\documentclass[1 [leqno,11pt]{amsart}
\usepackage{amssymb, amsmath,latexsym,amsfonts,amsbsy, amsthm,mathtools,graphicx,CJKutf8,CJKnumb,CJKulem,color}
\usepackage{float}
\usepackage{hyperref}

\makeatletter
\newif\ifmsbmloaded@
\def\loadmsbm{\msbmloaded@true
  \font\tenmsb=msbm10 scaled 1\@ptsize00
  \font\sevenmsb=msbm7 scaled 1\@ptsize00
  \font\fivemsb=msbm5 scaled 1\@ptsize00
  \alloc@8\fam\chardef\sixt@@n\msbfam
  \textfont\msbfam=\tenmsb
  \scriptfont\msbfam=\sevenmsb
  \scriptscriptfont\msbfam=\fivemsb
  }
\@addtoreset{equation}{section}

\makeatother \loadmsbm

\def\R{\mathbb R}

\def\Z{\mathbb Z}
\def\C{\mathbb C}

\def\T{{\mathbb T}}

\def\no{\noindent}

\def\f#1#2{\frac{#1}{#2}}
\def\f{\frac}

\def\eps{\epsilon}
\def\veps{\varepsilon}

\def\pa{\partial}
\def\p{\partial}

\def\om{\omega}
\def\Om{\Omega}

\def\be{\beta}
\def\ve{\varepsilon}
\def\na{\nabla}
\def\la{\lambda}

\def\al{\alpha}

\def\cL{{\mathcal L}}
\def\bu{\overline u}
\def\tu{\widetilde u}

\def\tW{\widetilde W}

\newcommand{\beq}{\begin{equation}}
\newcommand{\eeq}{\end{equation}}
\newcommand{\ben}{\begin{eqnarray}}
\newcommand{\een}{\end{eqnarray}}
\newcommand{\beno}{\begin{eqnarray*}}
\newcommand{\eeno}{\end{eqnarray*}}
\allowdisplaybreaks

\newtheorem{Theorem}{Theorem}[section]
\newtheorem{Lemma}[Theorem]{Lemma}

\newtheorem{remark}{Remark}[section]
\newtheorem{Proposition}{Proposition}[section]

\begin{document}

\title[Transition threshold for the 3D Couette flow]
{Transition threshold for the  3D Couette flow in a finite channel}

\author{Qi Chen}
\address{School of Mathematical Science, Peking University, 100871, Beijing, P. R. China}
\email{chenqi940224@gmail.com}

\author{Dongyi Wei}
\address{School of Mathematical Science, Peking University, 100871, Beijing, P. R. China}
\email{jnwdyi@163.com}

\author{Zhifei Zhang}
\address{School of Mathematical Science, Peking University, 100871, Beijing, P. R. China}
\email{zfzhang@math.pku.edu.cn}

\date{\today}

\maketitle

\begin{abstract}
In this paper, we study nonlinear stability of the 3D plane Couette flow $(y,0,0)$ at high Reynolds number ${Re}$ in a finite channel $\T\times [-1,1]\times \T$. It is well known that the plane Couette flow is linearly stable for any Reynolds number. However, it could become nonlinearly unstable and transition to turbulence for small but finite perturbations at high Reynolds number. This is so-called Sommerfeld paradox. One resolution of this paradox is to study the transition threshold problem, which is concerned with how much disturbance will lead to the instability of the flow and the dependence of disturbance on the Reynolds number. This work shows that if the initial velocity $v_0$ satisfies $\|v_0-(y,0,0)\|_{H^2}\le c_0{Re}^{-1}$ for some $c_0>0$ independent of $Re$, then the solution of the 3D Navier-Stokes equations is global in time and does not transition away from the Couette flow in the $L^\infty$ sense, and rapidly converges to a streak solution for $t\gg Re^{\f13}$ due to the mixing-enhanced dissipation effect. This result confirms the transition threshold conjecture proposed by Trefethen et al.(Science, 261(1993), 578-584). To this end, we develop the resolvent estimate method to establish the space-time estimates for the full linearized Navier-Stokes system around the flow $(V(t,y,z), 0,0)$,  where $V(t,y,z)$ is a small perturbation(but independent of $Re$) of  the Couette flow $y$.
\end{abstract}

\tableofcontents

\section{Introduction}

Hydrodynamic stability at high Reynolds number has been an important and very active field in the fluid mechanics. Beginning with Reynolds's famous paper in 1883 \cite{Rey},  many famous physicists and mathematicians made an important contribution to this field, such as Rayleigh, Kelvin, Orr, Sommerfeld, Heisenberg, Prandtl, Taylor, Arnold, Kolmogorov,  Lin etc. This field is mainly concerned with how the laminar flows become unstable and transition to turbulence \cite{Sch, Yag}.

On one hand, the eigenvalue analysis showed that the plane Couette flow is linearly stable for any Reynolds number $Re\ge 0$ \cite{Rom, DR}. It has been a folklore conjecture since Reynolds's experiment in 1883  that the pipe Poiseuille flow is linearly stable for any Reynolds number. In a recent work \cite{CWZ},  we prove the linear stability of pipe Poiseuille flow for general perturbations at high Reynolds number regime. On the other hand, the experiments showed that these flows could be unstable and transition to turbulence for small but finite perturbations at high Reynolds number \cite{Cha, GG, OK, Sch, Yag}. In addition, some laminar flows such as plane Poiseuille flow become turbulent at much lower Reynolds number than the critical Reynolds number predicted by the eigenvalue analysis.
This is so-called Sommerfeld paradoxes. The resolution of these paradoxes is a long-standing problem in fluid mechanics. There are many attempts to resolve these paradoxes(see \cite{Cha} and references therein).

The linear mechanism leading to the transition is very different from those due to the existence of growing modes(or unstable eigenvalues). This kind of transition is called subcritical transition or by-pass transition in physical literature. In \cite{TTR}, Trefethen et al provided an  important understanding of this transition from the viewpoint of pseudospectra of the linear operator. For both Couette flow and Poiseuille flow,  their pseudospectra includes the unstable domain, although their spectrum is stable. This phenomena is due to the non-normality of the linear operator. Psuedospectra has become an important concept for analyzing the stability of non-normal operators \cite{Tre}. Another consequence of non-normality gives rise to the transient growth of the solution of the linear evolution equation:
\beno
\pa_tu+A u=0, \quad u(0)=u_0,
\eeno
where $A$ is a non-normal operator. That is, the solution $\|u(t)\|_X$ could grow polynomially in time, even if $A$ has no unstable eigenvalues.

To understand how the interaction of linear and nonlinear mechanisms leads to the transition to turbulence, an important question firstly proposed by Trefethen et al. \cite{TTR} is to study the transition threshold problem, which is concerned with how much disturbance will lead to the instability of the flow and the dependence of disturbance on the Reynolds number. This idea may be traced back to Kelvin, who wrote in 1887(see \cite{Kel, TTR}): \smallskip

{\it
It seems probable, almost certain, indeed, that $\cdots$ the steady motion is stable for any viscosity, however small; and the practical unsteadiness pointed out by Stokes forty-four years ago, and so admirable investigated by Osborn Reynolds, is to be explained by limits of stability becoming narrower and narrower the smaller is the viscosity.
}\smallskip

The following mathematical version of  {\bf transition threshold problem} was formulated by Bedrossian, Germain and Masmoudi \cite{BGM-AM, BGM-BAMS}: \smallskip

{\it
Given a norm $\|\cdot\|_X$, find a $\beta=\beta(X)$ so that
\beno
&&\|u_0\|_X\le Re^{-\beta}\Longrightarrow  {stability},\\
&&\|u_0\|_X\gg Re^{-\beta}\Longrightarrow  {instability}.
\eeno
}
The exponent $\beta$ is referred to as the transition threshold. It was conjectured by Trefethen et al. in \cite{TTR}  that $\beta\le 1$:\smallskip

{\it Notwithstanding these qualification, we conjecture that transition to turbulence of eigenvalue-stable shear flows proceeds analogously to our model in that the destabilizing mechanism is essentially linear in the sense described above and the amplitude threshold for transition is $O(Re^\gamma)$ for some $\gamma<-1$.
}
\smallskip

A lot of works \cite{BT, Cha, DBL, LK, LHR, RSB, Wal, Yag}  in applied mathematics and physics are devoted to estimating $\beta$. Numerical experiments by Lundbladh, Henningson and  Reddy \cite{LHR} and  formal asymptotic analysis  by Chapman
\cite{Cha}  indicated that

1. Plane Couette flow
\begin{itemize}

\item[(1)] Numerical experiments: $\beta=1$ for streamwise perturbation and $\beta=\f 54$ for oblique perturbation;

\item[(2)] Asymptotic analysis: $\beta=1$ for streamwise  and  oblique perturbation.

\end{itemize}

2. Plane Poiseuille flow

\begin{itemize}

\item[(1)]  Numerical experiments: $\beta=\f74$ for streamwise  and  oblique perturbation;

\item[(2)]  Asymptotic analysis: $\beta=\f32$  for streamwise perturbation and $\beta=\f 54$ for oblique perturbation.
\end{itemize}
Furthermore, it was shown in \cite{Cha} why the numerically determined threshold exponents are not the true asymptotic values.
Formal asymptotic analysis in \cite{Cha}  confirms the conjecture for the Couette flow proposed by Trefethen et al. in \cite{TTR}.
\smallskip

In the absence of physical boundary(i.e., $\mathbb{T}\times \R\times \T$),  in the works \cite{BGM-AM, BGM-MAMS-1, BGM-MAMS-2},  Bedrossian, Germain and Masmoudi  made an important progress on the transition threshold problem for the 3-D Couette flow. It was shown that  $\beta\le 1$ for the perturbations in Gevrey class and $\beta\le \f32$ for the perturbations in Sobolev space.
More precisely, the authors in \cite{BGM-AM} showed that if the initial perturbation $u_0$ satisfies $\|u_0\|_{H^\sigma}\le \delta \nu^{\f32}$ for $\sigma>\f92$, then the solution is global in time, remains within $O(\nu^\f12)$ of the Couette flow in $L^2$ for any time, and converges to the streak solution for $t\gg \nu^{-\f13}$. In a recent work \cite{WZ-3D},  the later two authors proved that $\beta\le 1$ also for the perturbations in Sobolev space, which means that the regularity of the initial data(at least above $H^2$ regularity)  does not play an important role in determining the transition threshold.

In the presence of physical boundary, at high Reynolds number regime, the boundary layer  could  affect the stability of the flow. To understand the boundary layer effect, in a joint work \cite{CLWZ} with Li, we study the transition threshold problem of the 2-D Couette flow in a finite channel $\T\times  [-1,1]$. Since the 2-D Navier-Stokes equations have no lift-up effect, nonlinear effect is weaker so that the threshold is much smaller. More precisely, it was showed that if $\|u_0\|_{H^2}\le c_0\nu^\f12$ for some $c_0>0$,  then the solution will remain within $O(\nu^\f12)$ of the Couette flow in $L^\infty$ for any time. This result is consistent with one for the case of  $\Om=\T\times \R$
considered in \cite{BMV}. In a recent work \cite{MZ}, the threshold has been improved to  $\beta\le \f13$ when $\Om=\T\times \R$.
It remains a very interesting problem whether the threshold can be improved to $\beta\le \f13$ when $\Om=\T\times [-1,1]$.
Our previous work in 2D provides a foundation to study the 3D problem. In particular, the resolvent estimate method developed in \cite{CLWZ} is still very key in 3D case. Main challenges in 3D is to study how various linear effects(especially, boundary layer effect) and strong nonlinear  effect  interact to determine the transition threshold. \smallskip

The goal of this paper is to solve the transition threshold conjecture for the 3D plane Couette flow $U(y)=(y,0,0)$ in a finite channel $\Om=\T\times [-1,1]\times \T$. We consider the 3D incompressible Navier-Stokes equations at high Reynolds number $Re$ regime:
\begin{align*}
\left\{
\begin{aligned}
&\partial_tv-\nu\Delta v+v\cdot \nabla v+\nabla p=0,\\
&\nabla\cdot v=0,\\
&v(0,x,y,z)=v_0(x,y,z),\quad x,z\in\T,\, y\in [-1,1].
\end{aligned}
\right.
\end{align*}
where $v=\big(v^1(t,x,y,z),v^2(t,x,y,z),v^3(t,x,y,z)\big)$ is the velocity, $p(t,x,y,z)$ is the pressure, and $\nu=Re^{-1}>0$ is the viscosity coefficient.

We introduce the perturbation $u(t,x,y,z)=v(t,x,y,z)-U(y)$, which solves
\begin{align}\label{eq:NSp}
\left\{
\begin{aligned}
&\partial_t u-\nu\Delta u+y\partial_x u+\left(\begin{array}{l}u^2\\0\\0\end{array}\right)+\nabla p^{L}+u\cdot\nabla u+\nabla p^{NL}=0,\\
&\nabla\cdot u=0,\\
&u(0,x,y,z)=u_0(x,y,z),
\end{aligned}
\right.
\end{align}
together with the nonslip boundary condition
\ben\label{eq:NSp-bc}
u(t,x,\pm 1,z)=0.
\een
Here the pressure $p^{L}$ and $p^{NL}$ are determined by
\begin{align}
\label{eq:pressure}
\left\{
\begin{aligned}
&\Delta p^{L}=-2\partial_xu^2,\\
&\Delta p^{NL}=-\text{div}(u\cdot\nabla u)=-\partial_iu^j\partial_ju^i,\\
&(\partial_yp^{L}-\nu\Delta u^2)|_{y=\pm1}=0,\quad \partial_yp^{NL}|_{y=\pm1}=0.
\end{aligned}
\right.
\end{align}

To state our result, we define
\beno
P_0f=\overline{f}=\f1 {2\pi}\int_{\T}f(x,y,z)dx,\quad P_{\neq}f=f_{\neq}=f-P_0f.
\eeno

Our main result is stated as follows.

\begin{Theorem}\label{thm:stability}
Assume that $u_0\in H^1_0(\Om)\cap H^2(\Om)$ with $\text{div}\,u_0=0$. There exist constants $\nu_0, c_0, \epsilon, C>0,$ independent of $\nu$ so that if $\|u_0\|_{H^2}\le c_0\nu$, ${0<\nu\leq \nu_0}$,
then the solution $u$ of the system \eqref{eq:NSp} is global in time and satisfies the following stability estimates:

\begin{itemize}

\item Uniform bounds and decay of the background streak:
\begin{align}
&\|\bar{u}^1(t)\|_{H^2}+\|\bar{u}^1(t)\|_{L^{\infty}}\leq C\nu^{-1}\min(\nu t+\nu^{2/3},e^{-\nu t})\|u_{0}\|_{H^2},\label{eq:u-uniform0}\\
&\|\bar{u}^2(t)\|_{H^2}+\|\bar{u}^3(t)\|_{H^1}+\|(\bar{u}^2,\bar{u}^3)(t)\|_{L^{\infty}}\leq Ce^{-\nu t}\|u_{0}\|_{H^2}.\label{eq:u-uniform1}\end{align}

\item Rapid convergence to a streak:\begin{align}
&\|(\partial_x,\partial_z)\partial_xu_{\neq}(t)\|_{L^2}+\|(\partial_x,\partial_z)\nabla u_{\neq}^2(t)\|_{L^{2}}+\|(\partial_x^2+\partial_z^2)u_{\neq}^3(t)\|_{L^2}+ \nu^{1/4}\|u_{\neq}^2(t)\|_{H^2}\label{eq:u-uniform2}\\ \nonumber&\quad+ \nu^{1/3}\|(u_{\neq}^1,u_{\neq}^3)(t)\|_{H^1}+\|u_{\neq}^2(t)\|_{L^{\infty}}+\nu^{1/6}\|(u_{\neq}^1,u_{\neq}^3)(t)\|_{L^{\infty}}\leq Ce^{-2\epsilon\nu^{1/3}t}\|u_{0}\|_{H^2},\\&\|u_{\neq}\|_{L^{\infty}L^2}+\sqrt{\nu}\|t(u_{\neq}^1,u_{\neq}^3)\|_{L^{2}L^2}+\|\nabla u_{\neq}^2\|_{L^{\infty}L^2}+\|\nabla u_{\neq}^2\|_{L^{2}L^2}\leq C\|u_{0}\|_{H^2}.\label{eq:u-uniform3}
\end{align}
\end{itemize}

\end{Theorem}

Let us give some remarks on our results.

\begin{itemize}

\item[1.] Our rigorous analysis shows that various linear effects(including 3D lift-up effect, boundary layer effect, inviscid damping and enhanced dissipation) and nonlinear interaction play an import role in determining the transition threshold.  Surprisingly, the transition threshold obtained in this paper is consistent with one for the case of $\Om=\T\times \R\times \T$ obtained in \cite{WZ-3D}.
This shows that 3D lift-up may be the main mechanism leading to the instability of the flow even in the presence of boundary layer effect. Our explanation on this surprise result is that weak nonlinear interaction(or null structure of nonlinear terms)  and good linear mechanisms(inviscid damping and enhanced dissipation)  counteract  the bad effect of the boundary layer.

\item[2.] Global stability estimates in particular imply that
\beno
\|u(t)\|_{L^\infty}\le Cc_0e^{-\nu t}\to 0\quad \text{as}\quad t\to +\infty.
\eeno
This means that the 3D Couette flow is  nonlinearly stable in $L^\infty$ sense when the perturbation is  $o(\nu)$  in $H^2$.

\item[3.] In \cite{MZ}, the authors formulated the following question on nonlinear enhanced
dissipation and inviscid damping:\smallskip

{\it
Given a norm $\|\cdot\|_X\ (X\subset L^2)$, find a $\beta=\beta(X)$ so that for $\|u_0\|_X\ll \nu^{\beta} $ and for any $t>0$
\beno
\|u_{\neq}(t)\|_{L^2}\le Ce^{-c\nu^{1/3}t}\|u_{0}\|_{X}\quad \text{and}\quad \|u_{\neq}^2\|_{L^2L^2}\le C\|u_{0}\|_{X}.
\eeno
}
Our results answer this question for the 3D Couette flow.

\item[4.] The transition threshold problem is very interesting in an infinite channel $\Om=\R\times [-1,1]\times \T$. In this case, we need to understand the long wave effect in the $x$ variable on the stability. In fact, we conjecture that the threshold may be strictly less than 1 in this case.

\item[5.] The dynamics above the threshold should be a challenging problem, which is out of our current method.

\item[6.]  Formal asymptotic analysis conducted in \cite{Cha} indicates that the profile of shear flows may affect the transition threshold. From the results in \cite{Cha}, it seems reasonable to conjecture that the threshold $\beta\le \f32$ for the plane Poiseuille flow. In \cite{LWZ},  Li, Wei and Zhang proved that the threshold $\beta\le \f74$ for the 3D Kolmogorov flow.
It is a very interesting question whether one can improve the threshold to $\beta\le \f32$.

\item[7.] The transition threshold for the pipe Poiseuille flow is completely open. However, this flow is probably the most interesting and important, because it is close to the setting of the experiment conducted by Reynolds in 1883. In fact, the linear stability is just proved by our work \cite{CWZ}.

\end{itemize}

\medskip

\no{\bf Notations.}  Throughout this paper, we denote by $C$ a constant independent of $\nu, k,\ell$ and $c_0, \veps_0, \eps_1$, which may be different from line to line. Moreover, $\veps_0, \eps_1$ are absolute small constants independent of $\nu, k,\ell$.\smallskip

The following notations will be constantly used throughout this paper:

\begin{itemize}

\item $\Gamma_j=\big\{(x,j,z)|x,z\in\T\}$ for $j\in\{\pm1\big\}$ and $\partial\Omega=\Gamma_1\cup\Gamma_{-1}.$

\item $P_0f=\overline{f}=\f1 {2\pi}\int_{\T}f(x,y,z)dx,\,\, P_{\neq}f=f_{\neq}=f-P_0f$.

\item $\eta=(k^2+\ell^2)^\f12$ and $\delta=\nu^\f13|k|^{-\f13}$.

\item $\kappa={\partial_zV}/{\partial_yV}, \,  \rho_1=\dfrac{\partial_y\kappa+\kappa\partial_z\kappa}{\partial_y V(1+\kappa^2)},\,\rho_2=\dfrac{\partial_z\kappa-\kappa\partial_y\kappa}{(1+\kappa^2)}.$
\item We denote by $\|\cdot\|_{L^p}$ the $L^p(D)$ norm with $D=\Omega$ or $D=I=(-1,1)$, which is easy to distinguish from the context(for example, $D=\Om$ in sections 4, 6, 8, 9).

\item We denote by $\|\cdot\|_{H^k}$ the Sobolev norm $\|\cdot\|_{H^k(D)}$ with  $D=\Omega$ or $D=I=(-1,1)$.

\item We denote by $\|\cdot\|_{L^qL^p}$ the space-time norm $\|\cdot\|_{L^q(0,t; L^p(D))}$ with $D=\Om$ or $I$ and $t=T$( in sections 11, 12, 13) or $+\infty$(somewhere in section 10, 14).

\item Summation convention: the repeated upper and lower indices are summed over $i,j\in\{1,2,3\}$ and $\al,\be\in\{2,3\}.$

\end{itemize}

\section{Linear and nonlinear mechanisms affecting the threshold}

There are four kinds of linear effects: 3D lift-up, boundary layer, inviscid damping and enhanced dissipation, which play a key role in determining the transition threshold.

\subsection{3D lift-up effect}
To avoid the boundary, we consider $y\in \R$. In this case,  the linearized system of \eqref{eq:NSp}  reads
\begin{align*}
\partial_t u-\nu\Delta u+y\partial_x u+\left(u^2, 0, 0\right)-\nabla\Delta^{-1} 2\partial_xu^2=0.
\end{align*}
Introduce new variables $(\overline{x},y,z)=(x-ty,y,z)$ and set $\tu(t,\overline{x},y,z)={u}(t,{x},y,z)$, which solves
\begin{align*}
\partial_t \tu-\nu\Delta_{L} \tu+\left(\tu^2, 0, 0\right)-\nabla_{L}\Delta_{L}^{-1} 2\partial_{\overline{x}}\tu^2=0,
\end{align*}
where $\nabla_{L}=(\partial_{\overline{x}},\partial_y-t\partial_{\overline{x}},\partial_z)$ and $\Delta_{L}=\nabla_{L}\cdot\nabla_{L}. $
Notice that $P_0\tu=\bu$, and hence it reads
\begin{align*}
\partial_t\bu-\nu\Delta {\bu}+\left({\bu}^2,0, 0\right)=0.
\end{align*}
The solution of this linear problem is given by
\begin{align*}
\bu(t)=\left(\begin{array}{c}e^{\nu t\Delta}({\bu}^1(0)-t{\bu}^2(0))\\e^{\nu t\Delta}{\bu}^2(0)\\e^{\nu t\Delta}{\bu}^3(0)\end{array}\right).
\end{align*}
This means that
\beno
\|\bu^1(t)\|_{L^2}\le Cte^{-\nu t}\|u(0)\|_{L^2}.
\eeno
This linear growth for times $t\lesssim1/\nu$ is known as the {\bf lift-up effect} first observed in \cite{EP}.
This effect is related to the non-normality of the linearized operator,
which may give rise to the transient growth of the solution even if the operator is spectrally stable \cite{TTR, Tre}.

This turns out to be the main mechanism leading to nonlinear instability of the flow in 3D case,  which is absence in 2D fluid flow due to the beautiful structure of the vorticity $\om=\pa_yv^1-\pa_xv^2$:
\beno
\pa_t\om-\nu\Delta \om+v\cdot\na\om=0.
\eeno

\subsection{Inviscid damping}

Let us consider the 2D linearized Euler equation around a shear flow $(U(y),0)$ in terms of the vorticity:
\beno
\pa_t\om+U(y)\pa_x\om+U''(y)\pa_x(-\Delta)^{-1}\om=0.
\eeno
In particular, when $U(y)=y$,  there holds
\beno
\pa_t\om+y\pa_x\om=0.
\eeno
Thus, $\|\om(t)\|_{L^p}$ is conserved for any time. However, Orr [25]  observed that the velocity will tend to $0$ as $t\to \infty$.
More precisely, there holds
\beno
\|u(t)\|\le C(1+t)^{-1}\|u_0\|_{H^2},\quad \|u^2(t)\|_{L^2}\le C(1+t)^{-2}\|u_0\|_{H^3}.
\eeno
This is so-called the {\bf inviscid damping}, which is  due to the mixing of the vorticity induced by a shear flow. This phenomena is analogous to Landau damping in plasma physics found by Landau \cite{Lan}.

For general shear flows, linear inviscid damping is also a  difficult problem. In a series of work \cite{WZZ-CPAM, WZZ-APDE, WZZ-AM}, Wei, Zhang and Zhao proved the linear inviscid damping for monotone flows and non-monotone flows including the Poiseuille and Kolmogorov flows. Let us refer to \cite{Z1, Z2, BCV, WZZ-CMP, RZ, Jia, CZ} and references therein  for related works and recent progress on linear inviscid damping.

Nonlinear inviscid damping is a challenging problem. Nonlinear Landau damping was proved by Mouhot and Villani \cite{MV}.
Bedrossian and Masmoudi \cite{BM} proved nonlinear inviscid damping for the Couette flow in the domain $\Om=\T\times \R$.
On the other hand, nonlinear Landau damping and inviscid damping do not hold  for the perturbations in Sobolev spaces of low regularity \cite{LZ1, LZ2}. Let us refer to \cite{DM, IJ1, IJ2} for recent important progress on nonlinear inviscid damping.

Let us turn  to the 2D linearized Navier-Stokes equations around the Couette flow in a finite channel $\T\times [-1,1]$:
\ben\label{eq:2DNS}
\left\{
\begin{aligned}
&\pa_t\om-\nu\Delta\om+y\pa_x\om=0,\\
&\Delta \varphi=\om,\,\, \varphi|_{y=\pm 1}=\pa_y\varphi|_{y=\pm 1}=0,\quad  u=\big(-\pa_y\varphi,\pa_x\varphi\big).
\end{aligned}\right.
\een
In a joint work \cite{CLWZ} with Li, we established the inviscid damping result of \eqref{eq:2DNS}  in the sense
\beno
\|u_{\neq}\|_{L^2L^2}\le C\|\om(0)\|_{H^1},
\eeno
which plays an important role for 2D transition threshold problem. \smallskip

In 3D, $\Delta u^2$ has a similar structure as the vorticity in 2D:
\beno
\pa_tW-\nu\Delta W+y\pa_xW=0, \quad \Delta u^2=W,\quad u^2|_{y=\pm1}=\pa_y u^2|_{y=\pm 1}=0.
\eeno

\subsection{Enhanced dissipation}
Let  us consider the diffusion-convection equation in $\Om=\T\times \R$:
\beno
\pa_t\om-\nu\Delta \om+y\pa_x\om=0.
\eeno
Introduce new variables $(\overline{x},y)=(x-ty,y)$ and set $\widetilde{\om}(t,\overline{x},y)=\om(t,{x},y)$.
Then the solution $\widehat{\widetilde{\om}}(t,k,\xi)=\int_{\T\times \R} \widetilde{\om}(t,x,y)e^{-ikx-i\xi y}dxdy$ takes the form
\beno
\widehat{\widetilde{\om}}_{\neq}(t,k,\xi)=e^{-\nu(2\pi)^2\int_0^t(k^2+(\xi-k\tau)^2)d\tau}
\widehat{{\om}}_{\neq}(0,k,\xi).
\eeno
Due to $\int_0^t(k^2+(\xi-k\tau)^2)d\tau\geq k^2t^3/12 $,  we deduce that
 \begin{align*}
\|{\om}_{\neq}(t)\|_{L^{2}}\leq e^{-c\nu t^3}\|{\om_{\neq}}(0)\|_{L^2}\leq Ce^{-c\nu^{1/3}t}\|{\om_{\neq}}(0)\|_{L^2},
\end{align*}
which also gives
\ben\label{eq:om-enh}
\nu^\f16\|e^{c\nu^{1/3}t}{\om}_{\neq}(t)\|_{L^2L^{2}}\le C\|\om_{\neq}(0)\|_{L^2}.
\een
Here the exponent $\nu t^3$ gives a dissipation time scale $\nu^{-1/3}$, which is much shorter than the dissipation time scale $\nu^{-1}$. We refer to this phenomenon as the {\bf enhanced dissipation}, which is also due to the mixing mechanism.
For the system \eqref{eq:2DNS},  the following enhanced dissipation estimate was essentially proved in \cite{CLWZ}:
\beno
\nu^\f14\|e^{c\nu^{1/3}t}{\om}_{\neq}(t)\|_{L^2L^{2}}\le C\|\om_{\neq}(0)\|_{H^1}.
\eeno
Compared with \eqref{eq:om-enh}, the loss of  $\nu^\f1 {12}$ is due to the boundary layer effect. \smallskip

In \cite{CKR}, Constantin et al. gave a sufficient and necessary condition for general incompressible flow on a compact manifold.
However, the quantitative enhanced dissipation rate is usually hard to obtain except for some special flows such as shear flow, spiral flow and  Anosov flow \cite{CDE}.

The enhanced dissipation for the Kolmogorov flow $(e^{\nu t}\cos y,0)$, which is a solution of the 2D Navier-Stokes equations on the torus, has been proved by using different methods: resolvent estimate method \cite{LWZ, IMM}, wave operator method \cite{WZZ-AM} and hypocoercivity method \cite{WZ-SCM, BW}. More precisely, consider the linearized 2D Navier-Stokes equations around the Kolmogorov flow in $\T_{2\pi\delta}\times\T_{2\pi}$:
\beno
\pa_t\om-\nu\Delta \om+e^{-\nu t}\cos y\pa_x\big(1+\Delta^{-1}\big)\om=0.
\eeno
If $\delta\in (0,1)$, then it holds that for $t\lesssim \nu^{-1}$,
\beno
\|\om_{\neq}(t)\|_{L^2}\le Ce^{-c\nu^\f12 t}\|\om_{\neq}(0)\|_{L^2}.
\eeno
Here the enhanced dissipation rate is smaller than one for the Couette flow. This leads to conjecture that for stable monotone shear flows to the Euler equations, the enhanced dissipation rate should be $\nu^\f13$, and the rate should be $\nu^\f12$ for stable shear flows with non degenerate critical point.

In additional to the transition threshold problem, the enhanced dissipation also plays an important role for the suppression of blow-up in the Keller-Segel system \cite{KX, BH, He}  and axisymmertrization of 2D viscous vortices \cite{Ga}.  Let us refer to \cite{BC, GNR, LX} for more relevant works.

\subsection{Boundary layer effect}
Using the Laplace transform, the system \eqref{eq:2DNS}  can be reduced to solving the Orr-Sommerfeld(OS)  equation
\begin{align}\label{eq:OS}
\left\{
\begin{aligned}
&-\nu(\partial_y^2-k^2)^2\phi+ik(y-\lambda)(\partial_y^2-k^2)\phi=F,\\
&\phi(\pm1)=0,\quad\phi'(\pm1)=0.
\end{aligned}
\right.
\end{align}
In general, when $\nu\to 0$,  the solution of  \eqref{eq:OS} does not converge in a strong sense to the solution of  the  Rayleigh equation
\beno
ik(y-\lambda)(\partial_y^2-k^2)\phi=F,\quad \phi(\pm1)=0,
\eeno
because of the mismatch of their boundary conditions.

In \cite{CLWZ},  we decompose the solution of  \eqref{eq:OS}  as
\beno
\phi=\phi_{Na}+c_1\phi_{b,1}+c_2\phi_{b,2},
\eeno
where $\phi_{Na}$ solves the OS equation with the Navier-slip boundary condition:
\begin{align*}
\left\{
\begin{aligned}
&-\nu(\partial_y^2-k^2)^2\phi_{Na}+ik(y-\lambda)(\partial_y^2-k^2)\phi_{Na}=F,\\
&\phi_{Na}(\pm1)=0,\quad\phi_{Na}''(\pm1)=0,
\end{aligned}
\right.
\end{align*}
and $\phi_{b,i},i=1,2$ solves the homogeneous OS equation:
\begin{align*}
\left\{
\begin{aligned}
&-\nu (\partial_y^2-k^2)^2\phi_{b,1}+ik(y-\lambda) (\partial_y^2-k^2)\phi_{b,1}=0, \\
& \phi_{b,1}(\pm 1)=0, \quad\phi_{b,1}'(1)=1,\quad\phi_{b,1}'(-1)=0,
\end{aligned}
\right.
\end{align*}
and
\begin{align*}
\left\{
\begin{aligned}
&-\nu (\partial_y^2-k^2)^2\phi_{b,2}+ik(y-\lambda) (\partial_y^2-k^2)\phi_{b,2}=0, \\
& \phi_{b,2}(\pm 1)=0, \quad\phi_{b,2}'(-1)=1,\quad\phi_{b,2}'(1)=0.
\end{aligned}
\right.
\end{align*}
Via this decomposition, the boundary behavior of the solution $\phi$ can be described by $\phi_{b,i}$, and the coefficients $c_i, i=1,2$
are determined by $\phi_{Na}$. Let $w_{b,i}=(\pa_y^2-k^2)\phi_{b,i}$, which is a linear combination of the following two independent  Airy functions:
\begin{align*}
W_1(y)=Ai\big(e^{i\f{\pi}{6}}(L(y-\lambda-i k\nu))\big),\quad W_2(y)=Ai\big(e^{i\f{5\pi}{6}}(L(y-\lambda-i k\nu))\big),
\end{align*}
where $L=\big(\f k \nu\big)^\f13$. In some sense, this means that the width of the boundary layer is $\nu^{\f13}$. In particular,
there is a loss of $L$ when taking the derivative in $y$ variable.

Main advantage of this decomposition is that  the resolvent estimate of $\phi_{Na}$ can be derived by using the energy method under the Navier-slip boundary conditions, and the estimates of $\phi_{b,i}$ can be obtained by using the estimates of the Airy function. This idea introduced in  \cite{CLWZ} could be used to study the stability problem for general shear flows.

Recently, Grenier, Guo and Nguyen developed Rayleigh-Airy iteration method to solve the OS equation \cite{GGN-DM, GGN-AM}.
Gerard-Varet, Maekawa and Masmoudi  applied their method to the stability of the boundary layer shear flows \cite{GMM, GM}.

\subsection{Streak solution and nonlinear interaction}

 If the initial data in \eqref{eq:NSp} is independent of $x$, then so is the solution, i.e., $ u(t,x,y,z)=u(t,y,z)$. In this case, $(u^{2}, u^3)$ solves the 2D Navier-Stokes equations, and  $u^1$ solves the linear advection-diffusion equation
\begin{align*}
\left\{
\begin{aligned}
&\partial_t {u}^1-\nu\Delta {u}^1+(u^2\partial_y+u^3\partial_z){u}^1+u^2=0,\\
&\partial_t {u}^2-\nu\Delta {u}^2+(u^2\partial_y+u^3\partial_z){u}^2+\pa_yp=0,\\
&\partial_t {u}^3-\nu\Delta {u}^2+(u^2\partial_y+u^3\partial_z){u}^3+\pa_zp=0,\\
&\pa_yu^2+\pa_zu^3=0.
\end{aligned}\right.
\end{align*}
This solution is referred to as {\bf the streak.} Our result shows that for $t\gg \nu^{-\f13}$, the streak solutions describe the dynamics of the system if the perturbation is below the threshold.\smallskip

If we decompose the solution $u$ into $\bu+u_{\neq}$, where $\bu$ denotes zero mode and $u_{\neq}$ denotes nonzero mode, then nonlinear interactions can be classified as follows:
\begin{itemize}
\item zero mode and zero mode interaction: $0\cdot0 \to0$;

\item zero mode and nonzero mode interaction: $0\cdot\neq \to\neq$;

\item nonzero mode and nonzero mode interaction: $\neq\cdot\neq \to\neq$ or $ \neq\cdot\neq \to0.$
\end{itemize}

 Due to the lift-up effect, main nonlinear effect comes from the interaction between the streak solution and nonzero modes, especially, $\bu^1\pa_xu_{\neq}, u^j_{\neq}\pa_j\bu^1(j=2,3)$. This seems a primary source so that the solution could become unstable and transition to turbulence if  the perturbation exceeds some threshold.

To study how nonlinear and linear mechanisms interact to bring about transition to turbulence,  in \cite{TTR}, the authors considered a $2\times 2$ model problem:
\beno
\f {d u} {dt}=Au+\|u\|B u,
\eeno
 where
 \beno
A= \left(
 \begin{array}{l}
 -R^{-1} \quad 1\\
 0 \quad -2R^{-1}
 \end{array}\right),\qquad  B= \left(
 \begin{array}{l}
 0 \quad -1\\
 1 \quad  0
 \end{array}\right).
  \eeno
  and $R$ is a large parameter. Now the linear part has a transient growth due to the non-normality of $A$, and nonlinear term does not
  create or destroy energy since $B$ is skew-symmetric. This simple model has a strong nonlinear bootstrapping effect so that the threshold amplitude is of order $R^{-3}$ not $R^{-1}$. However,  the Navier-Stokes  equations are different from this simple model in two aspects: (1) there are infinitely many different modes, most of which do not experience non-normal linear growth; (2)  nonlinear interactions between different modes probably make the quadratic nonlinearity  unrealistically strong. Thus, they conjectured that the amplitude threshold of  transition for the Naver-stokes equation is $O(Re^\gamma)$ for some $\gamma<-1$.

Indeed, for the Navier-Stokes equations, nonlinear term has some good(null) structures similar to null forms introduced in \cite{Kla}. These structures  may avoid the worst nonlinear interactions. In some sense, $u^2$ could be viewed as a good component, $\bu^1$ a bad component, and $\pa_x, \pa_z$ good derivatives, while $\pa_y$
 is a bad derivative. Since the nonlinear term takes the form $u^i\pa_iu^j$, there are no worst interactions such as  the interaction between $\bu^1$ and itself. For the term $\bu^2\pa_y u_{\neq}$,
 although $\pa_y$ is bad,  $\bu^2$ is good. In other words,  a bad term(or derivative) always accompanies a good one for nonlinear interactions.\smallskip

\section{Key ideas and ingredients of the proof}

\subsection{Reformulation of the perturbation system}

Recall that the perturbation $u=v-U$ satisfies
\begin{align*}
&\big(\partial_t -\nu\Delta +y\partial_x\big)u+\left(\begin{array}{l}u^2\\0\\0\end{array}\right)+\nabla p^{L}+u\cdot\nabla u+\nabla p^{NL}=0.
\end{align*}

For the zero mode $\bu$, there holds
\begin{align}
&(\partial_t-\nu\Delta )\bu^1+\bu^2+\overline{u\cdot\nabla u^1}=0,\label{eq:u1-zero}\\
&(\partial_t-\nu\Delta )\bar{u}^j+\partial_j \bar{p}+(\bar{u}^2\partial_y+\bar{u}^3\partial_z)\bar{u}^j+\overline{u_{\neq}\cdot\nabla u^j_{\neq}}=0,\quad j=2,3.\label{eq:u23-zero}
\end{align}

To estimate the nonzero modes, we will use the formulation in terms of the shear wise velocity $u^2$ and vorticity $\om^2=\pa_zu^1-\pa_xu^3$:
\begin{align}\label{eq:u2-w2}
\left\{\begin{aligned}
  &\partial_t(\Delta u^2)-\nu\Delta^2 u^2+y\partial_x\Delta u^2
  +(\partial_x^2+\partial_z^2)(u\cdot\nabla
  u^2)\\
  &\qquad-\partial_y\big[\partial_x(u\cdot\nabla u^1)+\partial_z(u\cdot\nabla
  u^3)\big]=0,\\
   &\partial_t \omega^2- \nu\Delta\omega^2 +y\partial_x\omega^2+\partial_zu^2
  +\partial_z(u\cdot\nabla u^1)-\partial_x(u\cdot\nabla u^3)=0,\\
   &\partial_y u^2(t,x,\pm1,z)=u^2(t,x,\pm1,z)=0,\quad \omega^2(x,\pm1,z)=0.
     \end{aligned}\right.
\end{align}

The idea of using $\Delta u^2$ may go back to Kelvin's original paper \cite{Kel}. The coupled system of $(\Delta u^2, \om^2)$ was used in many physical literatures such as \cite{Cha, Sch}, and our recent work on the stability of 3D Kolmogorov flow \cite{LWZ}, and
blow-up criterion in terms of one velocity component \cite{CheZ, CZZ}.
The main advantage of using $\Delta u^2$ is that the equation of $\Delta u^2$ does not destroy the linear structure. This important point has played an important role in the works \cite{BGM-AM, WZ-3D}.

\subsection{Key ingredients when $\Om=\T\times \R\times \T$}

In this subsection, we will review the framework and key ingredients  when  $\Om=\T\times \R\times \T$  in \cite{WZ-3D}, which will help  understand the difficulties and ideas of this work.

For $a\ge 0$, we introduce two norms
\begin{align*}
\|w\|_{Y_0}=&\|w\|_{L^{\infty}L^2}+\nu^\f12\|\nabla w\|_{L^{2}L^2},\\
\|w\|_{X_a}=&\|e^{a\nu^{1/3}t}w\|_{L^{\infty}L^2}+\nu^\f12\|e^{a\nu^{1/3}t}\nabla w\|_{L^{2}L^2}\\
&\quad+\nu^{1/6}\|e^{a\nu^{1/3}t} w\|_{L^{2}L^2}+\|e^{a\nu^{1/3}t}\nabla\Delta^{-1}\partial_x w\|_{L^{2}L^2}.
\end{align*}
The norm of $Y_0$ corresponds to the heat diffusion.
The weight $e^{a\nu^{1/3}t}$ in $X_a$ corresponds  to the enhanced dissipation,  and the fourth part of $X_a$ corresponds to the inviscid damping.

In \cite{WZ-3D},  the following energy functionals were introduced:
\begin{align*}
&E_1=\|\bu\|_{L^{\infty}H^4}+
\nu^{\frac{1}{2}}\|\nabla \bu\|_{L^{2}H^4}+\big(\|\partial_t\bu\|_{L^{\infty}H^2}+\|\bu(1)\|_{H^2}\big)/\nu,\\
&E_2=\|\Delta \bu^2\|_{Y_0}+\|\bu^3\|_{Y_0}+\|\nabla \bu^3\|_{Y_0}+\|\min(\nu^{\frac{2}{3}}+\nu t,1)^{\frac{1}{2}}\Delta \bu^3\|_{Y_0},\\
&E_3=\|\Delta u_{\neq}^2\|_{X_2}+\| (\partial_x^2+\partial_z^2)u_{\neq}^3\|_{X_2}+\nu^{\frac{2}{3}}\|\Delta u_{\neq}^3\|_{X_3},\\
&E_4= \|e^{2\nu^{1/3}t}(\partial_x,\partial_z)u_{\neq}\|_{L^{\infty}H^3}+
\nu^{\frac{1}{2}}\|e^{2\nu^{1/3}t}\nabla(\partial_x,\partial_z)u_{\neq}\|_{L^{2}H^3},\\
\tag{E5}\label{E5eq}&E_5=\|\partial_x^2 u^2\|_{X_3}+\| \partial_x^2u^3\|_{X_3}.
\end{align*}

Let us give some explanations about the energy functional:

\begin{itemize}

\item $E_1$ is introduced to control the zero mode.  Due to the lift-up effect, $E_1$ is expected to be $o(1)$ at best.

\item $E_2$  is introduced to control good components $\bu^2, \bu^3$. Since there is no lift-up in the equations of $\bu^2$ and $\bu^3$,  $E_2$ is expected to be $o(\nu)$.

\item $E_4$ is mainly introduced to control $E_1$, since $E_1$ involves the fourth order derivative of the solution. Due to the lift-up effect, it is also expected to be $o(1)$.

\item $E_3$ is introduced to control good components $\Delta u^2$ and $(\partial_x^2+\partial_z^2)u_{\neq}^3$.

\item The most key part $E_5$ is introduced to control $E_3$. A key difference with $E_3$  is to use  the  $X_3$ norm instead of the $X_2$ norm, which is very crucial to control some nonlinear terms with the lift-up effect such as $\bu^1\pa_xu_{\neq}$ and $u_{\neq}^j\pa_j\bu^1(j=2,3)$.
 \end{itemize}

The estimates of $E_1$ and $E_2$ are based on the direct energy method. It holds that
\begin{align}
&E_1\le C\big(\|\bu(1)\|_{H^4}+\nu^{-1}\|\bu(1)\|_{H^2}+\nu^{-1}E_2+\nu^{-1}E_3^2+\nu^{-1}E_3E_4\big),\label{eq:E1-old}\\
&E_2\le C\big(\|u(1)\|_{H^2}+\nu^{-1}E_3^2\big).\label{eq:E2-old}
\end{align}

The estimate of $E_3$ is based on the space-time estimates for the following linearized system:
\beno
\cL W=\Delta f_1,\quad \cL U-2\partial_x\partial_z\Delta^{-2}W=f_2,
\quad  \cL=\pa_t-\nu \Delta+y\pa_x.
\eeno
That is, if $P_0 W=P_0U=P_0 f_1=P_0 f_2=0,$ then it holds that
\begin{align*}
\|W\|_{X_a}^2+\|(\partial_x^2+\partial_z^2)U\|_{X_a}^2\leq& C\Big(\|W(1)\|_{L^2}^2+\|U(1)\|_{H^2}^2\\&+\nu^{-1}\|e^{a\nu^{1/3}t} \nabla f_1\|_{L^2L^2}^2+\nu^{-1}\|e^{a\nu^{1/3}t} (\partial_x,\partial_z) f_2\|_{L^2L^2}^2\Big).
\end{align*}
Taking $(W,U)=(\Delta u^2_{\neq}, u^3_{\neq})$,  it was proved that
\ben\label{eq:E3-old}
E_3\le C\big(\|u(1)\|_{H^2}+\nu^{-1}E_2E_3+E_1^2E_3+E_5+\nu^{-1}E_3^2\big).
\een

The estimate of $E_4$ is based on the space-time estimates for the linear equation:
\beno
\cL w=\partial_xf_1+f_2+\text{div}f_3.
\eeno
That is, if $P_0 w=P_0f_1=P_0f_2=P_0f_3=0$, then it holds that
 \begin{align*}
&\|w\|_{X_a}^2\leq C\Big(\|w(1)\|_{L^2}^2+\|e^{a\nu^{1/3}t}\nabla f_1\|_{L^2L^2}^2+\nu^{-\frac{1}{3}}\|e^{a\nu^{1/3}t} f_2\|_{L^2L^2}^2+\nu^{-1}\|e^{a\nu^{1/3}t} f_3\|_{L^2L^2}^2\Big).
\end{align*}
On the right hand side, the second part used the inviscid damping, the third part used the enhanced dissipation, and the last part used the heat diffusion.
Taking $w=u_{\neq}$, it was proved that
\begin{align}
E_4\leq& C\Big(\|u(1)\|_{H^4}+\nu^{-1}E_3+E_4\big((E_3+E_2)/\nu+E_1\big)\nonumber\\&\qquad+E_1(E_3+E_5)/\nu+E_3E_2/\nu^2\Big).\label{eq:E4-old}
\end{align}

The estimate of $E_5$ is the most difficult.  For this part, we first need to establish the space-time estimates of  the linearized operator with variable coefficient:
\beno
\cL_V=\pa_t-\nu \Delta+V\pa_x,\quad V=y+\bu^1(t,x,z).
\eeno
This can be reduced to the case of $\cL$ by using the coordinate transform under the following key assumption:
\ben\label{ass:u1-old}
\|\bu^1\|_{L^\infty H^4}+\nu^{-1}\|\pa_t\bu^1\|_{L^\infty H^2}\le \veps.
\een
To estimate $E_5$, the most important idea in \cite{CLWZ}  is to introduce a quantity $W^2$ defined by
\beno
W^2=u_{\neq}^2+\kappa u_{\neq}^3,
\eeno
where  $\kappa(t,y,z)={\partial_zV}/{\partial_yV}$.
Then $\Delta W^2$ satisfies
\beno
\cL_V\Delta W^2=\Delta\big(-2\nu\nabla\kappa\cdot\nabla u_{\neq}^3\big)+\text{good terms}.
\eeno
However, the term $\Delta\big(-2\nu\nabla\kappa\cdot\nabla u_{\neq}^3\big)$ is still very singular. To handle it,  an important decomposition was introduced
\ben
 \nabla\kappa\cdot\nabla u_{\neq}^3=\rho_1\nabla V\cdot\nabla u_{\neq}^3+\rho_2(\partial_z-\kappa\partial_y)u_{\neq}^3,\label{eq:u3-decom}
\een
where
\beno
 \rho_1=\dfrac{\partial_y\kappa+\kappa\partial_z\kappa}{\partial_y V(1+\kappa^2)},\quad \rho_2=\dfrac{\partial_z\kappa-\kappa\partial_y\kappa}{(1+\kappa^2)}.
 \eeno
Since $(\partial_z-\kappa\partial_y)$ has a good commutative relation with $\cL_V$, it is a good derivative.  So, the second term in \eqref{eq:u3-decom} is good. To remove the singularity from the first term in \eqref{eq:u3-decom}, we need to introduce $W^{2,2}$ with solving
\beno
\cL_VW^{2,2}=-\rho_1\na V\cdot \na u^3_{\neq},\quad W^{2,2}(1)=0.
\eeno
We define
\beno
W^{2,1}=W^{2}-\nu W^{2,2},
\eeno
With this decomposition, it was found that $\Delta W^{2,1}$
satisfies a good equation:
\beno
\cL_V\Delta W^{2,1}=\text{good terms}.
\eeno
Then the space-time estimate $\|\Delta W^{2,1}\|_{X_3}$ together with the space-time estimate
\beno
\sum_{j=2}^3\big(\|\pa_x^2u^j_{\neq}\|_{X_3}+\|\pa_x(\pa_x-\kappa\pa_z)u^j_{\neq}\|_{X_3}\big)+\nu^{\f 23}\|\Delta u^3\|_{X_3}
\eeno
yields that
\ben\label{eq:E5-old}
E_5\le C\big(\|u(1)\|_{H^2}+\nu^{-1}E_3^2\big).
\een

With the estimates of $E_1$-$E_5$ and $\|u_0\|_{H^2}\le c_0\nu$, we can conclude by using a bootstrap argument that
\beno
E_1+E_4\le Cc_0, \quad E_2+E_3+E_5\le Cc_0\nu.
\eeno

\subsection{New ingredients and ideas when $\Om=\T\times [-1,1]\times \T$} It seems hard to apply Fourier multiplier method used in \cite{BGM-AM, BGM-MAMS-1} based on Fourier analysis to the case of finite channel.  There are two main advantages  of the framework introduced in \cite{WZ-3D}:

\begin{itemize}

\item[(1)]  the method is not strongly dependent on the use of Fourier analysis, although we used the Fourier transform for the space-time estimates of the linearized system;

\item[(2)] global stability is established in the Sobolev spaces of low regularity, which avoid the singularity due to the boundary layer effect when taking high order derivatives.

\end{itemize}

 Despite these advantages, there are still many challenging problems when applying this framework to the case of finite channel.  \smallskip

First of all,  the space-time estimates of $\cL_V$ strongly rely on the  assumption $\eqref{ass:u1-old}$. On the other hand, the estimate of $E_1$ depends on the energy $E_4$,  while $E_4$ involves the $H^4$ estimate of $u_{\neq}$. However, in the presence of the boundary, the high order derivative estimates in $y$ variable lead to singularity due to the boundary layer effect. Our new observations are:

\begin{itemize}

\item[(1)] The estimates of $E_3$ and $E_5$ do not depend on $E_4$(see \eqref{eq:E3-old} and \eqref{eq:E5-old});

\item[(2)] The assumption $\eqref{ass:u1-old}$ could be replaced by
\beno
\|\bar{u}^{1,0}\|_{L^{\infty}H^4}+\nu^{-1}\|\partial_t \bar{u}^{1,0}\|_{L^{\infty}H^2}\le \veps.
\eeno
\end{itemize}
Here we decompose $\bu^1=\bu^{1,0}+\bu^{1,\neq}$ with
\begin{align}
\label{eq:u10}&(\partial_t-\nu\Delta )\bar{u}^{1,0}+\bar{u}^2+\bar{u}^2\partial_y\bar{u}^{1,0}+\bar{u}^3\partial_z\bar{u}^{1,0}=0,\\
\label{eq:u1n}&(\partial_t-\nu\Delta )\bar{u}^{1,\neq}+\bar{u}^2\partial_y\bar{u}^{1,\neq}+\bar{u}^3\partial_z\bar{u}^{1,\neq}+\overline{u_{\neq}\cdot\nabla u_{\neq}^1}=0,\\
\label{boundary condition}& \bar{u}^{1,0}|_{t=0}=0,\quad\bar{u}^{1,\neq}|_{t=0}=\bar{u}^1(0),\quad\bar{u}^{1,0}|_{y=\pm1}=0,\quad \bar{u}^{1,\neq}|_{y=\pm1}=0.
\end{align}
The key point of this decomposition is that $\bu^{1,\neq}$ has better decay in $\nu$, and thus $\bu^{1,\neq}\pa_x$ could be
viewed as a perturbation. Therefore,  we avoid the use of the energy $E_4$. We introduce the following energy functional to control the zero mode:
\beno
E_1=E_{1,0}+\nu^{-2/3} E_{1,\neq},
\eeno
where
\begin{align*}
&E_{1,0}=\|\bar{u}^{1,0}\|_{L^{\infty}H^4}+\nu^{-1}\|\partial_t \bar{u}^{1,0}\|_{L^{\infty}H^2}
+\nu^{-\f12}\|\partial_t\bar{u}^{1,0}\|_{L^2H^3},\\
&E_{1,\neq}=\|\bar{u}^{1,\neq}\|_{L^{\infty}H^2}+\nu^{\f12}\|\nabla\bar{u}^{1,\neq}\|_{L^2H^2},
\end{align*}
and the energy $E_2$ is defined by
\begin{align*}
E_2=&\|\Delta\bar{u}^2\|_{L^{\infty}L^2}
+\nu^{\f12}\|\nabla\Delta\bar{u}^2\|_{L^2L^2}+\nu^{\f12}\|\Delta\bar{u}^2\|_{L^2L^2}+\nu^{-\f12}\|\partial_t\nabla\bar{u}^2\|_{L^2L^2}\\
&+\|\nabla\bar{u}^3\|_{L^{\infty}L^2}+\nu^{\f12}\|\Delta\bar{u}^3\|_{L^2L^2}+\nu^{\f12}\|\nabla\bar{u}^3\|_{L^2L^2}
+\nu^{-\f12}\|\partial_t\bar{u}^3\|_{L^2L^2}\\&+\|\min((\nu^{\frac{2}{3}}+\nu t)^{\frac{1}{2}},1-y^2)\Delta\bar{u}^3\|_{L^{\infty}L^2}+\nu^{-\f12}\|\min((\nu^{\frac{2}{3}}+\nu t)^{\frac{1}{2}},1-y^2)\nabla\partial_t\bar{u}^3\|_{L^{\infty}L^2}\\
&\quad+\nu^{\f12}\|\min((\nu^{\frac{2}{3}}+\nu t)^{\frac{1}{2}},1-y^2)\nabla\Delta\bar{u}^3\|_{L^{2}L^2}.
\end{align*}

Next we introduce a similar part of $E_3$ defined by
\begin{align*}
 E_3=E_{3,0}+E_{3,1},
\end{align*}
where $E_{3,0}$ and $E_{3,1}$ are given by
\begin{align*}
 &E_{3,0}=\nu^{\f12}\|e^{2\epsilon\nu^{\f13}t}(\partial_x,\partial_z) \Delta u^{2}_{\neq}\|_{L^2L^2}+\nu^{\f34}\|e^{2\epsilon\nu^{\f13}t}\nabla\Delta u^{2}_{\neq}\|_{L^2L^2}+\|e^{2\epsilon\nu^{\f13}t}(\partial_x,\partial_z)\nabla u^{2}_{\neq}\|_{L^{\infty}L^2}\\
&\qquad\quad+\|e^{2\epsilon\nu^{\f13}t}\partial_x\nabla u^{2}_{\neq}\|_{L^{2}L^2}+\|e^{2\epsilon\nu^{\f13}t}(\partial_x^2+\partial_z^2) u^{3}_{\neq}\|_{L^{\infty}L^2}+\nu^{\f12}\|e^{2\epsilon\nu^{\f13}t}(\partial_x^2+\partial_z^2)\nabla u^{3}_{\neq}\|_{L^2L^2},\\
&E_{3,1}=\nu^{\f13}\big(\|e^{2\epsilon\nu^{\f13}t}\nabla \omega^2_{\neq}\|_{L^{\infty}L^2}
   +\nu^\f12\|e^{2\epsilon\nu^{\f13}t}\Delta \omega^2_{\neq}\|_{L^{2}L^2}\big).
\end{align*}
Here and in what follows, the space-time norm $\|\cdot\|_{L^qL^p}=\|\cdot\|_{L^q(0,T; L^p(\Om))}$ for $T>0$.\smallskip

The estimate of $E_3$ is based on the space-time estimates for the coupled system \eqref{eq:u2-w2} of $(\Delta u^2, \om^2)$.
For this, we need to make the space-time estimates for the linear equation $\pa_tw-\nu \Delta w+y\pa_x w=f$ in a finite channel with
nonslip boundary condition or Navier-slip boundary condition. Due to the boundary layer effect, this is highly nontrivial.
In our work \cite{CLWZ}, we have developed the resolvent estimate method to solve this problem. In 3-D case, the idea is basically similar.
\smallskip

The adaption of $E_5$ to the present setting is the most challenging.
The first important observation is that it is enough to replace $E_5$(given by \eqref{E5eq})  by
\begin{align*}
  E_5=\nu^{1/6}\|e^{3\epsilon\nu^{1/3}t}\partial^2_xu^2_{\neq}\|_{L^2L^2}
  +\nu^{1/6}\|e^{3\epsilon\nu^{1/3}t}\partial_x^2u^3_{\neq}\|_{L^2L^2}.
\end{align*}
Let $V=y+\bu^{1,0}(t,y,z)$ and $Au=\mathbb{P}\Big(\nu\Delta u-V\pa_xu-\big(\pa_yV(u^2+\kappa u^3), 0,0\big)\Big)$, here $\mathbb{P}$ is the Leray projection.
To estimate $E_5$, we need to consider the following linearized system
\begin{align*}
\partial_tu_{\neq}-Au_{\neq}+\vec{g}=0.
\end{align*}
Thus, we need to establish the space-time estimates (without exponential growth) for the linearized system with variable coefficient, which is completely open
in 3D case.  In fact, even for the following linearized equation
\ben\label{eq:LS-scalar}
\left\{
\begin{aligned}
&\pa_tw-\nu\Delta w+V\pa_xw=f,\\
&\Delta \varphi=w,\quad \varphi|_{y=\pm1}=\pa_y\varphi|_{y=\pm1}=0,
\end{aligned}\right.
\een
the linear stability also remains unknown.  In a very recent work by  Almog and Helffer \cite{AH}, the linear stability of the Couette flow under a small perturbation $U(y)$  was just proved.\smallskip

 In this work, main challenges are that
\begin{itemize}

\item[(1)]  we need to consider a system, which is much more complicated than the scalar equation \eqref{eq:LS-scalar}. All the difficulties for the domain $\Om=\T\times \R\times \T$ still exist, while the main difficulty for \eqref{eq:LS-scalar} only lies in the boundary conditions.

\item[(2)]  we need to consider general perturbations of the Couette flow, which depend on $(t,y,z)$;

\item[(3)] we need to establish both linear stability and  uniform resolvent estimates, which should embody various linear effects: boundary layer,  enhanced dissipation and inviscid damping;

\item[(4)] we need to derive the space-time estimates from the resolvent estimates by using the Laplace transform. The trouble is that it is not direct in the case when the perturbation depends on $t$.
\end{itemize}

\subsection{Sketch of the estimate of $E_5$}

First of all,  to derive the space-time estimates from the resolvent estimate,  we will use the method of freezing the coefficient to estimate $E_5$, which is sketched as follows. \smallskip

$\bullet$  {\bf Separation of the time interval}:  $[0,T]=\bigcup I_j $, where  $I_j=[t_j,t_{j+1})\cap[0,T]$ with  $t_j=j\nu^{-\f13}$ for
$ j\in[0,\nu^{\f13}T)\cap\Z$. Here the choice of $t_j$ is due to the enhanced dissipation rate.  We define
\beno
&&V_j(y,z)=y+\overline{u}^{1,0}(t_j,y,z),\quad  A_j=A_{[V_j]}=\mathbb{P}\Big(\nu\Delta-V_j\pa_x-\big(\pa_yV_j(u^2+\kappa_j u^3),0,0\big)\Big).
\eeno
It can be reduced to considering  the system in each interval $I_j$:
\begin{align*}
&\partial_tu_{\neq}-A_ju_{\neq}+\vec{g}=0\quad\text{for}\quad t\in I_j.
\end{align*}

$\bullet$ {\bf Decomposition of the solution}:  we define
\beno
\vec{g}_{(j)}(t)=0\quad \text{for}\quad t\not\in I_j,\quad \vec{g}_{(j)}(t)=\vec{g}+\sum_{k=0}^{j-1}(A_k- A_j)\vec{u}_{[k]}\quad \text{for}\quad t\in I_j,
\eeno
and  let $\vec{u}_{[j]} $  solve
\begin{align*}
&\partial_t\vec{u}_{[j]}-A_j\vec{u}_{[j]}+\vec{g}_{(j)}=0,\\
&\vec{u}_{[0]}(0)=P_{\neq}u(0),\quad\vec{u}_{[j]}(0)=0\quad \text{for}\quad j\in(0,\nu^{\f13}T)\cap\Z.
\end{align*}
Then there holds
\begin{align*}
u_{\neq}=\sum_{k=0}^{j}\vec{u}_{[k]}\quad \text{for}\quad t\in I_j,\ j\in[0,\nu^{\f13}T)\cap\Z,\quad \vec{u}_{[j]}(t)=0\quad \text{for}\quad 0\leq t<t_j.
\end{align*}

$\bullet$ {\bf Space-time estimates for fixed $j$} :  for  $ j\in(0,\nu^{\f13}T)\cap\Z$,  there holds
\begin{align*}
&\|e^{4\epsilon\nu^{1/3}t}\vec{u}_{[j]}\|_{L^2Z_j^2}\leq Ce^{4\epsilon j}\|\vec{g}_{(j)}\|_{L^2Z_j^1},\\
&\|e^{4\epsilon\nu^{1/3}t}\vec{u}_{[0]}\|_{L^2Z_0^2}\leq C\big(\|u(0)\|_{H^2}+\|\vec{g}\|_{L^2(I_0,Z_0^1)}\big),
\end{align*}
which can be deduced from the resolvent estimates. The definition of  $Z_j^k(k=1,2)$ norm is given in section 14.2.

$\bullet$ {\bf Summation}:  For $j\in[0,\nu^{\f13}T)\cap\Z$, let
\beno
&&a_j=\sum_{k=0}^j(j-k+1)\|\vec{u}_{[k]}\|_{L^2(I_j,Z_{j}^2)},\quad  b_j=e^{-4\epsilon j}\|e^{4\epsilon\nu^{1/3}t}\vec{u}_{[j]}\|_{L^2Z_j^2},\\
&&E_6^2=\sum_{j=0}^{N}e^{6\epsilon j}\|u_{\neq}\|_{L^2(I_j,Z_j^2)}^2,\quad N=\max( [0,\nu^{\f13}T)\cap\Z).
\eeno
Using the following important facts that
\begin{align*}
&\|{v}\|_{Z_{j}^2}-\|{v}\|_{Z_{k}^2}\leq C|j-k|^{\f12}E_1\|{v}\|_{Z_{k}^2},\\
&\|(A_k- A_j)\vec{u}_{[k]}\|_{Z_j^1}\leq C|j-k|E_1\|\vec{u}_{[k]}\|_{Z_{j}^2},
\end{align*}
we can deduce from the space-time estimates  that
\begin{align*}
&b_j\le C\big(\|\vec{g}\|_{L^2(I_j,Z_j^1)}+E_1a_j\big),\\
&\|\vec{u}_{[k]}\|_{L^2(I_j,Z_{j}^2)}\leq\big(1+C|j-k|^{\f12}E_1\big)e^{-4\epsilon (j-k)}b_k,
\end{align*}
which will yield that
\beno
a_j^2 \leq C\sum_{k=0}^j(j-k+1)^{5}e^{-8\epsilon (j-k)}b_k^2\leq C\sum_{k=0}^je^{-7\epsilon (j-k)}b_k^2.
\eeno
Then we can conclude that
\begin{align*}
E_6^2\leq\sum_{j=0}^{N}e^{6\epsilon j}a_j^2&\leq  C\|u(0)\|_{H^2}^2+C\sum_{j=0}^{N}e^{6\epsilon j}\|\vec{g}\|_{L^2(I_j,Z_j^1)}^2+CE_1^2E_6^2.
\end{align*}

Through the above procedure, the problem is reduced to  considering the following linearized resolvent system
 \begin{align*}\left\{\begin{aligned}
    &\big(-\nu\Delta+V\partial_x-i\lambda-a\nu^{\f13}\big){u}+\big(\partial_yV(u^2+\kappa u^3),0,0\big)+\nabla P+\vec{g}=0,\\
    &\text{div}\,{u}=0,\quad \Delta P=-2\partial_yV(\partial_xu^2+\kappa\partial_xu^3),\\
    &{u}|_{y=\pm1}=(\partial_yP+\nu\Delta v^2)|_{y=\pm1}=0.
    \end{aligned}\right.
 \end{align*}
 We need to prove that if $\vec{g}=0 $ and $ \lambda\in\R,\ a\in[0,\epsilon_1]$, then $u=0$. To our knowledge, these results are
new and may be of independent interest. The proof is highly nontrivial.\smallskip

 Motivated by \cite{WZ-3D}, it is natural to introduce $W=u^2+\kappa u^3,\, U=u^3$. Then $(W,U)$ satisfies
  \begin{align*}\left\{\begin{aligned}
     & -\nu\Delta W+(\partial_xV-i\lambda)W-a\nu^{\f13}W+(\partial_y+\kappa\partial_z)P\\
     &\qquad\quad+(g^2+\kappa g^3)+2\nu\nabla\kappa\cdot\nabla U+\nu(\Delta\kappa)U=0,\\
    & -\nu\Delta U+(\partial_xV-i\lambda)U-a\nu^{\f13}U+\partial_zP+g^3=0,\\
    &\Delta P=-2\partial_x(\partial_yVW),\\
    &W|_{y=\pm1}=\partial_yW|_{y=\pm1}=U|_{y=\pm1}=0.
  \end{aligned}\right.
  \end{align*}
Taking Fourier transform in $x$, it can be reduced to the following system
\begin{align}\label{eq:WU-s3}
  \left\{
  \begin{aligned}
   &-\nu\Delta W+ik(V(y,z)-\lambda)W-a(\nu k^2)^{1/3}W+(\partial_y+\kappa\partial_z)P\\&\qquad+G_1+\nu(\Delta \kappa) U+2\nu\nabla\kappa\cdot\nabla U=0,\\&-\nu\Delta U+ik(V(y,z)-\lambda)U-a(\nu k^2)^{1/3}U+G_2+\partial_zP=0,\\
   & W|_{y=\pm1}=\partial_yW|_{y=\pm1}=U|_{y=\pm1}=0,
   \end{aligned}\right.
\end{align}
where
\beno
\Delta P=-2ik\partial_yVW,\quad \partial_xW=ikW,\ \partial_xU=ikU,\ \partial_xP=ikP.
\eeno

To estimate $E_5$, the key ingredient is to establish the following resolvent estimates for the system \eqref{eq:WU-s3} under the assumption \eqref{ass:V}:
\begin{align*}
&\nu^{\f{1}{3}}\big(\|\partial_x^2U\|_{L^2}^2+\|\partial_x(\partial_z-\kappa\partial_y)U\|_{L^2}^2\big)+
\nu\big(\|\nabla\partial_x^2U\|_{L^2}^2+\|\nabla\partial_x(\partial_z-\kappa\partial_y)U\|_{L^2}^2\big)
\\&\qquad+\nu^{\f{1}{3}} \|\partial_x\nabla W\|_{L^2}^2+\nu\|\partial_x\Delta W\|_{L^2}^2+\nu^{\f{5}{3}} \|\partial_x\Delta U\|_{L^2}^2\\
&\leq C\nu^{-1}\big(\|\nabla G_1\|_{L^2}^2+\|\partial_x G_2\|_{L^2}^2\big).
\end{align*}

To this end,  we first need to remove the singular part $\nu W_s$ of $W$ with $W_s$ given by
\begin{align*}
  \left\{
  \begin{aligned}
   &-\nu\Delta  W_s+ik(V(y,z)-\lambda) W_s-a(\nu k^2)^{1/3} W_s+\rho_1\nabla V\cdot\nabla U=0, \\
    & W_s|_{y=\pm1}=0,\quad \partial_x W_s=ik W_s.
   \end{aligned}\right.
\end{align*}
Thus, the good unknown $W_g$ seems to be naturally defined by
\beno
W_g=W-\nu W_s.
\eeno
Since $\Delta U|_{y=\pm 1}\neq 0$, we need to introduce a boundary corrector $U_b$ of $U$ defined by Lemma \ref{lem:UB}.  The new good unknown $W_g$ is defined by (for the case $|\lambda|<1$)
\begin{align*}
  W_g=W-\nu\theta W_s-\kappa U_b,
\end{align*}
where $\theta(y)=\theta_0(|y-\lambda|/(1-|\lambda|))$ with a fixed $\theta_0\in  C_0^{\infty}(\R)$ so that $\theta_0(y)=1 $ for $|y|\leq 1/4,$ $\theta_0(y)=0 $ for $|y|\geq 1/2.$ Now $W_g$ satisfies the following system
\begin{align}\label{eq:Wg-s3}
\left\{
\begin{aligned}
&-\nu\Delta W_g+ik(V(y,z)-\lambda)W_g-a(\nu k^2)^{1/3}W_g+(\partial_y+\kappa\partial_z)\big(P^1+P^2\big)=F,\\
&\Delta P^{1}=-2ik\partial_yVW_g,\quad \partial_yP^{1}|_{y=\pm1}=0,\quad \Delta P^{2}=0,\\\
& W_g|_{y=\pm1}=0,\quad \partial_xW_g=ikW_g,\,\,\pa_xP^j=ikP^j(j=1,2).
\end{aligned}\right.
\end{align}
Finally, $w_g=\Delta W_g$ satisfies
\beno
-\nu\Delta w_g+ik(V(y,z)-\lambda)w_g-a(\nu k^2)^{1/3}w_g=\text{good terms}
\eeno
together with $W_g|_{y=\pm 1}=0$ and good Neumann data $\pa_yW_g|_{y=\pm 1}$.\smallskip

The construction of $W_g$ may be the most key part of this paper. The second key  part is to establish the resolvent estimates for the linearized system (given $\partial_y\varphi|_{y=\pm1}$)
\begin{align}\label{eq:OS-NB-s3}
  \left\{
  \begin{aligned}
   &-\nu\Delta w+ik(V(y,z)-\lambda)w-a(\nu k^2)^{1/3}w=F_1+F_2,\\
   &\Delta \varphi=w,\quad \varphi|_{y=\pm1}=0,\\
   &\partial_xw=ikw,\quad \partial_xF_1=ikF_1, \quad \partial_xF_2=ikF_2.
      \end{aligned}\right.
\end{align}

\subsection{Resolvent estimates with non-vanishing  Neumann data}

To estimate $E_3$ and $E_5$, we need to establish the space-time estimates for the linearized system. In the presence of the boundary, we can not use the Fourier transform method introduced in \cite{WZ-3D}. Instead, we will use the resolvent estimate method developed in \cite{CLWZ}.

To estimate $E_3$, it is enough to establish the resolvent estimates for the following linearized system when $V=y$:
\begin{align*}
  \left\{
  \begin{aligned}
   &-\nu\Delta w+ik(V-\lambda)w-a(\nu k^2)^{1/3}w=F,\\
   &\Delta \varphi=w,\quad\varphi|_{y=\pm1}=0,
   \end{aligned}\right.
\end{align*}
together with the Navier-slip boundary condition(i.e., $w|_{y=\pm 1}=0$) or  the Neumann data $\pa_y\varphi|_{y=\pm 1}\neq 0$.
Even in the case of  Navier-slip boundary condition or nonslip boundary condition(i.e., $\pa_y\varphi|_{y=\pm 1}=0$),  the results in \cite{CLWZ} can not be applied to the 3D case. In this work, we will develop a general framework for $V$ satisfying  \eqref{ass:V},
then apply general results to the special case of $V=y$.

To estimate $E_5$, we need to establish the resolvent estimates for the  linearized system \eqref{eq:OS-NB-s3}.
One of the key differences with \cite{CLWZ} is that in our applications, $\pa_y\varphi|_{y=\pm 1}\neq 0$. Thus, the resolvent estimates have to show the precise dependence on the Neumann data $\pa_y\varphi|_{y=\pm 1}$. For example, we show that
\begin{align*}
&\nu^{\f16}|k|^{\f43}\|\nabla\varphi\|_{L^2} \leq C\Big(\|F_1\|_{L^2}+|\nu/k|^{-\f13}\|{F}_2\|_{H^{-1}}+\nu^{\f13}|k|^{\f76}\|\partial_y\varphi\|_{L^2(\partial\Omega)}\Big),\\
&\|w\|_{L^2} \leq C\Big(\|(|k(y-\lambda)/\nu|^{\f14}+|k/\nu|^{\f16})\partial_y\varphi\|_{L^2(\partial\Omega)}
+|\nu/k|^{\f16}\|\partial_y\partial_z\varphi\|_{L^2(\partial\Omega)}\\&\quad\qquad\qquad+(\nu k^2)^{-\f{5}{12}}\|F_1\|_{L^2}+\nu^{-\f34} |k|^{-\f{1}{2}}\|{F}_2\|_{H^{-1}}\Big).
\end{align*}
In the case of nonslip boundary condition(i.e., $\pa_y\varphi|_{y=\pm 1}=0$), the above result shows
\begin{align*}
&\nu^{\f16}|k|^{\f43}\|\nabla\varphi\|_{L^2} \leq C\big(\|F_1\|_{L^2}+|\nu/k|^{-\f13}\|{F}_2\|_{H^{-1}}\big),\\
&\|w\|_{L^2} \leq C\big((\nu k^2)^{-\f{5}{12}}\|F_1\|_{L^2}+\nu^{-\f34} |k|^{-\f{1}{2}}\|{F}_2\|_{H^{-1}}\big),
\end{align*}
which are the same as those in the case when $V=y$(see \cite{CLWZ}). This result in particular implies  the linear stability of the flow near the Couette flow under the nonslip boundary condition. Thus, our work also gives a new  proof of  linear stability  in \cite{AH}.\smallskip

Following the idea introduced in \cite{CLWZ}, we decompose the solution of \eqref{eq:OS-NB-s3} as $w=w_{Na}+w_{I}$, where $w_{Na}$ and $w_I$ solve
\begin{align*}
  \left\{
  \begin{aligned}
   &-\nu\Delta w_{Na}+ik(V-\lambda)w_{Na}-a(\nu k^2)^{1/3}w_{Na}=F,\\
   &w_{Na}=\Delta\varphi_{Na},\ \varphi_{Na}|_{y=\pm1}=0,\ w_{Na}|_{y=\pm1}=0,\\
   &-\nu\Delta w_{I}+ik(V-\lambda)w_{I}-a(\nu k^2)^{1/3}w_{I}=0,\\
   &\Delta  \varphi_I=w_I,\,\,\varphi_{I}|_{y=\pm1}=0   \end{aligned}\right.
\end{align*}
Since $w_{Na}$ satisfies the Navier-slip boundary condition, under the assumption \eqref{ass:V}, we can follow the energy method developed in \cite{CLWZ} to establish the resolvent estimates in the case when the force $F\in H^1, F\in L^2$ or $F\in H^{-1}$(see Proposition \ref{prop:res-nav-s1}) and weak type resolvent estimates when $F\in H^{-1}$(see Proposition \ref{prop:res-weak}).

New difficulty  is that the solution $w_I$ can not be expressed by the Airy function when $V\neq y$. Consider the homogeneous problem
\begin{align*}
  \left\{
  \begin{aligned}
   &-\nu\Delta w+ik(V-\lambda)w-a(\nu k^2)^{1/3}w=0,\\
   &\Delta  \varphi=w,\,\,\varphi|_{y=\pm1}=0.   \end{aligned}\right.
\end{align*}
Our new idea is to make the decomposition $w=w_{Na}+w_b$, where
\begin{align*}
  \left\{
  \begin{aligned}
   &-\nu\Delta w_{Na}+ik(V-\lambda)w_{Na}-a(\nu k^2)^{1/3}w_{Na}=-ik(V-y)w_{b},\\
   &w_{Na}=\Delta\varphi_{Na},\ \varphi_{Na}|_{y=\pm1}=0,\ w_{Na}|_{y=\pm1}=0,\\
   &-\nu\Delta w_{b}+ik(y-\lambda)w_{b}-a(\nu k^2)^{1/3}w_{b}=0,\\
   &w_{b}=\Delta\varphi_b,\ \varphi_b|_{y=\pm1}=0.
 \end{aligned}\right.
\end{align*}
Since $w_{b}$ solves the homogeneous OS equation with constant coefficient,  the solution $\widehat{w}_{b,\ell}=\int_{\mathbb{T}}e^{-i\ell z}w_b(x,y,z)dz$ can be expressed  as
\begin{align*}
   & \widehat{w}_{b,\ell}=\partial_y\widehat{\varphi}_{b,\ell}(1)w_{`1,\ell} +\partial_y\widehat{\varphi}_{b,\ell}(-1)w_{2,\ell},
\end{align*}
where $\widehat{\varphi}_{b,\ell}=\int_{\mathbb{T}}e^{-i\ell z}\varphi_b(x,y,z)dz$, and  $(w_{1,\ell}, w_{2,\ell})$  given by \eqref{eq:w1l} and \eqref{eq:w2l} is the boundary corrector.
The estimates of $(w_{1,\ell}, w_{2,\ell})$ can be obtained by using the properties of the Airy function. Thus, ${w}_b$ can be controlled by the Neumann data $\pa_y{\varphi}|_{y=\pm 1}$
and $\pa_y\varphi_{Na}|_{y=\pm 1}$ due to $\pa_y\varphi_b=\pa_y\varphi-\pa_y\varphi_{Na}$. Main reason why this decomposition works well is due to the fact that $(V-y)|_{y=\pm 1}=0$ so that $(V-y)w_b$ is a good remainder.\smallskip

It should be emphasized  that the estimates of the Neumann data $\pa_y\varphi_{Na}|_{y=\pm 1}$ are  very skilled.
The proof will be based on the following key fact: if $\Delta\varphi=w,\, \varphi|_{y=\pm 1}=0$, then we have
\begin{align*}
   \|\partial_y\varphi\|_{L^2(\{y=1\})} &=\f{1}{(2\pi)^2}\sup_{f\in\mathcal{F}_1}|\left\langle w,f\right\rangle|,
\end{align*}
where
\begin{align*}
   \mathcal{F}_1=\Big\{f(x,y,z)=\sum_{\ell \in\mathbb{Z}}a_{\ell}\f{\sinh(\eta(1+y))}{\sinh(2\eta)}e^{ikx+i\ell z}\Big|\|\{a_\ell\}\|_{\ell^2}=1,\ \sup\{|\ell||a_{\ell}\neq 0\}<+\infty\Big\}.
\end{align*}
To estimate $\left\langle w,f\right\rangle$, we need to use the resolvent estimates of $w_{Na}$, especially weak type resolvent estimates given by Proposition \ref{prop:res-weak}. \smallskip

We believe that the resolvent estimate method developed in this paper and \cite{CLWZ} could be applied to the stability problem of general shear flows and the other related problems.

\section{Resolvent estimates with Navier-slip boundary condition}

In this section, we establish the resolvent estimates of the Orr-Sommerfeld equation with variable coefficient:
\begin{align}\label{eq:OS-Nav}
  \left\{
  \begin{aligned}
   &-\nu\Delta w+ik(V(y,z)-\lambda)w-a(\nu k^2)^{1/3}w=F,\\
   &w|_{y=\pm1}=0,\,\, \partial_xw=ikw,\,\, \partial_xF=ikF,
   \end{aligned}\right.
\end{align}
where $V(y,z)$ satisfies
\ben\label{ass:V}
\|V-y\|_{H^4}<\veps_0,\quad (V(y,z)-y)|_{y=\pm1}=0,
\een
with $\veps_0$ small enough determined later.\smallskip

In this section, we always assume that $\la\in\R$ and $a\in [0,\eps_1]$ for some $\eps_1>0$ small enough
determined later, and  let $\varphi$ solve
\ben\label{def:varphi}
\Delta\varphi=w, \quad \varphi|_{y=\pm1}=0.
\een

Let us first give some basic estimates involving $V$.

\begin{Lemma}\label{lem:V}
  It holds that for any $(y,z)\in [-1,1]\times \T$,
  \begin{align*}
      |V(y,z)-y| \leq C\varepsilon_0(1-y^2),\quad \f{1}{2}\leq \f{j-V(y,z)}{j-y}\leq 2\quad j\in\{\pm1\}.
  \end{align*}
\end{Lemma}

\begin{proof} Since $\|V-y\|_{H^4}\leq \varepsilon_0$ and $V-y|_{y=\pm1}=0$, we deduce that for $y\in [-1,0]$,
  \begin{align*}
     |V(y,z)-y|=&\left|\int_{-1}^{y} \partial_y(V-y)dy\right|\leq (1+y)\|\nabla(V-y)\|_{L^\infty}\\
     \leq& C\|V-y\|_{H^4}(1+y)\leq C\varepsilon_0(1+y)\leq C\varepsilon_0(1-y^2).
  \end{align*}
Similarly, for $y\in[0,1]$, we have $|V-y|\leq C\varepsilon_0(1-y^2)$.

Using the first inequality of the lemma, we get
\begin{align*}& \f{j-V}{j-y}=1+\f{y-V}{j-y}\leq 1+\f{|V-y|}{|j-y|}\leq 1+C\varepsilon_0\f{1-y^2}{|j-y|}\leq 1+C\varepsilon_0,\\
   & \f{j-V}{j-y}=1+\f{y-V}{j-y}\geq 1-\f{|V-y|}{|j-y|}\geq 1-C\varepsilon_0\f{1-y^2}{|j-y|}\geq 1-C\varepsilon_0,
\end{align*}
which imply by taking $\varepsilon_0$ sufficiently small so that $C\varepsilon_0\leq\f12$ that
\begin{align*}
 \f{1}{2}\leq \f{j-V(y,z)}{j-y}\leq 2\quad j\in\{\pm1\}.
  \end{align*}
  \end{proof}

 \begin{Lemma}\label{lem:chi1}
Let $\chi_1=(V-\lambda-i\delta)^{-1}$ for some $\delta>0$. It holds that
  \begin{align*}
  &\|\chi_1\|_{L^\infty}\leq C\delta^{-1},\quad \|\chi_1\|_{L^\infty_{x,z}L^2_y}\leq C\delta^{-\f12},\\
     &\|\nabla\chi_1\|_{L^\infty}\leq C\delta^{-2},\quad \|\nabla\chi_1\|_{L^\infty_{x,z}L^2_y}\leq C\delta^{-\f32}.
  \end{align*}
 Here $C$ is a constant independent of $\lambda, \delta$.
\end{Lemma}

\begin{proof}
The first and third inequality is obvious. Note that
  \begin{align*}
     &\|1/(|V-\lambda|+\delta)\|_{L^\infty_{x,z}L^2_{y}}^2=\left\|\int_{-1}^{1} \f{1}{(|V-\lambda|+\delta)^2}dy\right\|_{L^\infty}\\
     &\leq C\left\|\int_{-1}^{1} \f{\partial_yV}{(|V-\lambda|+\delta)^2}dy\right\|_{L^\infty}\|1/(\partial_yV)^{-1}\|_{L^\infty}\leq C\|1/(|y-\lambda|+\delta)\|_{L^2_y}^2\leq C\delta^{-1},
  \end{align*}
which gives the second inequality. The last inequality can be proved similarly.
\end{proof}

 \subsection{Resolvent estimates}

 In this subsection, we establish the resolvent estimates in the case when
 $F\in H^1, L^2$ or $H^{-1}$. The proof is similar to the case of $V=y$
 considered in \cite{CLWZ}.

 \begin{Proposition}\label{prop:res-nav-s1}
Let $w\in H^2(\Omega)$ be a solution of \eqref{eq:OS-Nav}  with $F\in H^1(\Omega)$. Then it holds that
\begin{align*}
&\nu^{\f23}|k|^{\f13}\|\nabla w\|_{L^2} +(\nu k^2)^{\f13}\|w\|_{L^2} +\nu\|\Delta w\|_{L^2}+|k|\|(V-\lambda) w\|_{L^2}\leq C\|F\|_{L^2},\\&\nu^{\f23}|k|^{\f13}\|\nabla w\|_{L^2} +(\nu k^2)^{\f13}\|w\|_{L^2} +\nu\|\Delta w\|_{L^2}\leq C\nu^{\f16}|k|^{-\f23}\|\nabla F\|_{L^2},\\
&\nu^{\f23}|k|^{\f13}\|\nabla w\|_{L^2} +(\nu k^2)^{\f13}\|w\|_{L^2}\leq C\nu^{-\f13}|k|^{\f13}\|F\|_{H^{-1}},\\
&\nu^{\f16}|k|^{\f43}\|\nabla\varphi\|_{L^2} +\nu^{\f16}|k|^{\f56}\|w\|_{L^2_{x,z}L^1_y } \leq C\|F\|_{L^2}.
\end{align*}
If $\nu k^2\leq 1$, then we have
\begin{align*}
   &\nu^{\f16}|k|^{\f43}\|\nabla^2[(V-\lambda)\varphi]\|_{L^2} +\nu^{\f16}|k|^{\f{11}{6}}\|\nabla[(V-\lambda)\varphi]\|_{L^2_{x,z}L^\infty_y}\leq C\|F\|_{L^2}.
\end{align*}
\end{Proposition}

\begin{proof}
{\bf Step 1}. Case of $F\in L^2(\Om)$.\smallskip

Taking $L^2$ inner product with $w$ to \eqref{eq:OS-Nav}, and integrating by parts, we obtain
\begin{align}\label{eq:w-grad}
   & \nu\|\nabla w\|^2_{L^2} \leq \|F\|_{L^2}\|w\|_{L^2}+a(\nu k^2)^{\f13}\|w\|^2_{L^2}.
\end{align}
Taking $L^2$ inner product with $(V-\lambda)w$ to \eqref{eq:OS-Nav}, we get
\begin{align*}
   &\left\langle -\nu\Delta w+ik(V(y,z)-\lambda)w-a(\nu k^2)^{\f13}w, (V(y,z)-\lambda)w\right\rangle=\langle F,(V(y,z)-\lambda)w\rangle,
\end{align*}
and then taking the imaginary part, we get
\begin{align*}
   &|k|\|(V-\lambda) w\|^2_{L^2}
   \leq\left|\mathbf{Im}\big( \langle \nu\Delta w, (V-\lambda)w\rangle\big)\right|+ \|F\|_{L^2}\|(V-\lambda)w\|_{L^2},
\end{align*}
where
\begin{align*}
 \left|\mathbf{Im}\big( \langle \nu\Delta w, (V-\lambda)w\rangle\big)\right|=&\left|\mathbf{Im}\big( -\langle \nu\nabla w, (V-\lambda)\nabla w\rangle-\langle \nu\nabla w, (\nabla V) w\rangle\big)\right|\\ \nonumber=&\left|\mathbf{Im}\big(-\langle \nu\nabla w, (\nabla V) w\rangle\big)\right|\leq\nu\|\nabla V\|_{L^\infty}\|\nabla w\|_{L^2}\|w\|_{L^2}.
\end{align*}
Then we infer that
\begin{align}\label{eq:w-weight}
   |k|\|(V-\lambda)w\|^2_{L^2}
   \leq& C\big(\|\nabla w\|_{L^2}\|w\|_{L^2}+|k|^{-1}\|F\|^2_{L^2}\big).
\end{align}
Here we used $\|\nabla V\|_{L^\infty}\leq C\|V-y\|_{H^3}+1\leq C$.

Using the fact that
\begin{align*}
 -2\mathbf{Re}\left\langle (V-\lambda)w, \partial_yw\right\rangle=-\left\langle (V-\lambda), \partial_y|w|^2\right\rangle=\left\langle \partial_yV, |w|^2\right\rangle\geq\|w\|_{L^2}^2/2,
\end{align*}
we infer that
\begin{align*}
\|w\|_{L^2}^2&\leq 4\|(V-\lambda)w\|_{L^2}\|\partial_y w\|_{L^2}\\
  & \leq C\Big( \big(\nu|k|^{-1}\|\nabla w\|_{L^2}\|w\|_{L^2}\big)^{\f12}+|k|^{-1}\|F\|_{L^2}\Big)\|\nabla w\|_{L^2},\end{align*}
which along with \eqref{eq:w-grad} implies that (for $a$ small enough)
\begin{align}\label{eq:w-L2}
   &(\nu k^2)^{\f13}\|w\|_{L^2}\leq C\|F\|_{L^2}.
\end{align}
By \eqref{eq:w-weight}, \eqref{eq:w-L2} and \eqref{eq:OS-Nav}, we get
\beno
\nu\|\Delta w\|_{L^2}\leq |k|\|(V-\lambda)w\|_{L^2}+(\nu k^2)^{\f13}\|w\|_{L^2}+\|F\|_{L^2}.
\eeno
This shows that
\begin{align}
   & \nu^{\f23}|k|^{\f13}\|\nabla w\|_{L^2}+(\nu k^2)^{\f13}\|w\|_{L^2} +\nu\|\Delta w\|_{L^2}+|k|\|(V-\lambda)w\|_{L^2}\leq C\|F\|_{L^2}.\label{eq:w-L2-f}
\end{align}

{\bf Step 2.} Case of $F\in H^1(\Om)$.\smallskip

If  $\nu k^2\geq1$, the proof is obvious due to
  \begin{align*}
     & \nu^{\f23}|k|^{\f13}\|\nabla w\|_{L^2} +(\nu k^2)^{\f13}\|w\|_{L^2} +\nu\|\Delta w\|_{L^2}\\
     &\leq C\|F\|_{L^2}= C|k|^{-1}\|kF\|_{L^2}\leq C(\nu k^2)^{\f16}|k|^{-1}\|kF\|_{L^2}\leq C\nu^{\f16}|k|^{-\f23}\|\nabla F\|_{L^2}.
  \end{align*}
  
Now we assume  $\nu k^2\leq1$ and $k>0$ in this step (the case of $k<0$ can be proved by taking conjugation).  Let $\chi_1(y,z)=1/(V(y,z)-\lambda-i\delta)$ with $ \delta=(\nu/|k|)^{\f13}$. Then we have
\beno
ik(V-\lambda)w-a(\nu k^2)^{\f13}w=ik(V-\lambda+ia\delta)w.
\eeno
Taking $L^2$ inner product with $\chi_1w$ to \eqref{eq:OS-Nav}, and integrating by parts, we obtain
  \begin{align*}
 & \nu\big\langle\nabla w,\nabla(\chi_1w)\rangle +ik\langle w,(V-\lambda-ia\delta)\chi_1 w\big\rangle  = \langle F,\chi_1 w\rangle.
   \end{align*}
  Taking the imaginary part, we get
  \begin{align}\label{eq:w-H1-1}
     & |k|\left|\mathbf{Re}\langle w,(V-\lambda-ia\delta)\chi_1w\rangle\right| \leq \nu\left|\langle \nabla w,\nabla(\chi_1 w)\rangle\right| +\left|\langle F,\chi_1 w\rangle\right|.
  \end{align}
Using the fact that
  \begin{align*}
  &(V-\lambda-ia\delta)\chi_1=1+i(1-a)\delta\chi_1=1+i(1-a)(V(y,z)-\lambda+i\delta)\delta|\chi_1|^2,\\
  &{\rm thus},\quad \mathbf{Re}\big((V-\lambda+ia\delta)\chi_1\big)=1-(1-a)\delta^2|\chi_1|^2\geq 1-\delta^2|\chi_1|^2,
  \end{align*}
 we infer that
   \begin{align*}
   \left|\mathbf{Re}\langle w,(V-\lambda-ia\delta)\chi_1w\rangle\right| \geq\| w\|_{L^2}^2-\delta^2\|\chi_1 w\|_{L^2}^2.
  \end{align*}
  By Lemma \ref{lem:sob-f}, $\delta^2\|\chi_1 w\|_{L^2}^2\leq C \delta\| w\|_{L^2}\|\nabla w\|_{L^2},$ hence,
  \begin{align}\label{eq:w-H1-2}
    & |\mathbf{Re}\langle w,(V-\lambda-ia\delta)\chi_1w\rangle\rangle|\geq \f34\|w\|^2_{L^2}-C\delta^2\|\nabla w\|^2_{L^2}.
  \end{align}
By Lemma \ref{lem:chi1}, we have
  \begin{align}\label{eq:w-H1-3}
    &\left|\langle\nabla w,\nabla(\chi_1 w)\rangle\right|\leq \|\nabla w\|_{L^2}\|\nabla (\chi_1w)\|_{L^2}\leq C\|\nabla w\|_{L^2}\big(\delta^{-2}\|w\|_{L^2}+\delta^{-1}\|\nabla w\|_{L^2}\big),\\
  \label{eq:w-H1-4}
     & |\langle F,\chi_1w\rangle|\leq \|\chi_1 F\|_{L^2} \|w\|_{L^2}.
  \end{align}
  Summing up \eqref{eq:w-H1-1}-\eqref{eq:w-H1-4}, we obtain
  \begin{align*}
  &|k|\Big(\f34\|w\|_{L^2}^2-C\delta^2\|\nabla w\|^2_{L^2}\Big)\\
     &\leq C\nu\|\nabla w\|_{L^2}(\delta^{-2}\|w\|_{L^2}+\delta^{-1}\|\nabla w\|_{L^2})+C\|\chi_1 F\|_{L^2}\|w\|_{L^2},
  \end{align*}
which along with Lemma \ref{lem:sob-f} and by recalling $\nu|k|^{-1}=\delta^3$ gives
  \begin{align*}
     \|w\|_{L^2}^2\leq& C\Big( \delta^2\|\nabla w\|_{L^2}^2 +\nu|k|^{-1}\|\nabla w\|_{L^2}\big(\delta^{-2}\|w\|_{L^2}+\delta^{-1}\|\nabla w\|_{L^2}\big)+|k|^{-1}\|\chi_1 F\|_{L^2} \|w\|_{L^2} \Big)\\
     \leq& C\big( \delta\|\nabla w\|_{L^2}\|w\|_{L^2} +\delta^2\|\nabla w\|^2_{L^2}+\delta^{-\f12}|k|^{-\f32}\|\nabla F\|_{L^2} \|w\|_{L^2}\big).
  \end{align*}
 Then  Young's inequality  gives
  \begin{align}\label{eq:w-H1-44}
     & \|w\|_{L^2}\leq C\big(\delta\|\nabla w\|_{L^2}+\delta^{-\f12}|k|^{-\f32}\|\nabla F\|_{L^2}\big).
  \end{align}

  Next we estimate $\|\na w\|_{L^2}$ and $\|\Delta w\|_{L^2}$.   First of all, we have
   \begin{align}\label{eq:w-H1-5}
     &\nu\|\nabla w\|_{L^2}^2\leq a(\nu k^2)^{\f13}\|w\|^2_{L^2}+|\mathbf{Re}\langle F,w\rangle|.
  \end{align}
Taking the imaginary part to
\begin{align*}
   & \big\langle -\nu\Delta w+ik(V-\lambda)w-a(\nu k^2)^{\f13}w, \chi_1(y,z) F\big\rangle=\langle F,\chi_1
      F\rangle,
\end{align*}
we get
\begin{align*}
   &|k|\left|\mathbf{Re}\langle w,(V-\lambda-ia\delta)\chi_1 F\rangle\right|
   \leq\nu\|\Delta w\|_{L^2} \|\chi_1 F\|_{L^2}+ \left|\mathbf{Im}\big( \langle F, \chi_1 F\rangle\big)\right|.
\end{align*}
Using the facts that $(V-\lambda-ia\delta)\chi_1=1+i(1-a)\delta\chi_1,\ \mathbf{Im}\chi_1=\delta|\chi_1|^2,$
we have
  \begin{align*}
  &\left|\mathbf{Re}\langle w,(V-\lambda-ia\delta)\chi_1 F\rangle\right| \geq|\mathbf{Re}\langle F,w\rangle|-\delta\|w\|_{L^2}\|\chi_1 F\|_{L^2},\\&\left|\mathbf{Im}\langle F, \chi_1 F\rangle\right|=\delta\|\chi_1 F\|^2_{L^2}\leq C|k|^{-1}\|\nabla F\|^2_{L^2}.
  \end{align*}
 This shows that
  \begin{align*}
  &|\mathbf{Re}\langle F,w\rangle|\leq\delta\|w\|_{L^2}\|\chi_1 F\|_{L^2}+|k|^{-1}\big(\nu\|\Delta w\|_{L^2} \|\chi_1 F\|_{L^2}+\delta\|\chi_1 F\|^2_{L^2}\big),
  \end{align*}
 which along with \eqref{eq:w-H1-5} and $ \delta^3=\nu/|k|$ implies that
 \begin{align}\label{eq:w-H1-6}
 &\nu\|\nabla w\|_{L^2}^2\leq a(\nu k^2)^{\f13}\|w\|^2_{L^2}+\big(\delta\|w\|_{L^2}+\delta^3\|\Delta w\|_{L^2}\big)\|\chi_1 F\|_{L^2}+|k|^{-1}\delta\|\chi_1 F\|^2_{L^2}.
  \end{align}

Taking $L^2$ inner product with $\Delta w$ to \eqref{eq:OS-Nav}, we have
  \begin{align*}
 \nu\|\Delta w\|^2_{L^2}- k\mathbf{Im}\left(\langle \nabla((V-\lambda)w),\nabla w \rangle\right)-a(\nu k^2)^{\f13}\|\nabla w\|^2_{L^2}=\mathbf{Re}\big(\langle F,-\Delta w\rangle\big).
  \end{align*}
 Due to $\|\nabla V\|_{L^\infty}\leq C$, we have
  \begin{align*}
     &\nu\|\Delta w\|^2_{L^2}-a(\nu k^2)^{\f13}\|\nabla w\|^2_{L^2}-|\langle F,\Delta w\rangle|\\ &\quad\leq |k|\|\nabla V\|_{L^\infty}\|w\|_{L^2}\|\nabla w\|_{L^2}\leq C|k|\|w\|_{L^2}\|\nabla w\|_{L^2}.
  \end{align*}
 Due to $w|_{y=\pm 1}=0,$ we get by integration by parts and Lemma \ref{lem:sob-f} that
  \begin{align*}
  |\langle F,\Delta w\rangle|\leq& \|\nabla F\|_{L^2}\|\nabla w\|_{L^2}+\| F\|_{L^2(\partial\Omega)}\|\partial_y w\|_{L^2(\partial\Omega)}\\
     \leq& \|\nabla F\|_{L^2}\|\nabla w\|_{L^2}+C|k|^{-\f12}\|\nabla F\|_{L^2}\|\partial_y w\|_{L^2}^{\f12}\|\nabla\partial_y w\|_{L^2}^{\f12}\\
     \leq& \|\nabla F\|_{L^2}\|\nabla w\|_{L^2}+C|k|^{-\f12}\|\nabla F\|_{L^2}\|\nabla w\|_{L^2}^{\f12}\|\Delta w\|_{L^2}^{\f12}.\end{align*}
 Summing up, we arrive at
  \begin{align*}
      \|\Delta w\|^2_{L^2}\leq& a\nu^{-\f23} |k|^{\f23}\|\nabla w\|^2_{L^2}+C\nu^{-1}|k|\|w\|_{L^2}\|\nabla w\|_{L^2}\\&+\nu^{-1}\|\nabla F\|_{L^2}\|\nabla w\|_{L^2}+C\nu^{-1}|k|^{-\f12}\|\nabla F\|_{L^2}\|\nabla w\|_{L^2}^{\f12}\|\Delta w\|_{L^2}^{\f12}.
  \end{align*}
Then  Young's inequality gives
\begin{align*}
      \|\Delta w\|^2_{L^2}\leq& C\nu^{-\f23} |k|^{\f23}\|\nabla w\|^2_{L^2}+C\nu^{-1}|k|\|w\|_{L^2}\|\nabla w\|_{L^2}\\&+C\nu^{-1}\|\nabla F\|_{L^2}\|\nabla w\|_{L^2}+C\nu^{-\f43}|k|^{-\f23}\|\nabla F\|_{L^2}^{\f43}\|\nabla w\|_{L^2}^{\f23}\\ \leq& C\nu^{-\f23} |k|^{\f23}\|\nabla w\|^2_{L^2}+C\nu^{-\f43}|k|^{\f43}\|w\|_{L^2}^2+C\big(\nu^{-\f53}|k|^{-\f43}+\nu^{-\f43}|k|^{-\f23}\big)\|\nabla F\|_{L^2}^{2}.
 \end{align*}
Since $ \delta=(\nu/|k|)^{\f13},\ \nu k^2\leq 1,\ \nu^{-\f43}|k|^{-\f23}\leq\nu^{-\f53}|k|^{-\f43}=\delta^{-5}|k|^{-3}$, we have
\begin{align}\label{eq:w-H1-7}
\|\Delta w\|_{L^2}\leq C\delta^{-1} \|\nabla w\|_{L^2}+C\delta^{-2}\|w\|_{L^2}+C\delta^{-\f52}|k|^{-\f32}\|\nabla F\|_{L^2}.
  \end{align}
Plugging \eqref{eq:w-H1-7} into \eqref{eq:w-H1-6}, we get  by \eqref{eq:w-H1-44} that
 \begin{align*}\nu\|\nabla w\|_{L^2}^2\leq& a(\nu k^2)^{\f13}\|w\|^2_{L^2}+C\big(\delta\|w\|_{L^2}+\delta^2\|\nabla w\|_{L^2}+\delta^{\f12}|k|^{-\f32}\|\nabla F\|_{L^2}\big)\|\chi_1 F\|_{L^2}\\&+|k|^{-1}\delta\|\chi_1 F\|^2_{L^2}\\ \leq& a(\nu k^2)^{\f13}\|w\|^2_{L^2}+C\big(\delta^2\|\nabla w\|_{L^2}+\delta^{\f12}|k|^{-\f32}\|\nabla F\|_{L^2}\big)\|\chi_1 F\|_{L^2}+|k|^{-1}\delta\|\chi_1 F\|^2_{L^2},
  \end{align*}
 from which, Young's inequality, \eqref{eq:w-H1-44} and Lemma \ref{lem:sob-f}, we infer that
 \begin{align*}\nu\|\nabla w\|_{L^2}^2\leq& C\big(a(\nu k^2)^{\f13}\|w\|^2_{L^2}+\delta^{\f12}|k|^{-\f32}\|\nabla F\|_{L^2}\|\chi_1 F\|_{L^2}+|k|^{-1}\delta\|\chi_1 F\|^2_{L^2}\big)\\ \leq& Ca\big(\nu k^2)^{\f13}(\delta\|\nabla w\|_{L^2}+\delta^{-\f12}|k|^{-\f32}\|\nabla F\|_{L^2}\big)^2+C|k|^{-2}\|\nabla F\|_{L^2}^2\\ \leq& Ca\big(\nu\|\nabla w\|_{L^2}^2+|k|^{-2}\|\nabla F\|_{L^2}^2\big)+C|k|^{-2}\|\nabla F\|_{L^2}^2.
  \end{align*}
 Taking $\eps_1$ small  enough so that $Ca\leq C\eps_1\leq \f12$, we conclude
 \beno
 \nu\|\nabla w\|_{L^2}^2\leq C|k|^{-2}\|\nabla F\|_{L^2}^2,
 \eeno
 which along with \eqref{eq:w-H1-44} and  \eqref{eq:w-H1-7} gives
   \begin{align*}
     & \nu\|\Delta w\|_{L^2} + \nu^{\f23}|k|^{\f13}\|\nabla w\|_{L^2}+(\nu k^2)^{\f13}\|w\|_{L^2}\leq C\nu^{\f16}|k|^{-\f23}\|\nabla F\|_{L^2}.
  \end{align*}

{\bf Step 3.} Case of $F\in H^{-1}(\Om)$.\smallskip

If  $\nu k^2\geq1$, taking $L^2$ inner product with $w$ to \eqref{eq:OS-Nav}, and integrating by parts, we obtain
\begin{align*}
   & \nu\|\nabla w\|^2_{L^2} \leq \|F\|_{H^{-1}}\|w\|_{H^1}+a(\nu k^2)^{\f13}\|w\|^2_{L^2}\leq C\|F\|_{H^{-1}}\|\nabla w\|_{L^2}+a(\nu k^2)\|w\|^2_{L^2}.
\end{align*}
As $a(\nu k^2)\|w\|^2_{L^2}\leq a\nu\|\nabla w\|^2_{L^2}\leq (\nu/2)\|\nabla w\|^2_{L^2}$, we deduce that $\nu\|\nabla w\| \leq C\|F\|_{H^{-1}} $ and\begin{align*}
    \nu^{\f23}|k|^{\f13}\|\nabla w\|_{L^2} +(\nu k^2)^{\f13}\|w\|_{L^2}&\leq (\nu^{\f23}|k|^{\f13}+(\nu k^2)^{\f13}|k|^{-1})\|\nabla w\|_{L^2}\\&\leq 2\nu^{\f23}|k|^{\f13}\|\nabla w\|_{L^2}
     \leq C\nu^{-\f13}|k|^{\f13}\|F\|_{H^{-1}}.
\end{align*}
Now we assume  $\nu k^2\leq1$ and $k>0$ in this step. Then $\delta\leq1$. By Lemma \ref{lem:chi1}, we have
  \begin{align*}
      |\langle F,\chi_1w\rangle|\leq& \|F\|_{H^{-1}}\big(\|\nabla(\chi_1 w)\|_{L^2}+\|\chi_1 w\|_{L^2}\big)\\
      \leq& C\|F\|_{H^{-1}}\big(\|\chi_1\|_{L^\infty}\|\nabla w\|_{L^2}+\|\nabla\chi_1\|_{L^\infty}\|w\|_{L^2}+\|\chi_1\|_{L^\infty}\|w\|_{L^2}\big)\\
      \leq& C\delta^{-2}\|F\|_{H^{-1}}\big(\delta\|\nabla w\|_{L^2}+\|w\|_{L^2}+\delta\|w\|_{L^2}\big).
  \end{align*}
  As $\delta\leq1$, this shows that
  \begin{align}\label{eq:w-H-1-1}
  |\langle F,\chi_1w\rangle|\leq C\|F\|_{H^{-1}}\big(\delta^{-1}\|\nabla w\|_{L^2}+\delta^{-2}\|w\|_{L^2}\big).
  \end{align}
  Summing up \eqref{eq:w-H1-1}-\eqref{eq:w-H1-3} and \eqref{eq:w-H-1-1}, we get
  \begin{align*}
 & |k|\Big(\f34\|w\|_{L^2}^2-C\delta^2\|\nabla w\|^2_{L^2}\Big)\\
 & \leq \nu\|\nabla w\|_{L^2}\big(\delta^{-2}\|w\|_{L^2}+\delta^{-1}\|\nabla w\|_{L^2}\big) +\delta^{-2}\|F\|_{H^{-1}}\big(\delta\|\nabla w\|_{L^2}+\|w\|_{L^2}\big),
  \end{align*}
which gives
  \begin{align*}
     \|w\|_{L^2}^2\leq& C\Big(\delta^2\|\nabla w\|_{L^2}^2 +\nu|k|^{-1}\|\nabla w\|_{L^2}\big(\delta^{-2}\|w\|_{L^2}+\delta^{-1}\|\nabla w\|_{L^2}\big)\\
     &+\delta^{-2}|k|^{-1}\| F\|_{H^{-1}} \big(\delta\|\nabla w\|_{L^2}+\|w\|_{L^2}\big) \Big)\\
     \leq& C\Big( \delta\|\nabla w\|_{L^2}\|w\|_{L^2} +\delta^2\|\nabla w\|^2_{L^2}+\delta^{-2}|k|^{-1}\| F\|_{H^{-1}} \big(\delta\|\nabla w\|_{L^2}+\|w\|_{L^2}\big)\Big).
  \end{align*}
Then  Young's inequality gives
  \begin{align}\label{eq:w-H-1-2}
   \|w\|_{L^2}\leq C\big(\delta\|\nabla w\|_{L^2}+\delta^{-2}|k|^{-1}\|F\|_{H^{-1}}\big).
  \end{align}
  On the other hand, we have
  \begin{align*}
     \nu\|\nabla w\|_{L^2}^2\leq& a(\nu k^2)^{\f13}\|w\|^2_{L^2}+|\langle F,w\rangle|\\
     \leq& a(\nu k^2)^{\f13}\|w\|_{L^2}^{2}+\| F\|_{H^{-1}}\big(\|\nabla w\|_{L^2}+\|w\|_{L^2}\big)\\
     \leq &a(\nu k^2)^{\f13}\|w\|_{L^2}^{2}+\| F\|_{H^{-1}}\big(\|\nabla w\|_{L^2}+\delta^{-1}\|w\|_{L^2}\big).
  \end{align*}
from which and \eqref{eq:w-H-1-2},  and by taking  $\eps_1$ small enough so that $Ca^{\f12}\leq C\eps_1^{\f12}\leq \f12$, we infer that
  \begin{align}\label{9}
     & \nu^{\f23}|k|^{\f13}\|\nabla w\|_{L^2}+(\nu k^2)^{\f13}\|w\|_{L^2}\leq C\nu^{-\f13}|k|^{\f13}\| F\|_{H^{-1}}.
  \end{align}

{\bf Step 4.} Estimates of $\|w\|_{L^2_{x,z}L^1_y}$ and $\|\na\varphi\|_{L^2}$.\smallskip

By Lemma \ref{lem:chi1} and  \eqref{eq:w-L2-f},  we have
\begin{align*}
   \|w\|_{L^2_{x,z}L^1_y}=&\|\chi_1(V-\lambda-i\delta)w\|_{L^2_{x,z}L^1_y}\\\leq& \|\chi_1\|_{L^\infty_{x,z}L^2_{y}} \|(V-\lambda-i\delta)w\|_{L^2}\\
   \leq& C\delta^{-\f12}\big(\delta\|w\|_{L^2}+\|(V-\lambda)w\|_{L^2}\big)\\
   \leq &C\delta^{\f12}|\nu k^2|^{-\f13}\|F\|_{L^2}+C\delta^{-\f12}|k|^{-1}\|F\|_{L^2}\leq C\nu^{-\f16}|k|^{-\f56}\|F\|_{L^2},
\end{align*}
from which and Lemma \ref{lem:elliptic-weight}, we infer that
\begin{align}
  \nu^{\f16}|k|^{\f43}\|\nabla\varphi\|_{L^2}+\nu^{\f16}|k|^{\f56}\|w\|_{L^2_{x,z}L^1_y} \leq C\nu^{\f16}|k|^{\f56}\|w\|_{L^2_{x,z}L^1_y}\leq C\|F\|_{L^2}.\label{eq:w-L1}
\end{align}

{\bf Step 5.} Estimate of $\|\na^k((V-\lambda)\varphi)\|, k=1,2$.\smallskip

Let $F_1=(V-\lambda)\varphi$. Then we have
\beno
\Delta F_1=(V-\lambda)w+ 2\nabla V\cdot\nabla\varphi+ \varphi\Delta V, \quad F_1|_{y=\pm1}=0,\quad \partial_xF_1=ikF_1.
\eeno
Thus, we get by \eqref{eq:w-L2-f} and \eqref{eq:w-L1} that
  \begin{align*}
  \|\Delta F_1\|_{L^2}\leq&\|(V-\lambda)w\|_{L^2}+2\|\nabla V\cdot\nabla\varphi\|_{L^2}+\|\varphi\Delta V\|_{L^2}\\ \leq&C\big(|k|^{-1}\|F\|_{L^2}+\|\nabla\varphi\|_{L^2}+\|\varphi\|_{L^2}\big)
  \leq C\nu^{-\f16}|k|^{-\f43}\|F\|_{L^2}.
  \end{align*}
As $F_1|_{y=\pm1}=0,\ \partial_xF_1=ikF_1$, by Lemma \ref{lem:sob-f},  we have
\begin{align*}
 &\|\nabla^2 F_1\|_{L^2}\leq C\|\Delta F_1\|_{L^2}\leq C\nu^{-\f16}|k|^{-\f43}\|F\|_{L^2},\\&\|\nabla F_1\|_{L^2_{x,z}L^\infty_y}\leq C|k|^{-\f12}\|\nabla^2 F_1\|_{L^2}\leq C\nu^{-\f16}|k|^{-\f{11}{6}}\|F\|_{L^2}.
\end{align*}
\begin{align*}
   &\nu^{\f16}|k|^{\f43}\|\nabla^2[(V-\lambda)\varphi]\|_{L^2} +\nu^{\f16}|k|^{\f{11}{6}}\|\nabla[(V-\lambda)\varphi]\|_{L^2_{x,z}L^\infty_y}\leq C\|F\|_{L^2}.
\end{align*}
This completes the proof of the proposition.
\end{proof}\smallskip

The following proposition is a simple corollary of Proposition \ref{prop:res-nav-s1}.

\begin{Proposition}\label{prop:res-nav-s2}
Let $w\in H^2(\Omega)$ be a solution of \eqref{eq:OS-Nav} with $F\in L^2(\Omega)$. If $F=F_1+F_2+\partial_x f_1+\partial_y f_2+\partial_z f_3$ and $F_2|_{y=\pm1}=0$, then it holds that
\begin{align*}
&\nu^{\f23}\|\nabla w\|_{L^2} +(\nu |k|)^{\f13}\|w\|_{L^2} \leq C\Big(|k|^{-\f13}\|F_1\|_{L^2}+\nu^{\f16}|k|^{-1}\|\nabla F_2\|_{L^2}+\nu^{-\f13}\|(f_1,f_2,f_3)\|_{L^2}\Big).
\end{align*}
\end{Proposition}

\begin{proof}
We decompose $w$ as $w=w_1+w_2+w_3$, where $w_j(j=1,2,3)$ solves
\begin{align*}\left\{\begin{aligned}
                & -\nu\Delta w_1+ik(V(y,z)-\lambda)w_1-a(\nu k^2)^{\f13}w_1=F_1,\\&-\nu\Delta w_2+ik(V(y,z)-\lambda)w_2-a(\nu k^2)^{\f13}w_2=F_2,\\&-\nu\Delta w_3+ik(V(y,z)-\lambda)w_3-a(\nu k^2)^{\f13}w_3=\partial_xf_1+\partial_yf_2+\partial_zf_3,\\&
                w_j|_{y=\pm1}=0,\quad \partial_xw_j=ikw_j,\quad j=1,2,3.
              \end{aligned}\right.
\end{align*}
It follows from Proposition \ref{prop:res-nav-s1}  that
\begin{align*}
&\nu^{\f23}|k|^{\f13}\|\nabla w_1\|_{L^2} +(\nu k^2)^{\f13}\|w_1\|_{L^2} \leq C\|F_1\|_{L^2},\\&\nu^{\f23}|k|^{\f13}\|\nabla w_2\|_{L^2} +(\nu k^2)^{\f13}\|w_2\|_{L^2} \leq C\nu^{\f16}|k|^{-\f23}\|\nabla F_2\|_{L^2},\\
&\nu^{\f23}|k|^{\f13}\|\nabla w_3\|_{L^2} +(\nu k^2)^{\f13}\|w_3\|_{L^2}\leq C\nu^{-\f13}|k|^{\f13}\|(f_1,f_2,f_3)\|_{L^2}.
\end{align*}
Then we have
\begin{align}\label{12}
 \nu^{\f23}\|\nabla w\|_{L^2}+(\nu|k|)^{\f13}\|w\|_{L^2}\leq& |k|^{-\f13} \sum_{j=1}^3\Big(\nu^{\f23}|k|^{\f13}\|\nabla w_j\|_{L^2} +(\nu k^2)^{\f13}\|w_j\|_{L^2}\Big)\nonumber\\
   \leq& C\Big(|k|^{-\f13}\|F_1\|_{L^2} +\nu^{\f16}|k|^{-1}\|\nabla F_2\|_{L^2}+\nu^{-\f13}\|(f_1,f_2,f_3)\|_{L^2}\Big).
\end{align}
\end{proof}

The following proposition gives the estimates of $w$ when taking good derivatives $\pa_x$ and $\pa_z-\kappa\pa_y$.

\begin{Proposition}\label{prop:res-nav-good}
Let $w\in H^2(\Omega)$ be a solution of \eqref{eq:OS-Nav} with  $F=f_1+f_2+f_3$ and $P_0f_i=0(i=1,2,3)$.  Then it holds that
\begin{align*}
&(\nu |k|)^{\f13}\Big(\|\partial_x^2w\|_{L^2}+\|\partial_x(\partial_z-\kappa\partial_y)w\|_{L^2}\Big)+
\nu^{\f23}\Big(\|\nabla\partial_x^2w\|_{L^2}+\|\nabla\partial_x(\partial_z-\kappa\partial_y)w\|_{L^2}\Big) \\
&\quad\leq C\Big(\nu^{\f16}\| f_1\|_{H^2}+|k|^{-\f13}\|\partial_x^2 f_2\|_{L^2}+|k|^{-\f13}\|\partial_x(\partial_z-\kappa\partial_y) f_2\|_{L^2}+\nu^{-\f13}\| \partial_xf_3\|_{L^2}\Big).
\end{align*}
\end{Proposition}

\begin{proof}
 Thanks to $\partial_x\big[(V(y,z)-\lambda)w\big]=(V(y,z)-\lambda)\partial_xw$, $\partial_xf_1=ik f_1$, we have
  \begin{align*}
     & -\nu\Delta \partial_x^2w+ik(V(y,z)-\lambda)\partial_x^2w-a(\nu k^2)^{1/3}\partial_x^2w=ik\partial_xf_1+\partial_x^2f_2+\partial_x^2f_3.
  \end{align*}
Then it follows from Proposition \ref{prop:res-nav-s2} that
  \begin{align}
     &(\nu |k|)^{\f13}\|\partial_x^2w\|_{L^2}+\nu^{\f23}\|\nabla\partial_x^2w\|_{L^2}\nonumber\\
     &\leq C\big(\nu^{\f16}|k|^{-1}\|ik\nabla\partial_xf_1\|_{L^2}+ |k|^{-\f13}\|\partial_x^2 f_2\|_{L^2}+\nu^{-\f13}\|\partial_xf_3\|_{L^2}\big)\nonumber\\
     &\leq C\big( \nu^{\f16}\| f_1\|_{H^2}+|k|^{-\f13}\|\partial_x^2f_2\|_{L^2}+\nu^{-\f13}\|\partial_xf_3\|_{L^2}\big).\label{eq:w-x-good}
  \end{align}

Thanks to $(\partial_z-\kappa\partial_x)V(y,z)=0$, we have
  \begin{align*}
     & -\nu\Delta (\partial_z-\kappa\partial_y)w+ik(V(y,z)-\lambda)(\partial_z-\kappa\partial_y)w -a(\nu k^2)^{1/3}(\partial_z-\kappa\partial_y)w\\&=-\nu\Delta \kappa\partial_yw+2\nu(\partial_y(\partial_y\kappa\partial_yw)+\partial_z(\partial_z\kappa\partial_yw))
     +(\partial_z-\kappa\partial_y)F,
  \end{align*}
and we write
  \begin{align*}
     & (\partial_z-\kappa\partial_y)F=(\partial_zf_1-\kappa \partial_yf_1)+\partial_zf_3-\partial_y(\kappa f_3)+\partial_y\kappa f_3+(\partial_z-\kappa\partial_y)f_2.
  \end{align*}
  Thus, we have
  \begin{align*}
     &-\nu\Delta \partial_x (\partial_z-\kappa\partial_y)w+ik(V(y,z)-\lambda)(\partial_z-\kappa\partial_y)\partial_xw -a(\nu k^2)^{1/3}(\partial_z-\kappa\partial_y)\partial_xw\\
     &=-\nu\Delta \kappa\partial_x\partial_yw+2\nu(\partial_y(\partial_y\kappa\partial_y\partial_xw)+\partial_z(\partial_z\kappa\partial_y\partial_xw))
     \\&\quad +ik(\partial_zf_1-\kappa \partial_yf_1)+\partial_z\partial_xf_3+\partial_x(f_3\partial_y\kappa) -\partial_y\partial_x(\kappa f_3)+\partial_x(\partial_z-\kappa\partial_y)f_2.
  \end{align*}
By Proposition \ref{prop:res-nav-s2}  again, we get
  \begin{align*}
     & (\nu |k|)^{\f13}\|\partial_x(\partial_z-\kappa\partial_y)w\|_{L^2} +\nu^{\f23}\|\nabla\partial_x(\partial_z-\kappa\partial_y)w\|_{L^2}\\
     &\leq C\Big(\nu^{\f16}\|\nabla(\partial_zf_1-\kappa \partial_yf_1)\|_{L^2} +\nu^{-\f13}(\|\partial_xf_3\|_{L^2}+\|\partial_x(\kappa f_3)\|_{L^2}+\|f_3\partial_y\kappa\|_{L^2})\\
     &\quad+|k|^{-\f13}\|\partial_x(\partial_z-\kappa\partial_y)f_2\|_{L^2}+
     \nu^{\f23}(\|(\partial_y\kappa,\partial_z\kappa)\partial_y\partial_xw\|_{L^2}
     +\|\nabla h_1\|_{L^2})\Big),
  \end{align*}
 where $ \Delta h_1=\Delta \kappa\partial_x\partial_yw,\ h_1|_{y=\pm 1}=0.$
 Note that
\beno
\|\Delta \kappa\|_{L^{4}}\leq C\|\Delta \kappa\|_{H^{1}}\leq C\|\kappa\|_{H^3}\leq C,
\eeno
which gives
\begin{align*}
\|\nabla h_1\|_{L^2}^2=&-\langle\Delta h_1,h_1\rangle=-\langle\Delta \kappa\partial_x\partial_yf,h_1\rangle\\ \leq& \|\Delta \kappa\|_{L^{4}}\|\partial_x\partial_yw\|_{L^{2}}\|h_1\|_{L^{4}}\leq C\|\partial_x\nabla w\|_{L^{2}}\|h_1\|_{H^{1}}\leq C\|\nabla\partial_x^2 w\|_{L^{2}}\|\nabla h_1\|_{L^2}.
\end{align*}
Thus, $\|\nabla h_1\|_{L^2}\leq C\|\nabla\partial_x^2 w\|_{L^{2}}$.
This along with the facts that $\|\nabla\kappa\|_{L^\infty}+\|\nabla\kappa\|_{H^2}\leq C\|V-y\|_{H^4}\leq C\varepsilon_0$, shows that
\begin{align*}
   & (\nu |k|)^{\f13}\|\partial_x(\partial_z-\kappa\partial_y)w\|_{L^2} +\nu^{\f23}\|\nabla\partial_x(\partial_z-\kappa\partial_y)w\|_{L^2}\\
     &\leq C\Big(\nu^{\f16}\| f_1\|_{H^2} +\nu^{-\f13}\|\partial_xf_3\|_{L^2} +|k|^{-\f13}\|\partial_x(\partial_z-\kappa\partial_y)f_2\|_{L^2}+\nu^{\f23}\|\nabla\partial_x^2w\|_{L^2}\Big),
\end{align*}
from which and \eqref{eq:w-x-good}, we deduce our result.
\end{proof}

\subsection{Weak type resolvent estimates}

In this subsection, we always assume $\nu k^2\le 1$, so that $|k\delta|\le 1$($\delta=\nu^\f13|k|^{-\f13}$). We establish various weak type resolvent estimates, which will play an important role for the estimate of the Neumann data $\pa_y\varphi|_{y=\pm 1}$. Without lose of generality, we may assume $k>0$.

\begin{Proposition}\label{prop:res-weak}
Let $w\in H^2(\Omega)$ be a solution of \eqref{eq:OS-Nav}  with $F\in L^2(\Omega)$. If $f\in H^1(\Omega),\ j\in\{\pm1\},\ f|_{y=-j}=0$, then it holds that
\begin{align*}
    |\langle w(V-\widetilde{\lambda}),f\rangle| \le C|k|^{-1}\|F\|_{H^{-1}}\Big( \|f\|_{H^1}+\big(\|f\|_{L^2(\Gamma_j)}+\delta_1\|(\partial_x,\partial_z) f\|_{L^2(\Gamma_j)}\big)(|j-\lambda|+\delta)^{\f14} \delta^{-\f34} \Big),
 \end{align*}
 and
 \begin{align*}
  \|\langle w,f\rangle|\leq& C|k|^{-1}\|{F}\|_{H^{-1}}\Big(\big(\|f\|_{L^2(\Gamma_j)}+\delta_1\|(\partial_x,\partial_z) f\|_{L^2(\Gamma_j)}\big) (|j-\lambda|+\delta)^{-\f34}\delta^{-\f34}\\
&\quad
+\|\nabla f/(|V-\lambda|+\delta)\|_{L^2}+\delta^{-1}\|f/(|V-\lambda|+\delta)\|_{L^2}\Big),\\ |\langle w,f\rangle|\leq & C|k|^{-1}\|{F}\|_{H^{-1}}\Big(\delta^{-\f32}\|f\|_{L_{x,z}^2L_y^{\infty}} +\delta^{-\f12}\|(\partial_x,\partial_z)f\|_{L^2(\Gamma_j)}+\delta^{-1} \|\nabla f\|_{L^2}\Big)\\ \leq& C|k\delta|^{-\f32}\|{F}\|_{H^{-1}}\|\nabla f\|_{L^2} +C|k|^{-1}\delta^{-\f12}\|F\|_{H^{-1}}\|(\partial_x,\partial_z)f\|_{L^2(\Gamma_j)},\\|\langle w,f\rangle|\leq& C\nu^{-1}\|{F}\|_{H^{-1}}\|(1-y^2) f\|_{L^2},
  \end{align*}
where
\beno
\widetilde{\lambda}=\lambda-ia\delta,\quad \delta_1=\delta\big(|k(1-\lambda)|+1\big)^{-\f12}.
\eeno
\end{Proposition}

\begin{proof}
  First of all, for any $f_0\in H^1_0(\Omega)$ with $f_0|_{y=\pm1}=0$,  we get by integration by parts that
\begin{align*}
  \|F\|_{H^{-1}}\|f_0\|_{H^1} \geq& |\langle F,f_0\rangle|= |\langle -\nu \Delta w+ik(V-\tilde{\lambda})w, f_0\rangle|\\
  \geq& |\langle k(V-\tilde{\lambda})w,f_0 \rangle|-\nu\|\nabla w\|_{L^2}\|\nabla f_0\|_{L^2},
\end{align*}
from which and Proposition \ref{prop:res-nav-s1},  we infer that
\begin{align}
    \big|\big\langle w(V-\tilde{\lambda}),f_0\big\rangle\big|\leq  |k|^{-1}\|F\|_{H^{-1}}\|f_0\|_{H^1}.\label{eq:w-weak2}
\end{align}

Next we consider the case when $f\in H^1(\Omega),\ f(x,-1,z)=0$. In this case, for every $\delta_{*}\in(0,\delta]\subseteq [-1,1]$, let
\beno
\chi_2(y)=\max(1-(1-y)/\delta_{*},0),\quad f_0(x,y,z)=f(x,y,z)-f(x,1,z)\chi_2(y).
\eeno
Then we have $\chi_2\in H^1(\Omega)$, $f_0\in H^1_0(\Omega),\ I_2:={\rm supp}\chi_2=\mathbb{T}\times[1-\delta_*,1]\times \mathbb{T}$ and
\begin{align*}
   &\|\chi_{2}\|_{L^\infty}=1,\quad \|\chi_{2}\|_{L^\infty_{x,z} L_y^2}\leq\delta_{*}^{\f12},\quad\|\nabla\chi_2\|_{L^\infty_{x,z}L_y^2}\leq \delta_{*}^{-\f12},\\
   &\|(V-\tilde{\lambda})\chi_{2}\|_{L^\infty}\leq \|V-\tilde{\lambda}\|_{L^\infty([1-\delta_{*},1])}\|\chi_{2}\|_{L^\infty} \leq |1-\lambda|+\delta.
\end{align*}
Due to $ w(x,1,z)=0$, we have
\beno
|w(x,y,z)|=\left|\int_y^1\partial_yw(x,y_1,z)dy_1\right|\leq |1-y|^{\f12}\|\partial_yw(x,\cdot,z)\|_{L^2_y}\leq \delta_{*}^{\f12}\|\partial_yw(x,\cdot,z)\|_{L^2_y}
\eeno
for $y\in[1-\delta_{*},1]$ and then
\beno
 \|w\|_{L^2_{x,z}L^{1}_y(I_2)}\leq \delta_{*}^{\f32}\|\nabla w\|_{L^2}.
 \eeno
Then it follows from Proposition \ref{prop:res-nav-s1} and \eqref{eq:w-weak2} that
\begin{align*}
&|\langle w(V-\tilde{\lambda}),f\rangle|\\
&\leq |\langle w(V-\tilde{\lambda}),f(x,1,z)\chi_{2}\rangle|+|\langle w(V-\tilde{\lambda}),f_0\rangle|\\
&\leq  \|f\|_{L^2(\Gamma_1)}\|w\|_{L^2_{x,z}L^{1}_y(I_2)}\|(V-\tilde{\lambda})\chi_2\|_{L^{\infty}} +C|k|^{-1}\|{F}\|_{H^{-1}}\|f_0\|_{H^1}\\ &\leq  \|f\|_{L^2(\Gamma_1)}\delta_{*}^{\f32}\|\nabla w\|_{L^2}(|1-\lambda|+\delta)
+C|k|^{-1}\|{F}\|_{H^{-1}}\|f(x,1,z)\chi_2\|_{H^1} +C|k|^{-1}\|{F}\|_{H^{-1}}\|f\|_{H^1}\\
&\leq C\|f\|_{L^2(\Gamma_1)}\delta_{*}^{\f32}\nu^{-1}\|{F}\|_{H^{-1}}(|1-\lambda|+\delta)
+C|k|^{-1}\|{F}\|_{H^{-1}}\Big(\|f(x,1,z)\|_{L^2_{x,z}L^\infty_y}\|\nabla \chi_2\|_{L^\infty_{x,z}L^2_y} \\&\qquad+\big(\|\nabla f(x,1,z)\|_{L^2_{x,z}L^\infty_y}+\| f(x,1,z)\|_{L^2_{x,z}L^\infty_y}\big)\|\chi_2\|_{L^\infty_{x,z}L^2_y}\Big) +C|k|^{-1}\|{F}\|_{H^{-1}}\|f\|_{H^1}\\ &\leq C\|f\|_{L^2(\Gamma_1)}\delta_{*}^{\f32}\nu^{-1}\|{F}\|_{H^{-1}}(|1-\lambda|+\delta)
+C|k|^{-1}\|{F}\|_{H^{-1}}\Big(\|f\|_{L^2(\Gamma_1)}(\delta^{-\f12}_{*}+\delta_*^{\f12}) \\&\qquad+\|(\partial_x,\partial_z) f\|_{L^2(\Gamma_1)}\delta^{\f12}_{*}\Big) +C|k|^{-1}\|{F}\|_{H^{-1}}\|f\|_{H^1}\\ &\leq C|k|^{-1}\|f\|_{L^2(\Gamma_1)}\|{F}\|_{H^{-1}}\big(\delta_{*}^{\f32}(|1-\lambda|+\delta)\delta^{-3}+\delta_{*}^{-\f12}\big) +C|k|^{-1}\delta_*^{\f12}\|(\partial_x,\partial_z) f\|_{L^2(\Gamma_1)}\|F\|_{H^{-1}}\\ &\quad+C|k|^{-1}\|{F}\|_{H^{-1}}\|f\|_{H^1}.
\end{align*}
Here we used $ \nu^{-1}|k|=\delta^{-3}.$

Taking $\delta_{*}=(|1-\lambda|+\delta)^{-\f12}\delta^{\f32}\leq \delta_1$ due to $\nu k^2\le 1$, we obtain
\begin{align}\label{wyf2}
|\langle w(V-\tilde{\lambda}),f\rangle|\leq& C|k|^{-1}\big(\|f\|_{L^2(\Gamma_1)}+\delta_1\|(\partial_x,\partial_z) f\|_{L^2(\Gamma_1)}\big)\|{F}\|_{H^{-1}}(|1-\lambda|+\delta)^{\f14}\delta^{-\f34}\\ \nonumber& +C|k|^{-1}\|{F}\|_{H^{-1}}\|f\|_{H^1}.
\end{align}
This proves the first inequality of the proposition.

For $f\in H^1(\Omega),\ f|_{y=-1}=0,$ let $\phi=\chi_1f$, where $\chi_1=(V-\lambda-i\delta)^{-1}$. Then we have $\phi\in H^1(\Omega),\ \phi|_{y=-1}=0$. Thus, by \eqref{wyf2} and the fact $(V-\tilde{\lambda})\bar{\chi}_1=1+i(a-1)\delta\bar{\chi}_1$, we have
\begin{align*}
|\langle w,f\rangle|\leq& |\langle w(V-\tilde{\lambda}),\phi\rangle|+ |i(a-1)\delta\langle w,\chi_1f\rangle|\\
 \leq& |\langle w(V-\tilde{\lambda}),\phi\rangle|+ C\delta\|w\|_{L^2}\|\chi_1f\|_{L^2}\\
 \leq& C|k|^{-1}\big(\|\phi\|_{L^2(\Gamma_1)}+\delta_1\|(\partial_x,\partial_z) \phi\|_{L^2(\Gamma_1)}\big)\|{F}\|_{H^{-1}}(|1-\lambda|+\delta)^{\f14}\delta^{-\f34}\\ &+C|k|^{-1}\|{F}\|_{H^{-1}}\|\phi\|_{H^1}+C\delta\|w\|_{L^2}\|\chi_1f\|_{L^2}.
\end{align*}
Thanks to the facts that for $y\in [-1,1]$,
\beno
|\chi_1(y,z)|\leq C(|V-\lambda|+\delta)^{-1},\quad  |\nabla\chi_1(y,z)|\leq C(|V-\lambda|+\delta)^{-2},
\eeno
we deduce that
\begin{align*}
\|\phi\|_{L^2(\Gamma_1)}\leq&\|f\|_{L^2(\Gamma_1)}\|\chi_1(1,z)\|_{L^\infty_z}\leq C\|f\|_{L^2(\Gamma_1)}(|1-\lambda|+\delta)^{-1}, \\
\|(\partial_x,\partial_z)\phi\|_{L^2(\Gamma_1)}=&\|(\partial_x,\partial_z) f\|_{L^2(\Gamma_1)}\|\chi_1(1,z)\|_{L^\infty_z} +\|f\|_{L^2(\Gamma_1)}\|\nabla\chi_1(1,z)\|_{L^\infty_z}\\
\leq& C\|(\partial_x,\partial_z) f\|_{L^2(\Gamma_1)}(|1-\lambda|+\delta)^{-1} +C\|f\|_{L^2(\Gamma_1)}(|1-\lambda|+\delta)^{-2}, \\
\|\phi\|_{H^1}\leq&
\|\nabla f\chi_1\|_{L^2}+\|f\nabla\chi_1\|_{L^2}
+\|f\chi_1\|_{L^2}\\ \leq& C\|\nabla f/(|V-\lambda|+\delta)\|_{L^2}+C\|f/(|V-\lambda|+\delta)^2\|_{L^{2}} +C\|f/(|V-\lambda|+\delta)\|_{L^{2}}\\
\leq& C\|\nabla f/(|V-\lambda|+\delta)\|_{L^2}+C\delta^{-1}\|f/(|V-\lambda|+\delta)\|_{L^{2}},
\end{align*}
and by Proposition \ref{prop:res-nav-s1}, we get
\begin{align*}
&\delta\|w\|_{L^2}\|\chi_1f\|_{L^2}\leq C\delta^{-1}|k|^{-1}\|f/(|V-\lambda|+\delta)\|_{L^2}\|F\|_{H^{-1}}.
\end{align*}
Then,  using $\delta_1\leq\delta,\ \delta_1(|1-\lambda|+\delta)^{-2}\leq(|1-\lambda|+\delta)^{-1}$, we conclude that
\begin{align*}
|\langle w,f\rangle|\leq&  C|k|^{-1}\|{F}\|_{H^{-1}}\Big(
(\|f\|_{L^2(\Gamma_1)}+\delta_1\|(\partial_x,\partial_z) f\|_{L^2(\Gamma_1)})(|1-\lambda|+\delta)^{-\f34}\delta^{-\f34}\\&
+\|\nabla f/(|V-\lambda|+\delta)\|_{L^2}+\delta^{-1}\|f/(|V-\lambda|+\delta)\|_{L^{2}}\Big),
\end{align*}
which gives the second inequality.

The third inequality follows from the second inequality and the following facts that
\begin{align*}
&\|f\|_{L^2(\Gamma_j)}(|j-\lambda|+\delta)^{-\f34}\delta^{-\f34}\leq \|f\|_{L^2(\Gamma_j)}\delta^{-\f32}\leq \delta^{-\f32}\|f\|_{L^2_{x,z}L^{\infty}_y},\\
&\delta_1\|(\partial_x,\partial_z) f\|_{L^2(\Gamma_j)}(|j-\lambda|+\delta)^{-\f34}\delta^{-\f34}\leq \|(\partial_x,\partial_z) f\|_{L^2(\Gamma_j)}\delta^{-\f12},\\
&\|\nabla f/(|V-\lambda|+\delta)\|_{L^2}\leq\delta^{-1}\| \nabla f\|_{L^2},\\&\delta^{-1}\|f/(|V-\lambda|+\delta)\|_{L^{2}}\leq \delta^{-1}\|(|V-\lambda|+\delta)^{-1}\|_{L^\infty_{x,z}L^2_y}\|f\|_{L^2_{x,z} L_y^{\infty}}\leq C\delta^{-\f32}\|f\|_{L^2_{x,z}L^{\infty}_y},\\&\|f\|_{L_{x,z}^2L_y^{\infty}}=|k|^{-1}\|\partial_xf\|_{L_{x,z}^2L_y^{\infty}}\leq C|k|^{-\f12}\| \nabla f\|_{L^2},
\end{align*}
where we used Lemma \ref{lem:sob-f} in the last inequality.

By Hardy's inequality and Proposition \ref{prop:res-nav-s1}, we have (the fourth inequality)
\begin{align*}
   |\langle w,f\rangle|\leq & \left\|\f{w}{1-y^2}\right\|_{L^2}\|(1-y^2)f\|_{L^2}\\
   \leq& C\|\partial_yw\|_{L^2}\|(1-y^2)f\|_{L^2}\leq C\|\nabla w\|_{L^2}\|(1-y^2)f\|_{L^2}\\
   \leq& C\nu^{-1}\|F\|_{H^{-1}}\|(1-y^2)f\|_{L^2}.
\end{align*}

The case of $f|_{y=1}=0$ can be proved similarly.
\end{proof}

\subsection{Estimates of the Neumann data}

In this subsection, we will present some uniform estimates of the Neumann data  $\pa_y\varphi|_{y=\pm 1}$.

\begin{Proposition}\label{prop:res-nav-b1}
Let $w\in H^2(\Omega)$ be a solution of \eqref{eq:OS-Nav} with $F\in L^2(\Omega)$.
Then it holds that
\begin{align*}
&\nu^{\f16}|k|^{\f43}\|\varphi\|_{L^2}+(\nu k^2)^{\f13}\big(\|\nabla \varphi\|_{L^2}+\|(1-y^2)w\|_{L^2}\big) \leq C\big(\|(1-y^2)F\|_{L^2}+|\nu/k|^{\f13}\|F\|_{L^2}\big),\\
&\|\partial_z\partial_y\varphi\|_{L^2(\partial\Omega)}\leq C|\nu k|^{-\f12}\|F\|_{L^2},\\
&\|\partial_z\partial_y\varphi\|_{L^2(\partial\Omega)}\leq C\nu^{-\f56}|k|^{-\f16}\|{F}\|_{H^{-1}}.
\end{align*}
If  $\nu k^2\le 1$, then we have
\begin{align*}
&\big(1+|k(\lambda-j)|\big)\|\partial_y\varphi\|_{L^2(\Gamma_j)}\leq C\nu^{-\f16} |k|^{-\f56}\min(1,|k(\lambda+j)|+\nu^{\f16}|k|^{\f13})\|F\|_{L^2},\ j\in\{\pm1\},\\
&\big\||k(y-\lambda)|^{\f12}\partial_y\varphi\big\|_{L^2(\partial\Omega)}\leq C\nu^{-\f16}|k|^{-\f56}\big(\min(|k(\lambda-1)|^{\f12},|k(\lambda+1)|^{\f12}\big) +\nu^{\f16}|k|^{\f13}\big)\|F\|_{L^2}.
\end{align*}
\end{Proposition}

To estimate $\pa_y\varphi$, we need to use the following facts. First of all, we know that
\begin{align*}
   &\|\partial_y\varphi\|_{L^2(\Gamma_1)}^2=\f{1}{(2\pi)^2}\sum_{\ell\in\mathbb{Z}} \left|\left\langle w,\f{\sinh(\eta(y+1))}{\sinh(2\eta)} e^{i\ell z+ikx}\right\rangle\right|^2.
\end{align*}
Let $\{a_\ell\}$ be any complex valued sequence, such that $\{a_\ell\}_{\ell\in\mathbb{Z}}\in \ell_c^2(\mathbb{Z})$, i.e. $\big\{\ell \in\Z|a_\ell\neq0\big\}$ is a finite set and $\|\{a_\ell\}\|_{\ell^2}=1$.
Then the duality argument gives
\begin{align*}
  \|\partial_y\varphi\|_{L^2(\Gamma_1)} &=\f{1}{(2\pi)^2}\sup_{\{a_\ell\}\in l_c^2,\|\{a_\ell\}\|_{\ell^2}=1}\left|\sum_{\ell\in\mathbb{Z}}a_\ell\left\langle w,\f{\sinh(\eta(y+1))}{\sinh(2\eta)} e^{i\ell z+ikx}\right\rangle\right| \\&=\f{1}{(2\pi)^2}\sup_{\{a_\ell\}\in l_c^2,\ \|\{a_\ell\}\|_{\ell^2}=1}\left|\left\langle w,\sum_{\ell\in\mathbb{Z}}a_\ell\f{\sinh(\eta(y+1))}{\sinh(2\eta)} e^{i\ell z+ikx}\right\rangle\right|.
\end{align*}
We define the set
\begin{align*}
   \mathcal{F}_1=\left\{f(x,y,z)=\sum_{\ell\in\mathbb{Z}}a_\ell\f{\sinh(\eta(1+y))}{\sinh(2\eta)}e^{ikx+i\ell z}\bigg|\{a_\ell\}_{\ell\in\mathbb{Z}}\in \ell_c^2(\mathbb{Z}),\|\{a_\ell\}\|_{\ell^2}=1\right\}.
\end{align*}
Thus, we have
\begin{align}\label{eq:varphi-dual}
   \|\partial_y\varphi\|_{L^2(\Gamma_1)} &=\f{1}{(2\pi)^2}\sup_{f\in\mathcal{F}_1}|\left\langle w,f\right\rangle|.
\end{align}

The following lemma gives some basic estimates of functions in the set $\mathcal{F}_1$.

\begin{Lemma}\label{lem:F1}
For $f\in \mathcal{F}_1$, it holds that
 \begin{align}
  &\|f\|_{L^2(\Gamma_1)}\leq C,\quad  \|f\|_{L^2}\leq C|k|^{-\f12},\label{eq:F1-L2}\\
   &\|(1-y)\nabla f\|_{L^\infty_yL^2_{x,z}}\leq C,\quad \|(1-y)\nabla f\|_{L^2}\leq C|k|^{-\f12},\label{eq:F1-H1}\\
   &\|f\|_{L^2_{x,z}L^\infty_y}\le C,\quad \|(1-y)f\|_{L^2_{x,z}L^\infty_y}\le C|k|^{-1},\label{eq:F1-L2-infty}\\
   &\| f/(1+y)\|_{L^2_{x,z}L^\infty_y}\leq C,\quad\|f/(1+y)\|_{L^2}\leq C|k|^{-\f12}.\label{eq:F1-L2-w}
  \end{align}
\end{Lemma}

\begin{proof}
The first inequality of \eqref{eq:F1-L2} is obvious. The second one follows from
\begin{align*}
 \|f\|_{L^2}^2 =(2\pi)^2 \sum_{\ell\in\mathbb{Z}}\left\|a_l\f{\sinh(\eta(1+y))}{\sinh(2\eta)}\right\|^2_{L^2_y} \leq C\sum_{l\in\mathbb{Z}}(|a_\ell|^2\eta^{-1})\leq C|k|^{-1}.
\end{align*}

 Notice that\begin{align*}
     \|\nabla f\|_{L^2_{x,z}}^2&=(2\pi)^2 \sum_{\ell\in\mathbb{Z}}\bigg(\left|a_\ell\f{\eta\sinh(\eta(1+y))}{\sinh(2\eta)}\right|^2 +\left|a_\ell\f{\eta\cosh(\eta(1+y))}{\sinh(2\eta)}\right|^2\bigg)\\
     &\leq C\sum_{\ell\in\mathbb{Z}}\eta^2|a_\ell|^2e^{-2\eta(1-y)}\leq C(1-y)^{-2}\sum_{\ell\in\mathbb{Z}}|a_\ell|^2e^{-\eta(1-y)}\\
     &\leq C(1-y)^{-2}e^{-|k|(1-y)},
  \end{align*}
which gives  $\|(1-y)\nabla f\|_{L^2_{x,z}}\leq Ce^{-|k|(1-y)/2}$, and then
\begin{align*}
     &\|(1-y)\nabla f\|_{L^\infty_yL^2_{x,z}}\leq C\|e^{-|k|(1-y)/2}\|_{L^\infty_y}\leq C,\\ &\|(1-y)\nabla f\|_{L^2}\leq C\|e^{-|k|(1-y)/2}\|_{L^2_y}\leq  C|k|^{-\f12}.
  \end{align*}

 Since $f|_{y=-1}=0,\ \Delta f=0,$  we get by Lemma \ref{lem:har-max} that
\begin{align*}
  \|f\|_{L^2_{x,z}L^\infty_y}\le C\|f\|_{L^2(\Gamma_1)}\leq C,
\end{align*}
and by \eqref{eq:F1-H1} and \eqref{eq:F1-L2},  we have
\begin{align*}
 \|(1-y) f\|_{L^2_{x,z}L^\infty_y}\leq C|k|^{-\f12}\|\nabla((1-y) f)\|_{L^2}\leq C|k|^{-1}.
  \end{align*}

 Since $f|_{y=-1}=0,\ \Delta f=0,$ we have
\beno
\Delta((1-y)^2 f)=2f-4(1-y)\partial_y f,\quad (1-y)^2f|_{y=\pm1}=0,
\eeno
 and then the elliptic estimate gives
 \begin{align}\label{y2f}
 \|\nabla^2[(1-y)^2 f]\|_{L^2}&\leq C\|\Delta[(1-y)^2 f]\|_{L^2}\nonumber\\
 &\leq C(\|f\|_{L^2}+\|(1-y)\nabla f\|_{L^2})\leq C|k|^{-\f12}.
  \end{align}
Using the fact that $4f/(1+y)=(1-y)^2 f/(1+y)+(3-y)f,$ Hardy's inequality, \eqref{eq:F1-L2} and \eqref{y2f}, we have
\begin{align*}
     \| f/(1+y)\|_{L^2_{x,z}L^\infty_y}&\leq(\|(1-y)^2 f/(1+y)\|_{L^2_{x,z}L^\infty_y}+\|(3-y)f\|_{L^2_{x,z}L^\infty_y})/4\\&\leq \|\nabla[(1-y)^2 f]\|_{L^2_{x,z}L^\infty_y}+\|f\|_{L^2_{x,z}L^\infty_y}\\&\leq C|k|^{-\f12}\|\nabla^2[(1-y)^2 f]\|_{L^2}+C\leq C .
  \end{align*}
 Noting that $2f/(1+y)=(1-y) f/(1+y)+f,$ by Hardy's inequality and \eqref{eq:F1-H1}, we have
 \begin{align*}
     \| f/(1+y)\|_{L^2}&\leq(\|(1-y) f/(1+y)\|_{L^2}+\|f\|_{L^2})/2\\&\leq C\|\nabla[(1-y) f]\|_{L^2}+\|f\|_{L^2}\leq C|k|^{-\f12} .
  \end{align*}

This completes the proof of the lemma.
\end{proof}
\smallskip

Now we prove Proposition \ref{prop:res-nav-b1}.

\begin{proof}
    Let $W=(1-y^2)w$, which satisfies
  \begin{align*}\left\{\begin{aligned}
     &-\nu\Delta W+ik(V-\lambda)W-a(\nu k^2)^{1/3}W=(1-y^2)F+4\nu y\partial_yw+2\nu w,\\
     &W|_{y=\pm1}=0.
     \end{aligned}\right.
  \end{align*}
It follows from Proposition \ref{prop:res-nav-s1} that
  \begin{align*}
     \nu^{\f16}|k|^{\f56}\|W\|_{L^2_{x,z}L^1_y}+ (\nu k^2)^{\f13}\|W\|_{L^2}\leq& C\|(1-y^2)F+4\nu y\partial_yw+2\nu w\|_{L^2}\\
     \leq &C\|(1-y^2)F\|_{L^2}+C\nu\big(\|y\partial_yw\|_{L^2}+\|w\|_{L^2}\big),
  \end{align*}
and by  Proposition \ref{prop:res-nav-s1} again,
\begin{align*}
 \|y\partial_yw\|_{L^2}+\|w\|_{L^2}&\leq \|\nabla w\|_{L^2}+ \|w\|_{L^2}\\
 &\leq C\nu^{-\f23}|k|^{-\f13}(1+|\nu/k|^{\f13})\|F\|_{L^2}\leq C\nu^{-\f23}|k|^{-\f13}\|F\|_{L^2}.
\end{align*}
Then we obtain
\begin{align}\label{eq:W-L2}
   & \nu^{\f16}|k|^{\f56}\|W\|_{L^2_{x,z}L^1_y}+ (\nu k^2)^{\f13}\|W\|_{L^2}\leq C\big(\|(1-y^2)F\|_{L^2} +|\nu/k|^{\f13}\|F\|_{L^2}\big).
\end{align}
By Lemma \ref{lem:elliptic-weight}, we have
\beno
&&\|\nabla \varphi\|_{L^2}\leq C\|(1-y^2)w\|_{L^2}=C\|W\|_{L^2},\\
&&|k|^{\f12}\|\varphi\|_{L^2}\le C\|(1-y^2)w\|_{L^2_{x,z}L^1_y} \leq C\|W\|_{L^2_{x,z}L^1_y},
\eeno
which together with \eqref{eq:W-L2}  show that
\begin{align*}
   & \nu^{\f16}|k|^{\f43}\|\varphi\|_{L^2}+(\nu k^2)^{\f13}(\|\nabla \varphi\|_{L^2}+\|(1-y^2)w\|_{L^2}) \leq C\big(\|(1-y^2)F\|_{L^2}+|\nu/k|^{\f13}\|F\|_{L^2}\big).
\end{align*}

Thanks to $\int_{-1}^{1}\partial_z\partial_y\varphi(x,y_1,z)dy_1=0$, we get
\begin{align*}
   \|\partial_y\partial_z\varphi\|_{L^2(\partial\Omega)}\leq & \|\partial_y\partial_z\varphi\|_{L^2_{x,z}L^\infty_y}\leq C\|\partial_y\partial_z\varphi\|_{L^2}^{\f12} \|\partial^2_y\partial_z\varphi\|^{\f12}_{L^2}\leq C\|w\|_{L^2}^{\f12}\|\nabla w\|_{L^2}^{\f12},
\end{align*}
from which and Proposition \ref{prop:res-nav-s1}, we infer that
\begin{align*}
   \|\partial_y\partial_z\varphi\|_{L^2(\partial\Omega)}\leq C|\nu k|^{-\f12}\|F\|_{L^2},\quad  \|\partial_y\partial_z\varphi\|_{L^2(\partial\Omega)}\leq C\nu^{-\f56}|k|^{-\f16}\|F\|_{H^{-1}}.
   \end{align*}

For the third inequality of the proposition, we just consider the case of $j=1$. Another case is similar.
Notice that $\partial_y[(V-\lambda)\varphi]=(V-\lambda)\partial_y\varphi=(j-\lambda)\partial_y\varphi$ on $ \Gamma_j$. We get by Proposition \ref{prop:res-nav-s1}  that
\begin{align*}
   |j-\lambda|\|\partial_y\varphi\|_{L^2(\Gamma_j)}\leq& \|\partial_y[(V-\lambda)\varphi]\|_{L^2(\Gamma_j)} \\\leq& \|\nabla[(V-\lambda)\varphi]\|_{L_{x,z}^{2}L_y^\infty}\leq C\nu^{-\f16}|k|^{-\f{11}{6}}\|F\|_{L^2}.
\end{align*}
If $|\lambda-1|\geq |k|^{-1}$ and $|k(\lambda+1)|+\nu^{\f16}|k|^{\f13}\geq 1$, then $1+|k(\lambda-j)|\leq2|k(\lambda-j)|$, and then
\begin{align*}
(1+|k(\lambda-1)|)\|\partial_y\varphi\|_{L^2(\Gamma_1)}&\leq 2|k(\lambda-1)|\|\partial_y\varphi\|_{L^2(\Gamma_1)}\leq C\nu^{-\f16}|k|^{-\f{5}{6}}\|F\|_{L^2}\\ &= C\nu^{-\f16} |k|^{-\f56}\min(1,|k(\lambda+1)|+\nu^{\f16}|k|^{\f13})\|F\|_{L^2}.
\end{align*}
If  $|k(\lambda+1)|+\nu^{\f16}|k|^{\f13}\leq1$, then we have $1+|k(\lambda-1)|\leq 2|k|+2\leq 4|k|$.
Thus, by Proposition \ref{prop:res-nav-s1}, Lemma \ref{lem:V} and Lemma \ref{lem:F1}, we deduce that for $f\in\mathcal{F}_1$,
\begin{align*}
  |\langle w,f \rangle| &=|\left\langle (1+V)w,f/(1+V)\right\rangle|\\& =\left|\big\langle (V-\lambda)w,f/(1+V) \big\rangle+ \big\langle (\lambda+1)w,f/(1+V) \big\rangle\right| \\
  &\leq \|(V-\lambda)w\|_{L^2}\left\|f/(1+V)\right\|_{L^2} +|\lambda+1|\|w\|_{L^2_{x,z}L^1_{y}}\left\|f/(1+V )\right\|_{L^2_{x,z}L^\infty_y(\Omega_1)}\\
  &\leq C\Big(|k|^{-1}\left\|f/(1+y)\right\|_{L^2}+ \nu^{-\f16}|k|^{-\f56}|\lambda+1|\| f/(1+y)\|_{L^2_{x,z}L^\infty_y}\Big)\|F\|_{L^2}\\
  &\leq C\big(|k|^{-\f32}+\nu^{-\f16}|k|^{-\f56}|\lambda+1|\big)\|F\|_{L^2},
\end{align*}
which gives
\begin{align*}
   (1+|k(\lambda-1)|)|\langle w,f \rangle|&\leq 4|k||\langle w,f \rangle|\leq C\big(|k|^{-\f12}+\nu^{-\f16}|k|^{-\f56}|k(\lambda+1)|\big)\|F\|_{L^2}\\
   &= C\nu^{-\f16}|k|^{-\f56}\big(|k(\lambda+1)|+\nu^{\f16}|k|^{\f13}\big)\|F\|_{L^2}.
\end{align*}
If  $|\lambda-1|\leq |k|^{-1}$, we get by Proposition \ref{prop:res-nav-s1} and Lemma \ref{lem:F1} that
  \begin{align*}
   |\langle w,f \rangle|&\leq C\|f\|_{L^2_{x,z}L^\infty_{y}}\|w\|_{L^2_{x,z}L^1_y}\leq C\nu^{-\f16}|k|^{-\f56}\|F\|_{L^2}\\
   &\leq C(1+|k(\lambda-1)|)^{-1}\nu^{-\f16}|k|^{-\f56}\|F\|_{L^2}.
  \end{align*}
Combining with two cases, we get by \eqref{eq:varphi-dual} that
\begin{align*}
  (1+|k(1-\lambda)|)\|\partial_y\varphi\|_{L^2(\Gamma_1)} &=\f{1+|k(1-\lambda)|}{(2\pi)^2}\sup_{f\in\mathcal{F}_1} |\langle w,f\rangle| \\
  &\leq C\nu^{-\f16}|k|^{-\f56}\min(1,|k(\lambda+1)|+\nu^{\f16}|k|^{\f13}) \|F\|_{L^2}.
\end{align*}

For $j\in\{\pm1\} $, by the third inequality of the propsition, we have
\begin{align*}
 |k(\lambda-j)|^{\f12}\|\partial_y\varphi\|_{L^2(\Gamma_j)}\leq C\nu^{-\f16}|k|^{-\f56}|k(\lambda-j)|^{\f12}\|F\|_{L^2},
\end{align*}
and
\begin{align*}
   |k(\lambda-j)|^{\f12}\|\partial_y\varphi\|_{L^2(\Gamma_j)}\leq& (1+|k(\lambda-j)|)\|\partial_y\varphi\|_{L^2(\Gamma_j)}\\
   \leq& C\nu^{-\f16}|k|^{-\f56}(\min(1,|k(\lambda+j)|)+\nu^{\f16}|k|^{\f13})\|F\|_{L^2}\\
   \leq&
   C\nu^{-\f16}|k|^{-\f56}(|k(\lambda+j)|^{\f12}+\nu^{\f16}|k|^{\f13})\|F\|_{L^2},
\end{align*}
from which, it follows that
\begin{align*}
\big\||k(y-\lambda)|^{\f12}\partial_y\varphi\big\|_{L^2(\partial\Omega)}\leq C\nu^{-\f16}|k|^{-\f56}(\min(|k(\lambda-1)|^{\f12},|k(\lambda+1)|^{\f12}) +\nu^{\f16}|k|^{\f13})\|F\|_{L^2}.
\end{align*}

This completes the proof of the proposition.
\end{proof}\smallskip

Next we consider the case when $F\in H^{-1}$.

\begin{Proposition}\label{prop:res-weak-b}
Let $\nu k^2\le 1$, and $w\in H^2(\Omega)$ be a solution of \eqref{eq:OS-Nav} with  $F\in L^2(\Omega)$.  Then it holds that
\begin{align*}&\nu^{\f12}|k|\|\nabla\varphi\|_{L^2} \leq C\|{F}\|_{H^{-1}},\\
    &\nu^{\f12}|k|\|\varphi\|_{L^2} \leq C\max(1-|\lambda|,\nu^{\f13}|k|^{-\f13})\|{F}\|_{H^{-1}},\\
&\|(|k(y-\lambda)|+1)^{\f34}\partial_y\varphi\|_{L^2(\partial\Omega)} \leq C|\nu k|^{-\f12}\|{F}\|_{H^{-1}}.
  \end{align*}
\end{Proposition}

In what follows, we assume $\nu k^2\le 1$. We need the following lemmas.

\begin{Lemma}\label{lem:F1-lh}
Let $f\in \mathcal{F}_1$. We decompose $f=f^l+f^h$, where
\begin{align*}
   f^{l}(x,y,z)&=\sum_{\ell^2\leq N_1(k)}a_\ell\f{\sinh(\eta(1+y))}{\sinh(2\eta)}e^{ikx+i\ell z},\\
   f^{h}(x,y,z)&=\sum_{\ell^2> N_1(k)}a_\ell\f{\sinh(\eta(1+y))}{\sinh(2\eta)}e^{ikx+i\ell z},
\end{align*}
where $N_1(k)=\max(\delta^{-2}(|k(1-\lambda)|+1)-k^2,0)$. Then it holds that
  \begin{align*}
     &\|(f^{l},f^{h})\|_{L^2(\Gamma_1)}\leq C,\quad   \|(1-y^2)f^{h}\|_{L^2}\leq C\delta^{\f32}(|k(1-\lambda)|+1)^{-\f34},\\
     &\|(f^{l},f^{h})\|_{L^2}\leq C|k|^{-\f12},\quad \|\nabla f^{l}\|_{L^2}\leq C\delta^{-\f12}(|k(1-\lambda)|+1)^{\f14},\\
    &\|(f^{l},f^{h})\|_{L^2_{x,z}L^\infty_y}\leq C,\quad \|(1-y) f^{l}\|_{L^2_{x,z}L^\infty_y}\leq C|k|^{-1}.
  \end{align*}
\end{Lemma}

\begin{proof}
  Let $\delta_1\triangleq \delta(|k(1-\lambda)|+1)^{-\f12}$. It is easy to see that
 \beno
 \ell^2\leq N_1(k)\Leftrightarrow \eta\leq \delta_1^{-1},\quad \ell^2> N(k)\Leftrightarrow \eta>\delta_1^{-1}
 \eeno
 Then we have
 \begin{align*}
  \|(1-y^2)f^{h}\|^2_{L^2}=&\sum_{\ell^2>N_1(k)}(2\pi)^2\left\| a_\ell(1-y^2)\f{\sinh(\eta(1+y))}{\sinh(2\eta)}\right\|_{L^2_y}^2\leq C\sum_{\ell^2> N_1(k)}|a_\ell|^2\eta^{-3}\\
   \leq& C\sum_{\ell^2>N_1(k)}|a_\ell|^2\delta_1^3\leq C\delta_1^{3},
  \end{align*}
   and
   \begin{align*}
  \|\nabla f^{l}\|_{L^2}^2 =&(2\pi)^2 \sum_{\ell^2\leq N_1(k)}\bigg(\left\|a_\ell\f{\eta\sinh(\eta(1+y))}{\sinh(2\eta)}\right\|^2_{L^2_y} +\left\|a_\ell\f{\eta\cosh(\eta(1+y))}{\sinh(2\eta)}\right\|^2_{L^2_y}\bigg)\\
\leq& C\sum_{\ell^2\leq N_1(k)}|a_\ell|^2\eta=C\delta_1^{-1}.
\end{align*}

The proof of the other inequalities is the same as Lemma \ref{lem:F1}.
\end{proof}

\begin{Lemma}\label{lem:F1-V}
Let $f^l$ be as in Lemma \ref{lem:F1-lh}. Then it holds that
  \begin{align*}
     &\| f^{l}/(|V-\lambda|+\delta)\|_{L^2} +\delta\|\nabla f^{l}/(|V-\lambda|+\delta)\|_{L^2} \leq C\delta^{-\f12}\big(|k(1-\lambda)|+1\big)^{-\f34}.
  \end{align*}
\end{Lemma}

\begin{proof}
By Lemma \ref{lem:F1-lh} and Lemma \ref{lem:chi1}, we have
\begin{align*}
   &\|f^{l}/(|V-\lambda|+\delta)\|_{L^2}\leq \|1/(|V-\lambda|+\delta)\|_{L^\infty_{x,z}L^2_y}\|f^{l}\|_{L^2_{x,z}L^\infty_y} \leq C\delta^{-\f12} ,\\
   &\delta\|\nabla f^{l}/(|V-\lambda|+\delta)\|_{L^2}\leq \|\nabla f^{l}\|_{L^2}\leq C\delta^{-\f12}(|k(1-\lambda)|+1)^{\f14}.
\end{align*}
By Lemma \ref{lem:F1-lh},  Lemma \ref{lem:chi1},  Lemma \ref{lem:V} and $|k\delta|\leq 1$, we have \begin{align*}
   |1-\lambda|\|f^{l}/(|V-\lambda|+\delta)\|_{L^2}&\leq \|(V-\lambda)f^{l}/(|V-\lambda|+\delta)\|_{L^2}+\|(1-V)f^{l}/(|V-\lambda|+\delta)\|_{L^2}\\ &\leq \|f^{l}\|_{L^2}+C\|(1-y)f^{l}/(|V-\lambda|+\delta)\|_{L^2}\\&\leq C|k|^{-\f12}+C\|1/(|V-\lambda|+\delta)\|_{L^\infty_{x,z}L^2_y}\|(1-y)f^{l}\|_{L^2_{x,z}L^\infty_y} \\&\leq C|k|^{-\f12}+C\delta^{-\f12}|k|^{-1}\leq C\delta^{-\f12}|k|^{-1},
   \end{align*}
 and
 \begin{align*}
   &|1-\lambda|\delta\|\nabla f^{l}/(|V-\lambda|+\delta)\|_{L^2}\\ &\leq \delta\|(V-\lambda)\nabla f^{l}/(|V-\lambda|+\delta)\|_{L^2}+\delta\|(1-V)\nabla f^{l}/(|V-\lambda|+\delta)\|_{L^2}\\ &\leq\delta\|\nabla f^{l}\|_{L^2}+\|(1-V)\nabla f^{l}\|_{L^2}\leq\delta\|\nabla f^{l}\|_{L^2}+C\|(1-y)\nabla f^{l}\|_{L^2}\\ &\leq C\delta^{\f12}(|k(1-\lambda)|+1)^{\f14}+C|k|^{-\f12}\leq C\delta^{\f12}(|k(1-\lambda)|+1)^{\f14}.
\end{align*}

Summing up, we conclude that
\begin{align*}
   &(1+|k(\lambda-1)|)\|f^{l}/(|V-\lambda|+\delta)\|_{L^2}\leq C\delta^{-\f12} ,\\
   &(1+|k(\lambda-1)|)\delta\|\nabla f^{l}/(|V-\lambda|+\delta)\|_{L^2}\\
   &\qquad\leq C(\delta^{-\f12}+|k|\delta^{\f12})(|k(1-\lambda)|+1)^{\f14}\leq C\delta^{-\f12}(|k(1-\lambda)|+1)^{\f14},
\end{align*}
which show that
\begin{align*}
     &\| f^{l}/(|V-\lambda|+\delta)\|_{L^2} +\delta\|\nabla f^{l}/(|V-\lambda|+\delta)\|_{L^2}\\ \leq& C(1+|k(\lambda-1)|)^{-1}\delta^{-\f12}[1+(|k(1-\lambda)|+1)^{\f14}]\leq C\delta^{-\f12}(|k(1-\lambda)|+1)^{-\f34}.
  \end{align*}
\end{proof}

Now we are in a position to prove Proposition \ref{prop:res-weak-b}.

\begin{proof}
Using Proposition \ref{prop:res-weak} and  the fact that $ \varphi|_{y=\pm 1}=0$, we deduce that
\begin{align*}
     \|\nabla\varphi\|^2_{L^2}=|\langle w,\varphi\rangle|\leq C|k\delta|^{-\f32}\|{F}\|_{H^{-1}}\|\nabla \varphi\|_{L^2},
     \end{align*}
which gives
\begin{align}\label{eq:varphi-na}
 \|\nabla\varphi\|_{L^2}&\leq C|k\delta|^{-\f32}\|{F}\|_{H^{-1}}=C\nu^{-\f12}|k|^{-1}\|{F}\|_{H^{-1}}.
\end{align}

We denote
\beno
N(k)\triangleq\max(|\nu/k|^{-\f23}-k^2,0),\quad \widetilde{\lambda}=\lambda-ia\delta,
\eeno
and let $\Delta\phi=\varphi,\, \phi|_{y=\pm1}=0.$  We decompose $\phi=\phi^{l}+\phi^{h}$, where
  \begin{align*}
     &\phi^{l}(x,y,z)=\sum_{\ell^2\leq N(k)}\f{1}{2\pi}\int_{\mathbb{T}}\phi(x,y,z_1)e^{i\ell(z-z_1)}dz_1,\\
     &\phi^{h}(x,y,z)=\sum_{\ell^2> N(k)}\f{1}{2\pi}\int_{\mathbb{T}}\phi(x,y,z_1)e^{i\ell(z-z_1)}dz_1.
  \end{align*}
It is obvious that $\forall\ \alpha\in\mathbb{Z}^+,\ \beta=\{0,1,2\}$,
  \begin{align}\label{eq:phi-low}
     &\|(\partial_x,\partial_z)^{\alpha}\nabla^\beta\phi^{l}\|_{L^2}\leq \delta^{-\alpha}\|\nabla^\beta\phi^{l}\|_{L^2},\quad \|\nabla^\beta\phi^{h}\|_{L^2}\leq \delta^{\alpha}\|(\partial_x,\partial_z)^{\alpha}\nabla^\beta\phi^{h}\|_{L^2}.
  \end{align}

Next we discuss the following two cases.\smallskip

\no\textbf{Case 1.} $|\lambda|\leq1+\nu^{\f13}|k|^{-\f13}$.\smallskip

Without loss of generality, we may assume that $0\leq\lambda\leq1+\nu^{\f13}|k|^{-\f13}$ . Then we have
\beno
|1-\tilde{\lambda}|\leq \max(1-\lambda,\lambda-1)+a\delta\leq \max(1-|\lambda|,\nu^{\f13}|k|^{-\f13})+\delta\leq 2\max(1-|\lambda|,\nu^{\f13}|k|^{-\f13}).
\eeno
Let  $\phi_1(x,y,z)=\phi^{l}(x,y,z)/(y-1)$. Then we have
\begin{align*}
     &\phi_1(x,y,z)=\frac{1}{1-y}\int_{y}^1\partial_y\phi^{l}(y_1)dy_1= \int_{0}^1\partial_y\phi^{l}(x,1-(1-y)s,z)ds,\\ &\partial_y\phi_1(x,y,z)=\int_{0}^1s\partial_y^2\phi^{l}(x,1-(1-y)s,z)ds,
  \end{align*}
which imply that
\begin{align*}
     &\|\phi_1\|_{L^2}+\|\nabla\phi_1\|_{L^2}\leq C\big(\|\partial_y\phi^l\|_{L^2}+\|\partial_y\nabla\phi^l\|_{L^2}\big)\leq C\|\varphi\|_{L^2},\\
     &\|\partial_z\nabla\phi_1\|_{L^2}\leq C\|\partial_z\partial_y\nabla\phi^l\|_{L^2}\leq C\delta^{-1}\|\nabla^2\phi^l\|_{L^2}\leq C\delta^{-1}\|\varphi\|_{L^2}.
 \end{align*}

 Notice that $\phi_1|_{y=1}=(\partial_y\phi^{l})|_{y=1}$, $(\partial_x,\partial_z)\phi_1|_{y=1}=((\partial_x,\partial_z)\partial_y\phi^{l})|_{y=1}$, $\phi_1|_{y=-1}=0$. Then we infer that
 \begin{align*}
  \|\phi_1\|_{L^2(\Gamma_{1})}\leq& C\|\partial_y\phi^{l}\|_{L^2(\Gamma_1)}\leq \|\partial_y\phi^{l}\|_{L^2_{x,z}L^\infty_y}\\\leq& C\|\partial_y\phi^{l}\|_{L^2}^{\f12}\|\partial^2_y\phi^{l}\|^{\f12}_{L^2}\leq C|k|^{-\f12}\|\varphi\|_{L^2},
 \end{align*}
 and  by \eqref{eq:phi-low},
 \begin{align*}
   \|(\partial_x,\partial_z)\phi_1\|_{L^2(\Gamma_{1})}\leq& C \|(\partial_x,\partial_z)\partial_y\phi^{l}\|_{L^2(\Gamma_1)}\leq \|(\partial_x,\partial_z)\partial_y\phi^{l}\|_{L^2_{x,z}L^\infty_y}\\
    \leq& C\|(\partial_x,\partial_z)\partial_y\phi^{l}\|_{L^2}^{\f12} \|(\partial_x,\partial_z)\partial^2_y\phi^{l}\|^{\f12}_{L^2}\\ \leq& C\delta^{-1}\|\nabla\phi^{l}\|_{L^2}^{\f12} \|\nabla^2\phi^{l}\|^{\f12}_{L^2}\leq C\delta^{-1}|k|^{-\f12}\|\varphi\|_{L^2}.
 \end{align*}

By Proposition \ref{prop:res-weak} and Lemma \ref{lem:sob-f},  we get (here $ \delta_1$ is defined in Proposition \ref{prop:res-weak})
\begin{align*}
  |\langle w,\phi^l\rangle|=& |\langle w,(y-1)\phi_1\rangle|\\
  =&|\langle (V-\tilde{\lambda})w,\phi_1\rangle+(\tilde{\lambda}-1)\langle w,\phi_1\rangle +\langle w, (y-V)\phi_1\rangle|\\ \leq &C|k|^{-1}\|F\|_{H^{-1}}\Big( \|\phi_1\|_{H^1}+(\|\phi_1\|_{L^2(\Gamma_1)}+ \delta_1\|(\partial_x,\partial_z)\phi_1\|_{L^2(\Gamma_1)})(|1-\lambda|+\delta)^{\f14} \delta^{-\f34} \Big)\\
     &+|1-\tilde{\lambda}||\langle w,\phi_1\rangle|+ |\langle w, (y-V)\phi_1\rangle|,
  \end{align*}
  and
  \begin{align*}
     &|\langle w,\phi_1\rangle|\leq C|k\delta|^{-\f32}\|{F}\|_{H^{-1}}\|\nabla\phi_1\|_{L^2}+ C|k|^{-\f32}\delta^{-\f12}\|F\|_{H^{-1}}\|\partial_z\nabla \phi_1\|_{L^2}.
  \end{align*}
  By Lemma \ref{lem:V} and integration by parts(using $(y-V)\phi_1|_{y=\pm1}=0$), we get
  \begin{align*}
     |\langle w, (y-V)\phi_1\rangle|&=|\langle \varphi, \Delta[(y-V)\phi_1]\rangle|\\
     &\leq |\langle \varphi, (y-V)\Delta\phi_1\rangle| +2|\langle \varphi, \nabla(y-V)\cdot\nabla\phi_1\rangle|+|\langle \varphi, [\Delta(y-V)]\phi_1\rangle|\\
     &\leq C\varepsilon_0\|\varphi\|_{L^2}\big(\|(1-y)\Delta\phi_1\|_{L^2} +\|\nabla\phi_1\|_{L^2}+\|\phi_1\|_{L^2}\big)\\
     &\leq C\varepsilon_0\|\varphi\|_{L^2}\big(\|\Delta[(1-y)\phi_1] +2\partial_y\phi_1\|_{L^2} +\|\nabla\phi_1\|_{L^2}+\|\phi_1\|_{L^2}\big)\\
     &\leq C\varepsilon_0\|\varphi\|_{L^2}\big(\|\varphi\|_{L^2} +\|\nabla\phi_1\|_{L^2}+\|\phi_1\|_{L^2}\big)\leq C\varepsilon_0\|\varphi\|^2_{L^2}.
  \end{align*}

Summing up, we conclude that
\begin{align*}
     |\langle w,\phi^{l}\rangle| \leq &C|k|^{-1}\|F\|_{H^{-1}}\Big( \|\phi_1\|_{H^1}+(\|\phi_1\|_{L^2(\Gamma_1)}+ \delta_1\|(\partial_x,\partial_z)\phi_1\|_{L^2(\Gamma_1)})(|1-\lambda|+\delta)^{\f14} \delta^{-\f34} \Big)\\
     &+C|1-\tilde{\lambda}|\Big(|k\delta|^{-\f32}\|{F}\|_{H^{-1}}\|\nabla\phi_1\|_{L^2}+ |k|^{-\f32}\delta^{-\f12}\|F\|_{H^{-1}}\|\partial_z\nabla \phi_1\|_{L^2}\Big)+C\varepsilon_0\|\varphi\|^2_{L^2}\\
     \leq& C(|1-\lambda|+\delta)|k|^{-1}\|F\|_{H^{-1}} \Big(\delta^{-1}\|\phi_1\|_{H^1}+\delta^{-\f32}\big(\|\phi_1\|_{L^2(\Gamma_1)} +\delta\|(\partial_x,\partial_z\big)\phi_1\|_{L^2(\Gamma_1)})\\
     &+|k|^{-\f12}\delta^{-\f32}\|\nabla\phi_1\|_{L^2} +|k|^{-\f12}\delta^{-\f12}\|\partial_z\nabla\phi_1\|_{L^2}\Big) +C\varepsilon_0\|\varphi\|^2_{L^2}\\
     \leq&C(|1-\lambda|+\delta)|k|^{-1}\|F\|_{H^{-1}}\|\varphi\|_{L^2}\big(\delta^{-1} +|k|^{-\f12}\delta^{-\f32}\big)+C\varepsilon_0\|\varphi\|^2_{L^2}\\
     \leq& C\max(1-\lambda,|\nu/k|^{\f13})(\nu k^2)^{-\f12}\|F\|_{H^{-1}}\|\varphi\|_{L^2}+C\varepsilon_0\|\varphi\|^2_{L^2}.
  \end{align*}
  By \eqref{eq:varphi-na}, we have
  \begin{align*}
     |\langle w,\phi^{h}\rangle|\leq& \|\nabla\varphi\|_{L^2}\|\nabla\phi^h\|_{L^2}\leq \|\nabla\varphi\|_{L^2}(\delta^{2}\|\partial_z^2\nabla\phi^h\|_{L^2})\\
     \leq& C \delta^{2}\|\nabla\varphi\|^2_{L^2}\leq C\nu^{-\f13}|k|^{-\f83}\|F\|^2_{H^{-1}}.
  \end{align*}
  Thus, we obtain
  \begin{align*}
     \|\varphi\|^2_{L^2}&=|\langle w,\phi\rangle|\leq |\langle w,\phi^l\rangle|+| \langle w,\phi^h\rangle|\\
     &\leq C\max(1-\lambda,|\nu/k|^{\f13})(\nu k^2)^{-\f12}\|F\|_{H^{-1}}\|\varphi\|_{L^2} +C\nu^{-\f13}|k|^{-\f83}\|F\|^2_{H^{-1}}+C\varepsilon_0\|\varphi\|^2_{L^2},
  \end{align*}
 which  implies (taking $\varepsilon_0$ sufficiently small) that
  \begin{align*}
     \nu^{\f12}|k|\|\varphi\|_{L^2}&\leq  C\max(1-\lambda,|\nu/k|^{\f13})\|F\|_{H^{-1}}.
  \end{align*}

\no\textbf{Case 2.} $|\lambda|\geq1+\nu^{\f13}|k|^{-\f13}$.\smallskip

Let $f_1(x,y,z)=(V-\tilde{\lambda})^{-1}\phi$. Then $f_1|_{y=\pm1}=0$, and by Proposition \ref{prop:res-weak}, we have
\begin{align*}
   \|\varphi\|^2_{L^2}=&|\langle w(V-\tilde{\lambda}),f_1 \rangle|\leq C|k|^{-1}\|F\|_{H^{-1}}\|f_1\|_{H^1}.
\end{align*}
Using the fact that $|V-\tilde{\lambda}|^{-1}\leq C(1+\delta-|y|)^{-1}$, we deduce that
\begin{align*}
   &\left\|\f{\nabla\phi}{V-\tilde{\lambda}}\right\|_{L^2}\leq C \|\nabla\phi\|_{L^2_{x,z}L^\infty_y}\left\|(1+\delta-|y|)^{-1} \right\|_{L^\infty_{x,z}L^2_{y}}\leq C|k|^{-\f12}\delta^{-\f12}\|\varphi\|_{L^2},\\
   &\left\|\f{\phi(\nabla V)}{(V-\tilde{\lambda})^2}\right\|_{L^2} \leq \left\|\f{\phi}{1-|y|}\right\|_{L^2_{x,z}L^\infty_y} \left\|(1+\delta-|y|)^{-1}\right\|_{L^\infty_{x,z}L^2_y}\|\nabla V\|_{L^\infty}\\
   &\qquad\qquad\qquad\leq C\|\partial_y\phi\|_{L^2_{x,z}L^\infty_y}\left\|(1+\delta-|y|)^{-1}\right\|_{L^2} \leq C|k|^{-\f12}\delta^{-\f12}\|\varphi\|_{L^2},\\
   &\left\|\f{\phi}{V-\tilde{\lambda}}\right\|_{L^2}\leq C\left\|\f{\phi}{1-|y|}\right\|_{L^2}\leq C\|\partial_y\phi\|_{L^2} \leq C|k|^{-1}\|\varphi\|_{L^2},
\end{align*}
which give
\begin{align*}
   & \|f_1\|_{H^1}\leq \left\|\f{\nabla\phi}{V-\tilde{\lambda}}\right\|_{L^2} +\left\|\f{\phi(\nabla V)}{(V-\tilde{\lambda})^2}\right\|_{L^2} +\left\|\f{\phi}{V-\tilde{\lambda}}\right\|_{L^2} \leq C|k|^{-\f12}\delta^{-\f12}\|\varphi\|_{L^2}
\end{align*}
Thus, we conclude
\begin{align*}
   & \|\varphi\|^2_{L^2}\leq C|k|^{-1}|k|^{-\f12}\delta^{-\f12}\|F\|_{H^{-1}}\|\varphi\|_{L^2}= C(\nu k^2)^{-\f12}\delta\|F\|_{H^{-1}}\|\varphi\|_{L^2},
\end{align*}
which gives
\begin{align*}
   \nu^{\f12}|k|\|\varphi\|_{L^2}&\leq C |\nu/k|^{\f13}\|F\|_{H^{-1}}\leq C\max(1-|\lambda|,|\nu/k|^{\f13})\|F\|_{H^{-1}}
\end{align*}
Combining with both cases, we prove the second inequality of the proposition.

We decompose $f$ as $f=f^{h}+f^{l}$, where $f^l$ and $f^h$ is as in Lemma \ref{lem:F1-lh}. Recall that
$N_1(k)=\max(\delta^{-2}(|k(1-\lambda)|+1)-k^2,0)$. Then $l^2\leq N_1(k)\Leftrightarrow \eta\leq \delta^{-1}(|k(1-\lambda)|+1)^{\f12}$. Thus, we have
\beno
\|(\partial_x,\partial_z) f^{l}\|_{L^2(\Gamma_1)}\leq\delta^{-1}(|k(1-\lambda)|+1)^{\f12}\|f^{l}\|_{L^2(\Gamma_1)}=\delta_1^{-1}\|f^{l}\|_{L^2(\Gamma_1)}.
\eeno
Now, by Proposition \ref{prop:res-weak}, Lemma \ref{lem:F1-lh} and Lemma \ref{lem:F1-V}, we get
  \begin{align*}
    |\langle w, f^{l}\rangle| \leq& C|k|^{-1}\|{F}\|_{H^{-1}}\Big(\big(\|f^{l}\|_{L^2(\Gamma_1)}+\delta_1\|(\partial_x,\partial_z) f^{l}\|_{L^2(\Gamma_1)}\big) (|1-\lambda|+\delta)^{-\f34}\delta^{-\f34}\\
   &\quad
   +\|\nabla f^{l}/(|V-\lambda|+\delta)\|_{L^2}+\delta^{-1}\|f^{l}/(|V-\lambda|+\delta)\|_{L^2}\Big)\\
   \leq&C|k|^{-1}\|F\|_{H^{-1}}\Big((|1-\lambda|+\delta)^{-\f34}\delta^{-\f34} +C\delta^{-\f32}(|k(1-\lambda)|+1)^{-\f34}\Big)\\
   \leq &C|k|^{-1}\delta^{-\f32}(|k(1-\lambda)|+1)^{-\f34}\|F\|_{H^{-1}}
   =C|\nu k|^{-\f12}(|k(1-\lambda)|+1)^{-\f34}\|F\|_{H^{-1}},
  \end{align*}
and
  \begin{align*}
     |\langle w,f^{h}\rangle|&\leq C\nu^{-1}\|F\|_{H^{-1}}\|(1-y^2)f^{h}\|_{L^2}\leq C\nu^{-1}\delta^{\f32}(|k(1-\lambda)|+1)^{-\f34}\|F\|_{H^{-1}}\\
     &= C|\nu k|^{-\f12}(|k(1-\lambda)|+1)^{-\f34}\|F\|_{H^{-1}}.
  \end{align*}
  This shows that
   \begin{align*}
    |\langle w,f\rangle| &\leq |\langle w,f^l\rangle| +|\langle w,f^h\rangle|\leq C|\nu k|^{-\f12}(|k(1-\lambda)|+1)^{-\f34}\|F\|_{H^{-1}},
  \end{align*}
  which along with \eqref{eq:varphi-dual}  gives
\begin{align*}
  \|\partial_y\varphi\|_{L^2(\Gamma_1)}\leq  C|\nu k|^{-\f12}(|k(1-\lambda)|+1)^{-\f34}\|F\|_{H^{-1}}.
  \end{align*}

  For the case of $j=-1$,we can similarly get
  \begin{align*}
  \|\partial_y\varphi\|_{L^2(\Gamma_{-1})}\leq  C|\nu k|^{-\f12}(|k(1+\lambda)|+1)^{-\f34}\|F\|_{H^{-1}}.
  \end{align*}

  This completes the proof of the proposition.
\end{proof}

\section{$L^p$ estimate of the solutions for the homogeneous OS}

Let $w_{1,\ell}$ and $w_{2,\ell}$ be the solution to the homogeneous Orr-Sommerfeld equation
 \begin{align}\label{eq:w1l}
     \left\{\begin{aligned}
      &-\nu(\partial^2-\eta^2)w_{1,\ell}+ik(y-\lambda)w_{1,\ell}-a(\nu k^2)^{\f13}w_{1,\ell}=0,\\
      &w_{1,\ell}=(\partial_y-\eta^2)\varphi_{1,l}, \quad \varphi_{1,\ell}|_{y=\pm1}=0,\\
      &\partial_y\varphi_{1,\ell}(-1)=0,\quad \partial_y\varphi_{1,\ell}(1)=1,
     \end{aligned}\right.
  \end{align}
 and
  \begin{align}\label{eq:w2l}
     \left\{\begin{aligned}
      &-\nu(\partial^2-\eta^2)w_{2,\ell}+ik(y-\lambda)w_{2,\ell}-a(\nu k^2)^{\f13}w_{2,\ell}=0,\\
      &w_{2,\ell}=(\partial_y-\eta^2)\varphi_{2,l},\quad \varphi_{2,\ell}|_{y=\pm1}=0,\\
      &\partial_y\varphi_{2,\ell}(-1)=1,\quad \partial_y\varphi_{2,\ell}(1)=0.
     \end{aligned}\right.
  \end{align}
Here $\la\in \R,  a\in [0,\eps_1]$ with $\eps_1\le \delta_1$ and $\delta_1$ given by Lemma \ref{lem:A0-w}.
In the sequel, we assume that $k>0$ without loss of generality and let $L=(|k|/\nu)^{1/3}$(so $L=\delta^{-1}$). \smallskip

The goal of  this section is to  establish the following $L^p$ type estimates of $w_{1,\ell}$ and $w_{2,\ell}$, which could be viewed as boundary correctors in the case of nonslip boundary condition, hence describe the boundary behavior of the solution.

\begin{Proposition}\label{prop:w12-bounds}
There exists $k_0>1$ independent of $\nu$ so that
if $L\ge  k_0$, then we have
\begin{align*}
&\|w_{1,\ell}\|_{L^{\infty}}\leq C\Big(\big(|k(\lambda-1)/\nu|+\eta^2\big)^{\f12}+(|k|/\nu)^{\f13}\Big),\\
&\|w_{2,\ell}\|_{L^{\infty}}\leq C\Big(\big(|k(\lambda+1)/\nu|+\eta^2\big)^{\f12}+(|k|/\nu)^{\f13}\Big),\\
&\|(1-|y|)^{\alpha}w_{1,\ell}\|_{L^1}+\|(1-|y|)^{\alpha}w_{2,\ell}\|_{L^1}\leq CL^{-\alpha},\quad \alpha\geq 0,\\
&\|(1-|y|)^{\beta}w_{1,\ell}\|_{L^{\infty}}+\|(1-|y|)^{\beta}w_{2,\ell}\|_{L^{\infty}}\leq CL^{1-\beta},\quad\beta\geq1.
\end{align*}
\end{Proposition}

\begin{Proposition}\label{prop:w12-L2}
There exists $k_0>1$ independent of $\nu$ so that
if $L\ge  k_0$, then we have
\begin{align*}
&\|w_{1,\ell}\|_{L^{2}}\leq C\Big(\big(|k(\lambda-1)/\nu|^{\f14}+|k/\nu|^{\f16}+|k|^{\f12}\big)+(|k/\nu|^{\f13}+|k|)^{-\f12}|\ell|\Big),\\
&\|w_{2,\ell}\|_{L^{2}}\leq C\Big(\big(|k(\lambda+1)/\nu|^{\f14}+|k/\nu|^{\f16}+|k|^{\f12}\big)+(|k/\nu|^{\f13}+|k|)^{-\f12}|\ell|\Big),\\
&\|(1-|y|)^{\beta}w_{1,\ell}\|_{L^{2}}+\|(1-|y|)^{\beta}w_{2,\ell}\|_{L^{2}}\leq CL^{1/2-\beta},\quad\beta\geq1/2,\\
&\|(\partial_y,\eta)\varphi_{1,\ell}\|_{L^2}+\|(\partial_y,\eta)\varphi_{2,\ell}\|_{L^2}\leq C|\nu/k|^{\f16},\\
&|k|^{\f12}\|\varphi_{1,\ell}\|_{L^2}+|k|^{\f12}\|\varphi_{2,\ell}\|_{L^2}\leq C|\nu/k|^{\f13}.
\end{align*}
\end{Proposition}

In the following sections(section 6-15), we always assume $0<\nu\leq \nu_0\leq k_0^{-3}$. Then $L\ge  k_0$.

\subsection{Basic properties of the Airy function}

Let $Ai(y)$ be the Airy function, which is a nontrivial solution of $u''-yu=0$. We introduce some notations
\begin{align*}
&A_0(z)=\int_{e^{i\pi/6}z}^{\infty}Ai(t)dt=e^{i\pi/6}\int_{z}^{\infty}Ai(e^{i\pi/6}t)dt,\\
&\omega(z,x)=\frac{A_0(z+x)}{A_{0}(z)}=\exp\Big(\int_{0}^{x}\frac{A'_0(z+t)}{A_0(z+t)}dt\Big).
\end{align*}

The following two lemmas come from \cite{CLWZ}.

\begin{Lemma}\label{lem:Ao-more}
There exists $c>0,\ \delta_0$ so that for ${{\rm Im }z}\le \delta_0$,
\begin{align}
&\Big|\f{A_0'(z)}{A_0(z)}\Big|\le C\big(1+|z|^{\f12}\big),\quad {\rm Re}\f{A_0'(z)}{A_0(z)}\leq-c\big(1+|z|^{\f12}\big).\nonumber
\end{align}
\end{Lemma}

\begin{Lemma}\label{lem:A0-w}
There exists $\delta_1\in(0,\delta_0/2]$ so that for ${\rm Im}z\le \delta_1$ and $x\ge 0$,
\begin{align*}
  |\omega(z,x)|\leq e^{-\f{x}{3}}.
\end{align*}
\end{Lemma}

\begin{Lemma}\label{lem:Ld}
Let $z\in\mathbb{C}$. It holds that for any $x\ge 0$,
\begin{align}
\int_0^{x}\big(1+|t+z|^{\f12}\big)dt\gtrsim x|z|^{\f12}+|x|^{\f32}.
\end{align}
\end{Lemma}

\begin{proof}
Let  $z=z_1+iz_2$ with $z_1,z_2\in\mathbb{R}$.\smallskip

If $|z_1|\leq 1$, then we have
\beno
1+|t+z|^{\f12}\gtrsim 2+|t+z_1|^{\f12}+|z_2|^{\f12}\geq 2+|t|^{\f12}-|z_1|^{\f12}+|z_2|^{\f12}\geq 1+|t|^{\f12}+|z_2|^{\f12},
\eeno
which gives
  \begin{align*}
     &\int_{0}^{x}(1+|t+z|^{\f12})dt\gtrsim \int_{0}^{x}(1+|t|^{\f12}+|z_2|^{\f12})dt\gtrsim x+x^{\f32}+x|z_2|^{\f12}\geq x|z|^{\f12}+x^{\f32}.
  \end{align*}

  If $|z_1|\geq1$ and $x\geq|z_1|$, then we have
  \begin{align*}
     \int_{0}^{x}|t+z_1|^{\f12}dt&=\int_{0}^{|z_1|}|t+z_1|^{\f12}dt+\int_{|z_1|}^{x}|t+z_1|^{\f12} dt\geq \int_{0}^{|z_1|}(|z_1|^{\f12}-t^{\f12})dt+\int_{|z_1|}^{x}(t-|z_1|)^{\f12}dt\\
     &\gtrsim |z_1|^{\f32}+(x-|z_1|)^{\f32}\gtrsim x^{\f32}\gtrsim x|z_1|^{\f12}+x^{\f32},
  \end{align*}
which gives
  \begin{align*}
     &\int_{0}^{x}(1+|t+z|^{\f12})dt\gtrsim \int_{0}^{x}(1+|t+z_1|^{\f12}+|z_2|^{\f12})dt\gtrsim x|z_1|^{\f12}+x^{\f32}+x|z_2|^{\f12}\gtrsim x|z|^{\f12}+x^{\f32}.
  \end{align*}

  If $|z_1|\geq 1$ and $|z_1|\geq x$, then we have
  \begin{align*}
     &\int_{0}^{x}|t+z_1|^{\f12}dt\geq \int_{0}^{x/2}|t+z_1|^{\f12}dt\geq \int_{0}^{x/2}(|z_1|-t)^{\f12}dt\geq (|z_1|-x/2)^{\f12}x/2\gtrsim x|z_1|^{\f12},
  \end{align*}
  which gives

  \begin{align*}
     &\int_{0}^{x}(1+|t+z|^{\f12})dt\gtrsim \int_{0}^{x}(1+|t+z_1|^{\f12}+|z_2|^{\f12})dt\gtrsim x|z_1|^{\f12}+x|z_2|^{\f12}\gtrsim x|z|^{\f12}+x^{\f32}.
  \end{align*}

Summing up, we conclude the lemma.
\end{proof}

\begin{Lemma}\label{lem:A0-w2}
Let $\delta_0$ be as in Lemma \ref{lem:Ao-more}. Then it holds that for ${\rm Im}z\le \delta_0$ and $x\ge 0$,
\begin{align*}
  |\omega(z,x)|\leq e^{-c(x|z|^{\f12}+x^{3/2})}.
\end{align*}
\end{Lemma}

\begin{proof}
By Lemma \ref{lem:Ao-more} and Lemma \ref{lem:Ld}, we get
\begin{align*}
 |\omega(z,x)|\leq& \Big|\exp\Big({\rm Re}\int_{0}^{x}\frac{A'_0(z+t)}{A_0(z+t)}dt\Big)\Big|\\
 \leq&\exp\Big(-c\int_{0}^{x}(1+|z+t|^{\f12})dt\Big)\leq e^{-c(x|z|^{\f12}+x^{3/2})}.
\end{align*}
\end{proof}

\begin{Lemma}\label{lem:A0-w3}
Let $\delta_1$ be as in Lemma \ref{lem:A0-w}. There exists $k_0>1$ so that if $L\geq k_0$, $\eta\geq 1$, $ \mathbf{Im}z\leq \delta_1-\eta^2/L^2$, then we have
 \begin{align*}
&\Big|\sinh(2\eta)- \frac{\eta}{L}\int_{0}^{2L}\cosh\Big(2\eta-\frac{\eta t}{L}\Big)\omega(z,t) dt\Big|\\
&\quad\geq\sqrt{2}\Big|\sinh(2\eta)\omega(z,2L)- \frac{\eta}{L}\int_{0}^{2L}\cosh\Big(\frac{\eta t}{L}\Big)\omega(z,t) dt\Big|.
 \end{align*}
\end{Lemma}

\begin{proof}
Let $c_*>0$ be the constant $c$ in Lemma \ref{lem:A0-w2} and $b=\max(1/3,c_*|z|^{\f12})$. It follows from Lemma \ref{lem:A0-w} and Lemma \ref{lem:A0-w2} that for ${\rm Im}z\le \delta_1,\ x\ge0$ we have $|\omega(z,x)|\leq e^{-bx} $ and
\begin{align*}
\sinh(2\eta)- \Big|\frac{\eta}{L}\int_{0}^{2L}\cosh(2\eta-\frac{\eta t}{L})\omega(z,t) dt\Big|\geq&\sinh(2\eta)- \frac{\eta}{L}\int_{0}^{2L}\cosh\Big(2\eta-\frac{\eta t}{L}\Big)|\omega(z,t)| dt\\
\geq&\sinh(2\eta)- \frac{\eta}{L}\int_{0}^{2L}\cosh\Big(2\eta-\frac{\eta t}{L}\Big)e^{-bt} dt\\
=&\int_{0}^{2L}\sinh\Big(2\eta-\frac{\eta t}{L}\Big)(be^{-bt}) dt\\
\geq\int_{0}^{L/\eta}\sinh(2\eta-1)(be^{-bt}) dt
=&\sinh(2\eta-1)\big(1-e^{-bL/\eta}\big),
\end{align*}
where
 \begin{align*}
 \sinh(2\eta-1)=e^{2\eta-1}(1-e^{-2(2\eta-1)})/2\geq e^{2\eta-1}(1-e^{-2})/2\geq e^{-1}\sinh(2\eta)(1-e^{-2}),
 \end{align*}
 and
 \begin{align*}
 &b=\max(1/3,c_*|z|^{\f12})=[\max(1/9,c_*^2|z|)]^{\f12}\geq[(1+|z|)/(9+c_*^{-2})]^{\f12},\\&1+|z|\geq 1-\mathbf{Im}z\geq \eta^2/L^2,\quad b\geq c_1(1+|z|)^{\f12}=c_1\eta/L,\quad 1-e^{-bL/\eta}\geq 1-e^{-c_1},
 \end{align*}with $c_1=(9+c_*^{-2})^{-\f12}>0$.
 This shows that
 \begin{align*}
2\sinh(2\eta)\geq\Big|\sinh(2\eta)- \frac{\eta}{L}\int_{0}^{2L}\cosh\Big(2\eta-\frac{\eta t}{L}\Big)\omega(z,t) dt\Big|\geq& c_2\sinh(2\eta)
\end{align*}
with $c_2=e^{-1}(1-e^{-2})(1-e^{-c_1})\in (0,1)$.

On the other hand, we have
 \begin{align*}
  &\Big|\sinh(2\eta)\omega(z,2L)- \frac{\eta}{L}\int_{0}^{2L}\cosh\Big(\frac{\eta t}{L}\Big)\omega(z,t) dt\Big|\\ &\quad\leq \sinh(2\eta)|\omega(z,2L)|+ \frac{\eta}{L}\int_{0}^{2L}\cosh\Big(\frac{\eta t}{L}\Big)|\omega(z,t)| dt
  \\\ &\quad\leq \sinh(2\eta)e^{-2L/3}+\frac{\eta}{L}\int_{0}^{2L}\cosh\Big(\frac{\eta t}{L}\Big)e^{-t/3} dt
  \\&\quad=2\sinh(2\eta)e^{-2L/3}+ \int_{0}^{2L}\sinh\Big(\frac{\eta t}{L}\Big)\frac{e^{-t/3}}{3} dt\\
  &\quad\leq 2\sinh(2\eta)e^{-2L/3}+ \sinh(2\eta)\int_{0}^{2L}\frac{ t}{2L}\frac{e^{-t/3}}{3} dt\\&\quad {\leq}\sinh(2\eta)(2e^{-2L/3}+3/(2L))\leq (3/L)\sinh(2\eta).
 \end{align*}
 Here we used the fact that $(\sinh x)/x$ is increasing. Now the result follows by choosing $k_0=3\sqrt{2}/c_2>1$.
\end{proof}

\subsection{The solution of the homogeneous OS equation}
Let
\begin{align*}
u_1(y)=Ai(e^{i\f{\pi}{6}}y),\quad u_2(y)=Ai(e^{i\f{5\pi}{6}}y).
\end{align*}
Then $u_1$ and $u_2$ are two linearly independent solutions of $u''-iyu=0$. Hence,
\begin{align*}
W_{1,\ell}(y)=Ai\big(e^{i\f{\pi}{6}}(L(y-\lambda-i \nu\eta^2/k)+ia)\big),\quad W_{2,\ell}(y)=Ai\big(e^{i\f{5\pi}{6}}(L(y-\lambda-i \nu\eta^2/k)+ia)\big)
\end{align*}
are two linearly independent solutions of the homogeneous OS equation
\beno
-\nu (w''-\eta^2w)+ik(y-\lambda)w-a\nu^{\f13}|k|^{\f23}w=0.
\eeno
We denote
\begin{align*}
d=-1-\lambda-i\nu\eta^2/k,\quad\widetilde{d}=-1+\lambda-i\nu\eta^2/k.
\end{align*}

We have the following estimates for $W_{1,\ell}$ and $W_{2,\ell}$.

\begin{Lemma}\label{lem:W12}
It holds that
\begin{align*}
&\f{L}{|A_0(Ld+ia)|}\|W_{1,\ell}\|_{L^{\infty}}\leq C\Big(( |k(\lambda+1)/\nu|+\eta^2)^{\f12}+(|k|/\nu)^{\f13}\Big),\\
&\f{L}{|A_0(L\widetilde{d}+ia)|}\|W_{2,\ell}\|_{L^{\infty}}\leq C\Big(( |k(\lambda-1)/\nu|+\eta^2)^{\f12}+(|k|/\nu)^{\f13}\Big),
\end{align*}
and  for $\al\ge 0, \be\ge 1$,
\begin{align*}
&\f{L}{|A_0(Ld+ia)|}\|(1-|y|)^{\alpha}W_{1,\ell}\|_{L^1}+\f{L}{|A_0(L\widetilde{d}+ia)|}\|(1-|y|)^{\alpha}W_{2,\ell}\|_{L^1}\leq CL^{-\alpha},\\
&\f{L}{|A_0(Ld+ia)|}\|(1-|y|)^{\beta}W_{1,\ell}\|_{L^{\infty}}+\f{L}{|A_0(L\widetilde{d}+ ia)|}\|(1-|y|)^{\beta}W_{2,\ell}\|_{L^{\infty}}\leq CL^{1-\beta}.
\end{align*}
Here $C$ may depend on $\al,\beta$.
\end{Lemma}

\begin{proof}
   Thanks to the definition of $W_{1,\ell}$ and $\omega(z,x)$, Lemma \ref{lem:Ao-more}, Lemma \ref{lem:A0-w}, we have
  \begin{align*}
    \f{L|W_{1,\ell}(y)|}{|A_0(Ld+ia)|} &\leq \left|\f{L Ai(e^{i\f{\pi}{6}}L(y-\lambda-i\nu\eta^2/|k|)+ia)}{A_0(Ld+ia)}\right|= \left|\f{L A'_0(L(y+1)+Ld+ia)}{A_0(Ld+ia)}\right|\\
    &=L\left|\f{ A'_0(L(y+1)+Ld+ia)}{A_0(L(y+1)+Ld+ia)}\right| |\omega(Ld+ia,L(y+1))|\\
    & \leq CL(1+|Ld+ia+L(y+1)|^{\f12}) e^{-L(y+1)/3}\\
    &\leq C(L(1+|Ld|)^{\f12})+CL^{\f32}(y+1)^{\f12}e^{-L(y+1)/3}\\&\leq C(L+L|Ld|^{\f12})
    \leq C(|k(1+\lambda)/\nu|+\eta^2)^{\f12}+|k/\nu|^{\f13},
  \end{align*}
  here we use $L|Ld|^{\f12}\leq (L^3(|1+\lambda|+\nu\eta^{2}/|k|))^{\f12}=(|k(1+\lambda)/\nu|+\eta^2)^{\f12}$. This yields
  \begin{align*}
     &\f{L}{|A_0(Ld+ia)|}\|W_{1,\ell}\|_{L^{\infty}}\leq C\Big(( |k(\lambda+1)/\nu|+\eta^2)^{\f12}+|k/\nu|^{\f13}\Big).
  \end{align*}

  Thanks to Lemma \ref{lem:Ao-more}, Lemma \ref{lem:A0-w2}, we have
  \begin{align*}
     \f{L\|(1-|y|)^{\alpha}W_{1,\ell}\|_{L^1} }{|A_0(Ld+ia)|}
     =& L\int_{-1}^{1} (1-|y|)^{\alpha} \bigg|\f{Ai(e^{i\f{\pi}{6}}(L(y-\lambda-i\nu\eta^2/|k|)+ia))}{A_0(Ld+ia)}\bigg| dy\\
     =&L\int_{0}^{2} (1-|x-1|)^{\alpha} \bigg|\f{Ai(e^{i\f{\pi}{6}}(Lx+Ld+ia))}{A_0(Ld+ia)}\bigg| dx\\
     \leq &CL\int_{0}^{2} x^{\alpha}\bigg|\f{A_0'(Lx+Ld+ia)}{A_0(Lx+Ld+ia)}\bigg| |\omega(Ld+ia,Lx)|dx\\
     \leq&CL\int_{0}^{2}x^{\alpha} (1+|Lx+Ld+ia|)^{\f12}e^{-c(|Lx||Ld+ia|^{\f12}+|Lx|^{\f32})}dx\\
     \leq&CL\int_{0}^{2} x^{\alpha} (1+|Ld+ia|)^{\f12}e^{-c|Lx|(1+|Ld+ia|^{\f12})}dx\\
     \leq &CL^{-\alpha}(1+|Ld+ia|)^{-\f{\alpha}{2}}\int_{0}^{+\infty}t^{\alpha}e^{-ct}dt \leq CL^{-\alpha}.
     \end{align*}
  Here we use $(1+|Lx|)^\f12e^{-c|Lx|^{3/2}}\leq Ce^{-c|Lx|} $ for every $x\geq 0$ and fixed $c>0,$ and make a change of variable $t=Lx(1+|Ld+ia|)^{\f12}$.\smallskip

 Let $x=y+1\in[0,2]$ and $\beta\geq1 $. Thanks to Lemma \ref{lem:Ao-more}, Lemma \ref{lem:A0-w2}, we have
  \begin{align*}
     &\f{L(1-|y|)^{\beta}}{|A_0(Ld+ia)|}|W_{1,\ell}(y)|\\
     & \leq (1-|y|)^{\beta}\left|\f{L Ai(e^{i\f{\pi}{6}}(L(y-\lambda-i\nu\eta^2/|k|)+ia))}{A_0(Ld+ia)}\right|\\
     &= (1-|y|)^{\beta}\left|\f{L A'_0(L(y+1)+Ld+ia)}{A_0(Ld+ia)}\right|\\
    &=L(1-|y|)^{\beta}\left|\f{ A'_0(L(y+1)+Ld+ia)}{A_0(L(y+1)+Ld+ia)}\right| |\omega(Ld+ia,L(y+1))|\\
    & \leq CLx^{\beta}(1+|Ld+ia+Lx|^{\f12}) e^{-c(Lx|Ld+ia|^{\f12}+|Lx|^{\f32})}\\
    &\leq CLx^{\beta}(1+|Ld+ia|)^{\f12}e^{-c(Lx|Ld+ia|^{\f12}+Lx)}\\
    &\leq CL^{1-\beta} (1+|Ld+ia|)^{1/2-\beta/2} \big((Lx(1+|Ld+ia|)^{\f12})^{\beta}e^{-cLx(1+|Ld+ia|)^{\f12}}\big)\\
    &\leq CL^{1-\beta} (1+|Ld+ia|)^{1/2-\beta/2}\leq CL^{1-\beta},
   \end{align*}
which gives
  \begin{align*}
    \f{L}{|A_0(Ld+ia)|}\|(1-|y|)^{\beta}W_{1,\ell}\|_{L^\infty} &\leq CL^{1-\beta}.
  \end{align*}

Thus,  we finish the estimate for $W_{1,\ell}$. The proof for $W_{2,\ell}$ is similar.
  \end{proof}

\subsection{Proof of  Proposition \ref{prop:w12-bounds} and Proposition \ref{prop:w12-L2}}

The solution $w_{1,\ell}$ and $w_{2,\ell}$ of  \eqref{eq:w1l} and \eqref{eq:w2l} can be expressed as
\begin{align}
\label{two linear solutions}w_{1,\ell}=C_{11}W_{1,\ell}(y)+C_{12} W_{2,\ell}(y),\quad w_{2,\ell}=C_{21}W_{1,\ell}(y)+C_{22}W_{2,\ell}(y),
\end{align}
where $C_{ij}, i,j=1,2$ are constants (depending only on $\nu,k,\ell$). Thanks to the facts that
\begin{align*}
&\int_{-1}^1e^{\eta y}w_{1,\ell}(y)dy=e^\eta,\quad \int_{-1}^1e^{-\eta y}w_{1,\ell}(y)dy=e^{-\eta},
\end{align*}
we deduce that
\begin{align*}
\left\{
\begin{aligned}
&\sinh(2\eta)=C_{11}\int_{-1}^1\sinh(\eta(y+1))W_{1,\ell}(y)dy+C_{12}\int_{-1}^1\sinh(\eta(y+1))W_{2,\ell}(y)dy,\\
&0=C_{11}\int_{-1}^1\sinh(\eta(1-y))W_{1,\ell}(y)dy+C_{12}\int_{-1}^1\sinh(\eta(1-y))W_{2,\ell}(y)dy.
\end{aligned}
\right.
\end{align*}
We denote
\begin{align*}
&A_1=\int_{-1}^1\sinh(\eta(y+1))W_{1,\ell}(y)dy,\quad A_2=\int_{-1}^1\sinh(\eta(1-y))W_{2,\ell}(y)dy,\\
&B_1=\int_{-1}^1\sinh(\eta(1-y))W_{1,\ell}(y)dy,\quad B_2=\int_{-1}^1\sinh(\eta(y+1))W_{2,l}(y)dy.
\end{align*}
If $A_1A_2-B_1B_2\neq0 $, then we have
\begin{equation}\nonumber
\left(   C_{11} ,  C_{12}
\right)=\f{\sinh(2\eta)( A_2,  -B_1)}{A_1A_2-B_1B_2}.
\end{equation}
Similarly, we have
\begin{equation}\nonumber
\left(   C_{21},  C_{22}
 \right)=\f{\sinh(2\eta)( B_2,  -A_1)}{A_1A_2-B_1B_2}.
\end{equation}

\begin{Lemma}\label{lem:C-ij}
Let $k_0$ be as in Lemma \ref{lem:A0-w3} and $\delta_1$ be as in Lemma \ref{lem:A0-w}. If $L\geq k_0$, then it holds that
\begin{align*}
&|C_{11}|\leq \f{C}{|A_0(Ld+ia)|},\quad|C_{12}|\leq \f{CL}{|A_0(L\widetilde{d}+ia)|},\\
&|C_{21}|\leq \f{CL}{|A_0(Ld+ia)|},\quad
|C_{22}|\leq \f{C}{|A_0(L\widetilde{d}+ia)|}.
\end{align*}
\end{Lemma}

\begin{proof}

Let $y+1=x=\frac{t}{L}$. Due to $A'_0(z)=-e^{i\pi/6}Ai(e^{i\pi/6}z)$,
we have
\begin{align}
 \nonumber B_1  =&\int_{0}^{2}\sinh(\eta(2-x))Ai(e^{i\pi/6}(L(x+d)+ia))dx\\
  \nonumber =&\f1L\int_{0}^{2L}\sinh\Big(2\eta-\frac{\eta t}{L}\Big)Ai(e^{i\pi/6}((t+Ld)+ia))dt\\
  \nonumber =&-\frac{e^{-i\pi/6}}{L} \int_{0}^{2L}\sinh\Big(2\eta-\frac{\eta t}{L}\Big)A'_0(t+Ld+ia)dt\\
  \nonumber  =& -\frac{e^{-i\pi/6}}{L} \Big[- \sinh(2\eta)A_0(Ld+ia)+ \frac{\eta}{L}\int_{0}^{2L}\cosh\Big(2\eta-\frac{\eta t}{L}\Big)A_0(t+Ld+ia) dt\Big]\\
=& A_0(Ld+ia)\frac{e^{-i\pi/6}}{L} \Big[\sinh(2\eta)- \frac{\eta}{L}\int_{0}^{2L}\cosh\Big(2\eta-\frac{\eta t}{L}\Big)\omega(Ld+ia,t) dt\Big].\label{eq:B1}
\end{align}
Similarly, we have
 \begin{align*}
 A_1 & = -A_0(Ld+ia)\frac{e^{-i\pi/6}}{L}\Big [\sinh(2\eta)\omega(Ld+ia,2L)-  \frac{\eta}{L}\int_{0}^{2L}\cosh\Big(\frac{\eta t}{L}\Big)\omega(Ld+ia,t)dt\Big].
 \end{align*}
 Due to $ d=-1-\lambda-i\nu\eta^2/k,\ \mathbf{Im}(Ld+ia)=-L\nu\eta^2/k+a=-\eta^2/L^2+a\leq \delta_1-\eta^2/L^2$.
 Then we infer from Lemma \ref{lem:A0-w3} that
 \begin{align}\label{eq:A1-B1}
  \Big|\frac{A_1}{B_1}\Big| = \Big|\frac{\sinh(2\eta)\omega(Ld+ia,2L)-  \frac{\eta}{L}\int_{0}^{2L}\cosh(\frac{\eta t}{L})\omega(Ld+ia,t)dt}{\sinh(2\eta)- \frac{\eta}{L}\int_{0}^{2L}\cosh(2\eta-\frac{\eta t}{L})\omega(Ld+ia,t) dt}\Big|\le \f {\sqrt{2}} 2.
   \end{align}
Similarly, using ${Ai}(z)=\overline{{A{i}}(\bar{z})}$ and Lemma \ref{lem:A0-w3}, we have
 \begin{align}\label{eq:A2-B2}
        \Big|\frac{A_2}{B_2}\Big|\le \f {\sqrt{2}} 2.
          \end{align}
Now it follows from \eqref{eq:A1-B1} and \eqref{eq:A2-B2} that
\beno
|A_1A_2-B_1B_2|\gtrsim |B_1B_2|.
\eeno
From the proof of Lemma \ref{lem:A0-w3} and \eqref{eq:B1}, we know that
\begin{align*}
|B_1|\geq&\f{1}{L}|A_0(Ld+ia)|\Big[\sinh(2\eta)- \frac{\eta}{L}\int_{0}^{2L}\cosh\Big(2\eta-\frac{\eta t}{L}\Big)\omega(Ld+ia,t) dt\Big]\\
\geq&c_2\f{\sinh(2\eta)}{L}|A_0(Ld+ia)|.
\end{align*}
Similarly, $|B_2|\geq c_2\f{\sinh(2\eta)}{L}|A_0(L\widetilde{d}+ia)|$. Thus, we get
\begin{align*}
|A_1A_2-B_1B_2|\gtrsim|A_0(Ld+ia)||A_0(L\widetilde{d}+ia)|\f{[\sinh(2\eta)]^{2}}{L^2}.
\end{align*}
Furthermore, we can deduce that
\begin{align*}
&|B_1|\leq 2 \f{\sinh(2\eta)}{L} |A_0(Ld+ia)|,\quad
 |B_2|\leq 2\f{\sinh(2\eta)}{L}|A_0(L\widetilde{d}+ia)|,\\
&|A_1| \le 3\f{\sinh(2\eta)}{L^2} |A_0(Ld+ia)|,\quad
  |A_2| \leq 3\f{\sinh(2\eta)}{L^2}|A_0(L\widetilde{d}+ia)| .
\end{align*}

Summing up, we can conclude the estimates of $C_{ij}$.
\end{proof}

\smallskip

Now Proposition \ref{prop:w12-bounds} is an immediate consequence of Lemma \ref{lem:C-ij} and Lemma \ref{lem:W12}. Next, let us prove Proposition \ref{prop:w12-L2}.

\begin{proof}
By Proposition \ref{prop:w12-bounds} and H\"older inequality, we get
\begin{align*}
\|w_{1,\ell}\|_{L^{2}}\leq&\|w_{1,\ell}\|_{L^{\infty}}^{\f12}\|w_{1,\ell}\|_{L^1}^{\f12}\leq C\Big(( |k(\lambda-1)/\nu|+\eta^2)^{\f12}+(|k|/\nu)^{\f13}\Big)^{\f12}\\ \leq &C\Big(|k(\lambda-1)/\nu|^{\f14}+|k/\nu|^{\f16}+|\eta|^{\f12}\Big)\leq C\Big(|k(\lambda-1)/\nu|^{\f14}+|k/\nu|^{\f16}+|k|^{\f12}+|\ell|^{\f12}\Big)\\ \leq & C\Big(|k(\lambda-1)/\nu|^{\f14}+|k/\nu|^{\f16}+|k|^{\f12}+(|k/\nu|^{\f13}+|k|)^{\f12}+(|k/\nu|^{\f13}+|k|)^{-\f12}|\ell|\Big)\\ \leq & C\Big(|k(\lambda-1)/\nu|^{\f14}+|k/\nu|^{\f16}+|k|^{\f12}+(|k/\nu|^{\f13}+|k|)^{-\f12}|\ell|\Big),
\end{align*}
which gives the first inequality. The second inequality can be proved similarly.

For fixed $\beta\geq1/2, $ by Proposition \ref{prop:w12-bounds} and H\"older inequality, we have\begin{align*}
&\|(1-|y|)^{\beta}w_{1,\ell}\|_{L^{2}}\leq\|(1-|y|)^{\beta+\f12}w_{1,\ell}\|_{L^{\infty}}^{\f12}\|(1-|y|)^{\beta-\f12}w_{1,\ell}\|_{L^1}^{\f12}\leq CL^{1/2-\beta}.
\end{align*}
Similarly, $\|(1-|y|)^{\beta}w_{2,\ell}\|_{L^{2}}\leq CL^{1/2-\beta}$. This proves the third inequality.

By Lemma \ref{lem:ell-w} and the third inequality with $\beta=1$, we have
\begin{align*}
&\|(\partial_y,\eta)\varphi_{1,\ell}\|_{L^{2}}\leq\|(1-|y|)w_{1,\ell}\|_{L^{2}}\leq CL^{-1/2}= C|\nu/k|^{\f16}.
\end{align*}
Similarly, $\|(\partial_y,\eta)\varphi_{2,\ell}\|_{L^2}\leq C|\nu/k|^{\f16}$, thus we conclude the fourth inequality.

By Lemma \ref{lem:ell-w} and Proposition \ref{prop:w12-bounds} with $\alpha=1$, we have
\begin{align*}
&|k|^{\f12}\|\varphi_{1,\ell}\|_{L^2}\leq\eta^{\f12}\|\varphi_{1,\ell}\|_{L^{2}}\leq\|(1-|y|)w_{1,\ell}\|_{L^{1}}\leq CL^{-1}= C|\nu/k|^{\f13}.
\end{align*}
Similarly, $|k|^{\f12}\|\varphi_{2,\ell}\|_{L^2}\leq C|\nu/k|^{\f13}, $ thus we conclude the fifth inequality.\end{proof}

\section{Resolvent estimates via the Neumann boundary data}

In this section, we consider the Orr-Sommerfeld equation with variable coefficient:
\begin{align}\label{eq:OS-NB}
  \left\{
  \begin{aligned}
   &-\nu\Delta w+ik(V(y,z)-\lambda)w-a(\nu k^2)^{1/3}w=F,\\
   &\Delta \varphi=w,\quad \varphi|_{y=\pm1}=0, \\
   &\partial_xw=ikw,\quad \partial_xF=ikF.
   \end{aligned}\right.
\end{align}
Our goal  of this section is to control $w$ via the Neumann boundary data $\pa_y\varphi|_{y=\pm 1}$ and $F$.\smallskip

We assume that $\la\in \R, a\in [0,\eps_1]$ and $V$ satisfies \eqref{ass:V}. In section 6.1 and section 6.2, we also assume that
$L=\big(|k|/\nu\big)^\f13\ge k_0$ with $k_0$ given by Proposition \ref{prop:w12-bounds}.

\subsection{Resolvent estimates when $V=y$ and $F=0$}

\begin{Proposition}\label{prop:res-NB-y}
Let $w\in H^2(\Omega)$ be a solution of \eqref{eq:OS-NB} with $V=y$ and $F=0$. Then it holds that
\begin{align*}
&\|w\|_{L^2} \leq C\Big(\|(|k(y-\lambda)/\nu|^{\f14}+|k/\nu|^{\f16}+|k|^{\f12})\partial_y\varphi\|_{L^2(\partial\Omega)}
+(|k/\nu|^{\f13}+|k|)^{-\f12}\|\partial_y\partial_z\varphi\|_{L^2(\partial\Omega)}\Big),
\\&\|w\|_{L^2(\partial\Omega)} \leq C\Big(\|(|k(y-\lambda)/\nu|^{\f12}+|k/\nu|^{\f13}+|k|)\partial_y\varphi\|_{L^2(\partial\Omega)}
+\|\partial_y\partial_z\varphi\|_{L^2(\partial\Omega)}\Big),\\ &\|\nabla \varphi\|_{L^2}+\|(1-y^2)w\|_{L^2}\leq C|\nu/k|^{\f16}\|\partial_y\varphi\|_{L^2(\partial\Omega)},\\
&|k|^{\f12}\| \varphi\|_{L^2}\leq C|\nu/k|^{\f13}\|\partial_y\varphi\|_{L^2(\partial\Omega)},\\
&\|(1-y^2)^2w\|_{L^2}\leq C|\nu/k|^{\f12}\|\partial_y\varphi\|_{L^2(\partial\Omega)}.
\end{align*}
\end{Proposition}

\begin{proof}
  Let $\hat{w}_{\ell}=\int_{\mathbb{T}}e^{-i\ell z}w(x,y,z)dz$, which satisfies
  \begin{align*}
  \left\{
  \begin{aligned}
   &-\nu(\partial_y^2-\eta^2) \hat{w}_{\ell}+ik(y-\lambda)\hat{w}_{\ell}-a(\nu k^2)^{1/3}\hat{w}_{\ell}=0,\\
   &\hat{w}_{\ell}=(\partial^2_y-\eta^2)\hat{\varphi}_{\ell},\ \hat{\varphi}_{\ell}|_{y=\pm1}=0,\ \partial_x\hat{w}_{\ell}=ik\hat{w}_{\ell}.
   \end{aligned}\right.
\end{align*}
Then we have
\begin{align*}
   & \hat{w}_{\ell}=\partial_y\hat{\varphi}_{\ell}(1)w_{`1,\ell} +\partial_y\hat{\varphi}_{\ell}(-1)w_{2,\ell},
\end{align*}
where $w_{1,\ell}$ and $w_{2,\ell}$ are given by \eqref{eq:w1l} and \eqref{eq:w2l}.

By Plancherel's theorem,  we get
\begin{align*}
   &\|w\|_{L^2}^2=C\sum_{\ell\in\mathbb{Z}}\|\hat{w}_{\ell}\|_{L^2}^2 \leq C\sum_{\ell\in\mathbb{Z}}\Big(|\partial_y\hat{\varphi}_{\ell}(1)|^2\|w_{1,\ell}\|_{L^2}^2 +|\partial_y\hat{\varphi}_{\ell}(-1)|^2\|w_{2,\ell}\|_{L^2}^2\Big).
\end{align*}
By Proposition \ref{prop:w12-L2},  we have
\begin{align*}
   &\sum_{\ell\in\mathbb{Z}}|\partial_y\hat{\varphi}_{\ell}(1)|^2\|w_{1,\ell}\|_{L^2}^2
   \leq C \sum_{\ell\in\mathbb{Z}}\big( |k(\lambda-1)/\nu|^{\f12}+(|k|/\nu)^{\f13}+|k|+l^2(|k/\nu|^{\f13}+|k|)^{-1}\big) |\partial_y\hat{\varphi}_{\ell}(1)|^2\\
 &\quad=C\big(|k(\lambda-1)/\nu|^{\f12}+(|k|/\nu)^{\f13}+|k|\big) \|\partial_y\varphi\|^2_{L^2(\Gamma_{1})}+C(|k/\nu|^{\f13}+|k|)^{-1} \|\partial_z\partial_y\varphi\|_{L^2(\Gamma_1)}^2.
\end{align*}
Similarly we have
\begin{align*}
   &\sum_{\ell\in\mathbb{Z}}|\partial_y\hat{\varphi}_{\ell}(-1)|^2\|w_{2,\ell}\|_{L^2}^2\\ &\leq C\big(|k(\lambda+1)/\nu|^{\f12}+(|k|/\nu)^{\f13}+|k|\big) \|\partial_y\varphi\|^2_{L^2(\Gamma_{-1})}+C(|k/\nu|^{\f13}+|k|)^{-1} \|\partial_z\partial_y\varphi\|_{L^2(\Gamma_{-1})}^2.
\end{align*}
Thus, we conclude that
\begin{align*}
   &\|w\|_{L^2}^2\leq C\sum_{\ell\in\mathbb{Z}}\big(|\partial_y\hat{\varphi}_{\ell}(1)|^2\|w_{1,\ell}\|_{L^2}^2 +|\partial_y\hat{\varphi}_{\ell}(-1)|^2\|w_{2,\ell}\|_{L^2}^2\big)\\
   &\leq C\Big(\sum_{j\in\{-1,1\}}\big( |k(\lambda-j)/\nu|^{\f12}+(|k|/\nu)^{\f13}+|k|\big) \|\partial_y\varphi\|^2_{L^2(\Gamma_{j})}+(|k/\nu|^{\f13}+|k|)^{-1} \|\partial_z\partial_y\varphi\|_{L^2(\Gamma_{j})}^2\Big) \\
   &=C\|(|k(y-\lambda)/\nu|^{\f14}+|k/\nu|^{\f16}+|k|^{\f12})\partial_y\varphi\|_{L^2(\partial\Omega)}^2+C(|k/\nu|^{\f13}+|k|)^{-1} \|\partial_z\partial_y\varphi_l\|_{L^2(\partial\Omega)}^2 .
\end{align*}
This shows the first inequality.

By Plancherel's theorem again,  we have
\begin{align*}
   \|w\|_{L^2(\partial\Omega)}^2\leq& \|w\|^2_{L^\infty_yL^2_{z}}=C\|\sum_{\ell\in\mathbb{Z}}|\hat{w}_\ell|^2\|_{L^\infty_y}\leq C\sum_{\ell\in\mathbb{Z}}\|\hat{w}_{\ell}\|_{L^\infty}^2\\
   \leq &C\sum_{\ell\in\mathbb{Z}}\Big(|\partial_y\hat{\varphi}_{\ell}(1)|^2\|w_{1,\ell}\|_{L^\infty}^2 +|\partial_y\hat{\varphi}_{\ell}(-1)|^2\|w_{2,\ell}\|_{L^\infty}^2\Big).
\end{align*}
By Proposition \ref{prop:w12-bounds}, we have
\begin{align*}
   &\sum_{\ell\in\mathbb{Z}}|\partial_y\hat{\varphi}_{\ell}(1)|^2\|w_{1,\ell}\|_{L^\infty}^2
   \leq C \sum_{\ell\in\mathbb{Z}}\big(|k(\lambda-1)/\nu|+(|k|/\nu)^{\f23}+k^2+\ell^2\big) |\partial_y\hat{\varphi}_{\ell}(1)|^2\\
   &= C\big( |k(\lambda-1)/\nu|+(|k|/\nu)^{\f23}+k^2\big) \|\partial_y\varphi\|^2_{L^2(\Gamma_{1})} +C\|\partial_z\partial_y\varphi\|_{L^2(\Gamma_1)}^2,
\end{align*}
and
\begin{align*}
   &\sum_{\ell\in\mathbb{Z}}|\partial_y\hat{\varphi}_{\ell}(-1)|^2\|w_{2,\ell}\|_{L^\infty}^2\\ &\leq C\big( |k(\lambda+1)/\nu|+(|k|/\nu)^{\f23}+k^2\big) \|\partial_y\varphi\|^2_{L^2(\Gamma_{-1})} +C\|\partial_z\partial_y\varphi\|_{L^2(\Gamma_{-1})}^2,
\end{align*}
which imply that
\begin{align*}
   &\|w\|_{L^2(\partial\Omega)}^2\leq C\sum_{\ell\in\mathbb{Z}}\Big(|\partial_y\hat{\varphi}_{\ell}(1)|^2\|w_{1,\ell}\|_{L^\infty}^2 +|\partial_y\hat{\varphi}_{\ell}(-1)|^2\|w_{2,\ell}\|_{L^\infty}^2\Big)\\
   &\leq C\sum_{j\in\{-1,1\}}\big( |k(\lambda-j)/\nu|+(|k|/\nu)^{\f23}+k^2\big) \|\partial_y\varphi\|^2_{L^2(\Gamma_{j})}+C\sum_{j\in\{-1,1\}} \|\partial_z\partial_y\varphi\|_{L^2(\Gamma_{j})}^2 \\
   &=C \|(|k(y-\lambda)/\nu|^{\f12}+|k/\nu|^{\f13}+|k|)\partial_y\varphi\|_{L^2(\partial\Omega)}^2+C \|\partial_z\partial_y\varphi\|_{L^2(\partial\Omega)}^2,
\end{align*}
which gives  the second inequality.

By Plancherel's theorem and Proposition \ref{prop:w12-L2}, we get
\begin{align*}
\|\nabla\varphi\|_{L^2}^2=&C\sum_{\ell\in\mathbb{Z}}\|(\partial_y,\eta)\hat{\varphi}_{\ell}\|_{L^2}^2 \\
\leq& C\sum_{\ell\in\mathbb{Z}}\Big(|\partial_y\hat{\varphi}_{\ell}(1)|^2\|(\partial_y,\eta)\varphi_{1,\ell}\|_{L^2}^2 +|\partial_y\hat{\varphi}_{\ell}(-1)|^2\|(\partial_y,\eta)\varphi_{2,\ell}\|_{L^2}^2\Big)\\ \leq& C\sum_{\ell\in\mathbb{Z}}\Big(|\partial_y\hat{\varphi}_{\ell}(1)|^2|\nu/k|^{\f13} +|\partial_y\hat{\varphi}_{\ell}(-1)|^2|\nu/k|^{\f13}\Big)=C(\nu/|k|)^{\f13} \|\partial_y\varphi\|^2_{L^2(\partial\Omega)},
\end{align*}
and
\begin{align*}
|k|\|\varphi\|_{L^2}^2=&C|k|\sum_{\ell\in\mathbb{Z}}\|\hat{\varphi}_{\ell}\|_{L^2}^2 \\ \leq& C\sum_{\ell\in\mathbb{Z}}|k|\Big(|\partial_y\hat{\varphi}_{\ell}(1)|^2\|\varphi_{1,\ell}\|_{L^2}^2 +|\partial_y\hat{\varphi}_{\ell}(-1)|^2\|\varphi_{2,\ell}\|_{L^2}^2\Big)\\ \leq& C\sum_{\ell\in\mathbb{Z}}\Big(|\partial_y\hat{\varphi}_{\ell}(1)|^2|\nu/k|^{\f23} +|\partial_y\hat{\varphi}_{\ell}(-1)|^2|\nu/k|^{\f23}\Big)=C(\nu/|k|)^{\f23} \|\partial_y\varphi\|^2_{L^2(\partial\Omega)}.
\end{align*}

For $ \beta\in\{1,2\}$, by Plancherel's theorem and Proposition \ref{prop:w12-L2}, we get
\begin{align*}
  \|(1-y^2)^{\beta}w\|_{L^2}^2=&C\sum_{\ell \in\mathbb{Z}}\|(1-y^2)^{\beta}\hat w_\ell\|_{L^2}^2 \\ \leq& C\sum_{\ell\in\mathbb{Z}}\Big(|\partial_y\hat{\varphi}_{\ell}(1)|^2\|(1-y^2)^{\beta}w_{1,\ell}\|_{L^2}^2 +|\partial_y\hat{\varphi}_{\ell}(-1)|^2\|(1-y^2)^{\beta}w_{2,\ell}\|_{L^2}^2\Big)\\ \leq& C\sum_{\ell \in\mathbb{Z}}\Big(|\partial_y\hat{\varphi}_{\ell}(1)|^2L^{1-2\beta} +|\partial_y\hat{\varphi}_{\ell}(-1)|^2L^{1-2\beta}\Big)=CL^{1-2\beta} \|\partial_y\varphi\|^2_{L^2(\partial\Omega)},
\end{align*}
which gives
\begin{align*}
 &\|(1-y^2)w\|_{L^2}^2\leq CL^{-1} \|\partial_y\varphi\|^2_{L^2(\partial\Omega)}=C|\nu/k|^{\f13} \|\partial_y\varphi\|^2_{L^2(\partial\Omega)},\\
 &\|(1-y^2)^{2}w\|_{L^2}^2\leq CL^{-3} \|\partial_y\varphi\|^2_{L^2(\partial\Omega)}=C|\nu/k| \|\partial_y\varphi\|^2_{L^2(\partial\Omega)}.
\end{align*}
\end{proof}

\subsection{Resolvent estimates with general $V$ and $F=0$}

The following proposition gives the resolvent estimate of \eqref{eq:OS-NB} in the homogeneous case(i.e., $F=0$).
The proof is based on Proposition \ref{prop:res-NB-y} and the perturbation argument.

\begin{Proposition}\label{prop:res-NB-hom}
Let $w\in H^2(\Omega)$ be a solution of \eqref{eq:OS-NB} with $F=0$. Then it holds that
\begin{align*}
&\|w\|_{L^2} \leq C\Big(\|(|k(y-\lambda)/\nu|^{\f14}+|k/\nu|^{\f16}+|k|^{\f12})\partial_y\varphi\|_{L^2(\partial\Omega)}
+(|k/\nu|^{\f13}+|k|)^{-\f12}\|\partial_y\partial_z\varphi\|_{L^2(\partial\Omega)}\Big),
\\&\|w\|_{L^2(\partial\Omega)} \leq C\Big(\|(|k(y-\lambda)/\nu|^{\f12}+|k/\nu|^{\f13}+|k|)\partial_y\varphi\|_{L^2(\partial\Omega)}
+\|\partial_y\partial_z\varphi\|_{L^2(\partial\Omega)}\Big),\\ &\|\nabla \varphi\|_{L^2}+\|(1-y^2)w\|_{L^2}\leq C|\nu/k|^{\f16}\|\partial_y\varphi\|_{L^2(\partial\Omega)},\\
&|k|^{\f12}\| \varphi\|_{L^2}\leq C|\nu/k|^{\f13}\|\partial_y\varphi\|_{L^2(\partial\Omega)}.
\end{align*}
\end{Proposition}

\begin{proof}
We decompose $w=w_{Na}+w_I$, where $w_{Na}$ and $w_I$ solve
\begin{align*}
  \left\{
  \begin{aligned}
   &-\nu\Delta w_{Na}+ik(y-\lambda)w_{Na}-a(\nu k^2)^{1/3}w_{Na}=-ik(V(y,z)-y)w,\\
   &w_{Na}=\Delta\varphi_{Na},\ \varphi_{Na}|_{y=\pm1}=0,\ w_{Na}|_{y=\pm1}=0,\ \partial_xw_{Na}=ikw_{Na}.\\
   &-\nu\Delta w_{I}+ik(y-\lambda)w_{I}-a(\nu k^2)^{1/3}w_{I}=0,\\
   &w_{I}=\Delta\varphi_I,\ \varphi_I|_{y=\pm1}=0,\ \partial_xw_{I}=ikw_{I}.
   \end{aligned}\right.
\end{align*}Then we have\begin{align*}
   -\nu\Delta w_{Na}+ik(V(y,z)-\lambda)w_{Na}-a(\nu k^2)^{1/3}w_{Na}=-ik(V(y,z)-y)w_I.
\end{align*}
By Lemma \ref{lem:V}, we have
\begin{align}\label{eq: V-y}
   &\|(1-y^2)(V-y)w_I\|_{L^2}\leq C\varepsilon_0\|(1-y^2)^2w_I\|_{L^2},\\ \label{eq: V-y1}&\|(V-y)w_I\|_{L^2}\leq C\varepsilon_0\|(1-y^2)w_I\|_{L^2}.
\end{align}

By Proposition \ref{prop:res-nav-s1} and \eqref{eq: V-y1},  we have
\begin{align}\label{42}
   &\|w_{Na}\|_{L^2}\leq C\varepsilon_0|\nu/k|^{-\f13}\|(1-y^2)w_I\|_{L^2}.
\end{align}

Using Proposition \ref{prop:res-nav-b1} with $F=-ik(V-y)w$, we get by  by \eqref{eq: V-y} and \eqref{eq: V-y1} that
\begin{align}
   &\nu^{\f16}|k|^{\f43}\|\varphi_{Na}\|_{L^2}+ (\nu k^2)^{\f13}(\|\nabla \varphi_{Na}\|_{L^2}+\|(1-y^2)w_{Na}\|_{L^2})\label{31} \\
   &\quad \leq C\varepsilon_0\big(\|k(1-y^2)^2w_I\|_{L^2} +(\nu k^2)^{\f13}\|(1-y^2)w_I\|_{L^2}\big),\nonumber\\
   &\left\||k(y-\lambda)|^{\f12}\partial_y\varphi_{Na}\right\|_{L^2(\partial\Omega)} \label{33}\\ &\quad\leq C\varepsilon_0\nu^{-\f16}|k|^{\f16}\big(\min(|k(\lambda-1)|^{\f12} ,|k(\lambda+1)|^{\f12}) +\nu^{\f16}|k|^{\f13}\big)\|(1-y^2)w_I\|_{L^2},\nonumber
\end{align}
and
\begin{align}
   &\|\partial_y\partial_z\varphi_{Na}\|_{L^2(\partial\Omega)} \leq C\varepsilon_0\nu^{-\f12}|k|^{\f12}\|(1-y^2)w_I\|_{L^2},\label{32}\\
   &\|\partial_y\varphi_{Na}\|_{L^2(\partial\Omega)} \leq C\varepsilon_0 \nu^{-\f16}|k|^{\f16} \|(1-y^2)w_I\|_{L^2}.\label{40}
   \end{align}

For $w_I$, we get by Proposition \ref{prop:res-NB-y}  that
\begin{align}
\label{34}&\|w_I\|_{L^2} \leq C\Big(\|(|k(y-\lambda)/\nu|^{\f14}+|k/\nu|^{\f16}+|k|^{\f12})\partial_y\varphi_I\|_{L^2(\partial\Omega)}
\\ &\qquad\quad+(|k/\nu|^{\f13}+|k|)^{-\f12}\|\partial_y\partial_z\varphi_I\|_{L^2(\partial\Omega)}\Big),
\nonumber\\ \label{35}&\|w_I\|_{L^2(\partial\Omega)} \leq C\Big(\|(|k(y-\lambda)/\nu|^{\f12}+|k/\nu|^{\f13}+|k|)\partial_y\varphi_I\|_{L^2(\partial\Omega)}
+\|\partial_y\partial_z\varphi_I\|_{L^2(\partial\Omega)}\Big),\\ \label{36}&\|\nabla \varphi_I\|_{L^2}+\|(1-y^2)w_I\|_{L^2}\leq C|\nu/k|^{\f16}\|\partial_y\varphi_I\|_{L^2(\partial\Omega)},\\
\label{37}&|k|^{\f12}\| \varphi_I\|_{L^2}\leq C|\nu/k|^{\f13}\|\partial_y\varphi_I\|_{L^2(\partial\Omega)},\\
\label{41}&\|(1-y^2)^2w_I\|_{L^2}\leq C|\nu/k|^{\f12}\|\partial_y\varphi_I\|_{L^2(\partial\Omega)}.
\end{align}

By \eqref{36} and \eqref{40}, we get
\begin{align*}
   \|(1-y^2)w_I\|_{L^2}&\leq C|\nu/k|^{\f16}\|\partial_y\varphi_I\|_{L^2(\partial\Omega)}\leq C|\nu/k|^{\f16}\big(\|\partial_y\varphi\|_{L^2(\partial\Omega)}+ \|\partial_y\varphi_{Na}\|_{L^2(\partial\Omega)}\big)\\
   &\leq C|\nu/k|^{\f16}\|\partial_y\varphi\|_{L^2(\partial\Omega)} +C\varepsilon_0\|(1-y^2)w_I\|_{L^2}.
\end{align*}
Taking $C\varepsilon_0\leq1/2$, we conclude that
\begin{align}\label{eq: (1-y^2)w_I}
  & \|(1-y^2)w_I\|_{L^2} \leq C|\nu/k|^{\f16}\|\partial_y\varphi\|_{L^2(\partial\Omega)},
\end{align}
which along with \eqref{40} gives
\begin{align}\label{eq: partial varphi I}
   & \|\partial_y\varphi_I\|_{L^2(\partial\Omega)}\leq \|\partial_y\varphi\|_{L^2}+C\varepsilon_0|\nu/k|^{-\f16}\|(1-y^2)w_I\|_{L^2}\leq  C\|\partial_y\varphi\|_{L^2(\partial\Omega)}.
\end{align}
By  \eqref{33}, \eqref{40} and \eqref{eq: (1-y^2)w_I}, we have
\begin{align*}
   & \big\||k(y-\lambda)|^{\f12}\partial_y\varphi_{Na}\big\|_{L^2(\partial\Omega)}
   \leq C\varepsilon_0\big(\min(|k(\lambda-1)|^{\f12},|k(\lambda+1)|^{\f12})+\nu^{\f16}|k|^{\f13}\big) \|\partial_y\varphi\|_{L^2(\partial\Omega)},\\
  & \|\partial_y\varphi_{Na}\|_{L^2(\partial\Omega)}\leq C\varepsilon_0\nu^{-\f16}|k|^{\f16}\|(1-y^2)w_I\|_{L^2} \leq C\varepsilon_0\|\partial_y\varphi\|_{L^2(\partial\Omega)}.
\end{align*}
Using the interpolation,  we deduce that  or $\gamma\in[0,1]$,
\begin{align*}
   \||k(y-\lambda)|^{\f{\gamma}{2}}\partial_y\varphi_{Na}\|_{L^2(\partial\Omega)}\leq& \||k(y-\lambda)|^{\f12}\partial_y\varphi_{Na}\|_{L^2(\partial\Omega)}^{\gamma}\|\partial_y\varphi_{Na}\|_{L^2(\partial\Omega)}^{1-\gamma} \\   \leq &C\varepsilon_0\big(\min(|k(\lambda-1)|^{\f{\gamma}2},|k(\lambda+1)|^{\f{\gamma}2})+\nu^{\f{\gamma}6}|k|^{\f{\gamma}3}\big)\|\partial_y\varphi\|_{L^2(\partial\Omega)}\\
  \leq &C\varepsilon_0\Big(\||k(y-\lambda)|^{\f{\gamma}2}\partial_y\varphi\|_{L^2(\partial\Omega)}+\nu^{\f{\gamma}6}|k|^{\f{\gamma}3} \|\partial_y\varphi\|_{L^2(\partial\Omega)}\Big).
\end{align*}
Here we used the fact that  for $\al\geq0$,
\beno
\min(|k(\lambda-1)|^{\alpha},|k(\lambda+1)|^{\alpha})\|g\|_{L^2(\partial\Omega)} \leq \left\||k(y-\lambda)|^{\alpha}g\right\|_{L^2(\partial\Omega)}.
\eeno
This shows that for $\gamma\in[0,1]$,
\begin{align}\label{eq: weight partial varphi I}
   \big\||k(y-\lambda)|^{\f{\gamma}2}\partial_y\varphi_I\big\|_{L^2(\partial\Omega)} &\leq \big\||k(y-\lambda)|^{\f{\gamma}2}\partial_y\varphi_{Na}\big\|_{L^2(\partial\Omega)}+
   \big\||k(y-\lambda)|^{\f{\gamma}2}\partial_y\varphi\big\|_{L^2(\partial\Omega)}\\ \nonumber&\leq C\nu^{\f{\gamma}6}|k|^{\f{\gamma}3}\|\partial_y\varphi\|_{L^2(\partial\Omega)} +C\big\||k(y-\lambda)|^{\f{\gamma}2}\partial_y\varphi\big\|_{L^2(\partial\Omega)},
\end{align}
which along with \eqref{eq: partial varphi I} gives
\begin{align}
   &\|(|k(y-\lambda)/\nu|^{\f12}+ |k/\nu|^{\f13}+|k|) \partial_y\varphi_I\|_{L^2(\partial\Omega)} \label{38}\\
     &\qquad\leq C\|(|k(y-\lambda)/\nu|^{\f12}+ |k/\nu|^{\f13}+|k|) \partial_y\varphi\|_{L^2(\partial\Omega)},\nonumber
 \end{align}
and
\begin{align}
 &\|(|k(y-\lambda)/\nu|^{\f14}+ |k/\nu|^{\f16}+|k|^{\f12}) \partial_y\varphi_I\|_{L^2(\partial\Omega)} \label{39}\\
     &\qquad\leq C\|(|k(y-\lambda)/\nu|^{\f14}+ |k/\nu|^{\f16}+|k|^{\f12}) \partial_y\varphi\|_{L^2(\partial\Omega)}.\nonumber
\end{align}

By \eqref{42}, \eqref{34}, \eqref{39}, \eqref{32} and \eqref{eq: (1-y^2)w_I}, we have
\begin{align*}
   \|w\|_{L^2}\leq& \|w_{Na}\|_{L^2}+\|w_{I}\|_{L^2}\leq C|\nu/k|^{-\f13}\|(1-y^2)w_I\|_{L^2}\\ +& C\Big(\|(|k(y-\lambda)/\nu|^{\f14}+|k/\nu|^{\f16}+|k|^{\f12}) \partial_y\varphi_I\|_{L^2(\partial\Omega)} +(|k/\nu|^{\f13}+|k|)^{-\f12}\|\partial_y\partial_z\varphi_I\|_{L^2(\partial\Omega)}\Big)\\
   \leq &C|\nu/k|^{-\f16}\|\partial_y\varphi\|_{L^2(\partial\Omega)} +C\Big(\|(|k(y-\lambda)/\nu|^{\f14}+|k/\nu|^{\f16}+|k|^{\f12}) \partial_y\varphi\|_{L^2(\partial\Omega)}\\
   &+(|k/\nu|^{\f13}+|k|)^{-\f12}\|\partial_y\partial_z(\varphi-\varphi_{Na})\|_{L^2(\partial\Omega)}\Big)\\
   \leq &C|\nu/k|^{-\f16}\|\partial_y\varphi\|_{L^2(\partial\Omega)} +C(\|(|k(y-\lambda)/\nu|^{\f14}+|k/\nu|^{\f16}+|k|^{\f12}) \partial_y\varphi\|_{L^2(\partial\Omega)}\\
   &+(|k/\nu|^{\f13}+|k|)^{-\f12}\|\partial_y\partial_z\varphi\|_{L^2(\partial\Omega)}) +C(|k/\nu|^{\f13}+|k|)^{-\f12}\nu^{-\f12}|k|^{\f12}\|(1-y^2)w_{I}\|_{L^2}\\
   \leq &C\Big(\|(|k(y-\lambda)/\nu|^{\f14}+|k/\nu|^{\f16}+|k|^{\f12}) \partial_y\varphi\|_{L^2(\partial\Omega)}+(|k/\nu|^{\f13}+|k|)^{-\f12}\|\partial_y\partial_z\varphi\|_{L^2(\partial\Omega)}\Big),
\end{align*}
which gives the first inequality.

Due to $w_{Na}|_{y=\pm1}=0$, we get by  \eqref{35}, \eqref{38},  \eqref{32} and \eqref{eq: (1-y^2)w_I} , we have
\begin{align*}
   \|w\|_{L^2(\partial\Omega)}=& \|w_{I}\|_{L^2(\partial\Omega)}\leq  C\Big(\|(|k(y-\lambda)/\nu|^{\f12}+|k/\nu|^{\f13}+|k|) \partial_y\varphi_I\|_{L^2(\partial\Omega)} +\|\partial_y\partial_z\varphi_I\|_{L^2(\partial\Omega)}\Big)\\
   \leq &C\Big(\|(|k(y-\lambda)/\nu|^{\f12}+|k/\nu|^{\f13}+|k|) \partial_y\varphi\|_{L^2(\partial\Omega)}+\|\partial_y\partial_z(\varphi-\varphi_{Na})\|_{L^2(\partial\Omega)}\Big)\\
   \leq &C\Big(\|(|k(y-\lambda)/\nu|^{\f12}+|k/\nu|^{\f13}+|k|) \partial_y\varphi\|_{L^2(\partial\Omega)}+\|\partial_y\partial_z\varphi\|_{L^2(\partial\Omega)}\Big) \\&+\nu^{-\f12}|k|^{\f12}\|(1-y^2)w_{I}\|_{L^2}\\
   \leq &C\Big(\|(|k(y-\lambda)/\nu|^{\f12}+|k/\nu|^{\f13}+|k|) \partial_y\varphi\|_{L^2(\partial\Omega)} +\|\partial_y\partial_z\varphi\|_{L^2(\partial\Omega)}\Big),
\end{align*}
which gives the second inequality.

It follows from \eqref{41} and \eqref{eq: partial varphi I}  that
\begin{align}\label{43}
  \|(1-y^2)^2w_I\|_{L^2}\leq& C|\nu/k|^{\f12}\|\partial_y\varphi_I\|_{L^2(\partial\Omega)}\leq C|\nu/k|^{\f12}\|\partial_y\varphi\|_{L^2(\partial\Omega)}.
\end{align}
By \eqref{31}, \eqref{36}, \eqref{eq: partial varphi I}, \eqref{eq: (1-y^2)w_I}  and \eqref{43}, we have
\begin{align*}
 & \|\nabla\varphi\|_{L^2}+\|(1-y^2)w\|_{L^2} \\
 & \leq \|\nabla\varphi_{Na}\|_{L^2}+\|(1-y^2)w_{Na}\|_{L^2} +\|\nabla\varphi_I\|_{L^2}+\|(1-y^2)w_I\|_{L^2}\\
  &\leq C\Big(|\nu/k|^{-\f13}\|(1-y^2)^2w_I\|_{L^2}+ \|(1-y^2)w_I\|_{L^2}\Big)+ C|\nu/k|^{\f16}\|\partial_y\varphi_I\|_{L^2(\partial\Omega)}\\
  &\leq C\Big(|\nu/k|^{-\f13}\|(1-y^2)^2w_I\|_{L^2}+ \|(1-y^2)w_I\|_{L^2}\Big)+ C|\nu/k|^{\f16}\|\partial_y\varphi\|_{L^2(\partial\Omega)}\\
  &\leq C|\nu/k|^{\f16}\|\partial_y\varphi\|_{L^2(\partial\Omega)},
\end{align*}
which gives the third inequality.

By \eqref{31}, \eqref{37}, \eqref{eq: partial varphi I}, \eqref{eq: (1-y^2)w_I}and \eqref{43}, we get
\begin{align*}
  |k|^{\f12}\|\varphi\|_{L^2} \leq& |k|^{\f12}\|\varphi_{Na}\|_{L^2}+|k|^{\f12}\|\varphi_{I}\|_{L^2}\\
  \leq& C\Big(|\nu/k|^{-\f16}\|(1-y^2)^2w_{I}\|_{L^2}+ |\nu/k|^{\f16}\|(1-y^2)w_I\|_{L^2}\Big)+ C|\nu/k|^{\f13}\|\partial_y\varphi_I\|_{L^2(\partial\Omega)}\\
  \leq& C\Big(|\nu/k|^{-\f16}\|(1-y^2)^2w_{I}\|_{L^2}+ |\nu/k|^{\f16}\|(1-y^2)w_I\|_{L^2}\Big)+ C|\nu/k|^{\f13}\|\partial_y\varphi\|_{L^2(\partial\Omega)}\\
  \leq& C|\nu/k|^{\f13}\|\partial_y\varphi\|_{L^2(\partial\Omega)},
\end{align*}
which gives  the fourth inequality.
\end{proof}\smallskip

\subsection{Resolvent estimates with general $V$ and $F$}

The following proposition gives the resolvent estimates for the inhomogeneous equation.

\begin{Proposition}\label{prop:res-NB-ihom}
Let $\nu k^2\le 1$, and $w\in H^2(\Omega)$ be a solution of \eqref{eq:OS-NB} with $F\in L^2(\Omega)$ and $F=F_1+F_2$.
Then it holds that
\begin{align}
\label{eq:res-ihom1}&\nu^{\f16}|k|^{\f43}\|\varphi\|_{L^2} \leq C\Big(\|(1-y^2)F_1\|_{L^2}+|\nu/k|^{\f13}\|F_1\|_{L^2}\\& \qquad\qquad+|\nu/k|^{-\f13}\max(1-|\lambda|, \nu^{\f13}|k|^{-\f13})\|{F}_2\|_{H^{-1}}+|\nu k|^{\f12}\|\partial_y\varphi\|_{L^2(\partial\Omega)}\Big),\nonumber\\ &\label{eq:res-ihom2}\nu^{\f16}|k|^{\f43}\|\nabla\varphi\|_{L^2} \leq C\Big(\|F_1\|_{L^2}+|\nu/k|^{-\f13}\|{F}_2\|_{H^{-1}}+\nu^{\f13}|k|^{\f76}\|\partial_y\varphi\|_{L^2(\partial\Omega)}\Big),\\
&\label{eq:res-ihom3}\|w\|_{L^2} \leq C\Big(\|(|k(y-\lambda)/\nu|^{\f14}+|k/\nu|^{\f16})\partial_y\varphi\|_{L^2(\partial\Omega)}
+|\nu/k|^{\f16}\|\partial_y\partial_z\varphi\|_{L^2(\partial\Omega)}\\& \qquad\qquad+(\nu k^2)^{-\f{5}{12}}\|F_1\|_{L^2}+\nu^{-\f34} |k|^{-\f{1}{2}}\|{F}_2\|_{H^{-1}}\Big),\nonumber
\\&\label{eq:res-ihom4}\|w\|_{L^2(\partial\Omega)} \leq C\Big(\|(|k(y-\lambda)/\nu|^{\f12}+|k/\nu|^{\f13})\partial_y\varphi\|_{L^2(\partial\Omega)}
+\|\partial_y\partial_z\varphi\|_{L^2(\partial\Omega)}\\& \qquad\qquad+\nu^{-\f23} |k|^{-\f{5}{6}}\|F_1\|_{L^2}+\nu^{-1} |k|^{-\f{1}{2}}\|{F}_2\|_{H^{-1}}\Big).\nonumber
\end{align}
\end{Proposition}

\begin{proof}
  We decompose the solution $w$ as $w=w_{Na}^{(1)}+w_{Na}^{(2)}+w_I$, where $w_{Na}^{(j)}(j=1,2), w_I$ solve
  \begin{align*}
     \left\{\begin{aligned}
     &-\nu\Delta w_{Na}^{(j)}+ik(V(y,z)-\lambda)w_{Na}^{(j)}-a(\nu k^2)^{1/3}w_{Na}^{(j)}=F_j,\\
     &w_{Na}^{(j)}=\Delta\varphi_{Na}^{(j)},\ \varphi_{Na}^{(j)}|_{y=\pm1}=0,\ w_{Na}^{(j)}|_{y=\pm1}=0,\ \partial_xw_{Na}^{(j)}=ik\partial_xw_{Na}^{(j)},
     \end{aligned}\right.
  \end{align*}
  and
  \begin{align*}
     \left\{\begin{aligned}
     &-\nu\Delta w_I+ik(V(y,z)-\lambda)w_I-a(\nu k^2)^{1/3}w_I=0,\\
     &w_I=\Delta\varphi_I,\ \varphi_I|_{y=\pm1}=0,\ \partial_xw_I=ik\partial_xw_I.
     \end{aligned}\right.
  \end{align*}

 {\bf Step 1.} Proof of \eqref{eq:res-ihom1} and \eqref{eq:res-ihom2}.\smallskip

  By Proposition  \ref{prop:res-nav-b1}, we have
  \begin{align*}
     &|\nu k|^{\f12}\|\partial_y\varphi_{Na}^{(1)}\|_{L^2(\partial\Omega)}\leq C|\nu/k|^{\f13}\|F_1\|_{L^2},\\ &\nu^{\f13}|k|^{\f76}\|\partial_y\varphi_{Na}^{(1)}\|_{L^2(\partial\Omega)} \leq C(\nu k^2)^{\f16}\|F_1\|_{L^2}\leq C\|F_1\|_{L^2},
  \end{align*}
  and by Proposition \ref{prop:res-weak-b}, we have
  \begin{align*}
     &|\nu k|^{\f12}\|\partial_y\varphi_{Na}^{(2)}\|_{L^2(\partial\Omega)}\leq C\|F_2\|_{H^{-1}}\leq C|\nu/k|^{-\f13}\max(1-|\lambda|,\nu^{\f13}|k|^{-\f13})\|F_2\|_{H^{-1}},\\ &\nu^{\f13}|k|^{\f76}\|\partial_y\varphi_{Na}^{(2)}\|_{L^2(\partial\Omega)} \leq C\nu^{-\f16}|k|^{\f23}\|F_2\|_{H^{-1}}\leq C|\nu/k|^{-\f13}\|F_2\|_{H^{-1}}.
  \end{align*}
   By Proposition \ref{prop:res-NB-hom}, we have
  \begin{align}\label{16}
     \nu^{\f16}|k|^{\f43}\|\varphi_I\|_{L^2}\leq& C|\nu k|^{\f12}\|\partial_y\varphi_I\|_{L^2(\partial\Omega)}\nonumber\\
     \leq& C|\nu k|^{\f12}\big(\|\partial_y\varphi\|_{L^2(\partial\Omega)} +\|\partial_y\varphi_{Na}^{(1)}\|_{L^2(\partial\Omega)} +\|\partial_y\varphi_{Na}^{(2)}\|_{L^2(\partial\Omega)})\nonumber\\
     \leq &C|\nu k|^{\f12}\|\partial_y\varphi\|_{L^2(\partial\Omega)} +C|\nu/k|^{\f13}\|F_1\|_{L^2}+C|\nu/k|^{-\f13} \max(1-|\lambda|,\nu^{\f13}|k|^{-\f13})\|F_2\|_{H^{-1}} ,
  \end{align}
  and
  \begin{align}\label{17}
     \nu^{\f16}|k|^{\f43}\|\nabla\varphi_I\|_{L^2}\leq& C\nu^{\f13} |k|^{\f76}\|\partial_y\varphi_I\|_{L^2(\partial\Omega)}\nonumber\\
     \leq& C\nu^{\f13}|k|^{\f76}\big(\|\partial_y\varphi\|_{L^2(\partial\Omega)} +\|\partial_y\varphi_{Na}^{(1)}\|_{L^2(\partial\Omega)} +\|\partial_y\varphi_{Na}^{(2)}\|_{L^2(\partial\Omega)}\big)\nonumber\\
     \leq &C\nu^{\f13}|k|^{\f76}\|\partial_y\varphi\|_{L^2(\partial\Omega)} +C\|F_1\|_{L^2}+C|\nu/k|^{-\f13}\|F_2\|_{H^{-1}}.
  \end{align}
  By Proposition \ref{prop:res-nav-b1} and  Proposition \ref{prop:res-nav-s1}, we have
  \begin{align*}
     & \nu^{\f16}|k|^{\f43}\|\varphi_{Na}^{(1)}\|_{L^2}\leq C(\|(1-y^2)F_1\|_{L^2}+|\nu/k|^{\f13}\|F_1\|_{L^2}),\\
     &\nu^{\f16}|k|^{\f43}\|\nabla\varphi_{Na}^{(1)}\|_{L^2} \leq C\|F_1\|_{L^2},
  \end{align*}
 and by Proposition \ref{prop:res-weak-b}, we have
  \begin{align*}
     & \nu^{\f16}|k|^{\f43}\|\varphi_{Na}^{(2)}\|_{L^2}\leq C|\nu/k|^{-\f13}\max(1-|\lambda|,\nu^{\f13}|k|^{-\f13})\|F_2\|_{H^{-1}},\\
     &\nu^{\f16}|k|^{\f43}\|\nabla\varphi_{Na}^{(2)}\|_{L^2} \leq C|\nu/k|^{-\f13}\|F_2\|_{H^{-1}},
  \end{align*}
  which together with  \eqref{16} and \eqref{17} show that
  \begin{align*}
    \nu^{\f16}|k|^{\f43}\|\varphi\|_{L^2}\leq & \nu^{\f16}|k|^{\f43}\big(\|\varphi_I\|_{L^2}+\|\varphi_{Na}^{(1)}\|_{L^2} +\|\varphi_{Na}^{(2)}\|_{L^2}\big)\\
     \leq & C\Big(\|(1-y^2)F_1\|_{L^2}+|\nu/k|^{\f13}\|F_1\|_{L^2} +|\nu/k|^{-\f13}\max(1-|\lambda|,\nu^{\f13}|k|^{-\f13})\|{F}_2\|_{H^{-1}} \\&+|\nu k|^{\f12}\|\partial_y\varphi\|_{L^2(\partial\Omega)}\Big),
  \end{align*}
  and
  \begin{align*}
     \nu^{\f16}|k|^{\f43}\|\nabla\varphi\|_{L^2}\leq & \nu^{\f16}|k|^{\f43}\big(\|\nabla\varphi_I\|_{L^2}+\|\nabla\varphi_{Na}^{(1)}\|_{L^2} +\|\nabla\varphi_{Na}^{(2)}\|_{L^2}\big)\\ \leq & C\big(\|F_1\|_{L^2}+|\nu/k|^{-\f13}\|{F}_2\|_{H^{-1} }+\nu^{\f13}|k|^{\f76}\|\partial_y\varphi\|_{L^2(\partial\Omega)}\big).
       \end{align*}

{\bf Step 2.} Proof of \eqref{eq:res-ihom3}.\smallskip

 By Proposition  \ref{prop:res-nav-b1}, we have
  \begin{align*}
     &\|(|k(y-\lambda)/\nu|^{\f14}+|k/\nu|^{\f16})\partial_y\varphi_{Na}^{(1)}\|_{L^2(\partial\Omega)} \leq C\nu^{-\f14}\|(|k(y-\lambda)|^{\f14}+1)\partial_y\varphi_{Na}^{(1)}\|_{L^2(\partial\Omega)}\\ &\qquad\qquad\qquad\leq C\nu^{-\f14}\|(|k(y-\lambda)|+1)\partial_y\varphi_{Na}^{(1)}\|_{L^2(\partial\Omega)} \leq C(\nu k^2)^{-\f{5}{12}}\|F_1\|_{L^2},\\
     &|\nu/k|^{\f16}\|\partial_z\partial_y\varphi_{Na}^{(1)}\|_{L^2}\leq C(\nu k^2)^{-\f13}\|F_1\|_{L^2}\leq C(\nu k^2)^{-\f{5}{12}}\|F_1\|_{L^2}.
  \end{align*}
  and  by Proposition \ref{prop:res-weak-b},
  \begin{align*}
    &\|(|k(y-\lambda)/\nu|^{\f14}+|k/\nu|^{\f16})\partial_y\varphi_{Na}^{(2)}\|_{L^2(\partial\Omega)} \leq C\nu^{-\f14}\|(|k(y-\lambda)|^{\f14}+1)\partial_y\varphi_{Na}^{(2)}\|_{L^2(\partial\Omega)}\\
    &\qquad\qquad\qquad\leq C\nu^{-\f14}\|(|k(y-\lambda)|^{\f34}+1)\partial_y\varphi_{Na}^{(2)}\|_{L^2(\partial\Omega)} \leq C\nu^{-\f34}|k|^{-\f12}\|F_2\|_{H^{-1}},\\
     &|\nu/k|^{\f16}\|\partial_z\partial_y\varphi_{Na}^{(2)}\|_{L^2}\leq C\nu^{-\f23}|k|^{-\f13}\|F_2\|_{H^{-1}}\leq C\nu^{-\f34}|k|^{-\f12}\|F_2\|_{H^{-1}},
  \end{align*}
  which together with Proposition \ref{prop:res-NB-hom} show that
    \begin{align}
     \|w_I\|_{L^2}\leq &C\big(\|(|k(y-\lambda)/\nu|^{\f14}+|k/\nu|^{\f16}+|k|^{\f12})\partial_y\varphi_I\|_{L^2(\partial\Omega)}
  +(|k/\nu|^{\f13}+|k|)^{-\f12}\|\partial_y\partial_z\varphi_I\|_{L^2(\partial\Omega)}\big) \nonumber\\ \leq & C\big(\|(|k(y-\lambda)/\nu|^{\f14}+|k/\nu|^{\f16})\partial_y\varphi_I\|_{L^2(\partial\Omega)}
  +|\nu/k|^{\f16}\|\partial_y\partial_z\varphi_I\|_{L^2(\partial\Omega)}\big)\nonumber\\
  \leq &C\big(\|(|k(y-\lambda)/\nu|^{\f14}+|k/\nu|^{\f16})\partial_y\varphi\|_{L^2(\partial\Omega)} +|\nu/k|^{\f16}\|\partial_y\partial_z\varphi\|_{L^2(\partial\Omega)}\big)\nonumber\\
  &\quad+C\big(\|(|k(y-\lambda)/\nu|^{\f14}+|k/\nu|^{\f16})\partial_y\varphi_{Na}^{(1)}\|_{L^2(\partial\Omega)} +|\nu/k|^{\f16}\|\partial_y\partial_z\varphi_{Na}^{(1)}\|_{L^2(\partial\Omega)}\big)\nonumber\\
  &\quad+C\big(\|(|k(y-\lambda)/\nu|^{\f14}+|k/\nu|^{\f16})\partial_y\varphi_{Na}^{(2)}\|_{L^2(\partial\Omega)} +|\nu/k|^{\f16}\|\partial_y\partial_z\varphi_{Na}^{(2)}\|_{L^2(\partial\Omega)}\big)\nonumber\\
  \leq &C\big(\|(|k(y-\lambda)/\nu|^{\f14}+|k/\nu|^{\f16})\partial_y\varphi\|_{L^2(\partial\Omega)} +|\nu/k|^{\f16}\|\partial_y\partial_z\varphi\|_{L^2(\partial\Omega)}\big) \nonumber\\
  &\quad+C\big((\nu k^2)^{-\f{5}{12}}\|F_1\|_{L^2}+ \nu^{-\f34}|k|^{-\f12}\|F_2\|_{H^{-1}}\big).\nonumber
  \end{align}
  By Proposition \ref{prop:res-nav-s1}, we have
  \begin{align*}
     \|w_{Na}^{(1)}\|_{L^2}+\|w_{Na}^{(2)}\|\leq& C((\nu k^2)^{-\f13}\|F_1\|_{L^2} +\nu^{-\f23}|k|^{-\f13}\|F_2\|_{H^{-1}})\\
     \leq& C((\nu k^2)^{-\f{5}{12}}\|F_1\|_{L^2} +\nu^{-\f34}|k|^{-\f12}\|F_2\|_{H^{-1}}).
  \end{align*}
This shows that
  \begin{align*}
    \|w\|_{L^2} \leq& \|w_{I}\|_{L^2}+\|w_{Na}^{(1)}\|_{L^2}+\|w_{Na}^{(2)}\|_{L^2}\\
    \leq &C\big(\|(|k(y-\lambda)/\nu|^{\f14}+|k/\nu|^{\f16})\partial_y\varphi\|_{L^2(\partial\Omega)} +|\nu/k|^{\f16}\|\partial_y\partial_z\varphi_I\|_{L^2(\partial\Omega)}\big)\\
  &\quad+C\big((\nu k^2)^{-\f{5}{12}}\|F_1\|_{L^2}+ \nu^{-\f34}|k|^{-\f12}\|F_2\|_{H^{-1}}\big).
  \end{align*}

{\bf Step 3.} Proof of \eqref{eq:res-ihom4}.\smallskip

 By Proposition  \ref{prop:res-nav-b1}, we have
  \begin{align*}
     &\|(|k(y-\lambda)/\nu|^{\f12}+|k/\nu|^{\f13})\partial_y\varphi_{Na}^{(1)}\|_{L^2(\partial\Omega)} \leq C\nu^{-\f12}\|(|k(y-\lambda)|^{\f12}+1)\partial_y\varphi_{Na}^{(1)}\|_{L^2(\partial\Omega)}\\ &\qquad\qquad\quad\leq C\nu^{-\f12}\|(|k(y-\lambda)|+1)\partial_y\varphi_{Na}^{(1)}\|_{L^2(\partial\Omega)} \leq C\nu^{-\f23}|k|^{-\f56}\|F_1\|_{L^2},\\
     &\|\partial_z\partial_y\varphi_{Na}^{(1)}\|_{L^2(\p\Om)}\leq C|\nu k|^{-\f12}\|F_1\|_{L^2}\leq C\nu^{-\f23}|k|^{-\f56}\|F_1\|_{L^2},
  \end{align*}
  and  by Proposition \ref{prop:res-weak-b},
  \begin{align*}
    &\|(|k(y-\lambda)/\nu|^{\f12}+|k/\nu|^{\f13})\partial_y\varphi_{Na}^{(2)}\|_{L^2(\partial\Omega)} \leq C\nu^{-\f12}\|(|k(y-\lambda)|^{\f12}+1)\partial_y\varphi_{Na}^{(2)}\|_{L^2(\partial\Omega)}\\
    &\qquad\qquad\quad\leq C\nu^{-\f12}\|(|k(y-\lambda)|^{\f34}+1)\partial_y\varphi_{Na}^{(2)}\|_{L^2(\partial\Omega)} \leq C\nu^{-1}|k|^{-\f12}\|F_2\|_{H^{-1}},\\
     &\|\partial_z\partial_y\varphi_{Na}^{(2)}\|_{L^2(\p\Om)}\leq C\nu^{-\f56}|k|^{-\f16}\|F_2\|_{H^{-1}}\leq C\nu^{-1}|k|^{-\f12}\|F_2\|_{H^{-1}},
  \end{align*}
which together with Proposition \ref{prop:res-NB-hom}  show that
\begin{align*}
     \|w_I\|_{L^2(\partial\Omega)}\leq &C\big(\|(|k(y-\lambda)/\nu|^{\f12}+|k/\nu|^{\f13}+|k|)\partial_y\varphi_I\|_{L^2(\partial\Omega)}
  +\|\partial_y\partial_z\varphi_I\|_{L^2(\partial\Omega)}\big) \\ \leq & C\big(\|(|k(y-\lambda)/\nu|^{\f12}+|k/\nu|^{\f13})\partial_y\varphi_I\|_{L^2(\partial\Omega)}
  +\|\partial_y\partial_z\varphi_I\|_{L^2(\partial\Omega)}\big)\\
  \leq &C\big(\|(|k(y-\lambda)/\nu|^{\f12}+|k/\nu|^{\f13})\partial_y\varphi\|_{L^2(\partial\Omega)} +\|\partial_y\partial_z\varphi\|_{L^2(\partial\Omega)}\big)\nonumber\\
  &\quad+C\big(\|(|k(y-\lambda)/\nu|^{\f12}+|k/\nu|^{\f13})\partial_y\varphi_{Na}^{(1)}\|_{L^2(\partial\Omega)} +\|\partial_y\partial_z\varphi_{Na}^{(1)}\|_{L^2(\partial\Omega)}\big)\nonumber\\
  &\quad+C\big(\|(|k(y-\lambda)/\nu|^{\f12}+|k/\nu|^{\f13})\partial_y\varphi_{Na}^{(2)}\|_{L^2(\partial\Omega)} +\|\partial_y\partial_z\varphi_{Na}^{(2)}\|_{L^2(\partial\Omega)}\big)\\
  \leq &C\big(\|(|k(y-\lambda)/\nu|^{\f12}+|k/\nu|^{\f13})\partial_y\varphi\|_{L^2(\partial\Omega)} +\|\partial_y\partial_z\varphi\|_{L^2(\partial\Omega)}\big) \\
  &\quad+C\big(\nu^{-\f23}|k|^{-\f56}\|F_1\|_{L^2}+ \nu^{-1}|k|^{-\f12}\|F_2\|_{H^{-1}}\big).
  \end{align*}
Due to  $w_{Na}^{(j)}|_{y=\pm1}=0(j=1,2)$, we have
\begin{align*}
   \|w\|_{L^2(\partial\Omega)}=\|w_{I}\|_{L^2(\partial\Omega)} \leq C\big(&\|(|k(y-\lambda)/\nu|^{\f12}+|k/\nu|^{\f13})\partial_y\varphi\|_{L^2(\partial\Omega)} +\|\partial_y\partial_z\varphi\|_{L^2(\partial\Omega)}\big)\\
  &+C\big(\nu^{-\f23}|k|^{-\f56}\|F_1\|_{L^2}+ \nu^{-1}|k|^{-\f12}\|F_2\|_{H^{-1}}\big).
\end{align*}
\end{proof}

The following proposition gives some resolvent estimates relating to the inviscid damping effect. In what follows, we always assume $\nu k^2\le 1$.

\begin{Proposition}\label{prop:res-damp}
Let $\nu k^2\le 1$, and $w\in H^2(\Omega)$ be a solution of \eqref{eq:OS-NB} with $F\in L^2(\Omega)$ and $F=F_1+F_2$. Then it holds that
\begin{align*}
&\nu^{\f16}|k|^{\f43}\|\varphi\|_{L^2} \leq C\Big(\nu^{\f16}|k|^{-\f23}\|F_1\|_{L^2}+\nu^{\f16}|k|^{-\f23}\max(1-|\lambda|,\nu^{\f13}|k|^{-\f13})\|\nabla{F}_1\|_{L^2}\\ &\qquad\quad
+|\nu/k|^{-\f13}\max(1-|\lambda|,\nu^{\f13}|k|^{-\f13})\|{F}_2\|_{H^{-1}}+|\nu k|^{\f12}\|\partial_y\varphi\|_{L^2(\partial\Omega)}\Big),\\ &\nu^{\f16}|k|^{\f43}\|\nabla\varphi\|_{L^2} \leq C\Big(\nu^{\f16}|k|^{-\f23}\|\nabla F_1\|_{L^2}+|\nu/k|^{-\f13}\|{F}_2\|_{H^{-1}}+\nu^{\f13}|k|^{\f76} \|\partial_y\varphi\|_{L^2(\partial\Omega)}\Big),\\
&\|w\|_{L^2} \leq C\Big(\|(|k(y-\lambda)/\nu|^{\f14}+|k/\nu|^{\f16})\partial_y\varphi\|_{L^2(\partial\Omega)}
+|\nu/k|^{\f16}\|\partial_y\partial_z\varphi\|_{L^2(\partial\Omega)}\\&\qquad\quad+\nu^{-\f14}|k|^{-\f32}\|\nabla F_1\|_{L^2}+\nu^{-\f34} |k|^{-\f{1}{2}}\|{F}_2\|_{H^{-1}}\Big),
\\&\|w\|_{L^2(\partial\Omega)} \leq C\Big(\|(|k(y-\lambda)/\nu|^{\f12}+|k/\nu|^{\f13})\partial_y\varphi\|_{L^2(\partial\Omega)}
+\|\partial_y\partial_z\varphi\|_{L^2(\partial\Omega)}\\&\qquad\quad+\nu^{-\f12}|k|^{-\f32}\|\nabla F_1\|_{L^2}+\nu^{-1} |k|^{-\f{1}{2}}\|{F}_2\|_{H^{-1}}\Big).
\end{align*}
\end{Proposition}

We need the following lemmas.

\begin{Lemma}\label{Lem: weak damping}
Let  $F\in H^1(\Omega),\ w_1=\chi_1F,\ \chi_1=(V-\lambda-i\delta)^{-1},$ $\partial_xF=ikF$. Then it holds that for $f_0\in H^1_0(\Omega)$ with $\partial_xf_0=ikf_0$,
  \begin{align*}
    &|\langle w_1,f_0\rangle| \leq C\|\partial_yf_0\|_{L^2}\big(\max(1-|\lambda|,0)\| \nabla F\|_{L^2}+\|F\|_{L^2}\big),\\
    &|\langle w_1,f_0\rangle| \leq C|k|^{-1}\|\nabla f_0\|_{L^2}\|\nabla F\|_{L^2},\\
    &|\langle w_1,f_0\rangle| \leq C|k|^{-\f32}\|\nabla f_0\|_{L^2_{x,z}L^\infty_y}\|\nabla F\|_{L^2}.
  \end{align*}
\end{Lemma}

\begin{proof}
Without loss of generality,  we may assume $\lambda\geq0$. If $0\leq \lambda\leq 1$, let $y_{\lambda}(z)\in[-1,1]$ so that $V(y_{\lambda}(z),z)=\lambda$. Then we have
   \begin{align}\label{eq: 1-y lambda}
     1-y_\lambda &= (1-\lambda)\f{1-y_{\lambda}}{V(1,z)-V(y_\lambda,z)}\leq (1-\lambda)\|[\partial_yV]^{-1}\|_{L^\infty}\leq C(1-\lambda).
   \end{align}
First of all,  we have
   \begin{align*}
      |\langle w_1,f_0\rangle|&=\left|\int_{\mathbb{T}^2} \int_{-1}^{1}\f{F\bar{f}_0}{V-\lambda-i\delta} dydxdz\right|\\
    &\leq\left|\int_{\mathbb{T}^2}\int_{-1}^{1}\f{F(x,y,z) (\bar{f}_0(x,y,z)-\bar{f}_0(x,y_{\lambda},z))}{V-\lambda-i\delta} dydxdz\right| \\&\quad +\left|\int_{\mathbb{T}^2}\int_{-1}^{1}\f{F(x,y,z)\bar{f}_0(x,y_{\lambda},z)}{V-\lambda-i\delta} dydxdz\right|\\
    &\leq \|F\|_{L^2} \left\|\f{f_0(x,y,z)-f_0(x,y_{\lambda},z)}{V(y,z)-V(y_\lambda,z)} \right\|_{L^2} +\left|\int_{\mathbb{T}^2}\int_{-1}^{1}\f{F(x,y,z)\bar{f}_0(x,y_{\lambda},z)}{V-\lambda-i\delta} dydxdz\right|.
   \end{align*}
 By Hardy's inequality, we get
    \begin{align*}
       \left\|\f{f_0(x,y,z)-f_0(x,y_{\lambda},z)}{V(y,z)-V(y_\lambda,z)} \right\|_{L^2}&\leq \left\|\f{f_0(x,y,z)-f_0(x,y_{\lambda},z)}{y-y_\lambda} \right\|_{L^2}\left\|\f{y-y_\lambda}{V(y,z)-V(y_\lambda,z)}\right\|_{L^\infty}\\
       &\leq C\|\partial_yf_0\|_{L^2}\|[\partial_yV]^{-1}\|_{L^\infty}\leq C\|\partial_yf_0\|_{L^2}.
    \end{align*}
  By Lemma \ref{lem:hardy-V} (with $\delta_{*}=1-\lambda$) and \eqref{eq: 1-y lambda}, we have
    \begin{align*}
       &\left|\int_{\mathbb{T}^2}\int_{-1}^{1}\f{F(x,y,z)\bar{f}_0(x,y_{\lambda},z)}{V-\lambda-i\delta} dydxdz\right|\leq C\Big\|\big(\delta_*^{\f12}\|\partial_yF\|_{L^2_y} +\delta_*^{-\f12}\|F\|_{L^2_y}\big)|f_0(x,y_\lambda,z)|\Big\|_{L^1_{x,z}}\\
       &\leq C\big(\delta^{\f12}_*\|\partial_yF\|_{L^2}+ \delta_*^{-\f12}\|F\|_{L^2}\big)\|f_0(x,y_\lambda,z)\|_{L^2_{x,z}} \\ &\leq C\big(\delta^{\f12}_*\|\nabla F\|_{L^2}+ \delta_*^{-\f12}\|F\|_{L^2}\big)\left\|\int_{y_\lambda}^{1}\partial_y f_0(x,y_1,z)dy_1\right\|_{L^2_{x,z}}\\
       &\leq C\big(\delta^{\f12}_*\|\nabla F\|_{L^2}+ \delta_*^{-\f12}\|F\|_{L^2}\big)\|(1-y_\lambda)^{\f12}\partial_yf_0\|_{L^2}\leq C\big((1-\lambda)\| \nabla F\|_{L^2}+\|F\|_{L^2}\big)\|\partial_yf_0\|_{L^2}.
    \end{align*}
 This shows that for $\la\in [0,1]$,
    \begin{align*}
     |\langle w_1, f_0\rangle| &\leq C\big((1-\lambda)\| \nabla F\|_{L^2}+\|F\|_{L^2}\big)\|\partial_yf_0\|_{L^2}.
    \end{align*}
 If $\lambda\geq 1$,  by Hardy's inequality and Lemma \ref{lem:V}, we get
    \begin{align*}
       |\langle w_1,f_0\rangle|&=\left|\int_{\mathbb{T}^2} \int_{-1}^{1}\f{F\bar{f}_0}{V-\lambda-i\delta} dydxdz\right|\leq \|F\|_{L^2}\left\|\f{f_0}{1-V}\right\|_{L^2}\leq C\|F\|_{L^2}\left\|\f{f_0}{1-y}\right\|_{L^2}\\
       &\leq C\|F\|_{L^2}\|\partial_yf_0\|_{L^2}.
    \end{align*}
Combining two cases,  we obtain
   \begin{align*}
      |\langle w_1, f_0\rangle| &\leq C\big(\max(1-|\lambda|,0)\| \nabla F\|_{L^2}+\|F\|_{L^2}\big)\|\partial_yf_0\|_{L^2}.
   \end{align*}

  Using Lemma \ref{lem:hardy-V} with $\delta_{*}=|k|^{-1}$) and Lemma \ref{lem:sob-f}, we get
   \begin{align*}
     |\langle w_1,f_0\rangle| &=\left|\int_{\mathbb{T}^2}\int_{-1}^{1} \f{F\bar{f}_0}{V-\lambda-i\delta}dydxdz\right|\leq \int_{\mathbb{T}^2}\left|\int_{-1}^{1} \f{F\bar{f}_0}{V-\lambda-i\delta}dy\right|dxdz\\
     &\leq C|k|^{-\f12}\Big(\big\|\|\partial_y(F\bar{f}_0)\|_{L^2_y}\big\|_{L^1_{x,z}} +|k|\big\|\|F\bar{f}_0\|_{L^2_y}\big\|_{L^1_{x,z}}\Big)\\
     &\leq C|k|^{-\f12}\Big(\|\partial_yF\|_{L^2}\|f_0\|_{L^2_{x,z}L^\infty_{y}} +\|F\|_{L^2_{x,z}L^\infty_y}\|\partial_yf_0\|_{L^2} +|k|\|F\|_{L^2_{x,z}L^\infty_y}\|f_0\|_{L^2}\Big)\\
     &\leq C|k|^{-1}\|\nabla F\|_{L^2}\|\nabla f_0\|_{L^2}.
   \end{align*}

Finally, we have
\begin{align*}
     |\langle w_1,f_0\rangle| &\leq C|k|^{-\f12}\Big(\big\|\|\partial_y(F\bar{f}_0)\|_{L^2_y}\big\|_{L^1_{x,z}} +|k|\big\|\|F\bar{f}_0\|_{L^2_y}\big\|_{L^1_{x,z}}\Big)\\
     &\leq C|k|^{-\f12}\Big(\|\partial_yF\|_{L^2}\|f_0\|_{L^2_{x,z}L^\infty_{y}} +\|F\|_{L^2}\|\partial_yf_0\|_{L^2_{x,z}L^\infty_y} +|k|\|F\|_{L^2}\|f_0\|_{L^2_{x,z}L^\infty_y}\Big)\\
     &\leq C|k|^{-\f32}\|\nabla F\|_{L^2}\|\nabla f_0\|_{L^2_{x,z}L_y^\infty},   \end{align*}
here we used $\partial_xf_0=ikf_0$.
\end{proof}

 \begin{Lemma}\label{lem7.1}
Let  $F\in H^1(\Omega)$, $\Delta\varphi=w_1=\chi_1F$, $\varphi|_{y=\pm1}=0$, $\chi_1=(V-\lambda-i\delta)^{-1}$,
$\partial_xF=ikF$. Then it holds that
\begin{align*}&\|\nabla w_1\|_{L^2}+|\nu/k|^{-\f13}\|w_1\|_{L^2} \leq C\nu^{-\f12}\|\nabla F\|_{L^2},\\
&\|(1-y^2)w_1\|_{L^2}\leq C\|F\|_{L^2}+C|\nu k^2|^{-\f16}\max(1-|\lambda|,\nu^{\f13}|k|^{-\f13})\|\nabla F\|_{L^2},\\
&\|\varphi\|_{L^2}\leq C|k|^{-1}\big(\|F\|_{L^2}+\max(1-|\lambda|,\nu^{\f13}|k|^{-\f13})\|\nabla F\|_{L^2}\big),\\ &\|\nabla\varphi\|_{L^2} \leq C|k|^{-1}\|\nabla F\|_{L^2},\\&\|\partial_y\partial_z\varphi\|_{L^2(\partial\Omega)}\leq C\nu^{-\f13}|k|^{-\f16}\|\nabla{F}\|_{L^2},\\
&(1+|k(\lambda-j)|)\|\partial_y\varphi\|_{L^2(\Gamma_j)}\leq C|k|^{-\f12}\|\nabla{F}\|_{L^2}\\&\qquad\qquad+C\ln(2+(|k(\lambda-j)|+(\nu k^2)^{\f13})^{-1})\|F\|_{L_{x,z}^2L_y^{\infty}},\ \ j\in\{\pm 1\}.
\end{align*}
\end{Lemma}

\begin{proof}
By Lemma \ref{lem:sob-f} and $\nu k^2\leq 1$, we have
  \begin{align*}
      &\|\nabla w_1\|_{L^2}+|\nu/k|^{-\f13}\|w_1\|_{L^2}\\
      &\leq \|\nabla\chi_1\|_{L^\infty_{x,z}L^2_y}\|F\|_{L^2_{x,z}L^\infty_y} +\|\chi_1\|_{L^\infty}\|\nabla F\|_{L^2}+|\nu/k|^{-\f13}\|\chi_1\|_{L^\infty_{x,z}L^2_{y}}\|F\|_{L^2_{x,z}L^\infty_{y}}\\
      &\leq C\delta^{-\f32}|k|^{-\f12}\|\nabla F\|_{L^2} +C\delta^{-1}\|\nabla F\|_{L^2}+C\delta^{-1}\delta^{-\f12}|k|^{-\f12}\|\nabla F\|_{L^2}\\
      &\leq C\nu^{-\f12}\|\nabla F\|_{L^2}.
  \end{align*}

Without loss of generality. we may assume $\lambda\geq0$. Using the fact that
\begin{align*}
  |1-y^2| &\leq C(1-|y|)\leq C(1-V) = C[(1-\lambda)+(\lambda-V)]\leq C(\max(1-\lambda,0)+|V-\lambda|),
\end{align*}
we deduce that
\begin{align*}
  \|(1-y^2)w_1\|_{L^2} &\leq C\big(\max(1-\lambda,0)\|w_1\|_{L^2}+\|(V-\lambda)w_1\|_{L^2}\big)\\&\leq C\big(\max(1-|\lambda|,0)\delta^{-\f12}|k|^{-\f12}\|\nabla F\|_{L^2}+\|(V-\lambda)\chi_1\|_{L^\infty}\|F\|_{L^2}\big)\\
  &\leq C\big(\max(1-|\lambda|,\nu^{\f13}|k|^{-\f13})|\nu k^2|^{-\f16}\|\nabla F\|_{L^2}+\|F\|_{L^2}\big).
\end{align*}

Let $\phi$ be the unique solution to $\Delta\phi=\varphi$, $\phi|_{y=\pm1}=0$. Using Lemma \ref{Lem: weak damping} with $f_0=\phi$, we get
\begin{align*}
  \|\varphi\|_{L^2}^2 =|\langle w_1,\phi\rangle|\leq& C\big( \max(1-|\lambda|,0)\|\nabla F\|_{L^2}+\|F\|_{L^2}\big)\|\partial_y\phi\|_{L^2}\\
  \leq& C|k|^{-1}\big( \max(1-|\lambda|,0)\|\nabla F\|_{L^2}+\|F\|_{L^2}\big)\|\varphi\|_{L^2}\\
  \leq & C|k|^{-1}\big( \max(1-|\lambda|,|\nu/k|^{\f13})\|\nabla F\|_{L^2}+\|F\|_{L^2}\big)\|\varphi\|_{L^2},
\end{align*}
and
\begin{align*}
   \|\nabla\varphi\|_{L^2}^2 =&|\langle w_1,\varphi\rangle|\leq C|k|^{-1}\|\nabla F\|_{L^2}\|\nabla \varphi\|_{L^2}.
\end{align*}
This shows that
\begin{align*}
   &\|\varphi\|_{L^2}\leq C|k|^{-1}\big(\|F\|_{L^2}+ \max(1-|\lambda|,|\nu/k|^{\f13})\|\nabla F\|_{L^2}\big),\\
   &\|\nabla\varphi\|_{L^2}\leq C|k|^{-1}\|\nabla F\|_{L^2}.
\end{align*}

Using the first inequality of this lemma, we get
\begin{align*}
   \|\partial_y\partial_z\varphi\|_{L^2(\partial\Omega)}&\leq  \|\partial_y\partial_z\varphi\|_{L^2_{x,z}L^\infty_y}\leq C\|\partial_y\partial_z\varphi\|_{L^2}^{\f12} \|\partial_y^2\partial_z\varphi\|_{L^2}^{\f12}\leq C\|w_1\|^{\f12}_{L^2}\|\nabla w_1\|_{L^2}^{\f12}\\ &\leq C\nu^{-\f13}|k|^{-\f16}\|\nabla F\|_{L^2}.
\end{align*}

Finally, let us estimate $\|\partial_y\varphi\|_{L^2(\partial\Omega)}$. Recalling  \eqref{eq:varphi-dual}, we have
  \begin{align}
   \|\partial_y\varphi\|_{L^2(\Gamma_1)} &=\f{1}{(2\pi)^2}\sup_{f\in\mathcal{F}_1}|\left\langle w_1,f\right\rangle|.\label{eq:varphi-dual-2}
\end{align}
The estimate $\langle w_1,f\rangle$ will be split into three cases.\smallskip

\textbf{Case 1}. $\delta\leq1-\lambda\leq |k|^{-1}$. \smallskip

Let  $\chi_4=\max(1-|(V-\lambda)/(1-\lambda)|,0)$. Then $\chi_4\in H^1(\Omega)$ and it follows from Lemma \ref{lem:V} that
\begin{align*}
 &\|\chi_4\|_{L^\infty}=1,\quad |\nabla\chi_4|\leq |\nabla V/(1-\lambda)|\leq C/(1-\lambda),\quad \chi_4|_{y=1}=\chi_4|_{y\leq1-4(1-\lambda)}=0,\\&\|\chi_4/(1-y)\|_{L^\infty}\leq\|\nabla\chi_4\|_{L^\infty}\leq C/(1-\lambda),\quad \|\chi_4\|_{L^2_{y}L^\infty_{x,z}}\leq C(1-\lambda)^{\f12},\\& \|\chi_4/(1-y)\|_{L^2_{y}L^\infty_{x,z}}+\|\nabla\chi_4\|_{L^2_{y}L^\infty_{x,z}}\leq C(1-\lambda)^{-\f12}.
\end{align*}
Due to $\delta\leq1-\lambda\leq |k|^{-1}$, we get by Lemma \ref{lem:sob-f} that
\begin{align*}
  \| \nabla (\chi_4F)\|_{L^2}\leq& \|\nabla\chi_4\|_{L^2_{y}L^\infty_{x,z}}\|F\|_{L^\infty_{y}L^2_{x,z}}+
   \|\chi_4\|_{L^\infty}\|\nabla F\|_{L^2}\\
   \leq& C(1-\lambda)^{-\f12}|k|^{-\f12}\|\nabla F\|_{L^2}+
   \|\nabla F\|_{L^2}\leq C(1-\lambda)^{-\f12}|k|^{-\f12}\|\nabla F\|_{L^2},\\ \| \chi_4F\|_{L^2}\leq&\|\chi_4\|_{L^2_{y}L^\infty_{x,z}}\|F\|_{L^\infty_{y}L^2_{x,z}}\leq C(1-\lambda)^{\f12}|k|^{-\f12}\|\nabla F\|_{L^2}.
\end{align*}

For $f\in \mathcal{F}_1$, we get by Lemma \ref{lem:F1} that
\begin{align*}
   \|\partial_y(\chi_4f)\|_{L^2}\leq&\|(\partial_y\chi_4)f\|_{L^2}+\|\chi_4\partial_yf\|_{L^2}\\
   \leq&\|\nabla\chi_4\|_{L^2_{y}L^\infty_{x,z}}\|f\|_{L^\infty_{y}L^2_{x,z}}+
   \|\chi_4/(1-y)\|_{L^2_{y}L^\infty_{x,z}}\|(1-y)\partial_yf\|_{L^\infty_{y}L^2_{x,z}}\\ \leq&C(1-\lambda)^{-\f12}\big(\|f\|_{L^2_{x,z}L^\infty_{y}}+\|(1-y)\partial_yf\|_{L^\infty_{y}L^2_{x,z}}\big)\leq C(1-\lambda)^{-\f12}.
\end{align*}
Thus by Lemma \ref{Lem: weak damping}, we have
\begin{align}
  |\langle \chi_4F\chi_1,\chi_4f\rangle| &\leq C \|\partial_y(\chi_4f)\|_{L^2}\big(\max(1-|\lambda|,0)\| \nabla (\chi_4F)\|_{L^2}+\|\chi_4F\|_{L^2}\big)\nonumber\\
  &\leq C(1-\lambda)^{-\f12}\big((1-\lambda)\cdot(1-\lambda)^{-\f12}|k|^{-\f12}\|\nabla F\|_{L^2}+(1-\lambda)^{\f12}|k|^{-\f12}\|\nabla F\|_{L^2}\big)\nonumber\\
  &\leq  C|k|^{-\f12}\|\nabla F\|_{L^2}.\label{eq: F0}
\end{align}

Let $\Omega_3=\mathbb{T}\times[1-2|k|^{-1},1]\times \mathbb{T}$ and $\Omega_3^c=\mathbb{T}\times[-1,1-2|k|^{-1}]\times \mathbb{T}$. Then
\begin{align*}
 |\langle (1-\chi_4^2)F\chi_1, f\rangle|\leq&\|(1-\chi_4^2)F\chi_1\overline{f}\|_{L^1(\Omega_3^c)}+\|(1-\chi_4^2)F\chi_1\overline{f}\|_{L^1(\Omega_3)}\\ \leq & \|F\|_{L^2}\|f\|_{L^2}\|(1-\chi_4^2)\chi_1\|_{L^\infty(\Omega_3^c)}+
   \|F\|_{L^2_{x,z}L^\infty_y}\|f\|_{L^2_{x,z}L^\infty_y}\|(1-\chi_4^2)\chi_1\|_{L^\infty_{x,z}L^1_y(\Omega_3)}.
\end{align*}
Noticing that
\begin{align*}
  & |(1-\chi_4^2)\chi_1|\leq 2|(1-\chi_4)\chi_1|\leq2|(V-\lambda)/(1-\lambda)||\chi_1|\leq2/(1-\lambda),\\& |(1-\chi_4^2)\chi_1|\leq |\chi_1|\leq1/|V-\lambda|,
\end{align*}
we deduce that
\beno
|(1-\chi_4^2)\chi_1|\leq C/(1-\lambda+|V-\lambda|)\leq C/(1-V).
\eeno
Let $ \delta_1=1-\lambda$, and  then
  \begin{align*}
     &\|(1-\chi_4^2)\chi_1\|_{L^\infty_{x,z}L^1_y(\Omega_3)}\\
     &\leq\|1/(|V-\lambda|+\delta_1)\|_{L^{\infty}_{x,z}L^1_{y}(\Omega_3)}= \left\|\int_{1-2|k|^{-1}}^{1}\f{1}{|V-\lambda|+\delta_1}dy\right\|_{L^\infty}\nonumber\\ &\leq \left\|\int_{1-2|k|^{-1}}^{1}\f{\partial_yV}{|V-\lambda|+\delta_1}dy \right\|_{L^\infty} \left\|[\partial_yV]^{-1}\right\|_{L^\infty}\leq C\left\|\int_{V(1-2|k|^{-1},z)}^{1}\f{1}{|y_1-\lambda|+\delta_1}dy_1 \right\|_{L^\infty}\nonumber\\
     &\leq C\int_{1-4|k|^{-1}}^{1}\f{dy_1}{|y_1-\lambda|+\delta_1}= C\ln((|1-\lambda|+\delta_1)/\delta_1) +C\ln((|1-\lambda-4|k|^{-1}|+\delta_1)/\delta_1)\nonumber\\
     &\leq C\ln((4|k|^{-1}+\delta_1)/\delta_1)= C\ln(1+4|k\delta_1|^{-1})= C\ln\big(1+4/|k(\lambda-1)|\big),
  \end{align*}
and
\begin{align*}
  & \|(1-\chi_4^2)\chi_1\|_{L^\infty(\Omega_3^c)}\leq C\|(1-V)^{-1}\|_{L^\infty(\Omega_3^c)}\leq C\|(1-y)^{-1}\|_{L^\infty(\Omega_3^c)}\leq C|k|.
\end{align*}
Thus, we conclude that
  \begin{align*}
      |\langle (1-\chi_4^2)F\chi_1, f\rangle|\leq C|k|\| F\|_{L^2}\|f\|_{L^2} + C \ln(1+4/|k(\lambda-1)|)\|F\|_{L^2_{x,z}L^\infty_y}\|f\|_{L^2_{x,z}L^\infty_y},
  \end{align*}
  which along with \eqref{eq: F0}  and the facts that $|k|\| F\|_{L^2}\leq \|\nabla F\|_{L^2},\ 1\ge |k(\lambda-1)|\geq k\delta=(\nu k^2)^{\f13}, $ $\|f\|_{L^2}\leq C|k|^{-\f12}$ and $\|f\|_{L^2_{x,z}L^\infty_y}\leq C$ due to Lemma \ref{lem:F1}, gives
  \begin{align*}
     (1+|k(\lambda-1)|)|\langle w_1,f\rangle|&\leq2|\langle w_1,f\rangle|\leq 2|\langle \chi_4F\chi_1,\chi_4f\rangle| +2|\langle (1-\chi_4^2)F\chi_1, f\rangle|\\
     &\leq C|k|^{-\f12} \|\nabla F\|_{L^2}+\ln\big(2+(|k(\lambda-1)|+(\nu k^2)^{\f13})^{-1}\big)\|F\|_{L^2_{x,z}L^\infty_y}.
  \end{align*}

\textbf{Case 2}. $1-\delta\leq\lambda\leq 1+|k|^{-1}$. \smallskip

In this case, we have
\beno
|\lambda-1|-(\lambda-1)=2\max(1-\lambda,0)\leq 2\delta,\quad  \lambda-1+3\delta\geq |\lambda-1|+\delta.
\eeno
Let $\delta_1=|\lambda-1|+\delta $. By Lemma \ref{lem:V}, we have
\begin{align*}
  & |V-\lambda|+3\delta\geq \lambda-V+3\delta\geq 1-V+|\lambda-1|+\delta\geq (1-y)/2+\delta_1,\\&|\chi_1|\leq [\max(|V-\lambda|,\delta)]^{-1}\leq C(|V-\lambda|+3\delta)^{-1}\leq C(1-y+\delta_1)^{-1}.
\end{align*}
 Let $\Omega_3$ and $\Omega_3^c$ be defined as in \textbf{Case 1}, and then
\begin{align}
     \label{eq: chi1 case2}&\|\chi_1\|_{L^\infty_{x,z}L^1_y(\Omega_3)}\leq C\|(1-y+\delta_1)^{-1}\|_{L^1([1-2|k|^{-1},1])}=C\ln\big((2|k|^{-1}+\delta_1)/\delta_1\big),\\
      &\|\chi_1\|_{L^\infty(\Omega_3^c)}\leq C\|(1-y)^{-1}\|_{L^\infty(\Omega_3^c)}\leq C|k|,
  \end{align}
  which imply that
  \begin{align*}
  |\langle F\chi_1, f\rangle|\leq&\|F\chi_1\overline{f}\|_{L^1(\Omega_3^c)}+\|F\chi_1\overline{f}\|_{L^1(\Omega_3)}\\
  \leq & \|F\|_{L^2}\|f\|_{L^2}\|\chi_1\|_{L^\infty(\Omega_3^c)}+
   \|F\|_{L^2_{x,z}L^\infty_y}\|f\|_{L^2_{x,z}L^\infty_y}\|\chi_1\|_{L^\infty_{x,z}L^1_y(\Omega_3)}\\
\leq& C|k|\| F\|_{L^2}\|f\|_{L^2} + C \ln(1+2|k\delta_1|^{-1})\|F\|_{L^2_{x,z}L^\infty_y}\|f\|_{L^2_{x,z}L^\infty_y}.
  \end{align*}
 Using the facts that
 \beno
  |k|\| F\|_{L^2}\leq \|\nabla F\|_{L^2},\quad  k\delta_1=|k(\lambda-1)|+(\nu k^2)^{\f13}, \quad \|f\|_{L^2}\leq C|k|^{-\f12},\quad \|f\|_{L^2_{x,z}L^\infty_y}\leq C,
 \eeno
we deduce that
 \begin{align*}
     (1+|k(\lambda-1)|)|\langle w_1,f\rangle|&\leq2|\langle w_1,f\rangle|\\
     &\leq C|k|^{-\f12} \|\nabla F\|_{L^2}+\ln\big(2+(|k(\lambda-1)|+(\nu k^2)^{\f13})^{-1}\big)\|F\|_{L^2_{x,z}L^\infty_y}.
  \end{align*}

  \textbf{Case 3}.  $|\lambda-1|\ge |k|^{-1}$.\smallskip

Obviously, we have $(1-V){f}\in H^1_0(\Omega)$, and by Lemma \ref{lem:sob-f} and  Lemma \ref{lem:V}, we have
  \begin{align*}
    \|\nabla((1-V){f})\|_{L^2} \leq& \|(1-V)\nabla{f}\|_{L^2}+ \|\nabla V\|_{L^{\infty}}\|{f}\|_{L^2}\\
    \leq& C(\|(1-y)\nabla{f}\|_{L^2}+ \|{f}\|_{L^2})
     \leq C|k|^{-\f12},
  \end{align*}
and  $ |(V-\lambda)w_1|=|(V-\lambda)\chi_1F|\leq |F|$.
Thus, it follows from Lemma \ref{Lem: weak damping} that
\begin{align*}
  |(1-\lambda)\langle w_1,f\rangle| &= |\langle (V-\lambda)w_1,{f}\rangle+\langle w_1,(1-V)f\rangle|\\ &\leq C\big(\|(V-\lambda)w_1\|_{L^2}\|{f}\|_{L^2}+|k|^{-1}\|\nabla F\|_{L^2}\|\nabla ((1-V){f})\|_{L^2}\big)\\
  &\leq C\big(\| F\|_{L^2}|k|^{-\f12}+|k|^{-1}\|\nabla F\|_{L^2}|k|^{-\f12}\big)\leq C|k|^{-\f32}\|\nabla F\|_{L^2},
\end{align*}
which gives
\begin{align*}
  (1+|k(\lambda-1)|)|\langle w_1,f\rangle| &\leq C|k|^{-\f12}\|\nabla F\|_{L^2}.
\end{align*}

Summing up {\bf Case 1-Case 3},  we conclude that
\begin{align*}
   &(1+|k(\lambda-1)|)|\langle w_1,f\rangle|
     \leq C|k|^{-\f12}\|\nabla F\|_{L^2} +\ln\big(2+(|k(\lambda-1)|+(\nu k^2)^{\f13})^{-1}\big)\|F\|_{L^2_{x,z}L^\infty_y}.
\end{align*}
This along with \eqref{eq:varphi-dual-2}  shows  that
\begin{align*}
  (1+|k(\lambda-1)|)\|\partial_y\varphi\|_{L^2(\Gamma_1)} &\leq C\big(1+|k(\lambda-1)|\big)\sup_{f\in\mathcal{F}_1}|\langle w_1,f\rangle|\\&\leq C|k|^{-\f12}\|\nabla F\|_{L^2} +\ln\big(2+(|k(\lambda-1)|+(\nu k^2)^{\f13})^{-1}\big)\|F\|_{L^2_{x,z}L^\infty_y}.
\end{align*}

The proof of the case $j=-1$ is similar.
\end{proof}\smallskip

Now we are in a position to prove Proposition \ref{prop:res-damp}.

\begin{proof}
We  decompose $w$ as $w=w_1+w_2$, where $w_1$ and $w_2$ solve
  \begin{align*}\left\{\begin{aligned}
     &ik(V-\lambda-i\delta)w_1=F_1,\\
     &-\nu\Delta w_2+ik(V-\lambda)w_2-a(\nu k^2)^{\f13}w_2=F_2+\nu\Delta w_1+(a+1)(\nu k^2)^{\f13}w_1,\\
     &w_j=\Delta \varphi_j,\quad \varphi_j|_{y=\pm1}=0,\quad j=\{1,2\}.
  \end{aligned}\right.
  \end{align*}
  Let $F_{1,*}=(a+1)(\nu k^2)^{\f13}w_1$, $F_{2,*}=F_2+\nu\Delta w_1$. Then we have
  \begin{align*}
     &\|F_{1,*}\|_{L^2}\leq C(\nu k^2)^{\f13}\|w_1\|_{L^2},\quad \|(1-y^2)F_{1,*}\|_{L^2}\leq C(\nu k^2)^{\f13}\|(1-y^2)w_1\|_{L^2},\\
     &\|F_{2,*}\|_{H^{-1}}\leq \|F_2\|_{H^{-1}}+C\nu\|\nabla w_1\|_{L^2}.
  \end{align*}

  It follows from Proposition \ref{prop:res-NB-ihom} that
  \begin{align}
    \label{19}\nu^{\f16}|k|^{\f43}\|\varphi_2\|_{L^2} \leq& C\Big(\|(1-y^2)F_{1,*}\|_{L^2}+|\nu/k|^{\f13}\|F_{1,*}\|_{L^2}\nonumber\\&+|\nu/k|^{-\f13}\max(1-|\lambda|, \nu^{\f13}|k|^{-\f13})\|{F}_{2,*}\|_{H^{-1}}+|\nu k|^{\f12}\|\partial_y\varphi_2\|_{L^2(\partial\Omega)}\Big),\\
    \leq& C\Big((\nu k^2)^{\f13}\|(1-y^2)w_1\|_{L^2}+|\nu k|^{\f12}\|\partial_y\varphi_2\|_{L^2(\partial\Omega)}\nonumber\\&+\max(1-|\lambda|, \nu^{\f13}|k|^{-\f13})\big(|\nu/k|^{-\f13}\|{F}_2\|_{H^{-1}}+ \nu^{\f23}|k|^{\f13}(\|\nabla w_1\|_{L^2}+|\nu/k|^{-\f13}\|w_1\|_{L^2})\big)\Big)\nonumber
  \end{align}
    \begin{align}
   \label{20}\nu^{\f16}|k|^{\f43}\|\nabla\varphi_2\|_{L^2} \leq& C\Big(\|F_{1,*}\|_{L^2}+|\nu/k|^{-\f13}\|{F}_{2,*}\|_{H^{-1}}+ \nu^{\f13}|k|^{\f76}\|\partial_y\varphi_2\|_{L^2(\partial\Omega)}\Big)\\
    \leq& C\Big(\nu^{\f23}|k|^{\f13}(\|\nabla w_1\|_{L^2}+|\nu/k|^{-\f13}\|w_1\|_{L^2})+|\nu/k|^{-\f13}\|{F}_2\|_{H^{-1}}\nonumber\\
    &\qquad+ \nu^{\f13}|k|^{\f76}\|\partial_y\varphi_2\|_{L^2(\partial\Omega)}\Big),\nonumber
     \end{align}
 and
    \begin{align}
    \label{21}&\|w_2\|_{L^2} \leq C\Big(\|(|k(y-\lambda)/\nu|^{\f14}+|k/\nu|^{\f16})\partial_y\varphi_2\|_{L^2(\partial\Omega)}
    +|\nu/k|^{\f16}\|\partial_y\partial_z\varphi_2\|_{L^2(\partial\Omega)}\\& \qquad\qquad+(\nu k^2)^{-\f{5}{12}}\|F_{1,*}\|_{L^2}+\nu^{-\f34} |k|^{-\f{1}{2}}\|{F}_{2,*}\|_{H^{-1}}\Big),
    \nonumber\\
     &\qquad \leq C\Big(\|(|k(y-\lambda)/\nu|^{\f14}+|k/\nu|^{\f16})\partial_y\varphi_2\|_{L^2(\partial\Omega)}
    +|\nu/k|^{\f16}\|\partial_y\partial_z\varphi_2\|_{L^2(\partial\Omega)}\nonumber\\
    & \qquad\qquad+\nu^{\f14} |k|^{-\f{1}{2}}(\|\nabla w_1\|_{L^2}+|\nu/k|^{-\f13}\|w_1\|_{L^2})+\nu^{-\f34} |k|^{-\f{1}{2}}\|{F}_2\|_{H^{-1}}\Big),
    \nonumber\\
     \label{22}\|w_2\|_{L^2(\partial\Omega)} \leq& C\Big(\|(|k(y-\lambda)/\nu|^{\f12}+|k/\nu|^{\f13})\partial_y\varphi_2\|_{L^2(\partial\Omega)}
    +\|\partial_y\partial_z\varphi_2\|_{L^2(\partial\Omega)}\nonumber\\&\quad+\nu^{-\f23} |k|^{-\f{5}{6}}\|F_{1,*}\|_{L^2}+\nu^{-1} |k|^{-\f{1}{2}}\|{F}_{2,*}\|_{H^{-1}}\Big)\\
    \leq& C\Big(\|(|k(y-\lambda)/\nu|^{\f12}+|k/\nu|^{\f13})\partial_y\varphi_2\|_{L^2(\partial\Omega)}
    +\|\partial_y\partial_z\varphi_2\|_{L^2(\partial\Omega)}\nonumber\\
    &\quad+|k|^{-\f{1}{2}}(\|\nabla w_1\|_{L^2}+|\nu/k|^{-\f13}\|w_1\|_{L^2})+\nu^{-1} |k|^{-\f{1}{2}}\|{F}_2\|_{H^{-1}}\Big).
    \nonumber
  \end{align}

Thanks to  $w_1=-ik^{-1}\chi_1F_1$,  we get by Lemma \ref{lem7.1} that
\begin{align}
   &\|\nabla w_1\|_{L^2}+|\nu/k|^{-\f13}\|w_1\|_{L^2} \leq C\nu^{-\f12}|k|^{-1}\|\nabla F_1\|_{L^2},\label{23}\\
&\|(1-y^2)w_1\|_{L^2}\leq C|k|^{-1}\|F_{1}\|_{L^2}+C\nu^{-\f16}|k|^{-\f43}\max(1-|\lambda|,\nu^{\f13}|k|^{-\f13})\|\nabla F_1\|_{L^2},\label{24}\\&\|\varphi_1\|_{L^2}\leq Ck^{-2}\big(\|F_1\|_{L^2}+\max(1-|\lambda|,\nu^{\f13}|k|^{-\f13})\|\nabla F_1\|_{L^2}\big),\label{25}\\ &\|\nabla\varphi_1\|_{L^2} \leq Ck^{-2}\|\nabla F_1\|_{L^2},\label{26}\\ & \|\partial_y\partial_z\varphi_1\|_{L^2(\partial\Omega)}\leq C\nu^{-\f13}|k|^{-\f76}\|\nabla{F}_1\|_{L^2},\label{27}\\
&(1+|k(\lambda-j)|)\|\partial_y\varphi_1\|_{L^2(\Gamma_j)}\leq C|k|^{-\f32}\|\nabla{F}_1\|_{L^2}\label{28}\\&\qquad\qquad+C|k|^{-1}\ln(2+(|k(\lambda-j)|+(\nu k^2)^{\f13})^{-1})\|F_1\|_{L_{x,z}^2L_y^{\infty}},\ \ j\in\{\pm 1\}.\nonumber
\end{align}

Using the inequality $\|F_1\|_{L_{x,z}^2L_y^{\infty}}\leq C\|{F}_1\|_{L^2}^{\f12}\|\nabla{F}_1\|_{L^2}^{\f12},$ we find from \eqref{28}  that
\begin{align}
   & \|\partial_y\varphi_1\|_{L^2(\partial\Omega)}\leq C|k|^{-\f32}\|\nabla F_1\|_{L^2}+ C(\gamma)|k|^{-\f12}(\nu k^2)^{-\gamma}\| F_1\|_{L^2},\quad \forall\gamma>0\label{29}
\end{align}
Using \eqref{28} and the fact that  for any $0<\gamma_1\leq1$,
\beno
\ln(2+(|k(\lambda-j)|+(\nu k^2)^{\f13})^{-1})\leq C+C\nu^{-\gamma_1}(|k(\lambda-j)/\nu|+|k/\nu|^{\f23})^{-\gamma_1},
\eeno
we deduce that
\begin{align*}
&(|k(\lambda-j)/\nu|^{\gamma_1}+|k/\nu|^{\f23\gamma_1})\|\partial_y\varphi_1\|_{L^2(\Gamma_j)} \\ &\leq C\nu^{-1}|k|^{-\f32}\f{|k(\lambda-j)/\nu|^{\gamma_1}+ |k/\nu|^{\f23\gamma_1}}{|k(\lambda-j)/\nu|+\nu^{-1}} \|\nabla F_1\|_{L^2}\\ &\quad+C\nu^{-1}|k|^{-1}\f{|k(\lambda-j)/\nu|^{\gamma_1}+|k/\nu|^{\f23\gamma_1}} {|k(\lambda-j)/\nu|+\nu^{-1}} \big(1+\nu^{-\gamma_1}(|k(\lambda-j)/\nu|+|k/\nu|^{\f23})^{-\gamma_1}\big)\|F_1\|_{L^2_{x,z}L^\infty_y}\\
   &\leq C\nu^{-\gamma_1}|k|^{-\f32}\|\nabla F_1\|_{L^2}+C\nu^{-\gamma_1}|k|^{-1}\|F_1\|_{L^2_{x,z}L^\infty_y}\leq C\nu^{-\gamma_1}|k|^{-\f32}\|\nabla F_1\|_{L^2},
\end{align*}
which implies that for $0<\gamma_1\le 1$,
\begin{align}\label{30}
   &\|(|k(\lambda-y)/\nu|^{\gamma_1}+|k/\nu|^{\f23\gamma_1})\partial_y\varphi_1\|_{L^2(\partial\Omega)} \leq C\nu^{-\gamma_1}|k|^{-\f32}\|\nabla F_1\|_{L^2}.
\end{align}
We infer from  \eqref{19}, \eqref{23} and \eqref{24} that
\begin{align*}
  \nu^{\f16}|k|^{\f43}\|\varphi_2\|_{L^2} \leq& C\Big((\nu k^2)^{\f13}\|(1-y^2)w_1\|_{L^2}+|\nu k|^{\f12}\|\partial_y\varphi_2\|_{L^2(\partial\Omega)}+\max(1-|\lambda|, \nu^{\f13}|k|^{-\f13})\\& \times\big(|\nu/k|^{-\f13}\|{F}_2\|_{H^{-1}}+ \nu^{\f23}|k|^{\f13}(\|\nabla w_1\|_{L^2}+|\nu/k|^{-\f13}\|w_1\|_{L^2})\big)\Big),\\
  \leq& C\Big(|\nu k|^{\f12}\|\partial_y\varphi\|_{L^2(\partial\Omega)}+ |\nu k|^{\f12}\|\partial_y\varphi_1\|_{L^2}+ |\nu/k|^{\f13}\|F_1\|_{L^2}\\&+ \max(1-|\lambda|, \nu^{\f13}|k|^{-\f13})(|\nu/k|^{-\f13}\|{F}_2\|_{H^{-1}}+ \nu^{\f16}|k|^{-\f23}\|\nabla F_1\|_{L^2})\Big),
\end{align*}
which along with with \eqref{25} and \eqref{29}($\gamma=1/3$) gives
\begin{align*}
   \nu^{\f16}|k|^{\f43}\|\varphi\|_{L^2}\leq & \nu^{\f16}|k|^{\f43}\|\varphi_1\|_{L^2} +\nu^{\f16}|k|^{\f43}\|\varphi_2\|_{L^2}\\
  \leq& C\Big(|\nu k|^{\f12}\|\partial_y\varphi\|_{L^2(\partial\Omega)}+ |\nu k|^{\f12}\|\partial_y\varphi_1\|_{L^2(\partial\Omega)}+ \nu^{\f16}|k|^{-\f23}(1+(\nu k^2)^{\f16})\|F_1\|_{L^2}\\&+ \max(1-|\lambda|, \nu^{\f13}|k|^{-\f13})\big(|\nu/k|^{-\f13}\|{F}_2\|_{H^{-1}}+ \nu^{\f16}|k|^{-\f23}\|\nabla F_1\|_{L^2}\big)\Big)\\
  \leq& C\Big(|\nu k|^{\f12}\|\partial_y\varphi\|_{L^2(\partial\Omega)}+ \nu^{\f12}|k|^{-1}\|\nabla F_1\|_{L^2}+ \nu^{\f16}|k|^{-\f23}\|F_1\|_{L^2}\\&+ \max(1-|\lambda|, \nu^{\f13}|k|^{-\f13})\big(|\nu/k|^{-\f13}\|{F}_2\|_{H^{-1}}+ \nu^{\f16}|k|^{-\f23}\|\nabla F_1\|_{L^2}\big)\Big)\\
  \leq& C\Big(|\nu k|^{\f12}\|\partial_y\varphi\|_{L^2(\partial\Omega)}+ \nu^{\f16}|k|^{-\f23}\|F_1\|_{L^2}\\&+ \max(1-|\lambda|, \nu^{\f13}|k|^{-\f13})\big(|\nu/k|^{-\f13}\|{F}_2\|_{H^{-1}}+ \nu^{\f16}|k|^{-\f23}\|\nabla F_1\|_{L^2}\big)\Big).
\end{align*}
This proves the first inequality of the proposition.

It follows from \eqref{20}, \eqref{23}, \eqref{26} and \eqref{29} with $\gamma=1/6$ that
\begin{align*}
&\nu^{\f16}|k|^{\f43}\|\nabla\varphi\|_{L^2} \leq \nu^{\f16}|k|^{\f43}\|\nabla\varphi_2\|_{L^2} +\nu^{\f16}|k|^{\f43}\|\nabla\varphi_1\|_{L^2}\\
    &\leq C\Big(\nu^{\f23}|k|^{\f13}\big(\|\nabla w_1\|_{L^2}+|\nu/k|^{-\f13}\|w_1\|_{L^2}\big)+|\nu/k|^{-\f13}\|{F}_2\|_{H^{-1}}+ \nu^{\f13}|k|^{\f76}\|\partial_y\varphi_2\|_{L^2(\partial\Omega)}\Big)\\
    &\qquad+C\nu^{\f16}|k|^{-\f23}\|\nabla F_1\|_{L^2}\\
 &\leq C\Big(\nu^{\f16}|k|^{-\f23}\|\nabla F_1\|_{L^2}+|\nu/k|^{-\f13}\|{F}_2\|_{H^{-1}}+ \nu^{\f13}|k|^{\f76}\|\partial_y\varphi\|_{L^2(\partial\Omega)} +\nu^{\f13}|k|^{\f76}\|\partial_y\varphi_1\|_{L^2(\partial\Omega)}\Big)\\
    &\leq C\Big(\nu^{\f16}|k|^{-\f23}(1+(\nu k^2)^{\f16})\|\nabla F_1\|_{L^2} +\nu^{\f16}|k|^{\f13}\|F_1\|_{L^2}+|\nu/k|^{-\f13}\|{F}_2\|_{H^{-1}}+ \nu^{\f13}|k|^{\f76}\|\partial_y\varphi\|_{L^2(\partial\Omega)} \Big)\\
    &\leq C\Big(\nu^{\f16}|k|^{-\f23}\|\nabla F_1\|_{L^2}+|\nu/k|^{-\f13}\|{F}_2\|_{H^{-1}}+ \nu^{\f13}|k|^{\f76}\|\partial_y\varphi\|_{L^2(\partial\Omega)} \Big),
\end{align*}
which gives the second inequality of the proposition.

By \eqref{21} and \eqref{23}, we have
\begin{align*}
    &\|w\|_{L^2}\leq \|w_2\|_{L^2}+\|w_1\|_{L^2}\\
   &\leq C\Big(\|(|k(y-\lambda)/\nu|^{\f14}+|k/\nu|^{\f16})\partial_y\varphi_2\|_{L^2(\partial\Omega)}
    +|\nu/k|^{\f16}\|\partial_y\partial_z\varphi_2\|_{L^2(\partial\Omega)}\\
    &\qquad+\nu^{\f14} |k|^{-\f{1}{2}}(\|\nabla w_1\|_{L^2}+|\nu/k|^{-\f13}\|w_1\|_{L^2})+\nu^{-\f34} |k|^{-\f{1}{2}}\|{F}_2\|_{H^{-1}}+\nu^{-\f16}|k|^{-\f43}\|\nabla F_1\|_{L^2}\Big)\\
    &\leq C\Big(\|(|k(y-\lambda)/\nu|^{\f14}+|k/\nu|^{\f16})\partial_y\varphi\|_{L^2(\partial\Omega)}
    +|\nu/k|^{\f16}\|\partial_y\partial_z\varphi\|_{L^2(\partial\Omega)}\\
    &\qquad+\|(|k(y-\lambda)/\nu|^{\f14}+|k/\nu|^{\f16})\partial_y\varphi_1\|_{L^2(\partial\Omega)}
    +|\nu/k|^{\f16}\|\partial_y\partial_z\varphi_1\|_{L^2(\partial\Omega)}\\
    &\qquad\quad+\nu^{-\f14} |k|^{-\f{3}{2}}(1+(\nu k^2)^{\f{1}{12}})\|\nabla F_1\|_{L^2}+\nu^{-\f34} |k|^{-\f{1}{2}}\|{F}_2\|_{H^{-1}}\Big),\end{align*}
while  by \eqref{27},
\beno
|\nu/k|^{\f16}\|\partial_y\partial_z\varphi_1\|_{L^2(\partial\Omega)}\leq  C\nu^{-\f14} |k|^{-\f{3}{2}}(\nu k^2)^{\f{1}{12}}\|\nabla F_1\|_{L^2},
\eeno
and  by \eqref{30} with $\gamma_1=1/4$,
\begin{align*}
   & \|(|k(y-\lambda)/\nu|^{\f14}+|k/\nu|^{\f16})\partial_y\varphi_1\|_{L^2(\partial\Omega)} \leq C\nu^{-\f14}|k|^{-\f32}\|\nabla F_1\|_{L^2}.
\end{align*}
Thus, we conclude that
\begin{align*}
   \|w\|_{L^2}\leq & C\Big(\|(|k(y-\lambda)/\nu|^{\f14}+|k/\nu|^{\f16})\partial_y\varphi\|_{L^2(\partial\Omega)}
    +|\nu/k|^{\f16}\|\partial_y\partial_z\varphi\|_{L^2(\partial\Omega)}\\
    &+\nu^{-\f34} |k|^{-\f{1}{2}}\|{F}_2\|_{H^{-1}}
    +\nu^{-\f14} |k|^{-\f{3}{2}}\|\nabla F_1\|_{L^2}\Big),
\end{align*}
which gives  the third inequality of the proposition.

Using the interpolation $\|w_1\|_{L^2_{x,z}L^\infty_y}\le C\|w_1\|_{L^2}^\f12\|\pa_yw_1\|_{L^2}^\f12$ and \eqref{23},  we get
\begin{align*}
   \|w\|_{L^2(\partial\Omega)}\leq& \|w_1\|_{L^2(\partial\Omega)}+ \|w_2\|_{L^2(\partial\Omega)}\leq \|w_1\|_{L^2_{x,z}L^\infty_y}+ \|w_2\|_{L^2(\partial\Omega)}\\
   \leq &C|\nu/k|^{\f16}\big(\|\partial_yw_1\|_{L^2}+|\nu/k|^{-\f13}\|w_1\|_{L^2}\big)+ \|w_2\|_{L^2(\partial\Omega)}\\
   \leq& C\nu^{-\f13}|k|^{-\f76}\|\nabla F_1\|_{L^2} +\|w_2\|_{L^2(\partial\Omega)}.
\end{align*}
Then by \eqref{22}, \eqref{23}, \eqref{27} and \eqref{30} with $\gamma_1=1/2$, we get
\begin{align*}
   \|w\|_{L^2(\partial\Omega)}\leq& C\nu^{-\f13}|k|^{-\f76}\|\nabla F_1\|_{L^2} +\|w_2\|_{L^2(\partial\Omega)} \\ \leq &  C\nu^{-\f13}|k|^{-\f76}\|\nabla F_1\|_{L^2}+ C\Big(\|(|k(y-\lambda)/\nu|^{\f12}+|k/\nu|^{\f13}) \partial_y\varphi_2\|_{L^2(\partial\Omega)}\\
    &\quad+\|\partial_y\partial_z\varphi_2\|_{L^2(\partial\Omega)} +|k|^{-\f{1}{2}}(\|\nabla w_1\|_{L^2}+|\nu/k|^{-\f13}\|w_1\|_{L^2})+\nu^{-1} |k|^{-\f{1}{2}}\|{F}_2\|_{H^{-1}}\Big)\\
    \leq &  C\nu^{-\f12}|k|^{-\f32}(1+(\nu k^2)^{\f16})\|\nabla F_1\|_{L^2}+ C\Big(\|(|k(y-\lambda)/\nu|^{\f12}+|k/\nu|^{\f13}) \partial_y\varphi\|_{L^2(\partial\Omega)}\\
    &+\|\partial_y\partial_z\varphi\|_{L^2(\partial\Omega)}+\nu^{-1} |k|^{-\f{1}{2}}\|{F}_2\|_{H^{-1}}\Big)\\
    &+C\Big(\|(|k(y-\lambda)/\nu|^{\f12}+|k/\nu|^{\f13}) \partial_y\varphi_1\|_{L^2(\partial\Omega)}
    +\|\partial_y\partial_z\varphi_1\|_{L^2(\partial\Omega)}\Big)\\
    \leq &C\nu^{-\f12}|k|^{-\f32}(1+(\nu k^2)^{\f16})\|\nabla F_1\|_{L^2}+ C\Big(\|(|k(y-\lambda)/\nu|^{\f12}+|k/\nu|^{\f13}) \partial_y\varphi\|_{L^2(\partial\Omega)}\\
    &\quad+\|\partial_y\partial_z\varphi\|_{L^2(\partial\Omega)}+\nu^{-1} |k|^{-\f{1}{2}}\|{F}_2\|_{H^{-1}}\Big)\\
    \leq &C\nu^{-\f12}|k|^{-\f32}\|\nabla F_1\|_{L^2}+ C\Big(\|(|k(y-\lambda)/\nu|^{\f12}+|k/\nu|^{\f13}) \partial_y\varphi\|_{L^2(\partial\Omega)}\\
    &\quad+\|\partial_y\partial_z\varphi\|_{L^2(\partial\Omega)}+\nu^{-1} |k|^{-\f{1}{2}}\|{F}_2\|_{H^{-1}}\Big).
\end{align*}
This proves the last inequality of the proposition.
\end{proof}

\section{Resolvent estimates when $V=y$}

In this section, we consider the Orr-Sommerfeld equation when $V=y$:
\begin{align}\label{eq:OS-Nav-y}
  \left\{
  \begin{aligned}
   &-\nu(\partial^2_y-\eta^2)w+ik(y-\lambda)w-a(\nu k^2)^{1/3}w=F,\\
   &w|_{y=\pm1}=0.
   \end{aligned}\right.
\end{align}
Here \eqref{eq:OS-Nav-y} is slightly different from one considered in \cite{CLWZ} with $k^2$ instead of $\eta^2$. We will use the results from last section to establish the resolvent estimates of \eqref{eq:OS-Nav-y}.
\smallskip

In this section, we take $\lambda\in \mathbb{R}, a\in[0,\eps_1]$. Let $u=({\partial_y\varphi,-i\eta\varphi})$ with $\varphi$ solving  $(\partial_y^2-\eta^2)\varphi=w,\ \varphi|_{y=\pm 1}=0$.

\smallskip

Notice that if $V=y$ and $(w(y),F(y))$ solves \eqref{eq:OS-Nav-y}, then $(w(y)e^{i(kx+\ell z)},\ F(y)e^{i(kx+\ell z)})$ solves \eqref{eq:OS-Nav}.
Thus, it follows from Proposition \ref{prop:res-nav-s1} that

\begin{Proposition}\label{prop:res-nav-y1}
Let $w\in H^2(I)$ be a solution of \eqref{eq:OS-Nav-y}  with $F\in L^2(I)$. Then it holds that
\begin{align*}
&{\nu^{\f16}|k|^{\f56}\|w\|_{L^1}}+\nu^{\f23}|k|^{\f13}\|w'\|_{L^2}+(\nu k^2)^{\f13}\|w\|_{L^2}+|k|\|(y-\lambda)w\|_{L^2}\leq C\|{F}\|_{L^2},\\&\nu\|w'\|_{L^2}+\nu^{\f23}|k|^{\f13}\|w\|_{L^2} \leq C\|{F}\|_{H^{-1}}.
\end{align*}
\end{Proposition}

\begin{Proposition}\label{prop:res-y-uL2}
Let $w\in H^2(I)$ be a solution of \eqref{eq:OS-Nav-y}  with $F=ikf_1+\partial_yf_2+i\ell f_3\in L^2(I)$. If $\nu\eta^3\leq|k|$, then it holds that
\begin{align*}
 \nu^{\f16}|k|^{\f56}|\eta|^{\f12}\|u\|_{L^2}\leq C\|{F}\|_{L^2},\quad (\nu |k\eta|)^{\f12}\|u\|_{L^2} \leq C\|(f_1,f_2,f_3)\|_{L^2} .
\end{align*}
\end{Proposition}

\begin{proof} It follows from Lemma \ref{lem:ell-w} and Proposition $\ref{prop:res-nav-y1}$ that
\begin{align*}
 &\nu^{\f16}|k|^{\f56}|\eta|^{\f12}\|u\|_{L^2}\leq C\nu^{\f16}|k|^{\f56}\|w\|_{L^1}\leq C\|{F}\|_{L^2},
 \end{align*}
 which gives the first inequality.

 Thanks to $\nu |k|^2\leq \nu \eta^3/|k|\leq1,\ \eta\le \delta^{-1}$. By Proposition \ref{prop:res-weak-y} and
 Lemma \ref{lem:ell-w}, we get
 \begin{align*}
 \|u\|_{L^2}^2=\langle-w,\varphi\rangle\leq& C|k|^{-1}\|(f_1,f_2,f_3)\|_{L^2}\big(\delta^{-\f32}\|\varphi\|_{L^{\infty}}+\delta^{-1}\|(\partial_y \varphi,\eta \varphi)\|_{L^2}\big)\\ \leq& C|k|^{-1}\|(f_1,f_2,f_3)\|_{L^2}\big(\delta^{-\f32}|\eta|^{-\f12}+\delta^{-1}\big)\|u\|_{L^2}\\
 \leq& C|k|^{-1}\|(f_1,f_2,f_3)\|_{L^2}\delta^{-\f32}|\eta|^{-\f12}\|u\|_{L^2},
 \end{align*}
which gives
\beno
(\nu |k\eta|)^{\f12}\|u\|_{L^2}=|k|\delta^{\f32}|\eta|^{\f12}\|u\|_{L^2} \leq C\|(f_1,f_2,f_3)\|_{L^2}.
\eeno
\end{proof}

\begin{Proposition}\label{prop:res-nav-y-uL2}
Let $w\in H^2(I)$ be a solution of \eqref{eq:OS-Nav-y} with $F=F_1+\partial_yF_2$. If  $\nu\eta^3\leq|k|$,  then it holds that
\begin{align*}
&|k\eta|\|u\|_{L^2}+\nu^{\f16} \eta^{\f12}|k|^{\f56}\|w\|_{L^2}+(\nu |k\eta|)^{\f12}\|w'\|_{L^2}\leq C\big(\nu^{-\f16}|k|^{\f16}\eta^{\f12}\|{F}_1\|_{L^2}+\nu^{-\f12}|k\eta|^{\f12}\|{F}_2\|_{L^2}\big).
\end{align*}
\end{Proposition}

\begin{proof}
We decompose $w=w_1+w_2$, where $w_1, w_2$ solves
  \begin{align*}\left\{
  \begin{aligned}
   &-\nu(\partial_y^2-\eta^2)w_1+ik(y-\lambda)w_1-a(\nu k^2)^{\f13}w_1=F_1,\\
   &-\nu(\partial_y^2-\eta^2)w_2+ik(y-\lambda)w_2-a(\nu k^2)^{\f13}w_2=\partial_yF_2,\\
   &w_1|_{y=\pm1}=w_2|_{y=\pm1}=0.\end{aligned}
  \right.
  \end{align*}
Then the proposition follows from Proposition \ref{prop:res-nav-y1} and Proposition \ref{prop:res-y-uL2}.
\end{proof}
\smallskip

Applying Proposition \ref{prop:res-weak}  to the case when
\beno
&&V=y,\quad w=w_0(y)e^{i(kx+\ell z)},\quad f=f_0(y)e^{i(kx+\ell z)},\\
&&F=F_0(y)e^{i(kx+\ell z)}=\text{div}\big[(f_1(y),f_2(y),f_3(y))e^{i(kx+\ell z)})\big],
\eeno
we can deduce that

\begin{Proposition}\label{prop:res-weak-y}
Let $w\in H^2(I)$ be a solution of \eqref{eq:OS-Nav-y} with $F=ikf_1+\partial_yf_2+i\ell f_3\in L^2(I)$ and $\nu k^2\leq 1$.Then it holds that for $f\in H_0^1(I)$,
\begin{align*}
 |\langle w,f\rangle|\leq & C|k|^{-1}\|(f_1,f_2,f_3)\|_{L^2}\Big(\big(\delta^{-\f32}+\delta^{-\f12}\eta\big)\|f\|_{L^{\infty}}+\delta^{-1}\|(\partial_y f,\eta f)\|_{L^2}\Big).
  \end{align*}
\end{Proposition}

Applying Proposition \ref{prop:res-weak-b} to the case when $ V=y, F=\partial_yf_2(y)e^{ikx+i\ell z}$,  $w=w_0(y)e^{ikx+i\ell z}$, and applying Proposition \ref{prop:res-nav-b1} to the case when $V=y,$ $F=f_1(y)e^{ikx+i\ell z},\ w=w_0(y)e^{ikx+i\ell z}$,  we can deduce that

\begin{Proposition}\label{cor5}
Let $w\in H^2(I)$ be a solution of \eqref{eq:OS-Nav-y} and $\nu k^2\le 1$. If $F=\partial_yf_2\in L^2(I),$ then we have
\begin{align*}
   \big(|k(j-\lambda)|+1\big)^{\f34}|\partial_y\varphi(j)|\leq C|\nu k|^{-\f12}\|f_2\|_{L^2},\quad j=\pm 1.
  \end{align*}
If $F=f_1\in L^2(I)$, then we have
\begin{align*}
    \big(|k(j-\lambda)|+1\big)|\partial_y\varphi(j)|\leq C\nu^{-\f16}|k|^{-\f56}\|f_1\|_{L^2},\quad j=\pm 1.
  \end{align*}
In particular, if $F=f_1+\partial_yf_2$, then we have
\begin{align*}
   \big(|k(j-\lambda)|+1\big)^{\f34}|\partial_y\varphi(j)|\leq C|\nu k|^{-\f12}\|f_2\|_{L^2}+C\nu^{-\f16}|k|^{-\f56}\|f_1\|_{L^2}.
  \end{align*}
\end{Proposition}

\begin{Proposition}\label{prop:res-damping}
Let $w\in H^2(I)$ be a solution of \eqref{eq:OS-Nav-y} with $F\in H^1(I)$ and $F(\pm1)=0$. If  $\nu\eta^3\leq|k|$. then it holds that
  \begin{align*}
     |k\eta|\|u\|_{L^2}+\nu^{\f16} \eta^{\f12}|k|^{\f56}\|w\|_{L^2}+(\nu |k\eta|)^{\f12}\|w'\|_{L^2} &\leq C\big(\|F'\|_{L^2}+|\eta|\|F\|_{L^2}\big).
  \end{align*}
\end{Proposition}

\begin{proof}We write $w=w_1+w_2$, where $w_1=\chi_1 F/(ik)$ with $ \chi_1(y)=(y-\lambda-i\delta)^{-1}$.
Let $\varphi_j=(\partial_y^2-\eta^2)^{-1}w_j,\ u_j=({\partial_y\varphi_j,-i\eta\varphi_j}),\ j=1,2.$ Without lose of generality, we may assume $k>0$. It is easy to see that
\begin{align*}
&\|\chi_1\|_{L^2}\lesssim\delta^{-\f12},\quad \|\chi_1\|_{L^{\infty}}\lesssim\delta^{-1},\quad \|\chi_1'\|_{L^2}\lesssim\delta^{-\f32},\\
&ik(y-\lambda)w_1+(\nu k^2)^{1/3}w_1=F,\\
     &-\nu(\partial^2_y-\eta^2)w_2+ik(y-\lambda)w_2-a(\nu k^2)^{1/3}w_2\\&\quad=\nu\partial^2_yw_1+((a+1)(\nu k^2)^{1/3}-\nu\eta^2)w_1,\\
 &w_1(\pm 1)=0,\quad w_2(\pm 1)=0.
 \end{align*}
 Then it follows from Proposition \ref{prop:res-nav-y-uL2}  that
 \begin{align*}
&|k\eta|\|u_2\|_{L^2}+\nu^{\f16} \eta^{\f12}|k|^{\f56}\|w_2\|_{L^2}+(\nu |k\eta|)^{\f12}\|w_2'\|_{L^2}\\ &\leq C\Big(\nu^{-\f16}|k|^{\f16}\eta^{\f12}\|((a+1)(\nu k^2)^{1/3}-\nu\eta^2)w_1\|_{L^2}+\nu^{-\f12}|k\eta|^{\f12}\|\nu\partial_yw_1\|_{L^2}\Big)\\ &\leq C\Big(\nu^{-\f16}|k|^{\f16}\eta^{\f12}(\nu k^2)^{1/3}\|w_1\|_{L^2}+(\nu |k\eta|)^{\f12}\|w_1'\|_{L^2}\Big).
\end{align*}

As $\nu^{-\f16}|k|^{\f16}\eta^{\f12}(\nu k^2)^{1/3}=\nu^{\f16} \eta^{\f12}|k|^{\f56} $ and $w=w_1+w_2,\ u=u_1+u_2$, we conclude that
\begin{align*}
&|k\eta|\|u\|_{L^2}+\nu^{\f16} \eta^{\f12}|k|^{\f56}\|w\|_{L^2}+(\nu |k\eta|)^{\f12}\|w'\|_{L^2}\\ &\leq C\Big(|k\eta|\|u_1\|_{L^2}+\nu^{\f16} \eta^{\f12}|k|^{\f56}\|w_1\|_{L^2}+(\nu |k\eta|)^{\f12}\|w_1'\|_{L^2}\Big).
\end{align*}
Let $G=\|F'\|_{L^2} +\eta\|F\|_{L^2}$. Using the fact that $\|F\|_{L^\infty}\le C|\eta|^{-\f12}G$ by Lemma \ref{lem:ell-w}, we get
\begin{align*}
&\|w_1\|_{L^2}=\|\chi_1 F/(ik)\|_{L^2}\leq\|\chi_1\|_{L^2}\|F\|_{L^\infty}/|k|\leq C\delta^{-\f12}|\eta|^{-\f12}G/|k|=CG/(\nu^{\f16} \eta^{\f12}|k|^{\f56}),
\end{align*}
and by $\delta\le \eta^{-1}$,
\begin{align*}
\|w_1'\|_{L^2}=&\|(\chi_1 F)'/(ik)\|_{L^2}\leq \big(\|\chi_1'\|_{L^2}\|F\|_{L^\infty}+\|\chi_1\|_{L^\infty}\|F'\|_{L^2}\big)/|k|\\ \leq& C\big(\delta^{-\f32}|\eta|^{-\f12}G+\delta^{-1}G\big)/|k|\leq C\delta^{-\f32}|\eta|^{-\f12}G/|k|=CG/(\nu |k\eta|)^{\f12}.
\end{align*}
Thanks to the definition of $\varphi_1,\ w_1$, we have
\beno
(y-\lambda-i\delta)(\varphi_1''-\eta^2\varphi_1)=(y-\lambda-i\delta)w_1=F_1/(ik),\ \varphi_1(\pm 1)=0.
\eeno
Then it follows from  Proposition \ref{prop:LAP}  that
\begin{align*}
\|u_1\|_{L^2}\leq\|\partial_y\varphi_1\|_{L^2}+\eta\|\varphi_1\|_{L^2}\leq C\eta^{-1}\big(\|\partial_yF_1/(ik)\|_{L^2}+\eta\|F_1/(ik)\|_{L^2}\big)= C\eta^{-1}G/|k|.
\end{align*}
Summing up we conclude that
\begin{align*}
&|k\eta|\|u\|_{L^2}+\nu^{\f16} \eta^{\f12}|k|^{\f56}\|w\|_{L^2}+(\nu |k\eta|)^{\f12}\|w'\|_{L^2}\\ &\leq C\big(|k\eta|\|u_1\|_{L^2}+\nu^{\f16} \eta^{\f12}|k|^{\f56}\|w_1\|_{L^2}+(\nu |k\eta|)^{\f12}\|w_1'\|_{L^2}\big)\\&
\leq CG= C\big(\|F'\|_{L^2}+|\eta|\|F\|_{L^2}\big).
\end{align*}

This finishes the proof of the proposition.
\end{proof}

Now we consider the linear equation with non-vanishing boundary Neumann data:
\begin{align}\label{nonslip1}
\left\{
  \begin{aligned}
   & -\nu(\partial_y^2-\eta^2)w+ik(y-\lambda)w-a(\nu k^2)^{\f13}w=F,\\
   &(\partial_y^2-\eta^2)\varphi=w,\ \varphi|_{y=\pm1}=0.
  \end{aligned}\right.
\end{align}

Applying Proposition \ref{prop:res-NB-ihom}  to the case $F_1=f_1(y)e^{ikx+i\ell z},\ F_2=\partial_yf_2(y)e^{ikx+i\ell z}$, $V=y,$ and $w=w_0(y)e^{ikx+i\ell z}$,  we deduce that

\begin{Proposition}\label{prop:res-NB-y-w}
Let $w\in H^2(I)$ be a solution of \eqref{nonslip1} with $F=f_1+\p_y f_2$. If $\nu\eta^3\leq|k|,$  then it holds that
\begin{align*}
    \|w\|_{L^2}\leq& C\sum_{j\in\{\pm1\}}\big(|k(j-\lambda)/\nu|^{\f14}+|k/\nu|^{\f16}+\eta|\nu/k|^{\f16}\big)|\partial_y\varphi(j)|\\&\quad+C\big((\nu k^2)^{-\f{5}{12}}\|f_1\|_{L^2}+\nu^{-\f34} |k|^{-\f{1}{2}}\|f_2\|_{L^2}\big)\\ \leq& C\nu^{-\f14}\big((|k(1-\lambda)|+1)^{\f14}|\partial_y\varphi(1)|+(|k(1+\lambda)|+1)^{\f14}|\partial_y\varphi(-1)|\big)\\&+C\big((\nu k^2)^{-\f{5}{12}}\|f_1\|_{L^2}+\nu^{-\f34} |k|^{-\f{1}{2}}\|f_2\|_{L^2}\big).
  \end{align*}
\end{Proposition}

\begin{Proposition}\label{prop:res-NB-y-u}
Let $w\in H^2(I)$ be a solution of \eqref{nonslip1} with $F=f_1+\partial_yf_2$. If $\nu\eta^3\leq|k|$, then it holds that
\begin{align*}
&\eta^{\f12}\|u\|_{L^2}\le C\Big(\nu^{-\f12}|k|^{-\f12}\|f_2\|_{L^2}+\nu^{-\f16}|k|^{-\f56}\|f_1\|_{L^2}+|\partial_y\varphi(1)|+|\partial_y\varphi(-1)|\Big),
\end{align*}
where $u=(\pa_y\varphi, {-i\eta\varphi})$.
\end{Proposition}

\begin{proof}
Let $\big(w_{Na},\ \varphi_{Na}\big)$ solve
\begin{align*}
     \left\{\begin{aligned}
     & -\nu(\partial_y^2-\eta^2)w_{Na}+ik(y-\lambda)w_{Na}-a(\nu k^2)^{\f13}w_{Na}=F,\\
   &(\partial_y^2-\eta^2)\varphi_{Na}=w_{Na},\ w_{Na}|_{y=\pm1}=\varphi_{Na}|_{y=\pm1}=0,
     \end{aligned}\right.
  \end{align*} and $c_1=\varphi'(1)-\varphi_{Na}'(1),\ c_2=\varphi'(-1)-\varphi_{Na}'(-1)$.
By Proposition \ref{cor5}, we have
\begin{align*}
  |c_1|+ |c_2|\lesssim \nu^{-\f12}|k|^{-\f12}\|f_2\|_{L^2}+\nu^{-\f16}|k|^{-\f56}\|f_1\|_{L^2}+|\partial_y\varphi(1)|+|\partial_y\varphi(-1)|.
\end{align*}
Then we have $w=w_{Na}+c_1w_{1,\ell}+c_2w_{2,\ell}$ and
\begin{align*}
u=u_{Na}+c_1u_{1,\ell}+c_2u_{2,\ell},
\end{align*}
where $u_{Na}=(\partial_y\varphi_{Na},-i\eta\varphi_{Na}) $ and $u_{j,l}=(\partial_y\varphi_{j,\ell},-i\eta\varphi_{j,\ell})$ and $(\partial^{2}_{y}-\eta^2)\varphi_{j,\ell}=w_{j,\ell}$ with $\varphi_{j,\ell}(\pm 1)=0,\ j=1,2$.

Now we infer from  Proposition \ref{prop:res-nav-y-uL2}, Lemma \ref{lem:ell-w} and Proposition \ref{prop:w12-bounds} that
\begin{align*}
  \|u\|_{L^2} &\leq\|u_{Na}\|_{L^2}+|c_1|\|u_{1,\ell}\|_{L^2}+|c_2|\|u_{2,\ell}\|_{L^2}\\&\leq C\big(|\nu k\eta|^{-\f12}\|f_2\|_{L^2}+\nu^{-\f16}|k|^{-\f56}|\eta|^{-\f12}\|f_1\|_{L^2})+C|\eta|^{-\f12}(|c_1|\|w_{1,\ell}\|_{L^1}+|c_2|\|w_{2,\ell}\|_{L^1}\big)
  \\&\leq C|\eta|^{-\f12}\big(\nu^{-\f12}|k|^{-\f12}\|f_2\|_{L^2}+\nu^{-\f16}|k|^{-\f56}\|f_1\|_{L^2}+|\partial_y\varphi(1)|
  +|\partial_y\varphi(-1)|\big).
\end{align*}

The proof is completed.
\end{proof}

\section{Resolvent estimates for the simplified linearized  NS system}

In this section, we prove the resolvent estimates for a slightly simplified linearized NS system when $\nu k^2\le 1$:
\begin{align}\label{eq:LNS-full-s2}
  \left\{
  \begin{aligned}
   &-\nu\Delta  W+ik(V(y,z)-\lambda) W-a(\nu k^2)^{1/3} W+(\partial_y+\kappa\partial_z)\big(p^{L(0)}+p^{L(1)}\big)\\&\qquad+G+\nu(\Delta \kappa) U+2\nu\nabla\kappa\cdot\nabla U=0,\\&-\nu\Delta U+ik(V(y,z)-\lambda)U-a(\nu k^2)^{1/3}U+\partial_zp^{L(0)}=0,\\
   &\Delta p^{L(0)}=-2ik\partial_yV W,\quad \partial_yp^{L(0)}|_{y=\pm1}=0,\quad \Delta p^{L(1)}=0,\\& W|_{y=\pm1}=\partial_y W|_{y=\pm1}=U|_{y=\pm1}=0,
   \end{aligned}\right.
\end{align}
where
\beno
\partial_x W=ik W,\ \partial_xU=ikU,\ \partial_xp^{L(0)}=ikp^{L(0)},\ \partial_xp^{L(1)}=ikp^{L(1)}.
\eeno

In this section, we always assume $\nu k^2\le 1$, $\la\in \R, a\in [0,\eps_1]$ and $V$ satisfies \eqref{ass:V}.

\begin{Proposition}\label{prop:res-full-s2}
Let $W\in H^4(\Omega)$ and $U\in H^2(\Omega)$ be a solution of \eqref{eq:LNS-full-s2}. Then it holds that
\begin{align*}
&\nu^{\f{1}{3}}\big(\|\partial_x^2U\|_{L^2}^2+\|\partial_x(\partial_z-\kappa\partial_y)U\|_{L^2}^2\big)+\nu\big(\|\nabla\partial_x^2U\|_{L^2}^2+
\|\nabla\partial_x(\partial_z-\kappa\partial_y)U\|_{L^2}^2\big)
\\&\quad+\nu^{\f{1}{3}} \|\partial_x\nabla W\|_{L^2}^2+\nu\|\partial_x\Delta W\|_{L^2}^2+\nu^{\f{5}{3}} \|\partial_x\Delta U\|_{L^2}^2+\nu^{-1}\|\partial_x\nabla p^{L(1)}\|_{L^2}^2\leq C\nu^{-1}\|\nabla G\|_{L^2}^2.
\end{align*}
\end{Proposition}

\subsection{Resolvent estimates for a toy model}
Let us first study the following toy model:
\begin{align}\label{eq:LNS-toy}
  \left\{
  \begin{aligned}
   &-\nu\Delta  W+ik(V(y,z)-\lambda) W-a(\nu k^2)^{1/3} W+(\partial_y+\kappa\partial_z)p^{L}=F,\\
   &\Delta p^{L1}=-2ik\partial_yV W,\  W|_{y=\pm1}=0,\ \partial_x W=ik W.
   \end{aligned}\right.
\end{align}
We can decompose $p^{L}=p^{L(0)}+p^{L(1)}$, where
  \begin{align*}
     &\Delta p^{L(0)}=-2ik\partial_yV W,\quad \partial_yp^{L(0)}|_{y=\pm1}=0,\\
     &\Delta p^{L(1)}=0,\quad (\partial_yp^{L(1)}-\nu\Delta W-F)|_{y=\pm1}=0.
  \end{align*}

\begin{Proposition}\label{prop:res-toy}
Let $ W\in H^4(\Omega)$ be a solution of \eqref{eq:LNS-toy}. Then it holds that
\begin{align*}
&\nu^{\f12}|k|\|\nabla  W\|_{L^2} +\nu^{\f34}|k|^{\f12}\|\Delta W\|_{L^2} +\nu|k|^{\f12}\|\Delta W\|_{L^2(\partial\Omega)} +\|\partial_x\nabla p^{L(1)}\|_{L^2}\\
&\quad\leq C\Big(\|\nabla F\|_{L^2} +\nu^{\f12}|k|^{\f12}\|(|k(y-\lambda)|+1)^{\f12}\partial_y W\|_{L^2(\partial\Omega)} +\nu^{\f{11}{12}}|k|^{\f13}\|\partial_y\partial_z W\|_{L^2(\partial \Omega)}\Big),
\end{align*}
and
\begin{align*}
\nu^{\f12}|k|\| W\|_{L^2}\leq& C \max(1-|\lambda|,|\nu/k|^{\f13})\Big(\|\nabla F\|_{L^2}+\nu^{\f12}|k|^{-\f12} \|(|k(y-\lambda)|+1)^{\f12}\partial_y W\|_{L^2(\partial\Omega)} \\
&\quad+\nu^{\f{11}{12}}|k|^{-\f23}\|\partial_y\partial_z W\|_{L^2(\partial \Omega)}\Big) +C\nu^{\f56}|k|^{\f16}\|\partial_y W\|_{L^2(\partial\Omega)}.
\end{align*}
\end{Proposition}

\begin{proof}
We introduce
\beno
&&h_1(y,z)=\Delta V-2(\partial_y+\kappa\partial_z)\partial_yV,\\
 &&h_2(x,y,z)=2\nabla\kappa\cdot\nabla\partial_zp^{L(0)}+(\Delta \kappa)\partial_zp^{L(0)},\\
 &&h_3(x,y,z)=2\nabla\kappa\cdot\nabla\partial_zp^{L(1)}+(\Delta \kappa)\partial_zp^{L(1)}.
 \eeno
Let $w=\Delta W$, which satisfies
  \begin{align}
    \left\{\begin{aligned}
    &-\nu\Delta w+ik(V-\lambda)w-a(\nu k^2)^{\f13}w=-ikh_1 W-h_2-h_3+\Delta F,\\
    & W|_{y=\pm1}=0,\ \partial_x W=ik W.
    \end{aligned}\right.
  \end{align}
Thanks to  $\|\kappa\|_{H^3}\leq C\varepsilon_0$, it is easy to see that
  \begin{align*}
     & \|h_1\|_{H^2}\leq C\|\nabla^2 V\|_{H^2}+\|\kappa\partial_z\partial_yV\|_{H^2}\leq C\|(V-y)\|_{H^4}+C\|\kappa\|_{H^2}\|\partial_z\partial_yV\|_{H^2}\leq C\varepsilon_0,\\
     &\|h_2\|_{H^1}\leq C\|\nabla\kappa\|_{H^2}\|\nabla\partial_z p^{L(0)}\|_{H^1}+C\|\Delta\kappa\|_{H^1}\|\partial_zp^{L(0)}\|_{H^2}\leq C\|\kappa\|_{H^3}\|\Delta p^{L(0)}\|_{H^1}\\
     &\qquad\quad\leq C\varepsilon_0|k|\|\partial_yV\|_{H^2}\| W\|_{H^1}\leq C\varepsilon_0|k|\|\nabla W\|_{L^2},\\
     &\|h_2\|_{L^2}\leq 2\|\nabla\kappa\cdot\nabla\partial_zp^{L(0)}\|_{L^2}+\|(\Delta \kappa)\partial_zp^{L(0)}\|_{L^2}\leq 2\|\nabla\kappa\|_{L^\infty}\|\nabla\partial_zp^{L(0)}\|_{L^2}\\
     &\qquad\qquad+C\|\Delta\kappa\|_{H^1}\|\partial_zp^{L(0)}\|_{H^1}\leq C\varepsilon\|\Delta p^{L(0)}\|_{L^2}\leq C\varepsilon_0|k|\| W\|_{L^2}.
  \end{align*}
 and for any $f\in H^1$,
  \begin{align*}
     |\langle (\Delta \kappa)\partial_zp^{L(1)},f\rangle|\leq& \|\partial_zp^{L(1)}\|_{L^2}\|(\Delta \kappa)f\|_{L^2}\leq \|\nabla p^{L(1)}\|_{L^2}\|\Delta\kappa\|_{H^1}\|f\|_{H^1}\\
     \leq& C\varepsilon_0\|\nabla p^{L(1)}\|_{L^2}\|f\|_{H^1},
  \end{align*}
which gives
\beno
\|(\Delta \kappa)\partial_zp^{L(1)}\|_{H^{-1}}\leq C\varepsilon_0\|\nabla p^{L(1)}\|_{L^2}.
\eeno
Then  we have
  \begin{align}
     &\|h_2\|_{H^1}+\|\nabla(kh_1 W)\|_{L^2}\leq C\varepsilon_0|k|\|\nabla W\|_{L^2}+ C|k|\|h_1\|_{H^2}\| W\|_{H^1}\leq C\varepsilon_0|k|\|\nabla W\|_{L^2},\label{est: h2 H^1}\\
     &\|h_2\|_{L^2}+\|kh_1 W\|_{L^2}\leq C\varepsilon_0|k|\| W\|_{L^2},\label{est: h2 L^2}\\
      &\|h_3\|_{H^{-1}} \leq 2\|\text{div}(\partial_zp^{L(1)}\nabla\kappa)\|_{H^{-1}}+\|(\Delta \kappa)\partial_zp^{L(1)}\|_{H^{-1}}\label{est: h3 H-1}\\
      &\qquad\quad\leq C\big(\|\partial_zp^{L(1)}\nabla\kappa\|_{L^2}+\varepsilon_0\|\nabla p^{L(1)}\|_{L^2}\big)\leq C\varepsilon_0\|\nabla p^{L(1)}\|_{L^2}.\nonumber
  \end{align}
 We get by integration by parts that
  \begin{align*}
    \|\nabla p^{L(1)}\|_{L^2}^2 &\leq|\langle \Delta p^{L(1)},p^{L(1)}\rangle|+|\langle \partial_y p^{L(1)},p^{L(1)}\rangle_{\partial\Omega}|\leq \|\partial_yp^{L(1)}\|_{L^2(\partial\Omega)}\|p^{L(1)}\|_{L^2(\partial\Omega)}\\
    &\leq \|\partial_yp^{L(1)}\|_{L^2(\partial\Omega)}\|p^{L(1)}\|_{L^2_{x,z}L^\infty_y}\leq C|k|^{-\f12}\|\partial_yp^{L(1)}\|_{L^2(\partial\Omega)}\|\nabla p^{L(1)}\|_{L^2},
  \end{align*}
  then by $(\partial_yp^{L(1)}-\nu\Delta W-F)|_{\partial\Omega}=0$, we have
  \begin{align*}
  \|\partial_x\nabla p^{L(1)}\|_{L^2}&\leq C|k|^{\f12}\|\partial_yp^{L(1)}\|_{L^2(\partial\Omega)}\leq C|k|^{\f12}\big(\nu\|\Delta W\|_{L^2(\partial\Omega)}+\|F\|_{L^2(\partial\Omega)}\big)\\
  &\leq C\nu|k|^{\f12}\|\Delta W\|_{L^2(\partial\Omega)}+C\|\nabla F\|_{L^2}.
\end{align*}

Now it follows from Proposition \ref{prop:res-damp}, \eqref{est: h2 H^1} and \eqref{est: h3 H-1} that
  \begin{align*}
      &\nu^{\f12}|k|\|\nabla  W\|_{L^2} +\nu^{\f34}|k|^{\f12}\|\Delta W\|_{L^2} +\nu|k|^{\f12}\|\Delta W\|_{L^2(\partial\Omega)} +\|\partial_x \nabla p^{L(1)}\|_{L^2}\\
      &\leq \nu^{\f12}|k|\|\nabla  W\|_{L^2} +\nu^{\f34}|k|^{\f12}\|\Delta W\|_{L^2} +C\nu|k|^{\f12}\|\Delta W\|_{L^2(\partial\Omega)} +C\|\nabla F\|_{L^2}\\
      &\leq C\nu^{\f12}|k|^{-1}\big(\|\nabla(kh_1 W)\|_{L^2}+\|\nabla h_2\|_{L^2}\big)+C\big(\|\Delta F\|_{H^{-1}}+ \|h_3\|_{H^{-1}}\big)+C\nu^{\f23}|k|^{\f56}\|\partial_y W\|_{L^2(\partial\Omega)}\\
     &\quad+C\nu^{\f12}|k|^{\f12}\|(|k(y-\lambda)|^{\f14}+(\nu k^2)^{\f{1}{12}})\partial_y W\|_{L^2(\partial\Omega)} +C\nu^{\f{11}{12}}|k|^{\f13}\|\partial_y\partial_z W\|_{L^2(\partial \Omega)}\\
     &\quad+ C\nu^{\f12}|k|^{\f12}\|(|k(y-\lambda)|^{\f12}+(\nu k^2)^{\f{1}{6}})\partial_y W\|_{L^2(\partial\Omega)} + C\nu|k|^{\f12}\|\partial_y\partial_z W\|_{L^2(\partial\Omega)} +C\|\nabla F\|_{L^2}\\
     &\leq C\varepsilon_0\big(\nu^{\f12}\|\nabla W\|_{L^2}+\|\nabla p^{L(1)}\|_{L^2}\big)+C\|\nabla F\|_{L^2} +C\nu^{\f12}|k|^{\f12}\|(|k(y-\lambda)|+1)^{\f12}\partial_y W\|_{L^2(\partial\Omega)} \\&\quad+C\nu^{\f{11}{12}}|k|^{\f13}\|\partial_y\partial_z W\|_{L^2(\partial \Omega)}.
  \end{align*}
 Taking $\veps_0$ small enough so that  $C\varepsilon_0\leq 1/2$,  we conclude that
  \begin{align}\label{est: nabla Delta varphi nabla p}
     &\nu^{\f12}|k|\|\nabla  W\|_{L^2} +\nu^{\f34}|k|^{\f12}\|\Delta W\|_{L^2} +\nu|k|^{\f12}\|\Delta W\|_{L^2(\partial\Omega)} +\|\partial_x\nabla p^{L(1)}\|_{L^2}\\
     &\quad\leq C\Big(\|\nabla F\|_{L^2} +\nu^{\f12}|k|^{\f12}\|(|k(y-\lambda)|+1)^{\f12}\partial_y W\|_{L^2(\partial\Omega)} +\nu^{\f{11}{12}}|k|^{\f13}\|\partial_y\partial_z W\|_{L^2(\partial \Omega)}\Big).\nonumber
  \end{align}
  By Proposition \ref{prop:res-damp}, \eqref{est: h2 H^1}, \eqref{est: h2 L^2} and \eqref{est: nabla Delta varphi nabla p}, we get
  \begin{align*}
     \nu^{\f12}|k|\| W\|_{L^2}\leq& C\nu^{\f12}|k|^{-1}\max(1-|\lambda|,|\nu/k|^{\f13})\big(\|\nabla(kh_1 W)\|_{L^2}+\|\nabla h_2\|_{L^2}\big)+C\nu^{\f12}|k|^{-1}\times\\
     &\quad\big(\|kh_1 W\|_{L^2}+\|h_2\|_{L^2}\big)+C\max(1-|\lambda|,|\nu/k|^{\f13})\big(\|\Delta F\|_{H^{-1}}+\|h_3\|_{H^{-1}}\big)\\
     &\quad+C\nu^{\f56}|k|^{\f16}\|\partial_y W\|_{L^2(\partial\Omega)}\\
     &\leq C\max(1-|\lambda|,|\nu/k|^{\f13})(\nu^{\f12}\|\nabla W\|_{L^2}+\|\nabla F\|_{L^2} +|k|^{-1}\|\partial_x\nabla p^{L(1)}\|_{L^2})\\
     &\quad+ C\varepsilon_0\nu^{\f12}\| W\|_{L^2} +C\nu^{\f56}|k|^{\f16}\|\partial_y W\|_{L^2(\partial\Omega)}\\
     \leq& C \max(1-|\lambda|,|\nu/k|^{\f13})\Big(\|\nabla F\|_{L^2}+\nu^{\f12}|k|^{-\f12} \|(|k(y-\lambda)|+1)^{\f12}\partial_y W\|_{L^2(\partial\Omega)} \\
     &\quad+\nu^{\f{11}{12}}|k|^{-\f23}\|\partial_y\partial_z W\|_{L^2(\partial \Omega)}\Big)+ C\varepsilon_0\nu^{\f12}\| W\|_{L^2} +C\nu^{\f56}|k|^{\f16}\|\partial_y W\|_{L^2(\partial\Omega)}.
  \end{align*}
which gives by taking $\veps_0$ small enough so that  $C\varepsilon_0\leq 1/2$ that
  \begin{align*}
  \nu^{\f12}|k|\| W\|_{L^2}&\leq C \max(1-|\lambda|,|\nu/k|^{\f13})\Big(\|\nabla F\|_{L^2}+\nu^{\f12}|k|^{-\f12} \|(|k(y-\lambda)|+1)^{\f12}\partial_y W\|_{L^2(\partial\Omega)} \\
     &\quad+\nu^{\f{11}{12}}|k|^{-\f23}\|\partial_y\partial_z W\|_{L^2(\partial \Omega)}\Big) +C\nu^{\f56}|k|^{\f16}\|\partial_y W\|_{L^2(\partial\Omega)}.
  \end{align*}

This completes the proof of the proposition.
\end{proof}

\subsection{Proof of  Proposition \ref{prop:res-full-s2}  when $|\la|\ge 1-\nu^\f13|k|^{-\f13}$}
Let
\beno
G_3=G_1+\nu(\Delta \kappa)U+2\nu\nabla\kappa\cdot\nabla U.
\eeno
It follows from Proposition \ref{prop:res-toy} that
\begin{align*}
&\nu|k|\|\Delta W\|_{L^2}+\nu^{\f{1}{2}}|k|\|\nabla W\|_{L^2}\leq \nu^{\f34}|k|^{\f12}\|\Delta W\|_{L^2}+\nu^{\f{1}{2}}|k|\|\nabla W\|_{L^2}\leq C\|\nabla G_3\|_{L^2},\\
&\nu^{\f{1}{2}}|k|\|W\|_{L^2}\leq C\nu^{\f13}|k|^{-\f13}\|\nabla G_3\|_{L^2},\\
&\|\partial_x\nabla p^{L(1)}\|_{L^2}\leq C\|\nabla G_3\|_{L^2},
\end{align*}
where we used $|\la|\ge 1-\nu^\f13|k|^{-\f13}$ for the second inequality.

Since $\Delta p^{L(0)}=-2ik\partial_yVW,\ \partial_yp^{L(0)}|_{y=\pm1}=0$ and $\|\nabla V\|_{L^\infty}+\|\nabla^2 V\|_{L^\infty}\leq C$, we can deduce that
\begin{align}
&\| \nabla\partial_zp^{L(0)}\|_{L^2}\leq C\| \Delta p^{L(0)}\|_{L^2}\leq C|k|\|W\|_{L^2}\leq C\nu^{-\f16}|k|^{-\f13}\|\nabla G_3\|_{L^2},\label{eq: nabla pz p}\\&\| \partial_zp^{L(0)}\|_{H^2}\leq C\| \Delta \partial_zp^{L(0)}\|_{L^2}\leq C|k|\| \nabla W\|_{L^2}\leq C\nu^{-\f12}\|\nabla G_3\|_{L^2}.\label{eq: partial z p H2}
\end{align}
By Proposition \ref{prop:res-nav-s1} and \eqref{eq: nabla pz p}, we have
\begin{align*}
&\nu\|\Delta U\|_{L^2}\leq C\nu^{\f16}|k|^{-\f23}\|\nabla \partial_zp^{L(0)}\|_{L^2}\leq C|k|^{-1}\|\nabla G_3\|_{L^2}.
\end{align*}
Thanks to $ \|\kappa\|_{H^3}\leq C\varepsilon_0$, we have
\begin{align*}
\|\nabla G_3\|_{L^2}\leq&\|\nabla G_1\|_{L^2}+C\nu\|\Delta \kappa\|_{H^1} \|U\|_{H^2}+C\nu\|\nabla\kappa\|_{H^2}\|\nabla U\|_{H^1}\\ \leq&\|\nabla G_1\|_{L^2}+C\nu\varepsilon_0 \|\Delta U\|_{L^2}\leq\|\nabla G\|_{L^2}+C\varepsilon_0|k|^{-1} \|\nabla G_3\|_{L^2}.
\end{align*}
Taking $\varepsilon_0$ small enough so that $C\varepsilon_0<1/2 $, we conclude that\begin{align*}
&\|\nabla G_3\|_{L^2}\leq2\|\nabla G_1\|_{L^2},\quad \nu\|\Delta U\|_{L^2}\leq C|k|^{-1} \|\nabla G_1\|_{L^2}.
\end{align*}
Summing up, we obtain
\begin{align*}
   &\|\partial_x\nabla W\|^2_{L^2} +\nu\|\partial_x\Delta W\|_{L^2}^2+\nu\|\partial_x\Delta U\|^2_{L^2}  +\nu^{-1}\|\partial_x\nabla p^{L(1)}\|^2_{L^2} \leq C\nu^{-1}\|\nabla G_1\|^2_{L^2}.
\end{align*}
By Proposition \ref{prop:res-nav-good} and \eqref{eq: partial z p H2}, we have\begin{align*}
&\nu^{\f{1}{3}}\big(\|\partial_x^2U\|_{L^2}^2+\|\partial_x(\partial_z-\kappa\partial_y)U\|_{L^2}^2)+\nu\big(\|\nabla\partial_x^2U\|_{L^2}^2+\|\nabla\partial_x (\partial_z-\kappa\partial_y)U\|_{L^2}^2\big)\\ &\leq C\| \partial_zp^{L(0)}\|_{H^2}^2\leq C\nu^{-1}\|\nabla G_3\|_{L^2}^2\leq C\nu^{-1}\|\nabla G_1\|_{L^2}^2.
\end{align*}

This finished the proof of Proposition \ref{prop:res-full-s2}  when  $|\la|\ge 1-\nu^\f13|k|^{-\f13}$.

\begin{remark}\label{rem:la small}
We assume $|\la|\le 1-\nu^\f13|k|^{-\f13}$ in section 8.3-8.6.
Then \eqref{eq: partial z p H2} still holds true
since $\nu^{\f{1}{2}}|k|\|\nabla W\|_{L^2}\leq C\|\nabla G_3\|_{L^2}$ by Proposition \ref{prop:res-toy}.
Thus, we get by Proposition \ref{prop:res-nav-s1} that
\begin{align}
\nu^{\f23}|k|^{\f13}\|\nabla U\|_{L^2} +(\nu k^2)^{\f13}\|U\|_{L^2} +\nu\|\Delta U\|_{L^2}&\leq C\nu^{\f16}|k|^{-\f23}\|\nabla \partial_zp^{L(0)}\|_{L^2}\nonumber\\& \leq C\nu^{\f16}|k|^{-\f23}\|\Delta p^{L(0)}\|_{L^2} \leq C\nu^{\f16}|k|^{\f13}\|W\|_{L^2}.\label{eq: u priori}
\end{align}
\end{remark}

\subsection{Singular part of $W$}

To deal with main trouble term $-2\nu\nabla\kappa\cdot\nabla U$ appeared in \eqref{eq:LNS-full-s2}, we introduce the singular part of
$W$. Let
\begin{align}\label{eq: rho1 rho2}
   &\rho_1=\f{\partial_y\kappa+\kappa\partial_z\kappa}{\partial_yV(1+\kappa^2)}, \quad\rho_2=\f{\partial_z\kappa-\kappa\partial_y\kappa}{1+\kappa^2}.
\end{align}
We introduce $ W_s$ and $\tW_s$ defined by
\begin{align}\label{f1f2}
  \left\{
  \begin{aligned}
   &-\nu\Delta  W_s+ik(V(y,z)-\lambda) W_s-a(\nu k^2)^{1/3} W_s+\rho_1\nabla V\cdot\nabla U=0,\\&-\nu\Delta \tW_s+ik(V(y,z)-\lambda)\tW_s-a(\nu k^2)^{1/3}\tW_s+\nabla V\cdot\nabla U=0,\\
   & W_s|_{y=\pm1}=\tW_s|_{y=\pm1}=0,\ \partial_x W_s=ik W_s,\ \partial_x\tW_s=ik\tW_s.
   \end{aligned}\right.
\end{align}
By Proposition \ref{prop:res-nav-s1}, we have
\begin{align}
&\nu^{\f23}|k|^{\f13}\|\nabla  W_s\|_{L^2}+(\nu k^2)^{\f13}\| W_s\|_{L^2} +\nu\|\Delta  W_s\|_{L^2}+|k|\|(V-\lambda) W_s\|_{L^2}\nonumber\\
&\quad\leq C\|\rho_1\nabla V\cdot\nabla U\|_{L^2}\leq C \|\rho_1\|_{L^\infty}\|\nabla V\|_{L^\infty}\|\nabla U\|_{L^2}\leq  C\varepsilon_0\|\nabla U\|_{L^2},\label{est: f2}
\end{align}
and
\begin{align}
&\nu^{\f23}|k|^{\f13}\|\nabla \tW_s\|_{L^2}+(\nu k^2)^{\f13}\|\tW_s\|_{L^2} +\nu\|\Delta \tW_s\|_{L^2}+|k|\|(V-\lambda)\tW_s\|_{L^2}\nonumber\\
&\quad\leq C\|\nabla V\cdot\nabla U\|_{L^2}\leq C \|\nabla V\|_{L^\infty}\|\nabla U\|_{L^2}\leq  C\|\nabla U\|_{L^2}.\label{est: f1}
\end{align}

Thanks to $\partial_xV=0$, we have
\begin{align*}
   &\partial_x(\partial_z-\kappa\partial_y) (\nabla V\cdot\nabla U)=(\partial_z-\kappa\partial_y) (\nabla V)\cdot(\nabla \partial_xU)+\nabla V\cdot\partial_x(\partial_z-\kappa\partial_y)\nabla U,\\&\partial_x(\partial_z-\kappa\partial_y)\nabla U=\nabla \partial_x(\partial_z-\kappa\partial_y)U+(\nabla \kappa)\partial_x\partial_yU,\\
   &\partial_x^2(\nabla V\cdot\nabla U)=\nabla V\cdot(\nabla \partial_x^2 U),
\end{align*}
and
\begin{align}\label{u3xz}
 \|\partial_x(\partial_z-\kappa\partial_y)\nabla U\|_{L^2}\leq C\big(\|\nabla \partial_x(\partial_z-\kappa\partial_y)U\|_{L^2}+\|\nabla\partial_x^2U\|_{L^2}\big).
\end{align}
Then we infer from Proposition \ref{prop:res-nav-good} that
\begin{align}
&(\nu k^2)^{\f13}\|(\partial_z-\kappa\partial_y)\tW_s\|_{L^2}= |k|^{-1}(\nu k^2)^{\f13}\|\partial_x(\partial_z-\kappa\partial_y)\tW_s\|_{L^2}\nonumber\\ &\leq  C|k|^{-1}\big(\|\partial_x^2(\nabla V\cdot\nabla U)\|_{L^2}+ \|\partial_x(\partial_z-\kappa\partial_y)(\nabla V\cdot\nabla U)\|_{L^2}\big)\nonumber\\
&\leq  C|k|^{-1}\Big(\| \nabla V\cdot\nabla \partial_x^2U\|_{L^2}+ \|\nabla V\cdot[\nabla \partial_x(\partial_z-\kappa\partial_y)U]\|_{L^2} +\| \partial_x\partial_yU\nabla\kappa\cdot\nabla V\|_{L^2} \nonumber\\
&\qquad+\|(\partial_z-\kappa\partial_y) (\nabla V)\cdot(\nabla \partial_xU)\|_{L^2}\Big)\nonumber\\
&\leq  C|k|^{-1}\|\nabla V\|_{L^\infty}\Big(\|\nabla \partial_x^2 U\|_{L^2}+ \|\nabla \partial_x(\partial_z-\kappa\partial_y)U\|_{L^2} +\|\nabla\kappa\|_{L^\infty}\|\nabla \partial_xU\|_{L^2}\Big) \nonumber\\
&\qquad+C|k|^{-1}\|(\partial_z-\kappa\partial_y) (\nabla V)\|_{L^\infty}\|(\nabla \partial_xU)\|_{L^2}\nonumber\\
&\leq C|k|^{-1}\big(\|\nabla\partial_x^2U\|_{L^2} +\|\nabla\partial_x(\partial_z-\kappa\partial_y)U\|_{L^2}\big)\nonumber\\
&= C\big(\|\nabla\partial_xU\|_{L^2} +\|\nabla(\partial_z-\kappa\partial_y)U\|_{L^2}\big).\label{est: good f1}
\end{align}

Furthermore, we find that
\begin{align*}
  \left\{
  \begin{aligned}
   &-\nu\Delta ( W_s-\rho_1\tW_s)+ik(V(y,z)-\lambda)( W_s-\rho_1\tW_s)-a(\nu k^2)^{1/3}( W_s-\rho_1\tW_s)\\
   &\qquad=-\nu(\Delta\rho_1)\tW_s+2\nu\text{div}(\tW_s\nabla\rho_1),\\
   &( W_s-\rho_1\tW_s)|_{y=\pm1}=0,\ \partial_x( W_s-\rho_1\tW_s)=ik( W_s-\rho_1\tW_s).
   \end{aligned}\right.
\end{align*}
Let $h_1$ solve $\Delta h_1=\nu(\Delta\rho_1)\tW_s,\ h_1|_{y=\pm1}=0$. By Proposition \ref{prop:res-nav-s1} and Lemma \ref{Lem: bil good deri}, we have
\begin{align*}
&\|\nabla ( W_s-\rho_1\tW_s)\|_{L^2} +\nu^{-\f13} |k|^{\f13}\|( W_s-\rho_1\tW_s)\|_{L^2}\\ &\leq C\big(\nu^{-1}\|\nabla
h_1\|_{L^2}+\|\tW_s\nabla\rho_1\|_{L^2}\big)\leq C\|\rho_1\|_{H^2}\big(\|\tW_s\|_{L^2}+\|(\partial_z-\kappa\partial_y)\tW_s\|_{L^2}\big)
\\&\leq C\varepsilon_0\big(\|\tW_s\|_{L^2}+\|(\partial_z-\kappa\partial_y)\tW_s\|_{L^2}\big),
\end{align*}
which along with \eqref{est: good f1} and \eqref{est: f1} gives
\begin{align}\label{est: f2-rho1f1}
&\|\nabla ( W_s-\rho_1\tW_s)\|_{L^2} +|\nu/k|^{-\f13}\|( W_s-\rho_1\tW_s)\|_{L^2}\\ \nonumber&\qquad\leq C\varepsilon_0(\nu k^2)^{-\f13}\big(\|\nabla\partial_xU\|_{L^2} +\|\nabla(\partial_z-\kappa\partial_y)U\|_{L^2}\big).
\end{align}

\subsection{Boundary corrector of  $U$}

\begin{Lemma}\label{lem:UB}
  If $\nu k^2\leq 1,\ |\lambda|\leq 1-|\nu/k|^{\f13}$, then there exists $U_b$ so that
  \begin{align*}
  \left\{
\begin{aligned}
&-\nu\Delta U_b+ik(V(y,z)-\lambda)U_b-a(\nu k^2)^{1/3}U_b=f_b,\\
&-\nu\Delta^2 U_b+ik(V(y,z)-\lambda)\Delta U_b-a(\nu k^2)^{1/3}\Delta U_b=F_b,\\
& (ik(V(y,z)-\lambda)U_b-a(\nu k^2)^{1/3}U_b+\partial_zp^{L(0)})|_{y=\pm 1}=0,\\
& \Delta (U- U_b)|_{y=\pm 1}=0,\\
&U_b=0\quad \text{for}\quad |y|\leq(3+|\lambda|)/4,
\end{aligned}\right.
\end{align*}
where $f_b|_{y=\pm 1}=0$ and
\begin{align}
  &\|f_b\|_{L^2}+\|(1-|y|)\nabla f_b\|_{L^2}+ |k|\|(1-|y|)U_b\|_{L^2}\leq C\nu^{\f34}|k|^{-\f14}(1-|\lambda|)^{-\f74} \|W\|_{L^2}.
  \label{eq: w12 estimate}\\
   & \|F_b\|_{L^2}\leq C\nu^{-\f14}|k|^{\f34}(1-|\lambda|)^{-\f34}\|W\|_{L^2},
   \label{eq: w3 estimate}\\
   &\|\nabla G_b\|_{L^2}\leq C\nu^{\f14}|k|^{\f14}(1-|\lambda|)^{-\f54} \|W\|_{L^2}, \label{eq: w4 estimate}
\end{align}
with $G_b$ solving $\Delta G_b=F_b, G_b|_{y=\pm 1}=0$. Moreover,
\begin{align*}
   &\|\kappa U_b\|_{H^1}\leq C\varepsilon_0\nu^{\f14}|k|^{-\f34}(1-|\lambda|)^{-\f54}\|W\|_{L^2},\\
   &\|\kappa U_b\|_{H^2}\leq C\varepsilon_0\nu^{-\f14}|k|^{-\f14}(1-|\lambda|)^{-\f34}\|W\|_{L^2}.
\end{align*}
\end{Lemma}

\begin{proof}
 We introduce
\begin{align*}
   &\phi_{-1}(k,y,\ell)=\exp\bigg(-(1+y)\sqrt{\eta^2-a|k/\nu|^{\f23}-ik(1+\lambda)/\nu}\bigg),\\
   &\phi_1(k,y,\ell)=\exp\bigg(-(1-y)\sqrt{\eta^2-a|k/\nu|^{\f23}+ik(1-\lambda)/\nu}\bigg),
\end{align*}
here $\eta=\sqrt{k^2+\ell^2}$ and $\sqrt{z}$ be the branch of the square root defined on the complement of the non-positive real numbers. Then for $j\in\{\pm1\}$, $\phi_{j}$ satisfies
\begin{align*}
   \left\{\begin{aligned}
   &-\nu(\partial_y^2-\eta^2) \phi_j+ik(j-\lambda)\phi_j-a(\nu k^2)^{1/3}\phi_j=0,\\
   &\phi_j|_{y=j}=1.
   \end{aligned}\right.
\end{align*}
Furthermore,  due to   $2a|k/\nu|^{\f23}\leq |k(1\pm \lambda)/\nu|$, we deduce from Lemma \ref{lem:AB-ineq} that
\begin{align*}
     &|\phi_{-1}(k,y,\ell)|\leq \exp\big(-c(1+y)(\eta^2+|k(1+\lambda)/\nu|)^{\f12}\big) \leq C\exp\big(-c(1-|y|)|k(1+\lambda)/\nu|^{\f12}\big),\\
     &|\phi_{1}(k,y,\ell)|\leq \exp\big(-c(1-y)(\eta^2+|k(1-\lambda)/\nu|)^{\f12}\big) \leq C\exp\big(-c(1-|y|)|k(1-\lambda)/\nu|^{\f12}\big),\\
     &|(k,\partial_y,\ell)^m\phi_{j}(k,y,\ell)| \leq C(\eta^2+|k(1-\lambda)/\nu|)^{\f m2}\exp\big(-c|j-y|(\eta^2+|k(j-\lambda)/\nu|)^{\f12}\big),\end{align*}
  and for $m\in\mathbb{N},\ \alpha\geq 0,\ j\in\{\pm 1\}$,
  \begin{align}\label{mL2}&\|(1-|y|)^{\alpha}(k,\partial_y,l)^m\phi_{j}(k,y,l)\|_{L_y^2}\\ \nonumber&\leq C\big(\eta^2+|k(j-\lambda)/\nu|\big)^{\f m2}\||j-y|^{\alpha}\exp\big(-c|j-y|(\eta^2+|k(j-\lambda)/\nu|)^{\f12}\big)\|_{L^2(I)}\\ \nonumber&\leq C\big(\eta^2+|k(j-\lambda)/\nu|\big)^{\f m2-\f{\alpha}2-\f14}.
  \end{align}

\def\varphi{W}
We denote
  \begin{align*}
 w^{(j)}(x,z)=\f{\partial_zp^{L(0)}(x,j,z)}{ik(j-\lambda)-a(\nu k^2)^{\f13}},\quad \hat{w}^{(j)}(k,\ell)=\f{1}{(2\pi)^2}\int_{\mathbb{T}^2}w^{(j)} (x,z)e^{-ikx-i\ell z}dxdz,
  \end{align*}
and let $w_{1,j}$ be defined by
  \begin{align*}
     &w_{1,j}(x,y,z)=\sum_{\ell\in\mathbb{Z}}\phi_{j}(k,y,\ell)\hat{w}^{(j)}(k,l)e^{ikx+i\ell z}.
  \end{align*}
Then we have
  \begin{align}\label{eq: w1,j equations}\left\{\begin{aligned}
     &-\nu\Delta w_{1,j}+ik(j-\lambda)w_{1,j}-a(\nu k^2)^{1/3}w_{1,j}=0,\\ &(ik(V-\lambda)w_{1,j}-a(\nu k^2)^{\f13}w_{1,j}-\partial_zp^{L(0)})|_{y=j}=0.\end{aligned}\right.
  \end{align}
Using the fact that
  \begin{align}\nonumber\|\partial_zp^{L(0)}\|_{H^{1/2}(\Gamma_j)}\leq& \|p^{L(0)}\|_{H^{3/2}(\Gamma_j)}\leq C\|p^{L(0)}\|_{H^{2}}\\ \nonumber\leq& C\|\Delta p^{L(0)}\|_{L^{2}}\leq  C|k|\|\partial_yV\|_{L^\infty}\|\varphi\|_{L^2}\leq  C|k|\|\varphi\|_{L^2},
  \end{align}
  we deduce that
  \begin{align}
   \label{omH1/2}
    &\|w^{(j)}\|_{H^{1/2}}\leq |k|^{-1}(1-|\lambda|)^{-1}\|\partial_zp^{L(0)}\|_{H^{1/2}(\Gamma_j)} \leq C(1-|\lambda|)^{-1}\|\varphi\|_{L^2},\\ \label{omL2}
    &\| w^{(j)}\|_{L^2}\leq |k|^{-\f12}\|w^{(j)}\|_{H^{1/2}}\leq C|k|^{-\f12}(1-|\lambda|)^{-1}\|\varphi\|_{L^2}.
  \end{align}
 Thus, by Plancherel's theorem, \eqref{mL2} and \eqref{omL2}, we deduce that for $\al\ge m$,
  \begin{align}
   \|(1-|y|)^{\alpha}\nabla^m w_{1,j}\|_{L^2}=& C\Big(\sum_{\ell\in\mathbb{Z}}\|(1-|y|)^{\alpha}(k,\partial_y,l)^m\phi_{j}(k,y,\ell)\|_{L_y^2}^2| \hat{w}^{(j)}(k,l)|^2\Big)^{\f12} \nonumber\\
    \leq & C\Big(\sum_{\ell\in\mathbb{Z}}(\eta^2+|k(1-\lambda)/\nu|)^{m-\alpha-\f12}| \hat{w}^{(j)}(k,\ell)|^2\Big)^{\f12}\nonumber\\
    \leq& C|k(1-\lambda)/\nu|^{\frac{m-\alpha}{2}-\f14}\Big(\sum_{\ell\in\mathbb{Z}}|\hat{w}^{j}(k,\ell)|^2\Big)^{\f12}\nonumber\\
    \leq& C|k(1-\lambda)/\nu|^{\frac{m-\alpha}{2}-\f14}\|w^{(j)}\|_{L^2}\nonumber\\ \leq& C|k(1-\lambda)/\nu|^{\frac{m-\alpha}{2}-\f14}|k|^{-\f12}(1-|\lambda|)^{-1}\|\varphi\|_{L^2}\nonumber\\ =& C\nu^{\f{\alpha-m}{2}+\f14}|k|^{-\f{\alpha-m}{2}-\f34} (1-|\lambda|)^{-\f{\alpha-m}{2}- \f54}\|\varphi\|_{L^2}.
    \label{est: (1-y) nabla w1j L2}
  \end{align}
  For $m\in\mathbb{Z}^{+}$,  we get by \eqref{mL2} with $\alpha=m-1$,  \eqref{omH1/2} and \eqref{omL2}  that
    \begin{align}
   \|(1-|y|)^{m-1}\nabla^{m} w_{1,j}\|_{L^2}=& C\Big(\sum_{\ell\in\mathbb{Z}}\|(1-|y|)^{m-1}(k,\partial_y,\ell)^{m}\phi_{j}(k,y,\ell)\|_{L_y^2}^2| \hat{w}^{(j)}(k,l)|^2\Big)^{\f12} \nonumber\\
    \leq & C\Big(\sum_{\ell\in\mathbb{Z}}(\eta^2+|k(1-\lambda)/\nu|)^{\f12}| \hat{w}^{(j)}(k,\ell)|^2\Big)^{\f12}\nonumber\\
    \leq& C\|w^{(j)}\|_{H^{1/2}}+C|k(1-\lambda)/\nu|^{\f14}\|w^{(j)}\|_{L^2}\nonumber\\
    \leq& C(1-|\lambda|)^{-1}\|\varphi\|_{L^2}+C\nu^{-\f14}|k|^{-\f14}(1-|\lambda|)^{-\f34}\|\varphi\|_{L^2}\nonumber\\
    \leq& C\nu^{-\f14}|k|^{-\f14}(1-|\lambda|)^{-\f34}\|\varphi\|_{L^2},
    \label{est: nabla w1j L2}
  \end{align}Here we used the fact that $(1-|\lambda|)/|\nu k|\geq |\nu/k|^{\f13}/|\nu k|=(\nu k^2)^{-\f23}\geq 1.$\smallskip

  Take $\theta_0\in C_{0}^{\infty}(\mathbb{R})$ so that $\theta_0(y)=1$ for $|y|\leq 1/4$ and $\theta_0(y)=0$ for $|y|\geq1/2$.
  Let
  \beno
  \theta_j(y)=\theta_j(x,y,z)=\theta_0(2|y-j|/(1-|\lambda|)),\quad j\in\{\pm1\}.
    \eeno
Then  $\theta_j(x,y,z)=1$ for $|y-j|\leq (1-|\lambda|)/8, \,\theta_j(x,y,z)=0$ for $|y-j|\geq(1-|\lambda|)/4$, and there holds that
\begin{align*}
  |\nabla^{m}\theta_j| &\leq C(1-|\lambda|)^{-m},\quad |\nabla^{m+1}\theta_j| \leq C(1-|\lambda|)^{-m_1}(1-|y|)^{m_1-m-1},\quad m,m_1\in\mathbb{N}.
\end{align*}

Now we construct
\begin{align}\label{eq: w1 construction}
   U_b=\sum_{j=\pm1}\theta_{j}(y)w_{1,j}(x,y,z).
\end{align}
We find from  \eqref{eq: w1,j equations} that
\begin{align*}
   &-\nu\Delta U_b+ik(V-\lambda)U_b-a(\nu k^2)^{1/3}U_b=f_b,
\end{align*}
where
\beno
f_b=\sum_{j=\pm1}\big(-\nu\Delta \theta_jw_{1,j}-2\nu\nabla \theta_j\cdot\nabla w_{1,j}+ik(V-j)\theta_jw_{1,j}\big),
\eeno
and
\begin{align*}
  (ik(V-\lambda)U_b-a(\nu k^2)^{\f13}U_b)|_{y=j}=(ik(V-\lambda)w_{1,j}-a(\nu k^2)^{\f13}w_{1,j})|_{y=j}=\partial_zp^{L(0)}|_{y=j}.
\end{align*}
It is obvious that  $f_b|_{y=\pm1}=0$, which  implies that
\beno
(-\nu\Delta U_b+\partial_zp^{L(0)})|_{y=\pm 1}=0.
\eeno
On the other hand,  $(-\nu\Delta U+\partial_zp^{L(0)})|_{y=\pm 1}=0$. This shows that
\beno
\Delta (U-U_b)|_{y=\pm 1}=0.
\eeno
For $|y|\leq(3+|\lambda|)/4,\ j\in\{\pm 1\}$, we have $|y-j|\geq(1-|\lambda|)/4$,  $\theta_j(x,y,z)=0$, and then
$U_b(x,y,z)=0$ for $|y|\le (3+|\lambda|)/4$. \smallskip

Now we estimate $U_b$ and $f_b$.  By \eqref{est: (1-y) nabla w1j L2}, we have
\begin{align*}
   &\|U_b\|_{L^2}\leq C\sum_{j=\pm1}\|\theta_jw_{1,j}\|_{L^2}\leq C\nu^{\f14}|k|^{-\f34}(1-|\lambda|)^{-\f54}\|W\|_{L^2},\\
   &\|(1-|y|)U_{b}\|_{L^2}\leq C\sum_{j=\pm1}\left\| (1-|y|)w_{1,j}\right\|_{L^2}\leq C\nu^{\f34}|k|^{-\f54}(1-|\lambda|)^{-\f74}\|W\|_{L^2}.
\end{align*}
By \eqref{est: (1-y) nabla w1j L2} and the fact that $(1-|y|)|\nabla\theta_j|\leq C$, we have
\begin{align*}
   \|(1-|y|)\nabla U_b\|_{L^2}\leq & C\sum_{j=\pm1}\|(1-|y|)\nabla(\theta_jw_{1,j})\|_{L^2}\\
   \leq& C\sum_{j=\pm1}\big(\|(1-|y|)\nabla w_{1,j}\|_{L^2}+\|w_{1,j}\|_{L^2}\big)\\
   \leq &C\nu^{\f14}|k|^{-\f34}(1-|\lambda|)^{-\f54}\|W\|_{L^2}.
\end{align*}
Then we also have
\begin{align*}
  &\|\theta_jw_{1,j}\|_{L^2} \leq C\|w_{1,j}\|_{L^2}\leq C\nu^{\f14}|k|^{-\f34}(1-|\lambda|)^{-\f{5}{4}}\|W\|_{L^2},\\
  &\|(\Delta \theta_j)w_{1,j}\|_{L^2} \leq C(1-|\lambda|)^{-2}\|w_{1,j}\|_{L^2}\leq C\nu^{\f14}|k|^{-\f34}(1-|\lambda|)^{-\f{13}{4}}\|W\|_{L^2},\\
  &\|\nabla \theta_j\cdot\nabla w_{1,j}\|_{L^2} \leq C(1-|\lambda|)^{-2}\|(1-|y|)\nabla w_{1,j}\|_{L^2}\leq C\nu^{\f14}|k|^{-\f34}(1-|\lambda|)^{-\f{13}{4}}\|W\|_{L^2},\\
   &\|k(V-j)\theta_jw_{1,j}\|_{L^2}\leq C|k|\|(1-|y|)w_{1,j}\|_{L^2}\leq C\nu^{\f34}|k|^{-\f14}(1-|\lambda|)^{-\f74}\|W\|_{L^2}.
\end{align*}
Here we used 
$|(V(y,z)-j)\theta_j(y)|\leq C(1-|y|).$
Summing up, we obtain
\begin{align*}
  \|f_b\|_{L^2}=&\sum_{j=\pm1}\left\|-\nu\Delta \theta_jw_{1,j}-2\nu\nabla \theta_j\cdot\nabla w_{1,j}+ik(V-j)\theta_jw_{1,j}\right\|_{L^2}\\
  \leq& C\sum_{j=\pm1}\big(\nu\|(\Delta \theta_j)w_{1,j}\|_{L^2}+\nu \|\nabla \theta_j\cdot\nabla w_{1,j}\|_{L^2} +\|k(V-j)\theta_jw_{1,j}\|_{L^2}  \big)\\
  \leq &C \big(\nu^{\f54}|k|^{-\f34}(1-|\lambda|)^{-\f{13}{4}}+ \nu^{\f34}|k|^{-\f14}(1-|\lambda|)^{-\f74}\big)\|\varphi\|_{L^2}\\
  \leq& C\nu^{\f34}|k|^{-\f14}(1-|\lambda|)^{-\f74}\|\varphi\|_{L^2}.
\end{align*}
Using  the fact that $|\nabla^{m}\theta_j| \leq C(1-|\lambda|)^{-2}(1-|y|)^{2-m}, m\in\mathbb{Z}^+$, we get by  \eqref{est: (1-y) nabla w1j L2} that
\begin{align*}
   &\|(1-|y|)\nabla(\Delta\theta_j w_{1,j}+2\nabla \theta_j\cdot\nabla w_{1,j})\|_{L^2}\leq C\sum_{m_2=1}^3\|(1-|y|)|\nabla^{m_2}\theta_j||\nabla^{3-m_2}w_{1,j}|\|_{L^2}\\
   &\leq C(1-|\lambda|)^{-2}\sum_{m_2=1}^3\|(1-|y|)^{3-m_2}\nabla^{3-m_2}w_{1,j}\|_{L^2}
   \leq C\nu^{\f14}|k|^{-\f34}(1-|y|)^{-\f{13}{4}}\|W\|_{L^2}.
\end{align*}
Using $|(V-j)\theta_j|\leq C(1-|y|)$ and $|(\nabla V)\theta_j|+|(V-j)\nabla \theta_j|\leq C$ and \eqref{est: (1-y) nabla w1j L2}, we get
\begin{align*}
   &\|ik(1-|y|)\nabla[(V-j)\theta_jw_{1,j}]\|_{L^2}\\
   &\leq C|k|\Big(\|(1-|y|)(V-j)\theta_j\nabla w_{1,j}\|_{L^2}+\|(1-|y|)(V-j)(\nabla\theta_j)w_{1,j}\|_{L^2} \\&\quad+\|(1-|y|)(\nabla V)\theta_jw_{1,j}\|_{L^2}\Big)\\
   &\leq C|k|\big(\|(1-|y|)^2\nabla w_{1,j}\|_{L^2} +\|(1-|y|)w_{1,j}\|_{L^2}\big)\\
   &\leq C|k|\big(\nu^{\f34}|k|^{-\f54}(1-|\lambda|)^{-\f74}\big)\|\varphi\|_{L^2}=C\nu^{\f34}|k|^{-\f14}(1-|\lambda|)^{-\f74}\|\varphi\|_{L^2}.
\end{align*}
Thus, we have
\begin{align*}
 \|(1-|y|)\nabla f_b\|_{L^2}
   \leq& \sum_{j=\pm1}\Big(\nu\|(1-|y|)\nabla(\Delta\theta_j w_{1,j}+2\nabla \theta_j\cdot\nabla w_{1,j})\|_{L^2} \\
   &\quad+\|ik(1-|y|)\nabla[(V-j)\theta_jw_{1,j}]\|_{L^2}\Big)\\
   \leq & C\big(\nu^{\f54}|k|^{-\f34}(1-|\lambda|)^{-\f{13}{4}}+\nu^{\f34}|k|^{-\f14}(1-|\lambda|)^{-\f74}\big)\|\varphi\|_{L^2}\\
   \leq& C\nu^{\f34}|k|^{-\f14}(1-|\lambda|)^{-\f74}\|\varphi\|_{L^2}.
\end{align*}
This finished the proof of \eqref{eq: w12 estimate}.

It is easy to see that
\begin{align*}
   \|\kappa U_b\|_{H^1}\leq& \|\kappa (\nabla U_b)\|_{L^2}+\|(\nabla \kappa)U_b\|_{L^2}+\| \kappa U_b\|_{L^2}\leq C\varepsilon_0\big(\|(1-|y|)\nabla U_b\|_{L^2}+\|U_b\|_{L^2}\big)\\
   \leq& C\varepsilon_0\nu^{\f14}|k|^{-\f34}(1-|\lambda|)^{-\f54}\|W\|_{L^2},
\end{align*}
and  by \eqref{est: nabla w1j L2} and \eqref{est: (1-y) nabla w1j L2}, we have
\begin{align*}
   \|\kappa U_b\|_{H^2}\leq& C \|\Delta(\kappa U_b)\|_{L^2}\leq C\big(\|\kappa (\Delta U_b)\|_{L^2}+\|(\nabla \kappa)\nabla U_b\|_{L^2}+\|(\Delta\kappa)U_b\|_{L^2}\big)\\
   \leq& C\varepsilon_0\big(\|(1-|y|)\Delta U_b\|_{L^2}+\|\nabla U_b\|_{L^2}+\|U_b\|_{H^1}\big)\\
   \leq& C\varepsilon_0\sum_{j=\pm1}\big(\|(1-|y|)\Delta w_{1,j}\|_{L^2}+\|\nabla w_{1,j}\|_{L^2}+(1-|\lambda|)^{-1}\|w_{1,j}\|_{L^2}\big)\\
   \leq &C\varepsilon_0\big(\nu^{-\f14}|k|^{-\f14}(1-|\lambda|\big)^{-\f34}+\nu^{\f14}|k|^{-\f34}(1-|\lambda|)^{-\f94})\|W\|_{L^2}\\
   \leq & C\varepsilon_0\nu^{-\f14}|k|^{-\f14}(1-|\lambda|)^{-\f34}\|W\|_{L^2}.
\end{align*}
Here we used the facts that $U_b=\theta_{1}w_{1,1}+\theta_{-1}w_{1,1},\ 1-|\lambda|\geq |\nu/k|^\f13$, and that
\begin{align*}&\|\kappa/(1-|y|)\|_{L^{\infty}}\leq \|\nabla\kappa\|_{L^{\infty}}+\|\kappa\|_{L^{\infty}}\leq C\|\kappa\|_{H^3}\leq C\varepsilon_0,\\
   &\|(\Delta\kappa)U_b\|_{L^2}\leq C\|(\Delta\kappa)\|_{H^1}\|U_b\|_{H^1}\leq C\|\kappa\|_{H^3}\|U_b\|_{H^1}\leq C\varepsilon_0\|U_b\|_{H^1},\\   & \|(1-|y|)\Delta (\theta_{j}w_{1,j})\|_{L^2}\leq C\big(\|(1-|y|)\Delta w_{1,j}\|_{L^2}+\|\nabla w_{1,j}\|_{L^2}+(1-|\lambda|)^{-1}\|w_{1,j}\|_{L^2}\big),\\
  &\|\nabla (\theta_{j}w_{1,j})\|_{L^2}\leq C\big(\|\nabla w_{1,j}\|_{L^2}+(1-|\lambda|)^{-1}\|w_{1,j}\|_{L^2}\big).
\end{align*}

Direct calculations show that
\begin{align*}
   &-\nu\Delta^2U_b+ik(V-\lambda)\Delta U_b-a(\nu k^2)^{1/3}\Delta U_b=F_b,
\end{align*}
where
\begin{align*}
   F_b=&\Delta f_b-ik(\Delta V)U_b-2ik\nabla V\cdot\nabla U_b \\
   =&\sum_{j=\pm1}\Big[\Delta(-\nu\Delta\theta_jw_{1,j}-2\nu\nabla\theta_j\cdot\nabla w_{1,j})+ik\Delta\big((V-j)\theta_jw_{1,j}\big)-ik(\Delta V)\theta_jw_{1,j}\\
   &\qquad-2ik\nabla V\cdot\nabla(\theta_jw_{1,j})\Big]\\
   =&\sum_{j=\pm1}\big[\Delta(-\nu\Delta\theta_jw_{1,j}-2\nu\nabla\theta_j\cdot\nabla w_{1,j})+ik(V-j)\Delta(\theta_j w_{1,j})\big].
\end{align*}
Using the fact $|\nabla^{m}\theta_j| \leq C(1-|\lambda|)^{-4}(1-|y|)^{4-m}, m\in\mathbb{Z}^+$, we get by \eqref{est: (1-y) nabla w1j L2} with $\alpha=m=4-m_2$ that
\begin{align*}
 \|\Delta(\Delta\theta_j w_{1,j}+2\nabla \theta_j\cdot\nabla w_{1,j})\|_{L^2}\leq& C\sum_{m_2=1}^4\||\nabla^{m_2}\theta_j||\nabla^{4-m_2}w_{1,j}|\|_{L^2}\\
   \leq &C(1-|\lambda|)^{-4}\sum_{m_2=1}^4\|(1-|y|)^{4-m_2}\nabla^{4-m_2}w_{1,j}\|_{L^2}\\
   \leq& C\nu^{\f14}|k|^{-\f34}(1-|\lambda|)^{-\f{21}{4}}\|W\|_{L^2},
\end{align*}
and using the facts that  $|(V-j)\theta_j|\leq C(1-|y|)$, $|(V-j)\nabla\theta_j|\leq C(1-|\lambda|)^{-1}(1-|y|)$, $|(V-j)\nabla^2\theta_j|\leq C(1-\lambda)^{-1}$, we get  by \eqref{est: (1-y) nabla w1j L2} with $\alpha=m\in\{0,1\}$ and \eqref{est: nabla w1j L2} with $m=2$ that
\begin{align*}
   &\|ik(V-j)\Delta(\theta_jw_{1,j})\|_{L^2}\\
   &\leq C|k|\big(\|(1-|y|)\nabla^2w_{1,j}\|_{L^2} +(1-|\lambda|)^{-1}\|(1-|y|)\nabla w_{1,j}\|_{L^2}+(1-|\lambda|)^{-1}\|w_{1,j}\|_{L^2}\big)\\
   &\leq C(\nu^{-\f14}|k|^{\f34}(1-|\lambda|)^{-\f34} +\nu^{\f14}|k|^{\f14}(1-|\lambda|)^{-\f94})\|W\|_{L^2}
   \leq C\nu^{-\f14}|k|^{\f34}(1-|\lambda|)^{-\f34}\|W\|_{L^2}.
\end{align*}
Here we used $1-|\lambda|\geq |\nu/k|^\f13.$ This shows that
\begin{align*}
  \|F_b\|_{L^2} \leq& \nu\|\Delta(\Delta\theta_j w_{1,j}+2\nabla \theta_j\cdot\nabla w_{1,j})\|_{L^2}+\|ik(V-j)\Delta(\theta_jw_{1,j})\|_{L^2}\\
  \leq &C\big(\nu^{\f54}|k|^{-\f34}(1-|\lambda|)^{-\f{21}{4}}+\nu^{-\f14}|k|^{\f34}(1-|\lambda|)^{-\f34}\big)\|W\|_{L^2}\\
  \leq& C\nu^{-\f14}|k|^{\f34}(1-|\lambda|)^{-\f34}\|W\|_{L^2}.
\end{align*}

Since $ \Delta G_b=F_b,\ G_b|_{y=\pm 1}=0,$ by Hardy's inequality, we get
\begin{align*}
  &\|\nabla G_b\|_{L^2}^2=-\langle G_b, F_b\rangle \leq \|G_b/(1-|y|)\|_{L^2}\|(1-|y|)F_b\|_{L^2}\leq C\|\nabla G_b\|_{L^2}\|(1-|y|)F_b\|_{L^2},
\end{align*}
which gives
\begin{align*}
  \|\nabla G_b\|_{L^2} &\leq C\|(1-|y|)F_b\|_{L^2}\\
  &\leq C\sum_{j=\pm1}\big(\nu\|(1-|y|)\Delta(\Delta\theta_j w_{1,j}+2\nabla \theta_j\cdot\nabla w_{1,j})\|_{L^2}+\|ik(1-|y|)(V-j)\theta_j\Delta w_{1,j}\|_{L^2}\big)\\
  &\leq C\big(\nu^{\f74}|k|^{-\f54}(1-|\lambda|)^{-\f{23}{4}}+ \nu^{\f14}|k|^{\f14}(1-|\lambda|)^{-\f54}\big)\|W\|_{L^2}\\
  &\leq C\nu^{\f14}|k|^{\f14}(1-|\lambda|)^{-\f54}\|W\|_{L^2}.
\end{align*}
Here we used the facts that
\begin{align*}
  \|(1-|y|)\Delta(\Delta\theta_j w_{1,j}+2\nabla \theta_j\cdot\nabla w_{1,j})\|_{L^2}\leq& C\sum_{m_2=1}^4\|(1-|y|)|\nabla^{m_2}\theta_j||\nabla^{4-m_2}w_{1,j}|\|_{L^2}\\
   \leq &C(1-|\lambda|)^{-4}\sum_{m_2=1}^4\|(1-|y|)^{5-m_2}\nabla^{4-m_2}w_{1,j}\|_{L^2}\\
   \leq &C\nu^{\f34}|k|^{-\f54}(1-|\lambda|)^{-\f{23}{4}}\|W\|_{L^2},
\end{align*}
and
\begin{align*}
   \|k(1-|y|)(V-j)\Delta(\theta_j w_{1,j})\|_{L^2}\leq& C|k|\big(\|(1-|y|)^2\Delta w_{1,j}\|_{L^2}+\|(1-|y|)\nabla w_{1,j}\|_{L^2} +\|w_{1,j}\|_{L^2}\big)\\
   \leq& C\nu^{\f14}|k|^{\f14}(1-|\lambda|)^{-\f54}\|W\|_{L^2}.
\end{align*}
\end{proof}

\begin{Lemma}\label{lem:U-Ub}
It holds that
\begin{align*}
&(\nu k^2)^{\f13}\|\Delta (U-U_b)\|_{L^2} +|k|\|(V-\lambda)\Delta (U-U_b)\|_{L^2}\le C\nu^{-\f12}|k|\|W\|_{L^2}+C|k|\|\nabla W\|_{L^2},\\
&(\nu k^2)^{\f13}\|\Delta (U-U_b)-2ik\tW_s\|_{L^2}\leq C\nu^{-\f{1}{6}}|k|^{\f23}(1-|\lambda|)^{-1}\|W\|_{L^2}+C|k|\|\nabla W\|_{L^2},\\
&|k|\|\nabla^2 [(1-\theta)(U-U_b)]\|_{L^2}\leq C(1-|\lambda|)^{-1}\big(\nu^{-\f{1}{2}}|k|\|W\|_{L^2}+ |k|\|\nabla W\|_{L^2}\big).
\end{align*}
Here $\theta(x,y,z)=\theta_0(|y-\lambda|/(1-|\lambda|))$ with $\theta_0\in  C_0^{\infty}(\R)$) so that $\theta_0(y)=1 $ for $|y|\leq 1/4,$ $\theta_0(y)=0 $ for $|y|\geq 1/2.$
\end{Lemma}

\begin{proof}
Recall that $(U, U_b, \tW_s)$ satisfies
\begin{align*}\left\{
\begin{aligned}
&-\nu\Delta^2 U+ik(V-\lambda)\Delta U-a(\nu k^2)^{1/3}\Delta U+ik\big(2\nabla V\cdot\nabla U+(\Delta V)U\big)+\partial_z\Delta p^{L(0)}=0,\\
&-\nu\Delta^2 U_b+ik(V-\lambda)\Delta U_b-a(\nu k^2)^{1/3}\Delta U_b=F_b,\\
&-\nu\Delta \tW_s+ik(V(y,z)-\lambda)\tW_s-a(\nu k^2)^{1/3}\tW_s+\nabla V\cdot\nabla U=0,\\
&\Delta (U-U_b)|_{y=\pm1}=0,\quad  \tW_s|_{y=\pm1}=0.
\end{aligned}\right.
\end{align*}

By Proposition \ref{prop:res-nav-s1}, \eqref{eq: partial z p H2}, \eqref{eq: u priori} and Lemma \ref{lem:UB}, we have
\begin{align*}
&(\nu k^2)^{\f13}\|\Delta (U-U_b)\|_{L^2} +|k|\|(V-\lambda)\Delta (U-U_b)\|_{L^2}\nonumber\\
&\leq C\|ik(2\nabla V\cdot\nabla U+(\Delta V)U)+\partial_z\Delta p^{L(0)}+F_b\|_{L^2}\nonumber\\
&\leq C|k|\|\nabla V\|_{L^\infty}\|\nabla U\|_{L^2} +C|k|\|\Delta V\|_{L^\infty}\|U\|_{L^2}+C|k|\|\nabla W\|_{L^2}+ C\|F_b\|_{L^2}\nonumber\\
&\leq C\big(\nu^{-\f{1}{2}}|k|\|W\|_{L^2}+ \nu^{-\f16}|k|^{\f23}\|W\|_{L^2}\big)+C\big(|k|\|\nabla W\|_{L^2} +\nu^{-\f14}|k|^{\f34}(1-|\lambda|)^{-\f34}\|W\|_{L^2}\big)\nonumber\\
&\leq C\nu^{-\f12}|k|\|W\|_{L^2}+C|k|\|\nabla W\|_{L^2},
\end{align*}
and
\begin{align*}
&(\nu k^2)^{\f13}\|\Delta (U-U_b)-2ik\tW_s\|_{L^2}\\ \nonumber&\leq C\|ik(\Delta V)U+\partial_z\Delta p^{L(0)}\|_{L^2}+\nu^{-\f13} |k|^{\f13}\|\nabla G_b\|_{L^2}\\ \nonumber&\leq C|k|\|U\|_{L^2}+C|k|\|\nabla W\|_{L^2}+C\nu^{-\f{1}{12}}|k|^{\f{7}{12}}(1-|\lambda|)^{-\f54}\|W\|_{L^2}\\ \nonumber&\leq C\nu^{-\f{1}{6}}|k|^{\f23}(1-|\lambda|)^{-1}\|W\|_{L^2}+C|k|\|\nabla W\|_{L^2}.
\end{align*}
These show the first and second inequalities.\smallskip

Let $I_j=\{|y-\lambda|\leq(1-|\lambda|)/2^{j}\},\ I_j^c=[-1,1]\setminus I_j$. If $y\in I_2^c$ (i.e., $|y-\lambda|>(1-|\lambda|)/4$), then we have by Lemma \ref{lem:V} that
\begin{align*}
  |V-\lambda| &\geq |y-\lambda|-|V-y| \geq (1-|\lambda|)/2-C\varepsilon_0(1-|y|)\\& \geq (1-|\lambda|)/2-C\varepsilon_0(1-|\lambda|+|y-\lambda|)\geq (1-C\varepsilon_0)(1-|\lambda|)/2,
\end{align*}
which gives by taking $\varepsilon_0$ small enough so that $C\varepsilon_0\leq 1/2$ that
\begin{align*}
   &1-\theta\leq \mathbf{1}_{I_2^c}\leq 8(1-|\lambda|)^{-1}|V-\lambda|.
\end{align*}
Then by the first inequality of the lemma, we have
\begin{align*}
|k|\|(1-\theta)\Delta(U-U_b)\|_{L^2}&\leq C|k|(1-|\lambda|)^{-1}\|(V-\lambda)\Delta(U-U_b)\|_{L^2}\\
&\leq  C\big(1-|\lambda|)^{-1}(\nu^{-\f{1}{2}}|k|\|\varphi\|_{L^2}+|k|\|\nabla W\|_{L^2}\big).
\end{align*}
Since $(ik(V-\lambda)U_b-a(\nu k^2)^{\f13}U_b-\partial_zp^{L(0)})|_{y=\pm1}=0$, we have
$$U_b= \partial_zp^{L(0)}/(k(j-\lambda)-a(\nu k^2)^{\f13})\quad \text{on}\ \Gamma_j,\ j\in\{\pm1\}.$$
Thus, we get
\begin{align*}
   \|U_b\|_{H^{\f32}(\partial\Omega)}&\leq \sum_{j=\pm1}|k(j-\lambda)|^{-1}\|\partial_zp^{L(0)}\|_{H^{\f32}(\Gamma_j)}\leq 2|k(1-|\lambda|)|^{-1}\|\partial_zp^{L(0)}\|_{H^{\f32}(\partial\Omega)},
\end{align*}
where
\begin{align*}
   \|\partial_zp^{L(0)}\|_{H^{\f32}(\partial\Omega)}&\leq C\|\partial_zp^{L(0)}\|_{H^{2}}
   \leq C\|p^{L(0)}\|_{H^3}\leq C\|\nabla\Delta p^{L(0)}\|_{L^2}\leq C|k|\|\nabla W\|_{L^2}.
\end{align*}
This gives
\begin{align}
  \|U_b\|_{H^{\f32}(\partial\Omega)} & \leq 2|k(1-|\lambda|)|^{-1}\|\partial_zp^{L(0)}\|_{H^{\f32}(\partial\Omega)}\leq C(1-|\lambda|)^{-1}\|\nabla W\|_{L^2}.
   \label{est: w1 H32}
\end{align}
As $\text{supp}(\theta)\subseteq I_1$ and $U_b=0$ for $|y|\leq (3+|\lambda|)/4$, we have
\beno
|\nabla\theta||\nabla U_b|=|\nabla^2\theta||U_b|=0.
\eeno
Then by \eqref{eq: u priori}, \eqref{est: w1 H32} and standard elliptic estimate, we obtain
\begin{align}
&|k|\|\nabla^2 [(1-\theta)(U-U_b)]\|_{L^2}\nonumber\\ &\leq C|k|\|\Delta[(1-\theta)(U-U_b)]\|_{L^2} +C|k|\|(1-\theta)(U-U_b)\|_{H^{\f32}(\partial\Omega)}\nonumber\\
&= C|k|\|\Delta[(1-\theta)(U-U_b)]\|_{L^2} +C|k|\|U_b\|_{H^{\f32}(\partial\Omega)}\nonumber\\ &\leq C|k|\big(\| [(1-\theta)\Delta(U-U_b)]\|_{L^2}+\|\nabla\theta\cdot\nabla(U-U_b)\|_{L^2} +\|(\Delta\theta)(U-U_b)\|_{L^2}\big)\nonumber\\
&\quad+C|k|\|U_b\|_{H^{\f32}(\partial\Omega)}\nonumber\\ &\leq C|k|\big(\|[(1-\theta)\Delta(U-U_b)]\|_{L^2}+\|\nabla\theta\cdot\nabla U\|_{L^2} +\|(\Delta\theta)U\|_{L^2}\big)+C|k|\|U_b\|_{H^{\f32}(\partial\Omega)}\nonumber\\
&\leq C(1-|\lambda|)^{-1}(\nu^{-\f{1}{2}}|k|\|W\|_{L^2}+|k|\|\nabla W\|_{L^2}) +C(1-|\lambda|)^{-1}\nu^{-\f12}|k|\|W\|_{L^2}\nonumber\\ &\quad+ C(1-|\lambda|)^{-2}\nu^{-\f16}|k|^{\f23}\|W\|_{L^2}+ C|k|\|U_b\|_{H^{\f32}(\partial\Omega)}\nonumber\\
&\leq C(1-|\lambda|)^{-1}\big(\nu^{-\f{1}{2}}|k|\|W\|_{L^2}+ |k|\|\nabla W\|_{L^2}\big).\nonumber
\end{align}This is the third inequality.
\end{proof}

\subsection{Construction of good unknown}
We introduce the following good unknown
\begin{align}\label{def:good}
  W_g=W-\nu\theta W_s-\kappa U_b,
\end{align}
where $\theta=\theta_0(|y-\lambda|/(1-|\lambda|))$ with a fixed $\theta_0\in  C_0^{\infty}(\R)$ so that $\theta_0(y)=1 $ for $|y|\leq 1/4,$ $\theta_0(y)=0 $ for $|y|\geq 1/2.$

Let  us derive the equation of $W_g$. For this, we first introduce some notations.
Let $p^{L(01)}, p^{L(02)}$ and $p^{L(03)}$ solve
\begin{align*}
&\Delta p^{L(01)}=-2ik\partial_yV(\theta W_s),\quad  \partial_yp^{L(01)}|_{y=\pm1}=0,\\
&\Delta p^{L(02)}=-\partial_yV(\theta\rho_1 \Delta U),\quad \partial_yp^{L(02)}|_{y=\pm1}=0,\\
&\Delta p^{L(03)}=-2ik\partial_yV(\kappa U_b),\quad \partial_yp^{L(03)}|_{y=\pm1}=0.
\end{align*}
Let $p^{L(2)}=p^{L(0)}-\nu p^{L(01)}-p^{L(03)}$. Then  we have
\beno
\Delta p^{L(2)}=-2ik\partial_yVW_g,\quad \partial_yp^{L(2)}|_{y=\pm1}=0.
\eeno
We denote
\begin{align*}
F_1&=\theta\rho_1\nabla V\cdot\nabla U+(\partial_y+\kappa\partial_z)p^{L(01)}\\&=(\theta\rho_1\partial_y V\partial_y U+\partial_yp^{L(01)})+\kappa(\theta\rho_1\partial_y V\partial_z U+\partial_zp^{L(01)})\\&=F^{1}_{2}+\kappa F_{3}^{1},\end{align*}where \begin{align*}F_{j}^{l}&=(\theta\rho_1\partial_y V\partial_j U+\partial_jp^{L(0l)}),\ j\in\{2,3\},\ l\in\{1,2\},
\end{align*}
We denote
\begin{align}
F_2=&-\nu(\Delta \kappa) (U-U_b)-2\nu\nabla\kappa\cdot\nabla (U-U_b)+2\nu\theta\rho_1\nabla V\cdot\nabla U\nonumber\\=&-\nu(\Delta \kappa) (\theta U+U_1)-2\nu\nabla\kappa\cdot\nabla (\theta U+U_1)+2\nu\theta\rho_1\nabla V\cdot\nabla U\nonumber\\=&-\nu ((\Delta \kappa) \theta +2\nabla\kappa\cdot\nabla \theta)U+2\nu\theta(\rho_1\nabla V\cdot\nabla U-\nabla\kappa\cdot\nabla U)\nonumber\\&-\nu(\Delta \kappa) U_1-2\nu\nabla\kappa\cdot\nabla U_1\nonumber\\=&-\nu ((\Delta \kappa) \theta +2\nabla\kappa\cdot\nabla \theta) U-2\nu\theta\rho_2(\partial_z-\kappa\partial_y)U-\nu(\Delta \kappa) U_1-2\nu\nabla\kappa\cdot\nabla U_1,\label{eq: f2 define}
\end{align}
where
\beno
U_1=(1-\theta)U-U_b,\quad U=\theta U+U_1+U_b.
\eeno
We denote
\begin{align}\label{eq: F3 define}
F_3=F_2+\nu^2\big(2\nabla\theta\cdot\nabla W_s+\Delta\theta W_s\big)-\kappa f_b-\nu F_1-(\partial_y+\kappa\partial_z)p^{L(03)}.
\end{align}

Now, direct calculations show that
\begin{align*}
&-\nu\Delta (W-W_g)+ik(V(y,z)-\lambda)(W-W_g)-a(\nu k^2)^{1/3}(W-W_g)\\
&=\nu\theta\big[-\nu\Delta W_s+ik(V(y,z)-\lambda)W_s-a(\nu k^2)^{1/3}W_s\big]-\nu^2\big(2\nabla\theta\cdot\nabla W_s+\Delta\theta W_s\big)\\
&\quad+\kappa\big[-\nu\Delta U_b+ik(V(y,z)-\lambda)U_b-a(\nu k^2)^{1/3}U_b\big]-\nu(\Delta \kappa)U_b-2\nu\nabla\kappa\cdot\nabla U_b\\
&=-\nu\theta\rho_1\nabla V\cdot\nabla U-\nu^2(2\nabla\theta\cdot\nabla W_s+\Delta\theta W_s)+\kappa f_b-\nu(\Delta \kappa) U_b-2\nu\nabla\kappa\cdot\nabla U_b,
\end{align*}
which gives
\begin{align}
&-\nu\Delta W_g+ik(V(y,z)-\lambda)W_g-a(\nu k^2)^{1/3}W_g+(\partial_y+\kappa\partial_z)\big(p^{L(0)}-\nu p^{L(01)}+p^{L(1)}\big)\nonumber\\&=-\nu(\Delta \kappa) (U-U_b)-2\nu\nabla\kappa\cdot\nabla (U-U_b)+\nu\theta\rho_1\nabla V\cdot\nabla U+\nu^2(2\nabla\theta\cdot\nabla W_s+\Delta\theta W_s)\nonumber\\&\quad-\kappa f_b-\nu(\partial_y+\kappa\partial_z) p^{L(01)}-G_1\nonumber\\&=F_2+\nu^2(2\nabla\theta\cdot\nabla W_s+\Delta\theta W_s)-\kappa f_b-\nu F_1-G_1.\label{eq: varphi1 eq prior}
\end{align}
Thus, we conclude that
\begin{align}\label{eq:Wg}
\left\{
\begin{aligned}
&-\nu\Delta W_g+ik(V(y,z)-\lambda)W_g-a(\nu k^2)^{1/3}W_g+(\partial_y+\kappa\partial_z)(p^{L(2)}+p^{L(1)})\\
&\qquad\qquad+G_1-F_3=0,\\
&\Delta p^{L(2)}=-2ik\partial_yVW_g,\quad \partial_yp^{L(2)}|_{y=\pm1}=0,\\\
& W_g|_{y=\pm1}=0,\quad \partial_xW_g=ikW_g.
\end{aligned}\right.
\end{align}

Let us conclude this subsection by the following lemma.

\begin{Lemma}\label{lem: varphi varphi1}
If $|\lambda|\leq 1-\nu^{\f13}|k|^{-\f13}$, then we have
\begin{align*}
&\| \nu\theta W_s\|_{L^2}+\| \kappa U_b\|_{L^2}\leq C\varepsilon_0\nu^{\f16} |k|^{-\f23}\|W\|_{L^2},\\
&\|W\|_{L^2}\leq \|W_g\|_{L^2}+C\varepsilon_0\nu^{\f16}|k|^{-\f23}\|W\|_{L^2}, \\&\|\nabla W\|_{L^2}\leq\|\nabla W_g\|_{L^2}+C\varepsilon_0\nu^{-\f16} |k|^{-\f13}\|W\|_{L^2},\\&\|\Delta W\|_{L^2}\leq\|\Delta W_g\|_{L^2}+C\varepsilon_0\nu^{-\f12} \|W\|_{L^2}.
\end{align*}The second inequality implies that if $ \varepsilon_0$ is small enough, then $ \|W\|_{L^2}\leq 2\| W_g\|_{L^2}.$
\end{Lemma}

\begin{proof}
By \eqref{est: f2}, \eqref{eq: u priori}, \eqref{eq: w12 estimate} and the fact that $|\kappa/(1-|y|)|\leq C \|\nabla\kappa\|_{L^\infty}\leq C\|\kappa\|_{H^3}\leq C\varepsilon_0$, we deduce that
  \begin{align*}
    \|\nu\theta W_s\|_{L^2}+\|\kappa U_b\|_{L^2}\leq& C\nu\|W_s\|_{L^2}+\|\kappa/(1-|y|)\|_{L^\infty}\|(1-|y|)U_b\|_{L^2}\\
     \leq& C\varepsilon_0\nu(\nu k^2)^{-\f13}\|\nabla U\|_{L^2}+C\varepsilon_0\nu^{\f34}|k|^{-\f54}(1-|\lambda|)^{-\f74}\|W\|_{L^2} \\
     \leq& C\varepsilon_0\nu^{\f16}|k|^{-\f23}\|W\|_{L^2},
  \end{align*}
  which gives the first two inequalities.

 By \eqref{est: f2} and \eqref{eq: u priori}, we get
  \begin{align*}
    \|\nu\theta W_s\|_{H^1} &\leq C\nu\|\theta\|_{L^\infty}\|\nabla W_s\|_{L^2}+C\nu\|\nabla \theta\|_{L^\infty}\|W_s\|_{L^2}\\
    &\leq C\nu\varepsilon_0\big(\nu^{-\f23}|k|^{-\f13}+(1-|\lambda|)^{-1}(\nu k^2)^{-\f13}\big)\|\nabla U\|_{L^2}\\
    &\leq C\nu\varepsilon_0\nu^{-\f23}|k|^{-\f13}\|\nabla U\|_{L^2}\leq C\varepsilon_0\nu^{-\f16}|k|^{-\f13}\|W\|_{L^2},
  \end{align*}
and
  \begin{align*}
    \|\nu\theta W_s\|_{H^2} &\leq C\nu(\|\theta\|_{L^\infty}\|\Delta W_s\|_{L^2}+\|\nabla \theta\|_{L^\infty}\|\nabla W_s\|_{L^2} +\|\nabla^2 \theta\|_{L^\infty}\|W_s\|_{L^2})\\
    &\leq C\nu\varepsilon_0\big(\nu^{-1}+(1-|\lambda|)^{-1}\nu^{-\f23}|k|^{-\f13} +(1-|\lambda|)^{-2}(\nu k^2)^{-\f13}\big)\|\nabla U\|_{L^2}\\
    &\leq C\varepsilon_0\|\nabla U\|_{L^2}\leq C\varepsilon_0\nu^{-\f12}\|W\|_{L^2}.
  \end{align*}
from which and Lemma \ref{lem:UB}, we infer that
\begin{align*}
   \|\nabla W\|_{L^2}&\leq \|\nabla W_g\|_{L^2}+\big(\|\nu\theta W_s\|_{H^1}+\|\kappa U_b\|_{H^1}\big)\\
   &\leq \|\nabla W_g\|_{L^2}+C\varepsilon_0\nu^{-\f16}|k|^{-\f13}\|W\|_{L^2}+C\varepsilon_0 \nu^{\f14}|k|^{-\f34}(1-|\lambda|)^{-\f54}\|W\|_{L^2}\\
   &\leq \|\nabla W_g\|_{L^2}+C\varepsilon_0\nu^{-\f16}|k|^{-\f13}\|W\|_{L^2},
\end{align*}
and
\begin{align*}
   \|\Delta W\|_{L^2}&\leq \|\Delta W_g\|_{L^2}+ \|\nu\theta W_s\|_{H^2}+\|\kappa U_b\|_{H^2}\\&\leq \|\Delta W_g\|_{L^2}+ C\varepsilon_0\big(\nu^{-\f12}+\nu^{-\f14}|k|^{-\f14}(1-|\lambda|)^{-\f34}\big)\|W\|_{L^2}\\
   &\leq \|\Delta W_g\|_{L^2}+C\varepsilon_0\nu^{-\f12}\|W\|_{L^2}.
\end{align*}
\end{proof}

\subsection{Proof of  Proposition \ref{prop:res-full-s2}  when $|\la|\le 1-\nu^\f13|k|^{-\f13}$}

\subsubsection{$H^1$ estimate of source term}

Let us first estimate $\|\na F_1\|_{L^2}$.  Notice that
\beno
\Delta\big(p^{L(01)}-p^{L(02)}\big)=-2ik\partial_yV\theta(W_s-\rho_1\tW_s)+\partial_yV\theta\rho_1 (\Delta U-2ik\tW_s).
\eeno
Then by \eqref{est: f2-rho1f1}, Lemma \ref{lem:U-Ub},  $\|\rho_1\|_{L^\infty}\leq \|\rho_1\|_{H^2}\leq C\varepsilon_0$ and the fact that $\text{supp}(U_b)\cap \text{supp}(\theta)=\text{supp}(U_b)\cap I_1=\emptyset$,  we deduce that
\begin{align}
&\|\nabla^2 (p^{L(01)}-p^{L(02)})\|_{L^2}\leq C\|\Delta (p^{L(01)}-p^{L(02)})\|_{L^2}\nonumber\\
&\leq C|k|\|W_s-\rho_1\tW_s\|_{L^2}+C\|\rho_1\|_{L^\infty}\|\theta (\Delta U-2ik\tW_s)\|_{L^2}\nonumber\\
&\leq C |k|\|W_s-\rho_1\tW_s\|_{L^2}+C\varepsilon_0\|\theta (\Delta (U-U_b)-2ik\tW_s)\|_{L^2}\nonumber\\
&\leq C \varepsilon_0|k|^{-1}\big(\|\nabla\partial_x^2U\|_{L^2}+\|\nabla\partial_x (\partial_z-\kappa\partial_y)U\|_{L^2}\big)+C\varepsilon_0\|\Delta (U-U_b)-2ik\tW_s\|_{L^2}\nonumber\\ &\leq C\varepsilon_0|k|^{-1}\big(\|\nabla\partial_x^2U\|_{L^2}+\|\nabla\partial_x (\partial_z-\kappa\partial_y)U\|_{L^2}\big) \label{est: nabla2(p01-p02)}\\&\quad
+C\varepsilon_0\big(\nu^{-\f{1}{2}}(1-|\lambda|)^{-1}\|W\|_{L^2}+ \nu^{-\f13}|k|^{\f13}\|\nabla W\|_{L^2}\big).\nonumber
\end{align}

Let $ \phi_1=\theta\rho_1\partial_y V$, and then
\begin{align*}
\Delta F_{j}^{2}&=\Delta(\theta\rho_1\partial_y V\partial_j U)+\partial_j(\Delta p^{L(02)})=\Delta(\phi_1\partial_j U)-\partial_j(\phi_1\Delta U)\\&=\text{div}(\nabla\phi_1\partial_j U-\partial_j\phi_1\nabla U)+\partial_j(\nabla\phi_1\cdot\nabla U),
\end{align*}
with $ F_{2}^{l}|_{y=\pm1}=\partial_yF_{3}^{l}|_{y=\pm1}=0.$ Then we have
\begin{align*}
\|\nabla F_{j}^{2}\|_{L^2}^2=&-\big\langle\Delta F_{j}^{2},F_{j}^{2}\big\rangle=\big\langle\nabla\phi_1\partial_j U-\partial_j\phi_1\nabla U,\nabla F_{j}^{2}\big\rangle+\big\langle\nabla\phi_1\cdot\nabla U,\partial_jF_{j}^{2}\big\rangle,
\end{align*}
which gives
\ben
\|\nabla F_{j}^{2}\|_{L^2}\leq C\||\nabla\phi_1||\nabla U|\|_{L^2}. \label{eq:Fj2}
\een
Using the facts that  $0\leq\theta\leq 1,\ |\nabla\theta|\leq C(1-|\lambda|)^{-1},\ |\nabla V|+|\nabla^2 V|\leq C,\  \|\rho_1\|_{H^2}\leq C\varepsilon_0,$ we deduce that
\begin{align*}
|\nabla\phi_1|\leq C(1-|\lambda|)^{-1}|\rho_1|+C|\nabla\rho_1|,\end{align*}
and hence,
\begin{align}
\||\nabla\phi_1||\nabla U|\|_{L^2}&\leq  C(1-|\lambda|)^{-1}\|\rho_1\|_{L^{\infty}}\|\nabla U\|_{L^2}+C\||\nabla\rho_1||\nabla U|\|_{L^2}\nonumber\\ &\leq C(1-|\lambda|)^{-1}\|\rho_1\|_{H^2}\|\nabla U\|_{L^2}+C\|\nabla\rho_1\|_{H^1}\big(\|\nabla U\|_{L^2}+\|(\partial_z-\kappa\partial_y)\nabla U\|_{L^2}\big)\nonumber\\ &\leq C\varepsilon_0(1-|\lambda|)^{-1}\|\nabla U\|_{L^2}+C\varepsilon_0\big(\|\nabla U\|_{L^2}+\|(\partial_z-\kappa\partial_y)\nabla U\|_{L^2}\big)\nonumber\\ &\leq C\varepsilon_0(1-|\lambda|)^{-1}\nu^{-\f12}\|W\|_{L^2}+C\varepsilon_0|k|^{-1}
\big(\|\nabla\partial_x^2U\|_{L^2}+\| \nabla\partial_x(\partial_z-\kappa\partial_y)U\|_{L^2}\big),\label{eq: nabla phi1 nabla u}
\end{align}
here we used Lemma \ref{Lem: bil good deri}, \eqref{u3xz} and \eqref{eq: u priori}.

Note that for $j\in\{2,3\}$,
\begin{align*}
F_{j}^{1}-F_{j}^{2}&=\partial_j\big(p^{L(01)}-p^{L(02)}\big),\quad \|\nabla (F_{j}^{1}-F_{j}^{2})\|_{L^2}\leq\|\nabla^2 (p^{L(01)}-p^{L(02)})\|_{L^2}.
\end{align*}
We infer from \eqref{eq:Fj2}, \eqref{est: nabla2(p01-p02)} and  \eqref{eq: nabla phi1 nabla u}  that
\begin{align}
\nonumber\|\nabla F_{1}\|_{L^2}\leq& \|\nabla F_{2}^{1}\|_{L^2}+C\|\kappa\|_{H^2}\|\nabla F_{3}^{1}\|_{L^2}\\
 \leq& C\big(\||\nabla\phi_1||\nabla u|\|_{L^2}+\|\nabla^2 (p^{L(01)}-p^{L(02)})\|_{L^2}\big)\nonumber\\
 \leq&C\varepsilon_0|k|^{-1}\big(\|\nabla\partial_x^2U\|_{L^2}+\|\nabla\partial_x (\partial_z-\kappa\partial_y)U\|_{L^2}\big) \nonumber\\&
\quad+C\varepsilon_0\big((1-|\lambda|)^{-1}\nu^{-\f12}\|W\|_{L^2}+\nu^{-\f13}|k|^{\f13} \|\nabla W\|_{L^2}\big).\label{eq: nabla F1}
\end{align}

Next we estimate $\|F_2\|_{H^1}$. Using $|\na^j\theta|\leq C(1-|\la|)^{-j}$ and Lemma \ref{Lem: bil good deri}, we  deduce that
\begin{align}
\|F_2\|_{H^1}\leq&\nu\big(\|\Delta\kappa U\|_{H^1}\|\theta\|_{L^\infty}+ \|\Delta\kappa U\|_{L^2}\|\nabla \theta\|_{L^\infty}\big) +C\nu\big( \|\nabla\kappa U\|_{H^1}\|\nabla\theta\|_{L^\infty} \nonumber\\
&+\|\nabla\kappa U\|_{L^2}\|\nabla^2 \theta\|_{L^\infty}\big)+2\nu\big(\|\rho_2(\partial_z-\kappa\partial_y)U\|_{H^1}\|\theta\|_{L^\infty} \nonumber\\
&\quad+\|\rho_2(\partial_z-\kappa\partial_y)U\|_{L^2}\|\nabla\theta\|_{L^\infty}\big) +\nu\|(\Delta\kappa)U_1\|_{H^1}+2\nu\|\nabla\kappa\cdot\nabla U_1\|_{H^1}  \nonumber\\
\leq&C\nu\|\Delta \kappa U\|_{H^1}+C\nu(1-|\lambda|)^{-1}\big(\|\Delta \kappa U\|_{L^2}+\|\nabla \kappa U\|_{H^1}\big)+C\nu(1-|\lambda|)^{-2}\|\nabla \kappa U\|_{L^2}\nonumber\\&+C\nu\|\rho_2(\partial_z-\kappa\partial_y)U\|_{H^1} +C\nu(1-|\lambda|)^{-1}\|\rho_2(\partial_z-\kappa\partial_y)U\|_{L^2}\nonumber\\&
+C\nu\big(\|\Delta \kappa\|_{H^1}\|U_1\|_{H^2}+\|\nabla\kappa\|_{H^2}\|\nabla U_1\|_{H^1}\big)\nonumber\\
\leq&C\nu\|\Delta \kappa\|_{H^1}\big (\|U\|_{H^1}+\|(\partial_z-\kappa\partial_y)U\|_{H^1}\big)+C\nu(1-|\lambda|)^{-2}\|\nabla \kappa\|_{L^{\infty}}\|U\|_{L^2} \nonumber\\&+C\nu(1-|\lambda|)^{-1}\big(\|\Delta \kappa\|_{H^1} \|U\|_{H^1}+\|\nabla \kappa\|_{H^2} \|U\|_{H^1}\big) \nonumber\\&+C\nu\|\rho_2\|_{H^2}\|(\partial_z-\kappa\partial_y)U\|_{H^1}+C\nu(1-|\lambda|)^{-1}\|\rho_2\|_{L^{\infty}}
\|\nabla U\|_{L^2}\nonumber\\&
+C\nu\| \kappa\|_{H^3}\|U_1\|_{H^2}\nonumber\\
\leq&C\nu\big(\|\kappa\|_{H^3}+\|\rho_2\|_{H^2}\big) \Big((1-|\lambda|)^{-1}\|\nabla U\|_{L^2}+\|(\partial_z-\kappa\partial_y)U\|_{H^1} \label{est: F2 H1 prior}\\&\quad+(1-|\lambda|)^{-2}\|U\|_{L^2}+\|U_1\|_{H^2}\Big).\nonumber
\end{align}

Since $\text{supp}(\theta)\cap\text{supp}(U_b)= I_1\cap\text{supp}(U_b)=\emptyset$,  we have
\beno
 U_1=(U-U_b)-\theta U=(1-\theta)(U-U_b),
 \eeno
which along with Lemma \ref{lem:U-Ub} gives
\begin{align}
   \|U_1\|_{H^2}&=\|(1-\theta)(U-U_b)\|_{H^2}\leq C\|\nabla^2[(1-\theta)(U-U_b)]\|_{L^2}\nonumber\\ &\leq C(1-|\lambda|)^{-1}\big(\nu^{-\f12}|k|\|W\|_{L^2} +|k|\|W\|_{L^2}\big).\label{est: u1 varphi}
\end{align}
Then by the fact that $\|\kappa\|_{H^3}+\|\rho_2\|_{H^2}\leq C\varepsilon_0$, \eqref{eq: u priori} , \eqref{est: F2 H1 prior} and \eqref{est: u1 varphi}, we get
\begin{align}
   \|F_2\|_{H^1} \leq&C\nu\big(\|\kappa\|_{H^3}+\|\rho_2\|_{H^2}\big)\Big((1-|\lambda|)^{-1}\|\nabla U\|_{L^2}+\|(\partial_z-\kappa\partial_y)U\|_{H^1}\nonumber \\&\quad+(1-|\lambda|)^{-2}\|U\|_{L^2}+\|U_1\|_{H^2}\Big)\nonumber\\
   \leq&C\nu\varepsilon_0\Big((1-|\lambda|)^{-1}\nu^{-\f12}\|W\|_{L^2}+|k|^{-1}\|\nabla\partial_x (\partial_z-\kappa\partial_y)U\|_{L^2}\nonumber\\&
+(1-|\lambda|)^{-2}\nu^{-\f16}|k|^{-\f13}\|W\|_{L^2}+|k|^{-1}(1-|\lambda|)^{-1} \big(\nu^{-\f{1}{2}}|k|\|W\|_{L^2}+|k|\|\nabla W\|_{L^2}\big)\Big)\nonumber\\ \leq&C\nu\varepsilon_0\Big((1-|\lambda|)^{-1}\nu^{-\f12}\|W\|_{L^2}+|k|^{-1}\|\nabla\partial_x(\partial_z-\kappa\partial_y)U\|_{L^2}+ \nu^{-\f13}|k|^{\f13}\|\nabla W\|_{L^2}\Big).\label{est: F2 H1}
\end{align}

Finally, let us estimate $\|\na F_3\|_{L^2}$. By \eqref{est: f2} and  \eqref{eq: u priori} , we have
\begin{align}
\label{est: theta f2}
&\|\nu^2(2\nabla\theta\cdot\nabla W_s+\Delta\theta W_s)\|_{H^1}\\
&\leq  C\nu^2\Big(\|\nabla\theta\|_{L^\infty}\|\nabla W_s\|_{H^1}+\|\nabla^2\theta\|_{L^\infty}\|\nabla W_s\|_{L^2} +\|\Delta \theta\|_{L^\infty}\|W_s\|_{H^1}\nonumber\\
&\qquad\quad+\|\nabla\Delta\theta\|_{L^\infty} \|W_s\|_{L^2}\Big)\nonumber\\
&\leq C\nu^2(1-|\lambda|)^{-1}\|W_s\|_{H^2}+C\nu^2(1-|\lambda|)^{-2}\|W_s\|_{H^1} +C\nu^2(1-|\lambda|)^{-3}\|W_s\|_{L^2}\nonumber\\ &\leq C\nu^2(1-|\lambda|)^{-1}\big(\|\Delta W_s\|_{L^2}+\nu^{-\f13}|k|^{\f13}\|\nabla W_s\|_{L^2}+\nu^{-\f23}|k|^{\f23}\|W_s\|_{L^2}\big)\nonumber\\ &\leq C\nu\varepsilon_0(1-|\lambda|)^{-1}\|\nabla U\|_{L^2}\leq C\varepsilon_0 (1-|\lambda|)^{-1}\nu^{\f12}\|W\|_{L^2}.\nonumber
\end{align}
Since $\|\kappa/(1-|y|)\|_{L^{\infty}}+\|\nabla\kappa\|_{L^{\infty}}\leq C\|\nabla\kappa\|_{L^{\infty}}\leq C\varepsilon_0$,
we get by Lemma \ref{lem:UB} that
\begin{align}
\|\kappa f_b\|_{H^1}&\leq C\varepsilon_0\big(\|f_b\|_{L^2}+\|(1-|y|)\nabla f_b\|_{L^2}\big)\nonumber\\
&\leq C\varepsilon_0\nu^{\f34}|k|^{-\f14}(1-|\lambda|)^{-\f74}\|W\|_{L^2}\leq C\varepsilon_0\nu^{\f12}(1-|\lambda|)^{-1}\|W\|_{L^2}.\label{est: kappa w2}
\end{align}
Since  $\Delta p^{L(03)}=-2ik\partial_yV(\kappa U_b),\ \partial_yp^{L(03)}|_{y=\pm1}=0$, we get by Lemma \ref{lem:UB}  that \begin{align}
\|(\partial_y+\kappa\partial_z)p^{L(03)}\|_{H^1}\leq & C\|p^{L(03)}\|_{H^2}\leq C\|\Delta p^{L(03)}\|_{L^2}\leq C|k|\|\kappa U_b\|_{L^2}\nonumber\\ \leq& C|k|\|\kappa/(1-|y|)\|_{L^\infty}\|(1-|y|)U_b\|_{L^2}\nonumber\\
\leq&  C|k|\varepsilon_0\|(1-|y|)U_b\|_{L^2}\nonumber\\ \leq & C\varepsilon_0\nu^{\f34}|k|^{-\f14}(1-|\lambda|)^{-\f74}\|W\|_{L^2} \leq C\varepsilon_0\nu^{\f12}(1-|\lambda|)^{-1}\|W\|_{L^2}.\label{est: pL03 H1}
\end{align}
Recall that $F_3:=F_2+\nu^2(2\nabla\theta\cdot\nabla W_s+\Delta\theta W_s)-\kappa f_b-\nu F_1-(\partial_y+\kappa\partial_z)p^{L(03)}$. Then by \eqref{est: F2 H1}, \eqref{est: theta f2}, \eqref{est: kappa w2}, \eqref{eq: nabla F1} and \eqref{est: pL03 H1}, we conclude that
\begin{align}
\|\nabla F_3\|_{L^2}\leq&C\nu\varepsilon_0|k|^{-1}
\big(\|\nabla\partial_x^2U\|_{L^2}+\|\nabla\partial_x(\partial_z-\kappa\partial_y)U\|_{L^2}\big)\label{est: F3 H1}\\&
+C\varepsilon_0\big((1-|\lambda|)^{-1}\nu^{\f12}\|W\|_{L^2}+\nu^{\f23} |k|^{\f13}\|\nabla W\|_{L^2}\big).\nonumber
\end{align}

\subsubsection{Estimate of  Neumann data $\pa_yW_g$}

Using the fact that
\beno
(ik(V-\lambda)U_b-a(\nu k^2)^{\f13}U_b-\partial_zp^{L(0)})|_{y=\pm1}=(V-y)|_{y=\pm1}=0,
\eeno
 we find that
 \beno
 |k(y-\lambda)U_b|\leq |\partial_zp^{L(0)}|,\quad |k(y-\lambda)\partial_zU_b|\leq |\partial_z^2p^{L(0)}|\quad \text{on}\,\,\p\Om.
 \eeno
 Then by Lemma \ref{lem:sob-f}, we get
\begin{align*}
  \|k(y-\lambda)U_b\|_{L^2(\partial\Omega)}\leq& C\|\partial_zp^{L(0)}\|_{L^2(\partial\Omega)}\leq C\|\partial_zp^{L(0)}\|_{L^2_{x,z}L^\infty_y}\\ \leq& C|k|^{-\f12}\|p^{L(0)}\|_{H^2}
   \leq C|k|^{-\f12}\|\Delta p^{L(0)}\|_{L^2}\\
   \leq& C|k|^{\f12}\|\partial_yV\|_{L^\infty}\|W\|_{L^2}\leq  C|k|^{\f12}\|W\|_{L^2},
\end{align*}
and
\begin{align*}
   \|k(y-\lambda)\partial_zU_b\|_{L^2(\partial\Omega)}\leq&  C\|\partial^2_zp^{L(0)}\|_{L^2(\partial\Omega)}\leq C\|\partial^2_zp^{L(0)}\|_{L^2_{x,z}L^\infty_y}\\ \leq& C\|(\partial^2_z\partial_y,\partial^2_z)p^{L(0)}\|_{L^2}^{\f12} \|\partial^2_zp^{L(0)}\|_{L^2}^{\f12}\\
   \leq& C\|\nabla\Delta p^{L(0)}\|_{L^2}^{\f12}\|\Delta p^{L(0)}\|_{L^2}^{\f12}\leq C|k|\|\partial_yVW\|_{H^1}^{\f12} \|\partial_yVW\|_{L^2}^{\f12}\\
   \leq&C|k|\big(\|\partial_yV\|_{H^2}\|W\|_{H^1}\big)^{\f12}\big(\|\partial_yV\|_{L^\infty} \|W\|_{L^2}\big)^{\f12}\\
   \leq& C|k|\|\nabla W\|_{L^2}^{\f12}\|W\|_{L^2}^{\f12}.
\end{align*}

Using the fact that
\beno
[\partial_y(\theta W_s)]|_{y=\pm1}=\partial_y W|_{y=\pm1}=\kappa|_{y=\pm1}=0, \quad W_g=W-\nu\theta W_s-\kappa U_b,
\eeno
we deduce that
\begin{align*}
\|k(y-\lambda)\partial_yW_g\|_{L^2(\partial\Omega)} &\leq \|\partial_y\kappa\|_{L^\infty}\|k(y-\lambda)U_b\|_{L^2(\partial\Omega)}\leq C\varepsilon_0|k|^{\f12}\|U\|_{L^2},
\end{align*}
which  implies that
\begin{align*}
   \|\partial_yW_g\|_{L^2(\partial\Omega)}\leq& C|k(1-|\lambda|)|^{-1}\|k(y-\lambda)\partial_yW_g\|_{L^2(\partial\Omega)}\\
   \leq& C\varepsilon_0(1-|\lambda|)^{-1}|k|^{-\f12}\|W\|_{L^2}.
\end{align*}
By the interpolation, we get
\begin{align}
   \|(1+|k(y-\lambda)|)^{\alpha}\partial_yW_g\|_{L^2(\partial\Omega)}\leq& \|(1+|k(y-\lambda)|)\partial_yW_g\|_{L^2(\partial\Omega)}^{\alpha} \|\partial_yW_g\|_{L^2(\partial\Omega)}^{1-\alpha}\nonumber\\
   \leq&C\varepsilon_0|k|^{-\f12}\big((1-|\lambda|)^{-1}+|k|)^{\alpha}(1-|\lambda|)^{\alpha-1} \|W\|_{L^2}\nonumber\\
   \leq &C\varepsilon_0|k|^{-\f12}(1-|\lambda|)^{-1}|k|^{\alpha} \|W\|_{L^2}, \ \alpha\in[0,1].\label{est: partial yvarphi boundary}
\end{align}

Due to $\kappa|_{y=\pm1}=0, $ we have
  \begin{align*}
     &\partial_y\partial_z(\kappa U_b)|_{y=j}=(\partial_z\partial_y\kappa)U_b|_{y=j} +(\partial_y\kappa)\partial_zU_b|_{y=j}\quad j\in\{\pm1\},
  \end{align*}
  which implies  that
  \begin{align*}
     \|\partial_y\partial_z(\kappa U_b)\|_{L^2(\partial\Omega)} &\leq \|(\partial_y\partial_z\kappa)U_b\|_{L^2(\partial\Omega)}+ \|(\partial_y\kappa)\partial_zU_b\|_{L^2(\partial\Omega)}\\
     &\leq \|\partial_y\partial_z\kappa\|_{L^2_{z}(\partial\Omega)} \|U_b\|_{L^2_xL^\infty_z(\partial\Omega)}+ \|\partial_y\kappa\|_{L^\infty}\|\partial_zU_b\|_{L^2(\partial\Omega)}\\
     &\leq \|\partial_y\partial_z\kappa\|_{L^\infty_yL^2_z} \|(1,\partial_z)U_b\|_{L^2(\partial\Omega)}+C\|\kappa\|_{H^3}\|\partial_z U_b\|_{L^2(\partial\Omega)}\\
     &\leq C\varepsilon_0\big(\|U_b\|_{L^2(\partial\Omega)}+\|\partial_z U_b\|_{L^2(\partial\Omega)}\big).
  \end{align*}
  Thanks to $[\partial_z\partial_y(\theta W_s)]|_{y=\pm1}=\partial_z\partial_y
  W|_{y=\pm1}=0$, we obtain
  \begin{align*}
    \|\partial_y\partial_zW_g\|_{L^2(\partial\Omega)}&\leq \|\partial_y\partial_z(\kappa  U_b)\|_{L^2(\partial\Omega)}\leq C\varepsilon_0\big(\|U_b\|_{L^2(\partial\Omega)}+\|\partial_z U_b\|_{L^2(\partial\Omega)}\big)\\
    &\leq C\varepsilon_0|k|^{-1}\big(1-|\lambda|)^{-1} (\|k(y-\lambda)U_b\|_{L^2(\partial\Omega)}+\|k(y-\lambda)\partial_zU_b\|_{L^2(\partial\Omega)})\\
    &\leq C\varepsilon_0(1-|\lambda|)^{-1}\|W\|_{L^2}^{\f12}\|\nabla W\|_{L^2}^{\f12},
  \end{align*}
  which along with $1-|\lambda|\geq|\nu/k|^{\f13}=\nu^{\f23}|k|^{\f13}(\nu k^2)^{-\f13}\geq\nu^{\f23}|k|^{\f13} $ yields
  \begin{align}
   (1-|\lambda|)\nu^{\f{11}{12}}|k|^{\f13}\|\partial_y\partial_zW_g\|_{L^2(\partial\Omega)}\leq& C\varepsilon_0\nu^{\f{11}{12}}|k|^{\f13}\|W\|_{L^2}^{\f12}\|\nabla W\|_{L^2}^{\f12}\nonumber\\ \leq& C\varepsilon_0\big(\nu^{\f{1}{2}}\|W\|_{L^2}+\nu^{\f{4}{3}}|k|^{\f23}\|\nabla W\|_{L^2}\big)\nonumber\\
   \leq& C\varepsilon_0\big(\nu^{\f{1}{2}}\|W\|_{L^2}+(1-|\lambda|)\nu^{\f23} |k|^{\f13}\|\nabla W\|_{L^2}\big).\label{est: partial yz varphi1}
\end{align}

\subsubsection{Completion of the proof}

By Proposition \ref{prop:res-toy}, \eqref{est: F3 H1}, \eqref{est: partial yvarphi boundary} and \eqref{est: partial yz varphi1}, we get
\begin{align}
   \nu^{\f{1}{2}}|k|\|W_g\|_{L^2}
   \leq& C(1-|\lambda|)\Big(\|\nabla (G_1-F_3)\|_{L^2} +\nu^{\f12}|k|^{-\f12}\|(1+|k(y-\lambda)|)^{\f12}\partial_yW_g\|_{L^2(\partial\Omega)} \Big) \nonumber\\
   &+C\nu^{\f56}|k|^{\f16}\|\partial_yW_g\|_{L^2(\partial\Omega)} +(1-|\lambda|)\nu^{\f{11}{12}}|k|^{-\f23}\|\partial_y\partial_zW_g\|_{L^2(\partial \Omega)}\nonumber\\
   \leq &C(1-|\lambda|)\|\nabla G_1\|_{L^2} +C\nu\varepsilon_0|k|^{-1}(1-|\lambda|)
(\|\nabla\partial_x^2U\|_{L^2}+ \|\nabla\partial_x(\partial_z-\kappa\partial_y)U\|_{L^2})\label{est: varphi1 prior}\\&+
C\varepsilon_0\big(\nu^{\f12}\|W\|_{L^2}+(1-|\lambda|)\nu^{\f23} |k|^{\f13}\|\nabla W\|_{L^2}\big),\nonumber
\end{align}
and
\begin{align}
  \nonumber&\nu^{\f12}|k|\|\nabla W_g\|_{L^2}+\nu^{\f34}|k|^{\f12}\|\Delta W_g\|_{L^2} +\nu|k|^{\f12}\|\Delta W_g\|_{L^2(\partial\Omega)} +\|\partial_x\nabla p^{L(1)}\|_{L^2}\\ &\leq C\|\nabla(G_1-F_3)\|_{L^2}+C\nu^{\f12}|k|^{\f12} \|(1+|k(y-\lambda)|)^{\f12}\partial_yW_g\|_{L^2(\partial\Omega)}\nonumber\\&\quad+C\nu^{\f{11}{12}}|k|^{\f13}\|\partial_y\partial_zW_g\|_{L^2(\partial \Omega)}\nonumber\\
  &\leq C\|\nabla G_1\|_{L^2}+C\nu\varepsilon_0|k|^{-1}\big(\|\nabla\partial_x^2 U\|_{L^2}+\|\nabla\partial_x(\partial_z-\kappa\partial_y)U\|_{L^2}\big)\nonumber\\
  &\quad+C\varepsilon_0(1-|\lambda|)^{-1}( \nu^{\f12}+\nu^{\f12}|k|^{\f12})\|W\|_{L^2} +C\varepsilon_0\nu^{\f23}|k|^{\f13}\|\nabla W\|_{L^2}\nonumber\\
  &\leq  C\|\nabla G_1\|_{L^2}+C\nu\varepsilon_0|k|^{-1}\big(\|\nabla\partial_x^2 U\|_{L^2}+\|\nabla\partial_x(\partial_z-\kappa\partial_y)U\|_{L^2}\big) \label{est: nabla varphi1 prior}\\
  &\quad+C\varepsilon_0(1-|\lambda|)^{-1}\nu^{\f12}|k|^{\f12}\|W\|_{L^2}+ C\varepsilon_0\nu^{\f23}|k|^{\f13}\|\nabla W\|_{L^2}.\nonumber
\end{align}

Recalling that $p^{L(0)}=p^{L(2)}+\nu p^{L(01)}+p^{L(03)}$, we get by Proposition \ref{prop:res-nav-good} and  Lemma \ref{lem: varphi varphi1} that
\begin{align}
&\nu^{\f{1}{3}}\big(\|\partial_x^2U\|_{L^2}+\|\partial_x(\partial_z-\kappa\partial_y)U\|_{L^2}\big)+ \nu^{\f{2}{3}}
\big(\|\nabla\partial_x^2U\|_{L^2}+ \|\nabla\partial_x(\partial_z-\kappa\partial_y)U\|_{L^2}\big)\nonumber\\ &\leq C\nu^{\f{1}{6}}\| \partial_zp^{L(2)}\|_{H^2}+C|k|^{-\f13}\Big(\| \partial_x^2\partial_z(\nu p^{L(01)}+p^{L(03)})\|_{L^2}\nonumber\\
&\qquad+ \|\partial_x(\partial_z-\kappa\partial_y)\partial_z(\nu p^{L(01)}+p^{L(03)})\|_{L^2}\Big)\nonumber\\ &\leq C\nu^{\f{1}{6}}\| \partial_zp^{L(2)}\|_{H^2}+C|k|^{-\f13}\|\nabla \partial_x\partial_z(\nu p^{L(01)}+p^{L(03)})\|_{L^2}\nonumber\\ &\leq C\nu^{\f{1}{6}}\| \Delta\partial_zp^{L(2)}\|_{L^2}+C|k|^{-\f13}|k|\|\Delta(\nu p^{L(01)}+p^{L(03)})\|_{L^2}\nonumber\\ &\leq C\nu^{\f{1}{6}}\|\nabla W_g\|_{L^2}+C|k|^{-\f13}|k|^2\big(\| \nu\theta W_s\|_{L^2}+\| \kappa U_b\|_{L^2}\big)\nonumber\\ &\leq C\nu^{\f{1}{6}}\| \nabla W_g\|_{L^2}+C\varepsilon_0\nu^{\f16} |k|\|W\|_{L^2}.\label{est: u good deri prior}
\end{align}

Now we infer from  \eqref{est: varphi1 prior}, \eqref{est: u good deri prior} and Lemma \ref{lem: varphi varphi1} that
\begin{align*}
  \nu^{\f12}|k|(1-|\lambda|)^{-1}\|W\|_{L^2}&\leq C\|\nabla G_1\|_{L^2}+ C\varepsilon_0\nu^{\f12}|k|^{-1}\big(\|\nabla W_g\|_{L^2}+|k|\|W\|_{L^2})\\
  &\quad+C\varepsilon_0 \nu^{\f12}(1-|\lambda|)^{-1}\|W\|_{L^2}+ C\varepsilon_0\nu^{\f23}|k|^{\f13}\|\nabla W\|_{L^2}\\
  &\leq C\|\nabla G_1\|_{L^2}+ C\varepsilon_0\nu^{\f12}(1-|\lambda|)^{-1}\|W\|_{L^2}+C\nu^{\f12} \varepsilon_0\|\nabla W_g\|_{L^2},
\end{align*}
which gives by taking $\veps_0$ small enough so that $C\varepsilon_0\leq 1/2$ that
\begin{align}\label{est: varphi prior 2}
\nu^{\f12}|k|(1-|\lambda|)^{-1}\|W\|_{L^2} &\leq C\big(\|\nabla G_1\|_{L^2}+\varepsilon_0\nu^{\f12} \|\nabla W_g\|_{L^2}\big).
\end{align}
Then by \eqref{est: nabla varphi1 prior}, \eqref{est: u good deri prior}, \eqref{est: varphi prior 2} and Lemma \ref{lem: varphi varphi1},
we  get
\begin{align*}
  \nonumber&\nu^{\f12}|k|\|\nabla W_g\|_{L^2}+\nu^{\f34}|k|^{\f12}\|\Delta W_g\|_{L^2} +\nu|k|^{\f12}\|\Delta W_g\|_{L^2(\partial\Omega)} +\|\partial_x\nabla p^{L(1)}\|_{L^2}\\
  &\leq C\|\nabla G_1\|_{L^2}+C\varepsilon_0\nu^{\f12}|k|^{-1} (\|\nabla W_g\|_{L^2}+|k|\|W\|_{L^2})\\& \quad+C\varepsilon_0(1-|\lambda|)^{-1}\nu^{\f12}|k|^{\f12}\|W\|_{L^2}+C\varepsilon_0\nu^{\f23}|k|^{\f13}\|\nabla W\|_{L^2}\\
  &\leq C\|\nabla G_1\|_{L^2} +C\varepsilon_0\nu^{\f12}\|\nabla W_g\|_{L^2}+C\varepsilon_0\nu^{\f12}\|W\|_{L^2}\\
  &\leq C\|\nabla G_1\|_{L^2} +C\varepsilon_0\nu^{\f12}\|\nabla W_g\|_{L^2},
\end{align*}
which gives by taking $\veps_0$ small enough so that $C\varepsilon_0\leq 1/2$ that
\begin{align}\label{est: nabla varphi1 prior 2}
 &\nu^{\f12}|k|\|\nabla W_g\|_{L^2}+\nu^{\f34}|k|^{\f12}\|\Delta W_g\|_{L^2} +\|\partial_x\nabla p^{L(1)}\|_{L^2} \leq  C\|\nabla G_1\|_{L^2}.
\end{align}

Now, by \eqref{est: nabla varphi1 prior 2}, \eqref{est: varphi prior 2} and \eqref{est: u good deri prior}, we get
\begin{align}
  &\nu^{\f12}|k|(1-|\lambda|)^{-1}\|W\|_{L^2}\leq C\|\nabla G_1\|_{L^2},\label{est: varphi nabla G}\\
  &\nu^{\f{1}{3}} \big(\|\partial_x^2U\|_{L^2}+\|\partial_x(\partial_z-\kappa\partial_y)U\|_{L^2})+ \nu^{\f{2}{3}} \big(\|\nabla\partial_x^2U\|_{L^2}+ \|\nabla\partial_x(\partial_z-\kappa\partial_y)U\|_{L^2}\big)
  \label{est: u good deriv nabla G}\\
  &\leq C\nu^{-\f13}\|\nabla G_1\|_{L^2}.\nonumber
\end{align}
and by Lemma \ref{lem: varphi varphi1},
\begin{align}
\|\partial_x\nabla W\|_{L^2}&=|k|\|\nabla W\|_{L^2}\leq C|k|\big(\|\nabla W_g\|_{L^2}+\nu^{-\f16} |k|^{-\f13}\|W\|_{L^2}\big)\label{est: nabla varphi nabla G}\\
&\leq C\nu^{-\f{1}{2}}\|\nabla G_1\|_{L^2}+C\nu^{-\f{2}{3}}\|\nabla G_1\|_{L^2}\leq C\nu^{-\f{2}{3}}\|\nabla G_1\|_{L^2}.\nonumber
\end{align}
By \eqref{est: nabla varphi1 prior 2}, \eqref{est: varphi nabla G} and Lemma \ref{lem: varphi varphi1},  we have
\begin{align}
\|\partial_x\Delta W\|_{L^2}&=|k|\|\Delta W\|_{L^2}\leq C|k|\big(\|\Delta W_g\|_{L^2}+\nu^{-\f12} \|W\|_{L^2}\big)
\label{est: delta varphi nabla G}\\
&\leq C\nu^{-\f{3}{4}} |k|^{\f{1}{2}}\|\nabla G_1\|_{L^2}+C\nu^{-1}\|\nabla G_1\|_{L^2}\leq C\nu^{-1}\|\nabla G_1\|_{L^2}.\nonumber
\end{align}
By  \eqref{eq: u priori}, \eqref{est: nabla varphi1 prior 2} and Lemma \ref{lem: varphi varphi1}, we have
\begin{align}
\nonumber\nu\|\partial_x\Delta U\|_{L^2}&\leq C\nu^{\f16}|k|^{\f43}\|W\|_{L^2}\leq C\nu^{\f16}|k|^{\f43}\|W_g\|_{L^2}\\&\leq C\nu^{\f16}|k|\|\nabla W_g\|_{L^2}\leq  C\nu^{-\f{1}{3}}\|\nabla G_1\|_{L^2}.\label{est: delta u nabla G}
\end{align}

Finally, we can conclude the proposition by  \eqref{est: u good deriv nabla G}-\eqref{est: delta u nabla G} and \eqref{est: nabla varphi1 prior 2}.

\section{Resolvent estimates  for the full linearized NS system}

In this section, we consider the full linearized NS system
\begin{align}\label{eq:LNS-full}
  \left\{
  \begin{aligned}
   &-\nu\Delta W+ik(V(y,z)-\lambda)W-a(\nu k^2)^{1/3}W+(\partial_y+\kappa\partial_z)p^{L1}\\&\qquad+G_1+\nu(\Delta \kappa) U+2\nu\nabla\kappa\cdot\nabla U=0,\\&-\nu\Delta U+ik(V(y,z)-\lambda)U-a(\nu k^2)^{1/3}U+G_2+\partial_zp^{L1}=0,\\
   & W|_{y=\pm1}=\partial_yW|_{y=\pm1}=u|_{y=\pm1}=0,
   \end{aligned}\right.
\end{align}
where
\beno
\Delta p^{L1}=-2ik\partial_yVW,\quad  \partial_xW=ikW,\ \partial_xU=ikU,\ \partial_xp^{L1}=ikp^{L1}.
\eeno

We assume that  $\la\in \R, a\in [0,\eps_1]$ and $V$ satisfies \eqref{ass:V}.

\begin{Proposition}\label{prop:res-full}
Let $W\in H^4(\Omega),\ U\in H^2(\Omega)$ be a solution of \eqref{eq:LNS-full}. Then it holds that
\begin{align*}
&\nu^{\f{1}{3}}\big(\|\partial_x^2U\|_{L^2}^2+\|\partial_x(\partial_z-\kappa\partial_y)U\|_{L^2}^2\big)+
\nu\big(\|\nabla\partial_x^2U\|_{L^2}^2+\|\nabla\partial_x(\partial_z-\kappa\partial_y)U\|_{L^2}^2\big)
\\&\qquad+\nu^{\f{1}{3}} \|\partial_x\nabla W\|_{L^2}^2+\nu\|\partial_x\Delta W\|_{L^2}^2+\nu^{\f{5}{3}} \|\partial_x\Delta U\|_{L^2}^2\\
&\leq C\nu^{-1}\big(\|\nabla G_1\|_{L^2}^2+\|\partial_x G_2\|_{L^2}^2\big).
\end{align*}
\end{Proposition}

The proof  will be split into two cases: $\nu k^2\ge 1$ and $\nu k^2\le 1$.

\subsection{Case of $\nu k^2\ge 1$}

First of all,  we consider the following system
\begin{align}\label{eq:LNS-full-s1}
  \left\{
  \begin{aligned}
   &-\nu\Delta W+ik(V(y,z)-\lambda)W-a(\nu k^2)^{1/3}W+\partial_yp^{L1}+G=0,\\&-\nu\Delta U+ik(V(y,z)-\lambda)U-a(\nu k^2)^{1/3}U+\partial_zp^{L1}=0,\\
   &W|_{y=\pm1}=\partial_yW|_{y=\pm1}=U|_{y=\pm1}=0,\\& \partial_xW=ikW,\ \partial_xU=ikU,\ \partial_xp^{L1}=ikp^{L1}.
   \end{aligned}\right.
\end{align}
where $\Delta p^{L1}=-2ik(\partial_yVW+\partial_zVU)$.

\begin{Lemma}\label{lem:res-full-s1}
  Let $\nu k^2\ge 1$, and $ W\in H^4(\Omega),\ U\in H^2(\Omega)$ be a solution of \eqref{eq:LNS-full-s1}. Then it holds that
\begin{align*}
&\nu\big(\|\partial_x\Delta U\|_{L^2}+\|\partial_x\Delta W\|_{L^2}\big)\leq C\|\nabla G\|_{L^2}.
\end{align*}
\end{Lemma}
\begin{proof}
We get by integration by parts that
\begin{align*}
  &\big\langle ik(V-\lambda) W +\partial_yp^{L1},\Delta W\big\rangle +\big\langle ik(V-\lambda)\nabla W ,\nabla W\big\rangle+\big\langle ik(\nabla V) W ,\nabla W\big\rangle\\&=\big\langle \partial_y\Delta p^{L1}, W\big\rangle=-\big\langle \Delta p^{L1},\partial_y W\big\rangle=\big\langle 2ik(\partial_yV W+\partial_zVU),\partial_y W\big\rangle,
\end{align*}
and
\begin{align*}
  &\big\langle ik(V-\lambda) W +\partial_yp^{L1},\Delta W\big\rangle=\big\langle \nu\Delta  W+a(\nu k^2)^{1/3} W-G,\Delta W\big\rangle\\&=\nu\|\Delta W\|^2_{L^2}-a(\nu k^2)^{1/3}\|\nabla W\|_{L^2}^2-\big\langle G,\Delta W\big\rangle,
\end{align*}
from which, we infer that
\begin{align*}
  &\nu\|\Delta W\|^2_{L^2}-a(\nu k^2)^{1/3}\|\nabla W\|_{L^2}^2-\big\langle G,\Delta W\big\rangle +\big\langle ik(V-\lambda)\nabla W ,\nabla W\big\rangle\\
  &=\big\langle ik(\partial_yV,-\partial_zV) W,(\partial_y,\partial_z) W\big\rangle+\big\langle 2ik\partial_zVU,\partial_y W\big\rangle.
\end{align*}
Taking the real part, we get
\begin{align*}
  &\nu\|\Delta W\|^2_{L^2}-a(\nu k^2)^{1/3}\|\nabla W\|_{L^2}^2-\|G\|_{L^2}\|\Delta W\|_{L^2}-2|k|\|\partial_zV\|_{L^{\infty}}\|U\|_{L^2}\|\partial_y W\|_{L^2}\\ &\leq |k|\|\nabla V\|_{L^{\infty}}\| W\|_{L^2}\|(\partial_y,\partial_z) W\|_{L^2}\leq\|\nabla V\|_{L^{\infty}}(|k|^2\| W\|_{L^2}^2+\|(\partial_y,\partial_z) W\|_{L^2}^2)/2\\&=\|\nabla V\|_{L^{\infty}}\|\nabla W\|_{L^2}^2/2.\end{align*}
Notice that
   \begin{align*}
      &\|\Delta W\|^2_{L^2}=\|[(\partial^2_y+\partial_z^2)-k^2] W\|_{L^2}^2 =k^2\| W\|^2_{L^2}+2|k|\|(\partial_y,\partial_z) W\|_{L^2}^2+ \|(\partial_y^2+\partial_z^2) W\|_{L^2}^2,\\
      &\Rightarrow k^2\| W\|_{L^2}\leq \|\Delta W\|_{L^2},\quad  |k|\|\nabla  W\|_{L^2}\leq \|\Delta W\|_{L^2},
   \end{align*}
and   $\|\nabla V\|_{L^\infty}\leq \|\nabla(y)\|_{L^\infty}+\|\nabla(V-y)\|_{L^\infty}\leq 1+C\varepsilon_0$ and $\nu k^2\geq1$.
Take $\varepsilon_0,\ \eps_1$ sufficiently small so that $\|\nabla V\|_{L^{\infty}}/2+a\leq(1+C\varepsilon_0)/2+\eps_1\leq 3/4$,
and then
\begin{align*}
\big(\|\nabla V\|_{L^{\infty}}/2+a(\nu k^2)^{1/3}\big)\|\nabla  W\|_{L^2}^2\leq
\big(\|\nabla V\|_{L^{\infty}}/2+a\big)(\nu k^2)\|\nabla  W\|_{L^2}^2\leq(3/4)\nu \|\Delta W\|_{L^2}^2.
\end{align*}
Thus, we deduce that
that\begin{align*}
  &\nu\|\Delta W\|^2_{L^2}/4- \|G\|_{L^2}\|\Delta W\|_{L^2}\leq2|k|\|\partial_zV\|_{L^{\infty}}\|U\|_{L^2}\|\partial_y W\|_{L^2}\leq C\varepsilon_0\|U\|_{L^2}\|\Delta W\|_{L^2},
\end{align*}
which gives
  \begin{align}\label{est: varphi H3 prior}
&\nu\|\Delta W\|_{L^2}\leq C\|G\|_{L^2}+C\varepsilon_0\|U\|_{L^2}.
\end{align}

To proceed, we need to estimate the pressure $p^{L1}$. Let $F_1=(V-\lambda) W$, which satisfies
\beno
\Delta F_1=(V-\lambda)\Delta W+2\nabla V\cdot\nabla W+\Delta V W,\quad F_1|_{y=\pm 1}=\partial_yF_1|_{y=\pm 1}=0.
\eeno
We get by integration by parts that\begin{align*}
  &\big\langle \partial_yp^{L1},\Delta F_1\big\rangle= \big\langle \partial_y\Delta p^{L1},F_1\big\rangle=-\big\langle \Delta p^{L1},\partial_yF_1\big\rangle,\\&\big\langle ik(V-\lambda) W,\Delta F_1\big\rangle=\big\langle ikF_1,\Delta F_1\big\rangle=-ik\|\nabla F_1\|_{L^2}^2,\\&\big\langle \nu\Delta W,\Delta F_1\big\rangle=\big\langle \nu\Delta W,(V-\lambda)\Delta W\big\rangle+\big\langle \nu\Delta W,2\nabla V\cdot\nabla W+\Delta V W\big\rangle,\\&\big\langle -a(\nu k^2)^{1/3} W+G,\Delta F_1\big\rangle=a(\nu k^2)^{1/3}\big\langle \nabla W,\nabla F_1\big\rangle-\big\langle\nabla G,\nabla F_1\big\rangle.
\end{align*}
Then by \eqref{eq:LNS-full-s1}, we get
\begin{align*}
  &-\big\langle \nu\Delta W,(V-\lambda)\Delta W\big\rangle-\big\langle \nu\Delta W,2\nabla V\cdot\nabla W+\Delta V W\big\rangle-ik\|\nabla F_1\|_{L^2}^2\\&\quad+a(\nu k^2)^{1/3}\big\langle \nabla W,\nabla F_1\big\rangle-\big\langle\nabla G,\nabla F_1\big\rangle-\big\langle \Delta p^{L1},\partial_yF_1\big\rangle=0.
\end{align*}
Taking the imaginary part, we get
\begin{align*}
  &|k|\|\nabla F_1\|_{L^2}^2-a(\nu k^2)^{1/3}\|\nabla W\|_{L^2}\|\nabla F_1\|_{L^2}-\|\Delta p^{L1}\|_{L^2}\|\nabla F_1\|_{L^2}-\|\nabla G\|_{L^2}\|\nabla F_1\|_{L^2}\\ &\leq \nu\|\Delta W\|_{L^2}\|2\nabla V\cdot\nabla W+\Delta V W\|_{L^2}\leq C\nu\|\Delta W\|_{L^2}(\|\nabla W\|_{L^2}+\| W\|_{L^2})\leq C\nu|k|^{-1}\|\Delta W\|_{L^2}^2.
\end{align*}
which implies that
\begin{align*}
  \|\nabla F_1\|_{L^2}&\leq C(\nu k^2)^{\f13}|k|^{-1}\|\nabla W\|_{L^2}+C|k|^{-1}(\|\Delta p^{L1}\|_{L^2}+\|\nabla G\|_{L^2})+ C\nu^{\f12}|k|^{-1}\|\Delta W\|_{L^2}\\&\leq C|k|^{-1}(\|\Delta p^{L1}\|_{L^2}+\|\nabla G\|_{L^2})+ C\nu\|\Delta W\|_{L^2}.
\end{align*}
Here we used $\nu k^2\ge 1$ and
\begin{align*}
&(\nu k^2)^{\f13}|k|^{-1}\|\nabla W\|_{L^2}\leq(\nu k^2)^{\f12}|k|^{-1}\|\nabla W\|_{L^2}=\nu^{\f12}\|\nabla W\|_{L^2}\leq\nu^{\f12}|k|^{-1}\|\Delta W\|_{L^2}\leq\nu\|\Delta W\|_{L^2}.
\end{align*}
Since $-\nu\Delta  W+ikF_1-a(\nu k^2)^{1/3} W+\partial_yp^{L1}+G=0$, we have
\begin{align*}
\|\partial_yp^{L1}\|_{L^2}&\leq\nu\|\Delta  W\|_{L^2}+a(\nu k^2)^{1/3}\|  W\|_{L^2}+\|kF_1\|_{L^2}+\|G\|_{L^2}\\&\leq2\nu\|\Delta  W\|_{L^2}+\|\nabla F_1\|_{L^2}+|k|^{-1}\|\nabla G\|_{L^2}\\&\leq C\nu\|\Delta  W\|_{L^2}+C|k|^{-1}\big(\|\Delta p^{L1}\|_{L^2}+\|\nabla G\|_{L^2}\big).
\end{align*}Here we used $(\nu k^2)^{1/3}\|  W\|_{L^2}\leq\nu k^2\|  W\|_{L^2}\leq\nu\|\Delta  W\|_{L^2} .$ On the other hand, we have
\begin{align*}
|k|^{-1}\|\Delta p^{L1}\|_{L^2}=&|k|^{-1}\|2ik(\partial_yV W+\partial_zVU)\|_{L^2}\\ \leq& 2\|\partial_yV\|_{L^{\infty}}\| W\|_{L^2}+2\|\partial_zV\|_{L^{\infty}}\|u\|_{L^2}\leq C(\| W\|_{L^2}+\varepsilon_0\|U\|_{L^2})
\end{align*}
Thus, by Lemma \ref{delta f}, we have
\begin{align*}
\|\partial_zp^{L1}\|_{L^2}&\leq C\big(\|\partial_yp^{L1}\|_{L^2}+|k|^{-1}\|\Delta p^{L1}\|_{L^2}\big)\\
&\leq C\nu\|\Delta  W\|_{L^2}+C|k|^{-1}(\|\Delta p^{L1}\|_{L^2}+\|\nabla G\|_{L^2})\\&\leq C\nu\|\Delta  W\|_{L^2}+C(\| W\|_{L^2}+\varepsilon_0\|U\|_{L^2}+|k|^{-1}\|\nabla G\|_{L^2})\\&\leq C\|G\|_{L^2}+C\varepsilon_0\|U\|_{L^2}+C|k|^{-1}\|\nabla G\|_{L^2}\\
&\leq C\varepsilon_0\|U\|_{L^2}+C|k|^{-1}\|\nabla G\|_{L^2}.
\end{align*}
Here we used
\begin{align*}
\| W\|_{L^2}\leq |k|^{-2}\|\Delta  W\|_{L^2}\leq\nu\|\Delta  W\|_{L^2}\leq C\|G\|_{L^2}+C\varepsilon_0\|U\|_{L^2}.
\end{align*}

Now, by Proposition \ref{prop:res-nav-s1},  we have
  \begin{align*}
     \nu k^2\| U\|_{L^2}\leq\nu\|\Delta U\|_{L^2}&\leq C\|\partial_z p^{L1}\|_{L^2}\leq C\varepsilon_0\|U\|_{L^2}+C|k|^{-1}\|\nabla G\|_{L^2}.
  \end{align*}
Due to $\nu k^2\geq1$, taking $\veps_0$ small enough so that $C\varepsilon_0\leq1/2$, we obtain
\begin{align*}
     &\|U\|_{L^2}\leq\nu k^2\|U\|_{L^2}\leq C|k|^{-1}\|\nabla G\|_{L^2},\quad \nu\|\Delta U\|_{L^2}\leq C|k|^{-1}\|\nabla G\|_{L^2},
  \end{align*}
which along with  \eqref{est: varphi H3 prior}  gives
  \begin{align*}
    \nu\|\partial_x\Delta U\|_{L^2}+\nu\|\partial_x\Delta  W\|_{L^2}=&|k|\big(\nu\|\Delta U\|_{L^2}+\nu\|\Delta  W\|_{L^2}\big) \\ \leq& C\|\nabla G\|_{L^2}+C|k|\big(\|G\|_{L^2}+\varepsilon_0\|U\|_{L^2}\big)\leq C\|\nabla G\|_{L^2}.
  \end{align*}

This proves the lemma.
\end{proof}\smallskip

Now let us prove Proposition \ref{prop:res-full}  when $\nu k^2\ge 1$.

\begin{proof}We decompose $U=U_1+U_2$, where $(U_1,U_2)$ solves
\begin{align*}
  \left\{
  \begin{aligned}
   &-\nu\Delta U_1+ik(V(y,z)-\lambda)U_1-a(\nu k^2)^{1/3}U_1+\partial_zp^{L1}=0,\\
   &-\nu\Delta U_2+ik(V(y,z)-\lambda)U_2-a(\nu k^2)^{1/3}U_2+G_2=0,\\
   &U_1|_{y=\pm1}=U_2|_{y=\pm1} =0,\\
   & \partial_xU_1=ikU_1,\ \partial_xU_2=ikU_2.
   \end{aligned}\right.
\end{align*}
Let $G_3=G_1+\nu(\Delta\kappa)U_2+2\nu\nabla \kappa\cdot\nabla U_2$. Then we find that
\begin{align*}
\left\{
  \begin{aligned}
   &-\nu\Delta ( W-\kappa U_1)+ik(V(y,z)-\lambda)( W-\kappa U_1)-a(\nu k^2)^{1/3}( W-\kappa U_1)+\partial_yp^{L1}+G_3=0,\\&\quad-\nu\Delta U_1+ik(V(y,z)-\lambda)U_1-a(\nu k^2)^{1/3}U_1+\partial_zp^{L1}=0,\\
   &\Delta p^{L1}=-2ik\big[\partial_yV( W-\kappa U_1)+\partial_zVU_1\big],\\&( W-\kappa U_1)|_{y=\pm1}=\partial_y( W-\kappa U_1)|_{y=\pm1}=U_1|_{y=\pm1}=0.
   \end{aligned}\right.
\end{align*}
Then we infer from Lemma \ref{lem:res-full-s1} that
\begin{align*}
&\nu\big(\|\partial_x\Delta U_1\|_{L^2}+\|\partial_x\Delta W\|_{L^2}\big)\leq C\nu\big(\|\partial_x\Delta U_1\|_{L^2}+\|\partial_x\Delta( W-\kappa U_1)\|_{L^2}\big)\leq C\|\nabla G_3\|_{L^2}.
\end{align*}
And by Proposition \ref{prop:res-nav-s1},  we have
\begin{align*}
  \nu\|\Delta U_2\|_{L^2} &\leq C\|G_2\|_{L^2}.
\end{align*}
Thanks to the definition of $G_3$, we have
\begin{align*}
  \|\nabla G_3\|_{L^2} &\leq \|\nabla G_1\|_{L^2}+C\nu\|\nabla[(\Delta \kappa)U_2]\|_{L^2}+C\nu\|\nabla(\nabla\kappa\cdot\nabla U_2)\|_{L^2}\\
  &\leq \|\nabla G_1\|_{L^2}+C\nu\big(\|\Delta\kappa\|_{H^1}\|U_2\|_{H^2}+\|\nabla\kappa\|_{H^2}\|\nabla U_2\|_{H^1}\big)\\
  &\leq \|\nabla G_1\|_{L^2}+C\nu\varepsilon_0\|\Delta U_2\|_{L^2}\leq \|\nabla G_1\|_{L^2}+C\varepsilon_0\|G_2\|_{L^2}.
\end{align*}
This shows that
\begin{align*}
&\nu\big(\|\partial_x\Delta U\|_{L^2}+\|\partial_x\Delta W\|_{L^2}\big)\leq C\|\nabla G_1\|_{L^2}+C|k|\|G_2\|_{L^2}\leq C\big(\|\nabla G_1\|_{L^2}+\|\partial_xG_2\|_{L^2}\big),
\end{align*}
from which and  $\nu k^2\geq 1$, we infer that
\begin{align*}
    & \nu^{\f23}\big(\|\partial_x^2U\|_{L^2} +\|\partial_x(\partial_z-\kappa\partial_y)U\|_{L^2}\big)+ \nu\big(\|\nabla\partial_x^2U\|_{L^2} +\|\nabla\partial_x(\partial_z-\kappa\partial_y)U\|_{L^2}\big)\\
  &\qquad+\nu^{\f23}\|\partial_x\nabla W\|_{L^2}+ \nu\|\partial_x\Delta W\|_{L^2} +\nu^{\f43}\|\partial_x\Delta U\|_{L^2}\\
  &\leq C\Big(\nu^{\f23}|k|^{-1}\|\partial_x\Delta U\|_{L^2}+\nu\|\partial_x\Delta U\|_{L^2} + \nu^{\f23}|k|^{-1}\|\partial_x\Delta W\|_{L^2}+\nu\|\partial_x \Delta W\|_{L^2} +\nu^{\f43}\|\partial_x\Delta U\|_{L^2}\Big)\\
  &\leq C\nu(\|\partial_x\Delta U\|_{L^2}+\|\partial_x\Delta W\|_{L^2})\leq C\big(\|\partial_xG_2\|_{L^2}+\|\nabla G_1\|_{L^2}\big).
\end{align*}

This completes the proof when $\nu k^2\ge 1$.
\end{proof}

\subsection{Case of $\nu k^2\le 1$}

First of all,  we decompose $U=U_1+U_2$ and $p^{L1}=p^{L(0)}+p^{L(1)}$, where $\big(U_1, U_2, W\big)$ solve
\begin{align}\label{eq:LNS-full-s2-2}
  \left\{
  \begin{aligned}
   &-\nu\Delta  W+ik(V(y,z)-\lambda) W-a(\nu k^2)^{1/3} W+(\partial_y+\kappa\partial_z)p^{L1}\\&\qquad+G_1+\nu(\Delta \kappa) U+2\nu\nabla\kappa\cdot\nabla U=0,\\&-\nu\Delta U_1+ik(V(y,z)-\lambda)U_1-a(\nu k^2)^{1/3}U_1+\partial_zp^{L(0)}=0,\\
   &-\nu\Delta U_2+ik(V(y,z)-\lambda)U_2-a(\nu k^2)^{1/3}U_2+G_2+\partial_zp^{L(1)}=0,\\
   &\Delta p^{L(0)}=-2ik\partial_yV W,\quad \partial_yp^{L(0)}|_{y=\pm1}=0,\quad \Delta p^{L(1)}=0,\\
   & W|_{y=\pm1}=\partial_y W|_{y=\pm1}=U_1|_{y=\pm1}=U_2|_{y=\pm1} =0.
 \end{aligned}\right.
\end{align}

With this decomposition,  we can apply Proposition \ref{prop:res-full-s2} with $G=G_1+\nu(\Delta\kappa)U_2+2\nu\nabla \kappa\cdot\nabla U_2$
to the system \eqref{eq:LNS-full-s2-2} of $( W, U_1)$  to obtain
\begin{align}
&\nu^{\f{1}{3}}\big(\|\partial_x^2U_1\|_{L^2}^2+\|\partial_x(\partial_z-\kappa\partial_y)U_1\|_{L^2}^2\big)+\nu\big(\|\nabla\partial_x^2U_1\|_{L^2}^2+\|\nabla\partial_x(\partial_z-\kappa\partial_y)U_1\|_{L^2}^2\big)
\label{est: u1 prior nabla G2}\\&\quad+\nu^{\f{1}{3}} \|\partial_x\nabla W\|_{L^2}^2+\nu\|\partial_x\Delta W\|_{L^2}^2+\nu^{\f{5}{3}} \|\partial_x\Delta U_1\|_{L^2}^2+\nu^{-1}\|\partial_x\nabla p^{L(1)}\|_{L^2}^2\nonumber\\
&\qquad\leq C\nu^{-1}\|\nabla G\|_{L^2}^2.\nonumber
\end{align}
By Proposition \ref{prop:res-nav-s1} and \eqref{est: u1 prior nabla G2}, we get
\begin{align*}
  \nu^{\f23}|k|^{\f13}\|\nabla U_2\|_{L^2}+\nu\|\Delta U_2\|_{L^2} &\leq C\big(\|G_2\|_{L^2}+\|\nabla p^{L(1)}\|_{L^2}\big)\\
  &\leq C|k|^{-1}\big(\|\partial_xG_2\|_{L^2}+\|\nabla G\|_{L^2}\big).
\end{align*}
As in the case above, we have
\begin{align*}
  \|\nabla G\|_{L^2}
  &\leq \|\nabla G_1\|_{L^2}+C\nu\varepsilon_0\|\Delta U_2\|_{L^2}\leq \|\nabla G_1\|_{L^2}+C\varepsilon_0|k|^{-1}\big(\|\partial_x G_2\|_{L^2}+\|\nabla G\|_{L^2}\big),
\end{align*}
which gives by taking  $\veps_0$ small enough so that $C\varepsilon_0\leq1/2$ that
\begin{align}\label{est: G2 control}
   &\|\nabla G\|_{L^2}\leq C\big(\|\nabla G_1\|_{L^2}+\|\partial_xG_2\|_{L^2}\big).
\end{align}
Then we obtain
\begin{align}
  & \nu^{\f23}\big(\|\partial_x^2U_2\|_{L^2} +\|\partial_x(\partial_z-\kappa\partial_y)U_2\|_{L^2}\big)+ \nu\big(\|\nabla\partial_x^2U_2\|_{L^2} +\|\nabla\partial_x(\partial_z-\kappa\partial_y)U_2\|_{L^2}\big)\nonumber\\
  &\quad+\nu^{\f43}\|\partial_x\Delta U_2\|_{L^2}
  \leq C\big(\nu^{\f23}\|\partial_x\nabla U_2\|_{L^2}+\nu\|\partial_x\Delta U_2\|_{L^2}\big)\nonumber\\
  &\leq C|k|\big(\nu^{\f23}\|\nabla U_2\|_{L^2}+\nu\|\Delta U_2\|_{L^2}\big)
  \leq C\big(\|\partial_xG_2\|_{L^2}+\|\nabla G\|_{L^2}\big).
  \label{est: u2 prior nabla G2}
\end{align}
It follows from \eqref{est: u1 prior nabla G2}, \eqref{est: u2 prior nabla G2} and \eqref{est: G2 control}  that
\begin{align*}
&\nu^{\f{1}{3}}\big(\|\partial_x^2U\|_{L^2}^2+\|\partial_x(\partial_z-\kappa\partial_y)U\|_{L^2}^2\big)+\nu\big(\|\nabla\partial_x^2U\|_{L^2}^2+
\|\nabla\partial_x(\partial_z-\kappa\partial_y)U\|_{L^2}^2\big)
\\&\quad+\nu^{\f{1}{3}} \|\partial_x\nabla W\|_{L^2}^2+\nu\|\partial_x\Delta W\|_{L^2}^2+\nu^{\f{5}{3}} \|\partial_x\Delta U\|_{L^2}^2\leq C\nu^{-1}\big(\|\partial_xG_2\|_{L^2}^2+\|\nabla G_1\|_{L^2}^2\big).
\end{align*}

This proves the case of $\nu k^2\le 1$.

\section{Space-time estimates of the linearized equation}

In this section, we denote $\|F\|_{L^qL^p}=\|F\|_{L^q(0,T;L^p(I))}$ or $\|F\|_{L^q(0,+\infty;L^p(I))}$ when the function $F$ is extended to $t\ge T$.

\subsection{Space-time estimates with Navier-slip boundary condition}

In this subsection, we study the space-time estimate of the following
linearized  equation:
\begin{align}\label{eq:LNS-nav-full}
  \left\{\begin{aligned}
            &\partial_t\omega-\nu(\partial_y^2-\eta^2)\omega+iky\omega
            =-ikf_1-\partial_yf_2-i\ell f_3-f_4,\\
            &\omega|_{y=\pm1}=0,\quad \omega|_{t=0}=\omega_{in}.
         \end{aligned}
         \right.
\end{align}
 In what follows, we assume $a\in [0,\eps_1]$.

\begin{Proposition}\label{prop:TS-nav}
  Let $\omega$ be a solution of \eqref{eq:LNS-nav-full}  with
  $f_4(t,\pm1)=0$ and $\omega_{in}(\pm1)=0$. Then  it holds that
  \begin{align*}
    &\|e^{a\nu^{1/3}t}\omega\|^2_{L^\infty
    L^2}+\nu\|e^{a\nu^{1/3}t}\omega'\|^2_{L^2L^2} +(\nu \eta^2+(\nu
    k^2)^{1/3})\|e^{a\nu^{1/3}t}\omega\|^2_{L^2L^2} \\
    &\leq C\Big(\|\omega_{in}\|_{L^2}^2
    +\nu^{-1}\|e^{a\nu^{1/3}t}f_2\|^2_{L^2L^2}
    +(\eta|k|)^{-1}\|e^{a\nu^{1/3}t}\partial_yf_4\|^2_{L^2L^2}+\eta|k|^{-1}\|e^{a\nu^{1/3}t}f_4\|^2_{L^2L^2}\\
    &\qquad+\min((\nu \eta^2)^{-1} ,(\nu k^2)^{-1/3}) \|e^{a\nu^{1/3}t}(kf_1+\ell f_3)\|_{L^2L^2}^2\Big).
  \end{align*}
  Moreover, we have
 \begin{align*}
     &\|e^{a\nu^{1/3}t}\omega'\|^2_{L^\infty L^2}+\nu\|e^{a\nu^{1/3}t}\omega''\|^2_{L^2L^2}+\nu\eta^2\|e^{a\nu^{1/3}t}\omega'\|^2_{L^2L^2}\\
   &\leq C\|{\omega}'_{in}\|^2_{ L^2}+C\nu^{-\f23}|k|^{\f23}\Big(\|\omega_{in}\|_{L^2}^2
    +(\eta|k|)^{-1}\|e^{a\nu^{1/3}t}\partial_y{f}_4\|^2_{L^2L^2}+\eta|k|^{-1}\|e^{a\nu^{1/3}t}{f}_4\|^2_{L^2L^2}\Big)\\
   &\quad+C\nu^{-1}\Big(\|e^{a\nu^{1/3}t}(k{f}_1+\ell{f}_3)\|_{L^2L^2}^2
    +\nu^{-\f23}|k|^{\f23}\|e^{a\nu^{1/3}t}{f}_2\|^2_{L^2L^2}+\|e^{a\nu^{1/3}t}\partial_y{f}_2\|^2_{L^2L^2}\Big).
  \end{align*}
\end{Proposition}

Let $\widetilde{\omega}=e^{a\nu^{1/3}t}\omega$ and  $\widetilde{f}_j=e^{a\nu^{1/3}t}f_j$. Then $\widetilde{\omega} $ satisfies
\begin{align}\label{eqom}
  \left\{\begin{aligned}
            &\partial_t\widetilde{\omega}-\nu(\partial_y^2-\eta^2)\widetilde{\omega}+iky\widetilde{\omega}
            -a\nu^{1/3}\widetilde{\omega}
            =-ik\widetilde{f}_1-\partial_y\widetilde{f}_2-i\ell\widetilde{f}_3-\widetilde{f}_4,\\
            &\widetilde{\omega}|_{y=\pm1}=0,\quad \widetilde{\omega}|_{t=0}=\omega_{in}.
         \end{aligned}
         \right.
\end{align}
We decompose $\widetilde{\omega}$ as $\widetilde{\omega}=\omega_I+\omega_H$, where $\omega_H$ and
  $\omega_I$ solve
\begin{align}\label{eq:LNS-nav-inhom}
  \left\{\begin{aligned}
            &\partial_t\omega_{I}-\nu(\partial_y^2-\eta^2)\omega_{I}+iky\omega_I-a\nu^{1/3}\omega_I
            =-ik\widetilde{f}_1-\partial_y\widetilde{f}_2-i\ell\widetilde{f}_3-\widetilde{f}_4,\\
            &\omega_I|_{y=\pm1}=0,\quad \omega_I|_{t=0}=0,
         \end{aligned}
         \right.
\end{align}
and
\begin{align}\label{eq:LNS-nav-hom}
  \left\{\begin{aligned}
            &\partial_t\omega_{H}-\nu(\partial_y^2-\eta^2)\omega_{H}+iky\omega_H-a\nu^{1/3}\omega_H=0,\\
            &\omega_H|_{y=\pm1}=0,\quad \omega_H|_{t=0}=\omega_{in}.
         \end{aligned}
         \right.
\end{align}

\begin{Lemma}\label{lem:TS-nav-hom}
  Let $\omega_H$ be a solution of \eqref{eq:LNS-nav-hom} with
  $\omega_{in}(\pm1)=0$. Then  it holds that
  \begin{align*}
   &\|\omega_H\|^2_{L^\infty L^2}+\nu\|\omega'_H\|^2_{L^2L^2} +(\nu \eta^2+(\nu
   k^2)^{1/3})\|\omega_H\|^2_{L^2L^2}
    \leq C\|\omega_{in}\|^2_{L^2}.
  \end{align*}
\end{Lemma}

\begin{proof}
Let $L_{k,\ell}=\nu(k^2+\ell^2-\pa_y^2)+iky$ with $D(L_{k,\ell})={H^2\cap H_0^1}(-1,1)$. Then we have
\beno
\omega_H(t)=e^{-tL_{k,\ell}+ta\nu^{1/3}}\omega_{in}.
\eeno
Thanks to the fact that for $f\in D(L_{k,\ell})$
\begin{align}
\label{eq:accretive-Lk}{\rm Re}\langle L_{k,\ell}f,f \rangle=\nu (k^2+\ell^2)\|f\|_{L^2}^2+\nu\|f'\|_{L^2}^2,
\end{align}
the operator $L_{k,\ell}$ is accretive for any $k,\ell\in\Z$.

Thanks to Proposition \ref{prop:res-nav-y-uL2},  when $\nu\eta^3\le |k|$, there exists $c>0$ so that for any $k\in \Z$,
\begin{align*}
\Psi(L_{k,\ell})\geq c(\nu k^2)^{\f13}\geq 2a\nu^{\f13}
\end{align*}
for $0\leq a\leq \epsilon_1$ small enough, and when $\nu\eta^3\ge |k|$, we get by \eqref{eq:accretive-Lk}  that
\beno
\Psi(L_{k,\ell})\ge \nu\eta^2\ge \nu^\f13|k|^\f23\ge 2a\nu^{\f13}.
\eeno
Then it follows from Lemma \ref{lem:GP}  that
\begin{align}\nonumber
\|\omega_H(t)\|_{L^2}\leq e^{-t\Psi(L_{k,\ell})+{\pi}/{2}+ta\nu^{1/3}}\|\omega_{in}\|_{L^2}\leq Ce^{-\f {ct} 2(\nu k^2)^{\f13}}\|\omega_{in}\|_{L^2},
\end{align}
which gives
\begin{align}
\label{omHL2a}&\|\omega_H\|^2_{L^\infty L^2}+(\nu
   k^2)^{1/3}\|\omega_H\|^2_{L^2L^2}
    \leq C\|\omega_{in}\|^2_{L^2}.
\end{align}

The basic energy estimate yields that
\begin{align*}
  &\f12\f{d}{dt}\|\omega_H\|^2_{L^2}+\nu\|\omega'_H\|_{L^2}^2 +\nu
  \eta^2\|\omega_H\|_{L^2}^2=a\nu^{1/3}\|\omega_H\|_{L^2}^2,
\end{align*}
which gives
\begin{align}\label{omHL2b}
  &\|\omega_H\|^2_{L^\infty L^2}+\nu\|\omega'_H\|_{L^2L^2}^2 +\nu
  \eta^2\|\omega_H\|_{L^2L^2}^2\leq C\big(\|\omega_{in}\|^2_{L^2}+a\nu^{1/3}\|\omega_H\|_{L^2L^2}^2\big).
\end{align}

Now the result follows from \eqref{omHL2a}, \eqref{omHL2b} and the fact that $a\nu^{\f13}\leq (\nu k^2)^{\f13}$.
\end{proof}

\begin{Lemma}\label{lem:TS-nav-inhom}
 Let $\omega_I$ be a solution of  \eqref{eq:LNS-nav-inhom}  with $\widetilde{f}_4(t,\pm1)=0$. If $\nu \eta^3\leq |k|$, then  we have
   \begin{align*}
   &\|\omega_I\|^2_{L^\infty L^2}+\nu\|\omega'_I\|^2_{L^2L^2} +(\nu \eta^2+(\nu
    k^2)^{1/3})\|\omega_I\|^2_{L^2L^2} +\eta|k|\|u_I\|^2_{L^2L^2}\\
    &\leq C\Big((\nu k^2)^{-\f13} \|k\widetilde{f}_1+\ell\widetilde{f}_3\|_{L^2L^2}^2 +\nu^{-1}\|\widetilde{f}_2\|^2_{L^2L^2}\\
    &\qquad+(\eta|k|)^{-1}\|\partial_y\widetilde{f}_4\|^2_{L^2L^2}+\eta|k|^{-1}\|\widetilde{f}_4\|^2_{L^2L^2}\Big).
  \end{align*}
 Here $u_I=({\partial_y\varphi_I,-i\eta\varphi_I})$ and $\varphi_I=(\partial_y^2-\eta^2)^{-1}\omega_I.$
\end{Lemma}

\begin{proof}
We first extend the solution  $\omega_I$ to $t>T$ by solving \eqref{eq:LNS-nav-inhom}  with $\widetilde{f}_j=0$ for $t>T$. 
  Let
  \beno
 &&\hat{u}(\lambda,y)=\f{1}{2\pi}\int_{\mathbb{R}_+}u_I(t,y)e^{-i\lambda
  t}dt,\quad w(\lambda,y)=\f{1}{2\pi}\int_{\mathbb{R}_+}\omega_I(t,y)e^{-i\lambda
  t}dt,\\
   &&\hat{f}_j(\lambda,y)=\f{1}{2\pi}\int_{\mathbb{R}_+}\widetilde{f}_j(t,y)e^{-i\lambda
   t}dt,\quad j=1,2,3,4.
 \eeno
 Then we have
   \begin{align*}\left\{\begin{aligned}
     &-\nu(\partial_y^2-\eta^2)w+ik(y+\lambda/k)w-a\nu^{1/3}w=-ik\hat{f}_1-\partial_y\hat{f}_2
     -il\hat{f}_3-\hat{f}_4,\\
     &w|_{y=\pm1}=0.
   \end{aligned}\right.
   \end{align*}
 It follows from Proposition \ref{prop:res-nav-y-uL2} and Proposition \ref{prop:res-damping}  that
  \begin{align*}
    &\nu \|w'\|^2_{L^2}+ (\nu k^2)^{1/3}\|w\|_{L^2}^2
    +\eta|k|\|\hat{u}\|_{L^2}^2\\
    &\leq C\Big((\nu k^2)^{-1/3}\|k\hat{f}_1+\ell\hat{f}_3\|^2_{L^2}+\nu^{-1}\|\hat{f}_2\|_{L^2}^2
    +(\eta|k|)^{-1}(\|\partial_y\hat{f}_4\|^2_{L^2}+\eta^2\|\hat{f}_4\|^2_{L^2})\Big),
  \end{align*}
from which and  Plancherel's theorem,  we deduce that
  \begin{align}
    &\nu \|\omega_I'\|^2_{L_t^2L^2}+ (\nu k^2)^{1/3}\|\omega_I\|_{L_t^2L^2}^2
    +\eta|k|\|u_I\|_{L_t^2L^2}^2\nonumber\\ \nonumber
    &\leq C\big(\nu \|w'\|^2_{L_\lambda^2L^2}+ (\nu
    k^2)^{1/3}\|w\|_{L_\lambda^2L^2}^2
    +\eta|k|\|\widehat{u}\|_{L_\lambda^2L^2}^2\big)\\ \nonumber
    &\leq  C\Big((\nu k^2)^{-1/3}\|k\hat{f}_1+l\hat{f}_3\|^2_{L_\lambda^2L^2}+\nu^{-1}\|\hat{f}_2\|_{L_\lambda^2L^2}
    +(\eta|k|)^{-1}(\|\partial_y\hat{f}_4\|^2_{L_\lambda^2L^2}+\eta^2\|\hat{f}_4\|^2_{L_\lambda^2L^2})\Big)\\ \nonumber
   &\leq C\Big((\nu k^2)^{-\f13}\|k\widetilde{f}_1+l\widetilde{f}_3\|_{L_t^2L^2}^2+\nu^{-1}\|\widetilde{f}_2\|^2_{L_t^2L^2}
    +(\eta|k|)^{-1}(\|\partial_y\widetilde{f}_4\|^2_{L_t^2L^2}+\eta^{2}\|\widetilde{f}_4\|^2_{L_t^2L^2})\Big) .
  \end{align}
On the other hand, the basic energy estimate yields that
\begin{align*}
  &\f12\f{d}{dt}\|\omega_I\|^2_{L^2}+\nu\|\omega'_I\|_{L^2}^2 +\nu
  \eta^2\|\omega_I\|_{L^2}^2-a\nu^{1/3}\|\omega_I\|_{L^2}^2 \\
  &\leq \|k\widetilde{f}_1+\ell\widetilde{f}_3\|_{L^2}\|\omega_I\|_{L^2}+\|\widetilde{f}_2\|_{L^2}\|\partial_y\omega_I\|_{L^2}
  +\|\partial_y\widetilde{f}_4\|_{L^2}\|\partial_y\varphi_I\|_{L^2}
  +\eta^2\|\widetilde{f}_4\|_{L^2}\|\varphi_I\|_{L^2}\\
  &\leq
  \big((\nu k^2)^{-\f13}\|k\widetilde{f}_1+l\widetilde{f}_3\|_{L^2}^2+\nu^{-1}\|\widetilde{f}_2\|_{L^2}^2
  +(\eta|k|)^{-1}(\|\partial_y\widetilde{f}_4\|^2_{L^2}+\eta^{2}\|\widetilde{f}_4\|^2_{L^2})\big)\\&\quad+\big((\nu k^2)^{\f13}\|\omega_I\|_{L^2}^2+\nu\|\partial_y\omega_I\|_{L^2}^2+\eta|k|\|u_I\|_{L^2}^2\big)/4,
\end{align*}
here we used $ |\langle\widetilde{f}_4,\omega_I\rangle|=|\langle\widetilde{f}_4,(\partial_y^2-\eta^2)\varphi_I\rangle|
\leq\|\partial_y\widetilde{f}_4\|_{L^2}\|\partial_y\varphi_I\|_{L^2}
  +\eta^2\|\widetilde{f}_4\|_{L^2}\|\varphi_I\|_{L^2}$. This shows that
\begin{align*}
  &\|\omega_I\|_{L^\infty L^2}^2+\nu\|\omega'_I\|_{L^2L^2}^2 +\nu
  \eta^2\|\omega_I\|^2_{L^2L^2}\\
  &\leq C\Big((\nu k^2)^{-\f13}\|k\widetilde{f}_1+\ell\widetilde{f}_3\|_{L^2L^2}^2+\nu^{-1}\|\widetilde{f}_2\|_{L^2L^2}^2
  +(\eta|k|)^{-1}(\|\partial_y\widetilde{f}_4\|^2_{L^2L^2}+\eta^{2}\|\widetilde{f}_4\|^2_{L^2L^2})\Big)\\&\quad+C((\nu k^2)^{\f13}+a\nu^{\f13})\|\omega_I\|_{L^2L^2}^2+C\eta|k|\|u_I\|_{L^2L^2}^2.
\end{align*}

Summing up and noting that $a\nu^{\f13}\leq (\nu k^2)^{\f13}$, we  conclude the proof of  the lemma.
\end{proof}\smallskip

Now we are in a position to prove Proposition \ref{prop:TS-nav}.

\begin{proof}
We first consider the case of $\nu \eta^3\leq |k|$. In this case, we have $\nu \eta^2\le (\nu
    k^2)^{1/3}$. It follows from Lemma \ref{lem:TS-nav-hom} and lemma \ref{lem:TS-nav-inhom}  that
    \begin{align*}
   &\|\widetilde{\omega}\|^2_{L^\infty
    L^2}+\nu\|\widetilde{\omega}'\|^2_{L^2L^2} +(\nu \eta^2+(\nu
    k^2)^{1/3})\|\widetilde{\omega}\|^2_{L^2L^2}\\
    &\leq C\Big(\|\omega_{in}\|_{L^2}^2
    +\nu^{-1}\|\widetilde{f}_2\|^2_{L^2L^2}
    +(\eta|k|)^{-1}\|\partial_y\widetilde{f}_4\|^2_{L^2L^2}+\eta|k|^{-1}\|\widetilde{f}_4\|^2_{L^2L^2}\\&
    \qquad+(\nu k^2)^{-1/3}\|k\widetilde{f}_1+\ell\widetilde{f}_3\|_{L^2L^2}^2\Big).
  \end{align*}

  For the case of $\nu \eta^3\geq |k|$, we have $ \nu \eta^2\geq(\nu
    k^2)^{1/3},\ (\nu \eta^2)^{-1}\leq\eta|k|^{-1}.$
  The basic energy estimate yields that
  \begin{align*}
  &\f12\f{d}{dt}\|\widetilde{\omega}\|^2_{L^2}+\nu\|\widetilde{\omega}'\|_{L^2}^2 +(\nu \eta^2-a\nu^{1/3})\|\widetilde{\omega}\|_{L^2}^2 \\
  &\leq \|ik\widetilde{f}_1+i\ell\widetilde{f}_3+\widetilde{f}_4\|_{L^2}\|\widetilde{\omega}\|_{L^2}
  +\|\widetilde{f}_2\|_{L^2}\|\partial_y\widetilde{\omega}\|_{L^2}\\
  &\leq
  \big((\nu \eta^2)^{-1}\|ik\widetilde{f}_1+i\ell \widetilde{f}_3+\widetilde{f}_4\|_{L^2}^2+\nu^{-1}\|\widetilde{f}_2\|_{L^2}^2\big)\\&
  \quad+\big(\nu \eta^2\|\widetilde{\omega}\|_{L^2}^2+\nu\|\partial_y\widetilde{\omega}\|_{L^2}^2\big)/4,
\end{align*}
which shows that
\begin{align*}
  &\|\widetilde{\omega}\|_{L^\infty L^2}^2+\nu\|\widetilde{\omega}'\|_{L^2L^2}^2 +\nu
  \eta^2\|\widetilde{\omega}\|^2_{L^2L^2}\\
  &\leq C\big(\|{\omega}'_{in}\|^2_{ L^2}+(\nu \eta^2)^{-1}\|ik\widetilde{f}_1+i\ell\widetilde{f}_3+\widetilde{f}_4\|_{L^2L^2}^2+\nu^{-1}\|\widetilde{f}_2\|_{L^2L^2}^2\big)\\
  &\leq C\big(\|{\omega}'_{in}\|^2_{ L^2}+(\nu \eta^2)^{-1}\|ik\widetilde{f}_1+i\ell\widetilde{f}_3\|_{L^2L^2}^2+
  \eta|k|^{-1}\|\widetilde{f}_4\|_{L^2L^2}^2+\nu^{-1}\|\widetilde{f}_2\|_{L^2L^2}^2\big).
\end{align*}
This shows  the first inequality of the proposition.

It remains to prove the second inequality. The basic energy estimate yields that
\begin{align*}
   &\f{1}{2}\f{d}{dt}\|\widetilde{\omega}'\|^2_{L^2}+\nu\|\widetilde{\omega}''\|^2_{L^2} +(\nu \eta^2-a\nu^{\f13})\|\widetilde{\omega}'\|^2_{L^2}+\mathbf{Re}\Big(ik\int_{-1}^{1}\widetilde{\omega}\bar{\widetilde{\omega}}'dy\Big)\\
   &=\mathbf{Re}\big(\langle ik\widetilde{f}_1+\partial_y\widetilde{f}_2+i\ell\widetilde{f}_3+\widetilde{f}_4,
   \partial_y^2\widetilde{\omega}\rangle\big) =\mathbf{Re}\big(\langle ik\widetilde{f}_1+\partial_y\widetilde{f}_2+i\ell\widetilde{f}_3,\partial_y^2\widetilde{\omega}\rangle-\langle \partial_y\widetilde{f}_4,\partial_y\widetilde{\omega}\rangle\big)\\
   &\leq \nu^{-1}\|ik\widetilde{f}_1+\partial_y\widetilde{f}_2+i\ell\widetilde{f}_3\|^2_{L^2}+\nu\|\widetilde{\omega}''\|^2_{L^2}/4
   +\|\partial_y\widetilde{f}_4\|_{L^2}\|\partial_y\widetilde{\omega}\|^2_{L^2},
\end{align*}
which implies that
\begin{align*}
   & \f{d}{dt}\|\widetilde{\omega}'\|^2_{L^2}+\nu\|\widetilde{\omega}''\|^2_{L^2} +2\nu \eta^2\|\widetilde{\omega}'\|_{L^2}^2\\
   &\leq C\big(|k|\|\widetilde{\omega}\|_{L^2}\|\widetilde{\omega}'\|_{L^2}+a\nu^{\f13}\|\partial_y\widetilde{\omega}\|^2_{L^2} +\nu^{-1}\|ik\widetilde{f}_1+\partial_y\widetilde{f}_2+i\ell\widetilde{f}_3\|^2_{L^2}\big)+
   2\|\partial_y\widetilde{f}_4\|_{L^2}\|\partial_y\widetilde{\omega}\|^2_{L^2}\\
   &\leq C\big(\nu^{-\f13}|k|^{\f43}\|\widetilde{\omega}\|_{L^2}^2+(\nu k^2)^{\f13}\|\widetilde{\omega}'\|_{L^2}^2+a\nu^{\f13}\|\widetilde{\omega}'\|^2_{L^2}+
   \nu^{-1}\|ik\widetilde{f}_1+\partial_y\widetilde{f}_2+i\ell\widetilde{f}_3\|^2_{L^2}\big)\\&\quad+(\nu \eta^2)^{-\f12}(\nu k^2)^{-\f16}\|\partial_y\widetilde{f}_4\|_{L^2}^2+(\nu \eta^2+(\nu k^2)^{\f13})\|\partial_y\widetilde{\omega}\|^2_{L^2}/2.
\end{align*}
and hence,
\begin{align*}
   & \f{d}{dt}\|\widetilde{\omega}'\|^2_{L^2}+\nu\|\widetilde{\omega}''\|^2_{L^2} +\nu \eta^2\|\widetilde{\omega}'\|_{L^2}^2\\
   &\leq C\Big(\nu^{-\f13}|k|^{\f43}\|\widetilde{\omega}\|_{L^2}^2+(\nu k^2)^{\f13}\|\widetilde{\omega}'\|_{L^2}^2+
   \nu^{-1}\|ik\widetilde{f}_1+\partial_y\widetilde{f}_2+i\ell\widetilde{f}_3\|^2_{L^2}+
   \nu^{-\f23}\eta^{-1}|k|^{-\f13}\|\partial_y\widetilde{f}_4\|_{L^2}^2\Big),
\end{align*}
from which and the first inequality of the proposition, we deduce that
\begin{align*}
   &\|\widetilde{\omega}'\|^2_{L^\infty L^2}+\nu\|\widetilde{\omega}''\|^2_{L^2L^2} +\nu \eta^2\|\widetilde{\omega}'\|^2_{L^2L^2}\\
   &\leq C \Big(\|{\omega}'_{in}\|^2_{ L^2}+\nu^{-\f13}|k|^{\f43}\|\widetilde{\omega}\|_{L^2L^2}^2+(\nu k^2)^{\f13}\|\widetilde{\omega}'\|_{L^2L^2}^2+\nu^{-\f23}\eta^{-1}|k|^{-\f13}\|\partial_y\widetilde{f}_4\|_{L^2L^2}^2\\&
   \qquad+\nu^{-1}\|ik\widetilde{f}_1+\partial_y\widetilde{f}_2+i\ell\widetilde{f}_3\|^2_{L^2L^2}\Big)\\
   &\leq C \|{\omega}'_{in}\|^2_{ L^2}+C\nu^{-\f23}|k|^{\f23}\big(\|\widetilde{\omega}\|^2_{Y^0_{k,\ell}}+|\eta k|^{-1}\|\partial_y\widetilde{f}_4\|_{L^2}^2\big)+C\nu^{-1}\|ik\widetilde{f}_1+\partial_y\widetilde{f}_2+i\ell\widetilde{f}_3\|^2_{L^2L^2}
   \\
   &\leq C \|{\omega}'_{in}\|^2_{ L^2}+C\nu^{-\f23}|k|^{\f23}\Big(\|\omega_{in}\|_{L^2}^2
    +\nu^{-1}\|\widetilde{f}_2\|^2_{L^2L^2}
    +(\eta|k|)^{-1}\|\partial_y\widetilde{f}_4\|^2_{L^2L^2}+\eta|k|^{-1}\|\widetilde{f}_4\|^2_{L^2L^2}\\&
    \qquad+(\nu k^2)^{-1/3}\|k\widetilde{f}_1+l\widetilde{f}_3\|_{L^2L^2}^2\Big)
    +C\nu^{-1}\|ik\widetilde{f}_1+\partial_y\widetilde{f}_2+i\ell\widetilde{f}_3\|^2_{L^2L^2}\\
   &\leq C \|{\omega}'_{in}\|^2_{ L^2}+C\nu^{-\f23}|k|^{\f23}\big(\|\omega_{in}\|_{L^2}^2
    +(\eta|k|)^{-1}\|\partial_y\widetilde{f}_4\|^2_{L^2L^2}+\eta|k|^{-1}\|\widetilde{f}_4\|^2_{L^2L^2}\big)\\&
    \quad+C\nu^{-1}\big(\|k\widetilde{f}_1+\ell\widetilde{f}_3\|_{L^2L^2}^2
    +\nu^{-\f23}|k|^{\f23}\|\widetilde{f}_2\|^2_{L^2L^2}+\|\partial_y\widetilde{f}_2\|^2_{L^2L^2}\big).
\end{align*}
This proves the second inequality of the proposition.(See section 13 for the definition of the norm $\|\cdot\|_{Y^0_{k,\ell}}$).
\end{proof}

\subsection{Space-time estimates with non-slip boundary condition}

In this subsection, we study the space-time estimate of the following
linearized  equation with non-slip boundary condition:
\begin{align}\label{eq:LNS-non}
  \left\{\begin{aligned}
   & \partial_t\omega-\nu(\partial_y^2-\eta^2)\omega+iky\omega= F,\\
   &(\partial_y^2-\eta^2)\varphi=\omega,\ \partial_y\varphi|_{y=\pm1}=\varphi|_{y=\pm1}=0,\\
   &\omega|_{t=0}=\omega_{in}.\end{aligned}\right.
\end{align}
Here $\eta\ge k$ and $0<\nu\leq\nu_0$, $0\leq a\leq \epsilon_1\leq 1/8$.

\begin{Proposition}\label{prop:TS-non}
Let $\omega$ solve \eqref{eq:LNS-non} with $\partial_y\varphi_{in}|_{y=\pm1}=0$ and $F=ikf_1+\partial_yf_2+i\ell f_3$. Then it holds that
  \begin{align*}
     &|k\eta|^{\f12}\|e^{a\nu^{\f13}t}(\partial_y,\eta)\varphi\|_{L^2L^2}+\nu^{\f34}\|e^{a\nu^{\f13}t}\partial_y\omega\|_{L^2L^2}+
     \nu^{\f12}\eta\|e^{a\nu^{\f13}t}\omega\|_{L^2L^2}+\eta
     \|e^{a\nu^{\f13}t}(\partial_y,\eta)\varphi\|_{L^{\infty}L^2}\\&\quad+\nu^{\f14}\|e^{a\nu^{\f13}t}\omega\|_{L^{\infty}L^2}\leq C\nu^{-\f12}\|e^{a\nu^{\f13}t}(f_1,f_2,f_3)\|_{L^2L^2} +C\big(\eta^{-1}\|\partial_y\omega_{in}\|_{L^2}+\|\omega_{in}\|_{L^2}\big).
  \end{align*}
\end{Proposition}

The proof is based on the following lemmas.

\begin{Lemma}\label{lem:TS-non-s1}
Let $\omega$ solve \eqref{eq:LNS-non} with $\partial_y\varphi_{in}|_{y=\pm1}=0$ and  $F=f_1+\partial_yf_2$.
If $\nu\eta^3\leq |k|$, then we have
  \begin{align*}
     &|k\eta|^{\f12}\|e^{a\nu^{\f13}t}(\partial_y,\eta)\varphi\|_{L^2L^2}+\nu^{\f14}|k|^{\f12}\|e^{a\nu^{\f13}t}\omega\|_{L^2L^2}\\ &\leq C\nu^{-\f12}\big(\|e^{a\nu^{\f13}t}f_2\|_{L^2L^2}+|\nu/k|^{\f13}\|e^{a\nu^{\f13}t}f_1\|_{L^2L^2}\big) +C\big(\eta^{-1}\|\partial_y\omega_{in}\|_{L^2}+\|\omega_{in}\|_{L^2}\big).
  \end{align*}
\end{Lemma}

\begin{proof}
We first  extend the solution $\omega$ to $t>T$ by solving \eqref{eq:LNS-non} with $F=0$ for $t>T$, and extend the solution $\omega $ to $t<0$ by
\beno
\omega(t,y)=e^{-itky+(\nu k^2)^{1/3}t}\omega_{in}(y)\quad t<0,
\eeno
i.e., $\partial_t\omega+iky\omega-(\nu k^2)^{1/3}\omega= 0 $ for $t<0$, and extend $\varphi $ to $t<0$ by
\beno
(\partial_y^2-\eta^2)\varphi=\omega,\ \varphi|_{y=\pm1}=0\quad  \text{for}\,\, t<0,
\eeno
and extend $F,f_1,f_2$ to $t<0$ by
\begin{align*}F=-\nu(\partial_y^2-\eta^2)\omega+(\nu k^2)^{1/3}\omega,\ f_1=(\nu\eta^2+ (\nu k^2)^{1/3})\omega,\ f_2=-\nu\partial_y\omega\quad \text{for}\ t<0.\end{align*}
Then it holds that for $t\in \R$, $(\partial_y^2-\eta^2)\varphi=\omega,\ \varphi|_{y=\pm1}=0$ and
\begin{align*}
\partial_t\omega-\nu(\partial_y^2-\eta^2)\omega+iky\omega= F=f_1+\partial_yf_2.
\end{align*}
We denote
\beno
&&\hat{\varphi}(\lambda,y)=\f{1}{2\pi}\int_{\mathbb{R}}\varphi(t,y)e^{a\nu^{1/3}t-i\lambda t}dt,\quad  w(\lambda,y)=\f{1}{2\pi}\int_{\mathbb{R}}\omega(t,y)e^{a\nu^{1/3}t-i\lambda t}dt,\\
&&F_j(\lambda,y)=\f{1}{2\pi}\int_{\mathbb{R}}{f}_j(t,y)e^{a\nu^{1/3}t-i\lambda t}dt\quad j=1,2.
\eeno
Then we have
\begin{align*}
  \left\{\begin{aligned}
            &-\nu(\partial_y^2-\eta^2)w+ik(y+\lambda/k)w-a\nu^{1/3}w
            =F_1+\partial_yF_2,\\
            &(\partial_y^2-\eta^2)\hat{\varphi}=w,\quad \hat{\varphi}|_{y=\pm1}=0.
         \end{aligned}
         \right.
\end{align*}
It follows from Proposition \ref{prop:res-NB-y-w} and  Proposition \ref{prop:res-NB-y-u} that
\begin{align*}
  &\eta^{\f12}\|(\partial_y,\eta)\hat{\varphi}\|_{L^2_y}\leq C\big(\nu^{-\f12}|k|^{-\f12}\|F_2\|_{L_y^2}+\nu^{-\f16}|k|^{-\f56}\|F_1\|_{L_y^2}+|\partial_y\hat{\varphi}(\lambda,1)|+
  |\partial_y\hat{\varphi}(\lambda,-1)|\big),
  \\&\|w\|_{L_y^2}\leq C\nu^{-\f14}\big((|k(1+\lambda/k)|+1)^{\f14}|\partial_y\hat{\varphi}(\lambda,1)|+(|k(1-\lambda/k)|+1)^{\f14}|\partial_y\hat{\varphi}(\lambda,-1)|\big)
  \\&\qquad+C\big((\nu k^2)^{-\f{5}{12}}\|F_1\|_{L_y^2}+\nu^{-\f34} |k|^{-\f{1}{2}}\|F_2\|_{L_y^2}\big).
\end{align*}
Thus, we have
\begin{align*}
  \eta^{\f12}\|(\partial_y,\eta)\hat{\varphi}\|_{L^2_y}+\nu^{\f14}\|w\|_{L_y^2}\leq& C(\nu^{-\f12}|k|^{-\f12}\|F_2\|_{L_y^2}+\nu^{-\f16}|k|^{-\f56}\|F_1\|_{L_y^2})\\
  &+C\big((|k+\lambda|+1)^{\f14}|\partial_y\hat{\varphi}(\lambda,1)|+(|k-\lambda|+1)^{\f14}|\partial_y\hat{\varphi}(\lambda,-1)|\big).
\end{align*}
from which and  Plancherel's theorem, we deduce that
\begin{align*}
  &\eta^{\f12}\|e^{a\nu^{1/3}t}(\partial_y,\eta)\varphi(t,y)\|_{L^2_{t\in\R}L^2_y} +\nu^{\f14}\|e^{a\nu^{1/3}t}{\omega}(t,y)\|_{L^2_{t\in\R}L^2_y}\\
  &\leq \eta^{\f12}\|(\partial_y,\eta)\hat{\varphi}(\lambda,y)\|_{L^2_{\lambda}L^2_y} +\nu^{\f14}\|w(\lambda,y)\|_{L^2_{\lambda}L^2_y} \\
  &\leq C(\nu^{-\f12}|k|^{-\f12}\|F_2\|_{L^2_{\lambda}L_y^2}+\nu^{-\f16}|k|^{-\f56}\|F_1\|_{L^2_{\lambda}L_y^2})\\
  &\quad+C\big(\|(|k+\lambda|+1)^{\f14}\partial_y\hat{\varphi}(\lambda,1)\|_{L^2_{\lambda}}
  +\|(|k-\lambda|+1)^{\f14}\partial_y\hat{\varphi}(\lambda,-1)\|_{L^2_{\lambda}}\big)\\
  &\leq C\big(\nu^{-\f12}|k|^{-\f12}\|e^{a\nu^{1/3}t}f_2\|_{L^2_{t\in\R}L_y^2}+\nu^{-\f16}|k|^{-\f56}\|e^{a\nu^{1/3}t}f_1\|_{L^2_{t\in\R}L_y^2}\big)\\
  &\quad+C\big(\|(|k+\lambda|+1)c_1\|_{L^2(\R)}
  +\|(|k-\lambda|+1)c_2\|_{L^2(\R)}\big),
\end{align*}
here we denote
\begin{align*}
c_1(\lambda)&=\partial_y\hat{\varphi}(\lambda,1)= \int_{-1}^{1}\f{\sinh(\eta(1+y))}{\sinh(2\eta)}w(\lambda,y)dy,\\
c_2(\lambda)&=-\partial_y\hat{\varphi}(\lambda,-1) =\int_{-1}^{1}\f{\sinh(\eta(1-y))}{\sinh(2\eta)}w(\lambda,y)dy.
\end{align*}

Since $f_1=(\nu\eta^2+ (\nu k^2)^{1/3})\omega,\ f_2=-\nu\partial_y\omega$ for $t<0,$ $\nu\eta^3\leq|k|$, $\nu\eta^2\leq (\nu k^2)^{1/3}$, we have
\begin{align*}
  &\nu^{-\f12}|k|^{-\f12}\|e^{a\nu^{1/3}t}f_2\|_{L^2_t(-\infty,0)}+\nu^{-\f16}|k|^{-\f56}\|e^{a\nu^{1/3}t}f_1\|_{L^2_t(-\infty,0)} \\ &\leq \nu^{\f12}|k|^{-\f12}\|\partial_y\omega\|_{L^2_t(-\infty,0)}+C\nu^{\f16}|k|^{-\f16}\|\omega\|_{L^2_t(-\infty,0)}\\
 & \leq  C\nu^{\f12}|k|^{-\f12}\left\|((-ikt,|k/\nu|^{\f13})\omega_{in}(y),\partial_y\omega_{in}(y))e^{-ikty+\nu^{\f13}|k|^{\f23}t}\right\|_{L^2_t(-\infty,0])} \\
 & \leq C\nu^{\f12}|k|^{-\f12}\big(\nu^{-\f12}|\omega_{in}| +\nu^{-\f16}|k|^{-\f13}|\partial_y\omega_{in}|\big)\\
  &\leq C|k|^{-\f12}\big(|\omega_{in}|+ |\nu/k|^{\f13}|\partial_y\omega_{in}|\big)\leq C|k|^{-\f12}\big(|\omega_{in}|+\eta^{-1}|\partial_y\omega_{in}|\big),
\end{align*}
which shows that
\beno
 &\nu^{-\f12}|k|^{-\f12}\|e^{a\nu^{1/3}t}f_2\|_{L^2_{t<0}L_y^2}+\nu^{-\f16}|k|^{-\f56}\|e^{a\nu^{1/3}t}f_1\|_{L^2_{t<0}L_y^2}\leq C|k|^{-\f12}\big(\|\omega_{in}\|_{L^2}+\eta^{-1}\|\partial_y\omega_{in}\|_{L^2}\big).
\eeno
and then,
\begin{align*}
 &\nu^{-\f12}|k|^{-\f12}\|e^{a\nu^{1/3}t}f_2\|_{L^2_{t\in\R}L_y^2}+\nu^{-\f16}|k|^{-\f56}\|e^{a\nu^{1/3}t}f_1\|_{L^2_{t\in\R}L_y^2}
  \leq C|k|^{-\f12}\big(\|\omega_{in}\|_{L^2}+\eta^{-1}\|\partial_y\omega_{in}\|_{L^2}\big)
  \\&\qquad+\nu^{-\f12}|k|^{-\f12}\|e^{a\nu^{1/3}t}f_2\|_{L^2L^2}+\nu^{-\f16}|k|^{-\f56}\|e^{a\nu^{1/3}t}f_1\|_{L^2L^2}.
\end{align*}

It remains to estimate $c_1(\lambda)$ and $c_2(\lambda).$ Let
\begin{align*}
   a_1(t)&=e^{a\nu^{\f13}t}\int_{-1}^{1}\f{\sinh(\eta(1+y))}{\sinh(2\eta)}\omega(t,y)dy,\\
   \widetilde{a}_1(t)&=e^{(a\nu^{\f13}+\nu^{\f13}|k|^{\f23})t}\int_{-1}^{1}\f{\sinh(\eta(1+y))}{\sinh(2\eta)}\omega_{in}(y) e^{ikt(1-y)}dy.\end{align*}
Then we have
\beno
c_1(\lambda)=\f{1}{2\pi}\int_{\mathbb{R}}a_1(t)e^{-i\lambda t}dt,\quad  a_1(t)=e^{ikt}\widetilde{a}_1(t)\,\, \text{for}\,\, t<0,
\eeno
and  due to $\partial_y\varphi|_{y=\pm1}=\varphi|_{y=\pm1}=0$ for $t>0$, we have $a_1(t)=0$ for $t>0$.
By Plancherel's theorem, we get
\begin{align*}
  \|a_1(t)\|_{L^2(-\infty,0)}^2 &\leq  \|e^{-(a\nu^{\f13}+\nu^{\f13}|k|^{\f23})t}\widetilde{a}_1(t)\|_{L^2(\R)}^2 =\f{2\pi}{|k|}\int_{-1}^{1}\left|\f{\sinh(\eta(1+y))}{\sinh(2\eta)}\omega_{in}(y)\right|^2dy\\
  &\leq C|k|^{-1}\|\omega_{in}\|_{L^2}^2.
\end{align*}
Let  $b=(a\nu^{\f13}+\nu^{\f13}|k|^{\f23})\leq 2\nu^{\f13}|k|^{\f23}\leq 2$. For $t\leq 0$, we have
\begin{align*}
   e^{-(a\nu^{\f13}+\nu^{\f13}|k|^{\f23})t}\big(\partial_t\widetilde{a}_1(t)-b\widetilde{a}_{1}(t)\big)&= \int_{-1}^{1}\f{\sinh(\eta(1+y))}{\sinh(2\eta)}\omega_{in}(y) ik(1-y) e^{ikt(1-y)}dy,
\end{align*}
hence,
\begin{align*}
  \|\partial_t\widetilde{a}_1(t)-b\widetilde{a}_1(t)\|^2_{L^2(-\infty,0]}&\leq  \f{2\pi}{|k|}\int_{-1}^{1}\left|\f{\sinh(\eta(1+y))}{\sinh(2\eta)}ik(1-y)\omega_{in}(y)\right|^2dy \\&\leq C|k|\eta^{-2}\|\omega_{in}\|_{L^2}^2\leq C|k|^{-1}\|\omega_{in}\|_{L^2}^2.
\end{align*}
Then we infer that
\begin{align*}
  \|e^{-ikt}a_1(t)\|_{H^1(-\infty,0)}\leq & C\big(\|\partial_t\widetilde{a}_1(t)-b\widetilde{a}_1(t)\|_{L^2(-\infty,0)}+ (1+b)\|a_1(t)\|_{L^2(-\infty,0)}\big)\\
  \leq &C|k|^{-\f12}\|\omega_{in}\|_{L^2}.
\end{align*}
As $\langle \f{\sinh(\eta(1+y))}{\sinh(2\eta)}, \omega_{in}\rangle=0$, we have $\widetilde{a}_1(0)=0$, using also $a_1(t)=e^{ikt}\widetilde{a}_1(t)$ for $t<0$, $a_1(t)=0$ for $t>0$, we know that $a_1$ is continuous at $t=0$,  and that $e^{-ikt}a_1(t)\in H^1(\mathbb{R})$ and then
\begin{align*}
   \|e^{-ikt}a_1(t)\|_{H^1(\mathbb{R})}&=\|e^{-ikt}a_1(t)\|_{H^1(-\infty,0)}\leq C|k|^{-\f12}\|\omega_{in}\|_{L^2}.
\end{align*}
Recalling that $c_1(\lambda)=\f{1}{2\pi}\int_{\mathbb{R}}a_1(t)e^{-i\lambda t}dt$, we have
\begin{align*}
  \|(1+|\lambda+k|)c_1(\lambda)\|_{L^2} &=\|(1+|\lambda|)c_1(\lambda-k)\|_{L^2}= C\|e^{-ikt}a_{1}(t)\|_{H^1(\mathbb{R})}\leq C|k|^{-\f12}\|\omega_{in}\|_{L^2}.
\end{align*}
Similarly, we have
\beno
\|(1+|\lambda-k|)c_2(\lambda)\|_{L^2}\leq C|k|^{-\f12}\|\omega_{in}\|_{L^2}.
\eeno

 Summing up, we conclude that
 \begin{align*}
 &|k\eta|^{\f12}\|e^{a\nu^{\f13}t}(\partial_y,\eta)\varphi\|_{L^2L^2}+\nu^{\f14}|k|^{\f12}\|e^{a\nu^{\f13}t}\omega\|_{L^2L^2}\\
 &\leq |k|^{\f12}\Big(\eta^{\f12}\|e^{a\nu^{1/3}t}(\partial_y,\eta)\varphi(t,y)\|_{L^2_{t\in\R}L^2_y} +\nu^{\f14}\|e^{a\nu^{1/3}t}{\omega}(t,y)\|_{L^2_{t\in\R}L^2_y}\Big)\\ &\le  C\big(\|\omega_{in}\|_{L^2}+\eta^{-1}\|\partial_y\omega_{in}\|_{L^2}\big)
 +C\nu^{-\f12}\|e^{a\nu^{1/3}t}f_2\|_{L^2L^2}+C\nu^{-\f16}|k|^{-\f13}\|e^{a\nu^{1/3}t}f_1\|_{L^2L^2}.
\end{align*}

This completes the proof of the lemma.
\end{proof}

\begin{Lemma}\label{lem:TS-non-s2}
Let $\omega$ solve \eqref{eq:LNS-non} with $\partial_y\varphi_{in}|_{y=\pm1}=0$ and $F=ikf_1+\partial_yf_2+i\ell f_3$. Then it holds that
  \begin{align*}
  &\nu\eta^2
  \|e^{a\nu^{1/3}t}\omega\|^2_{L^2 L^2}+\eta^2\|e^{a\nu^{1/3}t}(\partial_y,\eta)\varphi\|^2_{L^\infty L^2}\\ &\leq \big(|k\eta|+2a\nu^\f13\eta^2\big)\|e^{a\nu^{\f13}t}(\partial_y,\eta)\varphi\|_{L^2L^2}^2+ C\nu^{-1}\|e^{a\nu^{\f13}t}(f_1,f_2,f_3)\|_{L^2L^2}^2 +C\|\omega_{in}\|_{L^2}^2,
 \end{align*}
 and
 \begin{align*}
 &\|e^{a\nu^{\f13}t}\omega\|^2_{L^{\infty}L^2}+{\nu}\|e^{a\nu^{\f13}t}(\partial_y,\eta)\omega\|_{L^2L^2}^2 \leq C\big(\|\omega_{in}\|_{L^2}^2+\nu^{-1}\|e^{a\nu^{\f13}t}(f_1,f_2,f_3)\|_{L^2L^2}^2\big)\\&
\qquad+C\nu^{-\f12}|k\eta|^{\f32}\|e^{a\nu^{\f13}t}\varphi\|_{L^2L^2}^2+C(| k/\eta|+\nu\eta^2)\|e^{a\nu^{\f13}t}\omega\|_{L^2L^{2}}^2.
  \end{align*}
\end{Lemma}

\begin{proof}
Taking $L^2$ inner product between \eqref{eq:LNS-non} and $\varphi$, we get
\begin{align*}
  &\big\langle(\partial_t-\nu(\partial_y^2-\eta^2)+iky)\omega,-\varphi\big\rangle =\big\langle ikf_1+\partial_yf_2+i\ell f_3 ,-\varphi\big\rangle,
\end{align*}
which gives
\begin{align*}
&\big\langle\partial_t(\partial_y,\eta)\varphi,(\partial_y,\eta)\varphi\big\rangle+\nu\|\omega\|_{L^2}^2 +ik\int_{-1}^{1}{\varphi}'\overline{\varphi} dy+ik\int_{-1}^{1}y|\varphi'|^2dy+ ik\eta^2\int_{-1}^{1}y|\varphi|^2dy\\
&=\big\langle ikf_1+\partial_yf_2+i\ell f_3,-\varphi\big\rangle=-\big\langle ikf_1+i\ell f_3,\varphi\big\rangle+\big\langle f_2,\partial_y\varphi\big\rangle.
\end{align*}
Taking the real part, we get
\begin{align*}
  &\f12\f{d}{dt}\|(\partial_y,\eta)\varphi\|^2_{L^2}+\nu\|\omega\|^2_{L^2}\\
&\leq |k|\int_{-1}^{1}|\varphi'{\varphi}|dy+\f{1}{\nu \eta^2}\|(f_1,f_3)\|_{L^2}^2+\f{\nu\eta^4}4\|\varphi\|_{L^2}^2+\f{1}{\nu \eta^2}\|f_2\|_{L^2}^2
+\f{\nu\eta^2}{4} \|\varphi'\|_{L^2}^2\\
  &\leq \f{|k|}{2\eta}\big(\|\varphi'\|_{L^2}^2+\eta^2\|\varphi\|_{L^2}^2\big)+\f{\nu\eta^2}{4}(\|\varphi'\|_{L^2}^2+\eta^2\|\varphi\|_{L^2}^2)
  +\f{1}{\nu \eta^2}\|(f_1,f_2,f_3)\|_{L^2}^2\\
&\leq\f{|k|}{2\eta}\|(\partial_y,\eta)\varphi\|^2_{L^2}+\f{\nu}{4}\|\omega\|_{L^2}^2+\f{1}{\nu \eta^2}\|(f_1,f_2,f_3)\|_{L^2}^2.
\end{align*}
This shows that
\begin{align}\label{dtphi}
&\f{d}{dt}\|e^{a\nu^{1/3}t}(\partial_y,\eta)\varphi\|^2_{L^2}+\f{3\nu}{2}\|e^{a\nu^{1/3}t}\omega\|^2_{L^2}\\ \nonumber &\leq
\big(|k/\eta|+2a\nu^\f13\big)\|e^{a\nu^{1/3}t}(\partial_y,\eta)\varphi\|^2_{L^2}+2/(\nu \eta^2)\|e^{a\nu^{1/3}t}(f_1,f_2,f_3)\|_{L^2}^2,
\end{align}
which gives
\begin{align*}
  &\|e^{a\nu^{1/3}t}(\partial_y,\eta)\varphi(t)\|^2_{L^2}+\nu\int_{0}^{t}\|e^{a\nu^{1/3}s}\omega(s)\|_{L^2}^2ds\\ &\leq
\big(|k/\eta|+2a\nu^\f13)\|e^{a\nu^{1/3}t}(\partial_y,\eta)\varphi\|^2_{L^2L^2}+2/(\nu \eta^2)\|e^{a\nu^{1/3}t}(f_1,f_2,f_3)\|_{L^2L^2}^2+\|(\partial_y,\eta)\varphi(0)\|_{L^2}^2.
\end{align*}
This gives the first inequality by noting that   $\eta^2\|(\partial_y,\eta)\varphi(0)\|_{L^2}^2\leq\|\omega(0)\|^2_{L^2}=\|\omega_{in}\|_{L^2}^2$.\smallskip

Taking $L^2$ inner product between \eqref{eq:LNS-non} and $\omega$, we get
\begin{align*}
  &\big\langle(\partial_t-\nu(\partial_y^2-\eta^2)+iky)\omega,\omega\big\rangle =\big\langle ikf_1+\partial_yf_2+i\ell f_3 ,\omega\big\rangle,
\end{align*}
which gives
\begin{align*}
& \big\langle\partial_t\omega,\omega\big\rangle+\nu\|(\partial_y,\eta)\omega\|_{L^2}^2 +ik\int_{-1}^{1}y|\omega|^2 dy=
\big\langle f_4,\omega\big\rangle-\big\langle f_2,\partial_y\omega\big\rangle+(\nu\partial_y\omega+f_2)\bar{\omega}|_{y=-1}^{y=1},
\end{align*}
here  $f_4=ikf_1+i\ell f_3$.
Taking the real part, we get
\begin{align*}
& \f12\f{d}{dt}\|\omega\|^2_{L^2}+\nu\|(\partial_y,\eta)\omega\|_{L^2}^2 \leq\|(f_2,\eta^{-1}f_4)\|_{L^2}\|(\partial_y,\eta)\omega\|_{L^2}
+\|\nu\partial_y\omega+f_2\|_{l^{1}(\{\pm1\})}\|\omega\|_{L^{\infty}}.
\end{align*}Here $\|f\|_{l^{1}(\{\pm1\})}:=|f(t,1)|+|f(t,-1)|$. Thus,
\begin{align*}
& \f{d}{dt}\|\omega\|^2_{L^2}+\nu\|(\partial_y,\eta)\omega\|_{L^2}^2 \leq \nu^{-1}\|(f_2,\eta^{-1}f_4)\|_{L^2}^2
+2\|\nu\partial_y\omega+f_2\|_{l^{1}(\{\pm1\})}\|\omega\|_{L^{\infty}}.
\end{align*}Let
\begin{align}\label{gamma12-def}
\gamma_1(y)=\dfrac{\sinh(\eta (y+1))}{\sinh(2\eta) },\quad \gamma_{-1}(y)=\dfrac{\sinh(\eta(1- y))}{\sinh(2\eta) }.
\end{align}
Since $(\partial_y^2-\eta^2)\varphi=\omega,\ \partial_y\varphi|_{y=\pm1}=\varphi|_{y=\pm1}=0,$ we find that
\begin{align*}
&\langle \omega,\gamma_j\rangle=\langle (\partial^2_y-\eta^2)\varphi,\gamma_j\rangle=\langle \varphi,(\partial^2_y-\eta^2)\gamma_j\rangle=0,\quad j\in\{\pm1\},
\end{align*} and $ \langle \partial_t\omega,\gamma_j\rangle=0$ for $j\in\{\pm1\} $, which implies  that
\begin{align*}
  0&=\langle \partial_t\omega,\gamma_j\rangle=\big\langle\nu(\partial_y^2-\eta^2)\omega-iky\omega+ikf_1+\partial_yf_2+ilf_3,\gamma_j\big\rangle\\
  &=\big\langle\nu\omega,(\partial_y^2-\eta^2)\gamma_j\big\rangle+\big\langle\omega,iky\gamma_j\big\rangle+\langle f_4,\gamma_j\big\rangle-\big\langle f_2,\partial_y\gamma_j\big\rangle+((\nu\partial_y\omega+f_2)\gamma_j-\nu\omega\partial_y\gamma_j)|_{-1}^{1}\\
  &=0+\big\langle(\partial_y^2-\eta^2)\varphi,iky\gamma_j\big\rangle+\langle f_4,\gamma_j\big\rangle-\big\langle f_2,\partial_y\gamma_j\big\rangle+j(\nu\partial_y\omega+f_2)(t,j)-(\nu\omega\partial_y\gamma_j)|_{-1}^{1},
\end{align*}
and we also have
\begin{align*}
  &\big\langle(\partial_y^2-\eta^2)\varphi,iky\gamma_j\big\rangle=\big\langle\varphi,(\partial_y^2-\eta^2)(iky\gamma_j)\big\rangle
  =\big\langle\varphi,2ik\partial_y\gamma_j\big\rangle.
\end{align*}
Thanks to $|\gamma_j'(-j)|=|\gamma_j'(j)|= \eta\coth(2\eta)\leq C\eta$ for $j\in\{\pm 1\},$ we get
\begin{align*}
&\|\gamma'_j\|_{L^2}^2
+\eta^2\|\gamma_j\|_{L^2}^2=-\langle \gamma_j,(\partial^2_y-\eta^2)\gamma_j\rangle+\gamma_j'\gamma_j|_{-1}^1=|\gamma_j'\gamma_j(j)|=|\gamma_j'(j)| \leq C\eta.
\end{align*}
Thus, we obtain
\begin{align*}
  &|(\nu\partial_y\omega+f_2)(t,j)|=\left|\big\langle\varphi,2ik\partial_y\gamma_j\big\rangle+\langle f_4,\gamma_j\big\rangle-\big\langle f_2,\partial_y\gamma_j\big\rangle-(\nu\omega\partial_y\gamma_j)|_{-1}^{1}\right|\\ &\leq 2|k|\|\varphi\|_{L^2}\|\gamma'_j\|_{L^2}+\|(f_2,\eta^{-1}f_4)\|_{L^2}\|(\partial_y,\eta)\gamma_j\|_{L^2}+\nu\|\gamma'_j\|_{l^{1}(\{\pm1\})}\|\omega\|_{L^{\infty}}
\\ &\leq C|k|\eta^{\f12}\|\varphi\|_{L^2}+C\eta^{\f12}\|(f_2,\eta^{-1}f_4)\|_{L^2}+C\nu\eta\|\omega\|_{L^{\infty}},
\end{align*}
and then
\begin{align*}
& \f{d}{dt}\|\omega\|^2_{L^2}+\nu\|(\partial_y,\eta)\omega\|_{L^2}^2 \leq \nu^{-1}\|(f_2,\eta^{-1}f_4)\|_{L^2}^2
+2\|\nu\partial_y\omega+f_2\|_{l^{1}(\{\pm1\})}\|\omega\|_{L^{\infty}}\\ &\leq\nu^{-1}\|(f_2,\eta^{-1}f_4)\|_{L^2}^2
+C\big(|k|\eta^{\f12}\|\varphi\|_{L^2}+\eta^{\f12}\|(f_2,\eta^{-1}f_4)\|_{L^2}+\nu\eta\|\omega\|_{L^{\infty}}\big)\|\omega\|_{L^{\infty}}\\ &\leq C\big(\nu^{-1}\|(f_2,\eta^{-1}f_4)\|_{L^2}^2
+\nu^{-\f12}|k\eta|^{\f32}\|\varphi\|_{L^2}^2+(|\nu k/\eta|^{\f12}+\nu\eta)\|\omega\|_{L^{\infty}}^2\big)\\ &\leq C\big(\nu^{-1}\|(f_1,f_2,f_3)\|_{L^2}^2
+\nu^{-\f12}|k\eta|^{\f32}\|\varphi\|_{L^2}^2+(|\nu k/\eta|^{\f12}+\nu\eta)\|\omega\|_{L^{2}}\|(\partial_y,\eta)\omega\|_{L^2}\big),
\end{align*}
and
\begin{align*}
& \f{d}{dt}\|e^{a\nu^{\f13}t}\omega\|^2_{L^2}+\f{\nu}{2}\|e^{a\nu^{\f13}t}(\partial_y,\eta)\omega\|_{L^2}^2 \\ &\leq C\Big(\nu^{-1}\|e^{a\nu^{\f13}t}(f_1,f_2,f_3)\|_{L^2}^2
+\nu^{-\f12}|k\eta|^{\f32}\|e^{a\nu^{\f13}t}\varphi\|_{L^2}^2+(| k/\eta|+\nu\eta^2)\|e^{a\nu^{\f13}t}\omega\|_{L^{2}}^2\Big),
\end{align*}
here we used the fact that $\nu^{\f13}\leq| 1/\eta|+\nu\eta^2\leq| k/\eta|+\nu\eta^2 .$ This shows that
\begin{align*}
& \|e^{a\nu^{\f13}t}\omega\|^2_{L^{\infty}L^2}+{\nu}\|e^{a\nu^{\f13}t}(\partial_y,\eta)\omega\|_{L^2L^2}^2 \leq C\big(\|\omega_{in}\|_{L^2}^2+\nu^{-1}\|e^{a\nu^{\f13}t}(f_1,f_2,f_3)\|_{L^2L^2}^2\big)\\&
\quad+C\nu^{-\f12}|k\eta|^{\f32}\|e^{a\nu^{\f13}t}\varphi\|_{L^2L^2}^2+C(| k/\eta|+\nu\eta^2)\|e^{a\nu^{\f13}t}\omega\|_{L^2L^{2}}^2.
\end{align*}
\end{proof}

Now let us  prove Proposition \ref{prop:TS-non}.

\begin{proof}\def\RHS{[RHS]}
We denote
\begin{align*}
   \RHS=\nu^{-\f12}\|e^{a\nu^{\f13}t}(f_1,f_2,f_3)\|_{L^2L^2} +\eta^{-1}\|\partial_y\omega_{in}\|_{L^2}+\|\omega_{in}\|_{L^2}.
  \end{align*}
We first consider the case of $\nu\eta^3\leq |k|$.  By Lemma  \ref{lem:TS-non-s1}, we have
\begin{align*}
     &|k\eta|^{\f12}\|e^{a\nu^{\f13}t}(\partial_y,\eta)\varphi\|_{L^2L^2}+\nu^{\f14}|k|^{\f12}\|e^{a\nu^{\f13}t}\omega\|_{L^2L^2}\\ &\leq C\nu^{-\f12}\Big(\|e^{a\nu^{\f13}t}f_2\|_{L^2L^2}+|\nu/k|^{\f13}\|e^{a\nu^{\f13}t}(ikf_1+i\ell f_3)\|_{L^2L^2}\Big) +C\big(\eta^{-1}\|\partial_y\omega_{in}\|_{L^2}+\|\omega_{in}\|_{L^2}\big)\\ &\leq C\nu^{-\f12}\|e^{a\nu^{\f13}t}(f_1,f_2,f_3)\|_{L^2L^2} +C\big(\eta^{-1}\|\partial_y\omega_{in}\|_{L^2}+\|\omega_{in}\|_{L^2}\big)=C\RHS,
  \end{align*}
  which along with Lemma  \ref{lem:TS-non-s2} and $\nu^{\f13}\eta\leq|k| $ implies that
  \begin{align}\label{CA2}
&\nu\eta^2
  \|e^{a\nu^{\f13}t}\omega\|^2_{L^2 L^2}+\eta^2\|e^{a\nu^{\f13}t}(\partial_y,\eta)\varphi\|^2_{L^\infty L^2}+|k\eta|\|e^{a\nu^{\f13}t}(\partial_y,\eta)\varphi\|_{L^2L^2}^2\leq C\RHS^2.
  \end{align}
 Thanks to $1\leq |k|\leq\eta,\ \nu\eta^3\leq |k|,\ |k\eta|^{\f32}\leq\eta^3,\ \nu\eta^2\leq| k/\eta|\leq1,$ we get by Lemma  \ref{lem:TS-non-s2} that
 \begin{align*}
 &\|e^{a\nu^{\f13}t}\omega\|^2_{L^{\infty}L^2}+{\nu}\|e^{a\nu^{\f13}t}\partial_y\omega\|_{L^2L^2}^2 \leq C\big(\|\omega_{in}\|_{L^2}^2+\nu^{-1}\|e^{a\nu^{\f13}t}(f_1,f_2,f_3)\|_{L^2L^2}^2\big)\\&
\qquad+C\nu^{-\f12}|k\eta|^{\f32}\|e^{a\nu^{\f13}t}\varphi\|_{L^2L^2}^2+C(| k/\eta|+\nu\eta^2)\|e^{a\nu^{\f13}t}\omega\|_{L^2L^{2}}^2\\ &\leq C\RHS^2
+C\nu^{-\f12}\eta^3\|e^{a\nu^{\f13}t}\varphi\|_{L^2L^2}^2+C\|e^{a\nu^{\f13}t}\omega\|_{L^2L^{2}}^2\\ &\leq C\RHS^2+C\nu^{-\f12}|k|^{-1}\RHS^2\leq C\nu^{-\f12}\RHS^2,
\end{align*}
which along with \eqref{CA2} gives our result when $\nu\eta^3\leq |k|$.\smallskip

Next we consider the case of  $\nu\eta^3\geq |k|$. In this case, we have (for $0\leq a\leq 1/8$)
\begin{align*}
\big(|k/\eta|+2a\nu^\f13\big)\|e^{a\nu^{1/3}t}(\partial_y,\eta)\varphi\|^2_{L^2}\leq \big(|k/\eta^3|+2a\nu^\f13\eta^{-2}\big)|\|e^{a\nu^{1/3}t}\omega\|^2_{L^2}
\leq \f {5\nu} 4\|e^{a\nu^{1/3}t}\omega\|^2_{L^2},
\end{align*}
from which and \eqref{dtphi}, we infer that
\begin{align*}
  &\f{d}{dt}\|e^{a\nu^{1/3}t}(\partial_y,\eta)\varphi\|^2_{L^2}+\f{\nu}4\|e^{a\nu^{1/3}t}\omega\|^2_{L^2}\leq2/(\nu \eta^2)\|e^{a\nu^{1/3}t}(f_1,f_2,f_3)\|_{L^2}^2,
\end{align*}
and then
\begin{align*}
{\nu}\eta^2\|e^{a\nu^{1/3}t}\omega\|^2_{L^2L^2}\leq C\nu^{-1}\|e^{a\nu^{\f13}t}(f_1,f_2,f_3)\|_{L^2L^2} +C\|\omega_{in}\|_{L^2}^2.
\end{align*}
This in turn gives
\begin{align*}
\big(|k\eta|+\nu^\f13\eta^2\big)\|e^{a\nu^{1/3}t}(\partial_y,\eta)\varphi\|^2_{L^2L^2}
&\le C{\nu}\eta^2\|e^{a\nu^{1/3}t}\omega\|^2_{L^2L^2}\\
&\leq C\nu^{-1}\|e^{a\nu^{\f13}t}(f_1,f_2,f_3)\|_{L^2L^2}^2 +C\|\omega_{in}\|_{L^2}^2,
\end{align*}
which along with Lemma  \ref{lem:TS-non-s2}  gives \eqref{CA2}.
Due to $\nu\eta^3\geq |k|$, we have $| k/\eta|\leq\nu\eta^2$ and $\nu^{-\f12}|k\eta|^{\f32}\leq |k\eta^3|.$  Then we get by Lemma  \ref{lem:TS-non-s2}  that
\begin{align*}
&\|e^{a\nu^{\f13}t}\omega\|^2_{L^{\infty}L^2}+{\nu}\|e^{a\nu^{\f13}t}\partial_y\omega\|_{L^2L^2}^2 \leq C(\|\omega_{in}\|_{L^2}^2+\nu^{-1}\|e^{a\nu^{\f13}t}(f_1,f_2,f_3)\|_{L^2L^2}^2)\\&
\qquad+C\nu^{-\f12}|k\eta|^{\f32}\|e^{a\nu^{\f13}t}\varphi\|_{L^2L^2}^2+C(| k/\eta|+\nu\eta^2)\|e^{a\nu^{\f13}t}\omega\|_{L^2L^{2}}^2\\ &\leq C\RHS^2
+C|k\eta^3|\|e^{a\nu^{\f13}t}\varphi\|_{L^2L^2}^2+C\nu\eta^2\|e^{a\nu^{\f13}t}\omega\|_{L^2L^{2}}^2\leq C\RHS^2,
\end{align*}
which along with \eqref{CA2} gives our result when $\nu\eta^3\ge |k|$.
\end{proof}

\section{Nonlinear interactions}

\subsection{Anisotropic bilinear estimates}

\begin{Lemma}\label{lemma-A.4}
For  $\{j,k\}=\{1,3\}$, it holds that
\begin{align}
\label{f1}&\|f_1f_2\|_{L^2}\leq C\big(\|\partial_kf_1\|_{H^1}+\|f_1\|_{H^1}\big)\big(\|\partial_jf_2\|_{L^2}+\|f_2\|_{L^2}\big),\\
\label{f2}&\|f_1f_2\|_{L^2}+\|\partial_j(f_1f_2)\|_{L^2}\leq C\|(\partial_x\partial_zf_1,\partial_xf_1,\partial_zf_1,f_1)\|_{H^1}\|(\partial_jf_2,f_2)\|_{L^2},\\
\label{f3}&\|f_1f_2\|_{L^2}+\|\partial_j(f_1f_2)\|_{L^2}\leq C\|(\partial_jf_1,f_1)\|_{H^1}\|(\partial_x\partial_zf_2,\partial_xf_2,\partial_zf_2,f_2)\|_{L^2},\\
\label{f4}&\|\nabla(f_1f_2)\|_{L^2}\leq C\|(\partial_x\partial_zf_1,\partial_xf_1,\partial_zf_1,f_1)\|_{H^1}\|f_2\|_{H^1},\\ \label{f5}&\|\nabla(f_1f_2)\|_{L^2}\leq C\big(\|\partial_kf_1\|_{H^1}+\|f_1\|_{H^1})(\|\partial_jf_2\|_{H^1}+\|f_2\|_{H^1}\big),
\end{align}
and
\begin{align}
 \label{f6}\|\nabla(f_1f_2)\|_{L^2}\leq C\big(&\|(\partial_kf_1,f_1)\|_{H^2}\|(\partial_jf_2,f_2)\|_{L^2}\\
 &+
\|(\partial_x\partial_zf_1,\partial_zf_1,\partial_zf_1,f_1)\|_{L^2}\|f_2\|_{H^2}\big).\nonumber
\end{align}
\end{Lemma}

\begin{proof}
By H\"older inequality and Sobolev embedding, we get
\begin{align*}
\|f_1f_2\|_{L^2}
\leq& \Big\| \|f_1\|_{L^{\infty}_y}\|f_2\|_{L^2_y}\Big\|_{L^2_{x,z}}
\leq C\Big\|\big(\|\partial_yf_1\|_{L^2_y}+\|f_1\|_{L^2_y}\big)\|f_2\|_{L^2_y}\Big\|_{L^2_{x,z}}\\
\leq &C\Big\|\big(\|\partial_yf_1\|_{L^{\infty}_zL^2_y}+\|f_1\|_{L^{\infty}_zL^2_y}\big)\|f_2\|_{L^2_{z,y}}\Big\|_{L^2_{x}}\\
\leq &C\Big\|\big(\|(\partial_z\partial_yf_1,\partial_yf_1)\|_{L^{2}_zL^2_y}+\|(\partial_zf_1,f_1)\|_{L^{2}_zL^2_y}\big)
\|f_2\|_{L^2_{z,y}}\Big\|_{L^2_{x}}\\
\leq &C\big\|(\partial_z\partial_yf_1,\partial_zf_1,\partial_yf_1,f_1)\big\|_{L^{2}}\|f_2\|_{L^{\infty}_xL^2_{z,y}}\\
\leq &C\big(\|\partial_zf_1\|_{H^1}+\|f_1\|_{H^1}\big)\big(\|\partial_xf_2\|_{L^2}+\|f_2\|_{L^2}\big).
\end{align*}
This proves the \eqref{f1} for the case of $(j,k)=(1,3)$, and the case of $(j,k)=(3,1)$ is similar.\smallskip

Using the fact that
\begin{align*}
&\|f\|_{L^{\infty}_{x,z}L^2_y}=\Big\| \|f\|_{L^{\infty}_xL^2_y}\Big\|_{L^{\infty}_{z}}
\leq C\Big\| \|(\partial_xf,f)\|_{L^{2}_xL^2_y}\Big\|_{L^{\infty}_{z}}\leq C\big\|(\partial_x\partial_zf,\partial_xf,\partial_zf,f)\big\|_{L^2},
\end{align*}
we infer that
\begin{align*}
\|(\partial_yf,f)\|_{L^{\infty}_{x,z}L^2_y}\leq C\big\|(\partial_x\partial_zf,\partial_xf,\partial_zf,f)\big\|_{H^1},
\end{align*}
which gives
\begin{align}\label{f7}
\|f_1f_2\|_{L^2}
\leq& C\Big\| \|(\partial_yf_1,f_1)\|_{L^2_y}\|f_2\|_{L^2_y}\Big\|_{L^2_{x,z}}\leq C \|(\partial_yf_1,f_1)\|_{L^{\infty}_{x,z}L^2_y}\|f_2\|_{L^2}\\
\nonumber
\leq &C\big\|(\partial_x\partial_zf_1,\partial_xf_1,\partial_zf_1,f_1)\big\|_{H^1}\|f_2\|_{L^2},
\label{f8}
\end{align}
and
\begin{align}
\|f_1f_2\|_{L^2}
\leq& C\Big\| \|(\partial_yf_1,f_1)\|_{L^2_y}\|f_2\|_{L^2_y}\Big\|_{L^2_{x,z}}\leq C \|(\partial_yf_1,f_1)\|_{L^2}\|f_2\|_{L^{\infty}_{x,z}L^2_y}\\
\nonumber \leq &C\|f_1\|_{H^1}\big\|(\partial_x\partial_zf_2,\partial_xf_2,\partial_zf_2,f_2)\big\|_{L^2}.
\end{align}

By \eqref{f1} and \eqref{f7}, we have
\begin{align*}
\|\partial_j(f_1f_2)\|_{L^2}\leq&\|\partial_jf_1f_2\|_{L^2}+\|f_1\partial_jf_2\|_{L^2}\\ \leq& C\|(\partial_k\partial_jf_1,\partial_jf_1)\|_{H^1}\|(\partial_jf_2,f_2)\|_{L^2}+
C\|(\partial_x\partial_zf_1,\partial_xf_1,\partial_zf_1,f_1)\|_{H^1}\|\partial_jf_2\|_{L^2}\\ \leq& C\|(\partial_x\partial_zf_1,\partial_xf_1,\partial_zf_1,f_1)\|_{H^1}\|(\partial_jf_2,f_2)\|_{L^2},
\end{align*}which gives \eqref{f2}.\smallskip

By \eqref{f8} and  \eqref{f1}, we have
\begin{align*}
\|\partial_j(f_1f_2)\|_{L^2}\leq&\|\partial_jf_1f_2\|_{L^2}+\|f_1\partial_jf_2\|_{L^2}\\ \leq& C\|\partial_jf_1\|_{H^1}\|(\partial_x\partial_zf_2,\partial_xf_2,\partial_zf_2,f_2)\|_{L^2}+
C\|(\partial_jf_1,f_1)\|_{H^1}\|(\partial_k\partial_jf_2,\partial_jf_2)\|_{L^2}\\ \leq& C\|(\partial_jf_1,f_1)\|_{H^1}\|(\partial_x\partial_zf_2,\partial_xf_2,\partial_zf_2,f_2)\|_{L^2},
\end{align*}
which gives \eqref{f3}.\smallskip

By \eqref{f7} and  \eqref{f8},  we get
\begin{align*}
\|\nabla(f_1f_2)\|_{L^2}\leq& \|f_1\nabla f_2\|_{L^2}+\|(\nabla f_1)f_2\|_{L^2}\\ \leq& C\|(\partial_x\partial_zf_1,\partial_xf_1,\partial_zf_1,f_1)\|_{H^1}\|\nabla f_2\|_{L^2}+
C\|\nabla(\partial_x\partial_zf_1,\partial_xf_1,\partial_zf_1,f_1)\|_{L^2}\|f_2\|_{H^1}\\ \leq& C\|(\partial_x\partial_zf_1,\partial_xf_1,\partial_zf_1,f_1)\|_{H^1}\|f_2\|_{H^1},
\end{align*}
which gives \eqref{f4}. \smallskip

By \eqref{f1}, we have
\begin{align*}
&\|f_1\nabla f_2\|_{L^2}\leq C(\|\partial_kf_1\|_{H^1}+\|f_1\|_{H^1})(\|\partial_j\nabla f_2\|_{L^2}+\|\nabla f_2\|_{L^2}),\\
&\|(\nabla f_1)f_2\|_{L^2}\leq C(\|\partial_k\nabla f_1\|_{L^2}+\|\nabla f_1\|_{L^2})(\|\partial_j f_2\|_{H^1}+\| f_2\|_{H^1}),
\end{align*}
which give \eqref{f5}.  By \eqref{f1} and \eqref{f8}, we have
\begin{align*}&\|(\nabla f_1)f_2\|_{L^2}\leq C(\|\partial_k\nabla f_1\|_{H^1}+\|\nabla f_1\|_{H^1})(\|\partial_j f_2\|_{L^2}+\| f_2\|_{L^2}),\\
&\|f_1\nabla f_2\|_{L^2}\leq C\|(\partial_x\partial_zf_1,\partial_xf_1,\partial_zf_1,f_1)\|_{L^2}\|\nabla f_2\|_{H^1},
\end{align*}
which give \eqref{f6}. This completes the proof.
\end{proof}

\begin{Lemma}\label{Lem: bilinear zero nonzero}
  If $\partial_xf_1=0$, then it holds that
  \begin{align*}
    &\|f_1f_2\|_{L^2}\leq
    C\|f_1\|_{H^1}(\|f_2\|_{L^2}+\|\partial_zf_2\|_{L^2}),\\
    &\|(\partial_x,\partial_z)(f_1f_2)\|_{L^2} \leq
    C\big(\|f_1\|_{H^1}+\|\partial_zf_1\|_{H^1}\big)\big(\|f_2\|_{L^2}
    +\|(\partial_x,\partial_z)f_2\|_{L^2}\big),\\
    &\|(\partial_x,\partial_z)(f_1f_2)\|_{L^2} \leq
    C\big(\|f_1\|_{L^2}+\|\partial_zf_1\|_{L^2}\big)\big(\|f_2\|_{H^1}
    +\|(\partial_x,\partial_z)f_2\|_{H^1}\big),\\
    &\|\partial_x(f_1f_2)\|_{L^2}\leq C\|f_1\|_{H^1}\big(\|\partial_xf_2\|_{L^2}
    +\|\partial_z\partial_x f_2\|_{L^2}\big).
  \end{align*}
\end{Lemma}

\begin{proof}The first inequality follows from \eqref{f1} in Lemma \ref{lemma-A.4} by taking $(j,k)=(3,1)$ and
 using $\partial_xf_1=0$.  The second and third inequality follows from \eqref{f2} and \eqref{f3} in Lemma \ref{lemma-A.4} by noting that
 $$\|(\partial_x\partial_zf_1,\partial_xf_1,\partial_zf_1,f_1)\|_{H^k}=\|(\partial_zf_1,f_1)\|_{H^k},\quad k=0,1.$$
 As $\partial_x(f_1f_2)=f_1\partial_xf_2,$  the fourth inequality follows from the first inequality.
\end{proof}

\begin{Lemma}\label{lemma-A.1}If $\partial_xf_1=0$, then it holds that
\begin{align*}
&\|f_1\|_{L^{\infty}}\leq C\big(\|f_1\|_{H^1}+\|\partial_zf_1\|_{H^1}\big),\\
&\|\nabla(f_1f_2)\|_{L^2}\leq C(\|f_1\|_{H^1}+\|\partial_zf_1\|_{H^1})\|f_2\|_{H^1},\\
&\|\nabla(f_1f_2)\|_{L^2}\leq C\|f_1\|_{H^1}(\|f_2\|_{H^1}+\|\partial_zf_2\|_{H^1}).
\end{align*}
\end{Lemma}

\begin{proof}
The first inequality follows by noting that
\begin{align*}
\|f_1\|_{L^{\infty}}\leq& C \big\|(\|\partial_yf_1\|_{L^2_y}+\|f_1\|_{L^2_y})\big\|_{L^{\infty}_z}\\
\leq&C\big(\|\partial_z\partial_yf_1\|_{L^2}+\|\partial_yf_1\|_{L^2}+\|\partial_zf_1\|_{L^2} +\|f_1\|_{L^2}\big)
\leq C\big(\|f_1\|_{H^1}+\|\partial_zf_1\|_{H^1}\big).
\end{align*}
The second and third inequality follows from \eqref{f4} and \eqref{f5} in Lemma \ref{lemma-A.4} by taking $(j,k)=(3,1)$ and using $\partial_xf_1=0$.
\end{proof}

\begin{Lemma}\label{Lem: bil good deri}
 Let $V$ satisfy $\|V-y\|_{H^4}\leq \varepsilon_0, \partial_xV=0,\ (V-y)|_{y=\pm1}=0$ and $\kappa=\partial_zV/\partial_yV$. If $\partial_xf_1=0,\,  P_0f_2=0$,  then we have
  \begin{align*}
     &\|f_1f_2\|_{L^2}\leq C\|f_1\|_{H^1}\big(\|f_2\|_{L^2}+\|(\partial_z-\kappa\partial_y)f_2\|_{L^2}\big),\\
     &\|\nabla(f_1f_2)\|_{L^2}\leq C\|f_1\|_{H^1}\big(\|f_2\|_{H^1}+\|(\partial_z-\kappa\partial_y)f_2\|_{H^1}\big).
       \end{align*}
 Let $h$ solve $\Delta h=f_1f_2,\ h|_{y=\pm1}=0$, and assume $\partial_xf_2=ik f_2,$ $k\in\Z\setminus\{0\}$. Then we have
   \begin{align*}
     &\|\nabla h\|_{L^2}\leq C\|f_1\|_{L^2}\big(\|f_2\|_{L^2}+\|(\partial_z-\kappa\partial_y)f_2\|_{L^2}\big).
  \end{align*}
 \end{Lemma}

\begin{proof}  Take $F_l(X,Y,Z)$ so that $F_l(x,V(t,y,z),z)=f_l(x,y,z)$. Using the facts that $(\partial_z-\kappa\partial_y)f_2(x,y,z)=\partial_ZF_2\big(x,V(t,y,z),z\big)$ and  for $k=0,1$,
  \begin{align*}
     &\|F_l\|_{H^k}\sim \|f_l\|_{H^k},\quad \|\nabla(F_1F_2)\|_{L^2}\sim \|\nabla(f_1f_2)\|_{L^2},\quad \|\partial_zF_2\|_{H^k}\sim \|(\partial_z-\kappa\partial_y)f_2\|_{H^k},
  \end{align*}
we can deduce the first two inequalities from Lemma \ref{Lem: bilinear zero nonzero} and Lemma \ref{lemma-A.1}.

Since $\partial_xf_1=0,\ \partial_xf_2=ik f_2,$ we have  $\partial_xh=ik h, $ and we can write $h(x,y,z)=e^{ikx}h_k(y,z),$. Thus,  $\|\nabla h\|_{L^2} \geq\| h\|_{L^2}$ and
\begin{align*}
\|\nabla h\|_{L^2}^2=&-\langle \Delta h,h\rangle=-\langle f_1f_2,h\rangle\leq \|f_1\|_{L^2}\|f_2h\|_{L^2}=\|f_1\|_{L^2}\|f_2h_k\|_{L^2}.
\end{align*}
By the first inequality of the lemma, we get
\begin{align*}
\|f_2h_k\|_{L^2}&\leq C\|h_k\|_{H^1}\big(\|f_2\|_{L^2}+\|(\partial_z-\kappa\partial_y)f_2\|_{L^2}\big)\\
&\leq C\|\nabla h\|_{L^2}\big(\|f_2\|_{L^2}+\|(\partial_z-\kappa\partial_y)f_2\|_{L^2}\big).
\end{align*}
Then we have
\begin{align*}
\|\nabla h\|_{L^2}^2\leq \|f_1\|_{L^2}\|f_2h_k\|_{L^2}\leq C\|f_1\|_{L^2}\|\nabla h\|_{L^2}\big(\|f_2\|_{L^2}+\|(\partial_z-\kappa\partial_y)f_2\|_{L^2}\big),
\end{align*}
which gives the third inequality.
\end{proof}

\subsection{The velocity estimates in terms of the energy}

\begin{Lemma}\label{lem:u-relation}
  It holds that for $k\ge 0$,
  \begin{align*}
  &\|\nabla^k(\partial_x,\partial_z)\partial_xu_{\neq}\|_{L^{2}}\leq C\big(\| \nabla^k(\partial_x^2+\partial_z^2)u_{\neq}^3\|_{L^2}+\| \nabla^{k+1}(\partial_x,\partial_z) u_{\neq}^2\|_{L^2}\big),\\
&\|\Delta(\partial_x,\partial_z)u_{\neq}\|_{L^2}\leq \|\Delta\omega_{\neq}^2\|_{L^2}+
  \|\nabla\Delta u_{\neq}^2\|_{L^2}.  \end{align*}
\end{Lemma}

\begin{proof}
  Thanks to $\text{div}u_{\neq}=\partial_xu_{\neq}^1+\partial_yu_{\neq}^2+\partial_zu_{\neq}^3=0,$
  we have
\begin{align*}
\|\nabla^k(\partial_x,\partial_z)\partial_xu_{\neq}\|_{L^{2}}\leq&\|\nabla^k(\partial_x,\partial_z)\partial_xu_{\neq}^1\|_{L^{2}}
+\|\nabla^k(\partial_x,\partial_z)\partial_xu_{\neq}^2\|_{L^{2}}+\|\nabla^k(\partial_x,\partial_z)\partial_xu_{\neq}^3\|_{L^{2}}\\
\leq&\|\nabla^k(\partial_x,\partial_z)\partial_yu_{\neq}^2\|_{L^{2}}+\|\nabla^k(\partial_x,\partial_z)\partial_zu_{\neq}^3\|_{L^{2}}
\\&+\|\nabla^k(\partial_x,\partial_z)\partial_xu_{\neq}^2\|_{L^{2}}+\|\nabla^k(\partial_x,\partial_z)\partial_xu_{\neq}^3\|_{L^{2}}\\ \leq&
C\big(\|\nabla^{k+1}(\partial_x,\partial_z)u_{\neq}^2\|_{L^{2}}+\|\nabla^k(\partial_x^2,\partial_x\partial_z,\partial_z^2)u_{\neq}^3\|_{L^{2}}
\big)\\ \leq&C\big(\| \nabla^k(\partial_x^2+\partial_z^2)u_{\neq}^3\|_{L^2}+\| \nabla^{k+1}(\partial_x,\partial_z) u_{\neq}^2\|_{L^2}\big).
\end{align*}

Using the formula $\|(\partial_x,\partial_z)(f_1,f_2)\|^2_{L^2}=\|(\partial_zf_1-\partial_xf_2,\partial_xf_1+\partial_zf_2)\|^2_{L^2}$, we can deduce that
  \begin{align*}
  &\|\Delta(\partial_x,\partial_z)(u_{\neq}^1,u_{\neq}^3)\|^2_{L^2}=\|\Delta\omega_{\neq}^2\|^2_{L^2}+
  \|\Delta\partial_y u_{\neq}^2\|^2_{L^2},
\end{align*}which implies the second inequality.
\end{proof}

\begin{Lemma}\label{lem:u1-H2}
  It holds that
  \begin{align*}
    &\|\bar{u}^1\|_{H^2}\leq CE_1\min(\nu t+\nu^{2/3},1).
  \end{align*}
\end{Lemma}

  \begin{proof}
 As $\bar{u}^{1,0}|_{t=0}=0$, we have
  \begin{align*}
     &\|\bar{u}^{1,0}\|_{H^2}\leq \int_{0}^{t}\|\partial_t\bar{u}^{1,0}(s)\|_{H^2}ds \leq C\nu tE_1,
  \end{align*}
on the other hand,  $\|\bar{u}^{1,0}\|_{H^2}\leq \|\bar{u}^{1,0}\|_{L^\infty H^4}\leq E_1$. Thus, we get
\beno
\|\bar{u}^{1,0}\|_{H^2}\leq CE_1\min(\nu t,1).
\eeno
 As $\|\bar{u}^{1,\neq}\|_{H^2}\leq E_{1,\neq}\leq \nu^{2/3}E_1,$  we get
  \begin{align*}
    \|\bar{u}^1\|_{H^2}\leq &\|\bar{u}^{1,0}\|_{H^2}+\|\bar{u}^{1,\neq}\|_{H^2}\leq CE_1\big( \min(\nu t,1)+\nu^{2/3}\big)\leq CE_1\min(\nu t+\nu^{2/3},1).
 \end{align*}
   \end{proof}

 \begin{Lemma}\label{lem:u10-Linfty}
 It holds that
\begin{align*}
&\left\|\f{\partial_z\bar{u}^{1,0}}{\min((\nu t)^{\frac{1}{2}},1-y^2)}\right\|_{L^{\infty}L^{\infty}}^2
\leq C\big(\|\bar{u}^{1,0}\|_{L^{\infty}H^4}^2+\nu^{-2}\|\partial_t\bar{u}^{1,0}\|_{L^{\infty}H^2}^2\big)\leq CE_{1,0}^2.
\end{align*}
\end{Lemma}

\begin{proof}As $\bar{u}^{1,0}|_{y=\pm1}=0$,  $\partial_z\bar{u}^{1,0}|_{y=\pm1}=0$, and then
\begin{align*}
\left\|\f{\partial_z\bar{u}^{1,0}}{1-y^2}\right\|_{L^{\infty}}^2\leq& C\|\partial_y\partial_z\bar{u}^{1,0}\|_{L^{\infty}}^2\leq C\|\partial_y\partial_z\bar{u}^{1,0}\|_{H^2}^2\leq C\|\bar{u}^{1,0}\|_{H^4}^2,
\end{align*}
which gives
\begin{align}\label{tri4.2}
\left\|\f{\partial_z\bar{u}^{1,0}}{1-y^2}\right\|_{L^{\infty}L^{\infty}}^2\leq C\|\bar{u}^{1,0}\|_{L^{\infty}H^4}^2.
\end{align}
On the other hand, we have
\begin{align*}
\|\partial_z\bar{u}^{1,0}\|_{L^{\infty}}^2&\leq  C\|\partial_z\bar{u}^{1,0}\|_{H^2}^2\leq C\|\bar{u}^{1,0}\|_{H^3}^2\leq C\|\bar{u}^{1,0}\|_{H^4}\|\bar{u}^{1,0}\|_{H^2},
\end{align*}
and  due to $\bar{u}^{1,0}|_{t=0}=0$, we have
\begin{align*}
\|\bar{u}^{1,0}(t)\|_{H^2}\leq\int_0^t\|\partial_s\bar{u}^{1,0}(s)\|_{H^2}ds\leq t\|\partial_t\bar{u}^{1,0}\|_{L^{\infty}H^2}.
\end{align*}
This  along with \eqref{tri4.2} gives
\begin{align*}
\left\|\f{\partial_z\bar{u}^{1,0}}{\min((\nu t)^{\frac{1}{2}},1-y^2)}\right\|_{L^{\infty}L^{\infty}}^2
\leq& C\big(\|\bar{u}^{1,0}\|_{L^{\infty}H^4}^2+
\|\bar{u}^{1,0}\|_{L^{\infty}H^4}\|{\bar{u}^{1,0}}/{(\nu t)}\|_{L^{\infty}H^2}\big)\\ \leq& C\big(\|\bar{u}^{1,0}\|_{L^{\infty}H^4}+
(\nu^{-1}\|\partial_t\bar{u}^{1,0}\|_{L^{\infty}H^2})^2\big)\leq CE_{1,0}^2.
\end{align*}
\end{proof}

\begin{Lemma}\label{lem:u23-zero}
  It holds that for $k\in\{2,3\}$,
  \begin{align*}
    & \|\bar{u}^2\|_{H^2}+\|\nabla\bar{u}^2\|_{H^1}
    +\|\bar{u}^3\|_{H^1}+\|\partial_z\bar{u}^3\|_{H^1} \leq CE_2,\\
    &\|\bar{u}^k\|_{L^\infty
    L^\infty}+\nu^{1/2}\|\nabla\bar{u}^k\|_{L^2L^\infty} \leq CE_2,\\
    &\|\nabla(\bar{u}^kf)\|_{L^2}+ \|\bar{u}^k\nabla f\|_{L^2} \leq
    CE_2\| f\|_{H^1}.
    \end{align*}
\end{Lemma}

\begin{proof}
Thanks to $\partial_y\bar{u}^2+\partial_z\bar{u}^3=0 $,  we have
\begin{align*}
\|\bar{u}^2\|_{H^2}+\|\nabla\bar{u}^2\|_{H^1}
    +\|\bar{u}^3\|_{H^1}+\|\partial_z\bar{u}^3\|_{H^1}&\leq C\big(\|\Delta\bar{u}^2\|_{L^2}+\|\nabla\bar{u}^3\|_{L^2}+\|\partial_y\bar{u}^2\|_{H^1}\big)\\
&\leq C\big(\|\Delta\bar{u}^2\|_{L^2}+\|\nabla\bar{u}^3\|_{L^2}\big)\leq CE_2.
\end{align*}

Thanks to Lemma \ref{lemma-A.1},  we have
\begin{align*}
&\|\bar{u}^2\|_{L^{\infty}}+\|\bar{u}^3\|_{L^{\infty}}\leq C\big(\|\bar{u}^2\|_{H^1}+\|\partial_z\bar{u}^2\|_{H^1}+\|\bar{u}^3\|_{H^1}+\|\partial_z\bar{u}^3\|_{H^1}\big)\leq CE_2,
\end{align*} for any $t\in[0,T]$,  and
\begin{align*}
&\|\nabla\bar{u}^2\|_{L^2L^{\infty}}\leq C\|\nabla\bar{u}^2\|_{L^2H^2}\leq C\|\nabla\Delta\bar{u}^2\|_{L^2L^2}\leq C\nu^{-1/2}E_2,\\
&\|\nabla\bar{u}^3\|_{L^{\infty}}
\leq C\big(\|\partial_z\nabla\bar{u}^3\|_{H^1}+\|\nabla\bar{u}^3\|_{H^1}\big)\leq C\big(\|\nabla\Delta\bar{u}^2\|_{L^2}+\|\Delta\bar{u}^3\|_{L^2}\big),
\end{align*}
for  any $t\in[0,T],$ and then
\begin{align*}
&\|\nabla\bar{u}^3\|_{L^2L^{\infty}}\leq C\big(\|\nabla\Delta\bar{u}^2\|_{L^2L^2}+\|\Delta\bar{u}^3\|_{L^2L^2}\big)\leq C\nu^{-1/2}E_2.
\end{align*}

By Lemma \ref{lemma-A.1} again, we get
\begin{align*}
   \|\nabla(\bar{u}^kf)\|_{L^2}+ \|\bar{u}^k\nabla f\|_{L^2} &\leq\big(\|\bar{u}^k\|_{H^1}+\|\partial_z\bar{u}^k\|_{H^1}\big)\| f\|_{H^1}+\|\bar{u}^k\|_{L^\infty
    }\|\nabla f\|_{L^2}\\&\leq
    CE_2\| f\|_{H^1}\quad
    k\in\{2,3\}.
  \end{align*}
\end{proof}

\begin{Lemma}\label{lem:u-nonzero}
  It holds that
  \begin{align*}
   &\nu^{1/6}
   \|e^{\f94\epsilon\nu^{1/3}t}\partial_x\nabla u^2_{\neq}\|_{L^2L^2}\leq  E^{\f14}_5E^{\f34}_3,\\
   &\nu^{1/6}
   \|e^{\f94\epsilon\nu^{1/3}t}\partial_x(\partial_x,\partial_z) u^k_{\neq}\|_{L^2L^2}\leq  CE^{\f12}_5E^{\f12}_3,\ \ k\in\{2,3\},\\
   &\nu^{1/6}
   \|e^{\f94\epsilon\nu^{1/3}t}\partial_x^2 u^1_{\neq}\|_{L^2L^2}\leq  C\big(E^{\f14}_5E^{\f34}_3+E^{\f12}_5E^{\f12}_3\big),\\
   &\nu^{1/2}\Big(\|e^{\f{17}{8}\epsilon\nu^{1/3}t}\partial_x\nabla u_{\neq}\|_{L^2L^2}+\|e^{\f{17}{8}\epsilon\nu^{1/3}t}\partial_z\nabla (u_{\neq}^2,u_{\neq}^3)\|_{L^2L^2}\Big)\leq  C\big(E^{\f18}_5E^{\f78}_3+E^{\f14}_5E^{\f34}_3\big).
  \end{align*}
\end{Lemma}

\begin{proof}
Since $\|e^{2\epsilon\nu^{1/3}t}\partial_x\nabla u^2_{\neq}\|_{L^2L^2}\leq E_{3,0}\le E_3$, and \begin{align*}
\|e^{\f52\epsilon\nu^{1/3}t}\partial_x\nabla u^2_{\neq}\|_{L^2L^2}\leq&\|e^{2\epsilon\nu^{1/3}t}\Delta u^2_{\neq}\|_{L^2L^2}^{\f12}\|e^{3\epsilon\nu^{1/3}t}\partial_x^2 u^2_{\neq}\|_{L^2L^2}^{\f12}\\ \leq&\|e^{2\epsilon\nu^{1/3}t}\partial_x\Delta u^2_{\neq}\|_{L^2L^2}^{\f12}(\nu^{-1/6}E_5)^{\f12}\\\leq&(\nu^{-1/2}E_3)^{\f12}(\nu^{-1/6}E_5)^{\f12}=\nu^{-1/3}E^{\f12}_5E^{\f12}_3,
\end{align*}
we get
 \begin{align*}
   \nu^{1/6}
   \|e^{\f94\epsilon\nu^{1/3}t}\partial_x\nabla u^2_{\neq}\|_{L^2L^2}\leq&\nu^{1/6}\|e^{2\epsilon\nu^{1/3}t}\partial_x\nabla u^2_{\neq}\|_{L^2L^2}^{\f12}\|e^{\f52\epsilon\nu^{1/3}t}\partial_x\nabla u^2_{\neq}\|_{L^2L^2}^{\f12}\\
   \leq& \nu^{1/6}E_3^{\f12}(\nu^{-1/3}E^{\f12}_5E^{\f12}_3)^{\f12}=E^{\f14}_5E^{\f34}_3.
\end{align*}

For $k\in\{2,3\} $,  we have
 \begin{align*}&\|\partial_x(\partial_x,\partial_z)u^k_{\neq}\|^2_{L^2}
      \leq \|\partial_x^2u^k_{\neq}\|_{L^2}\|(\partial_x^2+\partial_z^2)u^k_{\neq}\|_{L^2}
      \leq\|\partial_x^2u^k_{\neq}\|_{L^2}e^{-2\epsilon\nu^{1/3}t}E_3,\end{align*}
which gives
 \begin{align*}
 \|e^{\f94\epsilon\nu^{1/3}t}\partial_x(\partial_x,\partial_z) u^k_{\neq}\|_{L^2L^2}^2\leq&\|e^{\f52\epsilon\nu^{1/3}t}\partial_x^2 u^k_{\neq}\|_{L^1L^2}E_3\\
  \leq&\|e^{3\epsilon\nu^{1/3}t}\partial_x^2 u^k_{\neq}\|_{L^2L^2}\|e^{-\f12\epsilon\nu^{1/3}t}\|_{L^2(0,T)}E_3\\\leq& (\nu^{-1/6}E_5)(C\nu^{-1/6})E_3=C\nu^{-1/3}E_5E_3.
 \end{align*}
Due to  $\partial_x^2 u^1_{\neq}=-\partial_x\partial_y u^2_{\neq}-\partial_x\partial_z u^3_{\neq} $, the third inequality follows from the first two inequalities.

Notice that
\begin{align*}\|\partial_x\nabla u_{\neq}\|_{L^2}^2+\|\partial_z\nabla (u_{\neq}^2,u_{\neq}^3)\|_{L^2}^2
      &\leq \|\partial_x^2 u_{\neq}\|_{L^2}\|\Delta u_{\neq}\|_{L^2}+\|\partial_z\Delta (u_{\neq}^2,u_{\neq}^3)\|_{L^2}\|\partial_z (u_{\neq}^2,u_{\neq}^3)\|_{L^2}\\&
      \leq C\|(\partial_x,\partial_z)\Delta u_{\neq}\|_{L^2}\big(\|\partial_x^2 u_{\neq}\|_{L^2}+\|\partial_x\partial_z (u_{\neq}^2,u_{\neq}^3)\|_{L^2}\big),
      \end{align*}
 which along with Lemma \ref{lem:u-relation} gives
  \begin{align*}&\nu \Big(\|e^{\f{17}{8}\epsilon\nu^{1/3}t}\partial_x\nabla u_{\neq}\|_{L^2L^2}^2+\|e^{\f{17}{8}\epsilon\nu^{1/3}t}\partial_z\nabla (u_{\neq}^2,u_{\neq}^3)\|_{L^2L^2}^2\Big) \\
      &\leq C\nu\|e^{2\epsilon\nu^{1/3}t}\big(\partial_x,\partial_z)\Delta u_{\neq}\|_{L^2L^2}\Big(\|e^{\f{9}{4}\epsilon\nu^{1/3}t}\partial_x^2 u_{\neq}\|_{L^2L^2}+|e^{\f{9}{4}\epsilon\nu^{1/3}t}\partial_x\partial_z (u_{\neq}^2,u_{\neq}^3)\|_{L^2L^2}\Big)\\
      &\leq C\nu\Big(\|e^{2\epsilon\nu^{\f13}t}\nabla\Delta u_{\neq}^2\|_{L^2L^2}+\|e^{2\epsilon\nu^{\f13}t}\Delta\omega_{\neq}^2\|_{L^2L^2}\Big)
      \big(\nu^{-1/6}(E^{\f14}_5E^{\f34}_3+E^{\f12}_5E^{\f12}_3)\big)\\
      &\leq C\nu^{5/6}\big(\nu^{-3/4}E_{3,0}+\nu^{-5/6}E_{3,1}\big)
  \big(E^{\f14}_5E^{\f34}_3+E^{\f12}_5E^{\f12}_3\big)\leq CE_3\big(E^{\f14}_5E^{\f34}_3+E^{\f12}_5E^{\f12}_3\big).
  \end{align*}

This completes the proof of the lemma.
  \end{proof}

\subsection{Interaction between nonzero modes}

\begin{Lemma}\label{lem:int-nn}
  It holds that
 \begin{align}
\label{lemma2.1-1}&\|e^{4\epsilon\nu^{\f13}t}|u_{\neq}|^2\|_{L^2L^2}^2+\|e^{4\epsilon\nu^{\f13}t}u_{\neq}\cdot\nabla u_{\neq}\|_{L^2L^2}^2+\|e^{4\epsilon\nu^{\f13}t}\partial_x(u_{\neq}\cdot\nabla u_{\neq})\|_{L^2L^2}^2\\
\nonumber&\quad+\|e^{4\epsilon\nu^{\f13}t}\partial_z(u_{\neq}\cdot\nabla u_{\neq}^3)\|_{L^2L^2}^2+\|e^{4\epsilon\nu^{\f13}t}\nabla(u_{\neq}\cdot\nabla u_{\neq}^2)\|_{L^2L^2}^2\leq C\nu^{-1}E_3^4,
\end{align}
and
\begin{align}
\label{lemma2.1-3}&\|e^{4\epsilon\nu^{\f13}t}\nabla(u_{\neq}\cdot\nabla u_{\neq})\|_{L^2L^2}^2\leq C\nu^{-\f53}E_3^4.
\end{align}
In particular, we have
\begin{align}
\label{lemma2.1-4}\|(\nu^{\f23}+\nu t)^{\f12}\nabla(u_{\neq}\cdot\nabla u_{\neq}^3)\|_{L^2L^2}^2\leq C\nu^{-1}E_3^4.
\end{align}
\end{Lemma}

\begin{proof}
By \eqref{f1} in Lemma \ref{lemma-A.4}, we have
\begin{align*}
\||u_{\neq}|^2\|_{L^2}\leq& C\big(\|\partial_xu_{\neq}\|_{H^1}+\|u_{\neq}\|_{H^1}\big)\big(\|\partial_zu_{\neq}\|_{L^2}+\|u_{\neq}\|_{L^2}\big)\\\leq& C\|\nabla\partial_x^2 u_{\neq}\|_{L^2} \|(\partial_x,\partial_z)\partial_xu_{\neq}\|_{L^2},
\end{align*}
and for $k\in\{1,3\}$, by \eqref{f2} in Lemma \ref{lemma-A.4} we have
\begin{align*}
\|u_{\neq}^k\partial_ku_{\neq}\|_{L^2}+\|\partial_x(u_{\neq}^k\partial_ku_{\neq})\|_{L^2}&\leq C\|(\partial_x\partial_zu_{\neq}^k,\partial_xu_{\neq}^k,\partial_zu_{\neq}^k,u_{\neq}^k)\|_{H^1}
\|(\partial_x\partial_ku_{\neq},\partial_ku_{\neq})\|_{L^2}\\
&\leq C\|\nabla(\partial_x,\partial_z)\partial_xu_{\neq}\|_{L^2}\|(\partial_x,\partial_z)\partial_xu_{\neq}\|_{L^2},
\end{align*}
and by \eqref{f2}  again, we have
\begin{align*}
\|\partial_z(u_{\neq}^k\partial_ku_{\neq}^3)\|_{L^2}\leq& C\|(\partial_x\partial_zu_{\neq}^k,\partial_xu_{\neq}^k,\partial_zu_{\neq}^k,u_{\neq}^k)\|_{H^1}
\|(\partial_z\partial_ku_{\neq}^3,\partial_ku_{\neq}^3)\|_{L^2}\\
\leq&C\|\nabla(\partial_x,\partial_z)\partial_xu_{\neq}\|_{L^2}\|(\partial_x^2+\partial_z^2)u^3_{\neq}\|_{L^2}.
\end{align*}
For $k=2$, by \eqref{f3} in Lemma \ref{lemma-A.4}, we have
\begin{align*}
&\|u_{\neq}^k\partial_ku_{\neq}\|_{L^2}+\|(\partial_x,\partial_z)(u_{\neq}^k\partial_ku_{\neq})\|_{L^2}\\ &\leq C\|(\partial_xu_{\neq}^k,\partial_zu_{\neq}^k,u_{\neq}^k) \|_{H^{1}}\|(\partial_x\partial_z\partial_ku_{\neq},\partial_x\partial_ku_{\neq},\partial_z\partial_ku_{\neq},
\partial_ku_{\neq})\|_{L^2}\\
&\leq C \|(\partial_x,\partial_z)\nabla u_{\neq}^2\|_{L^{2}}\|\nabla(\partial_x,\partial_z)\partial_xu_{\neq}\|_{L^2}.
\end{align*}
Summing up,  we get by Lemma \ref{lem:u-relation}  that
\begin{align*}
&\||u_{\neq}|^2\|_{L^2}+\|u_{\neq}\cdot\nabla u_{\neq}\|_{L^2}+\|\partial_x(u_{\neq}\cdot\nabla u_{\neq})\|_{L^2}+\|\partial_z(u_{\neq}\cdot\nabla u_{\neq}^3)\|_{L^2}\\
&\leq C \|\nabla(\partial_x,\partial_z)\partial_xu_{\neq}\|_{L^2}\big(\|(\partial_x,\partial_z)\partial_xu_{\neq}\|_{L^2}+
\|(\partial_x^2+\partial_z^2)u^3_{\neq}\|_{L^2}+\|(\partial_x,\partial_z)\nabla u_{\neq}^2\|_{L^{2}}\big)\\
&\leq C\Big(\|(\partial_x,\partial_z)\Delta u_{\neq}^2\|_{L^2}+\|(\partial_x^2+\partial_z^2)\nabla u_{\neq}^3\|_{L^2}) (\|(\partial_x,\partial_z)\nabla u_{\neq}^2\|_{L^{2}}+\|(\partial_x^2+\partial_z^2) u_{\neq}^3\|_{L^2}\Big).
\end{align*}

For $k\in\{1,3\}$, by \eqref{f4} in Lemma \ref{lemma-A.4}, we have
\begin{align*}
\|\nabla(u_{\neq}^k\partial_ku_{\neq}^2)\|_{L^2}\leq &C\|(\partial_x\partial_zu_{\neq}^k,\partial_xu_{\neq}^k,\partial_zu_{\neq}^k,u_{\neq}^k)\|_{H^1}\|\partial_ku_{\neq}^2\|_{H^1}\\
\leq &C\|\nabla(\partial_x,\partial_z)\partial_x u_{\neq}\|_{L^2}
\|(\partial_x,\partial_z)\nabla u_{\neq}^2\|_{L^2},
\end{align*}
and for $k=2$, by \eqref{f5} in Lemma \ref{lemma-A.4}, we have
\begin{align*}
&\|\nabla(u_{\neq}^k\partial_ku_{\neq}^2)\|_{L^2}\leq \|(\partial_zu_{\neq}^k,u_{\neq}^k)\|_{H^1}\|(\partial_x\partial_ku_{\neq}^2,\partial_ku_{\neq}^2)\|_{H^1}\leq C\|(\partial_x,\partial_z)\nabla u_{\neq}^2\|_{L^{2}}\|\partial_x\Delta u_{\neq}^2\|_{L^2}.
\end{align*}
Then it follows from Lemma \ref{lem:u-relation} that
\begin{align*}
\|\nabla(u_{\neq}\cdot\nabla u_{\neq}^2)\|_{L^2}\leq& C ( \|(\partial_x,\partial_z)\Delta u_{\neq}^2\|_{L^2}+\|(\partial_x^2+\partial_z^2)\nabla u_{\neq}^3\|_{L^2}) \|(\partial_x,\partial_z)\nabla u_{\neq}^2\|_{L^{2}}.
\end{align*}
This shows that
\begin{align*}
&\|e^{4\epsilon\nu^{\f13}t}|u_{\neq}|^2\|_{L^2L^2}^2+\|e^{4\epsilon\nu^{\f13}t}u_{\neq}\cdot\nabla u_{\neq}\|_{L^2L^2}^2+\|e^{4\epsilon\nu^{\f13}t}\partial_x(u_{\neq}\cdot\nabla u_{\neq})\|_{L^2L^2}^2\\
\nonumber&+\|e^{4\epsilon\nu^{\f13}t}\partial_z(u_{\neq}\cdot\nabla u_{\neq}^3)\|_{L^2L^2}^2+\|e^{4\epsilon\nu^{\f13}t}\nabla(u_{\neq}\cdot\nabla u_{\neq}^2)\|_{L^2L^2}^2\\
&\leq C\big( \|e^{2\epsilon\nu^{\f13}t}(\partial_x,\partial_z)\Delta u_{\neq}^2\|_{L^2L^2}^2+\|e^{2\epsilon\nu^{\f13}t}\nabla(\partial_x^2+\partial_z^2) u_{\neq}^3\|_{L^2L^2}^2\big)\big(\|e^{2\epsilon\nu^{\f13}t}(\partial_x,\partial_z)\nabla u_{\neq}^2\|_{L^{\infty}L^2}^2\\
&\quad+\|e^{2\epsilon\nu^{\f13}t}(\partial_x^2+\partial_z^2) u_{\neq}^3\|_{L^{\infty}L^2}^2\big)
\leq C \nu^{-1}E_3^4.
\end{align*}
This proves \eqref{lemma2.1-1}.

For $k\in\{1,3\}$, by \eqref{f6} in Lemma \ref{lemma-A.4} and Lemma \ref{lem:u-relation},  we have
\begin{align*}
\|\nabla(u_{\neq}^k\partial_k u_{\neq})\|_{L^2}^2
\leq& C\Big(\|(\partial_zu_{\neq}^k,u_{\neq}^k)\|_{H^2}^2\|(\partial_x\partial_k u_{\neq},\partial_k u_{\neq})\|_{L^2}^2\\&\quad+
\|(\partial_x\partial_zu_{\neq}^k,\partial_xu_{\neq}^k,\partial_zu_{\neq}^k,u_{\neq}^k)\|_{L^2}^2\|\partial_k u_{\neq}\|_{H^2}^2\Big)\\ \leq& C\Big(\|\Delta(\partial_x,\partial_z)u_{\neq}^k\|_{L^2}^2\|(\partial_x,\partial_z) \partial_x u_{\neq}\|_{L^2}^2+
\|(\partial_x,\partial_z) \partial_xu_{\neq}^k\|_{L^2}^2\|\Delta\partial_k u_{\neq}\|_{L^2}^2\Big)\\ \leq&C\Big(\|\nabla\Delta u_{\neq}^2\|_{L^2}^2+\|\Delta\omega_{\neq}^2\|_{L^2}^2)(\|(\partial_x^2+\partial_z^2 )u_{\neq}^3\|_{L^2}^2+\|(\partial_x,\partial_z)\nabla u_{\neq}^2\|_{L^2}^2\Big),
\end{align*}
and for $k=2$,  by \eqref{f5} in Lemma \ref{lemma-A.4} and  Lemma \ref{lem:u-relation}, we have
\begin{align*}
\|\nabla(u_{\neq}^k\partial_k u_{\neq})\|_{L^2}^2
\leq& C\|(\partial_zu_{\neq}^k,u_{\neq}^k)\|_{H^1}\|(\partial_x\partial_k u_{\neq},\partial_k u_{\neq})\|_{H^1}\\ \leq& C\|(\partial_x,\partial_z)\nabla u_{\neq}^2\|_{L^2}^2\|\Delta\partial_x u_{\neq}\|_{L^2}^2\\ \leq&C\|(\partial_x,\partial_z)\nabla u_{\neq}^2\|_{L^2}^2\big(\|\nabla\Delta u_{\neq}^2\|_{L^2}^2+\|\Delta\omega_{\neq}^2\|_{L^2}^2\big).
\end{align*}
Thus, we arrive at
\begin{align*}
&\|e^{4\epsilon\nu^{\f13}t}\nabla(u_{\neq}\cdot\nabla u_{\neq})\|_{L^2L^2}^2\\
&\leq C\Big(\|e^{2\epsilon\nu^{\f13}t}\nabla\Delta u_{\neq}^2\|_{L^2L^2}^2+\|e^{2\epsilon\nu^{\f13}t}\Delta\omega_{\neq}^2\|_{L^2L^2}^2\Big)
\Big(\|e^{2\epsilon\nu^{\f13}t}(\partial_x^2+\partial_z^2 )u_{\neq}^3\|_{L^{\infty}L^2}^2\\
&+\|e^{2\epsilon\nu^{\f13}t}(\partial_x,\partial_z)\nabla u_{\neq}^2\|_{L^{\infty}L^2}^2\Big)\\
\leq&C\big(\nu^{-\f32}E_{3,0}^2+\nu^{-\f53}E_{3,1}^2\big) E_3^2\leq C \nu^{-\f53}E_3^4,
\end{align*}
which gives  \eqref{lemma2.1-3}.  Finally, we have
\begin{align*}
\|(\nu^{\f23}+\nu t)^{\f12}\nabla(u_{\neq}\cdot\nabla u_{\neq}^3)\|_{L^2L^2}^2\leq &C \nu^{\f23}\|(1+\nu^{\f13} t)^{\f12}e^{-4\epsilon\nu^{\f13}t}e^{4\epsilon\nu^{\f13}t}\nabla(u_{\neq}\cdot\nabla u_{\neq}^3)\|_{L^2L^2}^2 \\
\leq& C \nu^{\f23}\|e^{4\epsilon\nu^{\f13}t}\nabla(u_{\neq}\cdot\nabla u_{\neq}^3)\|_{L^2L^2}^2 \leq C \nu^{-1}E_3^4,
\end{align*}
which gives \eqref{lemma2.1-4}.
\end{proof}

\subsection{Interaction between zero modes}

\begin{Lemma}\label{lem:int-zz1}
It holds that
\begin{align*}
\|\Delta(\bar{u}^2\partial_y\bar{u}^{1,0}+\bar{u}^3\partial_z\bar{u}^{1,0})\|_{L^2}^2\leq &C\Bigg\{\big(\|\Delta\bar{u}^2\|_{L^2}^2+\|\nabla\bar{u}^3\|_{L^2}^2\big)\|\nabla\Delta\bar{u}^{1,0}\|_{L^2}^2\\
&+\|\min((\nu^{\f23}+\nu t)^{\frac{1}{2}},1-y^2)\Delta\bar{u}^3\|_{L^2}^2 \left\|\f{\partial_z\bar{u}^{1,0}}{\min((\nu t)^{\frac{1}{2}},1-y^2)}\right\|_{L^{\infty}}^2\Bigg\}.
\end{align*}
\end{Lemma}

\begin{proof}
For $k\in\{2,3\}$, by Lemma \ref{Lem: bilinear zero nonzero}, Lemma \ref{lemma-A.1} and $\partial_y\bar{u}^2+\partial_z\bar{u}^3=0 $, we have
\begin{align*}
&\|\bar{u}^k\Delta\partial_k\bar{u}^{1,0}\|_{L^2}^2+\|\nabla\bar{u}^k\nabla\partial_k\bar{u}^{1,0}\|_{L^2}^2\\
&\leq\|\bar{u}^k\|_{L^{\infty}}^2\|\Delta\partial_k\bar{u}^{1,0}\|_{L^2}^2+
C(\|\partial_z\nabla\bar{u}^k\|_{L^2}^2+\|\nabla\bar{u}^k\|_{L^2}^2)\|\nabla\partial_k\bar{u}^{1,0}\|_{H^1}^2 \\
&\leq C\big(\|\partial_z\bar{u}^k\|_{H^1}^2+\|\bar{u}^k\|_{H^1}^2\big) \|\nabla\partial_k\bar{u}^{1,0}\|_{H^1}^2\leq C\big(\|\Delta\bar{u}^2\|_{L^2}^2+\|\nabla\bar{u}^k\|_{L^2}^2\big)\|\nabla\Delta\bar{u}^{1,0}\|_{L^2}^2.
\end{align*}
It is easy to see that
\begin{align*}
&\|\Delta\bar{u}^2\partial_y\bar{u}^{1,0}\|_{L^2}^2\leq \|\Delta\bar{u}^2\|_{L^2}^2\|\partial_y\bar{u}^{1,0}\|_{L^{\infty}}^2\leq C\|\Delta\bar{u}^2\|_{L^2}^2\|\bar{u}^{1,0}\|_{H^{3}}^2\leq  C\|\Delta\bar{u}^2\|_{L^2}^2\|\nabla\Delta \bar{u}^{1,0}\|_{L^2}^2,
\end{align*}
and
\begin{align*}
\|\Delta\bar{u}^3\partial_z\bar{u}^{1,0}\|_{L^2}^2\leq \|\min((\nu^{\f23}+\nu t)^{\frac{1}{2}},1-y^2)\Delta\bar{u}^3\|_{L^2}^2 \left\|\f{\partial_z\bar{u}^{1,0}}{\min((\nu t)^{\frac{1}{2}},1-y^2)}\right\|_{L^{\infty}}^2.
\end{align*}Summing up, we conclude the result.
\end{proof}

\begin{Lemma}\label{lem:int-zz2}
It holds that
\begin{align*}
\|\partial_t\nabla(\bar{u}^2\partial_y\bar{u}^{1,0})\|_{L^2L^2}^2+\|\partial_t\nabla(\bar{u}^3\partial_z\bar{u}^{1,0})\|_{L^2L^2}^2\leq C\nu E_2^2E_{1,0}^2.
\end{align*}
\end{Lemma}

\begin{proof} First of all, we have
\begin{align*}
&\|\partial_t\nabla(\bar{u}^2\partial_y\bar{u}^{1,0})\|_{L^2L^2}^2\\
&\leq C\Big(\|\partial_t\nabla\bar{u}^2\|_{L^2L^2}^2\|\partial_y\bar{u}^{1,0}\|_{L^{\infty}L^{\infty}}^2+\|\nabla\bar{u}^2\|_{L^2L^{\infty}}^2
\|\partial_t\partial_y\bar{u}^{1,0}\|_{L^{\infty}L^2}^2\\
&+\|\bar{u}^2\|_{L^{\infty}L^{\infty}}^2\|\partial_t\nabla\partial_y\bar{u}^{1,0}\|_{L^2L^2}^2+
\|\partial_t\bar{u}^2\|_{L^2L^2}^2\|\nabla\partial_y\bar{u}^{1,0}\|_{L^{\infty}L^{\infty}}^2\Big)\\
&\leq C\Big(\|\partial_t\nabla\bar{u}^2\|_{L^2L^2}^2\|\bar{u}^{1,0}\|_{L^{\infty}H^4}^2+\|\nabla\Delta\bar{u}^2\|_{L^2L^2}^2\|\partial_t\bar{u}^{1,0}\|_{L^{\infty}H^1}^2\\
&+\|\Delta\bar{u}^2\|_{L^{\infty}L^2}^2\|\partial_t  \bar{u}^{1,0}\|_{L^2H^2}^2+\|\partial_t\bar{u}^2\|_{L^2L^2}^2\|\bar{u}^{1,0}\|_{L^{\infty}H^{4}}^2\Big)\\
&\leq C\Big(\nu E_2^2E_{1,0}^2+\nu^{-1}E_2^2(\nu E_{1,0})^2+E_2^2(\nu E_{1,0}^2)+\nu E_2^2E_{1,0}^2\Big)\leq C\nu E_2^2E_{1,0}^2.
\end{align*}

By Lemma \ref{lem:u10-Linfty} and Lemma \ref{lem:u23-zero}, we get
\begin{align*}
&\|\partial_t\nabla(\bar{u}^3\partial_z\bar{u}^{1,0})\|_{L^2L^2}^2\\
&\leq C\Bigg(\|\min((\nu^{\frac{2}{3}}+\nu t)^{\frac{1}{2}},1-y^2)\partial_t\nabla\bar{u}^3\|_{L^2L^2}^2 \left\|\f{\partial_z\bar{u}^{1,0}}{\min((\nu t)^{\frac{1}{2}},1-y^2)}\right\|_{L^{\infty}L^{\infty}}^2\\
&\quad+\|\nabla\bar{u}^3\|_{L^2L^{\infty}}^2\|\partial_t\partial_z\bar{u}^{1,0}\|_{L^{\infty}L^2}^2+
\|\bar{u}^3\|_{L^{\infty}L^{\infty}}^2\|\partial_t\nabla\partial_z\bar{u}^{1,0}\|_{L^2L^2}^2+
\|\partial_t\bar{u}^3\|_{L^2L^2}^2\|\nabla\partial_z\bar{u}^{1,0}\|_{L^{\infty}L^{\infty}}^2\Bigg)\\
&\leq C\Big(\|\min((\nu^{\frac{2}{3}}+\nu t)^{\frac{1}{2}},1-y^2)\partial_t\nabla\bar{u}^3\|_{L^2L^2}^2E_{1,0}^2\\
&\qquad+\nu^{-1}E_2^2\|\partial_t\partial_z\bar{u}^{1,0}\|_{L^{\infty}L^2}^2+E_2^2\|\partial_t\Delta\bar{u}^{1,0}\|_{L^2L^2}^2+
\nu E_2^2\|\bar{u}^{1,0}\|_{L^{\infty}H^{4}}^2\Big)\\
&\leq C\big(\nu E_2^2E_{1,0}^2+\nu E_2^2E_{1,0}^2\big).
\end{align*}

This completes the proof of the lemma.
\end{proof}

\subsection{Interaction between zero mode and nonzero mode}

The following lemma gives the reaction between $\bar{u}^1$ and $u^2_{\neq},u^3_{\neq}$.

\begin{Lemma}\label{lem:int-zn-1-23}
  It holds that
  \begin{align*}
    &\|e^{2\epsilon\nu^{1/3}t}(\partial_x,\partial_z)
    \big(\bar{u}^1\partial_xu^3_{\neq}\big)\|^2_{L^2L^2}
    +\|e^{2\epsilon\nu^{1/3}t}(\partial_x,\partial_z )
    (\bar{u}^1\partial_xu^2_{\neq})\|^2_{L^2L^2}\\&\quad
    +\|e^{2\epsilon\nu^{1/3}t}\partial_x
    ((u^2_{\neq}\partial_y+u^3_{\neq}\partial_z)\bar{u}^1)\|^2_{L^2L^2}
    \leq C\nu E^2_1E_3E_5.
  \end{align*}
\end{Lemma}

\begin{proof}
For  $k\in\{2,3\}$,  by Lemma \ref{Lem: bilinear zero nonzero} and Lemma \ref{lem:u1-H2},  we get
   \begin{align*}
      & \|(\partial_x,\partial_z)
    \big(\bar{u}^1\partial_xu^k_{\neq}\big)\|^2_{L^2} +\|\partial_x
    (u^k_{\neq}\partial_k\bar{u}^1)\|^2_{L^2} \\
      &\leq C\big(\|\bar{u}^1\|^2_{H^1}+\|\nabla\bar{u}^1\|^2_{H^1}\big)\big(\|\partial_x(\partial_x,\partial_z)u^k_{\neq}\|^2_{L^2}+ \|\partial_xu^k_{\neq}\|^2_{L^2}+ \|\partial_z\partial_xu^k_{\neq}\|^2_{L^2})\big)\\
      &\leq C\|\bar{u}^1\|^2_{H^2} \|\partial_x(\partial_x,\partial_z)u^k_{\neq}\|^2_{L^2}
      \leq C\nu^{\f43}E^2_1(1+\nu^{\f13} t)^2 \|\partial_x(\partial_x,\partial_z)u^k_{\neq}\|^2_{L^2}.
   \end{align*}
Then, using Lemma \ref{lem:u-nonzero} and the fact that $(1+\nu^{\f13} t)
\leq Ce^{\f14\epsilon\nu^{1/3}t}$, we deduce that
\begin{align*}
   &\|e^{2\epsilon\nu^{1/3}t}(\partial_x,\partial_z)
    \big(\bar{u}^1\partial_xu^3_{\neq}\big)\|^2_{L^2L^2}
    +\|e^{2\epsilon\nu^{1/3}t}(\partial_x,\partial_z )
    (\bar{u}^1\partial_xu^2_{\neq})\|^2_{L^2L^2}\\&\quad
    +\|e^{2\epsilon\nu^{1/3}t}\partial_x
    ((u^2_{\neq}\partial_y+u^3_{\neq}\partial_z)\bar{u}^1)\|^2_{L^2L^2}\\
   &\leq C\nu^{\f43}E^2_1\|(1+\nu^{\f13} t)e^{2\epsilon\nu^{1/3}t}\partial_x(\partial_x,\partial_z)(u^2_{\neq},u^3_{\neq})\|_{L^2L^2}^2\\
   &\leq C\nu^{\f43}E^2_1\|e^{\f94\epsilon\nu^{1/3}t}\partial_x(\partial_x,\partial_z)(u^2_{\neq},u^3_{\neq})\|_{L^2L^2}^2\\
   &\leq C\nu^{\f43}E^2_1(\nu^{-\f13}E_5E_3)
   =C\nu E^2_1E_3E_5.
   \end{align*}
\end{proof}

The following Lemma describes the reaction between $\bar{u}^1$ and $u^1_{\neq}$.

\begin{Lemma}\label{lem:int-zn-11}
  It holds that
  \begin{align*}
     \|e^{2\epsilon\nu^{1/3}t}\partial_x
     (\bar{u}^1\partial_xu^1_{\neq})\|^2_{L^2L^2}
     \leq C \nu\big (E^2_1E^{\f32}_3E^{\f12}_5+  E^2_1E_3E_5\big).
  \end{align*}
\end{Lemma}

\begin{proof}
By Lemma \ref{lem:u1-H2}, we have
\begin{align*}
     \|\partial_x
     (\bar{u}^1\partial_xu^1_{\neq})\|^2_{L^2}
     &\leq \|\bar{u}^1\|^2_{L^{\infty}}\|\partial_x^2u^1_{\neq}\|^2_{L^2}\leq C \|\bar{u}^1\|^2_{H^{2}}\|\partial_x^2u^1_{\neq}\|^2_{L^2}\leq C \nu^{\f43}E^2_1(1+\nu^{\f13} t)^2\|\partial_x^2u^1_{\neq}\|^2_{L^2},
  \end{align*}
 which along with Lemma \ref{lem:u-nonzero}  gives
\begin{align*}
   \|e^{2\epsilon\nu^{1/3}t}\partial_x
     (\bar{u}^1\partial_xu^1_{\neq})\|^2_{L^2L^2}
   \leq& C\nu^{\f43}E^2_1\|(1+\nu^{\f13} t)e^{2\epsilon\nu^{1/3}t}\partial_x^2u^1_{\neq}\|_{L^2L^2}^2
   \leq C\nu^{\f43}E^2_1\|e^{\f94\epsilon\nu^{1/3}t}\partial_x^2u^1_{\neq}\|_{L^2L^2}^2\\
   \leq& C\nu^{\f43}E^2_1\nu^{-\f13}\big(E^{\f12}_5E^{\f32}_3+E_5E_3\big)
   =C\nu \big(E^2_1E^{\f32}_3E^{\f12}_5+  E^2_1E_3E_5\big).
   \end{align*}
\end{proof}

The following lemma gives the reactions between $\bar{u}^2$, $\bar{u}^3$ with nonzero modes.
This lemma suggests that $\bar{u}^2$ and $\bar{u}^3$ are good components.

\begin{Lemma}\label{lem:int-nz-23}
  It holds that for $k\in\{2,3\}$,
  \begin{align*}
     &\|e^{2\epsilon\nu^{1/3}t}(\partial_x,\partial_z) (\bar{u}^k\nabla
     u_{\neq})\|_{L^2L^2} +\|e^{2\epsilon\nu^{1/3}t}(\partial_x,\partial_z)
     (u_{\neq}\cdot\nabla\bar{u}^k)\|_{L^2L^2}
     \leq C\nu^{-1/2}E_2E_3.
  \end{align*}
\end{Lemma}

\begin{proof}
  By Lemma \ref{Lem: bilinear zero nonzero}, Lemma \ref{lem:u23-zero} and Lemma \ref{lem:u-relation}, we get
  \begin{align*}
     &\|(\partial_x,\partial_z) (\bar{u}^k\nabla
     u_{\neq})\|_{L^2} +\|(\partial_x,\partial_z)
     (u_{\neq}\cdot\nabla\bar{u}^k)\|_{L^2}\\
     &\leq C\big(\|\bar{u}^k\|_{H^1} +\|\partial_z\bar{u}^k\|_{H^1}\big)\big(\| \nabla
     u_{\neq}\|_{L^2}+\|(\partial_x,\partial_z) \nabla
     u_{\neq}\|_{L^2}\big)\\&\quad+C\big(\|\nabla\bar{u}^k\|_{L^2} +\|\partial_z\nabla\bar{u}^k\|_{L^2}\big)\big(\|
     u_{\neq}\|_{H^1}+\|(\partial_x,\partial_z)
     u_{\neq}\|_{H^1}\big)\\
     &\leq C\big(\|\bar{u}^k\|_{H^1} +\|\partial_z\bar{u}^k\|_{H^1}\big)\|\nabla(\partial_x,\partial_z) \partial_x
     u_{\neq}\|_{L^2} \\
     &\leq CE_2\big(\|(\partial_x,\partial_z)\Delta u_{\neq}^2\|_{L^2}+\|(\partial_x^2+\partial_z^2)\nabla u_{\neq}^3\|_{L^2}\big),
  \end{align*}
which gives
\begin{align*}
     &\|e^{2\epsilon\nu^{1/3}t}(\partial_x,\partial_z) (\bar{u}^k\nabla
     u_{\neq})\|_{L^2L^2} +\|e^{2\epsilon\nu^{1/3}t}(\partial_x,\partial_z)
     (u_{\neq}\cdot\nabla\bar{u}^k)\|_{L^2L^2}\\
     &\leq CE_2\Big( \|e^{2\epsilon\nu^{1/3}t}(\partial_x,\partial_z)\Delta u_{\neq}^2\|_{L^2L^2}+\|e^{2\epsilon\nu^{1/3}t}(\partial_x^2+\partial_z^2)\nabla u_{\neq}^3\|_{L^2L^2}\Big)\\
     &\leq C\nu^{-1/2}E_2E_3.
  \end{align*}
 \end{proof}

 The following lemma will be used to estimate $E_{3,1}$.

\begin{Lemma}\label{lem:int-zz-11-pz}
  It holds that
  \begin{align*}
    &\|e^{2\epsilon\nu^{1/3}t}\partial_z(\bar{u}^1\partial_xu^1_{\neq})\|^2_{L^2L^2}+
    \|e^{2\epsilon\nu^{1/3}t}\partial_z((u^2_{\neq}\partial_y+u^3_{\neq}\partial_z)\bar{u}^1)\|^2_{L^2L^2}\\
   & \leq C\nu^{1/3}E^2_1E_{3} \big(E_5 +E_3^{\f34}E^{\f14}_5\big).
  \end{align*}
\end{Lemma}

\begin{proof}
By Lemma \ref{Lem: bilinear zero nonzero} and Lemma \ref{lem:u1-H2}, we have
  \begin{align*}
     &\|\partial_z(\bar{u}^1\partial_xu^1_{\neq})\|^2_{L^2}+
     \|\partial_z((u^2_{\neq}\partial_y+u^3_{\neq}\partial_z)\bar{u}^1)\|^2_{L^2}\\ &\leq  C\big(\|\bar{u}^1\|_{H^1}^2+\|\partial_z\bar{u}^1\|_{H^1}^2\big)\big(\|\partial_xu^1_{\neq}\|^2_{L^2}+
     \|(\partial_x,\partial_z)\partial_xu^1_{\neq}\|^2_{L^2}\big)\\&\quad+
     C\big(\|\partial_y\bar{u}^1\|_{L^2}^2+\|\partial_z\partial_y\bar{u}^1\|_{L^2}^2\big)\big(\|u^2_{\neq}\|^2_{H^1}+
     \|(\partial_x,\partial_z)u^2_{\neq}\|^2_{H^1}\big)\\&\quad+
     C\big(\|\partial_z\bar{u}^1\|_{L^2}^2+\|\partial_z^2\bar{u}^1\|_{L^2}^2\big)\big(\|u^3_{\neq}\|^2_{H^1}+
     \|(\partial_x,\partial_z)u^3_{\neq}\|^2_{H^1}\big)\\ &\leq C\|\bar{u}^1\|_{H^2}^2\big(\|\nabla\partial_xu^1_{\neq}\|^2_{L^2}+\|\nabla(\partial_x,\partial_z)u^2_{\neq}\|^2_{L^2}
     +\|\nabla(\partial_x,\partial_z)u^3_{\neq}\|^2_{L^2}\big)\\
     &\leq \nu^{\f43}E^2_1(1+\nu^{\f13} t)^2(\|\nabla\partial_xu_{\neq}\|^2_{L^2}+\|\nabla\partial_z(u^2_{\neq},u^3_{\neq})\|^2_{L^2}),
  \end{align*}
  which along with Lemma \ref{lem:u-nonzero}  gives
   \begin{align*}
   &\|e^{2\epsilon\nu^{1/3}t}\partial_z(\bar{u}^1\partial_xu^1_{\neq})\|^2_{L^2L^2}+
    \|e^{2\epsilon\nu^{1/3}t}\partial_z((u^2_{\neq}\partial_y+u^3_{\neq}\partial_z)\bar{u}^1)\|^2_{L^2L^2}\\
   &\leq C\nu^{\f43}E^2_1\Big(\|e^{\f{17}{8}\epsilon\nu^{1/3}t}\nabla\partial_xu_{\neq}\|_{L^2L^2}^2
   +\|e^{\f{17}{8}\epsilon\nu^{1/3}t}\nabla\partial_z(u^2_{\neq},u^3_{\neq})\|_{L^2L^2}^2\Big)\\
   &\leq C\nu^{\f43}E^2_1\nu^{-1}\big(E^{\f14}_5E^{\f74}_3+E^{\f12}_5E^{\f32}_3\big)
   \leq C\nu^{\f13} E^2_1E_{3}\big(E_5 +E_3^{\f34}E^{\f14}_5\big).
   \end{align*}
\end{proof}

\section{Energy estimates for zero mode}

\subsection{Estimate of $E_1$}

\begin{Proposition}\label{prop:E1}
It holds that
\begin{align}
&E_{1,0}\leq C\nu^{-1}\big(\|\bar{u}(0)\|_{H^2}+E_2+E_2E_{1,0}\big),\label{eq:E10}\\
&E_{1,\neq}\leq C\big(\|\bar{u}(0)\|_{H^2}+\nu^{-1}E_2E_{1,\neq}+\nu^{-\f43}E_3^2\big).\label{eq:E1n}
\end{align}
\end{Proposition}

\begin{proof}
{\bf Step 1.} Estimate of $E_{1,0}$.\smallskip

Thanks to  \eqref{eq:u10}, we know that
\begin{align}
\label{eq:u10-D}
(\partial_t-\nu\Delta )\Delta\bar{u}^{1,0}+\Delta\bar{u}^2+\Delta(\bar{u}^2\partial_y\bar{u}^{1,0}+\bar{u}^3\partial_z\bar{u}^{1,0})=0
\end{align}
with $\Delta\bar{u}^{1,0}|_{y=\pm1}=\bar{u}^{1,0}|_{y=\pm1}=0$.

Taking the time derivative to \eqref{eq:u10-D}, and then taking $L^2$ inner product with $\partial_t\Delta\bar{u}^{1,0}$ to the resulting equation, we obtain
\begin{align*}
&\f12\f d {dt}\|\partial_t\Delta\bar{u}^{1,0}\|_{L^2}^2+\nu\|\partial_t\nabla\Delta\bar{u}^{1,0}\|_{L^2}^2-\big\langle\partial_t\nabla\bar{u}^2,
\partial_t\nabla\Delta\bar{u}^{1,0}\big\rangle\\
&\quad-\big\langle\partial_t\nabla(\bar{u}^2\partial_y\bar{u}^{1,0}+\bar{u}^3\partial_z\bar{u}^{1,0}),\partial_t\nabla\Delta\bar{u}^{1,0}\big\rangle=0,
\end{align*}
which implies that
\begin{align*}
\f d {dt}\|\partial_t\Delta\bar{u}^{1,0}\|_{L^2}^2+\nu\|\partial_t\nabla\Delta\bar{u}^{1,0}\|_{L^2}^2\leq C\nu^{-1}\Big(\|\partial_t\nabla\bar{u}^2\|_{L^2}^2+\|\partial_t\nabla(\bar{u}^2\partial_y\bar{u}^{1,0}+\bar{u}^3\partial_z\bar{u}^{1,0})\|_{L^2}^2\Big).
\end{align*}
This gives that
\begin{align*}
&\|\partial_t\Delta\bar{u}^{1,0}\|_{L^{\infty}L^2}^2+\nu\|\partial_t\nabla\Delta\bar{u}^{1,0}\|_{L^2L^2}^2\\
&\leq C\Big(\|\partial_t\Delta\bar{u}^{1,0}(0)\|_{L^2}^2+\nu^{-1}\|\partial_t\nabla\bar{u}^2\|_{L^2L^2}^2+\nu^{-1}\|\partial_t\nabla(\bar{u}^2\partial_y\bar{u}^{1,0}
+\bar{u}^3\partial_z\bar{u}^{1,0})\|_{L^2L^2}^2\Big),
\end{align*}
which along with $\partial_t\Delta\bar{u}^{1,0}|_{t=0}=-\Delta \bar{u}^2|_{t=0}$ and the definition of $E_2$ gives
\begin{align*}
&\nu^{-2}\|\partial_t\Delta\bar{u}^{1,0}\|_{L^{\infty}L^2}^2+\nu^{-1}\|\partial_t\nabla\Delta\bar{u}^{1,0}\|_{L^2L^2}^2\\
\nonumber&\leq C\nu^{-2}\Big(\|u(0)\|_{H^2}^2+E_2^2+\nu^{-1}\|\partial_t\nabla(\bar{u}^{\al}\partial_{\al}\bar{u}^{1,0})\|_{L^2L^2}^2\Big),
\end{align*}
from which and Lemma \ref{lem:int-zz2}, we infer that
\begin{align*}
\nu^{-2}\|\partial_t\Delta\bar{u}^{1,0}\|_{L^{\infty}L^2}^2+\nu^{-1}\|\partial_t\nabla\Delta\bar{u}^{1,0}\|_{L^2L^2}^2\leq C\nu^{-2}\big(\|u(0)\|_{H^2}^2+E_2^2+E_2^2E_{1,0}^2\big).
\end{align*}

Thanks to \eqref{eq:u10-D}, we get
\begin{align*}
\nu^2\|\Delta^2\bar{u}^{1,0}\|_{L^{\infty}L^2}^2\leq C\Big(\|\partial_t\Delta\bar{u}^{1,0}\|_{L^{\infty}L^2}^2+\|\Delta\bar{u}^2\|_{L^{\infty}L^2}^2+\|\Delta(\bar{u}^2\partial_y\bar{u}^{1,0}
+\bar{u}^3\partial_z\bar{u}^{1,0})\|_{L^{\infty}L^2}^2\Big),
\end{align*}
which along with Lemma \ref{lem:int-zz1} and Lemma \ref{lem:u10-Linfty}  gives
\begin{align*}
\nu^2\|\Delta^2\bar{u}^{1,0}\|_{L^{\infty}L^2}^2\leq& C\big(\|u(0)\|_{H^2}^2+E_2^2+E_2^2E_{1,0}^2+\|\Delta(\bar{u}^2\partial_y\bar{u}^{1,0}+\bar{u}^3\partial_z\bar{u}^{1,0})\|_{L^{\infty}L^2}^2\big)\\
\leq&C\Big( \|u(0)\|_{H^2}^2+E_2^2+E_2^2E_{1,0}^2+\big(\|\Delta\bar{u}^2\|_{L^{\infty}L^2}^2+\|\nabla\bar{u}^3\|_{L^{\infty}L^2}^2)\|\nabla\Delta\bar{u}^{1,0}\|_{L^{\infty}L^2}^2\\
&+\|\min((\nu^{\f23}+\nu t\big)^{\frac{1}{2}},1-y^2)\Delta\bar{u}^3\|_{L^{\infty}L^2}^2\Big\|\f{\partial_z\bar{u}^{1,0}}{\min((\nu t)^{\frac{1}{2}},1-y^2)}\Big\|_{L^{\infty}L^{\infty}}^2\Big)\\
\leq&C\big(\|u(0)\|_{H^2}^2+E_2^2+E_2^2E_{1,0}^2\big).
\end{align*}

Since $\Delta\bar{u}^{1,0}|_{y=\pm1}=\bar{u}^{1,0}|_{y=\pm1}=\partial_t\bar{u}^{1,0}|_{y=\pm1}=0$, we obtain
\begin{align*}
E_{1,0}&=\|\bar{u}^{1,0}\|_{L^{\infty}H^4}+\nu^{-1}\|\partial_t \bar{u}^{1,0}\|_{L^{\infty}H^2}+\nu^{-\f12}\|\partial_t\bar{u}^{1,0}\|_{L^2H^3}\\&\leq C\big(\|\Delta\bar{u}^{1,0}\|_{L^{\infty}H^2}+\nu^{-1}\|\partial_t \Delta\bar{u}^{1,0}\|_{L^{\infty}L^2}+\nu^{-\f12}\|\partial_t\Delta\bar{u}^{1,0}\|_{L^2H^1}\big)\\&\leq C\big(\|\Delta^2\bar{u}^{1,0}\|_{L^{\infty}L^2}+\nu^{-1}\|\partial_t \Delta\bar{u}^{1,0}\|_{L^{\infty}L^2}+\nu^{-\f12}\|\partial_t\nabla\Delta\bar{u}^{1,0}\|_{L^2L^2}\big)\\&\leq C\nu^{-1}\big(\|\bar{u}(0)\|_{H^2}+E_2+E_2E_{1,0}\big).
\end{align*}
This proves \eqref{eq:E10}.\smallskip

{\bf Step 2.} Estimate of $E_{1,\neq}$.\smallskip

Recall that $\bar{u}^{1,\neq}$ satisfies
\begin{align*}
\left\{
\begin{aligned}
&(\partial_t-\nu\Delta )\bar{u}^{1,\neq}+\bar{u}^2\partial_y\bar{u}^{1,\neq}+\bar{u}^3\partial_z\bar{u}^{1,\neq}+\overline{u_{\neq}\cdot\nabla u_{\neq}^1}=0,\\
&\bar{u}^{1,\neq}|_{t=0}=\bar{u}^1(0),\quad \Delta\bar{u}^{1,\neq}|_{y=\pm1}=0,\quad \bar{u}^{1,\neq}|_{y=\pm1}=0.
\end{aligned}
\right.
\end{align*}
Thus, we have
\begin{align*}
(\partial_t-\nu\Delta )\Delta\bar{u}^{1,\neq}+\Delta(\bar{u}^2\partial_y\bar{u}^{1,\neq}+\bar{u}^3\partial_z\bar{u}^{1,\neq}+
\overline{u_{\neq}\cdot\nabla u_{\neq}^1})=0.
\end{align*}
Thanks to $\Delta\bar{u}^{1,\neq}|_{y=\pm1}=0$,  the energy estimate gives
\begin{align*}
\f d {dt}\| \Delta\bar{u}^{1,\neq}\|_{L^2}^2+2\nu\| \nabla\Delta\bar{u}^{1,\neq}\|_{L^2}^2-2\Big\langle \nabla(\bar{u}^2\partial_2\bar{u}^{1,\neq}+\bar{u}^3\partial_3\bar{u}^{1,\neq}+
\overline{u_{\neq}\cdot\nabla u_{\neq}^1}), \nabla\Delta\bar{u}^{1,\neq}\Big\rangle=0,
\end{align*}
which gives
\begin{align*}
&\f d {dt}\| \Delta\bar{u}^{1,\neq}\|_{L^2}^2+\nu\| \nabla\Delta\bar{u}^{1,\neq}\|_{L^2}^2\leq C\nu^{-1}\Big(\|\nabla (\bar{u}^2\partial_y\bar{u}^{1,\neq}+\bar{u}^3\partial_z\bar{u}^{1,\neq})\|_{L^2}^2+\|\nabla (\overline{u_{\neq}\cdot\nabla u_{\neq}^1})\|_{L^2}^2\Big).
\end{align*}
It is easy to see that
\begin{align*}
\|\nabla(\bar{u}^k\partial_k\bar{u}^{1,\neq})\|_{L^2}^2\leq& C\|\bar{u}^k\|_{H^1}^2\|\partial_k\bar{u}^{1,\neq}\|_{H^2}^2\leq C \|\nabla\bar{u}^k\|_{L^2}^2 \|\nabla\Delta\bar{u}^{1,\neq}\|_{L^2}^2.
\end{align*}
which gives
\begin{align*}
&\|\nabla (\bar{u}^2\partial_y\bar{u}^{1,\neq}+\bar{u}^3\partial_z\bar{u}^{1,\neq})\|_{L^2}^2
\leq C\big(\|\nabla\bar{u}^2\|_{L^2}^2+\|\nabla\bar{u}^3\|_{L^2}^2\big)\|\nabla\Delta\bar{u}^{1,\neq}\|_{L^2}^2,
\end{align*}
which along with Lemma \ref{lem:int-nn} gives
\begin{align*}
&\| \Delta\bar{u}^{1,\neq}\|_{L^{\infty}L^2}^2+\nu\| \nabla\Delta\bar{u}^{1,\neq}\|_{L^2L^2}^2\\ &\leq \| u(0)\|_{H^2}^2+C\nu^{-1}\Big(\|\nabla (\bar{u}^2\partial_y\bar{u}^{1,\neq}+\bar{u}^3\partial_z\bar{u}^{1,\neq})\|_{L^2L^2}^2+\|\nabla (\overline{u_{\neq}\cdot\nabla u_{\neq}^1})\|_{L^2L^2}^2\Big)\\ &\leq \| u(0)\|_{H^2}^2+C\nu^{-1}E_2^2\|\nabla\Delta\bar{u}^{1,\neq}\|_{L^2L^2}^2+C\nu^{-\f83}E_3^4.
\end{align*}
As $\Delta\bar{u}^{1,\neq}|_{y=\pm1}=\bar{u}^{1,\neq}|_{y=\pm1}=0$, we obtain
\begin{align*}
E_{1,\neq}^2&=\big(\|\bar{u}^{1,\neq}\|_{L^{\infty}H^2}+\nu^{\f12}\|\nabla\bar{u}^{1,\neq}\|_{L^2H^2}\big)^2\\&\leq C\big(\| \Delta\bar{u}^{1,\neq}\|_{L^{\infty}L^2}^2+\nu\| \nabla\Delta\bar{u}^{1,\neq}\|_{L^2L^2}^2\big) \\&\leq C\big(\|u(0)\|_{H^2}^2+\nu^{-2}E_2^2E_{1,\neq}^2+\nu^{-\f83}E_3^4\big).
\end{align*}
This proves \eqref{eq:E1n}.
\end{proof}

\subsection{Estimate of $E_2$}

Let us assume that  $\nu\in(0,\nu_0], \nu_0, \epsilon\in(0,1), \nu_0^{2/3}\leq 4\epsilon<\epsilon_1$. Then $e^{\nu t}\leq e^{4\epsilon\nu^{1/3} t}$ for $t>0.$

\begin{Proposition}\label{prop:E2}
It holds that
\beno
E_2\le C\big(1+\nu^{-1}E_2\big)\big(\|u(0)\|_{H^2}+\nu^{-1}E_3^2\big).
\eeno
\end{Proposition}

The proposition is an immediate consequence of the following lemmas.

\begin{Lemma}\label{lem:u23-H1}
It holds that
\begin{align*}
&\|e^{\nu t}(\bar{u}^2,\bar{u}^3)\|_{L^{\infty}L^2}^2+\nu\|e^{\nu t}(\nabla\bar{u}^2,\nabla\bar{u}^3)\|_{L^2L^2}^2\leq C\big(\|u(0)\|_{H^2}^2+\nu^{-2}E_3^4\big),\\
&\|e^{\nu t}(\nabla\bar{u}^2,\nabla\bar{u}^3)\|_{L^{\infty}L^2}^2+\nu^{-1}\|e^{\nu t}(\partial_t\bar{u}^2,\partial_t\bar{u}^3)\|_{L^2L^2}^2\\
&\qquad\leq C\big(1+\nu^{-2}E_2^2\big)\big(\|u(0)\|_{H^2}^2+\nu^{-2}E_3^4\big).
\end{align*}
\end{Lemma}

\begin{proof}
Recall that $\bar{u}^k(k=2,3)$ satisfies
\begin{align*}
(\partial_t-\nu\Delta )\bar{u}^k+\partial_k \bar{p}+(\bar{u}^2\partial_y+\bar{u}^3\partial_z)\bar{u}^k+\overline{u_{\neq}\cdot\nabla u^k_{\neq}}=0.
\end{align*}
As $\bar{u}^j|_{y=\pm1}=0$, $L^2$ energy estimate gives
\begin{align*}
&\f d {dt}\big(\|\bar{u}^2\|_{L^2}^2+\|\bar{u}^3\|_{L^2}^2\big)+2\nu\big(\|\nabla\bar{u}^2\|_{L^2}^2+\|\nabla\bar{u}^3\|_{L^2}^2\big)\\
&=2\langle \bar{p},\partial_y\bar{u}^2\rangle+2\langle \bar{p},\partial_z\bar{u}^3\rangle-2\sum\limits_{k=2,3}\big\langle(\bar{u}^2\partial_y+\bar{u}^3\partial_z)\bar{u}^k,\bar{u}^k\big\rangle\\&\quad-2\big\langle \overline{u_{\neq}\cdot\nabla u^2_{\neq}},\bar{u}^2\rangle-2\big\langle \overline{u_{\neq}\cdot\nabla u^3_{\neq}},\bar{u}^3\big\rangle.
\end{align*}
As $\partial_y\bar{u}^2+\partial_z\bar{u}^3=0$, we have
\begin{align*}
&\f d {dt}\big(\|\bar{u}^2\|_{L^2}^2+\|\bar{u}^3\|_{L^2}^2\big)+2\nu\big(\|\nabla\bar{u}^2\|_{L^2}^2+\|\nabla\bar{u}^3\|_{L^2}^2\big)\\
&=-2\sum\limits_{k=2,3}\langle\overline{u_{\neq}\cdot\nabla u^k_{\neq}},\bar{u}^k\rangle
=2\sum\limits_{k=2,3}\langle u_{\neq}\cdot\nabla \bar{u}^k,u^k_{\neq}\rangle\\
&\leq2\||u_{\neq}|^2\|_{L^2}\|(\nabla\bar{u}_2,\nabla\bar{u}_3)\|_{L^2},
\end{align*}
from which and the fact that $\|\nabla\bar{u}^j\|_{L^2}^2\geq(\pi/2)^2\|\bar{u}^j\|_{L^2}^2$, we infer that
\begin{align*}
&e^{2\nu t}\big(\|\bar{u}^2\|_{L^2}^2+\|\bar{u}^3\|_{L^2}^2\big)+\nu\int_0^te^{2\nu s}\big(\|\nabla\bar{u}^2(s)\|_{L^2}^2+\|\nabla\bar{u}^3(s)\|_{L^2}^2\big)ds\\
&\le C\|\bar{u}^2(0)\|_{L^2}^2+C\|\bar{u}^3(0)\|_{L^2}^2+C\nu^{-1}\int_0^te^{2\nu s}\||u_{\neq}(s)|^2\|_{L^2}^2ds,
\end{align*}
which  along with Lemma \ref{lem:int-nn} gives
\begin{align*}
&\|e^{\nu t}(\bar{u}^2,\bar{u}^3)\|_{L^{\infty}L^2}^2+\nu\|e^{\nu t}(\nabla\bar{u}^2,\nabla\bar{u}^3)\|_{L^2L^2}^2\\
&\leq C\big(\|u(0)\|_{L^2}^2+\nu^{-1}\|e^{\nu t}|u_{\neq}|^2\|_{L^2L^2}^2\big)\leq C\big(\|u(0)\|_{L^2}^2+\nu^{-2}E_3^4\big).
\end{align*}

Now $H^1$ energy estimate gives
\begin{align*}
&\nu\f d {dt}\big(\|\nabla\bar{u}^2\|_{L^2}^2+\|\nabla\bar{u}^3\|_{L^2}^2\big)+2\big(\|\partial_t\bar{u}^2\|_{L^2}^2+\|\partial_t\bar{u}^3\|_{L^2}^2\big)\\
&=-2\sum\limits_{k=2,3}\big\langle(\bar{u}^2\partial_y+\bar{u}^3\partial_z)\bar{u}^k,\partial_t\bar{u}^k\big\rangle-2\big\langle \overline{u_{\neq}\cdot\nabla u^2_{\neq}},\partial_t\bar{u}^2\big\rangle-2\big\langle \overline{u_{\neq}\cdot\nabla u^3_{\neq}},\partial_t\bar{u}^3\big\rangle.
\end{align*}
By Lemma \ref{lem:u23-zero}, we have
\begin{align*}
\|e^{\nu t}(\bar{u}^2\partial_y+\bar{u}^3\partial_z)\bar{u}^k\|_{L^2}\leq CE_2\|e^{\nu t}\nabla\bar{u}^k\|_{L^2},
\end{align*}
and by Lemma \ref{lem:int-nn}, we have
\beno
\|e^{\nu t}(u_{\neq}\cdot \nabla u_{\neq})\|_{L^2L^2}^2\le C\nu^{-1}E_3^4.
\eeno
Thus, we obtain
\begin{align*}
&\|e^{\nu t}(\nabla\bar{u}^2,\nabla\bar{u}^3)\|_{L^{\infty}L^2}^2+\nu^{-1} \|e^{\nu t}(\partial_t\bar{u}^2,\partial_t\bar{u}^3)\|_{L^2L^2}^2\\
&\leq C\|u(0)\|_{H^2}^2+C\nu \|e^{\nu t}(\nabla\bar{u}^2,\nabla\bar{u}^3)\|_{L^2L^2}^2+C\nu^{-2}E_3^4\\
&\quad+C\nu^{-1}E_2^2 \|e^{\nu t}(\nabla\bar{u}^2,\nabla\bar{u}^3)\|_{L^2L^2}^2\leq C\big(1+\nu^{-2}E_2^2\big)\big(\|u(0)\|_{H^2}^2+\nu^{-2}E_3^4\big).
\end{align*}
\end{proof}

\begin{Lemma}\label{lem:u2-H2}
It holds that
\begin{align*}
\|e^{\nu t}\Delta\bar{u}^2\|_{L^{\infty}L^2}^2+\nu^{-1}\|e^{\nu t}\nabla\partial_t\bar{u}^2\|_{L^2L^2}^2\leq C\big(1+\nu^{-2}E_2^2\big)\big(\|u(0)\|_{H^2}^2+\nu^{-2}E_3^4\big).
\end{align*}
\end{Lemma}

\begin{proof}
Recall that $\Delta \bar u^2$ satisfies
\begin{align}
\label{2.14.1}(\partial_t-\nu\Delta )\Delta\bar{u}^2+\Delta\partial_y \bar{p}+\Delta(\overline{u\cdot\nabla u^2})=0,\quad \nabla\bar{u}^2|_{y=\pm1}=0.
\end{align}
Taking the $L^2$ inner product with $-2\bar{u}^2$, we get
\begin{align*}
\f d {dt}\|\nabla\bar{u}^2\|_{L^2}^2+2\nu\|\Delta\bar{u}^2\|_{L^2}^2+2\langle\Delta \bar{p},\partial_y\bar{u}^2\rangle-2\big\langle\bar{u}^2\partial_y\bar{u}^2+\bar{u}^3\partial_z\bar{u}^2+\overline{u_{\neq}\cdot\nabla u^2_{\neq}},\Delta\bar{u}^2\big\rangle=0,
\end{align*}
which  gives
\begin{align}
\nu\|e^{\nu t}\Delta\bar{u}^2\|_{L^2L^2}^2\lesssim&\|u(0)\|_{H^2}^2+ \nu^{-1}\|e^{\nu t}\Delta \bar{p}\|_{L^2L^2}^2+\nu\| e^{\nu t}\nabla\bar{u}^2\|_{L^2L^2}^2\nonumber\\
&+\nu^{-1}\|e^{\nu t}(\bar{u}^2\partial_y+\bar{u}^3\partial_z)\bar{u}^2\|_{L^2L^2}^2+\nu^{-1}\|e^{\nu t} \overline{u_{\neq}\cdot \nabla u^2_{\neq}}\|_{L^2L^2}^2.\label{eq:Du2-est1}
\end{align}
Now by Lemma \ref{lemma-A.1} and Lemma \ref{lem:u23-H1}, we have
\begin{align*}
\|e^{\nu t}\nabla(\bar{u}^k\partial_k \bar{u}^j)\|_{L^2L^2}^2\leq& \|e^{\nu t}\bar{u}^k\|_{L^{\infty}H^1}^2\|(\partial_k \bar{u}^j,\partial_z\partial_k \bar{u}^j)\|_{L^2H^1}^2\\
\leq& C \|e^{\nu t}(\nabla\bar{u}^2,\nabla\bar{u}^3)\|_{L^{\infty}L^2}^2\|(\Delta\bar{u}^2,\Delta\bar{u}^3,\nabla\Delta\bar{u}^2)\|_{L^2L^{2}}^2\\ \leq& C\big(\|u(0)\|_{H^2}^2+\nu^{-2}E_3^4\big)\nu^{-1}E_2^2.
\end{align*}
This shows that
\begin{align}\label{eq:zz-est}
\nu^{-1}\|e^{\nu t}\nabla(\bar{u}\cdot\nabla \bar{u}^k)\|_{L^2L^2}^2\leq C\nu^{-2}E_2^2\big(\|u(0)\|_{H^2}^2+\nu^{-2}E_3^4\big),\quad k\in\{2,3\}.
\end{align}

Notice that
\begin{align*}
\Delta \bar{p}=-\overline{\partial_ju^i\partial_iu^j}=-\partial_j\bar{u}^i\partial_i\bar{u}^j-\overline{\partial_ju^i_{\neq}\partial_iu^j_{\neq}},
\end{align*}
where by Lemma \ref{lem:int-nn} and $\text{div}u_{\neq}=0$, we have
\begin{align*}
\|e^{\nu t}\overline{\partial_ju^i_{\neq}\partial_iu^j_{\neq}}\|_{L^2L^2}^2=\|e^{\nu t}\overline{\partial_j(u_{\neq}\cdot\nabla u^j_{\neq}})\|_{L^2L^2}^2\leq \|e^{\nu t}\partial_j(u_{\neq}\cdot\nabla u^j_{\neq})\|_{L^2L^2}^2\leq C\nu^{-1}E_3^4,
\end{align*}
and
\begin{align*}
&\| e^{\nu t}\partial_j\bar{u}^i\partial_i\bar{u}^j\|_{L^2L^2}^2=\| e^{\nu t}\partial_{\alpha}(\bar{u}\cdot\nabla \bar{u}^{\alpha})\|_{L^2L^2}^2\leq C\big(\|u(0)\|_{H^2}^2+\nu^{-2}E_3^4\big)\nu^{-1}E_2^2.
\end{align*}
This shows that
\begin{align}\label{eq:Dp-est}
\nu^{-1}\|e^{\nu t}\Delta \bar{p}\|_{L^2L^2}^2\leq C(1+\nu^{-2}E_2^2)\big(\|u(0)\|_{H^2}^2+\nu^{-2}E_3^4\big).
\end{align}
Then it follows from \eqref{eq:Du2-est1}, \eqref{eq:zz-est}, \eqref{eq:Dp-est} and Lemma \ref{lem:int-nn} that
\begin{align}\label{eq:Du2-est2}
\nu\|e^{\nu t}\Delta\bar{u}^2\|_{L^2L^2}^2\le C\big(1+\nu^{-2}E_2^2\big)\big(\|u(0)\|_{H^2}^2+\nu^{-2}E_3^4\big).
\end{align}

Next, we take the $L^2$ inner product  with $-2\partial_t\bar{u}^2$ to \eqref{2.14.1} to obtain
\begin{align*}
\|\partial_t\nabla\bar{u}^2\|_{L^2}^2+2\nu\f d {dt}\|\Delta\bar{u}^2\|_{L^2}^2+
2\langle\Delta\bar{p},\partial_t\partial_y\bar{u}^2\rangle+2\big\langle\nabla(\overline{u\cdot\nabla u^2}),\partial_t\nabla\bar{u}^2\big\rangle=0,
\end{align*}
which gives
\begin{align*}
\|\partial_t\nabla\bar{u}^2\|_{L^2}^2+\nu\partial_t\|\Delta\bar{u}^2\|_{L^2}^2\leq C\big(\|\Delta \bar{p}\|_{L^2}^2+\|\nabla(\overline{u\cdot\nabla u^2})\|_{L^2}^2\big),
\end{align*}
and then
\begin{align*}
&\nu^{-1}e^{2\nu t}\|\partial_t\nabla\bar{u}^2\|_{L^2}^2+\f d {dt}\big(e^{2\nu t}\|\Delta\bar{u}^2\|_{L^2}^2\big)\\
&\leq C\nu^{-1}\Big(e^{2\nu t}\|\Delta \bar{p}\|_{L^2}^2+e^{2\nu t}\|\nabla(\overline{u\cdot\nabla u^2})\|_{L^2}^2+\nu^2e^{2\nu t}\|\Delta\bar{u}^2\|_{L^2}^2\Big),
\end{align*}
which along with \eqref{eq:Du2-est2}, \eqref{eq:zz-est} and \eqref{eq:Dp-est} gives
\begin{align*}
&\|e^{\nu t}\Delta\bar{u}^2\|_{L^{\infty}L^2}^2+\nu^{-1}\|e^{\nu t}\partial_t\nabla\bar{u}^2\|_{L^2L^2}^2 \\
&\leq C(\|u(0)\|_{H^2}^2+\nu^{-1} \|e^{\nu t}\Delta \bar{p}\|_{L^2L^2}^2+\nu^{-1}\|e^{\nu t}\nabla(\overline{u\cdot\nabla u^2})\|_{L^2L^2}^2+\nu \|e^{\nu t}\Delta\bar{u}^2\|_{L^2L^2}^2)\\
&\le C\big(1+\nu^{-2}E_2^2\big)\big(\|u(0)\|_{H^2}^2+\nu^{-2}E_3^4\big).
\end{align*}
\end{proof}

\begin{Lemma}\label{lem:u2-H3}
It holds that
\begin{align*}
&\nu\|e^{\nu t}\nabla\Delta\bar{u}^2\|_{L^2L^2}^2+
\nu\|e^{\nu t}\Delta\bar{u}^3\|_{L^2L^2}^2\leq C\big(1+\nu^{-2}E_2^2\big)\big(\|u(0)\|_{H^2}^2+\nu^{-2}E_3^4\big).
\end{align*}
\end{Lemma}

\begin{proof}
Using the equation
\begin{align*}
(\partial_t-\nu\Delta )\partial_z\bar{u}^2+\partial_z\partial_y \bar{p}+\partial_z(\overline{u\cdot\nabla u^2})=0,\quad \nabla\bar{u}^2|_{y=\pm1}=0,
\end{align*}
we get by integration by parts that
\begin{align*}
&\|\partial_t \partial_z\bar{u}^2+\partial_z(\overline{u\cdot\nabla u^2})\|_{L^2}^2=\| \nu\Delta \partial_z\bar{u}^2-\partial_z\partial_y \bar{p}\|_{L^2}^2\\
&= \nu^2\| \Delta \partial_z\bar{u}^2\|_{L^2}^2+\|\partial_z\partial_y \bar{p}\|_{L^2}^2-2\nu\big\langle \Delta \partial_z\bar{u}^2, \partial_z\partial_y \bar{p}\big\rangle\\
&= \nu^2\| \Delta \partial_z\bar{u}^2\|_{L^2}^2+\|\partial_z\partial_y \bar{p}\|_{L^2}^2+2\nu\big\langle \nabla \partial_z\bar{u}^2, \nabla \partial_z\partial_y \bar{p}\big\rangle\\
&=\nu^2\| \Delta \partial_z\bar{u}^2\|_{L^2}^2+\|\partial_z\partial_y \bar{p}\|_{L^2}^2-2\nu\big\langle  \partial_z^2\partial_y\bar{u}^2, \Delta \bar{p}\big\rangle,
\end{align*}
which shows that
\begin{align}\label{eq:u2p-H3-est}
\nu^2\| \Delta \partial_z\bar{u}^2\|_{L^2}^2+\|\partial_z\partial_y \bar{p}\|_{L^2}^2\leq C\big(\| \Delta \bar{p} \|_{L^2}^2+\|\partial_t \partial_z\bar{u}^2+\partial_z(\overline{u\cdot\nabla u^2})\|_{L^2}^2\big).
\end{align}
Then by Lemma \ref{lem:u2-H2}, \eqref{eq:zz-est} and Lemma \ref{lem:int-nn}, we get
\begin{align}
&\nu\| e^{\nu t}\Delta \partial_z\bar{u}^2\|_{L^2L^2}^2+\nu^{-1}\|e^{\nu t}\partial_z\partial_y \bar{p}\|_{L^2L^2}^2\nonumber\\
&\leq C\nu^{-1}\Big(\| e^{\nu t}\Delta \bar{p} \|_{L^2L^2}^2+\|e^{\nu t}\partial_t \partial_z\bar{u}^2\|_{L^2L^2}^2+\|e^{\nu t}\partial_z(\bar{u}\cdot\nabla \bar{u}^2)\|_{L^2L^2}^2+\|e^{\nu t}\partial_z(\overline{u_{\neq}\cdot\nabla u^2_{\neq}})\|_{L^2L^2}^2\Big)\nonumber \\
&\leq C\big(1+\nu^{-2}E_2^2\big)\big(\|u(0)\|_{H^2}^2+\nu^{-2}E_3^4\big).\nonumber
\end{align}
Thanks to Lemma \ref{the estimate of delta bar{p}}, we find that
\begin{align*}
\|\partial_y^2\bar{p}\|_{L^2}^2+\|\partial_z^2\bar{p}\|_{L^2}^2\leq C\big(\|\partial_z\partial_y\bar{p}\|_{L^2}^2+\| \Delta \bar{p} \|_{L^2}^2\big),
\end{align*}
which gives
\begin{align}\label{eq:p-H2}
\nu^{-1}\big(\|e^{\nu t}\partial_y^2\bar{p}\|_{L^2L^2}^2+\|e^{\nu t}\partial_z^2\bar{p}\|_{L^2L^2}^2\big)\leq& C\nu^{-1}\big(\|e^{\nu t}\partial_z\partial_y\bar{p}\|_{L^2L^2}^2+\| e^{\nu t}\Delta \bar{p} \|_{L^2L^2}^2\big)\nonumber\\
\leq& C\big(1+\nu^{-2}E_2^2\big)\big(\|u(0)\|_{H^2}^2+\nu^{-2}E_3^4\big).
\end{align}

Now using the equation
\begin{align*}
(\partial_t-\nu\Delta )\partial_y\bar{u}^2+\partial_y^2 \bar{p}+\partial_y(\bar{u}^2\partial_y\bar{u}^2+\bar{u}^3\partial_z\bar{u}^2)+\partial_y(\overline{u_{\neq}\cdot\nabla u^2_{\neq}})=0,
\end{align*}
we deduce that
\begin{align*}
\nu\|\Delta \partial_y\bar{u}^2\|_{L^2}^2\leq C\nu^{-1}\Big(\|\partial_t\partial_y\bar{u}^2\|_{L^2}^2+\|\partial_y^2 \bar{p}\|_{L^2}^2+\|\partial_y(\bar{u}\cdot\nabla\bar{u}^2)\|_{L^2}^2+\|\partial_y(\overline{u_{\neq}\cdot\nabla u^2_{\neq}})\|_{L^2}^2\Big).
\end{align*}
Thus, we have
\begin{align*}
\nu\|e^{\nu t}\partial_y\Delta \bar{u}^2\|_{L^2L^2}^2\leq C\big(1+\nu^{-2}E_2^2\big)\big(\|u(0)\|_{H^2}^2+\nu^{-2}E_3^4\big).
\end{align*}

Using the fact that $\|\partial_z \bar{p}\|_{L^2}\leq \|\partial_z^2 \bar{p}\|_{L^2}$ and the equation $(\partial_t-\nu\Delta)\bar{u}^3+\partial_z\bar{p}+\overline{u\cdot\nabla u^3}=0$, we deduce that
\begin{align*}
\nu\|\Delta\bar{u}^3\|_{L^2}^2\leq\nu^{-1}\Big(\|\partial_t\bar{u}^3\|_{L^2}^2+\|\partial_z \bar{p}\|_{L^2}^2+\|(\bar{u}^2\partial_y\bar{u}^3+\bar{u}^3\partial_z\bar{u}^3)\|_{L^2}^2+\|\overline{u_{\neq}\cdot\nabla u_{\neq}^3}\|_{L^2}^2\Big).
\end{align*}
As above, we can obtain
\begin{align*}
\nu\|e^{\varepsilon\nu t}\Delta\bar{u}^3\|_{L^2L^2}^2\leq C\big(1+\nu^{-2}E_2^2\big)\big(\|u(0)\|_{H^2}^2+\nu^{-2}E_3^4\big).
\end{align*}
\end{proof}

\begin{Lemma}\label{lem:u3-H2-w}We assume that $E_2\leq\varepsilon_0\nu$, then $\exists$ $C>0$ such that
\begin{align*}
&\|\min((\nu^{\frac{2}{3}}+\nu t)^{\frac{1}{2}},1-y^2)\Delta\bar{u}^3\|_{L^{\infty}L^2}^2+\nu^{-1}
\|\min((\nu^{\frac{2}{3}}+\nu t)^{\frac{1}{2}},1-y^2)\partial_t\nabla \bar{u}_3\|_{L^2L^2}^2\\
\nonumber&
\quad+\nu\|\min((\nu^{\frac{2}{3}}+\nu t)^{\frac{1}{2}},1-y^2)\nabla\Delta\bar{u}^3\|_{L^2L^2}^2\leq C\big(1+\nu^{-2}E_2^2\big)\big(\|u(0)\|_{H^2}^2+\nu^{-2}E_3^4\big).
\end{align*}
\end{Lemma}

\begin{proof}
Recall that $\nabla\bar{u}^3$ satisfies
\begin{align}
\label{equation-2.22}(\partial_t-\nu\Delta )\nabla\bar{u}^3+\nabla\partial_z \bar{p}+\nabla(\overline{u\cdot\nabla u^3})=0.
\end{align}
For a smooth function $ \rho(t,y)$ satisfying $\rho|_{y=\pm 1}=0,\ 0\leq \rho\leq 1$, we have
\begin{align*}
&\|\rho(\partial_t-\nu\Delta )\nabla\bar{u}^3\|_{L^2}^2\\
&=
\|\rho\partial_t\nabla\bar{u}^3\|_{L^2}^2+\nu^2\|\rho\nabla\Delta\bar{u}^3\|_{L^2}^2-2\nu\langle\rho^2\partial_t\nabla\bar{u}^3,\nabla\Delta\bar{u}^3\rangle\\
&=\|\rho\partial_t\nabla\bar{u}^3\|_{L^2}^2+\nu^2\|\rho\nabla\Delta\bar{u}^3\|_{L^2}^2+
2\nu\big\langle\nabla\cdot[\rho^2\partial_t\nabla\bar{u}^3],\Delta\bar{u}^3\big\rangle\\
&=\|\rho\partial_t\nabla\bar{u}^3\|_{L^2}^2+\nu^2\|\rho\nabla\Delta\bar{u}^3\|_{L^2}^2+2\nu\big\langle\rho^2\partial_t\Delta\bar{u}^3, \Delta\bar{u}^3\big\rangle+4\nu\big\langle\rho\partial_y\rho\partial_t\partial_y\bar{u}^3,\Delta\bar{u}^3\big\rangle,
\end{align*}
which gives
\begin{align*}
&\|\rho(\partial_t-\nu\Delta )\nabla\bar{u}^3\|_{L^2}^2=\|\rho\partial_t\nabla\bar{u}^3\|_{L^2}^2+\nu^2\|\rho\nabla\Delta\bar{u}^3\|_{L^2}^2\\
&\quad+\nu\partial_t\|\rho\Delta\bar{u}^3\|_{L^2}^2-2\nu\big\langle(\rho\partial_t\rho)\Delta\bar{u}^3, \Delta\bar{u}^3\big\rangle+4\nu\big\langle\rho\partial_y\rho\partial_t\partial_y\bar{u}^3,\Delta\bar{u}^3\big\rangle,
\end{align*}
and then
\begin{align}\label{u3}
&\nu^{-1}\|\rho\partial_t\nabla\bar{u}^3\|_{L^2L^2}^2+\nu\|\rho\nabla\Delta\bar{u}^3\|_{L^2L^2}^2+\|\rho\Delta\bar{u}^3\|_{L^{\infty}L^2}^2\nonumber\\
\nonumber &\leq C\Big(\|{u}(0)\|_{H^2}^2+\nu^{-1}\|\rho(\partial_t-\nu\Delta )\nabla\bar{u}^3\|_{L^2L^2}^2\\
&\qquad\quad+\nu\big(\|\nu^{-1}|\rho\partial_t\rho|+|\partial_y\rho|^2\|_{L^{\infty}L^{\infty}}\big)
\|\Delta\bar{u}^3\|_{L^2L^2}^2\Big).
\end{align}

Now we take
\beno
\chi(x)=1-e^{-x},\quad \Psi(s,y)=s\chi\Big(\frac{1-y^2}{s}\Big),\quad \rho(t,y)=\Psi\big((\nu^{\frac{2}{3}}+\nu t)^{\frac{1}{2}},y\big).
\eeno
 Then we find that
 \beno
&&\chi(x)\sim\min(1,x),\quad |\chi'(x)|+|x\chi'(x)|\leq C\quad x\geq 0,\\
&&\Psi(t,y)\sim\min(t,1-y^2),\quad |\partial_s\Psi(s,y)|+|\partial_y\Psi(s,y)|\le C\quad s>0, y\in [-1,1].
\eeno
Thus, for $t\geq 0,\ y\in[-1,1],$  we have
\begin{align*}
&\rho(t,y)\sim\min\big((\nu^{\frac{2}{3}}+\nu t)^{\frac{1}{2}},1-y^2\big),
\end{align*}
and
\begin{align*}
&2\nu^{-1}|\rho\partial_t\rho|(t,y)+|\partial_y\rho(t,y)|^2\\ \nonumber &=(\nu^{\frac{2}{3}}+\nu t)^{-\frac{1}{2}}\cdot|\Psi\partial_s\Psi|((\nu^{\frac{2}{3}}+\nu t)^{\frac{1}{2}},y)+|\partial_y\Psi((\nu^{\frac{2}{3}}+\nu t)^{\frac{1}{2}},y)|^2\\ \nonumber &\leq C(\nu^{\frac{2}{3}}+\nu t)^{-\frac{1}{2}}\min((\nu^{\frac{2}{3}}+\nu t)^{\frac{1}{2}},1-y^2)+C\leq C,
 \end{align*}
 With this choice of $\rho$, we deduce that
 \begin{align*}
&\nu^{-1}\|\rho\partial_t\nabla\bar{u}^3\|_{L^2L^2}^2+\nu\|\rho\nabla\Delta\bar{u}^3\|_{L^2L^2}^2+\|\rho\Delta\bar{u}^3\|_{L^{\infty}L^2}^2\\
 \nonumber &\leq C\Big(\|u(0)\|_{H^2}^2+\nu^{-1}\|\rho[\nabla\partial_z \bar{p}+\nabla(\overline{u\cdot\nabla u^3})]\|_{L^2L^2}^2+\nu\|\Delta\bar{u}^3\|_{L^2L^2}^2\Big)\\
 \nonumber &\leq C\Big(\|u(0)\|_{H^2}^2+\nu\|\Delta\bar{u}^3\|_{L^2L^2}^2+\nu\|\na\pa_z\bar{p}\|_{L^2L^2}^2\\
 &\qquad+\nu^{-1}\|\nabla(\bar{u}\cdot\nabla \bar{u}^3)\|_{L^2L^2}^2+\nu^{-1}\|(\nu^{\frac{2}{3}}+\nu t)^{\frac{1}{2}}\nabla(\overline{u_{\neq}\cdot\nabla u^3_{\neq}})\|_{L^2L^2}^2\Big),
\end{align*}
which along with Lemma \ref{lem:u2-H3}, \eqref{eq:u2p-H3-est}, \eqref{eq:p-H2}, \eqref{eq:zz-est} and  Lemma \ref{lem:int-nn} gives
\begin{align*}
&\nu^{-1}\|\rho\partial_t\nabla\bar{u}^3\|_{L^2L^2}^2+\nu\|\rho\nabla\Delta\bar{u}^3\|_{L^2L^2}^2+\|\rho\Delta\bar{u}^3\|_{L^{\infty}L^2}^2\\
&\le C\big(1+\nu^{-2}E_2^2\big)\big(\|u(0)\|_{H^2}^2+\nu^{-2}E_3^4\big).
\end{align*}
This proves our result due to the choice of $\rho$.
\end{proof}

\section{Energy estimates for nonzero modes:semi-linear part}

In this part, the energy estimate is based on the formulation in terms of $(\Delta u^2,\om^2)$:
\begin{align*}
\left\{\begin{aligned}
  &\partial_t(\Delta u^2)-\nu\Delta^2 u^2+y\partial_x\Delta u^2
  +(\partial_x^2+\partial_z^2)(u\cdot\nabla
  u^2)-\partial_y[\partial_x(u\cdot\nabla u^1)+\partial_z(u\cdot\nabla
  u^3)]=0,\\
   &\partial_t \omega^2- \nu\Delta\omega^2 +y\partial_x\omega^2+\partial_zu^2
  +\partial_z(u\cdot\nabla u^1)-\partial_x(u\cdot\nabla u^3)=0,\\
   &\partial_y u^2(t,x,\pm1,z)=u^2(t,x,\pm1,z)=0,\quad
  u^2|_{t=0}(x,y,z)=u^2(0),\\
   &\omega^2(x,\pm1,z)=0,\quad
  \omega^2|_{t=0}=\partial_xu^3(0)-\partial_zu^1(0).
  \end{aligned}\right.
\end{align*}
We denote
\beno
\widehat{\Delta}=\widehat{\Delta}_{k,\ell}=\partial_y^2-k^2-\ell^2,\quad  f_{k,\ell}(y)=\f{1}{2\pi}\int_{\mathbb{T}^2}f(x,y,z)e^{-ikx-i\ell z}dxdz.
\eeno
Taking Fourier transform in $(x,z)$,  we obtain
\begin{align*}\left\{
\begin{aligned}
  &\partial_t(\widehat{\Delta} u^2_{k,\ell})-\nu\widehat{\Delta}^2
  u^2_{k,\ell}+iky\widehat{\Delta} u^2_{k,\ell} -(k^2+\ell^2)(u\cdot\nabla
  u^2)_{k,\ell}\\
  &\qquad\qquad-\partial_y\big[\partial_x(u\cdot\nabla u^1)+\partial_z(u\cdot\nabla
  u^3)\big]_{k,\ell}=0,\\
  &\partial_t \omega^2_{k,\ell}- \nu(\partial_y^2-k^2-l^2)\omega^2_{k,\ell}
  +iky\omega^2_{k,\ell}+i\ell u_{k,\ell}^2 +i\ell(u\cdot\nabla u^1)_{k,\ell}-ik(u\cdot\nabla u^3)_{k,\ell}=0,\\
  &\partial_yu^2_{k,\ell}|_{y=\pm1}=u^2_{k,\ell}|_{y=\pm1}=0,\quad
  u^2_{k,\ell}|_{t=0}=u^2_{k,\ell}(0),\\
  &\omega_{k,\ell}^2|_{y=\pm1}=0,\quad
  \omega^2_{k,\ell}|_{t=0}=iku^3_{k,\ell}(0)-i\ell u^1_{k,\ell}(0).
  \end{aligned}\right.
\end{align*}

Let $a\geq0$ and $\eta=\sqrt{k^2+\ell^2}$. We introduce the following norms:
\begin{align*}
  \|f\|_{X_{k,\ell}^{a}}^2=
  &\eta|k|\|e^{a\nu^{1/3}t}(-\partial_y,i\eta)f\|^2_{L^2 L^2}+\nu\eta^2
  \|e^{a\nu^{1/3}t}(\partial^2_y-\eta^2)f\|^2_{L^2 L^2}\\
  &+\nu^{3/2}\|e^{a\nu^{1/3}t}\partial_y(\partial^2_y-\eta^2)f\|^2_{L^2 L^2}
  +\eta^2\|e^{a\nu^{1/3}t}(-\partial_y,i\eta)f\|^2_{L^\infty L^2}\\&+\nu^{1/2}\|e^{a\nu^{1/3}t}(\partial^2_y-\eta^2)f\|^2_{L^{\infty} L^2},
\end{align*}
and
\begin{align*}
  \|f\|_{Y_{k,\ell}^{a}}^2=\|e^{a\nu^{1/3}t}f\|^2_{L^\infty L^2}+\nu
  \|e^{a\nu^{1/3}t}\partial_yf\|^2_{L^2 L^2}+((\nu k^2)^{1/3}+\nu\eta^2)
  \|e^{a\nu^{1/3}t}f\|^2_{L^2 L^2},
\end{align*}
and
\begin{align*}
  \|f\|^2_{X_a} &=\sum_{k\neq0;\ell\in\mathbb{Z}}
  \|\hat{f}(k,\ell)\|^2_{X^a_{k,\ell}},\quad
  \|f\|^2_{Y_a} = \sum_{k\neq0;\ell\in\mathbb{Z}}
  \|\hat{f}(k,\ell)\|^2_{Y^a_{k,\ell}}.
\end{align*}
Thus, it is easy to see that
\begin{align}\label{fX}
  &\|e^{a\nu^{\f13}t}\partial_x\nabla f_{\neq}\|_{L^{2}L^2}^2+\nu\|e^{a\nu^{\f13}t}(\partial_x,\partial_z) \Delta f_{\neq}\|_{L^2L^2}^2+\nu^{\f32}\|e^{a\nu^{\f13}t}\partial_y\Delta f_{\neq}\|_{L^2L^2}^2
  \\ \nonumber&\quad+\|e^{a\nu^{\f13}t}(\partial_x,\partial_z) \nabla f_{\neq}\|_{L^{\infty}L^2}^2+\nu^{1/2}\|e^{a\nu^{\f13}t} \Delta f_{\neq}\|_{L^{\infty}L^2}^2\leq \|f\|^2_{X_a}, \\ \label{fY}
 &\|e^{a\nu^{\f13}t} f_{\neq}\|_{L^{\infty}L^2}^2+\nu \|e^{a\nu^{\f13}t}\nabla f_{\neq}\|_{L^2L^2}^2\leq\|f\|^2_{Y_a}.
\end{align}

\begin{Lemma}\label{lem:E30}
It holds that
\beno
E^2_{3,0}\leq C\big(\|u^2_{\neq}\|^2_{X_{2\epsilon}}
   +\|\partial_x\omega^2_{\neq}\|^2_{Y_{2\epsilon}}\big).
   \eeno
\end{Lemma}

\begin{proof}

Using the fact $\omega^2=\partial_zu^1-\partial_xu^3$ and $\partial_xu^1+\partial_y
  u^2+\partial_zu^3=0$, we know that $(\partial_x^2+\partial_z^2) u_{\neq}^{3}
  =-\partial_x\omega_{\neq}^2-\partial_z\partial_yu_{\neq}^2 $ and
  \begin{align*}
&\|e^{2\epsilon\nu^{\f13}t}(\partial_x^2+\partial_z^2) u^{3}_{\neq}\|_{L^{\infty}L^2}+\nu^{\f12}\|e^{2\epsilon\nu^{\f13}t}(\partial_x^2+\partial_z^2)
\nabla u^{3}_{\neq}\|_{L^2L^2}\\ &\leq \|e^{2\epsilon\nu^{\f13}t}\partial_x\omega_{\neq}^2\|_{L^{\infty}L^2}+
\|e^{2\epsilon\nu^{\f13}t}\partial_z\partial_yu_{\neq}^2\|_{L^{\infty}L^2}\\&\quad
+\nu^{\f12}\|e^{2\epsilon\nu^{\f13}t}\partial_x\nabla\omega_{\neq}^2\|_{L^2L^2}
+\nu^{\f12}\|e^{2\epsilon\nu^{\f13}t}\partial_z\partial_y\nabla u_{\neq}^2\|_{L^2L^2},
\end{align*}
from which, \eqref{fX}, \eqref{fY} and the definition of $E_{3,0}$, we deduce the lemma.
\end{proof}\smallskip

In what follows, we take $\eps=\eps_1/8$.

\subsection{Estimate of $E_{3,0}$}

\begin{Proposition}\label{prop:E30}
It holds that
\begin{align*}
  E^2_{3,0} +\|u^2_{\neq}\|^2_{X_{2\epsilon}} \leq C\|u(0)\|^2_{H^2}+C \Big(E^4_3/\nu^2+E^2_2E^2_3/\nu^2+E^2_1E_3E_5+E^2_1E^{\f32}_3E^{\f12}_5\Big).
\end{align*}
\end{Proposition}

\begin{proof}
By Proposition \ref{prop:TS-non} and Proposition \ref{prop:TS-nav} and $ \partial_yu^2_{k,\ell}+iku^1_{k,\ell}+i\ell u^3_{k,\ell}=0$,  we have
\begin{align*}
  &\|u^2_{k,\ell}\|^2_{X_{k,\ell}^{2\epsilon}}\leq  C\Big(\|\widehat{\Delta}u^2_{k,\ell}(0)\|_{L^2}^2+
   (k^2+\ell^2)^{-1}\|\widehat{\Delta}(iku^1_{k,\ell}(0)+i\ell u^3_{k,\ell}(0))\|_{L^2}^2\Big)\\
   &\quad+ C\nu^{-1}\Big(\|e^{2\epsilon\nu^{1/3}t}[\partial_x(u\cdot\nabla
   u^1)+\partial_z(u\cdot\nabla
  u^3)]_{k,\ell}\|_{L^2L^2}^2+ \|e^{2\epsilon\nu^{1/3}t}(k,\ell)(u\cdot\nabla
  u^2)_{k,\ell}\|^2_{L^2L^2}\Big),
\end{align*}
and
\begin{align*}
  \|\omega^2_{k,\ell}\|^2_{Y^{2\epsilon}_{k,\ell}}
   \leq& C
  \Big(\|\omega^2_{k,\ell}(0)\|^2_{L^2}
 \\&+\min\big\{(\nu \eta^2)^{-1},(\nu
  k^2)^{-1/3}\big\}\|e^{2\epsilon\nu^{1/3}t}(k(u\cdot\nabla u^3)_{k,\ell}-\ell(u\cdot\nabla u^1)_{k,\ell})\|^2_{L^2L^2}\\
  &+(\ell^2(|k|\eta)^{-1})\|e^{2\epsilon\nu^{1/3}t}\partial_yu^2_{k,\ell}\|^2_{L^2L^2}
  +(\ell^2\eta|k|^{-1})\|e^{2\epsilon\nu^{1/3}t}u^2_{k,\ell}\|^2_{L^2L^2}\Big)\\
  \leq&C
  \Big(\|\omega^2_{k,\ell}(0)\|^2_{L^2}+\nu^{-1}\big(\|e^{2\epsilon\nu^{1/3}t}(u\cdot\nabla u^3)_{k,\ell}\|^2_{L^2L^2}+\|e^{2\epsilon\nu^{1/3}t}(u\cdot\nabla u^1)_{k,\ell}\|^2_{L^2L^2}\big)\\
  &+(\ell^2(|k|\eta)^{-1})\|e^{2\epsilon\nu^{1/3}t}\partial_yu^2_{k,\ell}\|^2_{L^2L^2}
  +(\ell^2\eta|k|^{-1})\|e^{2\epsilon\nu^{1/3}t}u^2_{k,\ell}\|^2_{L^2L^2}\Big).
\end{align*}
Using the fact that
\begin{align*}
   &\sum_{k\neq0;\ell\in\mathbb{Z}}k^2\Big((\ell^2(|k|\eta)^{-1}) \|e^{2\epsilon\nu^{1/3}t}\partial_yu^2_{k,\ell}\|^2_{L^2L^2}
  +(\ell^2\eta|k|^{-1})\|e^{2\epsilon\nu^{1/3}t}u^2_{k,\ell}\|^2_{L^2L^2}\Big)\\
  &\quad\leq\sum_{k\neq0;\ell\in\mathbb{Z}}\eta|k|\|e^{2\epsilon\nu^{1/3}t} (\partial_y,i\eta)u^2_{k,\ell}\|^2_{L^2L^2} \leq \sum_{k\neq0;\ell\in\mathbb{Z}}\|u^2_{k,\ell}\|^2_{X^{2\epsilon}_{k,\ell}} =\|u^2_{\neq}\|^2_{X_{2\epsilon}},
\end{align*}
we deduce that
\begin{align*}
  &\|\partial_x\omega^2_{\neq}\|^2_{Y_{2\epsilon}}\leq C\sum_{k\neq0;\ell\in\mathbb{Z}}k^2\Big(\|\omega^2_{k,\ell}(0)\|^2_{L^2}\\
  &\quad+ \nu^{-1}\big(\|e^{2\epsilon\nu^{1/3}t}(u\cdot\nabla u^3)_{k,\ell}\|^2_{L^2L^2}+
  \|e^{2\epsilon\nu^{1/3}t}(u\cdot\nabla u^1)_{k,\ell}\|^2_{L^2L^2}\big) \Big)\\
  &\quad+C\sum_{k\neq0;\ell\in\mathbb{Z}}k^2\Big((\ell^2(|k|\eta)^{-1})\|e^{2\epsilon\nu^{1/3}t} \partial_yu^2_{k,\ell}\|^2_{L^2L^2}
  +(\ell^2\eta|k|^{-1})\|e^{2\epsilon\nu^{1/3}t}u^2_{k,\ell}\|^2_{L^2L^2}\Big)\\
  &\leq C\|\omega^2_{\neq}(0)\|^2_{H^1}+
   C\nu^{-1}\left(\|e^{2\epsilon\nu^{1/3}t}\partial_x(u\cdot\nabla u^3)\|^2_{L^2L^2}+
   \|e^{2\epsilon\nu^{1/3}t}\partial_x(u\cdot\nabla u^1)\|^2_{L^2L^2}\right) +C\|u^2_{\neq}\|^2_{X_{2\epsilon}}.
\end{align*}
Then we infer from Lemma \ref{lem:E30} that
\begin{align*}
&E^2_{3,0}+\|u^2_{\neq}\|^2_{X_{2\epsilon}}\leq C\big(\|u^2_{\neq}\|^2_{X_{2\epsilon}}
   +\|\partial_x\omega^2_{\neq}\|^2_{Y_{2\epsilon}}\big)\\
   &\leq C\|\omega^2_{\neq}(0)\|^2_{H^1}+
   C\nu^{-1}\left(\|e^{2\epsilon\nu^{1/3}t}\partial_x(u\cdot\nabla u^3)\|^2_{L^2L^2}+
   \|e^{2\epsilon\nu^{1/3}t}\partial_x(u\cdot\nabla u^1)\|^2_{L^2L^2}\right) +C\|u^2_{\neq}\|^2_{X_{2\epsilon}}\\
  &\leq C\|\omega^2_{\neq}(0)\|^2_{H^1}+
   C\nu^{-1}\left(\|e^{2\epsilon\nu^{1/3}t}\partial_x(u\cdot\nabla u^3)\|^2_{L^2L^2}+
   \|e^{2\epsilon\nu^{1/3}t}\partial_x(u\cdot\nabla u^1)\|^2_{L^2L^2}\right)\\
&\quad+C\nu^{-1}\!\sum_{k\neq0;\ell\in\mathbb{Z}}\Big(\|e^{2\epsilon\nu^{1/3}t}[\partial_x(u\cdot\nabla
   u^1)+\partial_z(u\cdot\nabla
  u^3)]_{k,\ell}\|_{L^2L^2}^2+ \|e^{2\epsilon\nu^{1/3}t}(k,\ell)(u\cdot\nabla
  u^2)_{k,\ell}\|^2_{L^2L^2}\Big)\\
   &\quad+ C\sum_{k\neq0;\ell\in\mathbb{Z}} \Big(\|\widehat{\Delta}_{k,\ell}u^2_{k,\ell}(0)\|_{L^2}^2+
   (k^2+\ell^2)^{-1}\|\widehat{\Delta}_{k,l}(iku^1_{k,\ell}(0)+i\ell u^3_{k,\ell}(0))\|_{L^2}^2\Big)\\
   &\leq C\|\omega^2_{\neq}(0)\|^2_{H^1}+
   C\nu^{-1}\left(\|e^{2\epsilon\nu^{1/3}t}\partial_x(u\cdot\nabla u^3)\|^2_{L^2L^2}+
   \|e^{2\epsilon\nu^{1/3}t}\partial_x(u\cdot\nabla u^1)\|^2_{L^2L^2}\right)\\
  &\quad+C\nu^{-1}\Big(\|e^{2\epsilon\nu^{1/3}t}[\partial_x(u\cdot\nabla
   u^1)+\partial_z(u\cdot\nabla
  u^3)]_{\neq}\|_{L^2L^2}^2+ \|e^{2\epsilon\nu^{1/3}t}(\partial_x,\partial_z)(u\cdot\nabla
  u^2)_{\neq}\|^2_{L^2L^2}\Big)\\&\quad+C\Big(\|\Delta u_{\neq}^2(0)\|^2_{L^2}+\|\Delta u_{\neq}^1(0)\|^2_{L^2}+
  \|\Delta u_{\neq}^3(0)\|^2_{L^2}\Big)\\
   &\leq C\|u_{\neq}(0)\|^2_{H^2}+
   C\nu^{-1}\Big(\|e^{2\epsilon\nu^{1/3}t}(\partial_x,\partial_z)(u\cdot\nabla u^3)_{\neq}\|^2_{L^2L^2}+\|e^{2\epsilon\nu^{1/3}t}\partial_x(u\cdot\nabla u^1)_{\neq}\|^2_{L^2L^2}\\
  &\qquad+ \|e^{2\epsilon\nu^{1/3}t}(\partial_x,\partial_z)(u\cdot\nabla
  u^2)_{\neq}\|^2_{L^2L^2}\Big).
\end{align*}

Let us estimate each term on the right hand side. For $k\in\{2,3\}$,  we get by Lemma \ref{lem:int-nn}  that
\begin{align*}
   & \|e^{2\epsilon\nu^{1/3}t}(\partial_x,\partial_z)(u_{\neq}\cdot\nabla u_{\neq}^k)\|_{L^2L^2}^2\leq C\nu^{-1}E_3^4,
\end{align*}
and by Lemma \ref{lem:int-nz-23},
\begin{align*}
   & \|e^{2\epsilon\nu^{1/3}t}(\partial_x,\partial_z)(u_{\neq}\cdot\nabla \overline{u}^k)\|_{L^2L^2}^2\leq C\nu^{-1}E_2^2E_3^2,
\end{align*}
and by Lemma \ref{lem:int-zn-1-23} and Lemma \ref{lem:int-nz-23},
\begin{align*}
   &\|e^{2\epsilon\nu^{1/3}t}(\partial_x,\partial_z)(\overline{u}\cdot\nabla u_{\neq}^k)\|_{L^2L^2}^2\\ &\leq C\Big(\|e^{2\epsilon\nu^{1/3}t}(\partial_x,\partial_z)(\overline{u}^1\partial_x u_{\neq}^k)\|_{L^2L^2}^2+\|e^{2\epsilon\nu^{1/3}t}(\partial_x,\partial_z)((\overline{u}^2\partial_y+\overline{u}^3\partial_z) u_{\neq}^k)\|_{L^2L^2}^2\Big)\\ &\leq C\nu E_1^2E_3E_5+C\nu^{-1}E_2^2E_3^2.
\end{align*}
Noting that  $(fg)_{\neq}=\bar{f}g_{\neq}+f_{\neq}\bar{g}+(f_{\neq}g_{\neq})_{\neq}$. This shows that for $k=2,3$
\begin{align}\label{u32}
   & \|e^{2\epsilon\nu^{1/3}t}(\partial_x,\partial_z)(u\cdot\nabla u^k)_{\neq}\|_{L^2L^2}^2\leq C\big(\nu^{-1}E_3^4+\nu^{-1}E_2^2E_3^2+\nu E_1^2E_3E_5\big).
\end{align}

By Lemma \ref{lem:int-nn}, we have
\begin{align*}
   & \|e^{2\epsilon\nu^{1/3}t}\partial_x(u_{\neq}\cdot\nabla u_{\neq}^1)\|_{L^2L^2}^2\leq C\nu^{-1}E_3^4,
\end{align*}
and  by Lemma \ref{lem:int-zn-1-23},
\begin{align*}
   & \|e^{2\epsilon\nu^{1/3}t}\partial_x(u_{\neq}\cdot\nabla \overline{u}^1)\|_{L^2L^2}^2\leq C\nu E_1^2E_3E_5,
\end{align*}
and by Lemma \ref{lem:int-zn-11} and Lemma \ref{lem:int-nz-23},
\begin{align*}
   & \|e^{2\epsilon\nu^{1/3}t}\partial_x(\overline{u}\cdot\nabla u_{\neq}^1)\|_{L^2L^2}^2\\ &\leq C\Big(\|e^{2\epsilon\nu^{1/3}t}\partial_x(\overline{u}^1\partial_x u_{\neq}^1)\|_{L^2L^2}^2+\|e^{2\epsilon\nu^{1/3}t}\partial_x((\overline{u}^2\partial_y+\overline{u}^3\partial_z) u_{\neq}^1)\|_{L^2L^2}^2\Big)\\ &\leq C\nu E_1^2E_3^{\f32}E_5^{\f12}+C\nu^{-1}E_2^2E_3^2,
\end{align*}
which show that
\begin{align*}
   & \|e^{2\epsilon\nu^{1/3}t}\partial_x(u\cdot\nabla u^1)_{\neq}\|_{L^2L^2}^2\leq C\big(\nu^{-1}E_3^4+\nu^{-1}E_2^2E_3^2+\nu E_1^2E_3E_5+\nu E_1^2E_3^{\f32}E_5^{\f12}\big).
\end{align*}

Summing up, we conclude that
\begin{align*}
  E^2_{3,0} +\|u^2_{\neq}\|^2_{X_{2\epsilon}}& \leq C\|u(0)\|^2_{H^2}+C \Big(E^4_3/\nu^2+E^2_2E^2_3/\nu^2+E^2_1E_3E_5+E^2_1E^{\f32}_3E^{\f12}_5\Big).
\end{align*}

This completes the proof of the proposition.
\end{proof}

\subsection{Estimate of $E_{3,1}$}

\begin{Proposition}\label{prop:E31}
It holds that
\begin{align*}
   E^2_{3,1}
   \leq &C\Big(\|u(0)\|_{H^2}^2+\nu^{-2}E_3^4+\nu^{-\f43}E_2^2E_3^2+ E_1^2E_3E_5+ E_1^2E_3^{\f74}E_5^{\f14}+E_1^2E_3^{\f32}E_5^{\f12}\Big).
\end{align*}
\end{Proposition}

\begin{proof}
Recall that
\begin{align*}\left\{
\begin{aligned}
  &\partial_t \omega^2_{k,\ell}- \nu(\partial_y^2-\eta^2)\omega^2_{k,\ell}
  +iky\omega^2_{k,\ell}+i\ell u_{k,\ell}^2+i\ell(u\cdot\nabla u^1)_{k,\ell}-ik(u\cdot\nabla u^3)_{k,\ell}=0,\\
  &\omega_{k,\ell}^2(\pm1)=0,\quad
  \omega^2_{k,\ell}|_{t=0}=iku^3_{k,\ell}(0)-i\ell u^1_{k,\ell}(0).
  \end{aligned}\right.
\end{align*}
It follows from Proposition \ref{prop:TS-nav} that
\begin{align*}
  &\|(k,\ell)\omega^2_{k,\ell}\|^2_{Y^{2\epsilon}_{k,\ell}}\leq C\eta^2\|\omega^2_{k,\ell}\|^2_{Y^{2\epsilon}_{k,\ell}}\\
  &\leq C\eta^2\|\omega^2_{k,\ell}(0)\|^2_{L^2}+C\eta^2\big(\nu\eta^2\big)^{-1}\|e^{2\epsilon\nu^{1/3}t} (k(u\cdot\nabla u^3)_{k,l}-\ell(u\cdot\nabla u^1)_{k,\ell})\|^2_{L^2L^2}\\
  &\qquad+C\eta^2\min\{(\nu\eta^2)^{-1},(\nu
  k^2)^{-1/3}\}\|e^{2\epsilon\nu^{1/3}t}\ell u^2_{k,\ell}\|^2_{L^2L^2}\\
  &\leq C\eta^2\|\omega^2_{k,\ell}(0)\|^2_{L^2}+C(\nu k^2)^{-\f23}(\ell^2|k|\eta)\|e^{2\epsilon\nu^{1/3}t}u^2_{k,\ell}\|^2_{L^2L^2}\\
  &\quad+C\nu^{-1}\|e^{2\epsilon\nu^{1/3}t}(k(u\cdot\nabla u^3)_{k,\ell}-\ell(u\cdot\nabla u^1)_{k,\ell})\|^2_{L^2L^2},
\end{align*}
(here we used $\min\{(\nu\eta^2)^{-1},(\nu
  k^2)^{-1/3}\}\leq (\nu\eta^2)^{-1/2}(\nu
  k^2)^{-1/6}=(\nu k^2)^{-\f23}|k|\eta^{-1} $) and
\begin{align*}
  &\|e^{2\epsilon\nu^{1/3}t}\partial_y\omega^2_{k,\ell}\|^2_{L^\infty L^2}+\nu\|e^{2\epsilon\nu^{1/3}t}\partial_y^2\omega^2_{k,\ell}\|^2_{L^2L^2}+
  \nu\eta^2\|e^{2\epsilon\nu^{1/3}t}\partial_y\omega^2_{k,\ell}\|^2_{L^2L^2}\\
   &\leq C\Big(\|\partial_y\omega^2_{k,\ell}(0)\|^2_{L^2}+(\nu/|k|)^{-\f23}\|\omega^2_{k,\ell}(0)\|^2_{L^2}\Big)\\
  &\quad+C(\nu/|k|)^{-\f23}\Big((\ell^2(|k|\eta)^{-1})\|e^{2\epsilon\nu^{1/3}t}\partial_yu^2_{k,\ell}\|^2_{L^2L^2}
  +(\ell^2\eta|k|^{-1})\|e^{2\epsilon\nu^{1/3}t}u^2_{k,\ell}\|^2_{L^2L^2}\Big)\\
  &\quad+C\nu^{-1}\|e^{2\epsilon\nu^{1/3}t}(k(u\cdot\nabla u^3)_{k,\ell}-\ell(u\cdot\nabla u^1)_{k,\ell})\|^2_{L^2L^2}\\
  &\leq C\big(1+(\nu k^2)^{-\f23}\big)\Big(\|\partial_y\omega^2_{k,\ell}(0)\|^2_{L^2} +k^2\|\omega^2_{k,\ell}(0)\|^2_{L^2}\Big)\\
  &\quad+C(\nu k^2)^{-\f23}\Big((\ell^2|k|\eta^{-1})\|e^{2\epsilon\nu^{1/3}t}\partial_yu^2_{k,\ell}\|^2_{L^2L^2}
  +(\ell^2\eta|k|)\|e^{2\epsilon\nu^{1/3}t}u^2_{k,\ell}\|^2_{L^2L^2}\Big)\\
  &\quad+C\nu^{-1}\|e^{2\epsilon\nu^{1/3}t}(k(u\cdot\nabla u^3)_{k,\ell}-\ell(u\cdot\nabla u^1)_{k,\ell})\|^2_{L^2L^2},
\end{align*}
which show that
\begin{align*}
&\|e^{2\epsilon\nu^{1/3}t}\partial_y\omega^2_{k,\ell}\|^2_{L^\infty L^2}+\nu\|e^{2\epsilon\nu^{1/3}t}(\partial_y^2-\eta^2)\omega^2_{k,\ell}\|^2_{L^2L^2}/2
  +\|e^{2\epsilon\nu^{1/3}t}(k,\ell)\omega^2_{k,\ell}\|^2_{L^\infty L^2}\\
  &\leq\|e^{2\epsilon\nu^{1/3}t}\partial_y\omega^2_{k,\ell}\|^2_{L^\infty L^2}+\nu\|e^{2\epsilon\nu^{1/3}t}\partial_y^2\omega^2_{k,\ell}\|^2_{L^2L^2}\\&\qquad+
\nu\eta^2\|e^{2\epsilon\nu^{1/3}t}(k,\ell)\omega^2_{k,\ell}\|^2_{L^2L^2}
  +\|e^{2\epsilon\nu^{1/3}t}(k,\ell)\omega^2_{k,\ell}\|^2_{L^\infty L^2}\\
  &\leq
  \|e^{2\epsilon\nu^{1/3}t}\partial_y\omega^2_{k,\ell}\|^2_{L^\infty L^2}+\nu\|e^{2\epsilon\nu^{1/3}t}\partial_y^2\omega^2_{k,\ell}\|^2_{L^2L^2}
  +\|(k,\ell)\omega^2_{k,l}\|^2_{Y^{2\epsilon}_{k,\ell}}\\
  &\leq C\big(1+(\nu k^2)^{-\f23}\big)\Big(\|\partial_y\omega^2_{k,\ell}(0)\|^2_{L^2} +\eta^2\|\omega^2_{k,\ell}(0)\|^2_{L^2}\Big)\\
  &\quad+C(\nu k^2)^{-\f23}\Big((\ell^2|k|\eta^{-1})\|e^{2\epsilon\nu^{1/3}t}\partial_yu^2_{k,\ell}\|^2_{L^2L^2}
  +(\ell^2\eta|k|)\|e^{2\epsilon\nu^{1/3}t}u^2_{k,\ell}\|^2_{L^2L^2}\Big)\\
  &\quad+C\nu^{-1}\|e^{2\epsilon\nu^{1/3}t}(k(u\cdot\nabla u^3)_{k,\ell}-\ell(u\cdot\nabla u^1)_{k,\ell})\|^2_{L^2L^2}.
\end{align*}

Using the fact that
\begin{align*}
   &\sum_{k\neq0;\ell \in\mathbb{Z}}\Big((\ell^2|k|\eta^{-1}) \|e^{2\epsilon\nu^{1/3}t}\partial_yu^2_{k,\ell}\|^2_{L^2L^2}
  +(\ell^2\eta|k|)\|e^{2\epsilon\nu^{1/3}t}u^2_{k,\ell}\|^2_{L^2L^2}\Big)\\
  &\leq \sum_{k\neq0;\ell \in\mathbb{Z}}\eta|k|\|e^{2\epsilon\nu^{1/3}t} (\partial_y,i\eta)u^2_{k,\ell}\|^2_{L^2L^2} \leq \sum_{k\neq0;\ell\in\mathbb{Z}}\|u^2_{k,\ell}\|^2_{X^{2\epsilon}_{k,\ell}} =\|u^2_{\neq}\|^2_{X_{2\epsilon}},
\end{align*}
we deduce that
\begin{align*}
  &\nu^{\f23}\big(\|e^{2\epsilon\nu^{1/3}t}\nabla \omega^2_{\neq}\|^2_{L^{\infty}L^2}
   +\nu\|e^{2\epsilon\nu^{1/3}t}\Delta \omega^2_{\neq}\|^2_{L^{2}L^2}\big)\\
   &\leq \nu^{\f23}\sum_{k\neq0;\ell\in\mathbb{Z}} \Big(\|e^{2\epsilon\nu^{1/3}t}\partial_y\omega^2_{k,\ell}\|^2_{L^\infty L^2}+\nu\|e^{2\epsilon\nu^{1/3}t}(\partial_y^2-\eta^2)\omega^2_{k,\ell}\|^2_{L^2L^2}
  +\|e^{2\epsilon\nu^{1/3}t}(k,\ell)\omega^2_{k,\ell}\|^2_{L^\infty L^2}\Big)\\
  &\leq C\sum_{k\neq0;\ell \in\mathbb{Z}}\nu^{\f23}\big(1+(\nu k^2)^{-\f23}\big)\big(\|\partial_y\omega^2_{k,\ell}(0)\|^2_{L^2} +\eta^2\|\omega^2_{k,\ell}(0)\|^2_{L^2}\big)\\
  &\quad+C\sum_{k\neq0;\ell\in\mathbb{Z}}\Big((\ell^2|k|\eta^{-1})\|e^{2\epsilon\nu^{1/3}t}\partial_yu^2_{k,\ell}\|^2_{L^2L^2}
  +(\ell^2\eta|k|)\|e^{2\epsilon\nu^{1/3}t}u^2_{k,\ell}\|^2_{L^2L^2}\Big)\\
  &\quad+C\nu^{-\f13}\|e^{2\epsilon\nu^{1/3}t}(\partial_x(u\cdot\nabla u^3)-\partial_z(u\cdot\nabla u^1))_{\neq}\|^2_{L^2L^2}\\
  &\leq C\|\omega^2_{\neq}(0)\|^2_{H^1}+C\|u^2_{\neq}\|^2_{X_{2\epsilon}} +C\nu^{-\f13}\|e^{2\epsilon\nu^{1/3}t}(\partial_x(u\cdot\nabla u^3)-\partial_z(u\cdot\nabla u^1))_{\neq}\|^2_{L^2L^2}.
\end{align*}

Now let us estimate each term on the right hand side.  By \eqref{u32}, we have
\begin{align*}
   & \|e^{2\epsilon\nu^{1/3}t}\partial_x(u\cdot\nabla u^3)_{\neq}\|_{L^2L^2}^2\leq C\big(\nu^{-1}E_3^4+\nu^{-1}E_2^2E_3^2+\nu E_1^2E_3E_5\big).
\end{align*}
By Lemma \ref{lem:int-nn}, we have
\begin{align*}
   & \|e^{2\epsilon\nu^{1/3}t}\partial_z(u_{\neq}\cdot\nabla u_{\neq}^1)_{\neq}\|_{L^2L^2}^2\leq C\nu^{-\f53}E_3^4,
\end{align*}
and by Lemma \ref{lem:int-zz-11-pz}, we have
\begin{align*}
   & \|e^{2\epsilon\nu^{1/3}t}\partial_z(u_{\neq}\cdot\nabla \overline{u}^1)_{\neq}\|_{L^2L^2}^2\leq C\nu^{1/3}E^2_1E_{3} \big(E_5 +E_3^{\f34}E^{\f14}_5\big),
\end{align*}
and by Lemma \ref{lem:int-zz-11-pz} and Lemma \ref{lem:int-nz-23},
\begin{align*}
   & \|e^{2\epsilon\nu^{1/3}t}\partial_z(\overline{u}\cdot\nabla u_{\neq}^1)\|_{L^2L^2}^2\\ &\leq C\Big(\|e^{2\epsilon\nu^{1/3}t}\partial_z(\overline{u}^1\partial_x u_{\neq}^1)\|_{L^2L^2}^2+\|e^{2\epsilon\nu^{1/3}t}\partial_z((\overline{u}^2\partial_y+\overline{u}^3\partial_z) u_{\neq}^1)\|_{L^2L^2}^2\Big)\\ &\leq C\nu^{1/3}E^2_1E_{3}\big(E_5 +E_3^{\f34}E^{\f14}_5\big)+C\nu^{-1}E_2^2E_3^2,
\end{align*}
which show that
\begin{align*}
   & \|e^{2\epsilon\nu^{1/3}t}\partial_z(u\cdot\nabla u^1)_{\neq}\|_{L^2L^2}^2\leq C\Big(\nu^{-\f53}E_3^4+\nu^{-1}E_2^2E_3^2+\nu^{\f13} E_1^2E_3E_5+\nu^{\f13} E_1^2E_3^{\f74}E_5^{\f14}\Big).
\end{align*}

Summing up, we conclude that
\begin{align*}
   &E^2_{3,1}=\nu^{\f23}\big(\|e^{2\epsilon\nu^{1/3}t}\nabla \omega^2_{\neq}\|^2_{L^{\infty}L^2}
   +\nu\|e^{2\epsilon\nu^{1/3}t}\Delta \omega^2_{\neq}\|^2_{L^{2}L^2}\big)\\
   &\leq C\|\omega^2(0)\|_{H^1}^2+C\|u^2_{\neq}\|^2_{X_{2\epsilon}}+C\Big(\nu^{-2}E_3^4+\nu^{-\f43}E_2^2E_3^2+ E_1^2E_3E_5+ E_1^2E_3^{\f74}E_5^{\f14}\Big),
\end{align*}
which along with Proposition \ref{prop:E30} gives
\begin{align*}
  E^2_3 \leq C\|u(0)\|^2_{H^2}+C \Big(E^4_3/\nu^2+E^2_2E^2_3/\nu^2+E^2_1E_3E_5+E^2_1E^{\f74}_3E^{\f14}_5+E_1^2E_3^{\f32}E_5^{\f12}\Big).
\end{align*}

This completes the proof of the proposition.
\end{proof}

\section{Energy estimates for nonzero modes:quasi-linear part}

\subsection{Resolvent estimate of the linearized operator}
Let $\overset{\circ}{J}(\Omega),\ {J}_{2,0}^{(k)}(\Omega)$ denote the closure of the set
of vector field, which is smooth and solenoidal in $\Omega$ and vanish on $ \partial\Omega,$ in the topology, respectively, of $L^2(\Omega)$ and Sobolev space $W^{k,2}(\Omega).$ In this section, $u,v$ stands for generic functions rather than the solution
introduced in section 1. \smallskip

We define
\beno
&&Q(u,v)=\mathbb{P}\big(u\cdot\nabla{v}+{v}\cdot\nabla{u}\big),\\
&&Q_1(f,{v})=Q((f,0,0),{v}),\quad A_{[V]}{v}=\nu\mathbb{P}\Delta{v}-Q_1(V,{v}).
\eeno
Here $V$ satisfies \eqref{ass:V} and $\mathbb{P} $ is a projector in $L^2(\Omega)$ onto $\overset{\circ}{J}(\Omega). $
Then the operator $A_{[V]}$ defined on ${J}_{2,0}^{(2)}(\Omega)$ is  invariant in the subspace $\mathcal{H}=\big\{f\in L^2(\Omega)|P_0f=0\big\}$.

We denote by $m(\nu,V)$ the upper bound of the real parts of points of the spectrum of $A_{[V]}$ in the subspace ${J}_{2,0}^{(2)}(\Omega)\cap\mathcal{H}.$ More precisely, we consider $A_{[V]}$ as a closed linear operator in the Hilbert space $\overset{\circ}{J}(\Omega)\cap\mathcal{H}$ with the domain $D(A_{[V]})={J}_{2,0}^{(2)}(\Omega)\cap\mathcal{H}. $

Let $\kappa=\partial_zV/\partial_yV$. For ${v}=(v^1,v^2,v^3),$ we introduce the notations:
\begin{align*}
\|{v}\|_{Z_{[V]}^1}^2=&\nu^{-1}\big(\|\nabla (v^2+\kappa v^3)\|_{L^2}^2+\|\partial_x v^3\|_{L^2}^2\big),\\
\|{v}\|_{Z_{[V]}^2}^2=&\nu^{\f{1}{3}}\big(\|\partial_x^2v^3\|_{L^2}^2+\|\partial_x(\partial_z-\kappa\partial_y)v^3\|_{L^2}^2\big)+
\nu\big(\|\nabla\partial_x^2v^3\|_{L^2}^2+\|\nabla\partial_x(\partial_z-\kappa\partial_y)v^3\|^2_{L^2}\big)
\\&+\nu^{\f{1}{3}} \|\partial_x\nabla(v^2+\kappa v^3)\|_{L^2}^2+\nu\|\partial_x\Delta(v^2+\kappa v^3)\|^2_{L^2}+\nu^{\f{5}{3}} \|\partial_x\Delta v^3\|_{L^2}^2.
\end{align*}

The following Lemma ensures that $\|\cdot\|_{Z_{[V]}^2}$  defines a norm (in a subspace).
\begin{Lemma}\label{lem:Z norm}
  It holds that for any $v\in{J}_{2,0}^{(3)}(\Omega)\cap\mathcal{H}$,
  \begin{align*}
C^{-1}\nu^{\f{1}{3}}\|\partial_x^2{v}\|_{L^2}^2\leq\|{v}\|_{Z_{[V]}^2}^2 \leq C \nu^{\f{1}{3}}\|\partial_x{v}\|_{H^2}^2.
\end{align*}
and more precisely, we have
\begin{align*}
&\nu^{\f{1}{3}}\big(\|\partial_x^2v^2\|_{L^2}^2+\|\partial_x(\partial_z-\kappa\partial_y)v^2\|_{L^2}^2\big)+
\nu\big(\|\nabla\partial_x^2v^2\|_{L^2}^2+\|\nabla\partial_x(\partial_z-\kappa\partial_y)v^2\|_{L^2}^2\big)\leq C\|{v}\|_{Z_{[V]}^2}^2.
\end{align*}

\end{Lemma}

\begin{proof}
The upper bound of first inequality is obvious. For the lower bound, it is enough to note that
\beno
-\pa_x^2v^1=\pa_x\pa_y(v^2+\kappa v^3)+\pa_x(\pa_z-\kappa \pa_y)v^3-\pa_y\kappa\pa_xv^3.
\eeno

Direct calculations show that
\begin{align*}
&\|\partial_x(\partial_z-\kappa\partial_y)v^2\|_{L^2} +\|\partial_x^2v^2\|_{L^2} \\
&\leq \|\partial_x(\partial_z-\kappa\partial_y) (v^2+\kappa v^3)\|_{L^2} +\|\partial_x(\partial_z-\kappa\partial_y)(\kappa v^3)\|_{L^2} +\|\partial_x^2(v^2+\kappa v^3)\|_{L^2}+\|\kappa\partial_x^2v^3\|_{L^2}\\
&\leq \big(\|\kappa\|_{L^\infty}+1\big)\|\partial_x\nabla (v^2+\kappa v^3)\|_{L^2} +\|\kappa\|_{L^\infty}\|\partial_x(\partial_z-\kappa\partial_y) v^3\|_{L^2}+\|(\partial_z-\kappa\partial_y)\kappa\|_{L^\infty}\|\partial_xv^3\|_{L^2} \\
&\qquad+\|\partial_x^2(v^2+\kappa v^3)\|_{L^2}+\|\kappa\|_{L^\infty}\|\partial_x^2v^3\|_{L^2}\\
&\leq C\big(\|\partial_x\nabla (v^2+\kappa v^3)\|_{L^2} +\|\partial_x(\partial_z-\kappa\partial_y) v^3\|_{L^2}+\|\partial_xv^3\|_{L^2} +\|\partial_x^2v^3\|_{L^2}\big)\leq C\nu^{-\f16}\|{v}\|_{Z_{[V]}^2},
\end{align*}
and
\begin{align*}
&\|\nabla\partial_x(\partial_z-\kappa\partial_y)v^2\|_{L^2} +\|\nabla\partial_x^2v^2\|_{L^2} \\
&\leq \|\nabla\partial_x(\partial_z-\kappa\partial_y) (v^2+\kappa v^3)\|_{L^2} +\|\nabla\partial_x(\partial_z-\kappa\partial_y)(\kappa v^3)\|_{L^2}\\&\qquad+\|\nabla\partial_x^2(v^2+\kappa v^3)\|_{L^2}+\|\nabla(\kappa\partial_x^2v^3)\|_{L^2}\\
&\leq C\Big((1+\|\kappa \|_{H^2})\|\partial_x\nabla(v^2+\kappa v^3)\|_{H^1} +\|\nabla\partial_x(\kappa(\partial_z-\kappa\partial_y) v^3)\|_{L^2} +\|\nabla\partial_x (v^3(\partial_z-\kappa\partial_y)\kappa )\|_{L^2}\\&\qquad+\|\nabla\partial_x^2(v^2+\kappa v^3)\|_{L^2}+\|\kappa\|_{H^2}\|\partial_x^2v^3\|_{H^1}\Big)\\
&\leq C\Big(\|\partial_x\Delta(v^2+\kappa v^3)\|_{L^2} +\|\kappa\|_{H^2}\|\partial_x(\partial_z-\kappa\partial_y) v^3\|_{H^1} +\|(\partial_z-\kappa\partial_y)\kappa \|_{H^2}\|\partial_xv^3\|_{H^1}\\&\qquad+\|\nabla\partial_x^2(v^2+\kappa v^3)\|_{L^2}+\|\nabla\partial_x^2v^3\|_{L^2}\Big)\\
&\leq C\Big(\|\partial_x\Delta(v^2+\kappa v^3)\|_{L^2} +\|\nabla\partial_x(\partial_z-\kappa\partial_y) v^3\|_{L^2} +\|\nabla\partial^2_xv^3\|_{L^2}\Big)\leq C\nu^{-\f12}\|{v}\|_{Z_{[V]}^2}.
\end{align*}
Here we used $\|(\partial_z-\kappa\partial_y)\kappa )\|_{H^2}\leq C\big(\|\kappa\|_{H^3}+\|\kappa\|_{H^2}\|\partial_y\kappa\|_{H^2}\big)\leq C$.
\end{proof}

\begin{Proposition}\label{prop:res-A}
Let $\lambda\in\mathbb{C},\ \mathbf{Re}(\lambda)\in[0,\epsilon_1\nu^{\f13}].$ It holds that for any  $v\in{J}_{2,0}^{(3)}(\Omega)\cap\mathcal{H}$,
  \begin{align*}
 \|{v}\|_{Z_{[V]}^2} \leq C \|(A_{[V]}+\lambda){v}\|_{Z_{[V]}^1}.
  \end{align*}
\end{Proposition}

\begin{proof}
Let $\vec{g}=(A_{[V]}+\lambda){v}=(g^1,g^2,g^3)$ and $\lambda=i\lambda_i+a\nu^{\f13}$ with $\lambda_i\in\mathbb{R},$ $a\in[0,\epsilon_1)$. Due to $\partial_xV=0$, we have
\beno
-A_{[V]}{v}=\mathbb{P}\big(-\nu\Delta {v}+V\partial_x{v}+(\partial_yV(v^2+\kappa v^3),0,0)\big),
\eeno
Then we  get
 \begin{align*}
    \mathbb{P}(-\nu\Delta {v}+V\partial_x {v})+\mathbb{P}(\partial_yV(v^2+\kappa v^3),0,0)-i\lambda_i {v}-a\nu^{\f13}{v}=\vec{g}.
 \end{align*}
This means that
 \begin{align*}\left\{\begin{aligned}
    &\big(-\nu\Delta+V\partial_x-i\lambda_I-a\nu^{\f13}\big){v}+(\partial_yV(v^2+\kappa v^3),0,0)+\nabla P+\vec{g}=0,\\
    &\text{div}{v}=0,\quad \Delta P=-2\partial_yV(\partial_xv^2+\kappa\partial_xv^3),\\
    &{v}|_{y=\pm1}=(\partial_yP-\nu\Delta v^2)|_{y=\pm1}=0.
    \end{aligned}\right.
 \end{align*}
Let  $W=v^2+\kappa v^3,\ U=v^3$. Then $(W,U)$ satisfies
  \begin{align*}\left\{\begin{aligned}
     & -\nu\Delta W+(\partial_xV-i\lambda_i)W-a\nu^{\f13}W+(\partial_y+\kappa\partial_z)P\\
     &\qquad\quad+(g^2+\kappa g^3)+2\nu\nabla\kappa\cdot\nabla U+\nu(\Delta\kappa)U=0,\\
    & -\nu\Delta U+(\partial_xV-i\lambda_i)U-a\nu^{\f13}U+\partial_zP+g^3=0,\\
    &\Delta P=-2\partial_x(\partial_yVW),\quad W|_{y=\pm1}=\partial_yW|_{y=\pm1}=U|_{y=\pm1}=0.
  \end{aligned}\right.
  \end{align*}
Here we used $\partial_yW|_{y=\pm1}=(\partial_yv^2+(\partial_y\kappa)v^3+\kappa\partial_yv^3)|_{y=\pm1} =(-\partial_xv^1-\partial_zv^3+(\partial_y\kappa)v^3+\kappa\partial_yv^3)|_{y=\pm1}=0.$

 We denote
 \beno
 &&{v}_k(x,y,z)=\f{1}{2\pi}\int_{\mathbb{T}}e^{ik(x-x_1)}{v}(x_1,y,z)dx_1,\quad \vec{g}_k(x,y,z)=\f{1}{2\pi}\int_{2\pi\mathbb{T}}e^{ik(x-x_1)}\vec{g}(x_1,y,z)dx_1,\\
 &&{P}_k(x,y,z)=\f{1}{2\pi}\int_{2\pi\mathbb{T}}e^{ik(x-x_1)}{P}(x_1,y,z)dx_1,\quad  W_k=v^2_k+\kappa v^3_k,\quad U_k=v^3_k.
 \eeno
 Then $(W_k, U_k)$ satisfies
 \begin{align*}\left\{\begin{aligned}
     &-\nu\Delta W_k+ik(V-\lambda_i/k)W_k-a\nu^{\f13}W_k +(\partial_y+\kappa\partial_z)P_k\\ &\qquad+(g^2_k+\kappa g^3_k)+2\nu\nabla\kappa\cdot\nabla U_k+\nu(\Delta\kappa)U_k=0,\\
    &-\nu\Delta U_k+ik(V-\lambda_I/k)U_k-a\nu^{\f13}U_k+\partial_zP_k+g_k^3=0,\\
    &\Delta P_k=-2ik\partial_yVW_k,\,\, W_k|_{y=\pm1}=\partial_yW_k|_{y=\pm1}=U_k|_{y=\pm1}=0,\\
    &\partial_xW_k=ikW_k,\,\, \partial_xU_k=ikU_k,\,\, \partial_xP_k=ikP_k.
  \end{aligned}\right.\end{align*}
For $k\neq 0$,  we apply Proposition \ref{prop:res-full} to obtain
  \begin{align*}
&\nu^{\f{1}{3}}\big(\|\partial_x^2U_k\|_{L^2}^2+\|\partial_x(\partial_z-\kappa\partial_y)U_k\|_{L^2}^2\big)+
\nu\big(\|\nabla\partial_x^2U_k\|_{L^2}^2+\|\nabla\partial_x(\partial_z-\kappa\partial_y)U_k\|_{L^2}^2\big)
\\&\quad+\nu^{\f{1}{3}} \|\partial_x\nabla W_k\|_{L^2}^2+\nu\|\partial_x\Delta W_k\|_{L^2}^2+ \nu^{\f{5}{3}} \|\partial_x\Delta U_k\|_{L^2}^2\leq C\nu^{-1}\big(\|\nabla (g^2_k+\kappa g^3_k)\|_{L^2}^2+\|\partial_x g^3_k\|_{L^2}^2\big).
\end{align*}
which gives by Plancherel's  theorem that
\begin{align*}
&\nu^{\f{1}{3}}\big(\|\partial_x^2U\|_{L^2}^2+\|\partial_x(\partial_z-\kappa\partial_y)U\|_{L^2}^2\big)+
\nu\big(\|\nabla\partial_x^2U\|_{L^2}^2+\|\nabla\partial_x(\partial_z-\kappa\partial_y)U\|_{L^2}^2\big)
\\&\quad+\nu^{\f{1}{3}} \|\partial_x\nabla W\|_{L^2}^2+\nu\|\partial_x\Delta W\|_{L^2}^2+\nu^{\f{5}{3}} \|\partial_x\Delta U\|_{L^2}^2\leq C\nu^{-1}\big(\|\nabla (g^3+\kappa g^2)\|_{L^2}^2+\|\partial_x g^3\|_{L^2}^2\big).
\end{align*}
Recalling that $W=v^2+\kappa v^3, U=v^3$ and the definition of $Z_{[V]}^1,\ Z_{[V]}^2$, we infer that
\begin{align*}
   \|{v}\|^2_{Z_{[V]}^2}&\leq C\|\vec{g}\|^2_{Z_{[V]}^1}=C\|(A_{[V]}+\lambda){v}\|^2_{Z_{[V]}^1}.
\end{align*}

This finished the proof of the proposition.
\end{proof}\smallskip

Therefore, the spectrum of $A_{[V]}$ in the subspace ${J}_{2,0}^{(2)}(\Omega)\cap\mathcal{H}$ does not contain the region $\big\{\lambda\in\C| \mathbf{Re}(\lambda)\in[-\epsilon_1\nu^{\f13},0]\big\}$, so does that of $A_{[y+s(V-y)]}$ for $s\in[0,1],$ which implies $m(\nu,y+s(V-y))\not\in(-\epsilon_1\nu^{\f13},0). $ Moreover $\mathbb{P}\Delta $ is a dissipative self-adjoint operator with compact resolvent, and $A_{[y+s(V-y)]}$ for $s\in[0,1]$ form a continuous family of operators with relative compact perturbations. In particular, the spectrum of $A_{[y+s(V-y)]} $ is always discrete and depends continuously on $s$, and $m(\nu,y+s(V-y)) $ is a continuous function of $s$. We also know that $m(\nu,y)<0 $, thus,
\beno
m(\nu,y)\leq -\epsilon_1\nu^{\f13},\quad m(\nu,y+s(V-y))\leq -\epsilon_1\nu^{\f13}\quad \text{for}\,\, s\in[0,1].
\eeno
Due to  $m(\nu,V)\leq -\epsilon_1\nu^{\f13}$,  there holds that for  $\mu\in(0,-m(\nu,V))$,
\beno
\|e^{A_{[V]}t}f\|_{H^1}\leq C(\nu,\mu,V)e^{-\mu t}\|f\|_{H^1}\quad \text{for any}\,\, f\in {J}_{2,0}^{(2)}(\Omega)\cap\mathcal{H},\ t>0.
\eeno

\subsection{Space-time estimate via freezing the coefficient in time}

Let $t_j=j\nu^{-\f13},\ I_j=[t_j,t_{j+1})\cap[0,T]$. We define
\beno
&&V_j(y,z)=y+\overline{u}^{1,0}(t_j,y,z),\quad \kappa_j=\partial_zV_j/\partial_yV_j,\quad A_j=A_{[V_j]},\\
&&u^{1,1}=y+\overline{u}^{1}-V_j,\quad t\in I_j,\quad j\in[0,\nu^{\f13}T)\cap\Z.
\eeno
Then $ e^{a\nu^{1/3}t}\sim e^{aj}$ for $t\in I_j.$  We denote
\beno
\|{v}\|_{Z_{j}^l}=\|{v}\|_{Z_{[V_j]}^l} \quad \text{for} \quad j\in[0,\nu^{\f13}T)\cap\Z,\ l\in\{1,2\}.
\eeno
Recall that  $u_{\neq}$ satisfies
 \begin{align*}
&\partial_tu_{\neq}-\nu\mathbb{P}\Delta u_{\neq}+Q_1(y,u_{\neq})+Q(\overline{u},u_{\neq})+\mathbb{P}(u_{\neq}\cdot\nabla u_{\neq})_{\neq}=0,
\end{align*}
We can write
\begin{align*}
Q_1(y,u_{\neq})+Q(\overline{u},u_{\neq})=&Q_1(y+\overline{u}^1,u_{\neq})+Q\big((0,\overline{u}^2,\overline{u}^3),u_{\neq}\big)\\
=&Q_1(V_j,u_{\neq})+Q_1(u^{1,1},u_{\neq})+Q\big((0,\overline{u}^2,\overline{u}^3),u_{\neq}\big).
\end{align*}
Then we have
\begin{align*}
&\partial_tu_{\neq}-A_j u_{\neq}+\vec{g}=0\quad\text{for}\quad t\in I_j,
\end{align*}
where
\beno
\vec{g}=Q_1(u^{1,1},u_{\neq})+Q((0,\overline{u}^2,\overline{u}^3),u_{\neq})+\mathbb{P}(u_{\neq}\cdot\nabla u_{\neq})_{\neq}.
\eeno
We define $\vec{g}_{(j)},\ \vec{u}_{[j]} $ for $j\in[0,\nu^{\f13}T)\cap\Z$ iteratively by solving
\beno
\vec{g}_{(j)}(t)=0\quad \text{for}\quad t\not\in I_j,\quad \vec{g}_{(j)}(t)=\vec{g}+\sum_{k=0}^{j-1}(A_k- A_j)\vec{u}_{[k]}\quad \text{for}\quad t\in I_j,
\eeno
\begin{align*}
&\partial_t\vec{u}_{[j]}-A_j\vec{u}_{[j]}+\vec{g}_{(j)}=0,\quad \vec{u}_{[j]}(t)\in{J}_{2,0}^{(3)}(\Omega)\cap\mathcal{H},\quad \text{for}\quad t\in [0,+\infty),\\
&\vec{u}_{[0]}(0)=P_{\neq}u(0),\quad\vec{u}_{[j]}(0)=0\quad \text{for}\quad j\in(0,\nu^{\f13}T)\cap\Z.
\end{align*}
Then we find that
\begin{align*}
u_{\neq}=\sum_{k=0}^{j}\vec{u}_{[k]}\quad \text{for}\quad t\in I_j,\ j\in[0,\nu^{\f13}T)\cap\Z,\quad \vec{u}_{[j]}(t)=0\quad \text{for}\quad 0\leq t<t_j.
\end{align*}

\begin{Proposition}\label{prop:TS-uj}
Let $\eps=\eps_1/8$ and $V_j$ satisfy \eqref{ass:V} for $j\in[0,\nu^{\f13}T)\cap\Z$. Then it holds that
\begin{align*}
&\|e^{4\epsilon\nu^{1/3}t}\vec{u}_{[j]}\|_{L^2Z_j^2}\leq C\|e^{4\epsilon\nu^{1/3}t}\vec{g}_{(j)}\|_{L^2Z_j^1}\leq Ce^{4\epsilon j}\|\vec{g}_{(j)}\|_{L^2Z_j^1}\quad \text{for}\quad j\in(0,\nu^{\f13}T)\cap\Z,\\
&\|e^{4\epsilon\nu^{1/3}t}\vec{u}_{[0]}\|_{L^2Z_0^2}\leq C\big(\|u(0)\|_{H^2}+\|e^{4\epsilon\nu^{1/3}t}\vec{g}_{(0)}\|_{L^2Z_0^1}\big)
\leq C\big(\|u(0)\|_{H^2}+\|\vec{g}\|_{L^2(I_0,Z_0^1)}\big).
\end{align*}
\end{Proposition}

\begin{proof}
Here we just establish a priori estimates of the solution under the assumption $e^{4\epsilon\nu^{1/3}t}\vec{u}_{[j]}\in\\ L^2(0,+\infty;Z_j^2) $. Rigorous justification could consult section 4.2 in \cite{BMM}.\smallskip

We decompose $\vec{u}_{[0]}={u}_{H}+{u}_{I}$, where ${u}_{H}$ and ${u}_{I}$ solve
\begin{align*}\left\{\begin{aligned}
  &(\partial_t-\nu\Delta +y\partial_x){u}_{I}+(u^2_{I},0,0)+\nabla P_{I}+\vec{g}_{(0)}=0,\\
  &\text{div}{u}_{I}=0,\quad \Delta P=-2\partial_xu^2_{I},\quad {u}_{I}|_{t=0}=0,\quad {u}_{I}|_{y=\pm1}=0,
    \end{aligned}\right.
\end{align*}
and \begin{align*}\left\{\begin{aligned}
  &(\partial_t-\nu\Delta +y\partial_x){u}_{H}+(u^2_{H},0,0)+\nabla P_{H}=0,\\
  &\text{div}{u}_{H}=0,\quad \Delta P=-2\partial_xu^2_{H},\quad {u}_{H}|_{t=0}=P_{\neq}u(0),\quad {u}_{H}|_{y=\pm1}=0.
    \end{aligned}\right.
\end{align*}
That is,
\begin{align*}\left\{\begin{aligned}
  &\partial_t{u}_{I}-A_0 {u}_{I}+\vec{g}_{(0)}=0,\\
  &\partial_t{u}_{H}-A_0{u}_{I}=0.
  \end{aligned}\right.
\end{align*}

For $\lambda\in\mathbb{R}$, let
\beno
&&{v}_j(\lambda,x,y,z)=\int_{0}^{\infty}\vec{u}_{[j]}(t) e^{-it\lambda+4\epsilon\nu^{\f13}t}dt\quad j\in(0,\nu^{\f13}T)\cap\mathbb{Z},\\
&&{v}_0(\lambda,x,y,z)=\int_{0}^{\infty}{u}_{I}(t) e^{-it\lambda+4\epsilon\nu^{\f13}t}dt,\quad \widetilde{\vec{g}}_{(j)}(\lambda)=\int_{0}^{\infty}\vec{g}_{(j)}(t)e^{-it\lambda+4\epsilon\nu^{\f13}t}\quad j\in[0,\nu^{\f13}T)\cap\mathbb{Z}.
\eeno
Then we have
\begin{align*}
   & (A_j-i\lambda+4\epsilon\nu^{\f13}){v}_j=\widetilde{\vec{g}}_{(j)},\quad {v}_j\in J^{(3)}_{2,0}(\Omega)\cap \mathcal{H},\quad j\in[0,\nu^{\f13}T)\cap\mathbb{Z}.
\end{align*}
It follows from Proposition \ref{prop:res-A} that
\begin{align*}
     \|{v}_j\|_{Z_j^2}\leq C\|(A_j-i\lambda+4\epsilon\nu^{\f13}){v}_j\|_{Z_j^1}= C\|\widetilde{\vec{g}}_{(j)}\|_{Z_j^1}.
  \end{align*}
By Plancherel's theorem, it holds that for $j\in[0,\nu^{\f13}T)\cap\mathbb{Z}$,
\begin{align*}
  \|e^{4\epsilon\nu^{1/3}t}\vec{u}_{[j]}\|_{L^2_tZ^2_j} &\leq C\|{v}_j\|_{L^2_{\lambda}Z_j^2}\leq C \|\widetilde{\vec{g}}_{(j)}\|_{L^2_{\lambda}Z_j^1}\leq C\|e^{4\epsilon\nu^{1/3}t}\vec{g}_{(j)}\|_{L^2Z_j^1}\leq Ce^{4\epsilon j}\|\vec{g}_{(j)}\|_{L^2Z^1_j}.
\end{align*}
This proves the first inequality of the lemma, and
\begin{align}\label{est: prior uI0}
  \|e^{4\epsilon\nu^{1/3}t}{u}_{I}\|_{L^2Z^2_j} \leq Ce^{4\epsilon j}\|\vec{g}_{(0)}\|_{L^2Z^1_0}.
\end{align}

Let $\omega^2_H=\partial_zu^1_{H}-\partial_xu^3_{H}.$ Due to  $\partial_xu^1_{H}+\partial_yu^2_{H}+\partial_zu^3_{H}=0,$ we know that $(\partial_x^2+\partial_z^2)u^3_{H}=-\partial_x\omega^2_H-\partial_z\partial_yu^2_{H}.$
 Using the fact that $\|(\partial_x,\partial_z)(f_1,f_2)\|_{L^2}^2=\|(\partial_zf_1-\partial_xf_2,\partial_x f_1+\partial_zf_2)\|_{L^2}^2$, we deduce that
 \begin{align*}
 &\|\partial_x(\partial_x,\partial_z)(u^1_{H},u^3_{H})\|^2_{L^2}=\|\partial_x\omega^2_H\|_{L^2}^2+\|\partial_x\partial_y u^2_{H}\|_{L^2}^2\leq\|\partial_x\omega_H^2\|_{L^2}^2+\|\partial_x\nabla u^2_{H}\|_{L^2}^2,\\&\|\partial_x\nabla(\partial_x,\partial_z)(u^1_{H},u^3_{H})\|^2_{L^2}=\|\partial_x\nabla\omega_H^2\|_{L^2}^2+\|\partial_x\nabla\partial_y u^2_{H}\|_{L^2}^2\leq\|\partial_x\nabla\omega_H^2\|_{L^2}^2+C\|\partial_x\Delta u^2_{H}\|_{L^2}^2,\\
    &\|(\partial_x,\partial_z)\Delta (u^1_{H},u^3_{H})\|_{L^2}^2= \|\Delta \omega^2_H\|_{L^2}^2+\|\partial_y\Delta u^2_{H}\|_{L^2}^2.
 \end{align*}
On the other hand,  for $j=0$, $V_j=y,\ \kappa_j=0.$ Thus,
\begin{align*}
   \|{u}_{H}\|^2_{Z_{0}^2} =&\nu^{\f13}\|\partial_x(\partial_x,\partial_z)u^3_{H}\|^2_{L^2} +\nu\|\partial_x\nabla(\partial_x,\partial_z)u^3_{H}\|^2_{L^2}+\nu^{\f13} \|\partial_x\nabla u^2_{H}\|_{L^2}^2\\&+\nu\|\partial_x\Delta u^2_{H}\|_{L^2}^2
   +\nu^{\f53}\|\partial_x\Delta u^3_{H}\|^2_{L^2}.
\end{align*}
This shows that
  \begin{align}
   \|{u}_H\|^2_{Z^2_{0}}&\leq C\big(\nu^{\f13}\|\partial_x\omega_H^2\|^2_{L^2}+\nu\|\partial_x\nabla \omega_H^2\|^2_{L^2}+\nu^{\f13} \|\partial_x\nabla u^2_{H}\|_{L^2}^2+\nu\|\partial_x\Delta u^2_{H}\|_{L^2}^2\big)
   \label{Z01 control}\\
   &\quad+C\nu^{\f53}\big(\|\Delta \omega_H^2\|_{L^2}+\|\partial_y\Delta u^2_{H}\|_{L^2}\big).\nonumber
 \end{align}

 Recall  that ${u}_H$ satisfies
\begin{align*}\left\{\begin{aligned}
  &(\partial_t-\nu\Delta +y\partial_x){u}_{H}+(u^2_{H},0,0)+\nabla P_H=0,\\
  &\text{div}{u}_{H}=0,\ \Delta P=-2\partial_xu^2_{H},\ {u}_H|_{t=0}=P_{\neq}u(0).
  \end{aligned}\right.
\end{align*}
Then $(\omega_H^2,\Delta u^2_{H})$ satisfies
\begin{align*}
\left\{\begin{aligned}
  &\partial_t(\Delta u^2_{H})-\nu\Delta(\Delta u^2_{H})+y\partial_x(\Delta u^2_{H})=0,\\
  &u^2_{H}|_{y=\pm1}=\partial_yu^2_{H}|_{y=\pm1}=0,\ u^2_{H}|_{t=0}=u_{\neq}^2(0),\\
  &\partial_t\omega_H^2-\nu\Delta\omega_H^2+y\partial_x\omega_H^2+\partial_zu^2_{H}=0,\\
  &\omega_H^2|_{y=\pm1}=0,\ \omega_H^2|_{t=0}=(\partial_zu_{\neq}^1(0)-\partial_xu_{\neq}^3(0)).
  \end{aligned}\right.
\end{align*}
Let $\hat{\Delta}=\hat{\Delta}_{k,\ell}:=\partial_y^2-k^2-\ell^2,\ \eta=\sqrt{k^2+\ell^2},$ and
\beno
u^j_{k,\ell}(y)=\int_{\mathbb{T}^2}u^j_{H}(x,y,z)e^{-i(kx+\ell z)}dxdz,\quad \omega^2_{k,\ell}(y)=\int_{\mathbb{T}^2}\omega_H^2(x,y,z)e^{-i(kx+\ell z)}dxdz.
\eeno
Taking Fourier transformation in $x,z$, we obtain
\begin{align*}\left\{\begin{aligned}
  &\partial_t(\hat{\Delta}u^2_{k,\ell})-\nu\hat{\Delta}(\hat{\Delta}u^2_{k,\ell}) +iky(\hat{\Delta}u^2_{k,\ell})=0,\\
  &u^2_{k,\ell}|_{y=\pm1}=\partial_yu^2_{k,\ell}|_{y=\pm1}=0,\\
  &\partial_t\omega^2_{k,\ell}-\nu\hat{\Delta}\omega^2_{k,\ell} +iky\omega^2_{k,\ell}+i\ell u^2_{k,\ell}=0,\\
  &\omega^2_{k,\ell}|_{y=\pm1}=0.
  \end{aligned}\right.
\end{align*}
By Proposition \ref{prop:TS-nav}, we have
\begin{align*}
   & \nu\|e^{a\nu^{\f13}t}(\partial_y,\eta) \omega^2_{k,\ell}\|_{L^2L^2}^2+\nu^{\f13}\|e^{a\nu^{\f13}t}\omega^2_{k,\ell}\|_{L^2L^2}^2
  \leq C\Big(\|\omega^2_{k,\ell}(0)\|_{L^2}^2+|\eta k|^{-1}\|e^{a\nu^{\f13}t}(\partial_y,\eta)i\ell u^{2}_{k,\ell}\|_{L^2L^2}^2\Big),\\
   & \nu\|e^{a\nu^{\f13}t}(\partial_y,\eta) (k,\ell)\omega^2_{k,\ell}\|_{L^2L^2}^2+\nu^{\f13}\|e^{a\nu^{\f13}t}(k,\ell)\omega^2_{k,\ell}\|_{L^2L^2}^2
   \\&\quad\leq C\Big((k^2+\ell^2)\|\omega^2_{k,\ell}(0)\|_{L^2}^2+\min((\nu \eta^2)^{-1} ,(\nu k^2)^{-1/3})\|e^{a\nu^{\f13}t}(k,\ell)i\ell u^{2}_{k,\ell}\|_{L^2L^2}^2\Big),
\end{align*}
and
\begin{align*}
   \nu\|e^{a\nu^{\f13}t}\partial_y^2 \omega^2_{k,\ell}\|_{L^2L^2}^2\leq& C\nu^{-\f23}|k|^{\f23}\Big(\|\omega^2_{k,\ell}(0)\|_{L^2}^2+|\eta k|^{-1}\|e^{a\nu^{\f13}t}(\partial_y,\eta)i\ell u^2_{k,\ell}\|_{L^2L^2}^2\Big)\\
   &\quad+C\|\partial_y\omega^2_{k,\ell}(0)\|_{L^2}^2.
\end{align*}
Noting that  $\min((\nu \eta^2)^{-1} ,(\nu k^2)^{-1/3})\leq \nu^{-2/3} \eta^{-1}|k|^{-1/3}$, we deduce that
\begin{align}\label{est: u3 prior1}
   &\nu\|e^{a\nu^{\f13}t}k(\partial_y,\eta) \omega^2_{k,\ell}\|_{L^2L^2}^2+\nu^{\f13}\|e^{a\nu^{\f13}t}k \omega^2_{k,\ell}\|_{L^2L^2}^2+\nu^{\f53}\|e^{a\nu^{\f13}t}(\partial_y^2-\eta^2) \omega^2_{k,\ell}\|_{L^2L^2}^2\\
   &\quad\leq C\Big(\|\eta \omega^2_{k,\ell}(0)\|_{L^2}^2+\|\partial_y\omega^2_{k,\ell}(0)\|_{L^2}^2 +|k\eta|\|e^{a\nu^{\f13}t}(\partial_y,\eta)u^2_{k,\ell}\|_{L^2L^2}^2\Big).\nonumber
\end{align}
By proposition \ref{prop:TS-non}, we have
\begin{align*}
   &\nu\big\|e^{a\nu^{\f13}t}\eta(\partial_y^2-\eta^2) u^2_{k,\ell}\big\|^2_{L^2L^2}+ |k\eta|\|e^{a\nu^{\f13}t}(\partial_y,\eta) u^2_{k,\ell}\|_{L^2L^2}^2
   +\nu^{\f32}\|e^{a\nu^{\f13}t}\partial_y(\partial_y^2-\eta^2)u^2_{k,\ell}\|^2_{L^2L^2}\\
   &\leq C\Big(\|\eta^{-1}\partial_y(\partial_y^2-\eta^2)u^2_{k,\ell}(0)\|_{L^2}^2+ \|(\partial_y^2-\eta^2)u^2_{k,\ell}(0)\|^2_{L^2}\Big),
\end{align*}
which along with  \eqref{est: u3 prior1} gives
\begin{align*}
   &\nu\|e^{a\nu^{\f13}t}k(\partial_y,\eta) \omega^2_{k,\ell}\|_{L^2L^2}^2+\nu^{\f13}\|e^{a\nu^{\f13}t}k \omega^2_{k,\ell}\|_{L^2L^2}^2+\nu^{\f53}\|e^{a\nu^{\f13}t}(\partial_y^2-\eta^2)  \omega^2_{k,\ell}\|_{L^2L^2}^2\\
   &+\nu\big\|e^{a\nu^{\f13}t}\eta(\partial_y^2-\eta^2) u^2_{k,\ell}\big\|^2_{L^2L^2}+ |k\eta|\|e^{a\nu^{\f13}t}(\partial_y,\eta) u^2_{k,\ell}\|_{L^2L^2}^2
   +\nu^{\f32}\|e^{a\nu^{\f13}t}\partial_y(\partial_y^2-\eta^2)u^2_{k,\ell}\|^2_{L^2L^2}\\
   &\leq C\Big(\|\eta \omega^2_{k,\ell}(0)\|_{L^2}^2 +\|(\partial_y^2-\eta^2)(u^1_{k,\ell},u^3_{k,l})(0)\|_{L^2}^2+ \|(\partial_y^2-\eta^2)u^2_{k,\ell}(0)\|^2_{L^2}\Big).
\end{align*}
Here we used  $iku^1_{k,\ell}+\partial_yu^2_{k,\ell}+i\ell u^3_{k,\ell}=0.$ This shows that
\begin{align*}
   &\nu\|e^{a\nu^{\f13}t}\partial_x\nabla \omega_H^2\|_{L^2L^2}^2+\nu^{\f13}\|e^{a\nu^{\f13}t}\partial_x \omega_H^2\|_{L^2L^2}^2+\nu^{\f53}\|e^{a\nu^{\f13}t}\partial_y\nabla \omega_H^2\|_{L^2L^2}^2+\nu\big\|e^{a\nu^{\f13}t}\partial_x\Delta u^2_{k,l}\big\|^2_{L^2L^2}\\
   &\quad+\|e^{a\nu^{\f13}t}\partial_x\nabla u^2_{H}\|_{L^2L^2}^2
   +\nu^{\f32}\|e^{a\nu^{\f13}t}\partial_y\Delta u^2_{H}\|^2_{L^2L^2}  \leq C\|u(0)\|^2_{H^2}.
\end{align*}
Thanks to \eqref{Z01 control} and $\nu^{\f53}\leq\nu^{\f32},$ and taking $a=4\epsilon\le \eps_1,$ we conclude that
\begin{align*}
   &\|e^{4\epsilon\nu^{\f13}t}{u}_{H}\|_{L^2Z^2_0}\leq C\|u(0)\|_{H^2},
\end{align*}
from which and \eqref{est: prior uI0}, we deduce that
\begin{align*}
   \|e^{4\epsilon\nu^{\f13}t}\vec{u}_{[0]}\|_{L^2Z^2_0}&\leq \|e^{4\epsilon\nu^{\f13}t}{u}_{I}\|_{L^2Z^2_0}+\|e^{4\epsilon\nu^{\f13}t}{u}_{H}\|_{L^2Z^2_0}\leq C\big(\|u(0)\|_{H^2}+\|e^{4\epsilon\nu^{\f13}t}\vec{g}_{(0)}\|_{L^2Z^1_0}\big)\\
   &\leq C\big(\|u(0)\|_{H^2}+\|\vec{g}\|_{L^2(I_0,Z^1_0)}\big).
\end{align*}

This proves the proposition.\end{proof}

The following lemma gives some important properties of $V_j$.

\begin{Lemma}\label{lem:Vj}
Let $V_j$ satisfy \eqref{ass:V} for $j\in[0,\nu^{\f13}T)\cap\Z$. For $j,k\in[0,\nu^{\f13}T)\cap\Z,\ {v}=(v^1,v^2,v^3)\in{J}_{2,0}^{(3)}(\Omega)\cap\mathcal{H},$ we have
\begin{align*}
&\|V_j-V_k\|_{H^2}+\|\kappa_j-\kappa_k\|_{H^1}\leq C\nu^{\f23}|j-k|E_1,\quad \|\kappa_j-\kappa_k\|_{H^2}\leq C\nu^{\f13}|j-k|^{\f12}E_1,\\&\|u^{1,1}\|_{H^2}\leq C\nu^{\f23}E_1,\quad \|{v}\|_{Z_{j}^2}-\|{v}\|_{Z_{k}^2}\leq C|j-k|^{\f12}E_1\|{v}\|_{Z_{k}^2}.
\end{align*}
For $j\in[0,\nu^{\f13}T)\cap\Z,\ t\in I_j$ we have
\begin{align*}
&\|\kappa_j\nabla\overline{u}^3\|_{H^1}\leq CE_1E_2.
\end{align*}
For $j\in[0,\nu^{\f13}T)\cap\Z,\ \vec{f}=(f^1,f^2,f^3)\in{H}_{0}^{1}(\Omega)$, we have
\begin{align*}&\nu^{1/2}\|\mathbb{P} \vec{f}\|_{Z_{j}^1}\leq C\big(\|\nabla f^2\|_{L^2}+\|(\partial_x,\partial_z)f^3\|_{L^2}+\|\partial_xf^1\|_{L^2}+\nu^{\f13}j^{\f12}\|\nabla f^3\|_{L^2}\big).
\end{align*}
\end{Lemma}
\begin{proof}
Since $\|\partial_yV_j\|_{L^\infty}\geq 1-\|\partial_y(V_j-y)\|_{L^\infty}\geq 1-C\|V_j-y\|_{H^4}\geq 1-C\varepsilon_0$,  we get $|\partial_yV_j|\geq1/2$ by taking $\veps_0$ small enough so that $C\varepsilon_0\leq1/2$.
Then we have
\begin{align*}
   \big\|\f{1}{\partial_yV_j}\big\|_{H^3}\leq C\max\big(\|(\partial_yV_j)^{-1}\|_{L^\infty}^{4},1\big)\big(\|\partial_yV_j\|_{H^3}^3+1)\leq C(1+\|V_j-y\|_{H^4}\big)^3\leq C.
\end{align*}
and then
\begin{align*}
   & \|\kappa_j\|_{H^3}\leq C\|\partial_zV_j\|_{H^3}\big\|\f{1}{\partial_yV_j}\big\|_{H^3}\leq CE_1.
\end{align*}

Thanks to $V_j-V_k=\bar{u}^{1,0}(t_j,y,z)-\bar{u}^{1,0}(t_k,y,z)=\int_{t_k}^{t_j} \partial_t\bar{u}^{1,0}(s,y,z)ds$, we deduce that
\begin{align*}
&\|V_j-V_k\|_{H^2}\leq |t_j-t_k|\|\partial_t\bar{u}^{1,0}\|_{L^\infty H^2}\leq C\nu^{-\f13}|j-k|\nu E_1\leq C\nu^{\f23}|j-k|E_1.
\end{align*}
Using the formula
\begin{align*}
  \kappa_j-\kappa_k =\f{\partial_zV_j}{\partial_yV_j}- \f{\partial_zV_k}{\partial_yV_k}=&\f{\partial_z(V_j-V_k)}{\partial_yV_j}+ \f{(\partial_zV_k)\partial_y(V_k-V_j)}{(\partial_yV_k)(\partial_yV_j)},
\end{align*}
we get
\begin{align*}
   \|\kappa_j-\kappa_k\|_{H^1}\leq& C\|V_j-V_k\|_{H^2}\left(\left\|\f{1}{\partial_yV_j}\right\|_{H^2}+ \left\|\f{\kappa_k}{\partial_yV_j}\right\|_{H^2}\right)\\ \leq&  C\|V_j-V_k\|_{H^2}\leq C\nu^{\f23}|j-k|E_1.
\end{align*}
By the interpolation, we get
\begin{align*}
 \|\kappa_j-\kappa_k\|_{H^2}\leq& C\|\kappa_j-\kappa_k\|_{H^1}^{\f12} \|\kappa_j-\kappa_k\|_{H^3}^{\f12}\leq C\|\kappa_j-\kappa_k\|_{H^1}^{\f12}(\|\kappa_j\|_{H^3}+\|\kappa_k\|_{H^3})^{\f12}\\
 \leq&C\|\kappa_j-\kappa_k\|_{H^1}^{\f12}E_1^{\f12}\leq C\nu^{\f13}|j-k|^{\f12}E_1.
\end{align*}
For $t\in I_j$, then we have $u^{1,1}=\bar{u}^1(t,y,z)-\bar{u}^{1,0}(t_j,y,z)=\bar{u}^{1,\neq}+\int_{t_j}^{t} \partial_t\bar{u}^{1,0}(s,y,z)dz$, so
\begin{align*}
   \|u^{1,1}\|_{H^2}&\leq \|\bar{u}^{1,\neq}\|_{H^2}+\int_{t_j}^{t}\|\partial_t\bar{u}^{1,0}(s)\|_{H^2}ds \leq E_{1,\neq}+|t-t_j|\|\partial_t\bar{u}^{1,0}(s)\|_{L^\infty H^2}\\
   &\leq E_{1,\neq}+\nu^{-\f13}\nu E_{1,0}\leq \nu^{\f23}E_1.
\end{align*}

Direct calculations show that
\begin{align*}
   &\nu^{\f16}\big(\|\partial_x(\partial_z-\kappa_j\partial_y)v^3\|_{L^2}- \|\partial_x(\partial_z-\kappa_k\partial_y)v^3\|_{L^2}\big)\leq \nu^{\f16}\|(\kappa_j-\kappa_k)\partial_y\partial_xv^3\|_{L^2}\\ &\leq \nu^{\f16} \|\kappa_j-\kappa_k\|_{L^\infty}\|\nabla\partial_xv^3\|_{L^2}\leq C\nu^{\f12}|j-k|^{\f12}E_1\|\nabla\partial_x^2v^3\|_{L^2}\leq C|j-k|^{\f12}E_1\|{v}\|_{Z_{k}^2},
\end{align*}
and
\begin{align*}
   &\nu^{\f12}\big(\|\nabla\partial_x(\partial_z-\kappa_j\partial_y)v^3\|_{L^2}- \|\nabla\partial_x(\partial_z-\kappa_k\partial_y)v^3\|_{L^2}\big)\leq \nu^{\f12}\|\nabla[(\kappa_j-\kappa_k)\partial_x\partial_yv^3]\|_{L^2}\\
  &\leq C\nu^{\f12}\|\kappa_j-\kappa_k\|_{H^2}\|\partial_x\partial_yv^3\|_{H^1}\leq C\nu^{\f56}|j-k|^{\f12}E_1\|\partial_x\Delta v^3\|_{L^2}\leq  C|j-k|^{\f12}E_1\|{v}\|_{Z_{k}^2},
\end{align*}
and
\begin{align*}
   &\nu^{\f16}(\|\partial_x\nabla(v^2+\kappa_jv^3)\|_{L^2}- \|\partial_x\nabla(v^2+\kappa_jv^3)\|_{L^2})\leq \nu^{\f16}\|\partial_x\nabla[(\kappa_j-\kappa_k)v^3]\|_{L^2}\\
   &\leq C\nu^{\f16}\|\kappa_j-\kappa_k\|_{H^2}\|\partial_xv^3\|_{H^1}\leq C\nu^{\f12}|j-k|^{\f12}E_1\|\nabla\partial^2_x v^3\|_{L^2}\leq C|j-k|^{\f12}E_1\|{v}\|_{Z_{k}^2},
\end{align*}
and
\begin{align*}
   &\nu^{\f12}(\|\partial_x\Delta(v^2+\kappa_jv^3)\|_{L^2}- \|\partial_x\Delta(v^2+\kappa_jv^3)\|_{L^2})\leq \nu^{\f12}\|\partial_x\Delta[(\kappa_j-\kappa_k)v^3]\|_{L^2}\\
  &\leq C\nu^{\f12}\|\kappa_j-\kappa_k\|_{H^2}\|\partial_xv^3\|_{H^2}\leq C\nu^{\f56}|j-k|^{\f12}E_1\|\partial_x\Delta v^3\|_{L^2}\leq C|j-k|^{\f12}E_1\|{v}\|_{Z_{k}^2}.
\end{align*}
This shows that
\begin{align*}
  \|{v}\|_{Z_j^2}-\|{v}\|_{Z_k^2} &\leq C|j-k|^{\f12}E_1\|{v}\|_{Z_k^2}.
\end{align*}

Since $|\kappa_j|\leq C\|\nabla\kappa_j\|_{L^\infty}(1-|y|)\leq CE_1(1-|y|)$, $V_0-y=\bar{u}^{1,0}|_{t=0}=0,$ we have $\partial_z V_j=0,\ \kappa_0=0,$ and  then
\beno
\|\kappa_j\|_{H^2}=\|\kappa_j-\kappa_0\|_{H^2}\leq C\nu^{\f13}|j|^{\f12}E_1= C(\nu t_j)^{\f12}E_1,
\eeno
which gives
    \begin{align*}
      |\kappa_j|\leq& CE_1\min((\nu+\nu t_j)^{\f12},(1-|y|))\leq CE_1\min((\nu^{\f23}+\nu t_j)^{\f12},(1-|y|))\\
      \leq &CE_1\min((\nu^{\f23}+\nu t)^{\f12},(1-|y|))\quad\text{for}\quad t\in I_j.
    \end{align*}
    Then we infer that for $t\in I_j$
    \begin{align*}
       \|\kappa_j \bar{u}^3\|_{H^2}\lesssim&\|\Delta(\kappa_j \bar{u}^3)\|_{L^2}+\|\kappa_j\bar{u}^3\|_{L^2}\lesssim \|\kappa_j\Delta\bar{u}^3\|_{L^2}+\|\nabla\kappa_j\cdot\nabla\bar{u}^3\|_{L^2} +\|(\Delta\kappa_j,\kappa_j)\bar{u}^3\|_{L^2}\\
       \lesssim& E_1\|\min((\nu^{\f23}+\nu t)^{\f12},(1-|y|))\Delta\bar{u}^3\|_{L^2}+\|\nabla\kappa_j\|_{H^2}\|\nabla\bar{u}^3\|_{L^2}\\
       &+\|(\Delta\kappa_j,\kappa_j)\|_{H^1}\|\bar{u}^3\|_{H^1}\leq CE_1E_2,
    \end{align*}
    and then
    \begin{align*}
    \|\kappa_j\nabla\bar{u}^3\|_{H^1}\leq C\big(\|\kappa_j\bar{u}^3\|_{H^2}+\|(\nabla\kappa_j)\bar{u}^3\|_{H^1}\big)
    \leq C\big(\|\kappa_j\bar{u}^3\|_{H^2}+\|\nabla\kappa_j\|_{H^2}\|\bar{u}^3\|_{H^1}\big) \leq CE_1E_2.
    \end{align*}

 For $\vec{f}\in H^1_0(\Omega)$, we know that
\begin{align*}\left\{\begin{aligned}
&\mathbb{P}\vec{f}=\vec{f}+\nabla p,\\
&\Delta p=-\text{div}{\vec{f}},\quad \partial_yp|_{y=\pm1}=0.
\end{aligned}\right.
\end{align*}
Without loss of generality, we may assume $\langle p,1\rangle=0.$ Due to $\kappa_0=0$, we have
\begin{align*}
\|\nabla (f^2+\kappa_jf^3)\|_{L^2}\leq& \|\nabla f^2\|_{L^2}+ C\|\kappa_j\|_{H^2}\|f^3\|_{H^1}\leq C\big( \|\nabla f^2\|_{L^2}+ \|\kappa_j-\kappa_0\|_{H^2}\|\nabla f^3\|_{L^2}\big)\\
\leq &C\big(\|\nabla f^2\|_{L^2}+ \nu^{\f13}j^{\f12}\|\nabla f^3\|_{L^2}\big),
\end{align*}
and
\begin{align*}
  \|\nabla (\partial_yp+\kappa_j\partial_zp)\|_{L^2}+\|\partial_x\partial_zp\|_{L^2}\leq C\big(\|p\|_{H^2}+\|\kappa_j\|_{H^2}\|\partial_zp\|_{H^1}\big)\leq C\|\Delta p\|_{L^2}\leq C\|\text{div} \vec{f}\|_{L^2},
\end{align*}
which show that
\begin{align*}
   \nu^{\f12}\|\mathbb{P}\vec{f}\|_{Z_j^1}&\leq C\big(\|\nabla (f^2+\kappa_j f^3)\|_{L^2}+\|\partial_xf^3\|_{L^2}+ \|\nabla (\partial_yp+\kappa_j\partial_zp)\|_{L^2}+\|\partial_x\partial_zp\|_{L^2}\big)\\
   &\leq C\big(\|\nabla f^2\|_{L^2}+ \nu^{\f13}j^{\f12}\|\nabla f^3\|_{L^2} + \|\partial_xf^3\|_{L^2}+\|\text{div} \vec{f}\|_{L^2}\big)\\
   &\leq C\big(\|\nabla f^2\|_{L^2}+\|(\partial_x,\partial_z)f^3\|_{L^2} +\|\partial_xf^1\|_{L^2}+\nu^{\f13}j^{\f12}\|\nabla f^3\|_{L^2}\big).
\end{align*}

This proves the lemma.
\end{proof}

\subsection{Estimates of quadratic form}

\begin{Lemma}\label{lem:Q-z1n}
It holds that  for $\partial_xf=0, {v}=(v^1,v^2,v^3)\in{J}_{2,0}^{(3)}(\Omega)\cap\mathcal{H},$
\beno
\|Q_1(f,{v})\|_{Z_{[V]}^1}\leq C\nu^{-\f{2}{3}}\|f\|_{H^2}\|{v}\|_{Z_{[V]}^2}.
\eeno
\end{Lemma}

\begin{proof}
 Since $Q_1(f, {v})=\mathbb{P}(f\partial_x {v}+({v}\cdot\nabla f,0,0))$, we know that
   \begin{align*}\left\{\begin{aligned}
     &Q_1(f,{v})= f\partial_x{v}+({v}\cdot\nabla f,0,0) +\nabla P,\\
     &\Delta P=-\text{div}\big( f\partial_x{v}+({v}\cdot\nabla f,0,0)\big)=-2\nabla f\cdot \partial_x{v},\\
     &Q_1(f,{v})^2|_{y=\pm1}=\partial_yP|_{y=\pm1}=0.\end{aligned}\right.
  \end{align*}
  Since
  \begin{align*}
     &\|\nabla(f\partial_xv^{2}+\kappa f\partial_xv^{3})\|_{L^2}\leq C\|f\|_{H^2}\|\partial_xv^{2}+\kappa\partial_xv^{3}\|_{H^1}\leq C\|f\|_{H^2}\|\nabla(\partial_xv^{2}+\kappa\partial_xv^{3})\|_{L^2},\\
     &\|\partial_x(fv^3)\|_{L^2}=\|f\partial_x(v^3)\|_{L^2}\leq \|f\|_{L^\infty}\|\partial_xv^3\|_{L^2}\leq C\|f\|_{H^2}\|\partial_xv^3\|_{L^2}.
  \end{align*}
 we infer that
  \begin{align*}
     \|f\partial_x{v}+({v}\cdot\nabla f,0,0)\|_{Z_{[V]}^1}\leq& C\nu^{-\f12}\big(\|\nabla(f\partial_xv^{2}+\kappa f\partial_xv^{3})\|_{L^2} +\|\partial_x(fv^3)\|_{L^2}\big)\\
     \leq& C\nu^{-\f12}\|f\|_{H^2}\big(\|\nabla(\partial_xv^2+\kappa\partial_xv^3)\|_{L^2} +\|\partial_x(fv^3)\|_{L^2}\big)\\
     \leq&   C\nu^{-\f23}\|f\|_{H^2}\|{v}\|_{Z_{[V]}^2}.
  \end{align*}

Thanks to $\partial_yP|_{y=\pm1}=0$,  we get by Lemma \ref{Lem: bil good deri}  and Lemma \ref{lem:Z norm} that
  \begin{align*}
     \|P\|_{H^2}\leq& C\|\Delta P\|_{L^2}\leq C\|\nabla f\cdot\partial_x{v}\|_{L^2}\leq C\big(\|(\partial_yf)\partial_xv^2\|_{L^2}+\|(\partial_zf)\partial_xv^3\|_{L^2}\big)\\
     \leq &C\|\nabla f\|_{H^1}\big(\|\partial_xv^2\|_{L^2}+ \|\partial_x(\partial_z-\kappa\partial_y)v^2\|_{L^2} +\|\partial_xv^3\|_{L^2}+ \|\partial_x(\partial_z-\kappa\partial_y)v^3\|_{L^2}\big)\\
     \leq &C\nu^{-\f16}\|f\|_{H^2}\|{v}\|_{Z_{[V]}^2},
  \end{align*}
which gives
  \begin{align*}
     &\|\nabla(\partial_yP+\kappa\partial_zP)\|_{L^2}\leq C(\|P\|_{H^2}+\|\kappa\|_{H^2}\|\partial_zP\|_{H^1} )\leq C\|P\|_{H^2}\leq C\nu^{-\f16}\|f\|_{H^2}\|{v}\|_{Z_{[V]}^2},\\
     &\|\partial_x\partial_zP\|_{L^2}\leq C\|P\|_{H^2}\leq C\nu^{-\f16}\|f\|_{H^2}\|{v}\|_{Z_{[V]}^2}.
  \end{align*}
This shows that
  \begin{align*}
     &\|\nabla P\|_{Z_{[V]}^{1}}\leq C\nu^{-\f12}\big(\|\nabla(\partial_yP+\kappa\partial_zP)\|_{L^2} +\|\partial_x\partial_zP\|_{L^2}\big)\leq C\nu^{-\f23}\|f\|_{H^2}\|{v}\|_{Z_{[V]}^2}.
  \end{align*}
Thus, we conclude that
  \begin{align*}
     \|Q_1(f, {v})\|_{Z_{[V]}^{1}}& \leq \|f\partial_x{v}+({v}\cdot\nabla f,0,0)\|_{Z_{[V]}^1}+\|\nabla P\|_{Z_{[V]}^{1}}\leq C\nu^{-\f23}\|f\|_{H^2}\|{v}\|_{Z_{[V]}^{2}}.
  \end{align*}
 \end{proof}

\begin{Lemma}\label{lem:Q-z23-n}For $j\in[0,\nu^{\f13}T)\cap\Z,\ {v}=(v^1,v^2,v^3)\in{J}_{2,0}^{(3)}(\Omega)\cap\mathcal{H}, t\in I_j,$ we have
\begin{align*}
\|Q((0,\overline{u}^2,\overline{u}^3), {v})\|_{Z_{j}^1}\leq C\nu^{-1}E_2\|{v}\|_{Z_{j}^2}.
\end{align*}
\end{Lemma}

\begin{proof}

Let $\alpha,\beta\in\{2,3\}$, $t\in I_j,$ and
    \begin{align*}
      &W^2_j=v^2+\kappa_jv^3,\\ &g_{\alpha,1}=(\bar{u}^2\partial_y+\bar{u}^3\partial_z)v^\alpha,\quad g_{\alpha,2}={v}\cdot\nabla \bar{u}^{\alpha},\\
      &G_{2,1}=g_{2,1}+\kappa_jg_{3,1}=(\bar{u}^2\partial_y+\bar{u}^3\partial_z)W^2_j -v^3(\bar{u}^2\partial_y+\bar{u}^3\partial_z)\kappa_j,\\
      &G_{2,2}=g_{2,2}+\kappa_jg_{3,2},\quad \Delta p^{(3)}=-2\partial_\alpha\bar{u}^{\beta}\partial_{\beta}v^{\alpha},\quad \partial_yp^{(3)}|_{y=\pm1}=0.
    \end{align*}
    Then we find that
    \begin{align*}
      Q((0,\bar{u}^2,\bar{u}^3),{v})&=\mathbb{P}\big( (\bar{u}^2\partial_y+\bar{u}^3\partial_z){v}+ (0,{v}\cdot\nabla \bar{u}^2,{v}\cdot\nabla\bar{u}^3)\big)\\
      &=\mathbb{P}\big(((\bar{u}^2\partial_y+\bar{u}^3\partial_z)v^1, g_{2,1}+g_{2,2}, g_{3,1}+g_{3,2})\big).
    \end{align*}
    and
    \begin{align*}
       \text{div}((\bar{u}^2\partial_y+\bar{u}^3\partial_z){v}+ (0,{v}\cdot\nabla \bar{u}^2,{v}\cdot\nabla\bar{u}^3))=2\partial_{\alpha}\bar{u}^{\beta} \partial_{\beta}v^{\alpha}.
    \end{align*}
Thus, we have
  \begin{align}\label{eq: Q fenliang}
     Q((0,\bar{u}^2,\bar{u}^3),{v})^2=g_{2,1}+g_{2,2}+\partial_y p^{(3)},\quad
     Q((0,\bar{u}^2,\bar{u}^3),{v})^3=g_{3,1}+g_{3,2}+\partial_z p^{(3)}.
  \end{align}

  Let us first  claim that
  \begin{align}\label{est: E5 bar23}
    &\|\nabla G_{2,1}\|_{L^2}+\|\partial_xg_{3,1}\|_{L^2}+\|\nabla g_{2,2}\|_{L^2}+\|(\partial_x,\partial_z)g_{3,2}\|_{L^2}+\|\nabla(\kappa_jg_{3,2}) \|_{L^2}\\
    &\qquad+\|\Delta p^{(3)}\|_{L^2}\leq C\nu^{-\f12}E_2\|{v}\|_{Z_{j}^2}.\nonumber
  \end{align}
Then by \eqref{eq: Q fenliang}, we have
  \begin{align*}
     &\|\nabla \big(Q((0,\bar{u}^2,\bar{u}^3),{v})^2+ \kappa_jQ((0,\bar{u}^2,\bar{u}^3),{v})^3\big)\|_{L^2}\\
     &= \|\nabla(G_{2,1}+G_{2,2}+(\partial_y+\kappa_j\partial_z)p^{(3)})\|_{L^2}\\
     &\leq \|\nabla G_{2,1}\|_{L^2}+\|\nabla g_{2,2}\|_{L^2}+ \|\nabla(\kappa_jg_{3,2})\|_{L^2}+C(1+\|\kappa_j\|_{H^3})\|p^{(3)}\|_{H^2}\\
     &\leq \|\nabla G_{2,1}\|_{L^2}+\|\nabla g_{2,2}\|_{L^2}+ \|\nabla(\kappa_jg_{3,2})\|_{L^2}+C\|\Delta p^{(3)}\|_{L^2}\leq C\nu^{-\f12}E_2\|{v}\|_{Z_{j}^2},
  \end{align*}
  and
  \begin{align*}
     & \|\partial_x(Q((0,\bar{u}^2,\bar{u}^3), {v})^3)\|_{L^2}\leq \|\partial_xg_{3,1}\|_{L^2}+\|\partial_xg_{3,2}\|_{L^2}+\|\partial_x\partial_zp^{(3)}\|_{L^2} \leq C\nu^{-\f12}E_2\|{v}\|_{Z_{j}^2}.
  \end{align*}
  This gives
  \beno
  \|Q((0,\bar{u}^2,\bar{u}^3), {v})\|_{L^2}\leq CE_{2}\|{v}\|_{Z_j^2}.
  \eeno

  It remains to prove  \eqref{est: E5 bar23}. By Lemma \ref{lem:u23-zero}, we have
  \begin{align*}
     \|\nabla G_{2,1}\|_{L^2}&\leq CE_2\big(\|W^2_j\|_{H^2}+ \|v^3\nabla\kappa_j\|_{H^1}\big)\leq CE_2\big(\|\Delta W^2_j\|_{L^2}+ \|v^3\|_{H^1}\|\nabla\kappa_j\|_{H^2}\big)\\
     &\leq CE_2\big(\|\partial_x\Delta W^2_j \|_{L^2}+ \|\nabla\partial_x^2v^3\|_{L^2}\big)\leq C\nu^{-\f12}E_2\|{v}\|_{Z_{j}^2}.
  \end{align*}
Thanks to $\partial_xg_{3,1}=(\bar{u}^2\partial_y+\bar{u}^3\partial_z)\partial_xv^3$, we get by Lemma \ref{lem:u23-zero} that
  \begin{align*}
    \|\partial_xg_{3,1}\|_{L^2} &\leq C\big(\|\bar{u}^2\|_{L^\infty}+\|\bar{u}^3\|_{L^\infty}\big)\|\partial_x\nabla v^3\|_{L^2}\leq CE_2\|\partial_x\nabla v^3\|_{L^2}\leq C\nu^{-\f12}E_2\|{v}\|_{Z_j^2},
  \end{align*}
We get by Lemma \ref{Lem: bil good deri}  that
  \begin{align*}
   \|\nabla g_{2,2}\|_{L^2}\leq& \|\nabla(v^{\alpha}\partial_{\alpha}\bar{u}^2)\|_{L^2} \\ \leq& C\|(\partial_y,\partial_z)\bar{u}^2\|_{H^1}(\| (v^2,v^3)\|_{H^1}+\| (\partial_z-\kappa_j\partial_y) (v^2,v^3)\|_{H^1})\\
   \leq &CE_2\big(\| \nabla(v^2,v^3)\|_{L^2}+\| \nabla(\partial_z-\kappa_j\partial_y) (v^2,v^3)\|_{L^2}\big)\leq C\nu^{-\f12}E_2\|{v}\|_{Z_j^2}.
  \end{align*}
  We write
  \begin{align*}
   g_{3,2}={v}\cdot\nabla\bar{u}^3=(v^2\partial_y+v^3\partial_z) \bar{u}^3= W^2_j\partial_y\bar{u}^3+v^3(\partial_z\bar{u}^3-\kappa_j\partial_y\bar{u}^3).
  \end{align*}
  Then by \eqref{f8} and Lemma \ref{lem:u23-zero}, we get
  \begin{align*}
     \|(\partial_x,\partial_z)(W^2_j\partial_y\bar{u}^3)\|_{L^2}&\leq \|((\partial_x,\partial_z)W^2_j)\partial_y\bar{u}^3\|_{L^2} +\|W^2(\partial_z\partial_y\bar{u}^3)\|_{L^2}\\ &\leq C\|(\partial_x,\partial_z)W^2_j\|_{H^1}\|(\partial_z,1)\partial_y\bar{u}^3\|_{L^2} +C\|W^2_j\|_{H^2}\|\partial_z\partial_y\bar{u}^3\|_{L^2}\\
     &\leq CE_2\|\Delta W^2_j\|_{L^2}\leq C\nu^{-\f12}E_2\|{v}\|_{Z_j^2},
  \end{align*}
 Due to Lemma \ref{lem:Vj},  $\|\kappa_j\nabla \bar{u}^3\|_{H^1}\leq CE_2$,  which along with Lemma \ref{lem:u23-zero} gives  $\|\partial_z\bar{u}^3-\kappa_j\partial_y\bar{u}^3\|_{H^1}\leq CE_2$, and then
\begin{align*}
   \|\nabla(\kappa_j W^2_j\partial_y\bar{u}^3)\|_{L^2}&\leq C\|W^2_j\|_{H^2}\|\kappa_j\partial_y\bar{u}^3\|_{H^1}\leq C\|\Delta W^2_j\|_{L^2}E_2\leq C\nu^{-\f12}E_2\|{v}\|_{Z_j^2}.
\end{align*}
By Lemma \ref{Lem: bil good deri},  we have
  \begin{align*}
     \|v^3(\partial_z\bar{u}^3-\kappa_j\partial_y\bar{u}^3)\|_{H^1}&\leq C\big(\|v^3\|_{H^1}+\|(\partial_z-\kappa_j\partial_y)v^3\|_{H^1}\big) \|\partial_z\bar{u}^3-\kappa_j\partial_y\bar{u}^3\|_{H^1}\nonumber\\
     &\leq CE_2\big(\|\nabla v^3\|_{L^2}+\|\nabla(\partial_z-\kappa_j\partial_y)v^3\|_{L^2}\big)\leq C\nu^{-\f12}E_2\|{v}\|_{Z_j^2},
  \end{align*}
which gives
\begin{align*}
  \|\nabla[\kappa_j v^3(\partial_z\bar{u}^3-\kappa_j\partial_z\bar{u}^3)]\|_{L^2} &\leq C\|\kappa_j\|_{H^2}\|v^3(\partial_z\bar{u}^3-\kappa_j\partial_z\bar{u}^3)\|_{H^1}
  \leq C\nu^{-\f12}E_2\|{v}\|_{Z_j^2}.
\end{align*}

Summing up,  we obtain
\begin{align*}
   &\|(\partial_x,\partial_z)g_{3,2}\|_{L^2}+\|\nabla(\kappa_j g_{2,3})\|_{L^2}\\
   &\leq C\big(\|(\partial_x,\partial_z)(W^2_j\partial_y\bar{u}^3)\|_{L^2} +\|v^3(\partial_z\bar{u}^3-\kappa_j\partial_z\bar{u}^3)\|_{H^1}
   +\|\nabla(\kappa_j W^2_j\partial_y\bar{u}^3)\|_{L^2} \\ &\qquad+\|\nabla[\kappa_j v^3(\partial_z\bar{u}^3-\kappa_j\partial_z\bar{u}^3)]\|_{L^2}\big)\leq C\nu^{-\f12}E_2\|{v}\|_{Z_j^2}.
\end{align*}
Using $\partial_y\bar{u}^2+\partial_z\bar{u}^3=0$, we may write
  \begin{align*}
     \Delta p^{(3)}&=-2\partial_\alpha\bar{u}^{\beta}\partial_\beta v^{\alpha}=-2\partial_\beta(\partial_\alpha\bar{u}^\beta v^{\alpha})=-2\big(\partial_yg_{2,2}+\partial_zg_{3,2}\big),
  \end{align*}
  therefore,
  \begin{align*}
     \|\Delta p^{(3)}\|_{L^2}&\leq C(\|\nabla g_{2,2}\|_{L^2} +\|\partial_z g_{3,2}\|_{L^2})\leq C\nu^{-\f12}E_2\|{v}\|_{Z_j^2}.
  \end{align*}
This proves \eqref{est: E5 bar23} and the lemma.
\end{proof}

\subsection{Estimate of $E_5$}

 Let $N=\max( [0,\nu^{\f13}T)\cap\Z)$ and we define
 \begin{align*}
&E_6^2=\sum_{j=0}^{N}e^{6\epsilon j}\|u_{\neq}\|_{L^2(I_j,Z_j^2)}^2,
\end{align*}
By Lemma \ref{lem:Z norm}, we have
\begin{align*}
&\nu^{\f{1}{3}}\big(\|\partial_x^2u_{\neq}^2\|_{L^2}^2+\|\partial_x^2u_{\neq}^3\|_{L^2}^2\big)\leq C\|u_{\neq}\|_{Z_j^2}^2,
\end{align*}
and then
\begin{align*}
  E^2_5=&\nu^{1/3}\|e^{3\epsilon\nu^{1/3}t}\partial^2_xu^2_{\neq}\|_{L^2L^2}^2
  +\nu^{1/3}\|e^{3\epsilon\nu^{1/3}t}\partial_x^2u^3_{\neq}\|^2_{L^2L^2}\\ \leq& C\sum_{j=0}^{N}\nu^{1/3}e^{6\epsilon j}\big(\|\partial^2_xu^2_{\neq}\|_{L^2(I_j,L^2)}^2
  +\|\partial_x^2u^3_{\neq}\|^2_{L^2(I_j,L^2)}\big)\\
  \leq& C\sum_{j=0}^{N}e^{6\epsilon j}\|u_{\neq}\|_{L^2(I_j,Z_j^2)}^2=CE_6^2.
\end{align*}

\begin{Proposition}\label{prop:E5}
It holds that
\begin{align*}
E_5^2\le E_6^2\leq C\|u(0)\|_{H^2}^2+C\big(E_1^2+\nu^{-2}E_2^2\big)E_6^2+C\nu^{-2}E_3^4.
\end{align*}

\end{Proposition}

\begin{proof}
For $j\in[0,\nu^{\f13}T)\cap\Z$, let
\beno
a_j=\sum_{k=0}^j(j-k+1)\|\vec{u}_{[k]}\|_{L^2(I_j,Z_{j}^2)},\quad  b_j=e^{-4\epsilon j}\|e^{4\epsilon\nu^{1/3}t}\vec{u}_{[j]}\|_{L^2Z_j^2}.
\eeno

Thanks to the definition of $\vec{g}_{(j)}$, we have
\begin{align*}
&\|\vec{g}_{(j)}\|_{L^2Z_j^1}\leq \|\vec{g}\|_{L^2(I_j,Z_j^1)}+\sum_{k=0}^{j-1}\|(A_k- A_j)\vec{u}_{[k]}\|_{L^2(I_j,Z_j^1)}.
\end{align*}
Notice that
\begin{align*}
&(A_k- A_j)\vec{u}_{[k]}=(A_{[V_k]}- A_{[V_j]})\vec{u}_{[k]}=Q_1(V_j-V_k,\vec{u}_{[k]}).
\end{align*}
We get  by Lemma \ref{lem:Q-z1n}  and Lemma \ref{lem:Vj}  that
\begin{align*}
&\|(A_k- A_j)\vec{u}_{[k]}\|_{Z_j^1}\leq C\nu^{-\f{2}{3}}\|V_j-V_k\|_{H^2}\|\vec{u}_{[k]}\|_{Z_{j}^2}\leq C|j-k|E_1\|\vec{u}_{[k]}\|_{Z_{j}^2}.
\end{align*}
Then we infer from Proposition \ref{prop:TS-uj} that
\beno
b_0\leq C\big(\|u(0)\|_{H^2}+\|\vec{g}\|_{L^2(I_0,Z_0^1)}\big),
\eeno
and  for $j\in(0,\nu^{\f13}T)\cap\Z$,
\begin{align*}
b_j&\leq C\|\vec{g}_{(j)}\|_{L^2Z_j^1}\leq C\|\vec{g}\|_{L^2(I_j,Z_j^1)}+C\sum_{k=0}^{j-1}\|(A_k- A_j)\vec{u}_{[k]}\|_{L^2(I_j,Z_j^1)}
\\&\leq C\|\vec{g}\|_{L^2(I_j,Z_j^1)}+C\sum_{k=0}^{j-1}|j-k|E_1\|\vec{u}_{[k]}\|_{L^2(I_j,Z_j^1)}\leq C\big(\|\vec{g}\|_{L^2(I_j,Z_j^1)}+E_1a_j\big).
\end{align*}

For $j,k\in[0,\nu^{\f13}T)\cap\Z$,  we get by  Lemma \ref{lem:Vj} that
\begin{align*}
\|\vec{u}_{[k]}\|_{L^2(I_j,Z_{j}^2)}&\leq \big(1+C|j-k|^{\f12}E_1\big)\|\vec{u}_{[k]}\|_{L^2(I_j,Z_{k}^2)}\leq \big(1+C|j-k|^{\f12}E_1\big)e^{-4\epsilon j}\|e^{4\epsilon\nu^{1/3}t}\vec{u}_{[k]}\|_{L^2Z_{k}^2}\\&\leq(1+C|j-k|^{\f12}E_1)e^{-4\epsilon (j-k)}b_k,
\end{align*}
which gives
\begin{align*}
a_j&=\sum_{k=0}^j(j-k+1)\|\vec{u}_{[k]}\|_{L^2(I_j,Z_{j}^2)}\leq\sum_{k=0}^j(j-k+1)(1+C|j-k|^{\f12}E_1)e^{-4\epsilon (j-k)}b_k\\&\leq C\sum_{k=0}^j(j-k+1)^{\f32}e^{-4\epsilon (j-k)}b_k,
\end{align*}
and hence,
\beno
a_j^2 \leq C\sum_{k=0}^j(j-k+1)^{5}e^{-8\epsilon (j-k)}b_k^2\leq C\sum_{k=0}^je^{-7\epsilon (j-k)}b_k^2.
\eeno
Thus, we obtain
\begin{align*}
\sum_{j=0}^{N}e^{6\epsilon j}a_j^2&\leq C\sum_{j=0}^{N}e^{6\epsilon j}\sum_{k=0}^je^{-7\epsilon (j-k)}b_k^2=C\sum_{k=0}^{N}\sum_{j=k}^{N}e^{6\epsilon k-\epsilon (j-k)}b_k^2\\ &\leq C\sum_{j=0}^{N}e^{6\epsilon j}b_j^2 \leq C\|u(0)\|_{H^2}^2+C\sum_{j=0}^{N}e^{6\epsilon j}\big(\|\vec{g}\|_{L^2(I_j,Z_j^1)}^2+E_1^2a_j^2\big).
\end{align*}
Then we conclude that \begin{align*}
E_6^2\leq\sum_{j=0}^{N}e^{6\epsilon j}a_j^2&\leq  C\|u(0)\|_{H^2}^2+C\sum_{j=0}^{N}e^{6\epsilon j}\|\vec{g}\|_{L^2(I_j,Z_j^1)}^2+CE_1^2E_6^2.
\end{align*}

Thanks to the definition of $\vec{g} $, we have
\begin{align*}
\|\vec{g}\|_{L^2(I_j,Z_j^1)}\leq& \|Q_1(u^{1,1},u_{\neq})\|_{L^2(I_j,Z_j^1)}+\|Q((0,\overline{u}^2,\overline{u}^3),u_{\neq})\|_{L^2(I_j,Z_j^1)}\\&+\|\mathbb{P}(u_{\neq}\cdot\nabla u_{\neq})_{\neq}\|_{L^2(I_j,Z_j^1)}.
\end{align*}
By Lemma \ref{lem:Q-z1n}, Lemma \ref{lem:Vj} and Lemma \ref{lem:Q-z23-n}, we have
\begin{align*}
&\|Q_1(u^{1,1},u_{\neq})\|_{L^2(I_j,Z_j^1)}\leq C\nu^{-\f{2}{3}}\|u^{1,1}\|_{L^{\infty}H^2}\|u_{\neq}\|_{L^2(I_j,Z_j^2)}\leq CE_1\|u_{\neq}\|_{L^2(I_j,Z_j^2)},\\ &\|Q((0,\overline{u}^2,\overline{u}^3),u_{\neq})\|_{L^2(I_j,Z_j^1)}\leq C\nu^{-1}E_2\|u_{\neq}\|_{L^2(I_j,Z_j^2)},
\end{align*}
and
\begin{align*}&\nu^{1/2}\|\mathbb{P}(u_{\neq}\cdot\nabla u_{\neq})_{\neq}\|_{Z_{j}^1}\leq C(\|\nabla(u_{\neq}\cdot\nabla u_{\neq}^2)\|_{L^2}+\|(\partial_x,\partial_z)(u_{\neq}\cdot\nabla u_{\neq}^3)\|_{L^2}\\&\qquad+\|\partial_x(u_{\neq}\cdot\nabla u_{\neq}^1)\|_{L^2}+\nu^{\f13}j^{\f12}\|\nabla(u_{\neq}\cdot\nabla u_{\neq}^3)\|_{L^2})\\ &\leq Ce^{\epsilon j}\Big(\|\nabla(u_{\neq}\cdot\nabla u_{\neq}^2)\|_{L^2}+\|(\partial_x,\partial_z)(u_{\neq}\cdot\nabla u_{\neq}^3)\|_{L^2}+\|\partial_x(u_{\neq}\cdot\nabla u_{\neq}^1)\|_{L^2}\\&\qquad+\nu^{\f13}\|\nabla(u_{\neq}\cdot\nabla u_{\neq}^3)\|_{L^2}\Big).
\end{align*}
Summing up, we get by Lemma \ref{lem:int-nn}  that
\begin{align*}
&\sum_{j=0}^{N}e^{6\epsilon j}\|\vec{g}\|_{L^2(I_j,Z_j^1)}^2\leq  C\sum_{j=0}^{N}e^{6\epsilon j}\Big(\|Q_1(u^{1,1},u_{\neq})\|_{L^2(I_j,Z_j^1)}^2+\|Q((0,\overline{u}^2,\overline{u}^3),u_{\neq})\|_{L^2(I_j,Z_j^1)}^2\\&\qquad+\|\mathbb{P}(u_{\neq}\cdot\nabla u_{\neq})_{\neq}\|_{L^2(I_j,Z_j^1)}^2\Big)\\ &\leq C\sum_{j=0}^{N}e^{6\epsilon j}\big(E_1^2\|u_{\neq}\|_{L^2(I_j,Z_j^2)}^2+\nu^{-2}E_2^2\|u_{\neq}\|_{L^2(I_j,Z_j^2)}^2\big)+\sum_{j=0}^{N}e^{8\epsilon j}\nu^{-1}\Big(\|\nabla(u_{\neq}\cdot\nabla u_{\neq}^2)\|_{L^2(I_j,L^2)}^2\\&\qquad+\|(\partial_x,\partial_z)(u_{\neq}\cdot\nabla u_{\neq}^3)\|_{L^2(I_j,L^2)}^2+\|\partial_x(u_{\neq}\cdot\nabla u_{\neq}^1)\|_{L^2(I_j,L^2)}^2+\nu^{\f23}\|\nabla(u_{\neq}\cdot\nabla u_{\neq}^3)\|_{L^2(I_j,L^2)}^2\Big)\\ &\leq C\big(E_1^2+\nu^{-2}E_2^2\big)E_6^2+C\nu^{-1}\Big(\|e^{4\epsilon\nu^{\f13}t}\nabla(u_{\neq}\cdot\nabla u_{\neq}^2)\|_{L^2L^2}^2+\|e^{4\epsilon\nu^{\f13}t}(\partial_x,\partial_z)(u_{\neq}\cdot\nabla u_{\neq}^3)\|_{L^2L^2}^2\\&\qquad+\|e^{4\epsilon\nu^{\f13}t}\partial_x(u_{\neq}\cdot\nabla u_{\neq}^1)\|_{L^2L^2}^2+\nu^{\f23}\|e^{4\epsilon\nu^{\f13}t}\nabla(u_{\neq}\cdot\nabla u_{\neq}^3)\|_{L^2L^2}^2\Big)\\ &\leq C\big(E_1^2+\nu^{-2}E_2^2\big)E_6^2+C\nu^{-2}E_3^4.
\end{align*}
Then we infer that
\begin{align*}
E_6^2&\leq C\|u(0)\|_{H^2}^2+C\sum_{j=0}^{N}e^{6\epsilon j}\|\vec{g}\|_{L^2(I_j,Z_j^1)}^2+CE_1^2E_6^2\\
&\leq C\|u(0)\|_{H^2}^2+C\big(E_1^2+\nu^{-2}E_2^2\big)E_6^2+C\nu^{-2}E_3^4.
\end{align*}

This completes the proof of the  proposition.
\end{proof}

\section{Global stability and long-time behavior}

In this section, we prove Theorem \ref{thm:stability}. Let us assume that  $\nu\in(0,\nu_0], \nu_0, \epsilon\in(0,1), \nu_0^{2/3}\leq 4\epsilon<\epsilon_1$. Then $e^{\nu t}\leq e^{4\epsilon\nu^{1/3} t}$ for $t>0.$

\subsection{Global existence and uniqueness}

The classical well-posedness theory ensures that there exists a unique solution $u\in C\big([0,T^*),H^2(\Omega)\big)\cap L_{loc}^2\big([0,T^*),H^3(\Omega)\big)$ to the Navier-Stokes equations \eqref{eq:NSp}, where $T^*$ is the maximal existence time of the solution.  Furthermore, for the linear equation \eqref{eq:u10} and \eqref{eq:u1n},  it is easy to prove the existence of the solution with
\beno
&&\bar{u}^{1,\neq}\in C\big([0,T^*),H^2(\Omega)\big)\cap L_{loc}^2\big([0,T^*),H^3(\Omega)\big),\\
&&\bar{u}^{1,0}\in C\big([0,T^*),H^4(\Omega)\big)\cap L_{loc}^2\big([0,T^*),H^5(\Omega)\big).
\eeno
To be precise, we write $E_j=E_j(T)$ for $j\in\{1,2,3,5\}$ and $T\in (0,T^*)$. Then $E_j(T)$ is a continuous and increasing function of $T\in(0,T^*)$, and
\begin{align*}
&\lim_{T\to 0+}E_j(T)\leq C\|u(0)\|_{H^2}\leq Cc_0\nu,\quad j\in\{1,2,3,5\}.
\end{align*}
Here all norms are taken over the interval $[0, T]$ unless stated otherwise, such as
\begin{align*}
&\|f\|_{L^pH^s}=\big\|\|f(t)\|_{H^s(\Omega)}\big\|_{L^p(0,T)},\quad \|f\|_{L^pL^q}=\big\|\|f(t)\|_{L^q(\Omega)}\big\|_{L^p(0,T)}.
\end{align*}

 Our goal is to prove that $T^*=+\infty$. The proof  is based on a continuity argument. Let us first assume that
\ben\label{ass:boot}
E_1\le \ve_1, \quad E_2\le \ve_1\nu,\quad  E_3\le \ve_1\nu.
\een
where $\ve_1$ is determined later.  First of all, we take $\veps_1<\veps_0$  so that
\beno
\|V_j-y\|_{H^4}\le E_1<\veps_0\quad \text{for}\,\, j\in[0,\nu^{\f13}T)\cap\Z.
\eeno
Now it follows from Proposition \ref{prop:E1},  Proposition \ref{prop:E2}, Proposition \ref{prop:E30},  Proposition \ref{prop:E31} and Proposition \ref{prop:E5}
that
\begin{align*}
&E_{1,0}\leq C\nu^{-1}\big(\|\bar{u}(0)\|_{H^2}+E_2+E_2E_{1,0}\big),\\
&E_{1,\neq}\leq C\big(\|\bar{u}(0)\|_{H^2}+\nu^{-1}E_2E_{1,\neq}+\nu^{-\f43}E_3^2\big),\\
&E_2\le C\big(1+\nu^{-1}E_2\big)\big(\|u(0)\|_{H^2}+\nu^{-1}E_3^2\big),\\
& E^2_{3,0} \leq C\|u(0)\|^2_{H^2}+C \Big(E^4_3/\nu^2+E^2_2E^2_3/\nu^2+E^2_1E_3E_5+E^2_1E^{\f32}_3E^{\f12}_5\Big),\\
&E^2_{3,1} \leq C\Big(\|u(0)\|_{H^2}^2+\nu^{-2}E_3^4+\nu^{-\f43}E_2^2E_3^2+ E_1^2E_3E_5+ E_1^2E_3^{\f74}E_5^{\f14}+E_1^2E_3^{\f32}E_5^{\f12}\Big),\\
&E_5^2\le E_6^2\leq C\|u(0)\|_{H^2}^2+C\big(E_1^2+\nu^{-2}E_2^2\big)E_6^2+C\nu^{-2}E_3^4.
\end{align*}
It is easy to see that by taking $\veps_1$ small enough, we can deduce from \eqref{ass:boot}  that
\begin{align*}
 E_3+ E_5 \leq C\|u(0)\|_{H^2}\le Cc_0\nu.
\end{align*}
Then we can get
\beno
 E_2\leq C\big(\|u(0)\|_{H^2}+\varepsilon_1E_3\big)\leq C\|u(0)\|_{H^2}\le Cc_0\nu,
 \eeno
 and then
 \begin{align*}
 &E_{1,0}\leq C\nu^{-1}\big(\|\bar{u}(0)\|_{H^2}+E_2\big)\leq C\nu^{-1}\|{u}(0)\|_{H^2}\le Cc_0,\\
 &E_{1,\neq} \leq C\big(\|\bar{u}(0)\|_{H^2}+\nu^{-\f13}\varepsilon_1 E_3\big)\leq C\nu^{-\f13}\|\bar{u}(0)\|_{H^2}\le Cc_0\nu^\f23.
  \end{align*}
Thus, we have
  \begin{align*}
  E_1=E_{1,0}+\nu^{-2/3}E_{1,\neq}\leq Cc_0.
\end{align*}
Now we take  $c_0>0$ small enough so that $Cc_0<\veps_1/2$. We define
\begin{align*}
&T_0=\sup\big\{T\in(0,T^*)|\max(\nu E_1(T),E_2(T),E_3(T))\le \ve_1\nu \big\}.
\end{align*}
The argument as above implies that $T_0=T^*$, and the argument as below implies that $\|u(t)\|_{L^{\infty}}\leq C\nu^{-1}\|u(0)\|_{H^2}$ for any $t\in [0,T^*)$, which in turn implies $T^*=+\infty.$

\subsection{Global stability estimates}

By Lemma \ref{lem:u23-H1}, we have
\begin{align*}
 e^{\nu t}\big(\|\Delta\bar{u}^2(t)\|_{L^2}+\|\nabla\bar{u}^3(t)\|_{L^2}\big) \leq C\big(\|u(0)\|_{H^2}+\nu^{-1}E_3^2\big)\leq C\|u(0)\|_{H^2}.
 \end{align*}
 As $\partial_y\bar{u}^2+\partial_z\bar{u}^3=0$, we have
 \begin{align*}
\|\bar{u}^2\|_{H^2}+\|\bar{u}^3\|_{H^1}+\|\partial_z\bar{u}^3\|_{H^1}
\leq C(\|\Delta\bar{u}^2\|_{L^2}+\|\nabla\bar{u}^3\|_{L^2}),
\end{align*}
which implies that
\begin{align*}
\|\bar{u}^2(t)\|_{H^2}+\|\bar{u}^3(t)\|_{H^1}+\|\partial_z\bar{u}^3(t)\|_{H^1} \leq Ce^{-\nu t}\|u(0)\|_{H^2}.
\end{align*}
Thanks to Lemma \ref{lemma-A.1}, we have
\begin{align*}
&\|\bar{u}^2(t)\|_{L^{\infty}}+\|\bar{u}^3(t)\|_{L^{\infty}}\leq C\big(\|\bar{u}^2(t)\|_{H^2}+\|\bar{u}^3(t)\|_{H^1}+\|\partial_z\bar{u}^3(t)\|_{H^1}\big)\leq Ce^{-\nu t}\|u(0)\|_{H^2}.
\end{align*}
This proves \eqref{eq:u-uniform1}. \smallskip

Now we prove \eqref{eq:u-uniform2}. First of all, we have
\begin{align*}
 & e^{2\epsilon\nu^{\f13}t}\big(\|(\partial_x,\partial_z)\nabla u_{\neq}^2(t)\|_{L^{2}}+\|(\partial_x^2+\partial_z^2)u_{\neq}^3(t)\|_{L^2}\big)\leq E_{3,0}\leq E_{3}\leq C\|u(0)\|_{H^2},
 \end{align*}
 and by Lemma \ref{lem:u-relation}, we have
 \begin{align*}
 \|(\partial_x,\partial_z)\partial_x u_{\neq}(t)\|_{L^{2}}\leq C\big(\|(\partial_x,\partial_z)\nabla u_{\neq}^2(t)\|_{L^{2}}+\|(\partial_x^2+\partial_z^2)u_{\neq}^3(t)\|_{L^2}\big)\leq Ce^{-2\epsilon\nu^{\f13}t}\|u(0)\|_{H^2}.
 \end{align*}
 By \eqref{fX} and Proposition \ref{prop:E30}, we have
 \begin{align*}
 &\nu^{1/2}e^{4\epsilon\nu^{\f13}t}\|\Delta u_{\neq}^2(t)\|_{L^{2}}^2\leq \|u^2_{\neq}\|^2_{X_{2\epsilon}}\\
& \leq C\|u_{\neq}(0)\|^2_{H^2}+C \big(E^4_3/\nu^2+E^2_2E^2_3/\nu^2+E^2_1E_3E_5+E^2_1E^{\f32}_3E^{\f12}_5\big)\leq C\|u(0)\|^2_{H^2}.
\end{align*}
By the definition of $E_{3,1}$, we have
\begin{align*}
 &e^{2\epsilon\nu^{\f13}t}\|\nabla \omega_{\neq}^2(t)\|_{L^{2}}\leq \nu^{-1/3}E_{3,1}\leq \nu^{-1/3}E_{3}\leq C\nu^{-1/3}\|u(0)\|_{H^2}.\end{align*}
 Using the fact that
 \begin{align*}
  &\|\nabla(\partial_x,\partial_z)(u_{\neq}^1,u_{\neq}^3)\|^2_{L^2}=\|\nabla\omega_{\neq}^2\|^2_{L^2}+
  \|\nabla\partial_y u_{\neq}^2\|^2_{L^2},
\end{align*}
we deduce that
\begin{align*}
  &\|\nabla(\partial_x,\partial_z)(u_{\neq}^1,u_{\neq}^3)(t)\|_{L^2}\leq C\big(\|\nabla\omega_{\neq}^2(t)\|_{L^2}+
  \|\Delta u_{\neq}^2(t)\|_{L^2}\big)\\ \nonumber&\leq C\big(\nu^{-1/3}e^{-2\epsilon\nu^{\f13}t}\|u(0)\|_{H^2}+\nu^{-1/4}e^{-2\epsilon\nu^{\f13}t}\|u(0)\|_{H^2}\big)\leq C \nu^{-1/3}e^{-2\epsilon\nu^{\f13}t}\|u(0)\|_{H^2}.
\end{align*}
Then we obtain
\begin{align*}
  &\|(u_{\neq}^1,u_{\neq}^3)(t)\|_{H^1}\leq C\|\nabla\partial_x(u_{\neq}^1,u_{\neq}^3)(t)\|_{L^2}\leq  C \nu^{-1/3}e^{-2\epsilon\nu^{\f13}t}\|u(0)\|_{H^2},\\&\nu^{1/4}\| u_{\neq}^2(t)\|_{H^{2}}\leq C\nu^{1/4}\|\Delta u_{\neq}^2(t)\|_{L^{2}}\leq Ce^{-2\epsilon\nu^{\f13}t}\|u(0)\|_{H^2}.
\end{align*}

By \eqref{eq:u-uniform2}, we have
\beno
\|u_{\neq}(t)\|_{L^2}\leq\|\partial_x^2u_{\neq}(t)\|_{L^2}\leq Ce^{-2\epsilon\nu^{1/3}t}\|u(0)\|_{H^2},
\eeno
which gives
\begin{align*}
&\|u_{\neq}\|_{L^{\infty}L^2}+\sqrt{\nu}\|t(u_{\neq}^1,u_{\neq}^3)\|_{L^{2}L^2}\\ &\leq C\|e^{-2\epsilon\nu^{1/3}t}\|_{L^{\infty}(0,+\infty)}\|u(0)\|_{H^2}+\sqrt{\nu}\|te^{-2\epsilon\nu^{1/3}t}\|_{L^2(0,+\infty)}\|u(0)\|_{H^2}\leq C\|u(0)\|_{H^2}.
\end{align*}
 By the definition of $E_3$, we have
 \begin{align*}
 \|\nabla u_{\neq}^2\|_{L^{\infty}L^2}+\|\nabla u_{\neq}^2\|_{L^{2}L^2}\leq CE_3\leq C\|u(0)\|_{H^2}.
\end{align*}
This proves \eqref{eq:u-uniform3}. \smallskip

It remains to prove the stability estimate in $L^\infty$ norm. By Lemma \ref{lem:sob-Linfty}, we get
\begin{align*}
  &\|u_{\neq}^2(t)\|_{L^{\infty}}\leq C\|(\partial_x,\partial_z)\partial_x u_{\neq}^2(t)\|_{L^{2}}^{\f12}\|(\partial_x,\partial_z)\nabla u_{\neq}^2(t)\|_{L^{2}}^{\f12}\leq  C e^{-2\epsilon\nu^{\f13}t}\|u(0)\|_{H^2},\\&\|u_{\neq}^j(t)\|_{L^{\infty}}\leq C\|(\partial_x,\partial_z)\partial_x u_{\neq}^j(t)\|_{L^{2}}^{\f12}\|(\partial_x,\partial_z)\nabla u_{\neq}^j(t)\|_{L^{2}}^{\f12}\leq  C\nu^{-1/6} e^{-2\epsilon\nu^{\f13}t}\|u(0)\|_{H^2},
\end{align*}
here $j\in\{1,3\}$. Thus, we also conclude \eqref{eq:u-uniform2}.\smallskip

Recall that $\bar{u}^{1}$ satisfies
\begin{align*}
\left\{
\begin{aligned}
&(\partial_t-\nu\Delta )\bar{u}^{1}+\bar{u}^{2}+\bar{u}^2\partial_y\bar{u}^{1}+\bar{u}^3\partial_z\bar{u}^{1}+\overline{u_{\neq}\cdot\nabla u_{\neq}^1}=0,\\
&\bar{u}^{1}|_{t=0}=\bar{u}^1(0),\quad \Delta\bar{u}^{1}|_{y=\pm1}=0,\quad \bar{u}^{1}|_{y=\pm1}=0.
\end{aligned}
\right.
\end{align*}
Then  $\Delta\bar{u}^{1}$ solves
\begin{align}
\label{u1a}(\partial_t-\nu\Delta )\Delta\bar{u}^{1}+\Delta\bar{u}^{2}+\Delta(\bar{u}^2\partial_y\bar{u}^{1}+\bar{u}^3\partial_z\bar{u}^{1}+
\overline{u_{\neq}\cdot\nabla u_{\neq}^1})=0.
\end{align}Since $\Delta\bar{u}^{1}|_{y=\pm1}=0$,  $H^2$ energy estimate gives
\begin{align*}
\f d {dt}\| \Delta\bar{u}^{1}\|_{L^2}^2+2\nu\| \nabla\Delta\bar{u}^{1}\|_{L^2}^2-2\big\langle \nabla(\bar{u}^{2}+\bar{u}^2\partial_2\bar{u}^{1}+\bar{u}^3\partial_3\bar{u}^{1}+
\overline{u_{\neq}\cdot\nabla u_{\neq}^1}), \nabla\Delta\bar{u}^{1}\big\rangle=0,
\end{align*}
which gives
\begin{align*}
\f d {dt}\| \Delta\bar{u}^{1}\|_{L^2}^2+\nu\| \nabla\Delta\bar{u}^{1}\|_{L^2}^2\leq C\nu^{-1}\Big(\|\nabla (\bar{u}^2\partial_y\bar{u}^{1}+\bar{u}^3\partial_z\bar{u}^{1})\|_{L^2}^2+\|\nabla \bar{u}^{2}\|_{L^2}^2+\|\nabla (\overline{u_{\neq}\cdot\nabla u_{\neq}^1})\|_{L^2}^2\Big).
\end{align*}
Thanks to $\|\nabla\Delta\bar{u}^{1}\|_{L^2}^2\geq(\pi/2)^2\|\Delta\bar{u}^{1}\|_{L^2}^2$,  we deduce that
\begin{align*}
\| e^{\nu t}\Delta\bar{u}^{1}\|_{L^{\infty}L^2}^2&+\nu\| e^{\nu t}\nabla\Delta\bar{u}^{1}\|_{L^2L^2}^2\leq \| u(0)\|_{H^2}^2+C\nu^{-1}\| e^{\nu t}\nabla(\bar{u}^2\partial_y\bar{u}^{1}+\bar{u}^3\partial_z\bar{u}^{1})\|_{L^2L^2}^2\\
&+C\nu^{-1}\| e^{\nu t}\nabla \bar{u}^{2}\|_{L^2L^2}^2+C\nu^{-1}\|e^{\nu t}\nabla (\overline{u_{\neq}\cdot\nabla u_{\neq}^1})\|_{L^2L^2}^2.
\end{align*}
Using the fact that
\begin{align*}
\|\nabla(\bar{u}^k\partial_k\bar{u}^{1})\|_{L^2}^2\leq& \|\bar{u}^k\partial_k\bar{u}^{1}\|_{H^1}^2\leq C\|\bar{u}^k\|_{H^2}^2\|\partial_k\bar{u}^{1}\|_{H^1}^2\\
\leq& C \|\Delta\bar{u}^k\|_{L^2}^2 \|\bar{u}^{1}\|_{H^2}^2\quad k=2,3,
\end{align*}
we infer that
\begin{align*}
&\|\nabla (\bar{u}^2\partial_y\bar{u}^{1}+\bar{u}^3\partial_z\bar{u}^{1})\|_{L^2}^2
\leq C\big(\|\Delta\bar{u}^2\|_{L^2}^2+\|\Delta\bar{u}^3\|_{L^2}^2\big)\|\bar{u}^{1}\|_{H^2}^2.
\end{align*}
Then by Lemma \ref{lem:u2-H2}, Lemma \ref{lem:u2-H3} and Lemma \ref{lem:u1-H2}, we get
\begin{align*}
\|e^{\nu t}\nabla (\bar{u}^2\partial_y\bar{u}^{1}+\bar{u}^3\partial_z\bar{u}^{1})\|_{L^2L^2}^2
\leq& C\|e^{\nu t}(\Delta\bar{u}^2,\Delta\bar{u}^3)\|_{L^{2}L^2}^2\|\bar{u}^{1}\|_{L^{\infty}H^2}^2\\ \leq& C\nu^{-1}\big(\|u(0)\|_{H^2}^2+\nu^{-2}E_3^4\big)E_1^2\leq C\nu^{-1}\|u(0)\|_{H^2}^2.
\end{align*}
By Lemma \ref{lem:u23-H1}, we have
\begin{align*}
&\| e^{\nu t}\nabla \bar{u}^{2}\|_{L^2L^2}^2\leq C\nu^{-1}\big(\|u(0)\|_{H^2}^2+\nu^{-2}E_3^4\big)\leq C\nu^{-1}\|u(0)\|_{H^2}^2.
\end{align*}
Since $0<\nu\leq \nu_0,\ \nu_0^{2/3}\leq 4\epsilon$,  we get by \eqref{lemma2.1-3}  that
\begin{align*}
&\|e^{\nu t}\nabla (\overline{u_{\neq}\cdot\nabla u_{\neq}^1})\|_{L^2L^2}^2\leq\|e^{4\epsilon\nu^{\f13}t}\nabla(u_{\neq}\cdot\nabla u_{\neq})\|_{L^2L^2}^2\leq C\nu^{-\f53}E_3^4\leq C\nu^{\f13}\|u(0)\|_{H^2}^2.
\end{align*}

Summing up, we conclude that
\begin{align*}
&\| e^{\nu t}\Delta\bar{u}^{1}\|_{L^{\infty}L^2}^2\leq \| u(0)\|_{H^2}^2+C(\nu^{-2}\|u(0)\|_{H^2}^2+\nu^{-\f83}E_3^4)\leq C\nu^{-2}\|u(0)\|_{H^2}^2,
  \end{align*}
 which gives
  \begin{align*}
 \|\bar{u}^1(t)\|_{H^2}\leq C\|\Delta\bar{u}^1(t)\|_{L^2}\leq C\nu^{-1}e^{-\nu t}\|u(0)\|_{H^2}.
 \end{align*}
 On the other hand, by Lemma \ref{lem:u1-H2}, we have
 \begin{align*}
    &\|\bar{u}^1(t)\|_{H^2}\leq CE_1\cdot(\nu t+\nu^{2/3})\leq C\nu^{-1}(\nu t+\nu^{2/3})\|u(0)\|_{H^2}.
  \end{align*}
  Thus, we obtain
  \begin{align*}
  \|\bar{u}^1(t)\|_{L^{\infty}}\leq C \|\bar{u}^1(t)\|_{H^2}\leq C\nu^{-1}\min(\nu t+\nu^{2/3},e^{-\nu t})\|u(0)\|_{H^2}.
  \end{align*}
  This proves \eqref{eq:u-uniform0}.

\section{Appendix}

\subsection{Sobolev inequalities}

\begin{Lemma}\label{lem:sob-f}
 If $f$ satisfies $\partial_xf=ikf, |k|\ge 1$, then we have
  \begin{align*}
    \|f\|_{L^2_{x,z}L^\infty_y}+\delta^{\f12}\|\chi_1 f\|_{L^2}+\|f\|_{L^2(\partial\Omega)}\leq C\| f\|_{L^2}^{\f12}\|\nabla f\|_{L^2}^{\f12}\leq C|k|^{-\f12}\|\nabla f\|_{L^2}.
  \end{align*}
 Here $\chi_1$ is  given by Lemma \ref{lem:chi1}, and $C$ is a constant independent of $\lambda, \delta, k$.
\end{Lemma}

\begin{proof}
Due to $\partial_xf=ikf$, we have
\beno
\| f\|_{L^2}\leq|k|\| f\|_{L^2}\leq\|\partial_x f\|_{L^2}\leq\|\nabla f\|_{L^2} .
\eeno
For fixed $x,z,$ we have $\| f\|_{L^\infty_y}\leq C\| f\|_{L^2_y}^{\f12}\|(\partial_y,1) f\|_{L^2_y}^{\f12}, $
which shows that
\begin{align}\label{fL2a}
     &\|f\|_{L^2_{x,z}L^\infty_y}\leq C\| f\|_{L^2}^{\f12}\|(\partial_y,1) f\|_{L^2}^{\f12}\leq C\| f\|_{L^2}^{\f12}\|\nabla f\|_{L^2}^{\f12}\leq C|k|^{-\f12}\|\nabla f\|_{L^2}.
  \end{align}
By Lemma \ref{lem:chi1}, we get
\begin{align*}
     &\delta^{\f12}\|\chi_1 f\|_{L^2}+\|f\|_{L^2(\partial\Omega)}\leq \delta^{\f12}\|\chi_1\|_{L^\infty_{x,z}L^2_y}\|f\|_{L^2_{x,z}L^\infty_y}+2\|f\|_{L^\infty_yL^2_{x,z}}\leq C\|f\|_{L^2_{x,z}L^\infty_y},
  \end{align*}
 which along with \eqref{fL2a}  gives the lemma.
   \end{proof}

\begin{Lemma}\label{lem:sob-Linfty} If $P_0f=0$, then we have
\begin{align*}&\|f\|_{L^{\infty}}\leq C\|(\partial_x,\partial_z)\partial_x f\|_{L^{2}}^{\f12}\|(\partial_x,\partial_z)\nabla f\|_{L^{2}}^{\f12}.
\end{align*}
\end{Lemma}

\begin{proof}Let $f_{k,\ell}(y)=\f{1}{2\pi}\int_{\mathbb{T}^2}f(x,y,z)e^{-ikx-i\ell z}dxdz$. Then we have
\begin{align*}
 \|f\|_{L^{\infty}}\leq& C\sum_{k\neq 0;\ell\in\Z}\|f_{k,\ell}\|_{L^{\infty}}\leq C\sum_{k\neq 0;\ell\in\Z}\|f_{k,\ell}\|_{L^{2}}^{1/2}\|(\partial_y,1)f_{k,\ell}\|_{L^{2}}^{1/2}\\ \leq& C\Big(\sum_{k\neq 0;\ell\in\Z}k^2\eta^2\|f_{k,\ell}\|_{L^{2}}^{2}\Big)^{\f14}\Big(\sum_{k\neq 0;\ell\in\Z}\eta^2\|(\partial_y,1)f_{k,\ell}\|_{L^{2}}^{2}\Big)^{\f14}\Big(\sum_{k\neq 0;\ell\in\Z}\frac{1}{|k|(k^2+\ell^2)}\Big)^{\f12}\\ \leq& C\|(\partial_x,\partial_z)\partial_x f\|_{L^{2}}^{\f12}\|(\partial_x,\partial_z)\nabla f\|_{L^{2}}^{\f12}\Big(\sum_{k\neq 0}\frac{1}{|k|^2}\Big)^{\f12}\leq C\|(\partial_x,\partial_z)\partial_x f\|_{L^{2}}^{\f12}\|(\partial_x,\partial_z)\nabla f\|_{L^{2}}^{\f12}.
\end{align*}
Here $\eta^2=k^2+\ell^2 $ and we used the fact that for $k\neq 0$,
\begin{align*}
  &\sum_{\ell\in\Z}\frac{1}{k^2+\ell^2}\leq \frac{1}{k^2}+\int_{\R}\frac{dz}{k^2+z^2}\leq\frac{1}{k^2}+\frac{\pi}{|k|}\leq\frac{\pi+1}{|k|}.
\end{align*}
\end{proof}

 \subsection{Elliptic estimates with the weight}

 \begin{Lemma}\label{lem:elliptic-weight}
  Let $\varphi$ be a unique solution of $\Delta \varphi=w,\ \varphi(\pm1)=0$.
It holds that
  \begin{align*}
    &\|\nabla\varphi\|_{L^2} \leq C|k|^{-\f12}\|w\|_{L^2_{x,z}L^1_y},\\
    &\|\nabla\varphi\|_{L^2}\leq C\|(1-y^2)w\|_{L^2},\\
    &\|\varphi\|_{L^2}\leq C|k|^{-\f12}\|(1-y^2)w\|_{L^2_{x,z}L^1_y}.
  \end{align*}
 Here $C$ is a constant independent of $k$.
\end{Lemma}

\begin{proof}
Thanks to $\partial_xw=ikw$ and $\Delta\varphi=w$, we have $\partial_x\varphi=ik\varphi.$ By Lemma \ref{lem:sob-f}, we have
\beno
\|\varphi\|_{L^2_{x,z}L^\infty_y}\leq C|k|^{-\f12}\|\nabla \varphi\|_{L^2},
\eeno
 and by integration by parts,
 \begin{align*}
    \|\nabla\varphi\|_{L^2}^2=&-\langle \varphi,\Delta\varphi\rangle=-\langle \varphi,w\rangle\\
    \leq&\|\varphi\|_{L^2_{x,z}L^\infty_y}\|w\|_{L^2_{x,z}L^1_y}\leq C|k|^{-\f12}\|\nabla\varphi\|_{L^2}\|w\|_{L^2_{x,z}L^1_y},
  \end{align*}
 which gives the first inequality.

 Due to $\varphi|_{y=\pm1}=0,$ we get by Hardy's inequality  that
 \begin{align*}
    \|\nabla\varphi\|_{L^2}^2=-\langle \varphi,w\rangle&\leq\|\varphi/(1-y^2)\|_{L^2}\|(1-y^2)w\|_{L^2}\\
    &\leq C\|\nabla\varphi\|_{L^2}\|(1-y^2)w\|_{L^2},
  \end{align*}
  which gives the second inequality.

  Let $ \phi$ solve $\Delta\phi=\varphi, \phi|_{y=\pm1}=0$. Then we have
  \beno
  \partial_x\phi=ik\phi,\quad \|\phi\|_{H^2}\leq C\|\varphi\|_{L^2}.
  \eeno
 As $\phi|_{y=\pm1}=0,$ for fixed $x,z,$ we have $\| \phi/(1-|y|)\|_{L^\infty_y}\leq \|\partial_y \phi\|_{L^\infty_y},$ and then
 \beno
 \| \phi/(1-|y|)\|_{L^2_{x,z}L^\infty_y}\leq \|\partial_y \phi\|_{L^2_{x,z}L^\infty_y},
 \eeno
 and by Lemma \ref{lem:sob-f}, we have
 \begin{align*}
    &\|\partial_y \phi\|_{L^2_{x,z}L^\infty_y}\leq C|k|^{-\f12}\|\partial_y^2\phi\|_{L^2}\leq C|k|^{-\f12}\| \phi\|_{H^2}\leq C|k|^{-\f12}\| \varphi\|_{L^2}.
  \end{align*}
This shows that
\beno
 \| \phi/(1-|y|)\|_{L^2_{x,z}L^\infty_y}\leq C|k|^{-\f12}\| \varphi\|_{L^2},
\eeno
and then,
 \begin{align*}
    \|\varphi\|_{L^2}^2&=\langle \varphi,\Delta\phi\rangle=\langle \Delta\varphi,\phi\rangle=\langle w,\phi\rangle\\&\leq\|(1-|y|)w\|_{L^2_{x,z}L^1_y}\| \phi/(1-|y|)\|_{L^2_{x,z}L^\infty_y}\\&\leq C\|(1-y^2)w\|_{L^2_{x,z}L^1_y}|k|^{-\f12}\| \varphi\|_{L^2},
  \end{align*}
 which gives the third inequality.
 \end{proof}

\begin{Lemma}\label{lem:ell-w}
 Let $\varphi$ be a unique solution of $(\partial_y^2-\eta^2)\varphi=w,\ \varphi(\pm1)=0$.
 Then it holds that
  \begin{align*}&C^{-1}|\eta|\|\varphi\|_{L^{\infty}}^2\leq\|\varphi'\|_{L^2}^2+\eta^2\|\varphi\|_{L^2}^2=\langle-w,\varphi\rangle\le C|\eta|^{-1}\|w\|_{L^1}^2,\\
    &\|\varphi\|_{L^2}  \leq C\eta^{-\f12}\|(1-|y|)w\|_{L^1},\\
    &\|(\partial_y,\eta)\varphi\|_{L^2}\leq C\|(1-|y|)w\|_{L^2}.
  \end{align*}
\end{Lemma}

\begin{proof}
The first inequality follows from
\begin{align*}&\|\varphi\|_{L^{\infty}}\leq C\|\varphi'\|^{\f12}_{L^2}\|\varphi\|^{\f12}_{L^2}\leq C\big(|\eta|^{-1}\|\varphi'\|_{L^2}^2+|\eta|\|\varphi\|_{L^2}^2\big)^{\f12}=C|\eta|^{-\f12}\|(\partial_y,\eta)\varphi\|_{L^2},\\
&\|\varphi'\|_{L^2}^2+\eta^2\|\varphi\|_{L^2}^2=\langle-w,\varphi\rangle\leq\|w\|_{L^1}\|\varphi\|_{L^{\infty}}\leq C|\eta|^{-\f12}\|w\|_{L^1}\|(\partial_y,\eta)\varphi\|_{L^2}.
\end{align*}

Let $ \varphi_1$ solve  $(\partial_y^2-\eta^2)\varphi_1=\varphi,\ \varphi_1(\pm1)=0$. Then we have
$$\|\varphi\|_{L^2}^2=\eta^4\|\varphi_1\|_{L^2}^2+2\eta^2\|\partial_y\varphi_1\|_{L^2}^2+\|\partial^2_y\varphi_1\|_{L^2}^2.$$

Since $\varphi_1(y)=\int_{-1}^{y}\varphi_1'(z)dz $ for $y\in[-1,0] $, $\varphi_1(y)=-\int_{y}^{1}\varphi_1'(z)dz $ for $y\in[0,1] $, we conclude
  \begin{align*}
    \left\|\f{\varphi_1}{1-|y|}\right\|_{L^\infty} \leq& \|\partial_y\varphi_1\|_{L^\infty}\leq C\|\partial_y\varphi_1\|_{L^2}^{\f12} \|\partial^2_y\varphi_1\|_{L^2}^{\f12}
     \leq C\eta^{-\f12}\|\varphi\|_{L^2},
  \end{align*}
which gives
  \begin{align*}
    \|\varphi\|^2_{L^2} =&\left|\big\langle \varphi_1,w \big\rangle\right|
    \leq \left\|\f{\varphi_1}{1-|y|}\right\|_{L^\infty}\|(1-|y|)w\|_{L^1}\\
    \leq &C\eta^{-\f12}\|\varphi\|_{L^2}\|(1-|y|)w\|_{L^1},
  \end{align*}
 and thus the second inequality.

By Hardy's inequality, we get
\begin{align*}
    \|(\partial_y,\eta)\varphi\|_{L^2}^2=&|\langle \varphi,w \rangle|\leq \left\|\f{\varphi}{1-|y|}\right\|_{L^2}\|(1-|y|)w\|_{L^2}\\
    \leq &C\|\partial_y\varphi\|_{L^2}\|(1-|y|)w\|_{L^2},
\end{align*}
which gives the third inequality.
\end{proof}

\subsection{Some basic properties of harmonic function}

\begin{Lemma}\label{lem1.harmonic}
If $f$ is a harmonic function in $\Omega:=\mathbb{T}\times[-1,1]\times\mathbb{T}$, then we have
\begin{align*}
&\|\nabla f\|_{L^2}\leq C\|\partial_yf\|_{L^2}.
\end{align*}
\end{Lemma}

\begin{proof}
As $\Delta f=0$,  we get by taking Fourier transform in $x,z$  that
\begin{align*}
\big(\partial_y^2-(k^2+\ell^2)\big){f}_{k,\ell}(y)=0,
\end{align*}
which means that
\begin{align*}
{f}_{k,\ell}=C_1e^{\sqrt{k^2+\ell^2}y}+C_2e^{-\sqrt{k^2+\ell^2}y},\quad {f}_{k,\ell}'=\sqrt{k^2+\ell^2}\Big(C_1e^{\sqrt{k^2+\ell^2}y}-C_2e^{-\sqrt{k^2+\ell^2}y}\Big),
\end{align*}
where $C_j=a_j+ib_j,a_j,b_j\in\mathbb{R}(j=1,2)$.

For $k^2+{\ell}^2\geq1$, we have
\begin{align*}
&\int_{-1}^1\big|C_1e^{\sqrt{k^2+\ell^2}y}\pm C_2e^{-\sqrt{k^2+\ell^2}y}\big|^2dy\\
&=\int_{-1}^1\big|a_1e^{\sqrt{k^2+\ell^2}y}\pm a_2e^{-\sqrt{k^2+\ell^2}y}\big|^2dy+\int_{-1}^1\big|b_1e^{\sqrt{k^2+\ell^2}y}\pm b_2e^{-\sqrt{k^2+\ell^2}y}\big|^2dy\\
&=\f{a_1^2+a_2^2}{2\sqrt{k^2+\ell^2}}\Big(e^{2\sqrt{k^2+\ell^2}}-e^{-2\sqrt{k^2+\ell^2}}\Big)\pm 4a_1a_2+\f{b_1^2+b_2^2}{2\sqrt{k^2+\ell^2}}\Big(e^{2\sqrt{k^2+\ell^2}}-e^{-2\sqrt{k^2+\ell^2}}\Big)\pm 4b_1b_2\\
&\sim \f{a_1^2+a_2^2}{\sqrt{k^2+\ell^2}}\Big(e^{2\sqrt{k^2+\ell^2}}-e^{-2\sqrt{k^2+\ell^2}}\Big)+\f{b_1^2+b_2^2}{\sqrt{k^2+\ell^2}}
\Big(e^{2\sqrt{k^2+\ell^2}}-e^{-2\sqrt{k^2+\ell^2}}\Big),
\end{align*}
which implies
\begin{align*}
(k^2+\ell^2)\|{f}_{k,\ell}\|_{L_y^2}^2\leq C\|{f}_{k,\ell}'\|_{L_y^2}^2,
\end{align*}
and then the lemma follows by using Plancherel's formula.
\end{proof}

\begin{Lemma}\label{delta f}If $f\in H^2(\Omega),\ \partial_xf=ikf,\ |k|\geq 1$,  then we have
\begin{align*}
\|\nabla f\|_{L^2}\leq C\big(\|\partial_yf\|_{L^2}+|k|^{-1}\|\Delta f\|_{L^2}\big).
\end{align*}
\end{Lemma}

\begin{proof}
We decompose $f=f_1+f_2$, where $f_1,f_2$ solve
\begin{align*}
&\Delta f_1=0,\quad
\Delta f_2=\Delta f,\quad f_2|_{y=\pm1}=0.
 \end{align*}
It is easy to see that
 \begin{align}
\label{f1f2a} |k|\|\nabla f_2\|_{L^2}=\|\partial_x\nabla f_2\|_{L^2}\leq\|\nabla^2f_2\|_{L^2}\leq C\|\Delta f\|_{L^2}.
 \end{align}
 Thanks to Lemma \ref{lem1.harmonic} and \eqref{f1f2a}, we have
\begin{align*}
\|\nabla f_1\|\leq C\| \partial_yf_1\|_{L^2}\leq C\big(\| \partial_yf\|_{L^2}+\| \partial_yf_2\|_{L^2}\big)\leq C\big(\| \partial_yf\|_{L^2}+|k|^{-1}\|\Delta f\|_{L^2}\big),
\end{align*}
which gives
\begin{align*}
\|\nabla f\|\leq\|\nabla f_1\|+\|\nabla f_2\|_{L^2}\leq C\big(\| \partial_yf\|_{L^2}+|k|^{-1}\|\Delta f\|_{L^2}\big).
\end{align*}
\end{proof}

\begin{Lemma}\label{the estimate of delta bar{p}}If $f$ is a function in $[-1,1]\times2\pi\mathbb{T}$, i.e. $f=f(y,z)$,  then we have
\begin{align*}
\|\partial_y^2f\|_{L^2}+\|\partial_z^2f\|_{L^2}\leq C\big(\|\partial_z\partial_yf\|_{L^2}+\|\Delta f\|_{L^2}\big).
\end{align*}
\end{Lemma}

\begin{proof}
We make the same decomposition for $f$ as in Lemma \ref{delta f}.
First of all, we have
 \begin{align}
\label{tri4.1} \|\partial_y^2f_2\|_{L^2}+2\|\partial_z\partial_yf_2\|_{L^2}+\|\partial_z^2f_2\|_{L^2}\leq C\|\Delta f\|_{L^2}.
 \end{align}
 Thanks to Lemma \ref{lem1.harmonic}, we have
 \begin{align*}
 \|\partial_z^2f_1\|\leq C\| \partial_z\partial_yf_1\|_{L^2},
 \end{align*}
which along with \eqref{tri4.1} gives
\begin{align*}
\|\partial_z^2f_1\|\leq C\| \partial_z\partial_yf_1\|_{L^2}\leq C\big(\| \partial_z\partial_yf\|_{L^2}+\| \partial_z\partial_yf_2\|_{L^2}\big)\leq C\big(\| \partial_z\partial_yf\|_{L^2}+\|\Delta f\|_{L^2}\big).
\end{align*}
This shows that
\begin{align*}
\|\partial_z^2f\|_{L^2}\leq \|\partial_z^2f_1\|_{L^2}+\|\partial_z^2f_2\|_{L^2}\leq C\big(\| \partial_z\partial_yf\|_{L^2}+\|\Delta f\|_{L^2}\big),
\end{align*}
which also gives
\begin{align*}
\|\partial_y^2f\|_{L^2}\leq \|\partial_z^2f\|_{L^2}+\|\Delta f\|_{L^2}\leq C\big(\| \partial_z\partial_yf\|_{L^2}+\|\Delta f\|_{L^2}\big).
\end{align*}
\end{proof}

\subsection{Maximal inequality of harmonic function}

\begin{Lemma}\label{lem:har-max}
   Let $j\in\{\pm 1\}$. If $\Delta f=0,\ f|_{y=-j}=0$ and $f\in C^0(\Omega)$, then we have
   \begin{align*}
      &\|f\|_{L^2_{x,z}L^\infty_y(\Omega)}\leq 6\|f\|_{L^2(\Gamma_{j})}.
   \end{align*}
\end{Lemma}

We need to use some definitions and conclusions from Chapter 2 in \cite{Gra}. The centered Hardy-Littlewood maximal function is defined by
 \begin{align*}
      &\mathcal{M}(h)(q)=\sup_{r>0}\mathop{\text{Avg}}_{B(q,r)}|h|=\sup_{r>0}\frac{1}{\pi r^2}\|h\|_{L^1(B(q,r))},
   \end{align*}
where $B(q,r):=\{q'\in\R^2||q'-q|< r\} $ for $q\in\R^2$. It is well-known that $\mathcal{M}$ is bounded from $L^p(\R^n)$ to $L^p(\R^n)$ with constant at most $3^{n/p}p/(p-1)$. In particular, we have
\begin{align}\label{Mh0}
      &\|\mathcal{M}(h)\|_{L^2(\R^2)}\leq 6\|h\|_{L^2(\R^2)}\quad \forall\ h\in L^2(\R^2) .
   \end{align}
The following lemma shows that the same is true with $L^2(\T^2) $ instead of $L^2(\R^2)$.
Here we identify a function on $\T^2$ with a 
function on $\R^2$ satisfying $f(q+2\pi m)=f(q),\ \forall\ m\in\Z^2$.

\begin{Lemma}\label{lem:max}
For every $h\in L^2(\T^2)$, we have $\|\mathcal{M}(h)\|_{L^2(\T^2)}\leq 6\|h\|_{L^2(\T^2)}. $
\end{Lemma}

\begin{proof}   For $a>0$, let $Q_a=\{(x,z)\in\R^2||x|<a,|z|<a\},\ \chi_a=\mathbf{1}_{Q_a}$ and\begin{align*}
      &\mathcal{M}_{a}(h)(q)=\sup_{0<r<a}\frac{1}{\pi r^2}\|h\|_{L^1(B(q,r))}.
   \end{align*}
It is easy to see that  for $\ h\in L^2(\T^2),\ q\in\R^2,$
 \begin{align}\label{Mh1}
      &\mathcal{M}_{a}(h)(q)\uparrow\mathcal{M}(h)(q)\quad \text{as}\,\, a\uparrow+\infty,\\ \label{Mh2}&\|h\|_{L^2(Q_{\pi n})}^2=\|h\chi_{\pi n}\|_{L^2(\R^2)}^2=n^2\|h\|_{L^2(\T^2)}^2,\ \forall\ n\in\Z,n>0.
   \end{align}

For $m,n\in \Z,\ m,n>0,\ q\in Q_{\pi n},\ 0<r<\pi m,$ we have $B(q,r)\subset Q_{\pi (m+n)}$, and\begin{align*}
      &\|h\|_{L^1(B(q,r))}=\|h\chi_{\pi (m+n)}\|_{L^1(B(q,r))}\leq\pi r^2\mathcal{M}(h\chi_{\pi (m+n)})(q),
   \end{align*}
which implies that  $\mathcal{M}_{\pi m}(h)(q)\leq\mathcal{M}(h\chi_{\pi (m+n)})(q) $ and then by \eqref{Mh0}),
\begin{align*}
      &\|\mathcal{M}_{\pi m}(h)\|_{L^2(Q_{\pi n})}\leq\|\mathcal{M}(h\chi_{\pi (m+n)})\|_{L^2(\R^2)}\leq 6\|h\chi_{\pi (m+n)}\|_{L^2(\R^2)}.
   \end{align*}
which along with  \eqref{Mh2}  gives
\begin{align*}
 \|\mathcal{M}_{\pi m}(h)\|_{L^2(\T^2)}\leq 6(1+m/n)\|h\|_{L^2(\T^2)},\ \forall\ m,n\in \Z,\ m,n>0,\ h\in L^2(\T^2).
   \end{align*}
Letting $n\to+\infty$, we get
\begin{align*}
      &\|\mathcal{M}_{\pi m}(h)\|_{L^2(\T^2)}\leq 6\|h\|_{L^2(\T^2)},\ \forall\ m\in \Z,\ m>0,\ h\in L^2(\T^2).
   \end{align*}
Letting $m\to+\infty$, using \eqref{Mh1} and monotone convergence theorem, we get
\begin{align*}
      &\|\mathcal{M}(h)\|_{L^2(\T^2)}\leq 6\|h\|_{L^2(\T^2)},\ \forall\ h\in L^2(\T^2).
   \end{align*}

This completes the proof.
 \end{proof}

   Now we give the proof of Lemma \ref{lem:har-max}.
    \begin{proof}
   Without lose of generality, we only need to consider the case of $j=1$.  Let $f_0(x,z)=|f(x,1,z)|$ and let\begin{align*}
    &P(x,z)=c/(1+x^2+z^2)^{3/2},
  \end{align*}
  where $c=\Gamma(3/2)/\pi^{3/2}$ is a constant so that
  \begin{align*}
    &\int_{\R^2}P(x,z)dxdz=1,
  \end{align*}
  The function $P$ is called the Poisson kernel. We define $L^1$ dilates $P_t$ of the Poisson
kernel $P$ by setting
\begin{align*}
    &P_t(x,z)=t^{-2}P(x/t,z/t), \quad t>0.
  \end{align*}
  Then the function\begin{align*}
    &F(x,y,z)=(f_0*P_{1-y})(x,z),
  \end{align*}solves the Dirichlet problem\begin{align*}
    &\Delta F=0\quad \text{on}\ \R\times(-\infty,1)\times\R,\quad F(x,1,z)=f_0(x,z)\quad \text{on}\ \R^2.
  \end{align*}
  We also have $F\in C^0(\Omega)$ and $F\geq 0$ on $ \Omega.$
   Since $f|_{y=-1}=0,\ |f||_{y=1}=f_0=F|_{y=1}$, we have $|f|\leq F$ on $\partial \Omega.$ For fixed $ \theta\in\R,$ let $f_{[\theta]}=\mathbf{Re}(e^{i\theta}f)$. Since $\Delta f=0$, we have $\Delta f_{[\theta]}=0$ on $ \Omega,$ and $f_{[\theta]}$ is real valued. Now we have $\Delta F=\Delta f_{[\theta]}=0$ on $ \Omega,$ and $f_{[\theta]}\leq |f|\leq F$ on $ \partial\Omega$. Using the maximum principle, we deduce that $f_{[\theta]}\leq  F$ on $\Omega$ for every $ \theta\in\R.$ Thus,  $|f|=\sup_{\theta\in\R}f_{[\theta]}\leq  F$ on $\Omega,$ and then
   \begin{align*}
      &\|f\|_{L^2_{x,z}L^\infty_y(\Omega)}\leq \|F\|_{L^2_{x,z}L^\infty_y(\Omega)},
   \end{align*}
Since $F(x,y,z)\leq \mathcal{M}(f_0)(x,z)$ for $x,z\in\R^2,\ y\leq 1$ (see \cite{Gra}), we get
\begin{align*}
      &\|f\|_{L^2_{x,z}L^\infty_y(\Omega)}\leq \|F\|_{L^2_{x,z}L^\infty_y(\Omega)}\leq \|\mathcal{M}(f_0)\|_{L^2(\T^2)},
   \end{align*}
  which along with Lemma \ref{lem:max}  shows that
   \begin{align*}
      &\|f\|_{L^2_{x,z}L^\infty_y(\Omega)}\leq \|\mathcal{M}(f_0)\|_{L^2(\T^2)}\leq 6\|f_0\|_{L^2(\T^2)}=6\|f\|_{L^2(\Gamma_{1})}.
   \end{align*}
   The case of $f|_{y=1}=0$ can be proved similarly.\end{proof}

\subsection{Limiting absorption principle}

In this section, we establish the limiting absorption principle for the Rayleigh equation
\begin{align}\label{eq:Ray}
(y-c)(\Phi''-\alpha^2\Phi)=\omega, \quad \Phi(-1)=\Phi(1)=0,
\end{align}
where $c\in \C,\, \mathbf{Im}(c)\neq 0$.

\begin{Proposition}\label{prop:LAP}
Let $\alpha\ge 1$, $\mathbf{Im}(c)\neq 0$. Then the unique solution $\Phi$ to \eqref{eq:Ray}
satisfies
\begin{align*}
&\|\partial_y\Phi\|_{L^2}+\alpha\|\Phi\|_{L^2}\leq C\alpha^{-1}\big(\|\partial_y\om\|_{L^2}+\alpha\|\om\|_{L^2}\big),\\
&\|\Phi\|_{L^2}\leq C\alpha^{-1}\big(\max(1-|c_r|,0)\|\partial_y\om\|_{L^2}+\|\om\|_{L^2}\big),
\end{align*}
where $c_r=\mathbf{Re}(c)$ is the real part of $c$.
\end{Proposition}

We need the following lemma.
\begin{Lemma}\label{lem:hardy}
  Let $a,\lambda\in\R,\ f\in H^1([-1,1]),\ \delta_{*}>0,\ a\neq0$.
  If $f(\pm1)=0$ or $1-|\lambda|\geq\delta_{*}>0$, then it holds that
\begin{align*}
 \left|\int_{-1}^{1}\f{f(y)}{y-\lambda+ia}dy\right|\leq C\big(\delta_*^{\f12}\|f'\|_{L^2}+\delta_*^{-\f12}\|f\|_{L^2}\big),
  \end{align*}
where $C$ is a constant independent of $a, \delta_*$.
\end{Lemma}

\begin{proof}
 For the case of $1-|\lambda|\geq \delta_*$,  we have $[\lambda-\delta_*,\lambda+\delta_*]\subseteq[-1,1]$, and then
  \begin{align*}
    \left|\int_{-1}^{1}\f{f(y)}{y-\lambda+ia}dy\right| \leq& \left|\int_{\lambda-\delta_*}^{\lambda+\delta_*}\f{f(y)}{y-\lambda+ia}dy\right| +\left|\int_{|y-\lambda|\geq\delta_*}\f{f(y)}{y-\lambda+ia}dy\right|\\
    \leq &\left|\int_{\lambda}^{\lambda+\delta_*}\f{f(y)}{y-\lambda+ia}+ \f{f(2\lambda-y)}{\lambda-y+ia}dy\right|+\|f\|_{L^2} \|(y-\lambda)^{-1}\|_{L^2\big(\{|y-\lambda|\geq\delta_*\}\big)}\\
    \leq &\left|\int_{\lambda}^{\lambda+\delta_*}\f{f(y)}{y-\lambda+ia}+ \f{f(2\lambda-y)}{\lambda-y+ia}dy\right|+C\delta_*^{-\f12}\|f\|_{L^2},
  \end{align*}
where
\begin{align*}
   &\left|\int_{\lambda}^{\lambda+\delta_*}\f{f(y)}{y-\lambda+ia}+ \f{f(2\lambda-y)}{\lambda-y+ia}dy\right|\\
   &\leq \left|\int_{\lambda}^{\lambda+\delta_*} \f{(y-\lambda)(f(y)-f(2\lambda-y))}{(y-\lambda)^2+a^2}dy\right|+ \left|\int_{\lambda}^{\lambda+\delta_*} \f{\epsilon(f(y)+f(2\lambda-y))}{(y-\lambda)^2+a^2}dy\right|\\
   &\leq \left|\int_{\lambda}^{\lambda+\delta_*} \f{(y-\lambda)\int_{\lambda-y}^{y-\lambda}f'(z+\lambda)dz}{(y-\lambda)^2+a^2}dy\right| +\|f\|_{L^\infty}\left|\int_{\lambda}^{\lambda+\delta_*}\f{2a}{(y-\lambda)^2+a^2} dy\right|\\
   &\leq \left|\int_{\lambda}^{\lambda+\delta_*} \f{C(y-\lambda)^{\f32}\|f'\|_{L^2}}{(y-\lambda)^2+a^2}dy\right| +2\|f\|_{L^\infty}\arctan\big(\f{\delta_*}{a}\big) \\
   &\leq C\left(\int_{0}^{\delta_*}\f{z^{\f12}}{z+|a|}dz\right)\|f'\|_{L^2}+C\|f\|_{L^\infty}
   \leq C\big(\delta_*^{\f12}\|f'\|_{L^2}+\delta_*^{-\f12}\|f\|_{L^2}\big).
\end{align*}
This proves the case of  $1-|\lambda|\geq\delta_*$.\smallskip

For the case of $f(\pm1)=0$, we can first extend $f$ to be a function in $H^1(\R)$ by taking $f(y)=0$ for $|y|\geq 1$, then follow
the proof as above.
\end{proof}

Now let us prove Proposition \ref{prop:LAP}.

\begin{proof}
It is easy to get that
  \begin{align*}
     \|\partial_y\Phi\|_{L^2}^2+\alpha^2\|\Phi\|^2_{L^2}= -\int_{-1}^{1}\f{\omega\bar{\Phi}}{y-c}dy.
  \end{align*}
Thanks to $\omega\bar{\Phi}(\pm1)=0$,  we get by Lemma \ref{lem:hardy} with $\delta_*=\alpha^{-1}$ that   \begin{align*}
     \|\partial_y\Phi\|_{L^2}^2+\alpha^2\|\Phi\|^2_{L^2}\leq& \left|\int_{-1}^{1}\f{\omega\bar{\Phi}}{y-c}dy\right|\leq C(\alpha^{-\f12}\|\partial_y(\omega\bar{\Phi})\|_{L^2} +\alpha^{\f12}\|\omega\bar{\Phi}\|_{L^2})\\
     \leq& C\Big(\alpha^{-\f12}\big(\|\partial_yw\|_{L^2}\|\Phi\|_{L^\infty} +\|\omega\|_{L^\infty}\|\partial_y\Phi\|_{L^2}\big)+\alpha^{\f12} \|\omega\|_{L^2}\|\Phi\|_{L^\infty}\Big)\\
     \leq & C\Big(\alpha^{-\f12}\big(\|\partial_yw\|_{L^2}\|\partial_y\Phi\|^{\f12}_{L^2}\|\Phi\|^{\f12}_{L^2} +\big(\|\partial_y\omega\|^{\f12}_{L^2}\|\omega\|_{L^2}^{\f12}+\|\omega\|_{L^2}\big)\|\partial_y\Phi\|_{L^2}\big)\\
     &\quad+\alpha^{\f12} \|\omega\|_{L^2}\|\partial_y\Phi\|^{\f12}_{L^2}\|\Phi\|^{\f12}_{L^2}\Big)\\
     \leq &C\alpha^{-1}\big(\|\partial_y\omega\|_{L^2}+\alpha\|\omega\|_{L^2}\big) \big(\|\partial_y\Phi\|_{L^2}+\alpha\|\Phi\|_{L^2}\big).
  \end{align*}
This shows that
\beno
\|\partial_y\Phi\|_{L^2}+\alpha\|\Phi\|_{L^2}\leq C\alpha^{-1}\big(\|\partial_y\om\|_{L^2}+\alpha\|\om\|_{L^2}\big).
\eeno

For the second inequality, let $\phi$ be a unique solution of $(\partial_y^2-\alpha^2)\phi=\Phi,\ \phi(\pm1)=0$. Then we have
\begin{align*}
 \|\Phi\|^2_{L^2}=\big\langle\Phi,(\partial_y^2-\alpha^2)\phi\big\rangle=-\int_{-1}^{1}\f{\omega\bar{\phi}}{y-c}dy \leq\left|\int_{-1}^{1}\f{\omega\bar\phi}{y-c}dy\right|.
 \end{align*}
If $|c_r|\geq1$, we get by Hardy's inequality that
  \begin{align*}
     & \|\Phi\|_{L^2}^2\leq \left|\int_{-1}^{1}\f{\omega\bar\phi}{y-c}dy\right|\leq\|\omega\|_{L^2} \left\|\f{\phi}{1-|y|}\right\|_{L^2}\leq \|\omega\|_{L^2}\|\partial_y\phi\|_{L^2}\leq C\alpha^{-1}\|\omega\|_{L^2}\|\Phi\|_{L^2},
  \end{align*}
 which gives
 \beno
 \|\Phi\|_{L^2}\leq C\alpha^{-1}\|\omega\|_{L^2}.
 \eeno
If $|c_r|<1$, by Lemma \ref{lem:hardy} with $\delta_*=1-|c_r|$ and Hardy's inequality, we have
  \begin{align*}
     \left|\int_{-1}^{1}\f{\omega\bar{\phi}}{y-c}dy\right|\leq&  \left|\int_{-1}^{1}\f{\omega(y)\bar{\phi}(c_r)}{y-c}dy\right|+  \left|\int_{-1}^{1}\f{\omega(y)\big(\bar{\phi}(y)-\bar{\phi}(c_r)\big)}{y-c}dy\right|\\
     \leq&|\phi(c_r)|\left|\int_{-1}^{1}\f{\omega(y)}{y-c}dy\right|+\|\omega\|_{L^2} \left\|\f{\phi(y)-\phi(c_r)}{y-c_r}\right\|_{L^2}\\
     \leq& C|\phi(c_r)|\big((1-|c_r|)^{\f12}\|\partial_y\omega\|_{L^2}+ (1-|c_r|)^{-\f12}\|\omega\|_{L^2}\big)+C\|\omega\|_{L^2}\|\partial_y\phi\|_{L^2}\\
     \leq& C|\phi(c_r)|\big((1-|c_r|)^{\f12}\|\partial_y\omega\|_{L^2}+ (1-|c_r|)^{-\f12}\|\omega\|_{L^2}\big)+C\alpha^{-1}\|\omega\|_{L^2}\|\Phi\|_{L^2}.
  \end{align*}
Thanks to $\phi(\pm1)=0$, we have
  \begin{align*}
     |\phi(c_r)|&\leq \min\left(\bigg|\int_{-1}^{c_r}\partial_y\phi(y)dy\bigg|, \bigg|\int_{c_r}^{1}\partial_y\phi(y)dy\bigg| \right)\\&\leq \min\left((1+c_r)^{\f12},(1-c_r)^{\f12}\right)\|\partial_y\phi\|_{L^2}\\
     &\leq (1-|c_r|)^{\f12}\|\partial_y\phi\|_{L^2}\leq \alpha^{-1}(1-|c_r|)^{\f12}\|\Phi\|_{L^2}.
  \end{align*}
 Summing up, we obtain
  \begin{align*}
 \|\Phi\|^2_{L^2}\leq \left|\int_{-1}^{1}\f{\omega\bar{\phi}}{y-c}dy\right| \leq C\alpha^{-1}\big((1-|c_r|)\|\partial_y\omega\|_{L^2} +\|\omega\|_{L^2}\big)\|\Phi\|_{L^2}.
  \end{align*}
This shows the second inequality.
\end{proof}

The following lemma is a variant of Lemma \ref{lem:hardy}.

  \begin{Lemma}\label{lem:hardy-V}
Let $V$ satisfy \eqref{ass:V}. Let $ a,\lambda\in\R,$ $a\neq0$, $f\in H^1(\Omega)$, $f(x,\pm1,z)=0$ or $0<\delta_{*}\leq 1-|\lambda|$. Then it holds that
  \begin{align*}
  \left|\int_{I}\f{f(x,y,z)}{V(y,z)-\lambda+ia}dy\right| \leq C\big(\delta_*^{\f12}\|\partial_yf\|_{L^2_y}+ \delta_*^{-\f12}\|f\|_{L^2_y}\big).
  \end{align*}
 where $C$ is a  constant independent of $a, \lambda, x, z$ and $\delta_*$.
\end{Lemma}

\begin{proof}
Make a change of variable: $(x,y,z)\rightarrow (x,V(y,z),z)\triangleq(X,Y,Z)$. We denote
\beno
F(x,V(y,z),z)= f(x,y,z),\quad \psi\big(V(y,z),z\big)=\partial_yV(y,z),\quad  H(V(y,z),z)=\partial_y^2V(y,z).
\eeno
For  fixed $x,z$, it is obvious that
\begin{align*}
 \f12\leq|\psi|\leq 2,\quad \|H\|_{L^\infty_Y}\leq 1,\quad \psi(V(y,z),z)\partial_{Y}F(x,V(y,z),z)=\partial_yf(x,y,z)
\end{align*}
 so that
 \beno
 \|F\|_{L^2_{Y}}\leq C\|f\|_{L^2_y},\quad \|\partial_{Y}F\|_{L^2_{Y}}\leq C\|\partial_yf\|_{L^2_y}.
 \eeno
  Then we have
  \begin{align*}
     & \left|\int_{-1}^{1}\f{f(x,y,z)}{V-\lambda+ia}dy\right|=\left|\int_{-1}^{1} \f{F(x,Y,z)\psi(Y,z)^{-1}}{Y-\lambda+ia}dY\right|.
  \end{align*}
  By Lemma \ref{lem:hardy} and using the fact that $\partial_Y(\psi^{-1})=-\psi^{-3}H$, we get
  \begin{align*}
     & \left|\int_{-1}^{1} \f{F(x,Y,z)\psi^{-1}(Y,z)}{Y-\lambda+ia}dY\right|\\
     &\leq C\big(\delta_*^{\f12}\|\partial_{Y}(F\psi^{-1})\|_{L^2_{Y}} +\delta^{-\f12}_{*}\|F\psi^{-1}\|_{L^2_{Y}}\big)\\
     &\leq C\big(\delta_*^{\f12}\|\partial_{Y}F\|_{L^2_{Y}}\|\psi^{-1}\|_{L^\infty_Y} +\delta^{\f12}_*\|F\|_{L^2_{Y}}\|\psi^{-3}\|_{L^\infty_Y}\|H\|_{L^\infty_Y} +\delta^{-\f12}_{*}\|F\|_{L^2_{Y}}\|\psi^{-1}\|_{L^\infty_Y}\big)\\
     &\leq C\big(\delta_*^{\f12}\|\partial_{y}f\|_{L^2_{y}} +\delta^{\f12}_*\|f\|_{L^2_y} +\delta^{-\f12}_{*}\|f\|_{L^2_{y}}\big)\leq C\big(\delta_*^{\f12}\|\partial_yf\|_{L^2_y}+\delta^{-\f12}_*\|f\|_{L^2_y}\big).
  \end{align*}
This proves the lemma.
\end{proof}

\subsection{A simple algebraic inequality}

\begin{Lemma}\label{lem:AB-ineq}
 Let $D=A_1-A_2+iB$ with $A_1,A_2,B\in\mathbb{R}$, $|B|\geq 2A_2$ and $|B|,A_1,A_2>0$. Then there exists a constant $c>0$, so that
\begin{align*}
   &\mathbf{Re}\big(\sqrt{D}\big)\geq c\big(A_1+|B|\big)^{\f12},
\end{align*}
here $\sqrt{z}$ be the branch of the square root defined on the complement of the non-positive real numbers.
\end{Lemma}
\begin{proof}
without loss of generality, we may assume that $B>0$ . Then $\sqrt{D}=a+ib$ with $a,b\in\mathbb{R}$ and $a,b>0$. Then we have
\beno
a^2-b^2=A\triangleq A_1-A_2,\quad 2ab=B.
\eeno
A direct calculation gives
  \begin{align*}
    a^2 = A+b^2\geq -B/2+b^2\geq b^2-ab\Rightarrow&(a/b)^2+(a/b)-1\geq0.
  \end{align*}
  As $a/b> 0$, we get $a/b\geq (\sqrt{5}-1)/2$. Thus, there exist constants $c_1>c_2>0$, such that
   \begin{align*}
      a^2&\geq c_1(a^2+b^2)=c_1|D|= c_1((A_1-A_2)^2+B^2)^{\f12}\geq c_2(A_1-A_2+|B|)\geq c_2(A_1+|B|/2),
   \end{align*}
 which implies that
  \begin{align*}
     \mathbf{Re}(\sqrt{D})&=a\geq c(A_1+|B|)^{\f12}.
  \end{align*}
\end{proof}

\subsection{Gearhart-Pr\"{u}ss type lemma}

An operator $A$ in a Hilbert space $H$ is accretive if ${\rm Re} \langle Af,f\rangle\geq 0$ for all $f\in {D}(A),$ or equivalently $\|(\la+A)f\|\geq \la\|f\|$ for all $f\in {D}(A)$ and all $\la>0$. The operator $A$ is called m-accretive if in addition any $\la<0$ belongs to the resolvent set of $A$. We define
\beno
\Psi(A)=\inf\big\{\|(A-i\la)f\|: f\in {D}(A),\ \la\in \mathbb{R},\  \|f\|=1\big\}.
\eeno

The following Gearhart-Pr\"{u}ss type lemma  comes from \cite{Wei}.

\begin{Lemma}\label{lem:GP}
Let $A$ be an m-accretive operator in a Hilbert space $H$. Then we have
\beno
\|e^{-tA}\|\leq e^{-t\Psi(A)+{\pi}/{2}}\quad  \text{for any}\,\,t\geq 0.
\eeno
\end{Lemma}

\section*{Acknowledgement}
The authors thank Te Li for helpful discussions. Z. Zhang is partially supported by NSF of China under Grant 11425103.

\end{document}